\documentclass[10pt]{amsart}

\usepackage{latexsym,exscale,enumerate,amsfonts,amssymb,xparse, mathtools}
\usepackage[normalem]{ulem}
\usepackage{amsmath,amsthm,amsfonts,amssymb,amscd, stmaryrd,textcomp}
\usepackage[bookmarks=false]{hyperref}
\usepackage[usenames,dvipsnames]{xcolor}
\usepackage{enumitem}
\usepackage{verbatim}
\usepackage{ytableau}
\usepackage{euscript}

\usepackage{bbm}

\colorlet{green}{black!30!green}

\definecolor{myblue}{RGB}{109, 156, 179}

\usepackage{xcolor}
\definecolor{myorange}{rgb}{1,0.647,0}

\definecolor{mypurple}{cmyk}{.6,.9,0, .11}

\newcommand{\red}[1]{{\color{red}#1}}
\newcommand{\gr}[1]{{\color{green}#1}}
\newcommand{\bl}[1]{{\color{blue}#1}}

\newcommand{\grb}{\gr{\bullet}}
\newcommand{\blb}{\bl{\bullet}}
\newcommand{\redb}{\red{\bullet}}

\numberwithin{equation}{section}

\theoremstyle{definition}
\newtheorem{thm}{Theorem}[section]
\newtheorem{cor}[thm]{Corollary}

\newtheorem{conj}[thm]{Conjecture}
\newtheorem{lem}[thm]{Lemma}
\newtheorem{rem}[thm]{Remark}
\newtheorem{conv}[thm]{Convention}

\newtheorem{prop}[thm]{Proposition}
\newtheorem{defn}[thm]{Definition}
\newtheorem{example}[thm]{Example}
\newtheorem{question}[thm]{Question}

\newtheorem{nota}[thm]{Notation}
\newtheorem{hypothesis}[thm]{Hypothesis}

\usepackage{tikz}
\usetikzlibrary{cd}
\usetikzlibrary{snakes}
\usetikzlibrary{decorations.markings}
\usetikzlibrary{decorations.pathreplacing}
\usetikzlibrary{arrows,shapes,positioning}
\usetikzlibrary{decorations.pathmorphing}
\tikzset{anchorbase/.style={baseline={([yshift=-0.5ex]current bounding box.center)}},
tinynodes/.style={font=\tiny,text height=0.75ex,text depth=0.15ex},
smallnodes/.style={font=\scriptsize,text height=0.75ex,text depth=0.15ex},
>={Latex[length=1mm, width=1.5mm]},
overcross/.style={line width=4pt,color=white},
}
\tikzstyle directed=[postaction={decorate,decoration={markings,
	mark=at position #1 with {\arrow{>}}}}]
\tikzstyle rdirected=[postaction={decorate,decoration={markings,
	mark=at position #1 with {\arrow{<}}}}]
    
\newcommand{\CQGbox}[1]{
\begin{tikzpicture}
\node[draw, fill=white,rounded corners=4pt,inner sep=3pt] (X) at (0,.75) {\scs$#1$};
\end{tikzpicture}}

\newcommand{\CQGsbox}[1]{
\begin{tikzpicture}
\node[draw, fill=white,rounded corners=4pt,inner sep=2pt] (X) at (0,.75) {\scs$#1$};
\end{tikzpicture}}

\newcommand{\NB}[1]{
\begin{tikzpicture}[anchorbase]
\node[draw, thick,  fill=white,rounded corners=4pt,inner sep=3pt] (X) at (0,.75) {\scs$#1$};
\end{tikzpicture}}

\newcommand{\arc}[3]{arc[start angle=#1,end angle=#2,radius=#3]}

\newcommand{\Hom}{{\rm Hom}}

\newcommand{\End}{{\rm End}}

\renewcommand{\to}{\rightarrow}

\newcommand{\id}{{\rm id}}

\newcommand{\sgn}{{\rm sgn}}

\newcommand{\Sym}{\operatorname{Sym}}

\let\hat=\widehat
\let\tilde=\widetilde
\newcommand{\scs}{\scriptsize}
\def\C{{\mathbb C}}

\def\R{{\mathbb R}}
\def\Z{{\mathbb Z}}
\def\Q{{\mathbb Q}}
\def\K{{\mathbb K}}
\def\X{{\mathbb X}}
\def\Zge{\mathbb{Z}_{\ge 0}}

\newcommand{\sln}[1][N]{\mathfrak{sl}_{#1}}
\newcommand{\gln}[1][N]{\mathfrak{gl}_{#1}}

\newcommand{\slm}{\mathfrak{sl}_{m}}
\newcommand{\glm}{\mathfrak{gl}_{m}}
\newcommand{\spn}[1][2n]{\mathfrak{sp}_{#1}}
\newcommand{\son}[1][2n+1]{\mathfrak{so}_{#1}}
\newcommand{\som}[1][m]{\mathfrak{so}_{#1}}

\newcommand{\Rep}{\mathbf{Rep}}

\newcommand{\HT}{\mathbf{HT}}
\newcommand\mydots{\makebox[1em][c]{$\cdot$\hfil$\cdot$\hfil$\cdot$}}
\newcommand\mysdots{\makebox[.5em][c]{$\cdot$\hfil$\cdot$\hfil$\cdot$}}
\newcommand{\ot}{\otimes}

\DeclareMathOperator{\Kar}{Kar}

\newcommand{\SG}{\mathfrak{S}}

\newcommand{\one}{\mathbbm{1}}
\newcommand{\Br}{\mathrm{Br}}
\newcommand{\ev}{\mathbf{ev}}
\newcommand{\coev}{\mathbf{coev}}
\newcommand{\Cl}{\mathrm{Cl}}

\newcommand{\Gr}[2]{\mathrm{Gr}_{#1}(\C^{#2})}
\newcommand{\flip}{\mathrm{flip}}
\newcommand{\bP}{\overline{P}}
\newcommand{\mirror}{\mathsf{m}}
\newcommand{\rP}{\tilde{P}}

\newcommand{\tr}{\operatorname{Tr}}
\newcommand{\trq}{\operatorname{Tr}_q}

\newcommand{\dU}{\dot{U}_q}
\newcommand{\ZdU}{\prescript{}{\Z}{\dot{U}}_q}
\newcommand{\dS}{\dot{S}_q}
\newcommand{\ZdS}{\prescript{}{\Z}{\dot{S}}_q}
\newcommand{\dSn}{\dot{S}_q^{\le N}}
\newcommand{\ZdSn}{\prescript{}{\Z}{\dot{S}}_q^{\le N}}

\newcommand{\iqUsom}[1][m]{{U}^{\iota}_{-q^2}(\som[#1])}

\newcommand{\wt}{\mathbf{a}}
\newcommand{\wtt}{\mathbf{b}}
\newcommand{\wtn}{\mathbf{n}} 
\newcommand{\wtone}{\mathbf{1}}
\newcommand{\wtk}{\mathbf{k}} 
\newcommand{\RM}{{R}}
\newcommand{\QW}{c}
\newcommand{\iQW}{\prescript{\iota}{}{T}}
\newcommand{\SH}{\mathrm{SH}}

\newcommand{\ee}{\alpha}
\newcommand{\EE}{\mathcal{E}}
\newcommand{\FF}{\mathcal{F}}
\newcommand{\UU}{\mathcal{U}}
\newcommand{\U}{\check{\mathcal{U}}}
\newcommand{\cUU}{\check{\mathcal{U}}}

\newcommand{\XX}{\FECat}
\newcommand{\cXX}{\check{\FECat}}
\newcommand{\cXXtn}{\check{\FECat}^{\le 2n}}
\newcommand{\SSS}{\mathcal{S}}
\newcommand{\cSSS}{\check{\mathcal{S}}}
\newcommand{\SSSn}{\mathcal{S}^{\leq N}}
\newcommand{\cSSSn}{\check{\mathcal{S}}^{\leq N}}
\newcommand{\SSStn}{\mathcal{S}^{\leq 2n}}
\newcommand{\cSSStn}{\check{\mathcal{S}}^{\leq 2n}}

\newcommand{\hComp}{\star}

\newcommand{\parti}{\lambda}
\newcommand{\Schur}[1][\parti]{\mathfrak{s}_{#1}}

\newcommand{\swcl}{{\color{green} {\mathrm{LC}}}}
\newcommand{\swclp}{{\color{purple} {\mathrm{LC}}}}
\newcommand{\swcr}{{\color{green} {\mathrm{RC}}}}
\newcommand{\FEk}{\mathbf{X}^{(k)}}
\newcommand{\FEkk}{\mathbf{X}^{(k+1)}}
\newcommand{\FE}{\mathbf{X}}

\newcommand{\ourmod}{\mathrm{mod} \;}

\newcommand{\qsh}{\mathbf{q}}
\newcommand{\qdim}{\mathrm{dim}_{\qsh}}
\newcommand{\tsh}{\mathbf{t}}
\newcommand{\ssh}{\mathbf{s}}

\newcommand{\bgen}{\beta}
\newcommand{\bword}{\underline{\bgen}}
\newcommand{\brgroup}[1][m]{\mathrm{Br}_{#1}}

\renewcommand{\Im}{\operatorname{Im}}
\newcommand{\rker}{\operatorname{Rker}}
\newcommand{\lker}{\operatorname{Lker}}

\newcommand{\Ya}[1][]{\mathsf{Y}_{\!#1}}
\newcommand{\Pa}[1][]{\reflectbox{\rotatebox[origin=c]{180}{$\mathsf{Y}$}}^{\!#1}}
\newcommand{\cu}[1][]{\mathbf{U}_{#1}}
\newcommand{\ca}[1][]{\reflectbox{\rotatebox[origin=c]{180}{$\mathbf{U}$}}^{#1}}
\newcommand{\xx}{\mathsf{X}}
\newcommand{\hh}{\mathsf{H}}
\newcommand{\II}{\reflectbox{\rotatebox[origin=c]{90}{$\hh$}}}
\newcommand{\nx}{\mathsf{x}}

\newcommand{\ak}{\mathbb K}

\newcommand{\grdim}{{\rm grdim}}

\newcommand{\BC}{\mathbb{B}}
\newcommand{\EBC}{{^{\sigma}\mathbb{B}}}
\newcommand{\tauEBC}{{^{\tau}\mathbb{B}}}

\newcommand{\Cat}{\mathcal{A}}
\newcommand{\ECat}[1][\Cat]{#1^{\sigma}}
\newcommand{\ChCat}[1][\Cat]{\mathrm{Ch}(#1)}
\newcommand{\ChECat}{\ChCat[\ECat]}
\newcommand{\dgCat}[1][\Cat]{{\EuScript C}(#1)}
\newcommand{\dgECat}{{\EuScript C}(\ECat)}
\newcommand{\hCat}[1][\Cat]{{\mathcal K}(#1)}
\newcommand{\dCat}[1][\Cat]{{\EuScript D}(#1)}
\newcommand{\Vect}{\mathrm{Vect}}
\newcommand{\sVect}{\mathrm{sVect}}
\newcommand{\tauKzero}[1]{K_0^\tau(#1)}
\newcommand{\wKzero}[1]{K_0^\sigma(#1)}
\newcommand{\Kzero}[1]{K_0(#1)}
\newcommand{\Ide}{\mathcal{I}}
\newcommand{\RR}{\EuScript R}
\newcommand{\SLC}[1]{\{#1\}_{2n+1,\mathcal{S}}}

\newcommand{\Catn}{\Cat_n}

\newcommand{\CatBC}{\Cat(\mathbb{B})}
\newcommand{\ECatBC}{\ECat(\mathbb{B})}

\newcommand{\sigast}{{\sigma\ast}}
\newcommand{\sigbast}{{\sigma_b\ast}}
\newcommand{\tauast}{{\tau\ast}}

\newcommand{\sh}{\mathrm{sh}}

\newcommand{\For}{\mathsf{U}}

\newcommand{\LIP}{\mathbf{L}}

\newcommand{\pr}{\mathsf{pr}}
\newcommand{\inc}{\mathsf{in}}

\newcommand{\BFoam}[1][m]{{\EuScript B}_{#1}}
\newcommand{\EBFoam}[1][m]{({\EuScript{B}}_{#1})^{\tau}}
\newcommand{\BnFoam}[1][m]{{\EuScript B}_{#1}^{n}}
\newcommand{\EBnFoam}[1][m]{({\EuScript{B}}_{#1}^{n})^{\tau}}

\newcommand{\crazyF}{{F}}
\DeclareMathOperator{\ind}{ind}

\newcommand{\indX}{
  \mathchoice
    {X\mskip-12.5mu X}
    {X\mskip-12.5mu X}
    {X\mskip-12mu X}
    {X\mskip-11.5mu X}
}

\newcommand{\FECat}{
  \mathchoice
    {\mathcal{F}\mskip-6mu \mathcal{E}\!}
    {\mathcal{F}\mskip-6mu \mathcal{E}\!}
     {\mathcal{F}\mskip-5.5mu \mathcal{E}\!}
    {\mathcal{F}\mskip-5mu \mathcal{E}\!}
}

\newcommand{\FECatprime}{
  \mathchoice
    {\mathcal{F}\mskip-6mu \mathcal{E}'\!\!_{q}}
    {\mathcal{F}\mskip-6mu \mathcal{E}'\!\!_{q}}
     {\mathcal{F}\mskip-5.5mu \mathcal{E}'\!\!_{q}}
    {\mathcal{F}\mskip-5mu \mathcal{E}'\!\!_{q}}
}

\newcommand{\Ztwogroupoid}{\Omega_{\Z/2}}

\begin{document}
%

\title[]{Spin Link Homology}

\author{Elijah Bodish}
\address{Department of Mathematics, Massachusetts Institute of Technology, Building 2, Cambridge, MA 02139-4307, USA}
\email{ebodish@mit.edu}

\author{Ben Elias}
\address{Department of Mathematics, University of Oregon,
Fenton Hall, Eugene, OR 97403-1222, USA}
\email{belias@uoregon.edu}

\author{David E. V. Rose}
\address{Department of Mathematics, University of North Carolina, 
Phillips Hall CB \#3250, UNC-CH, Chapel Hill, NC 27599-3250, USA}
\email{davidrose@unc.edu}

\begin{abstract}
We put a new spin on Khovanov--Rozansky homology. 
That is, we equip $\Lambda^n$-colored $\sln[2n]$ Khovanov--Rozansky 
homology with an involution whose $\pm 1$-eigenspaces are link invariants.
When $n=1,2,3$ (and assuming technical conjectures for $n \geq 4$), 
we prove that this refined invariant categorifies
the spin-colored $\son$ quantum link polynomial.
Along the way, we partially develop the theory of 
quantum $\son$ webs 
and make contact with $\iota$quantum groups.
\end{abstract}

\maketitle

\setcounter{tocdepth}{1}

\tableofcontents

%
\section{Introduction}
%

For each finite-dimensional simple complex Lie algebra $\mathfrak{g}$, 
work of Reshetikhin--Turaev \cite{RT1} defines
a Laurent polynomial-valued invariant of a link $\mathcal{L} \subset S^3$ 
with components colored by finite-dimensional representations of the 
corresponding quantum group $U_q(\mathfrak{g})$.
In the case that $\mathfrak{g} = \sln[2]$ and the link is colored by the defining representation, 
this invariant is the much-celebrated Jones polynomial \cite{Jones}.
In pioneering work \cite{Kh1}, 
Khovanov showed that the Jones polynomial admits a \emph{categorification} 
taking the form of a bigraded 
homological link invariant $H(\mathcal{L})$ 
from which the Jones polynomial can be recovered by taking the Euler characteristic.
Subsequent works of Khovanov--Rozansky \cite{KhR} (and others \cite{Man,CK02,MS2}) 
construct analogous \emph{link homology theories} in the $\mathfrak{g}=\sln$ case. 
Khovanov--Rozansky homology is defined in terms of explicit chain complexes, 
and a number of subsequent formulations make it even more accessible 
and computable (e.g.~the cobordism and foam formulations given in \cite{BN2,QR1,ETW}). 
These explicit descriptions underlie many applications of 
Khovanov(--Rozansky) homology in low-dimensional topology 
(e.g.~\cite{Ras,MWW,RenWillis}).

In his ICM address \cite{Kh6}, Khovanov posed the ``difficult open problem''
of categorifying the polynomial invariants of links associated
to arbitrary $\mathfrak{g}$ and their irreducible representations. 
The first solution to this problem was obtained by Webster in \cite{Web}. 
Therein, he defines algebras that categorify tensor product representations 
of $U_q(\mathfrak{g})$, 
and constructs derived functors corresponding to tangles. 
These derived functors categorify the Reshetikhin--Turaev invariants of tangles, 
hence, in particular, the $U_q(\mathfrak{g})$ Reshetikhin--Turaev link invariants.
However, Webster's categorical invariants are famously difficult to compute:
for example, we are unaware of anyone having computed 
the Webster homology of the unknot outside of type $A$.
Lack of computability has severely hindered the development 
and application of these link homologies. 

In the present paper we provide a new construction for link homology 
(not a priori related to Webster's) in the case 
that $\mathfrak{g} = \son$ and the link is colored by the \emph{spin representation}.
Our construction is based on a new ``folded categorical skew Howe duality,'' 
a surprising connection between spin-colored $\son$ link invariants 
and $\Lambda^n$-colored $\sln[2n]$ link invariants that is only visible at the categorified level. 
Along the way, we study endomorphisms of tensor powers 
of the quantum spin representation in detail; 
this work thus constitutes the first steps towards solving Kuperberg's ``webs problem'' in type $B$.
We also make contact with $\iota$quantum groups and their categorifications.

\subsection{Our construction, in a nutshell}
	\label{ss:nutshell}

The starting point for our investigation 
is a meta principle pioneered by Lusztig \cite{Lus4}: 
that representation-theoretic structures in non-simply laced type should arise via 
the process of \emph{folding} categorical structures in simply laced type along diagram automorphisms.
This principle is well-known; nevertheless, 
until now, the implementation of this principle in the context of link homology 
has remained elusive. The example we study is the folding of the root system $A_{2n-1}$ 
to the root system of type $B_n$.

Our construction was originally motivated by examination of the invariants of colored unknots. 
Consider first the case $n=2$. 
The root system $B_2 (=C_2)$ can be viewed as folded from the root system $A_3$, 
with (the highest weight of) the $\sln[4]$ representation $\Lambda^2\C^4$
corresponding to (the highest weight of) the spin representation $S$ of $\son[5]$:
\[
\begin{tikzpicture}[anchorbase]
	\draw (0,0) to (2,0);
	\draw[fill=white] (0,0) circle (.125);
	\draw[fill=white] (1,0) circle (.125);
	\draw[fill=white] (2,0) circle (.125);
	\node at (1,0) {$\ast$};
\end{tikzpicture} 
\quad \xmapsto{\text{fold}} \quad
\begin{tikzpicture}[anchorbase]
	\draw (0,.05) to (1,.05);
	\draw (0,-.05) to (1,-.05);
	\draw[fill=white] (0,0) circle (.125);
	\draw[fill=white] (1,0) circle (.125);
	\node at (0,0) {$\ast$};
	\node at (.5,0) {$<$};
\end{tikzpicture} \, .
\]
The associated quantum invariants $P_\mathfrak{g}$ 
of unknots (colored by these representations) are 
\begin{equation}
	\label{eq:unknots}
P_{\sln[4]}\left( 
\begin{tikzpicture}[scale =.5, anchorbase]
	\draw[very thick] (0,0) node[right=5pt,yshift=5pt]{\scs$\Lambda^2$} circle (.5);
\end{tikzpicture} \right) 
= q^4 + q^2 + 2 + q^{-2} + q^{-4}
\quad \text{and} \quad
P_{\son[5]}\left( 
\begin{tikzpicture}[scale =.5, anchorbase]
	\draw[very thick] (0,0) node[right=5pt,yshift=5pt]{\scs$S$} circle (.5);
\end{tikzpicture} \right) 
= q^4 + q^2 + q^{-2} + q^{-4} \, .
\end{equation}
The starting observation is that the latter invariant is obtained from the former by removing 
the summand $2$, 
and that categorification provides a rigorous means for doing so.

When $N$ is understood, let $\Lambda^k$ denote the exterior product $\Lambda^k \C^N$, 
or its quantum analogue.
In the $\sln$ Khovanov--Rozansky theory, 
the invariant of a $\Lambda^k$-colored 
unknot is a degree-shifted version 
of the cohomology of the Grassmannian $\Gr{k}{N}$. 
In the case of \eqref{eq:unknots}, the $\sln[4]$ homology 
of the $\Lambda^2$-colored unknot is $\qsh^{-4} H^\ast(\Gr{2}{4})$, 
which has a basis indexed by partitions $\parti$ with Young diagrams 
fitting inside a $2 \times 2$ box; 
here the power of $\qsh$ indicates a shift 
of this graded vector space (down, by $4$). 
Observe that the graded dimension of $\qsh^{-4} H^\ast(\Gr{2}{4})$ 
is exactly $q^4 + q^2 + 2 + q^{-2} + q^{-4}$.
On the other hand, there is a natural involution on 
$H^\ast(\Gr{2}{4})$, given by taking the transpose partition 
(induced by the involution on $\Gr{2}{4}$ 
that takes the perpendicular $2$-plane). 
If we instead take the graded trace of this involution, 
we recover $q^4 + q^2 + q^{-2} + q^{-4}$.

This behavior persists for all $n \geq 2$.
Specifically, the $\sln[2n]$ Khovanov--Rozansky homology 
of the $\Lambda^n$-colored unknot is 
isomorphic to $\qsh^{-n^2} H^\ast(\Gr{n}{2n})$
and this space admits an analogous involution $\tau$ 
with graded trace that is equal to the spin-colored 
$\son$ unknot invariant. Precisely, 
\[
\tr \big(
\begin{tikzcd}[column sep=1pt]
\arrow[loop left, looseness=3, "\tau"] \phantom{\Gr{n}{2n}}  
\end{tikzcd}
\!\!\!\!\!\!\!\!\!\!\!\!\!\!\!\!\!\!\!\!\!\!\!\!\!\!\!\!
\qsh^{-n^2}H^\ast(\Gr{n}{2n})
	\big) 
	= \prod_{i=1}^n (q^{2i-1} + q^{1-2i})
	= P_{\son}\left( 
\begin{tikzpicture}[scale =.5, anchorbase]
	\draw[very thick] (0,0) node[right=5pt,yshift=5pt]{\scs$S$} circle (.5);
\end{tikzpicture} \right) \, .
\]
A consequence of the main results of this paper is that 
these observations may be extended to all links.

\begin{thm}[Corollary \ref{cor:decatastrace} and Remark \ref{rem:conditional}]
	\label{thm:introtrace}
Let $\mathcal{L} \subset S^3$ be a link. 
\begin{itemize}
\item For all $n \geq 1$, the $\Lambda^n$-colored $\sln[2n]$ Khovanov--Rozansky homology 
$H_{\sln[2n]}(\mathcal{L}^{\Lambda^n})$ admits an involution $\tau$ that preserves the bidegree. 
The bigraded eigenspaces of $\tau$ are link invariants.
\item Assume that $n=1,2,3$,  
or that Conjectures
\ref{conj:R3}, \ref{conj:TrX}, and \ref{conj:TrX2} hold.
Then
\begin{equation}
	\label{eq:introtrace}
\tr \big(
\begin{tikzcd}
H_{\sln[2n]}(\mathcal{L}^{\Lambda^n}) \arrow[loop left, looseness=3, "\tau"]
\end{tikzcd}
	\big) 
= \rP_{\son}(\mathcal{L}^S) \, ,
\end{equation}
i.e.,
the pair $\big( H_{\sln[2n]}(\mathcal{L}^{\Lambda^n}) , \tau \big)$ 
categorifies an appropriate renormalization $\rP_{\son}(\mathcal{L}^S)$
of the spin-colored $\son$ quantum link polynomials.
\qed
\end{itemize}
\end{thm}

The two parts of this theorem represent two distinct goals of this paper. 
The first is to define a new categorical link invariant, and this we achieve for all $n \geq 1$.
The second goal is to determine what link polynomial our invariant categorifies. 
We expect it to categorify the spin-colored link polynomial for all $n \geq 1$, 
but technical issues of a combinatorial nature (primarily in the decategorification) 
obstruct our ability to verify this for $n > 3$. 
With sufficient time, 
it should be possible to check the relevant conjectures for any fixed value of $n$.

Theorem \ref{thm:introtrace} is a slight reformulation and decategorification of our main construction, 
which defines \emph{spin link homology} $H_{\son}(\mathcal{L}^S)$, 
a link invariant, valued in bigraded super vector spaces, that refines 
$\Lambda^n$-colored $\sln[2n]$ Khovanov--Rozansky homology. 
As will be clear from our construction,
this invariant is as computable\footnote{
That said, we do not include computations beyond the unknot here. 
The paper is already quite long, 
and (our) PhD students need things to do.}
as the Khovanov--Rozansky theory.
We thus propose that $H_{\son}(\mathcal{L}^S)$ 
are the first readily computable link homologies  
associated to simple complex $\mathfrak{g} \neq \sln$.

Before proceeding, a word of caution.
As the picture is painted above, 
it seems as though the Dynkin diagram automorphism of $A_{2n-1}$ 
has been directly transformed into an involution $\tau$ on type $A$ link homology. 
This is not at all what we do! 
Our construction of $\tau$ arises through duality, 
which we now explain.

\subsection{Our construction, in more depth}\label{ss:inmoredepth}

Suppose one wishes to compute the $\sln$ Reshetikhin--Turaev invariant of a tangle
with components colored by fundamental representations. 
One can view this invariant as a morphism between an $m$-fold tensor product of the representations 
$\Lambda^a$ (with $0\leq a \leq N$), for some\footnote{The number $m$ could be higher than 
the number of boundary points of the tangle, 
as one may need to add additional copies 
$\Lambda^0$ and/or $\Lambda^N$ 
of the trivial representation 
to be the source or target of ``cup'' and ``cap'' tangles.} 
natural number $m$. 
Pioneering results of Cautis, Kamnitzer, Licata, and Morrison \cite{CKL,CKM} establish a duality between 
the subcategory of $\Rep(U_q(\sln))$ consisting of such $m$-fold tensor products, 
and the idempotented quantum group $\dU(\glm)$. 
In order to state the relevant representation theory correctly, 
we now replace $\sln$ with $\gln$; 
there is little distinction\footnote{The modern view is that categorical constructions 
are most naturally associated with $\gln$, 
so we denote them thusly; however, as in the literature, 
we continue to refer to the link invariants as $\sln$ (Khovanov--Rozansky) link homology.} 
between $\gln$ and $\sln$ link invariants. 

A consequence of the \emph{quantum skew Howe duality} 
proved in \cite{CKM}
is the existence of a full functor
\begin{equation}
	\label{eq:SHintro}
\dU(\glm) \xrightarrow{\SH} \Rep(U_q(\gln)) \, .
\end{equation}
Here, one views the idempotented algebra $\dU(\glm)$ as a category 
with one object for each $\glm$ weight $\wt = (a_1,\ldots,a_m)$, 
and $\SH(\wt) = \Lambda^\wt := \Lambda^{a_1} \ot \cdots \ot \Lambda^{a_m}$.
After factoring through a quotient $\dU^{\leq N}(\glm)$, 
the functor induced by \eqref{eq:SHintro} is fully faithful. 
Consequently, the link polynomials $P_{\sln}(\mathcal{L})$, 
which are defined using the braided monoidal structure on $\Rep(U_q(\gln))$, 
can be described entirely in terms of $\dU^{\leq N}(\glm)$. 
The same is true at the categorical level \cite{Cautis}, 
with $H_{\sln}(\mathcal{L})$ admitting a formulation in 
the bounded homotopy category of complexes over an analogous quotient 
$\cUU_q^{\leq N}(\glm)$ of the categorified quantum group $\cUU_q(\glm)$.
The categorified quantum group is reviewed in full detail in \S \ref{s:BCQG} 
and the analogous quotient is recalled in Definition \ref{def:CSQ} 
(where it is denoted $\cSSSn_q(\glm)$).
In order to precisely state our results, we remind the reader that 
$1$-morphisms of $\cUU_q(\glm)$ are generated by 
elements $\EE_i^{(k)} \one_{\wt}$ and $\FF_i^{(k)} \one_{\wt}$ that lift 
divided powers of the Chevalley generators of $\dU(\glm)$.

\begin{rem} 
A categorification of \cite{CKM} is given in \cite{QR1}.
Therein, a $2$-category of $\gln$ foams is constructed, 
which categorifies the image of \eqref{eq:SHintro}.
As in the decategorified case, 
this foam $2$-category is in duality with $\cUU_q(\glm)$.  \end{rem}

Now consider the case where $N = 2n$. 
As we show in this paper, 
upon passage through (categorical) skew Howe duality, 
the Dynkin automorphism of $\gln[2n]$ manifests as an involution on $\dU(\glm)$
which is akin to the Chevalley involution.
Namely, it swaps the Chevalley generators $e_i \leftrightarrow f_i$ in $\dU(\glm)$
and sends a weight $\wt$ to the weight $2\wtn - \wt = (2n-a_1, \ldots, 2n-a_m)$. 
Precisely, we prove the following result at the categorical level.

\begin{thm}[Theorems \ref{thm:symmetry} and \ref{thm:involution} 
	and Corollaries \ref{cor:order4}, \ref{cor:thickinvolution} and \ref{cor:atlastinv}]
	\label{thm:invintro}
For each $n \geq 1$, there is an order $4$ 
automorphism\footnote{Our automorphism $\tau_n$ 
is unrelated to other symmetries of $\cUU_q(\glm)$ 
appearing in the literature, despite its similar action on $1$-morphisms.}
$\tau_n$ of the 
categorified quantum group $\cUU_q(\glm)$ that swaps generating $1$-morphisms 
$\EE_i^{(k)} \one_{\wtn + \wt} \leftrightarrow \FF_i^{(k)} \one_{\wtn - \wt}$.
Further, $\tau_n$ restricts to an involution on the $2$-subcategory 
$\cXX_q(\glm) \subset \cUU_q(\glm)$ generated by the $1$-endomorphisms
$\FF_i^{(k)}\EE_i^{(k)} \one_{\wt}$ and $\EE_i^{(k)}\FF_i^{(k)} \one_{\wt}$.
This involution descends to the corresponding $2$-subcategory 
$\cXXtn_q(\glm) \subset \cUU_q^{\leq 2n}(\glm)$, 
where it extends the involution on $H^\ast(\Gr{n}{2n})$.
\qed
\end{thm}

To explain the final sentence of Theorem \ref{thm:invintro}, 
we point out that the endomorphism algebra of the identity $1$-morphism $\one_\wt$
of the weight $\wt$ in $\cUU^{\le N}_q(\glm)$ 
is isomorphic to $\otimes_{i=1}^m H^\ast(\Gr{a_i}{N})$, 
which is the $\sln$ Khovanov--Rozansky homology of the $\wt$-colored unlink
(up to degree shift). 
When $N=2n$, our involution $\tau$ of $\cXXtn_q(\glm)$ restricts in weight $\wtn = (n,\ldots,n)$
to give an involution of $\otimes_{i=1}^m H^\ast(\Gr{n}{2n})$,
which agrees with (the $m$-fold tensor product of) the involution from \S \ref{ss:nutshell}.
Note that the weight $\wtn$ corresponds to labeling each component of the unlink with $\Lambda^n$, 
the fundamental representation which folds to the spin representation in type $B$.

\begin{rem}
We began this project by searching for an extension of the involution on 
$H^\ast(\Gr{n}{2n})$ to all of $\cUU^{\le 2n}_q(\glm)$. 
We were bemused to find that it extended not to an involution, 
but an automorphism of order $4$. 
That its restriction to $\cXXtn_q(\glm)$ is an involution is crucial
to our constructions below. \end{rem}

The setting for our link invariant is the monoidal subcategory 
$\BnFoam \subset \cXXtn_q(\glm)$ of $1$-endomorphisms of the object $\wtn$, 
or, more precisely, the associated \emph{equivariant category} $\EBnFoam$ 
of this category with respect to the involution $\tau = \tau_n$.
Objects of $\EBnFoam$ are equivariant structures:
pairs $(X,\varphi_X)$, where $X$ is an object of $\BnFoam$ and $\varphi_X$ is an isomorphism
$\varphi_X \colon X \xrightarrow{\cong} \tau(X)$ in $\BnFoam$
satisfying $\tau(\varphi_X) \circ \varphi_X = \id_X$.
Morphisms 
$f \colon (X,\varphi_X) \to (Y,\varphi_Y)$ 
are morphisms 
$f \colon X \to Y$ in $\BnFoam$ satisfying 
$\varphi_Y \circ f = \tau(f) \circ \varphi_X$.

The equivariant category $\EBnFoam$ itself also admits a $\Z/2$-action, 
given by $\varphi_X \mapsto -\varphi_X$. 
In particular, the monoidal identity $\one = \one_{\wtn}$ of $\BnFoam$
gives rise to two equivariant objects $(\one,\id_\one)$ and $(\one, -\id_\one)$
which are exchanged by this $\Z/2$-action.

The Rickard complexes of \cite{CK,CKL,Cautis} 
(which generalize the Rouquier complex \cite{CR,Rou3}) 
associate a bounded complex 
$C(\bgen_{i_1}^\pm \cdots \bgen_{i_r}^\pm) \in \hCat[\BnFoam]$ 
to each word in the generators $\{\bgen_i\}_{i=1}^{m-1}$ 
of the $m$-strand braid group $\Br_m$.
We show that these complexes determine 
\emph{equivariant Rickard complexes}
$C^\tau(\bgen_{i_1}^\pm \cdots \bgen_{i_r}^\pm) \in \hCat[\EBnFoam]$. 
In the setting of $\hCat[\BnFoam]$, 
braid relations are categorified by canonical homotopy equivalences, 
and therefore the Rickard complexes canonically associate 
a complex $C(\bgen) \in \hCat[\BnFoam]$ to each $\bgen \in \Br_m$.
We show that the same is true for the equivariant Rickard complexes, 
and then apply appropriate representable functors to obtain invariants of links.

\begin{thm}[Theorems \ref{thm:Ctaubraid} and \ref{thm:typebLH}]
	\label{thm:SLHintro}
The equivariant Rickard complexes assign a complex 
	$C^\tau(\bgen) \in \hCat[\EBnFoam]$ to each $\bgen \in \Br_m$, 
	which is well-defined up to canonical homotopy equivalence.	
Let $\sVect_{\K}^\Z$ denote the category of $\Z$-graded super vector spaces 
and let $\RR \colon \EBnFoam \to \sVect_{\K}^\Z$
be the representable functor with even and odd components
\[
\RR(x)_{\bar{0}} = \Hom_{\EBnFoam}\!\big((\one,\id_{\one}),x\big) \, , \quad
	\RR(x)_{\bar{1}} = \Hom_{\EBnFoam}\!\big((\one, -\id_{\one}),x\big)  \, .
\]
Then, the homology $H_{\son}(\mathcal{L}_\bgen)$ 
of the complex $\qsh^{-mn^2}\RR(C^\tau(\bgen))$ is an invariant 
of the link $\mathcal{L}_\bgen \subset S^3$ arising as the braid closure of $\bgen$.
\qed
\end{thm}

\begin{rem} 
There is no difference between super vector spaces and $\Z/2$-graded vector spaces 
when only considered as monoidal categories.
Given this, one would not typically use super vector spaces unless the braiding on this braided monoidal category
(which does distinguish it from $\Z/2$-graded vector spaces) were relevant. 
While it is not relevant in this paper, we use the language of super vector spaces rather than
$\Z/2$-graded vector spaces for two reasons.
First, to avoid confusion, as many other gradings already appear in the paper. 
Second, because our results are conveniently stated in terms of super dimension 
(which, if one considers the relevant pivotal structure on super vector spaces, 
equals the dimension in this category).
\end{rem}

\subsection{Our construction, decategorified}
	\label{ss:introdecat}

For each $n \geq 1$, our construction produces a homological link invariant. 
Conditional upon the assumptions in Theorem \ref{thm:introtrace},
an appropriate normalization of 
the graded Euler characteristic of $H_{\son}(\mathcal{L}_\bgen)$ 
equals the spin-colored $\son$ link polynomial 
$\bP_{\son}(\mathcal{L}_\bgen^S) := P_{\son}(\mirror \mathcal{L}_\bgen^S)$ 
of the mirror link $\mirror \mathcal{L}$. 

\begin{thm}[Theorem \ref{thm:decat} and Remark \ref{rem:conditional}]
	\label{thm:introdecat}
Assume that $n=1,2,3$
or that Conjectures \ref{conj:R3}, \ref{conj:TrX}, and \ref{conj:TrX2} hold.
If $\beta \in \Br_m$ and $\epsilon(\bgen)$ denotes the braid exponent, 
then
\[
\bP_{\son}(\mathcal{L}_\beta^{\mathcal{S}}) 
	= (-1)^{n \epsilon(\bgen)+m\binom{n+1}{2}}
	q^{\frac{1}{2} n \epsilon(\bgen)} 
		\sum_{i} (-1)^i \dim_\qsh \big( H^{i}_{\son}(\mathcal{L}_\beta^\mathcal{S}) \big) \, .
\]
\qed
\end{thm}

This result and Theorem \ref{thm:introtrace} are essentially reformulations of one another.
Indeed, for an equivariant object $(X,\varphi_X)$ in $\EBnFoam$, 
the space of morphisms $\Hom_{\BnFoam}(\one,X)$ admits an involutory action of $\tau$, 
whose $+1$-eigenspace is $\Hom_{\EBnFoam}((\one,\id_{\one}),-)$ 
and whose $-1$-eigenspace is $\Hom_{\EBnFoam}((\one,-\id_{\one}),-)$. 
Meanwhile, the dimension of a super vector space $V = V_{\bar{0}} \oplus V_{\bar{1}}$ 
is $\dim(V) = \dim (V_{\bar{0}}) - \dim (V_{\bar{1}})$; 
if $V_{\bar{0}}$ and $V_{\bar{1}}$ arise as the $\pm 1$-eigenspaces of an involution $\tau$ on $V$,
this also equals $\tr(\tau)$.

From this perspective, Theorem \ref{thm:SLHintro} can be repackaged as the statement that 
$\tau$ descends to an involution on the complex $\Hom_{\BnFoam}(\one,C(\bgen))$
that computes $\Lambda^n$-colored $\sln[2n]$ Khovanov--Rozansky homology
(and is compatible with braid/Markov moves).
Theorem \ref{thm:introdecat} then implies Theorem \ref{thm:introtrace}.
We emphasize that the conjectural results on which those theorems rely 
for $n \geq 4$ are primarily related to open questions in the decategorified setting.
Regardless, for all $n \geq 1$ we still obtain the involution $\tau$, and hence
our link invariants from Theorem \ref{thm:SLHintro}.

\subsection{Intertwiners for the spin representation}
	\label{ss:introinter}

Theorem \ref{thm:introdecat} relies on new results 
that we establish concerning the endomorphism algebras 
of tensor powers $S^{\otimes m}$ of the spin representation
in $\Rep(U_q(\son))$. 
We believe these results to be of independent interest.
Although they do not depend on our categorified construction,
they are informed by it.

For the moment, we forget the involution $\tau$ and focus on the
(surjective) algebra homomorphism
\begin{equation} \label{eq:inwtnintro} 
1_{\wtn} \dU(\gln[2]) 1_{\wtn} \to \End_{U_q(\gln[2n])}(\Lambda^n \ot \Lambda^n) 
	\end{equation}
arising from the $m=2$ case of 
\eqref{eq:SHintro} restricted to endomorphisms of the object $\wtn=(n,n)$. 
For $k \geq 1$, 
let $\chi^{(k)} := f^{(k)} e^{(k)} 1_{\wtn} \in \dU(\gln[2])$ 
where $e^{(k)} := \frac{1}{[k]!}e^k$ and $f^{(k)} := \frac{1}{[k]!}f^k$ 
are the divided power elements.
It is straightforward to compute from the relations of $\dU(\gln[2])$ 
(found in Definition \ref{def:Udotgl}) that
\begin{equation} 
	\label{eq:xxintro}
\chi^{(k)} \cdot \chi = [k][k+1] \chi^{(k)} + [k+1]^2 \chi^{(k+1)} \, .
	\end{equation}
Consequently, we obtain endomorphisms (also) denoted by
$\chi^{(k)} \in \End_{U_q(\gln[2n])}(\Lambda^n \ot \Lambda^n)$ 
that again satisfy \eqref{eq:xxintro}. 
Moreover, \cite{CKM} describes explicitly how the braiding 
$R_{n,n} \in \End_{U_q(\gln[2n])}(\Lambda^n \ot \Lambda^n)$ 
can be expressed as a linear combination of the elements 
$\{ \chi^{(k)} \}_{0\leq k \leq n}$.

The left-hand side of \eqref{eq:inwtnintro} is categorified by 
$\one_{\wtn} \cUU_q(\gln[2]) \one_{\wtn}$, 
with indecomposable $1$-morphisms $\FEk := \FF^{(k)} \EE^{(k)} \one_{\wtn}$
corresponding to $\chi^{(k)}$ in the Grothendieck ring.
There is a direct sum decomposition
\begin{equation} 
	\label{eq:XXintro}
\FEk \ot \FE \cong [k][k+1] \FEk \oplus [k+1]^2 \FE^{(k+1)}
	\end{equation}
lifting \eqref{eq:xxintro}, 
where here multiplication by a quantum integer corresponds to an appropriate 
direct sum of grading shifts of $\FE^{(k)}$. 
For example, $[2] \FE = \qsh^1 \FE \oplus \qsh^{-1} \FE$.
Note that e.g.~$[k][k+1] \FEk$ can be viewed either as a direct sum of shifted copies of $\FEk$, 
or equivalently as a formal tensor product of $\FEk$
with a ``multiplicity space'' whose graded dimension equals $[k][k+1]$.

Now we reintroduce the involution $\tau$. 
As we show, the $\FEk$ admit equivariant structures,
and $\tau$ induces a $\Z/2$-action on the multiplicity spaces in \eqref{eq:XXintro}.
In Corollary \ref{cor:traceonV}, we compute the trace of this action.
Up to a sign, the traces of $\tau$ on the above multiplicity spaces of dimension $[k][k+1]$ 
and $[k+1]^2$ are given by Laurent polynomials that we denote 
$``[k][k+1]"$ and $``[k+1][k+1]"$. 
Paralleling \eqref{eq:unknots}, these latter quantities are obtained from their unfolded counterparts 
by replacing certain coefficients with integers congruent to them modulo $2$.
For example,
\[ 
[3][4] = [6]+[4]+[2] = q^5 + 2 q^3 + 3 q + 3 q^{-1} + 2 q^{-3} + q^{-5}
	\]
while
\[
``[3][4]" = [6]-[4]+[2] = q^5 + q + q^{-1} + q^{-5} \, .
	\]
Please allow us to introduce this ``incorrect'' arithmetic on quantum numbers, 
which we sympathetically call the \emph{devil's arithmetic}, 
in \S \ref{s:typeB} below.

Theorems \ref{thm:introtrace} and \ref{thm:introdecat} suggest the existence of elements 
$\xx^{(k)} \in \End_{U_q(\son)}(S \otimes S)$ analogous to the elements $\chi^{(k)}$. 
Working solely at the decategorified level, we establish the following.

\begin{thm}[Theorems \ref{thm:X} and \ref{T:iqWeyl=typeBbraid}]
	\label{thm:introbraiding}
There is a basis $\lbrace \xx^{(k)}\rbrace_{k=0}^{n} \subset \End_{U_q(\son)}(S\ot S)$
in which the braiding is given by
\begin{equation}
	\label{eq:introbraiding}
\RM_{S, S}= q^{\frac{n}{2}} \sum_{k=0}^n q^{-k}\xx^{(k)} \, .
\end{equation}
The structure coefficients for multiplication in this basis are determined by
\begin{equation}
	\label{eq:xxintro2}
\xx^{(k)} \xx^{(1)}  = (-1)^{k} ``[k][k+1]"\xx^{(k)} + (-1)^k``[k+1][k+1]"\xx^{(k+1)} \, .
	\end{equation}
Further, the elements 
$\xx_i^{(k)}:= \id_S^{\ot i-1} \ot \xx^{(k)} \ot \id_S^{\ot m-i-1}$
for $1 \leq i \leq m-1$ and $0\leq k \leq n$
generate $\End_{U_q(\son)}((S)^{\ot m})$.
\qed
\end{thm}

We propose the elements $\xx^{(k)}$ as a new \emph{canonical basis} 
for the endomorphism space $\End_{U_q(\son)}(S \ot S)$.
Indeed, they arise as a twisted decategorification of the 
indecomposable objects $\FEk = \FF^{(k)} \EE^{(k)} \one_{\wtn}$
which correspond to canonical basis elements in $\dU(\gln[2])$.

In more detail, 
there is a decategorification procedure for categories such as $\EBnFoam$, 
called the \emph{weighted Grothendieck group} $\tauKzero{-}$. 
We recall this construction (in the case of $\Z/2$-actions)
in \S \ref{subsec:equivariantization}.
Loosely stated, while the structure coefficients in the ordinary Grothendieck group 
are dimensions of multiplicity spaces, 
the structure coefficients of the weighted Grothendieck group 
are the super dimensions of those same multiplicity spaces. 
As we show,
an appropriate renormalization of the classes $[\FEk]_\tau \in \tauKzero{\EBnFoam[2]}$
satisfy equation \eqref{eq:xxintro2}.
Since all indecomposable $1$-morphisms in $\BnFoam[2]$ 
are of the form $\FEk$ for $0 \leq k \leq n$, 
standard results on weighted Grothendieck groups give the following.

\begin{thm}[Corollary \ref{cor:XmnDecat}]
	\label{thm:introKzero1}
For all $n \geq 1$,
there is an isomorphism of $\C(q)$-algebras
\[
\C(q) \ot_{\Z[q^\pm]} \tauKzero{\EBnFoam[2]} \xrightarrow{\cong} \End_{U_q(\son)}(S \ot S)
\]
that sends the class of (an appropriate equivariant structure on) $\FEk$ to $\xx^{(k)}$.
\qed
\end{thm}

Since this is the folded analogue of \eqref{eq:inwtnintro}, 
we refer to this result as a \emph{folded skew Howe duality}. 
More generally, we propose the following generalization.

\begin{conj}
	\label{conj:FSH}
For all $n \geq 1$ and $m \geq 2$, 
there is an isomorphism of $\C(q)$-algebras
\[
\C(q) \ot_{\Z[q^\pm]} \tauKzero{\EBnFoam} \xrightarrow{\cong} \End_{U_q(\son)}(S^{\ot m})
\]
that sends the class of (an appropriate equivariant structure on) 
$\FF_i^{(k)} \EE_i^{(k)} \one_\wtn$ to $\xx_i^{(k)}$.
\end{conj}

Further evidence for Conjecture \ref{conj:FSH} is as follows.
In Proposition \ref{prop:iSerre}, we prove that the elements 
$\xx_i$ and $\xx_{i \pm 1}$ satisfy a ``Reidemeister III''-like relation \eqref{eq:iSerre}
in $\End_{U_q(\son)}(S^{\ot m})$. 
In Theorem \ref{T:iserre-in-twisted-K0} and Corollary \ref{cor:XmnDecat}, 
we then prove that this relation is categorified by a direct sum decomposition in $\EBnFoam$. 
Consequently, Conjecture \ref{conj:FSH} would follow from the following two conjectures:
\begin{itemize}
\item That the ``Reidemeister III''-like relation \eqref{eq:iSerre}, 
	together with \eqref{eq:xxintro2}, 
	gives a presentation for the $\C(q)$-algebra
	$\End_{U_q(\son)}(S^{\ot m})$.
\item That the Grothendieck group $\tauKzero{\EBnFoam}$ is generated as a ring 
	by the classes of the $\FEk_i$, 
	and the dimensions of $\tauKzero{\EBnFoam}$ and $\End_{U_q(\son)}(S^{\ot m})$ are equal.
\end{itemize}
We discuss these conjectures further in the body of the paper.

\begin{rem}
Theorem \ref{thm:introKzero1} and Conjecture \ref{conj:FSH} can be viewed as a 
categorification of an instance of \cite[Theorem 1.1]{HongShen}. 
There, working in the setting of a connected almost simple algebraic group $G$ 
equipped with a Dynkin automorphism $\sigma$, 
Hong and Shen prove that the dimensions of invariant spaces for 
the algebraic group $G_\sigma$ associated to $(G,\sigma)$ via folding can 
be computed as the trace of the involution induced by $\sigma$ on invariant spaces for $G$.
In the $G = SL_{2n}$ case,
folded skew Howe duality asserts that the invariant space itself can be recovered via folding 
by considering involutions on the categories that categorify these invariant spaces.
\end{rem}

We emphasize that folded skew Howe duality is necessarily a product of categorification: 
it shows that the algebra $1_{\wtn} \dU(\glm) 1_{\wtn}$ has a ``secret'' relation 
with $\End_{U_q(\son)}(S^{\ot m})$ that is only visible by first considering the categorification 
$\one_{\wtn} \cUU_q(\glm) \one_{\wtn}$, 
then passing to the equivariant (subquotient) category $\EBnFoam$, 
and finally taking the weighted Grothendieck ring.
This procedure remains invisible at the decategorified level: 
there is no obvious method to modify the relations of \eqref{eq:XXintro} 
into those of \eqref{eq:xxintro2}. 
Said another way, if one only knows the ordinary dimension $\dim V_{\bar{0}} + \dim V_{\bar{1}}$ 
for a super vector space $V$, 
one can not deduce the value of the super dimension $\dim V_{\bar{0}} - \dim V_{\bar{1}}$.

Nevertheless, algebras related to $\tauKzero{\EBnFoam}$ 
\emph{have} previously appeared in the literature.

\subsection{Relationship to existing spin literature and further results}

The endomorphism algebra $\End_{U_q(\son)}(S^{\ot m})$ 
has previously been studied by Wenzl \cite{Wenzl-Spin} and Reshetikhin \cite{Resh-unpub1}, 
and also by Deligne \cite{Deligne-spin} and McNamara--Savage \cite{MS-spin} 
in the classical ($q=1$) setting.
All of these works (save for Wenzl's), 
use the diagrammatic language for monoidal categories 
to describe intertwiners between tensor products of $S$ 
and the vector representation of (quantum) $\son$.

In order to establish the results in \S \ref{ss:introinter}, 
we further develop the diagrammatic calculus for 
the spin representation of $U_q(\son)$.
Our advances in this direction are the expression \eqref{eq:introbraiding} 
for the braiding in terms of our canonical basis from Theorem \ref{thm:introbraiding}, 
the ``$H=I$'' relation given in Lemma \ref{lem:H1reln} below, 
and the interpretation of numerous structure coefficients 
as arising from the devil's arithmetic.

Reshetikhin notes that his relation \cite[Equation 5.29]{Resh-unpub1}, 
which we re\"{e}stablish below in Corollary \ref{C:Rstkn-Reln}, 
can be used to define a $q$-analogue of the Clifford algebra.
Later, in \cite[Equation 3.9]{Wenzl-Spin},
Wenzl constructs a distinguished endomorphism 
$C \in \End_{U_{q}(\son)}(S \otimes S)$ using these $q$-Clifford algebras. 
Via calculations with this element, Wenzl goes on to establish
a relation to an algebra defined twenty years 
earlier by Gavrilik--Klimyk\footnote{Their goal was to define a $q$-analogue of 
$\som$ which, unlike the usual Drinfeld--Jimbo quantum $\som$, 
was adapted to the chain of embeddings $\mathfrak{so}_{m-1}\subset \som$.}. 

\begin{defn}[{\cite{GK-Usom}}]
	\label{def:GK}
Let $U'_q(\som)$ be the $\C(q)$-algebra 
generated by $b_1, \ldots, b_{m-1}$ subject to the relations
\[
b_ib_j = b_jb_i \quad |i-j|>1\, , 
\]
and
\[
b_i^2b_{i\pm1} + b_{i\pm1}b_i^2 = [2]b_ib_{i\pm1}b_i + b_{i\pm1} \, .
\]
\end{defn} 

Wenzl's calculations in \cite[Section 4]{Wenzl-Spin} and \cite[Theorem 5.2]{Wenzl-Spin} 
give the following.

\begin{thm}[Wenzl]
	\label{thm:Wenzl}
There is a surjective algebra homomorphism
\begin{equation}
	\label{eq:Wenzlmap}
U'_{-q^2}(\som) \to \End_{U_q(\son)}(S^{\otimes m})
\end{equation}
such that 
$b_i\mapsto \id_S^{\otimes i-1}\otimes C \otimes \id_S^{\otimes m-i-1}$.
\end{thm}

\begin{rem}
	\label{rem:nonclassical}
A consequence of Wenzl's theorem is that the vector space $S^{\otimes m}$ 
decomposes into a direct sum of irreducible representations of $U'_{-q^2}(\som)$. 
However, the irreducible $U'_{-q^2}(\som)$-representations that appear are not 
$q$- (or $-q^2$)-analogues of irreducible representations of $\som$. 
Wenzl overcomes this issue by studying ``non-classical" representations of 
$U'_{-q^2}(\som)$ first introduced in \cite{KlimykReps}; 
see \cite[Theorem 2.1]{Wenzl-Spin}. 
\end{rem}

In Appendix \ref{S:wenzlapproach}, 
we explicitly relate our approach to Wenzl's endomorphism $C$. 
Proposition \ref{P:CequalsH1} implies that
\begin{equation}
	\label{eq:CvsX}
\xx = C - \frac{1}{[2]} \id_{S \otimes S} \, .
\end{equation}
The most novel feature of our work in this direction is the 
simple formula \eqref{eq:introbraiding} for the braiding
after passing from Wenzl's $C$ to our $\xx$.
To our knowledge, such a formula has not previously 
appeared in the literature on the quantum spin representation. 
Further, our element $\xx$ also (conjecturally) sheds light on the kernel of \eqref{eq:Wenzlmap};
see Conjecture \ref{conj:itsthekernel}.

\begin{rem} 
With the computations of Lemma \ref{L:H-satisfies-i-quantum-gp-Serre}, 
we can reprove the existence of the homomorphism \eqref{eq:Wenzlmap}. 
Note that we do not reprove the surjectivity of \eqref{eq:Wenzlmap}, 
but rather use Wenzl's result to establish surjectivity in Theorem \ref{thm:introbraiding}. \end{rem}

\subsection{Categorifying $U'_{-q^2}(\som)$}
	\label{ss:catGK}

Taken together, our results thus far suggest an approach to
the categorification of the algebra $U'_{-q^2}(\som)$ itself. 
Inspired by \eqref{eq:CvsX} we make the following definition.

\begin{defn}
	\label{def:idpb}
Inside $U'_{-q^2}(\som[2])$, let $b = b_1$. For $k \geq 0$, define elements 
 $\{\nx^{(k)}\}_{k \geq 0} \subset U'_{-q^2}(\som[2])$
recursively as follows:
\begin{equation}
	\label{eq:idpb}
\begin{gathered}
\nx^{(0)} := 1 \, , \quad 
\nx^{(1)} = \nx := b-\frac{1}{[2]} \, , \\
\nx^{(k)} \nx = (-1)^{k}``[k+1][k]" \nx^{(k)} + (-1)^{k}``[k+1] [k+1]" \nx^{(k+1)} \, .
\end{gathered}
\end{equation}
Set $\nx_i := b_i - \frac{1}{[2]}$ in $U'_{-q^2}(\som)$ and 
define $\nx_i^{(k)}$ analogously. 
\end{defn}

It is easy to deduce that $\{\nx^{(k)}\}_{k \geq 0}$ is a basis for $U'_{-q^2}(\som[2])$. 
Paralleling Theorem \ref{thm:introKzero1}, 
we obtain a categorification of $U'_{-q^2}(\som[2])$ in which indecomposables 
correspond to this basis.

\begin{thm}[Corollary \ref{cor:K0(X2)}]
	\label{thm:introKzero2}
For all $n \geq 1$, consider the involution $\tau_n$ of 
$\one_{\wtn} \cUU_q(\gln[2]) \one_{\wtn}$ given by Theorem \ref{thm:invintro}.
There is an isomorphism of $\C(q)$-algebras
\[
U'_{-q^2}(\som[2]) \xrightarrow{\cong}  
\C(q) \ot_{\Z[q^\pm]} K_0^{\tau_n}((\one_{\wtn} \cUU_q(\gln[2]) \one_{\wtn})^{\tau_n}) \, .
\]
that is intertwined with the isomorphism in Theorem \ref{thm:introKzero1} 
by Wenzl's homomorphism \eqref{eq:Wenzlmap}.
In other words, it sends $\nx^{(k)}$ to the class of (an appropriate equivariant structure on) 
$\FEk = \FF^{(k)} \EE^{(k)} \one_{\wtn}$. \qed
\end{thm}

We direct the reader distressed by the appearance of seemingly many 
(one for each $n \geq 1$) categorifications of $U'_{-q^2}(\som[2])$ to \S \ref{ss:dependence-on-n}. 
The $m=2$ case of the results there give equivalences between the 
$\one_{\wtn} \cUU_q(\gln[2]) \one_{\wtn}$ (for various $n$) 
that intertwine the involutions $\tau_n$.

Continuing the parallel, we expect that Theorem \ref{thm:introKzero2} generalizes from $m=2$ 
to all $m \geq 2$. 
Recall that $\cXX_q(\glm)$ is the full subcategory of
$\cUU_q(\glm)$ generated by the $1$-endomorphisms 
$\FF_i^{(k)}\EE_i^{(k)}$ and $\EE_i^{(k)}\FF_i^{(k)}$. 
We let $\BFoam := \one_{\wtn} \cXX_q(\glm) \one_{\wtn}$, suppressing $n$ from
the notation (all these categories as $n$ varies are identified in \S \ref{ss:dependence-on-n}). 

\begin{conj}
	\label{conj:introKzero2}
There is an isomorphism of $\C(q)$-algebras
\[
U'_{-q^2}(\som) \xrightarrow{\cong}  
\C(q) \ot_{\Z[q^\pm]} \tauKzero{\EBFoam}\, .
\]
that sends $\nx_i^{(k)}$ to the class of (an appropriate equivariant structure on) 
$\FEk_i = \FF_i^{(k)} \EE_i^{(k)} \one_{\wtn}$.
\end{conj}

We again discuss some evidence for this conjecture. 
The relations of Definition \ref{def:GK} can be transformed into the following statement.

\begin{prop}
	\label{prop:iSerreintro}
The elements $\nx_i^{(k)} \in U'_{-q^2}(\som)$ satisfy
\begin{equation}
	\label{eq:iSerreintro}
\begin{gathered}
\nx_i \nx_j = \nx_j \nx_i  \quad |i-j|>1 \, , \\
\nx_{i} \nx_{i\pm1} \nx_{i} 
	= \nx_{i}^{(2)} \nx_{i\pm1} + \nx_{i\pm1} \nx_{i}^{(2)} + [2] \nx_{i}^{(2)} + \nx_{i} \, .
\end{gathered}
\end{equation}
These relations (together with the definition of $\nx_i^{(2)}$) 
give a presentation for $U'_{-q^2}(\som)$.
\qed
\end{prop}

In this context, we think of \eqref{eq:iSerreintro} 
as a variant on the usual quantum group Serre relations. 
We refer to it as the \emph{devil's Serre relation}.
Replacing $\nx_i$ with $\xx_i \in \End_{U_q(\son)}(S^{\otimes m})$ transforms
the devil's Serre relation into the ``Reidemeister III''-like 
relation \eqref{eq:iSerre} discussed above, 
so Conjecture \ref{conj:introKzero2} is again intertwined with 
Conjecture \ref{conj:FSH} by Wenzl's theorem.
As additional evidence for Conjecture \ref{conj:introKzero2}, 
Theorem \ref{T:iserre-in-twisted-K0} and Corollary \ref{cor:K0(Xm)}
provide a lift of the relations \eqref{eq:iSerreintro} 
to isomorphisms in the $\tau_n$-equivariant category of $\BFoam$
and establish the existence of the $\C(q)$-algebra homomorphism appearing 
in Conjecture \ref{conj:introKzero2}.

We make one further statement about the decategorified setting, 
inspired by categorification. 
Recall Wenzl's surjective homomorphism
\[ U'_{-q^2}(\som) \to \End_{U_q(\son)}(S^{\otimes m}). \]
We propose above that the left-hand side is categorified (in the equivariant sense) by $\BFoam$, 
and the right-hand side by $\EBnFoam$. 
The relationship between $\BFoam$ and $\EBnFoam$ is the projection to $\one_{\wtn}$ 
of the surjection
\[ \cUU_q(\glm) \to \cUU_q^{\leq 2n}(\glm) \, , \]
which categorifies the surjection from $\dU(\glm)$ to a particular Schur algebra. 
The object $\FEk_i$ is sent to the zero object in this quotient, 
for any $k > n$. 

Analogously, Wenzl's homomorphism satisfies
\[
\nx_i^{(k)}
\mapsto 
\begin{cases}
\xx_i^{(k)} & 1 \leq k \leq n \, , \\
0 & k > n \, ,
\end{cases}
\]
by e.g.~Remark \ref{rem:Xtoohigh} below. 
Denote by $I^{> n}\subset U_{-q^2}'(\som)$ the two sided ideal 
generated by the elements $\nx_i^{(n+1)}$ for $i=1,\ldots, m-1$. 
We propose that the ideal $I^{>n}$ is exactly the kernel of Wenzl's homomorphism.

\begin{conj} \label{conj:itsthekernel}
The algebra homomorphism in Theorem \ref{thm:Wenzl} induces an isomorphism
\[
U'_{-q^2}(\som)^{\le n} := U'_{-q^2}(\som)/I^{> n} 
\xrightarrow{\cong} \End_{U_q(\son)}(S^{\otimes m}) \, .
\]
\end{conj}

It is straightforward to verify this conjecture when $n=1$, 
by comparing the algebras with the Temperley--Lieb algebra.
One can also verify the case $n=2$ by direct calculation, 
using type $B_2 = C_2$ webs \cite{Kup,BERT}. 
For example, compare \eqref{eq:iSerreintro} to \cite[equation (5.56b)]{BERT}.

\begin{rem}
	\label{rem:itsthekernelsketch}
A consequence of the conjecture is that 
the algebra $U'_{-q^2}(\som)^{\le n}$ is a finite dimensional algebra. 
Assume that we knew that $U_{-q^2}'(\som)^{\le n}$ is finite dimensional. 
Since finite dimensional $U_{-q^2}'(\som)$-modules are completely reducible \cite{KlimykReps}, 
it would follow that $U_{-q^2}'(\som)^{\le n}$ is a finite direct sum of algebras $\End(L)$, 
where $L$ is a simple module for $U_{-q^2}'(\som)$ that is annihilated by $I^{>n}$.
This is the folded skew Howe duality analogue of \cite[Lemma 4.4.2]{CKM}. 
We expect that one can use \cite[Theorem 5.3(a)]{Wenzl-Spin} to show that 
each such $L$ is isomorphic to a direct summand of the $U_{-q^2}'(\som)$-module $S^{\otimes m}$, 
and that this exhausts all (isomorphism classes) of irreducible summands appearing in $S^{\otimes m}$. 
Then, Conjecture \ref{conj:itsthekernel} would follow by mimicking 
the proof of \cite[Theorem 4.4.1]{CKM}. 
\end{rem}

The takeaway is that the proof sketch in Remark \ref{rem:itsthekernelsketch} 
would prove Conjecture \ref{conj:itsthekernel} once one can show that 
$U_{-q^2}(\som)^{\le n}$ is finite dimensional.
The resolution of our Conjecture \ref{conj:R3} would imply finite-dimensionality
(so Proposition \ref{prop:R3} shows that we already 
know finite-dimensionality when $n=3$).

\subsection{Relationship to $\iota$quantum groups}
	\label{ss:iquantum}
	
The algebra $U'_q(\som)$ appearing in Theorem \ref{thm:Wenzl} 
is a special case of a so-called $\iota$quantum group.
The latter are algebras that arise in the theory of quantum symmetric pairs.
We refer the reader to Wang's 2022 ICM address \cite{WangICM} 
for comprehensive details and references to the vast body of work on this subject, 
and recall only the immediately pertinent information here.

Quantum symmetric pairs are parametrized 
by Satake diagrams \cite[Section 1.2]{WangICM}, 
which determine a semisimple Lie algebra $\mathfrak{g}$ 
and an involution $\theta$ of $\mathfrak{g}$. 
The pair $(\mathfrak{g}, \mathfrak{g}^{\theta})$ 
consisting of the Lie algebra and its fixed-point subalgebra 
is a (classical) symmetric pair. 
Analogously, a \emph{quantum symmetric pair} is a tuple
$(U_q(\mathfrak{g}), U^{\iota}_q(\mathfrak{g}^{\theta}))$
where $U_q^{\iota}(\mathfrak{g}^{\theta})$ is a 
(coideal) subalgebra of 
$U_q(\mathfrak{g})$, referred to as an $\iota$quantum group. 

\begin{rem}
In the quasi-split case\footnote{This means there are no 
``black dots'' on the Satake diagram.}, 
the Satake diagram is determined by a pair $(D, t)$, 
where $D$ is the Dynkin diagram associated to $\mathfrak{g}$, 
and $t$ is a Dynkin diagram involution\footnote{In \cite{WangICM} 
this Dynkin involution is denoted by $\tau$.},
see e.g.~\cite[Example 1.1]{WangICM}. 
A quasi-split pair is called split if $t = \id$
and, in this case, $\theta$ is the 
Chevalley involution which swaps generators $e_i$ and $f_i$, 
and acts by $-1$ on the Cartan subalgebra $\mathfrak{h}$.
\end{rem}

The algebra $U'_q(\som)$ can be identified with the $\iota$quantum group 
associated to the split symmetric pair $(\slm, \som)$; see~\cite[Remark 2.4]{Letzter}. 
In particular, there is an embedding $U'_q(\som) \to U_q(\mathfrak{sl}_m)$ given by 
$b_i\mapsto f_i+q^{-1}e_ik_i^{-1}$, which induces the isomorphism 
$U'_q(\som) \xrightarrow{\cong} U_q^{\iota}(\som)$. 
Henceforth, we identify $U_q'(\som)$ with $U_q^{\iota}(\som)$.

We now contrast our results with the future goals and recent results 
from the $\iota$quantum group literature. 
The punchline is that our constructions do \textbf{not} agree.
Hence, we are hopeful that the constructions in the present paper will 
shed new light on the theory of $\iota$quantum groups.

Using based modules for $U_q(\mathfrak{g})$ 
and the quasi $K$-matrix \cite[Theorem 3.1]{WangICM}, 
it is possible to construct the so-called $\iota$canonical basis 
in the restriction of a finite dimensional based $U_q(\mathfrak{g})$-module 
to $U^{\iota}_q(\mathfrak{g}^{\theta})$. 
Then, in analogy with Lusztig's construction of canonical bases for quantum groups
using tensor products of based modules, 
Bao and Wang \cite{BW-icanonical}, 
defined the $\iota$canonical basis in 
the modified (i.e.~idempotented) form of  
the $\iota$quantum group $U^{\iota}_q(\mathfrak{g}^{\theta})$. 

\begin{example}\label{Ex:A1-divided-powers}
In the special case of the split quantum symmetric pair 
corresponding to $(A_{1}, \id)$, 
the $\iota$canonical basis is computed explicitly in \cite[Example 3.3]{WangICM}. 
Here, there are two modified versions 
$\dU^{\iota}(\mathfrak{so}_2) 1_{\bar{0}}$ and $\dU^{\iota}(\mathfrak{so}_2) 1_{\bar{1}}$
of the $\iota$quantum group $U_q^{\iota}(\mathfrak{so}_2)$, 
where the parity corresponds to the parity of weights in irreducible 
$\sln[2]$ representations.

Denoting the generator of $\dU^{\iota}(\mathfrak{so}_2) 1_{\epsilon}$ 
by $b_{\epsilon}$, 
the $\iota$canonical basis consists of the $\iota$divided powers, 
which are given by
\begin{equation}\label{E:i-divided-powers}
b_{\epsilon}^{(0)} = 1 \, , \quad 
b_{\epsilon}^{(1)} = b_{\epsilon} \, , \quad \text{and} \quad 
b_{\epsilon}^{(k)}b_{\epsilon} = 
	[k+1]b_{\epsilon}^{(k+1)} + \delta_{\overline{k},\epsilon}[k]b_{\epsilon}^{(k-1)}
\end{equation}
where $\overline{k} = k \ \ourmod 2$.
\end{example}

From this example, one can see that the basis $b_{\epsilon}^{(k)}$ for $U^{\iota}_{-q^2}(\som[2])$ 
does \textbf{not} agree with our basis $\nx^{(k)}$ (nor is our basis dependent on parity). 
We henceforth refer to our elements $\{\nx^{(k)}\}$ 
as the \emph{devil's divided powers}.

\begin{rem}\label{R:Bao-Wang-approach}
Bao and Wang's $\iota$canonical basis for $U^{\iota}_q(\mathfrak{g}^{\theta})$ 
is characterized as the unique basis which is ``asymptotically compatible" 
with the $\iota$canonical basis of tensor products 
of highest weight $U_q(\mathfrak{g})$ modules \cite[Theorem G]{BW-icanonical}. 
The $\iota$canonical basis of a based $U_q(\mathfrak{g})$ module 
is defined using the quasi $K$-matrix \cite[Theorem E]{BW-icanonical},
which is an element of the larger ambient quantum group 
$U_q(\mathfrak{g})$ \cite[Proposition C]{BW-icanonical}.

It would be very interesting to give an algebraic description 
of the devil's divided power basis obtained from our approach 
to the categorification of the $\iota$quantum group $\iqUsom$.
One might imagine defining an $\iota$canonical basis for 
non-classical representations from Remark \ref{rem:nonclassical}, 
then describing the $\iota$canonical basis of $\iqUsom$ 
as the elements which are compatible with the $\iota$canonical basis 
in the non-classical representations. 
However, since non-classical representations are not restricted from $U_{-q^2}(\mathfrak{sl}_m)$,
it is unclear how to mimic Bao and Wang's construction. 
\end{rem}

Although Bao--Wang's theory of $\iota$canonical bases 
is developed for all Satake diagrams, 
until recently
there was only one quantum symmetric pair admitting a categorification: 
the quasi-split quantum symmetric pair corresponding to $(A_{2n}, t)$, 
where $t$ is the non-trivial Dynkin diagram automorphism \cite{BSWW}. 
In very recent work \cite{brundan2023nilbrauercats, brundan2023nilbrauer}, 
Brundan, Wang, and Webster 
found an approach to categorifying the based algebras 
in Example \ref{Ex:A1-divided-powers} using the nilBrauer category.
Evidently, 
our equivariant categorification of the $(A_1,\id)$ case is not 
na\"{i}vely related to the nilBrauer category, 
as they categorify different bases of the same algebra. 
The resolution to the problem posed in Remark \ref{R:Bao-Wang-approach} 
may clarify whether there is a categorical connection with the nilBrauer category. 

In another direction, 
recall that the Drinfeld--Jimbo quantum groups possess 
quantum Weyl group elements \cite[Chapter 8]{JantzenQgps}
that generate an action of the corresponding braid group 
on finite-dimensional representations.
Similarly, is is expected that there exists an $\iota$quantum Weyl group 
corresponding to the relative Weyl group of the associated Satake diagram. 
There has been extensive progress defining 
the $\iota$quantum analogue of the braid group action on 
$\iota$quantum groups themselves \cite[Section 7]{WangICM}. 
Further, there is a proposal for the $\iota$quantum Weyl group elements 
in $U_q^{\iota}(\mathfrak{so}_2)$ given in \cite[Section 16.3]{zhang-thesis}. 
These elements would act on modules for the $\iota$quantum group 
so that the $\iota$quantum braid group action on the $\iota$quantum group 
is ``conjugation" by the $\iota$quantum Weyl group elements. 

In our setting, the decategorification of the equivariant Rickard complexes 
provide natural candidates for the $\iota$quantum Weyl group elements. 
We propose the following.

\begin{defn}
The \emph{devil's quantum Weyl group generators} in $\iqUsom$
are the formal expressions
\[
\iQW_i := \sum_{k\ge 0} q^{-k} \nx_i^{(k)}
\]
for $1 \leq i \leq m-1$.
\end{defn}

As a consequence of Theorem \ref{thm:introbraiding}, 
the elements $\iQW_i$ act invertibly in any representation of $\iqUsom$ 
that contains $I^{> N}$ in its kernel for some $N > 0$. 
One can combine Theorem \ref{thm:introbraiding}, equation \eqref{eq:CvsX},
and Wenzl's Theorem \ref{thm:Wenzl} to yield the following.

\begin{thm}[Theorem \ref{T:iqWeyl=typeBbraid}]
	\label{T:enhanced-Wenzl}
The surjective algebra homomorphism in Theorem \ref{thm:Wenzl} is such that 
\[
\iQW_i \mapsto 
	q^{-n/2} \id_{S}^{\otimes (i-1)} \otimes R_{S,S} \otimes \id_{S}^{\otimes (m-i-1)} \, .
\]
\end{thm}

\begin{rem}
One can view Theorem \ref{T:enhanced-Wenzl} as 
a type $B$ analogue of a well-known family of results in type $A$ 
that relate the quantum Weyl group elements with the $R$-matrix via Howe duality; 
see \cite[Theorem 6.5]{Valerio-qWeyl} and \cite[Corollary 6.2.3]{CKM}.
\end{rem}

We now outline the remainder of the paper.
First, we point out that the majority of our results
(e.g.~all those explicitly stated thus far)
are contained in:
\begin{itemize}
\item Section \ref{s:typeB}, where we study quantum $\son$ representation theory,
\item Section \ref{s:inv}, where we introduce an involution on a $2$-subcategory of 
	categorified quantum $\glm$,
\item Section \ref{s:ECQG}, where we study the equivariant category for this involution, and
\item Section \ref{s:SLH}, where we define and study our link homology.
\end{itemize}
Each of these sections is immediately preceded by a section developing 
pertinent background 
(on quantum groups in \S \ref{s:BQG}, 
on their categorifications in \S \ref{s:BCQG},
on decompositions/equivariant categories in \S \ref{sec:equiv}, 
and on type $A$ link homology in \S \ref{s:LH}).
The reader comfortable with this background may focus on the even-numbered 
sections and backfill the background material as needed.
(That said, some of our treatment of the background is novel, 
and these sections do contain some new concepts/results.)
We begin in \S \ref{s:Conv} with our categorical conventions.

\subsection*{Acknowledgements}
We thank
William Ballinger,
Jon Brundan,
Lily Gergle,
Jiuzu Hong,
Peter McNamara,
Alistair Savage,
and
Haihan Wu
for helpful discussions.
We also thank our once (and hopefully future) 
collaborator Logan Tatham, 
as aspects of this work
grew out of our previous collaboration with him.

\subsection*{Funding}

E.B. was partially supported by 
the NSF MSPRF-2202897
and the University of Oregon's Lokey Doctoral Fellowship.
B.E. was partially supported by
NSF grants DMS-2201387 and DMS-2039316.
D.R. was partially supported by 
NSF CAREER grant DMS-2144463 
and Simons Collaboration Grant 523992.

\section{Conventions}
	\label{s:Conv}

\subsection{The base ring $\K$}
	\label{ss:K}

In this paper we use constructions at two different categorical levels. 
Downstairs, 
we have structures at the categorical level of traditional representation theory,
e.g.~the quantum group $U_q(\mathfrak{g})$, 
its category of finite-dimensional representations $\Rep(U_q(\mathfrak{g}))$, 
and algebras of intertwiners (endomorphism spaces in $\Rep(U_q(\mathfrak{g}))$).
Upstairs, 
we have structures arising in categorical representations theory, 
e.g.~the categorified quantum group $\UU_q(\glm)$ and related subcategories and quotients.

Downstairs\footnote{All the algebras downstairs can be defined over more general rings. 
The resulting algebras have interesting modular representation theory, 
as well as connection to $3$-manifold invariants. 
However, the typical approach to categorification
(realizing the algebra as the Grothendieck group of an additive/triangulated category) 
means that downstairs the algebras are always in characteristic zero and generic $q$.}, 
our choice of scalars will not be particularly significant, provided the variable $q$ is generic. 
Typically, we will use $\C(q)$ or at times $\C(q^{\frac{1}{d}})$ for an appropriate integer $d$;
the essential feature here is that some constructions 
in the representation theory of $U_q(\mathfrak{g})$ 
require the invertibility of certain quantum integers 
and the existence of certain fractional powers of $q$. 
(E.g.~in type $B$, we require $[2]^{-1}$ and $q^{\frac{1}{2}}$.)

On the other hand, 
there are more subtleties for our choice of base ring upstairs, 
which we henceforth always denote by $\K$.
The categorified quantum group $\UU_q(\glm)$ will be a $\K$-linear category 
and our categorified link invariants will be $\K$-modules.
The reader who so desires may set $\K = \C$ and ignore any further technical discussion. 
However, there is general interest both in developing structures in categorical representation theory 
and in defining link homologies integrally (in the present setting, $\Z[\frac{1}{2}]$ is more appropriate) 
or over other fields. 
This leads to a number of technicalities, 
due both to subtleties when dealing with equivariant categories and 
to the choice of $\K$ in the existing literature.

We now mention some of these technicalities.
The category $\UU_q(\glm)$, which plays a central role for us upstairs, 
can be defined for any commutative ring $\K$.
However, some common assertions about $\UU_q(\glm)$ 
(e.g.~that certain objects are indecomposable) 
assume that $\K$ is an integral domain or is local.
Additionally, most literature concerning the Grothendieck group of $\UU_q(\glm)$ 
assumes that $\K$ is a field (however, note the exception \cite{KLMS}).  
Further, when dealing with equivariant categories associated with involutions, 
it will be essential that $2$ is invertible in $\K$.
Finally, the results in the literature that construct braid group actions in the 
context of categorical representation theory are generally only proved under the assumption 
that $\K$ is a field (although it is folklore that they hold integrally).

Although the goal at the outset of this project was to define the titular spin link homology, 
along the way we established results in categorical representation theory that we 
believe are of independent interest. 
Given that our results speak to a number of different audiences, 
we have attempted to work (and to set up for future work) 
over as general of a commutative ring $\K$ as possible.
We have attempted (and hopefully succeeded!) to clearly indicate what 
is being assumed about $\K$ in each section of the paper, 
and we often re\"{e}mphasize what is being assumed of $\K$ 
in most of our ``major'' results.
In particular, much of our work concerning the categorified quantum group $\UU_q(\glm)$ 
is done over a general integral domain,
as are our decategorification results when $m=2$. 
When considering equivariant categories associated with involutions, 
we will additionally impose the condition that $2$ is invertible 
(in order to diagonalize $\Z/2$-representations).
Finally, we will assume that $\K$ is a field (in which $2 \neq 0$)
for certain decategorification results/conjectures,
when dealing with the categorified braiding, 
and ultimately when defining our categorified link invariant.

To the last point, 
although we work over a field when defining spin link homology, 
the construction can in principle be carried out over the ring $\K = \Z[\frac{1}{2}]$.
This would require:
\begin{enumerate}
	\item A proof that the (bounded) Rickard complexes appearing 
		in Definition \ref{def:Rickard} are invertible over $\Z$,
	\item a proof that the Rickard complexes braid over $\Z$, and
	\item a subsequent adaptation of the proofs of 
		Propositions \ref{prop:SLHM1}, \ref{prop:SLHM2+}, and \ref{prop:SLHM2-}
		to the integral setting.
\end{enumerate}
We note that the first two items above are generally accepted folklore in the link homology community, 
but (to our knowledge) a proof\footnote{We encourage someone to fill this gap in the literature.} 
has not appeared in the literature. 

Meanwhile, we can say the most about the Grothendieck group of various ($2$-)categories considered
under additional assumptions on $\K$.
Assuming $\K$ is a field of characteristic zero, 
we can use results of Webster \cite{Web4,WebSchur} which state that $\UU_q(\glm)$ is mixed. 
Further, when $\K$ is algebraically closed, 
we can classify the indecomposable objects in the equivariant category of a mixed category.
It is under these assumptions that we can say the most about weighted Grothendieck groups, 
thus we expect the decategorification conjectures (for $m > 2$) 
in Sections \ref{ss:introinter} and \ref{ss:catGK} to be most-easily accessible under this assumption.
Nevertheless, we are able to establish some decategorification results over more general $\K$.

\subsection{Categorical conventions}
	\label{ss:conv}

We now record our conventions for various categorical constructions 
which take place ``upstairs,'' in the terminology of \S \ref{ss:K}. 
We state these results for categories, but remark that all constructions in this section extend to 
$2$-categories by applying them to the constituent $\Hom$-categories.
For us, ``$2$-category'' will always mean weak $2$-category 
(also called bicategories in the literature),
although some that appear will be strict.
As mentioned above, the base ring for these constructions is $\K$.

\subsubsection{(Graded) linear categories}

A $\K$-linear category will mean a category enriched in $\K$-modules;
such a category is additive if and only if it admits all finite 
coproducts, which are necessarily biproducts. 
We refer to biproducts as \emph{direct sums}. 
We can always pass from a $\K$-linear category to an additive $\K$-linear category 
by formally adjoining finite direct sums. 
Even in a $\K$-linear category which is not additive we can discuss direct sums, 
which may or may not exist.

Analogously, a $\Z$-graded $\K$-linear category is a category 
enriched in the category of $\Z$-graded $\K$-modules 
(which itself has morphisms the \emph{degree-zero} linear maps).
Throughout, we will denote various grading shift functors 
(that are not of a homological nature) by powers of $\qsh$.
For example, given a $\Z$-graded $\K$-module $V$ and $k \in \Z$, 
$\qsh^k V$ denotes the $\Z$-graded $\K$-module which, 
in degree $m$, agrees with $V$ in degree $m - k$.
Given a finitely generated free graded $\K$-module
(e.g.~a finite-dimensional graded vector space), 
we let $\qdim(V) \in \Zge[q^{\pm}]$ denote its graded dimension.
Our conventions therefore imply that $\qdim(\qsh^k V) = q^k \qdim(V)$.

The following standard construction allows to pass from a $\Z$-graded $\K$-linear category 
to an additive $\Z$-graded $\K$-linear category that is equipped with a grading shift autoequivalence.

\begin{defn}
	\label{def:Zadd}
Let $\Cat$ be a $\Z$-graded $\K$-linear category. 
The \emph{$\Z$-additive closure} of $\Cat$ is the category whose objects are formal 
expressions $\bigoplus_{i \in I} \qsh^{k_i} X_i$, where $I$ is a finite set, 
$k_i \in \Z$, and $X_i$ are objects of $\Cat$. 
Morphisms are given by matrices:
\[
\Hom \Big(\bigoplus_{i \in I} \qsh^{k_i} X_i , \bigoplus_{j \in J} \qsh^{\ell_j} Y_j \Big)
:= \bigoplus_{i \in I, j \in J} \qsh^{\ell_j - k_i} \Hom(X_i,Y_j) \, .
\]
\end{defn}

Next, for an additive category $\Cat$
we let $\Kzero{\Cat}$ denote its (split) Grothendieck group.
This is the quotient of the free abelian group 
generated by isomorphism classes $[X]$ of objects $X \in \Cat$ 
by the relation $[X \oplus Y] = [X] + [Y]$.
An additive category is \emph{Krull--Schmidt} if each object can be 
decomposed as a finite sum of objects with local endomorphism rings
(graded local, in the event that $\Cat$ is graded). 
In such a category, 
$\Kzero{\Cat}$ is free abelian with basis the classes of non-isomorphic 
indecomposable objects.
If $\Cat$ is $\Z$-graded and is equipped with a grading shift autoequivalence 
(e.g.~if we're in the setting of Definition \ref{def:Zadd}), 
we can endow $\Kzero{\Cat}$ with the structure of a $\Z[q^{\pm}]$-module 
via the relation $q [X] := [\qsh X]$.

A pre-additive category $\Cat$ is \emph{Karoubian} if every idempotent 
endomorphism in $\Cat$ splits.
Any Krull--Schmidt category is Karoubian, 
and in categories enriched in (graded) finitely dimensional $\K$-vector spaces (over a field $\K$), 
Karoubian implies Krull--Schmidt; see e.g.~\cite[Theorem 11.53]{soergelbook}.
If $\Cat$ is not necessarily Karoubian, 
we can pass to its \emph{Karoubi envelope} $\Kar(\Cat)$ 
wherein objects are idempotent endomorphisms. 
The category $\Kar(\Cat)$ is always Karoubian, 
and is equivalent to $\Cat$ in the event that $\Cat$ is itself Karoubian.

\subsubsection{Categories of complexes}

Next, we establish our conventions for homological algebra.
We emphasize at the outset that all complexes considered in this paper will be bounded.

Given a $\K$-linear category $\Cat$, we let $\Cat[\tsh^\pm]$ denote the category of 
(cohomologically indexed) finite sequences in $\Cat$. 
Explicitly, objects are sequences $X = (X_i)_{i \in \Z}$ with $X_i = 0$ for all but finitely many $i \in \Z$, 
and morphisms are given by
\[
\Hom_{\Cat[\tsh^\pm]}(X,Y) := \bigoplus_{k \in \Z} \Hom^k_{\Cat[\tsh^\pm]}(X,Y)
\, , \quad 
\Hom^k_{\Cat[\tsh^\pm]}(X,Y) := \prod_{i \in \Z} \Hom^k_{\Cat}(X_i,Y_{i+k}) \, .
\]
This category is enriched in $\Z$-graded $\K$-modules.
By slight abuse of notation, we will denote objects in $\Cat[\tsh^\pm]$ by $\bigoplus_{i \in \Z} \tsh^i X_i$. 
Here, in contrast to Definition \ref{def:Zadd}, we view the grading (shift) $\tsh$ as having a homological nature.

\begin{defn}\label{D:cats-of-complexes}
Let $\Cat$ be a $\ak$-linear category.
A bounded \emph{chain complex} over $\Cat$ is a pair $(X, d_X)$ with $X \in \Cat[\tsh^\pm]$ 
and $d_X \in \End^1_{\Cat[\tsh^\pm]}(X)$ such that $d_X^2=0$. 

The \emph{dg category of bounded chain complexes} over $\Cat$ is the category $\dgCat$ 
with objects bounded chain complexes $(X,d_X)$ over $\Cat$ and morphism spaces 
the complexes of $\K$-modules
\begin{equation}
	\label{eq:dgCatHom}
\Hom_{\dgCat}\big( (X,d_X), (Y,d_Y) \big) 
:= \Big( \bigoplus_{k \in \Z} \tsh^k \Hom^k_{\Cat[\tsh^\pm]}(X,Y) , D \Big)
\end{equation}
where the differential $D$ is defined by 
$D(f) = d_Y\circ f - (-1)^{|f|}f\circ d_X$.
\end{defn}

One benefit of working with $\dgCat$ is that
other familiar categories of chain complexes are easily recovered from it, 
so it provides a unified setting for studying homological algebra.
The (usual) category $\ChCat$ of bounded chain complexes has the same objects as $\dgCat$ 
and morphisms
\[
\Hom_{\ChCat}\big( (X,d_X), (Y,d_Y) \big) := 
	\ker\big( D \colon \Hom^0_{\Cat[\tsh^\pm]}(X,Y) \to \Hom^1_{\Cat[\tsh^\pm]}(X,Y) \big),
\]
while the homotopy category $\hCat$ of bounded chain complexes has the same objects as $\dgCat$ 
and morphisms given by zeroth homology:
\[
\Hom_{\hCat}\big( (X,d_X), (Y,d_Y) \big) := H^0 \big( \Hom_{\dgCat}\big( (X,d_X), (Y,d_Y) \big)  \big) \, .
\]
It follows that $\ChCat$ is a (non-full) subcategory of $\dgCat$, 
while $\hCat$ is a quotient of $\ChCat$.

Further unpacking the definitions, we see that morphisms in $\ChCat$ are \emph{chain maps}: 
morphisms $f \in \Hom^0_{\Cat[\tsh^\pm]}(X,Y)$ such that $d_Y \circ f = f \circ d_X$.
Similarly, morphisms in $\hCat$ are homotopy classes of chain maps: 
the quotient of the space of chain maps by those that are \emph{null-homotopic}, 
i.e.~those that can be written as $D(h) = d_Y \circ h + h \circ d_X$ for $h \in \Hom^{-1}_{\Cat[\tsh^\pm]}(X,Y)$.
We will denote the corresponding equivalence relation on chain maps by $\sim$.

A morphism $f \in \Hom^0_{\dgCat}\big( (X,d_X), (Y,d_Y) \big)$ in $\dgCat$ is called a 
\emph{homotopy equivalence} provided there exists $g \in \Hom^0_{\dgCat}\big( (Y,d_Y) , (X,d_X) \big)$
such that $\id_X \sim g \circ f$ and $\id_Y \sim f \circ g$. 
We will write $(X,d_X) \simeq (Y,d_Y)$ if there exists a homotopy equivalence between these complexes, 
and refer to such complexes as \emph{homotopy equivalent}.
Note that homotopy equivalent complexes are isomorphic in $\hCat$. 
In the case that $\Cat$ is abelian, we let $\dCat$ denote the \emph{bounded derived category} 
of $\Cat$, which is the localization of $\hCat$ at the class of quasi-isomorphisms.

Finally, note that we can consider $\dgCat$ in the event that $\Cat$ is itself $\Z$-graded.
In this case, the $\Hom$-spaces \eqref{eq:dgCatHom} are complexes of $\Z$-graded 
$\K$-modules, hence are $\Z \times \Z$-graded, 
while the $\Hom$-spaces in $\ChCat$ are only $\Z$-graded (by $\qsh$-degree).
We require homotopy equivalences and quasi-isomorphisms to have $\qsh$-degree zero, 
so this $\Z$-grading is inherited by $\hCat$ and $\dCat$.

\subsubsection{Over a field}

When $\K$ is a field, 
we denote the category of $\K$-vector spaces by $\Vect_\K$.
Given an abelian group $\Gamma$, we denote by $\Vect_\K^\Gamma$ the category of 
$\Gamma$-graded vector spaces and degree-zero linear maps.

We let $\sVect_{\K}$ denote the category of super vector spaces. 
The objects in this category are $\Z/2$-graded vector spaces 
and morphisms are degree-zero linear maps.
As $\K$-linear monoidal categories, 
$\sVect_{\K}$ and $\Vect_\K^{\Z/2}$ are indistinguishable, 
but $\sVect_{\K}$ has a non-trivial symmetric monoidal structure.
We refer to the categorical dimension\footnote{Computed in a general rigid symmetric monoidal category 
as the scalar multiple of the identity in the following composition of the braiding and the 
evaluation/coevaluation: $\one\rightarrow X\otimes X^*\rightarrow X^*\otimes X\rightarrow \one$.} 
computed in the symmetric monoidal category $\sVect_{\K}$ as the \emph{super dimension}. 
The latter is characterized by 
an even line having super dimension $+1$ and an odd line having super dimension $-1$.

In $\sVect_{\K}$, we will denote grading shift (in super degree) by $\ssh$, 
since we also consider $\Gamma$-graded super vector spaces $\sVect_\K^\Gamma$
and wish to reserve $\qsh$ to denote grading shift in that setting.
The super dimension 
of a $\Z$-graded super vector space $V = V_0 \oplus V_1$ 
is equal to 
\begin{equation}
	\label{eq:svsdim}
\dim_\qsh(V) = \dim_\qsh(V_0) - \dim_\qsh(V_1) \, .
\end{equation}
Note that we therefore use $\qdim$ to denote both the graded dimension 
of a graded vector space and 
the graded super dimension of a graded super vector space.
There is no cause for confusion, 
since we can view vector spaces as super vector spaces whose odd degree part is zero.

%
\section{Background on quantum groups and link invariants}
	\label{s:BQG}
%

\subsection{The quantum group}
	\label{ss:QG}
	
We recall background on the quantum group $U_q(\mathfrak{g})$ 
associated to a finite-dimensional simple complex Lie algebra $\mathfrak{g}$, 
for the purpose of establishing conventions and notation. 
Given such a Lie algebra $\mathfrak{g}$, 
there is an associated root system $\Phi(\mathfrak{g})$ and we let 
\[
\{ \ee_i \}_{i \in I} \subset \Phi_+(\mathfrak{g}) \subset \Phi(\mathfrak{g})
\]
denote the subsets of simple roots and positive roots, respectively.
The simple roots are indexed by the finite set $I$ of nodes in the 
corresponding Dynkin diagram,
and the cardinality $| I |$ of this subset is the rank of $\mathfrak{g}$.
If $W_{\mathfrak{g}}$ is the corresponding Weyl group (generated by reflections in the root hyperplanes), 
then we denote the standard $W_{\mathfrak{g}}$-invariant inner product on 
$\ee,\ee' \in \R \Phi(\mathfrak{g})$ by $(\ee , \ee')$, 
which is uniquely determined by the condition that $(\ee , \ee) = 2$ 
when $\ee \in \Phi(\mathfrak{g})$ is a root of minimal length (i.e.~a short root).
For $i,j \in I$, we will occasionally abbreviate by writing $i \cdot j := (\ee_i,\ee_j)$.
The coroot associated to a root $\alpha$ is $\ee^{\vee} := \frac{2}{(\ee , \ee)} \ee$, 
therefore the Cartan matrix $(a_{ij})_{i,j \in I}$ is given by
\[
a_{ij} := (\ee_i^{\vee} , \ee_j) 
= 2 \frac{(\ee_i ,\ee_j)}{(\ee_i , \ee_i)}
=2 \frac{i \cdot j}{i \cdot i} \, .
\]
Write $\rho,\rho^{\vee}\in \R\Phi(\mathfrak{g})$ to denote the elements such that $(\rho, \alpha_i^{\vee})=1$ and $(\rho^{\vee}, \alpha_i) = 1$, for all $i\in I$.

Let $q$ be an indeterminate.
Given $i \in I$ and $m \in \Zge$, set 
\[
q_i := q^{\frac{(\ee_i , \ee_i)}{2}} \, , \quad
[m]_i := [m]_{q_i} := \frac{q_i^m- q_i^{-m}}{q_i- q_i^{-1}}
\]
so if $\ee_i$ is a short simple root then $[m]_i$ agrees 
with the usual quantum integer $[m] := \dfrac{q^m - q^{-m}}{q-q^{-1}}$.

\begin{defn}
	\label{quantumgroupconventions}
Let $U_q(\mathfrak{g})$ be the $\C(q)$-algebra 
generated by elements 
$e_i, f_i, k_i^{\pm 1}$ for $i \in I$, subject to the following relations:
\begin{enumerate}
\item $k_i k_i^{-1} = 1 = k_i k_i^{-1}$,
\item $k_i k_j = k_j k_i$,
\item $k_i e_j= q^{(\ee_i , \ee_j)} e_j k_i$,
\item $k_i f_j= q^{-(\ee_i , \ee_j)}f_j k_i$,
\item $e_i f_i = f_i e_i + \dfrac{k_i - k_i^{-1}}{q_i - q_i^{-1}}$,
\item $e_i f_j = f_j e_i$ if $i\ne j$, and
\item the ``quantum Serre relations'' (of e.g.~\cite[4.3 (R5) and (R6)]{JantzenQgps}).
\end{enumerate}
\end{defn}

Define the divided powers by $e_i^{(a)} := \dfrac{e_i^a}{[a]_i !}$ and $f_i^{(a)} := \dfrac{f_i^a}{[a]_i !}$.  

The $\C(q)$-algebra $U_q(\mathfrak{g})$ is a Hopf algebra and, 
following the conventions in \cite{CP,ST},
the structure maps are given on generators as follows:
\begin{enumerate}
\item $\Delta(e_i) = e_i \otimes k_i + 1\otimes e_i$, 
	$\Delta(f_i) = f_i\otimes 1 + k_i^{-1}\otimes f_i$, 
	and $\Delta(k_i^{\pm}) = k_i ^{\pm}\otimes k_i^{\pm}$.
\item $\mathbf{S}(e_i) = -e_i k_i^{-1}$, 
	$\mathbf{S}(f_i) = - k_i f_i$, 
	and $\mathbf{S}(k_i^{\pm})= k_i^{\mp}$.
\item $\epsilon(e_i)= 0$, 
	$\epsilon(f_i) = 0$, 
	and $\epsilon(k_i) = 1$.
\end{enumerate}

In the present paper, we will be interested in the following quantum groups. 
In both cases, we let $\{ \epsilon_i \}_{i=1}^{N}$ be the standard basis for the 
vector space $\R^N$.

\begin{example}
	\label{ex:slm}
Let $\mathfrak{g} = \sln$, the simple Lie algebra of type $A_{N-1}$. 
In this case $I= \{1,\ldots,N-1\}$ and $\Phi(\sln)$ is the root system with simple roots: 
\[
\{\ee_i = \epsilon_i- \epsilon_{i+1} \}_{i=1}^{N-1} \subset \R^N \, .
\]
In this case, $\R \Phi(\sln)$ is a hyperplane in $\R^N$ and the inner product
is given by restricting the standard inner product $(\epsilon_i , \epsilon_j) = \delta_{ij}$
on $\R^N$ to this hyperplane.
\end{example}

\begin{example}
	\label{ex:son}
Let $\mathfrak{g} = \son$, the simple Lie algebra of type $B_n$. 
In this case $I= \{1,\ldots,n\}$ and $\Phi(\son)$ is the root system with simple roots: 
\[
\{\ee_i = \epsilon_i- \epsilon_{i+1} \}_{i=1}^{n-1} \cup \{\ee_n = \epsilon_n\} \subset \R^n \, .
\]
In this case, $\R \Phi(\son) = \R^n$, and the inner product $(\epsilon_i , \epsilon_j) = 2\delta_{ij}$
is a multiple of the standard inner product on $\R^n$.
\end{example}

Later, we will also be interested in the quantum group $\dU(\glm)$ associated with the (non-simple) 
Lie algebra $\glm$. See \S \ref{ss:idempotent} for the relevant definition.

\subsection{Representations of $U_q(\mathfrak{g})$}
	\label{ss:Rep}

We let $\Rep(U_q(\mathfrak{g}))$ denote the category of 
(type 1) finite-dimensional representations of $U_q(\mathfrak{g})$. 
Every representation $V \in \Rep(U_q(\mathfrak{g}))$ admits a weight space decomposition
$V \cong \bigoplus_{\mu \in X(\mathfrak{g})} V_\mu$ indexed by the weight lattice $X(\mathfrak{g})$.
Recall that the latter consists of all $\mu \in \R \Phi(\mathfrak{g})$ so that $(\ee^{\vee},\mu) \in \Z$.
When $v \in V_\mu$, we will write $\mathrm{wt}(v) = \mu$.
Further, $\Rep(U_q(\mathfrak{g}))$ is semisimple, with irreducibles $V(\lambda)$ indexed by 
dominant weights $\lambda \in X_+(\mathfrak{g})$.

Since $U_q(\mathfrak{g})$ is a Hopf algebra, $\Rep(U_q(\mathfrak{g}))$ is a 
rigid monoidal category. 
Moreover, this category is pivotal and braided.
We now review the latter in depth, since it will play an important role in our considerations.
For this, we must work with representations over the field $\C(q^{\frac{1}{d}})$, 
where $d$ is the index\footnote{This value for $d$ always suffices.
For certain $\mathfrak{g}$ it is possible to use a smaller value for $d$;
see e.g.~\cite{LeIntegrality} for details. 
(There, this parameter is denoted $D$.)} 
of the root lattice in the weight lattice, 
which we do for the duration.

Given $V,W \in \Rep(U_q(\mathfrak{g}))$,
let $\flip_{V, W}$ denote the $\C(q^{\frac{1}{d}})$-linear map $V\otimes W \longrightarrow W\otimes V$ 
that sends $v\otimes w \mapsto w\otimes v$. 
The braiding on $\Rep(U_q(\mathfrak{g}))$ is the invertible operator
\begin{equation}
	\label{eq:braiding}
\RM_{V, W}:= \flip_{V,W} \circ R \colon V\otimes W \longrightarrow W\otimes V
\end{equation}
where $R$ is the \emph{universal $R$-matrix}.
We now recall the approach to the later taken in \cite{KirResh}, 
following \cite{KamnTin, Tingley3, ST}.
For this, we need the following ingredients:
\begin{itemize}
\item Let $q^{(\mathrm{wt}(-) , \mathrm{wt}(-))}$
denote the operator that acts on weight vectors of the form $v \otimes w \in V_{\mu_1} \otimes W_{\mu_2}$ in
a tensor product $V \otimes W$ of finite-dimensional representations by
\begin{equation}
	\label{eq:qwtwt}
q^{(\mathrm{wt}(-) , \mathrm{wt}(-))} (v \otimes w) := q^{(\mu_1, \mu_2)} v\otimes w \, .
\end{equation}

\item Let $T_i$ be the operator that acts on weight vectors $v \in V_{\mu}$ as
\begin{equation}
	\label{eq:LT}
T_i(v) := \sum_{\substack{a, b \ge 0 \\ b-a = (\ee_i^{\vee} , \mu)}}
		(-q_i)^b e_i^{(a)} f_i^{(b)}(v).
\end{equation}
\end{itemize}

The element $T_i$ in \eqref{eq:LT} is equal to Lusztig's quantum Weyl group element $T_{i, +1}''$, 
in the simplified form given e.g.~in \cite[Remark 2.1]{Cautis}. 
These elements satisfy the relations of the type $\mathfrak{g}$ braid group $B_\mathfrak{g}$. 
We have that $\mathrm{wt}(T_i(v)) = s_i ( \mathrm{wt}(v))$, 
where $s_i \in W_\mathfrak{g}$ is the corresponding generator for the Weyl group.
Given any $w \in W_\mathfrak{g}$, we can choose a reduced expression 
$w=s_{i_1} \cdots s_{i_\ell}$ and set
\[
T_w := T_{i_1} \cdots T_{i_\ell}
\]
which is well-defined by Matsumoto's theorem \cite{Matsumoto}. 
Of particular importance is the operator $T_{w_0}$ 
that is associated with the longest element $w_0 \in W_{\mathfrak{g}}$.

In \cite[Theorem 3]{KirResh}, 
Kirillov--Reshetikhin describe the universal $R$ matrix 
in terms of the operator $T_{w_0}$. 
Their result, in the conventions\footnote{In \cite{KamnTin}, 
the authors state that they work with Lusztig's $T_{i,-1}''$;
however, it is clear from \cite[Lemma 5.6]{KamnTin} that they in fact 
work with $T_{i,+1}''$.} 
of \cite[Theorem 7.1]{KamnTin}, is as follows.

\begin{thm}
	\label{thm:KR}
$R = q^{(\mathrm{wt}(-) , \mathrm{wt}(-))} \circ (T_{w_0}^{-1}\otimes T_{w_0}^{-1}) \circ \Delta(T_{w_0})$
\end{thm}

For our considerations in Section \ref{s:trivalent} below,
we will use that the operator $(T_{w_0}^{-1}\otimes T_{w_0}^{-1}) \circ \Delta(T_{w_0})$ 
admits a certain description.
For this, let $U_q(\mathfrak{g})^{\ge 0}$ and $U_q(\mathfrak{g})^{\le 0}$ 
denote the $\C(q)$-subalgebras of $U_q(\mathfrak{g})$ 
generated by $\{e_i, k_i^{\pm}\}_{i \in I}$ and $\{f_i,k_i^{\pm}\}_{i \in I}$, respectively.
These algebras are graded, 
respectively by $\Zge \Phi_+(\mathfrak{g})$ and $\Zge (-\Phi_+(\mathfrak{g}))$, 
and for $\mu \in \Zge \Phi_+(\mathfrak{g})$ we denote the homogeneous components as 
$U_q(\mathfrak{g})_{\mu}^{\ge 0}$ and $U_q(\mathfrak{g})_{-\mu}^{\le 0}$.
We let $U_q(\mathfrak{g})^{> 0}$ and $U_q(\mathfrak{g})^{< 0}$ 
denote the graded $\C(q)$-subalgebras generated by 
the $\{e_i\}_{i \in I}$ and $\{f_i\}_{i \in I}$ alone.

\begin{prop}
	\label{prop:R=xy}
There exist elements 
$\{x_{i,\mu}^+\}_{i=1}^{d_\mu} \subset U_q(\mathfrak{g})^{> 0}$ 
and $\{y_{i,\mu}^-\}_{i=1}^{d_\mu} \subset U_q(\mathfrak{g})^{< 0}$ so that
\[
\big(T_{w_0}^{-1}\otimes T_{w_0}^{-1} \big) \circ \Delta(T_{w_0})
= 1 \otimes 1 + \sum_{\substack{\mu \in \Zge \Phi_+(\mathfrak{g}) \\ \mu \neq 0}} \ \sum_{i=1}^{d_\mu} x_{i,\mu}^+ \otimes y_{i,\mu}^-
\]
\end{prop}
\begin{proof}
As noted e.g.~in \cite[Lemma 3.3.13]{Tingley3} the ``standard'' $R$-matrix appearing in Theorem \ref{thm:KR} admits a description 
of the form
\[
q^{(\mathrm{wt}(-) , \mathrm{wt}(-))} \circ
\left(1 \otimes 1 + \sum_{\substack{\mu \in \Zge \Phi_+(\mathfrak{g}) \\ \mu \neq 0}} \ 
	\sum_{i=1}^{d_\mu} x_{i,\mu}^+ \otimes y_{i,\mu}^- \right)
\]
with $\{x_{i,\mu}^+\}_{i=1}^{d_\mu} \subset U_q(\mathfrak{g})^{> 0}$ 
and $\{y_{i,\mu}^-\}_{i=1}^{d_\mu} \subset U_q(\mathfrak{g})^{< 0}$. 
The result then follows from Theorem \ref{thm:KR}.
\end{proof}

\subsection{Quantum link polynomials}
	\label{ss:linkpoly}

Our approach to link invariants in the present paper is through the classical Alexander and Markov Theorems.
We let 
$\brgroup = \langle \bgen_1 , \ldots, \bgen_{m-1} \mid 
	\bgen_{i}\bgen_{i+1}\bgen_{i} = \bgen_{i+1}\bgen_{i}\bgen_{i+1} \rangle$
denote the (type $A$) braid group on $m$ strands.

\begin{thm}[\cite{Alexander}]
	\label{thm:Alexander}
Let $\mathcal{L} \subset S^3$ be a link in the $3$-sphere.
There is a braid $\beta$ such that the braid closure $\mathcal{L}_{\beta}$ is isotopic to $\mathcal{L}$. 
\end{thm}

\begin{thm}[\cite{Markov}]
	\label{thm:Markov}
Given two braids $\beta \in \brgroup$ and $\beta' \in \brgroup[m']$, 
their closures $\mathcal{L}_{\beta}$ and $\mathcal{L}_{\beta'}$ are isotopic if and only if 
$\beta'$ an be obtained from $\beta$ by a sequence of the following ``Markov moves.''
\begin{enumerate}
\item[(MI)] Conjugate $\beta$ in $\Br_m$.
\item[(MII)] Replace $\beta\in \Br_m$ by $\beta_{m+1}^{\pm 1} \beta \in \Br_{m+1}$ or vice-versa.
\end{enumerate}
\end{thm}

Namely, we will first find invariants of braids $\beta\mapsto T_{\beta}\in X$ 
taking values in a set or category $X$ (possibly with additional structure).
We then compose with a function or functor $X \xrightarrow{[-]} A$ such that 
$[T_{\beta\beta'}]= [T_{\beta'\beta}]$ and $[T_{\beta}]= [T_{\beta_{m+1}^{\pm1} \beta }]$. 
It follows from Theorems \ref{thm:Alexander} and \ref{thm:Markov} that $[T_\bgen]$ 
is an invariant of the link $\mathcal{L}_{\bgen}$.
In practice, we will encounter $[T_{\bgen}]$ that are conjugation invariant, 
but only satisfy invariance under the second Markov move (MII) up to a sign and a power of $q$. 
In this case $[T_{\bgen}]$ will be an invariant of \emph{framed} links, 
i.e.~links endowed with a framing on their normal bundle. 
It is possible to introduce a renormalization factor to obtain an invariant of unframed links, 
but we prefer to work in the present setup. See Remarks \ref{rem:framing} and \ref{rem:framing2} for details.

In this section, we recall
the Reshetikhin--Turaev link invariants defined using the quantum group $U_q(\mathfrak{g})$
associated to a finite-dimensional simple complex Lie algebra $\mathfrak{g}$.
Our exposition is adapted to links presented as braid closures;
see \cite{RT1} for the original construction, which treats all link (and more-generally tangle) diagrams.

Let $\beta = \beta_{i_1}^{\epsilon_1}\cdots \beta_{i_d}^{\epsilon_d} \in \brgroup$ for $\epsilon_i = \pm1$.
Given an $m$-tuple $\vec{\lambda} = (\lambda_1, \ldots, \lambda_m)$ 
of highest weights of finite-dimensional irreducible representations of $\mathfrak{g}$, 
this induces a coloring of $\beta$, 
where the strand meeting the $k^{th}$ point at the bottom of $\beta$ is colored by the weight $\lambda_k$.
If the crossing $\beta_{i_s}^{\epsilon_s}$ is colored by the pair $(\lambda_{l_s},\lambda_{r_s})$ at its bottom, 
then we set
\[
\RM(\beta, \vec{\lambda}) := 
\RM_{\lambda_{l_1},\lambda_{r_1}}^{\epsilon_1} \cdots
\RM_{\lambda_{l_d}, \lambda_{r_d}}^{\epsilon_d}
\]
where here $\RM_{\lambda, \mu}$ denotes the isomorphism
$\RM_{V(\lambda),V(\mu)} \colon V(\lambda) \otimes V(\mu) \to  V(\mu) \otimes V(\lambda)$
from \eqref{eq:braiding},
tensored with appropriate identity morphisms.

We say that $\beta$ is \emph{$\vec{\lambda}$-balanced} if the strand meeting the $k^{th}$ point at the top of 
$\beta$ is also colored by $\lambda_k$. 
In this case, $R(\beta,\vec{\lambda})$ is a braid invariant valued in
$\End_{U_q(\mathfrak{g})} \big( V(\lambda_1) \otimes \cdots \otimes V(\lambda_d) \big)$. 
Since $\Rep\big( U_q(\mathfrak{g}) \big)$ is pivotal, 
composing with the quantum trace $\trq$ 
(i.e.~taking a closure in the associated graphical calculus)
yields an invariant of braid conjugacy classes 
taking values in $\C(q^{\frac{1}{d}})$.  

\begin{thm}[{\cite[Section 6.1]{RT1}}]
	\label{thm:RT}
The braid conjugacy invariant
\begin{equation}
	\label{eq:RT}
P_{\mathfrak{g}}(\mathcal{L}_{\beta}^{\vec{\lambda}})
:= \trq \big(\RM(\beta, \vec{\lambda}) \big)
\end{equation}
descends to give an invariant of the framed link $\mathcal{L}_{\beta}$.
\end{thm}

A priori, 
$P_{\mathfrak{g}}(\mathcal{L}_{\beta}^{\vec{\lambda}})$
takes values in $\C(q^{\frac{1}{d}})$; 
however, it is possible to show\footnote{An even stronger result 
in \cite{LeIntegrality} shows that there exists $r(\mathcal{L}) \in \Q$ depending only on 
the linking matrix of $\mathcal{L}$ such that 
$q^{r(\mathcal{L})} P_{\mathfrak{g}}(\mathcal{L}^{\vec{\lambda}}) \in \Z[q^{\pm 2}]$.} 
that 
$P_{\mathfrak{g}}(\mathcal{L}_{\beta}^{\vec{\lambda}}) \in \Z[q^{\frac{1}{d}},q^{-\frac{1}{d}}]$.
We thus refer to 
$P_{\mathfrak{g}}(\mathcal{L}_{\beta}^{\vec{\lambda}})$
as the 
\emph{$\vec{\lambda}$-colored $U_q(\mathfrak{g})$ link polynomial}. 
When $\vec{\lambda} = (\lambda,\ldots,\lambda)$, we will denote this invariant simply by
$P_{\mathfrak{g}}(\mathcal{L}_{\beta}^{\lambda})$

\begin{rem}
	\label{rem:framing}
As defined in \eqref{eq:RT}, 
$P_{\mathfrak{g}}(\mathcal{L}_{\beta}^{\vec{\lambda}})$
changes by a factor of
$\nu_{\lambda} = (-1)^{( 2\lambda, \rho^{\vee})} q^{\pm (\lambda, \lambda+2\rho)}$
upon applying a positive/negative Reidemeister~I move on a $\lambda$-colored component,
i.e.~ graphically:
\begin{equation} \label{eq:curlsforframing}
P_{\mathfrak{g}} \left(
\begin{tikzpicture}[scale=.5,anchorbase]
	\draw[very thick] (.8,-.4) to [out=180,in=270] (0,1);
	\draw[overcross] (0,-1) to [out=90,in=180] (.8,.4);
	\draw[very thick] (0,-1) to [out=90,in=180] (.8,.4);
	\draw[very thick] (.8,.4) to [out=0,in=90] (1.1,0) node[right=-2pt, black]{\scs$\lambda$}
		to [out=270,in=0] (.8,-.4);
\end{tikzpicture}
\! \right)
= (-1)^{( 2\lambda, \rho^{\vee})} q^{(\lambda, \lambda+2\rho)}
P_{\mathfrak{g}} \left( \
\begin{tikzpicture}[scale=.5,anchorbase]
	\draw[very thick] (0,-1) to [out=90,in=270] node[right=-2pt, black]{\scs$\lambda$} (0,1);
\end{tikzpicture}
\! \right)
\quad \text{and} \quad
P_{\mathfrak{g}} \left(
\begin{tikzpicture}[scale=.5,anchorbase,yscale=-1]
	\draw[very thick] (.8,-.4) to [out=180,in=270] (0,1);
	\draw[overcross] (0,-1) to [out=90,in=180] (.8,.4);
	\draw[very thick] (0,-1) to [out=90,in=180] (.8,.4);
	\draw[very thick] (.8,.4) to [out=0,in=90] (1.1,0) node[right=-2pt, black]{\scs$\lambda$}
		to [out=270,in=0] (.8,-.4);
\end{tikzpicture}
\! \right)
= (-1)^{( 2\lambda, \rho^{\vee})} q^{-(\lambda, \lambda+2\rho)}
P_{\mathfrak{g}} \left( \
\begin{tikzpicture}[scale=.5,anchorbase]
	\draw[very thick] (0,-1) to [out=90,in=270] node[right=-2pt, black]{\scs$\lambda$} (0,1);
\end{tikzpicture}
\! \right) \, .
\end{equation}
It follows that two diagrams for a given link yield the same value of 
$P_{\mathfrak{g}}(\mathcal{L}_{\beta}^{\vec{\lambda}})$
if and only if they have the same writhe ($= \#\{\text{positive crossings}\} - \#\{\text{negative crossings}\}$).
\end{rem}

\begin{rem} \label{rem:framing2}
Let $P(\bgen)$ be a braid-conjugacy invariant which satisfies 
the $\pm$-Markov II moves only up to 
a factor of $c^{\pm1}$ for some invertible scalar $c$.
For example, $P(\beta) = P_{\mathfrak{g}}(\mathcal{L}_{\beta}^{\lambda})$ 
satisfies this property with $c = (-1)^{( 2\lambda, \rho^{\vee})} q^{(\lambda, \lambda+2\rho)}$ by \eqref{eq:curlsforframing}. 
It is well known that $P(\beta)$ then descends to an invariant of framed links. 
Since an appropriate citation has eluded us, 
we provide a quick sketch.
Let, $\epsilon(\bgen)$ denote the braid exponent, 
then $c^{-\epsilon(\bgen)}P(\bgen)$ is invariant under the Markov moves MI and MII, 
therefore is an invariant of unframed links. 
Since the braid exponent equals the writhe of the braid closure,
two braid closures have the same value for $P(\bgen)$ 
if they have the same writhe, i.e.~if they have the same (blackboard) framing.
\end{rem}

\begin{rem}
	\label{rem:Rconv}
The link polynomials 
$P_{\mathfrak{g}}(\mathcal{L}_{\beta}^{\vec{\lambda}})$
are defined using the braiding on $\Rep(U_q(\mathfrak{g}))$. 
The latter is essentially unique up to conventions \cite{KhoroshkinTolstoy}, 
but not literally unique since e.g.~the inverse to a braiding is again a braiding.
Since we wish to use Theorem \ref{thm:KR} and Proposition \ref{prop:R=xy} below, 
our conventions for the braiding/$R$-matrix are taken from \cite{KamnTin,ST} 
and they lead to behavior described in Remark \ref{rem:framing}.
However, it is common in the link homology literature to work with 
link polynomials that correspond to the inverse choice of $R$-matrix. 
We denote these link polynomials by 
$\bP_{\mathfrak{g}}(\mathcal{L}_{\beta}^{\vec{\lambda}})$
and observe that
\begin{equation}
	\label{eq:Pbar}
\bP_{\mathfrak{g}}(\mathcal{L}_{\beta}^{\vec{\lambda}})(q) = 
P_{\mathfrak{g}}(\mathcal{L}_{\beta}^{\vec{\lambda}})(q^{-1}) =
P_{\mathfrak{g}}(\mirror \mathcal{L}_{\beta}^{\vec{\lambda}})(q) =
P_{\mathfrak{g}}(\mathcal{L}_{\mirror \beta}^{\vec{\lambda}})(q)
\end{equation}
where here $\mirror \mathcal{L}$ and $\mirror \beta$ denote the mirror 
link/braid (obtained by switching over/under information at all crossings).
Consequently, we have
\[
\bP_{\mathfrak{g}} \left(
\begin{tikzpicture}[scale=.5,anchorbase]
	\draw[very thick] (.8,-.4) to [out=180,in=270] (0,1);
	\draw[overcross] (0,-1) to [out=90,in=180] (.8,.4);
	\draw[very thick] (0,-1) to [out=90,in=180] (.8,.4);
	\draw[very thick] (.8,.4) to [out=0,in=90] (1.1,0) node[right=-2pt, black]{\scs$\lambda$}
		to [out=270,in=0] (.8,-.4);
\end{tikzpicture}
\! \right)
= (-1)^{( 2\lambda, \rho^{\vee})} q^{-(\lambda, \lambda+2\rho)}
\bP_{\mathfrak{g}} \left( \
\begin{tikzpicture}[scale=.5,anchorbase]
	\draw[very thick] (0,-1) to [out=90,in=270] node[right=-2pt, black]{\scs$\lambda$} (0,1);
\end{tikzpicture}
\! \right)
\quad \text{and} \quad
\bP_{\mathfrak{g}} \left(
\begin{tikzpicture}[scale=.5,anchorbase,yscale=-1]
	\draw[very thick] (.8,-.4) to [out=180,in=270] (0,1);
	\draw[overcross] (0,-1) to [out=90,in=180] (.8,.4);
	\draw[very thick] (0,-1) to [out=90,in=180] (.8,.4);
	\draw[very thick] (.8,.4) to [out=0,in=90] (1.1,0) node[right=-2pt, black]{\scs$\lambda$}
		to [out=270,in=0] (.8,-.4);
\end{tikzpicture}
\! \right)
= (-1)^{( 2\lambda, \rho^{\vee})} q^{(\lambda, \lambda+2\rho)}
\bP_{\mathfrak{g}} \left( \
\begin{tikzpicture}[scale=.5,anchorbase]
	\draw[very thick] (0,-1) to [out=90,in=270] node[right=-2pt, black]{\scs$\lambda$} (0,1);
\end{tikzpicture}
\! \right) \, .
\]
\end{rem}

In Section \ref{ss:SLP}, we further study the invariant 
$P_{\mathfrak{g}}(\mathcal{L}_{\beta}^{\vec{\lambda}})$
for $\mathfrak{g} = \son$ and when each entry of $\vec{\lambda}$ corresponds to the spin representation.
Later, in Section \ref{s:LH}, we consider the $\mathfrak{g} = \sln$ case of the invariant 
$P_{\mathfrak{g}}(\mathcal{L}_{\beta}^{\vec{\lambda}})$ in more detail, 
and review its categorification via $\sln$ Khovanov--Rozansky homology. 
This forms the foundation for our categorification of the invariant from \S \ref{ss:SLP}.

%
\section{Type $B$ Intertwiners}
	\label{s:typeB}
%

In this section, we study representations of $U_q(\son)$ in depth. 
In this case, there is a distinguished \emph{spin representation} $S \in \Rep(U_q(\son))$, 
and our main result is the construction of a ``canonical basis'' for the 
endomorphism algebra $\End_{U_q(\son)}(S \otimes S)$.
Although we discuss this material before our categorification results that in appear in later sections, 
this presentation is somewhat ahistorical: 
the existence of this basis (and the corresponding structure coefficients for multiplication)
was initially suggested by the techniques appearing in \S \ref{s:ECQG}.

In order to state this result precisely, 
we need the following, which we affectionately refer to as the \emph{devil's product}.

\begin{defn}
	\label{def:devil}
Let $m \leq n \in \Zge$ and set
\[
``[m][n]" = ``[n][m]" := \sum_{i=0}^{m-1} (-1)^i [n+m-2i-1] \, .
\]
\end{defn}

In other words, the expression
\[
``[m][n]" = [n+m- 1] - [n+m - 3] + [n+m - 5] - \cdots + (-1)^{m-1} [n-m+1]
\]
is obtained from the usual expression for the product of quantum integers
\[
[m][n] = \sum_{i=0}^{m-1} [n+m-2i-1]
	= [n+m- 1] + [n+m - 3] + [n+m - 5] + \cdots + [n-m+1]
\]
by introducing alternating signs. It is essential that $m \leq n$ above; 
one does not obtain the same result by swapping $m$ and $n$ in the formula.
Note that if $m=0$, the summation is empty, so by convention $``[0][n]"=0$ for all $n \in \Z_{\geq 0}$.
See Appendix \ref{S:combinatorial-identities} for further discussion.

We now precisely record the main result of this section.

\begin{thm}
	\label{thm:X}
Let $S \in \Rep(U_q(\son))$ be the spin representation.
There is a basis 
\[
\lbrace \xx^{(i)}\rbrace_{i=0}^{n} \subset \End_{U_q(\son)}(S\ot S)
\]
such that
\[
\xx^{(i)} \xx^{(1)}  = (-1)^i``[i+1][i+1]"\xx^{(i+1)} + (-1)^{i} ``[i][i+1]"\xx^{(i)}
\]
and so that the braiding is given by
\[
\RM_{S, S}= q^{\frac{n}{2}} \sum_{i=0}^n q^{-i}\xx^{(i)} \, .
\]
\qed
\end{thm}

This result is established in Propositions \ref{itoxchangeofbasis} and \ref{P:decat-spin-rickard}.
Later on, Theorems 
\ref{T:div-powers-in-twisted-K0}, 
\ref{T:iserre-in-twisted-K0}, 
and \ref{thm:Ctaubraid} 
categorify this result.

\subsection{Representations of $U_q(\son)$}

In this section, we review the fundamental representations 
$\{V(\varpi_1),\ldots,V(\varpi_{n-1}),S\}$ of $U_q(\son)$.
In doing so, we will follow the conventions and notation
of Example \ref{ex:son}.
Further, as mentioned in Section \ref{ss:Rep}, 
we consider representations over the field $\C(q^{\frac{1}{2}})$.

For $1 \leq k \leq n-1$, 
the fundamental weights are $\varpi_k =\epsilon_1 + \cdots + \epsilon_{k}$ and
the representations $V(\varpi_k)$ are quantized analogues of the 
exterior powers $\Lambda^k(\C^{2n+1})$, 
while $\varpi_n = \frac{1}{2}(\epsilon_1 + \cdots + \epsilon_{n})$ and 
$S:=V(\varpi_{n})$ is the (quantum) $\emph{spin representation}$.
We begin by studying the latter in depth, since it plays a leading role in the present work.

For $i,j = \{1,\ldots,n\}$, let
\[
\| i,j \| :=
\begin{cases}
-2 & \text{if } j=i \\
2 & \text{if } j=i+1 \\
0 & \text{else}
\end{cases}
\]
and for $J \subset \{1,\ldots,n\}$ denote $\| i, J \| = \sum_{j\in J} \| i, j \| \in \{-2,0,2\}$. 
Also, set 
\[
q^J = \prod_{j\in J} (-q)(-q^2)^{n-j} = \prod_{j\in J} (-1)^{n-j+1} q^{2(n-j)+1}
\]
and $q^{-J} = (q^J)^{-1}$.

\begin{defn}\label{def:spin}
The \emph{spin representation} of $U_q(\son)$ is
the $\C(q^{\frac{1}{2}})$-vector space $S$ with basis $\{x_J\}_{J \in \mathcal{P}}$ 
where $\mathcal{P}$ is the power set of $\lbrace 1, 2, \ldots, n\rbrace$.
For $i \in \{1, \ldots, n-1\}$, the action of $U_q(\son)$ on $S$ is specified as follows:
\begin{enumerate}
\item $k_i^{\pm1} x_J = q^{\pm \| i, J \|}x_J$, 
\item $k_n^{\pm1} x_J = q^{\pm(1+ \| n, J \|)} x_J$
\item $e_i x_J = x_{J \smallsetminus \{ i \} \cup \{ i+1\} }$ if $i\in J$ and $i+1 \notin J$,
\item $e_n x_J = x_{J \smallsetminus \{ n\}}$ if $n\in J$,
\item $f_i x_J= x_{J \smallsetminus \{ i+1\} \cup \{ i\} }$ if $i+1\in J$ and $i\notin J$,
\item $f_n x_J = x_{J\cup \{ n \}}$ if $n\notin J$,
\end{enumerate}
and where all $e_j$ and $f_j$ for $j \in \{1,\ldots,n\}$ act by zero otherwise.
\end{defn}

Using the relations in Definition \ref{quantumgroupconventions},
it is straightforward to check that $S$ is indeed an irreducible $U_q(\son)$-representation.
Note that $x_{\emptyset} \in S$ is thus a highest weight vector with 
$\mathrm{wt}(x_{\emptyset}) = \varpi_n = \frac{1}{2}(\epsilon_1 + \cdots + \epsilon_{n})$,
and hence $S\cong V(\varpi_n)$. 

\begin{lem}\label{qWeylLemma}
Let $x_J \in S$ and $i\in \{ 1, \ldots, n \}$. Exactly one of the following holds:
\begin{enumerate}
\item $(\alpha_i^{\vee}, \mathrm{wt}(x_J)) = 1$ and $T_i x_J= -q_i f_i x_J$,
\item $(\alpha_i^{\vee}, \mathrm{wt}(x_J))= -1$ and $T_i x_J= e_i x_J$, or
\item $(\alpha_i^{\vee}, \mathrm{wt}(x_J)) = 0$ and $T_i x_J = x_J$. 
\end{enumerate}
\end{lem}
\begin{proof}
Using Definition \ref{def:spin} and the formulae for the simple roots from Example \ref{ex:son}, 
we compute that 
\begin{equation}
	\label{eq:wtxJ}
\mathrm{wt}(x_J) = \frac{1}{2} \left( \sum_{i \notin J} \epsilon_i - \sum_{i \in J} \epsilon_i  \right) \, .
\end{equation}
It is then an easy consequence that 
$(\alpha_i^{\vee}, \mathrm{wt}(x_J))\in \{ -1, 0, 1\}$ for all $\{i\}, J \subset \{1,\ldots,n\}$, 
thus $S$ decomposes as a sum of $1$- and $2$-dimensional irreducible representations 
upon restriction to the subalgebra
$U_{q_i}(\sln[2]) \subset U_q(\son)$ generated by $e_i, f_i$, and $k_i^{\pm 1}$. 
The result then follows from \eqref{eq:LT}.
\end{proof}

Using Lemma \ref{qWeylLemma}, 
we can compute the action of $T_{w_0}$ on certain basis vectors $x_J \in S$, 
where as always $w_0 \in W_{\son}$ denotes the longest element.
Recall that the type $B_n$ Weyl group $W_{\son}$ is the \emph{hyperoctahedral group}
which, as always, admits a presentation as a Coxeter group
\[
W_{\son} := \langle s_1,\ldots, s_n 
	\mid s_i^2 = e, \, s_is_{i+1}s_i = s_{i+1}s_{i}s_{i+1} \text{ if } i \leq n-2, \,
		s_{n-1}s_n s_{n-1}s_n = s_n s_{n-1}s_n  s_{n-1} \rangle \, .
\]
However, $W_{\son}$
can also be identified with the subgroup of $\mathrm{Perm}\big(\{-n,\ldots,-1,1,\ldots,n\} \big) \cong \SG_{2n}$
consisting of permutations that satisfy $\sigma(-i) = -\sigma(i)$ via the assignment:
\[
s_i \mapsto 
\begin{cases}
(i-n, i-1-n)(n-i, n-i+1) & \text{ if } 1 \leq i \leq n-1 \\
(-1,1) & \text{ if } i=n \, .
\end{cases}
\]
(Here $(k,\ell)$ denotes the transposition that interchanges $k$ and $\ell$.)
Under this inclusion, the longest element $w_0 \in W_{\son}$ is identified with 
the longest element of $\SG_{2n}$, the half-twist permutation.

\begin{lem}\label{tw0lemma}
\[
T_{w_0}(x_{\lbrace 1, \ldots, i\rbrace}) = q^{\lbrace i+1, \ldots, n\rbrace}x_{\lbrace i+1, \ldots, n\rbrace}
\]
\end{lem}
\begin{proof}
Set 
\[
w_{\{1, \ldots, i \} }^{ \lbrace 1, \ldots, n \rbrace} 
	:= (s_n)(s_{n-1}s_n)\cdots(s_{i+1}\cdots s_{n-1}s_n) 
	\, , \quad
w_{\lbrace 1, \ldots , n\rbrace}^{\lbrace i+1, \ldots , n\rbrace} 
	:= (s_{i}\cdots s_n)(s_{i-1}\cdots s_n)\cdots (s_1\cdots s_n) \, ,
\]
and let $w_0^{1, \ldots , n-1}$ denote the longest element of the parabolic subgroup 
$\langle s_1, \ldots, s_{n-1} \rangle \subset W_{\son}$.
For each $i \in \{1,\ldots,n\}$ the longest element $w_0 \in W_{\son}$ can be written as
\[
w_0= w_{\lbrace 1, \ldots, n\rbrace}^{\lbrace i+1, \ldots , n\rbrace} 
	w_0^{1, \ldots , n-1} w_{\lbrace1, \ldots , i\rbrace}^{\lbrace 1, \ldots n\rbrace}
\]
which e.g.~can be verified using the aforementioned inclusion $W_{\son} \hookrightarrow \SG_{2n}$.

Next, Definition \ref{def:spin} gives that
\[
x_{\lbrace1, \ldots , n\rbrace} 
= (f_n)(f_{n-1}f_n) \cdots (f_{i+2} \cdots f_{n-1}f_n) (f_{i+1} \cdots f_{n-1} f_n) x_{\lbrace 1, \ldots, i\rbrace}
\]
and
\[
x_{\lbrace i+1, \ldots, n\rbrace} 
= (e_{i} \cdots e_{n-1}e_n) \cdots (e_2\cdots e_{n-1}e_n)(e_1 \cdots e_{n-1}e_n)x_{\lbrace 1, \ldots n\rbrace} \, ,
\]
so repeated application of Lemma \ref{qWeylLemma} implies that
\[
T_{w_{\lbrace1, \ldots , i\rbrace}^{\lbrace 1, \ldots n\rbrace}}(x_{\lbrace 1, \ldots , i\rbrace}) 
	= q^{\lbrace i+1, \ldots, n \rbrace} x_{\lbrace 1, \ldots n\rbrace}
\]
and
\[
T_{w_{\lbrace 1, \ldots n\rbrace}^{\lbrace i+1, \ldots n\rbrace}}(x_{\lbrace 1, \ldots, n\rbrace})= x_{\lbrace i+1, \ldots, n\rbrace} \, .
\]
Since $(\alpha_i^{\vee}, \mathrm{wt}(x_{\{1,\ldots,n\}})) = 0$ for $1 \leq i \leq n-1$, 
Lemma \ref{qWeylLemma} also gives that 
\[
T_{w_0^{1, \ldots , n-1}}(x_{\lbrace 1, \ldots , n\rbrace})= x_{\lbrace 1, \ldots , n\rbrace}
\] 
and the result follows.
\end{proof}

We next discuss the remaining fundamental representations.

\begin{defn}
For $i= 1, 2, \ldots n-1$, 
let $V_i := V(\varpi_i)$ denote the Weyl module for $U_q(\son)$ 
with fixed highest weight vector $v_i^+ \in V_i$ of weight $\varpi_i= \epsilon_1+ \epsilon_2 + \cdots + \epsilon_i$.
\end{defn}

We will let $v_i^- \in V_i$ denote the unique vector in the lowest weight space of $V_i$ 
such that $v_i^+= T_{w_0}(v_i^-)$.
Note that $\mathrm{wt}(v_i^-) = w_0(\varpi_i) = - \varpi_i$.
Extending this notation, we will also write $V_0:= \C(q^{\frac{1}{2}})$ to denote the trivial $U_q(\son)$-module
(with fixed highest weight vector $v_0^+:= 1 = v_0^-$) and denote the 
distinguished highest and lowest weight vectors of $S$ 
by $v_n^+ := v_\emptyset$ and $v_n^- := v_{\{1,\ldots,n\}}$.

The representation $V_1$ admits the following explicit description, 
which we will use below to study certain morphisms in $\Rep(U_q(\son))$.

\begin{prop}
	\label{prop:V1basis}
The representation $V_1$ has a basis $\{ a_1, a_2, \ldots, a_n, u, b_n, \ldots, b_2, b_1 \}$
such that 
\[
\mathrm{wt}(a_i)=\epsilon_i \, , \quad \mathrm{wt}(u) = 0 \, , \quad \mathrm{wt}(b_i) = -\epsilon_i \, .
\]
With respect to this basis, 
the $U_q(\son)$ action is given by
\begin{enumerate}
\item $f_i a_i= a_{i+1}$ and $f_i b_{i+1} = b_i$ for $1 \leq i \leq n-1$,
\item $e_i a_i= a_{i-1}$ and $e_i b_i = b_{i+1}$ for $1 \leq i \leq n-1$,
\item $f_n a_n= u$ and $f_n u= [2]b_n$,
\item $e_n b_n = u$ and $e_n u= [2]a_n$,
\item $k_i v= q^{(\ee_i, \mathrm{wt}(v))} x$, for $1 \leq i \leq n$,
\end{enumerate}
and where all $e_i$ and $f_i$ for $1 \leq i \leq n$ act by zero otherwise.
\end{prop}
\begin{proof}
An exercise using Definition \ref{quantumgroupconventions}.
\end{proof}

If we identify $v_1^+ = a_1$, then a computation using \eqref{eq:LT} 
shows that $v_1^{-} = b_1$.

\subsection{Trivalent vertices}
	\label{s:trivalent}

We now begin our study of morphisms in $\Rep(U_q(\son))$.
It is a standard fact that
\begin{equation}
	\label{eq:StSdecomp}
S\otimes S \cong \C(q^{\frac{1}{2}}) \oplus \left( \bigoplus_{i=1}^{n-1}V(\varpi_i) \right) \oplus V(2\varpi_n) 
	= \left( \bigoplus_{i=0}^{n-1} V_i  \right) \oplus V(2\varpi_n)
\end{equation}
and we first describe the inclusions of the submodules $V_i$ for $0 \leq i \leq n-1$.
Note: in this subsection we will break our conventions from \S \ref{ss:QG}
by using the symbol $I$ to denote arbitrary subsets of $\{1,\ldots,n\}$ 
(as opposed to letting it denote the set of Dynkin nodes, 
as is done in other sections of the paper).
Given such $I \subset \{1,\ldots,n\}$, we will denote its complement by 
$I^c = \{1,\ldots,n\} \smallsetminus I$.

\begin{prop}
	\label{prop:Y}
For each $0 \leq i \leq n-1$, 
there is a unique map $\Ya[i] \colon V_i \longrightarrow S\otimes S$ 
of $U_q(\son)$-modules such that 
\begin{equation}\label{highestwti}
\Ya[i](v_i^+):=  \sum_{\substack {I \cap J = \emptyset \\ I \cup J  = \lbrace i+1, \ldots, n\rbrace}} q^{J} x_{I} \otimes x_J \, .
\end{equation}
\end{prop}

\begin{proof}
Equation \eqref{eq:wtxJ} implies that $\Ya[i](v_i^+)$ is a non-zero vector in the $\varpi_i$-weight space of $S\otimes S$, 
so it suffices to show that it is a highest weight vector, i.e.~that $U_q(\son)^{> 0} \cdot \Ya[i](v_i^+)= 0$. 
Recall that $e_\ell$ acts on tensor products via $\Delta(e_\ell) = e_\ell \otimes k_\ell + 1 \otimes e_\ell$.
For $\ell \in \lbrace 1, \ldots, n-1\rbrace$, we thus compute:
\begin{align*}
e_\ell \cdot \Ya[i](v_i^+) &= \sum_{\substack {I\cap J = \emptyset \\ I\cup J  = \lbrace i+1, \ldots, n\rbrace}} q^{J} e_\ell (x_I) \otimes k_\ell (x_J)
	+ \sum_{\substack {I\cap J = \emptyset \\ I\cup J  = \lbrace i+1, \ldots, n\rbrace}} q^{J} x_I \otimes e_\ell (x_J) \\
&= \sum_{\substack {I \cap J = \emptyset \\ I \cup J = \lbrace i+1, \ldots, n\rbrace \\ \ell \in I, \ell+1\notin I}} 
	q^{J}q^{\| \ell, J \|} x_{I\smallsetminus \lbrace \ell \rbrace \cup \lbrace \ell+1\rbrace}\otimes x_{J}
		+ \sum_{\substack {I\cap J = \emptyset \\ I\cup J  = \lbrace i+1, \ldots, n\rbrace \\ \ell \in J, \ell+1 \notin J}} 
			q^{J} x_I\otimes x_{J\smallsetminus \lbrace \ell \rbrace \cup \lbrace \ell+1\rbrace}
\end{align*}
which is zero when $\ell \le i$, since both summations are empty.
If $\ell \in \lbrace i+1, \ldots , n-1\rbrace$, we then have
\[
\sum_{\substack{I\cap J = \lbrace \ell +1 \rbrace \\ I\cup J = \lbrace i+1, \ldots , n\rbrace \smallsetminus \lbrace \ell \rbrace}} 
	\left(q^{J \cup \lbrace \ell \rbrace \smallsetminus \lbrace \ell+1\rbrace} + q^{J}q^{\| \ell, J \|}\right)x_I\otimes x_J \, .
\]
Since 
\[
q^{J \cup \lbrace \ell \rbrace \smallsetminus \lbrace \ell +1\rbrace} 
= \frac{(-q)(-q^2)^{n-\ell}}{(-q)(-q^2)^{n-(\ell+1)}} q^J = -q^{2} q^J = - q^J q^{\| \ell , J \|} \, ,
\]
it follows that $e_\ell \cdot \Ya[i](v_i^+)= 0$ in this case as well. 
Similarly, one can compute that $e_n\cdot \Ya[i](v_i^+) = 0$. 
\end{proof}

We next compute the value of $\Ya[i]$ on our distinguished lowest-weight vectors $v_i^- \in V_i$.

\begin{lem}
We have that
\begin{equation}
	\label{eq:Yvminus}
\Ya[i](v_i^-)=  \sum_{\substack {I\cap J = \emptyset \\ I\cup J  = \lbrace i+1, \ldots, n\rbrace}} q^{J} x_{J^c}\otimes x_{I^c} \, .
\end{equation}
\end{lem}

\begin{proof}
Let $w_i^- \in S \otimes S$ denote the right-hand side of \eqref{eq:Yvminus}.
We leave it as an exercise to verify that 
$\mathrm{wt}(w_i^-) = -\varpi_i$ and that $w_i^{-}$ is a lowest weight vector, i.e.~that
$U_q(\son)^{< 0} \cdot w_i^- = 0$. 
It then follows that 
$\Ya[i](v_i^{-}) = \chi \cdot w_i^-$
for some $\chi \in \C(q^{\frac{1}{2}})$. 
We will show that $\chi =1$ by computing the value of 
$(T_{w_0}^{-1}\otimes T_{w_0}^{-1}) \circ \Delta(T_{w_0})$ on both $\Ya[i](v_i^{-})$ and $\chi \cdot w_i^-$.

First, since $x_{\{1,\ldots,n\}} \in S$ is a lowest weight vector, 
Proposition \ref{prop:R=xy} gives that 
\begin{equation}
	\label{computeRmatrix1}
\left((T_{w_0}^{-1}\otimes T_{w_0}^{-1}) \circ \Delta(T_{w_0}) \right)(w_i^-) = 
	\chi \cdot x_{\lbrace 1,\ldots, n\rbrace}\otimes x_{\lbrace 1, \ldots, i\rbrace}
		+ \sum_{\substack{I', J' \\ I' \ne\lbrace 1, \ldots, n\rbrace}} \xi_{I',J'} x_{I'} \otimes x_{J'},
\end{equation}
for some $\xi_{I',J'}\in \C(q^{\frac{1}{2}})$.
Next, we compute that
\begin{align*}
\big((T_{w_0}^{-1}\otimes T_{w_0}^{-1}) \circ \Delta(T_{w_0}) \big) (\Ya[i](v_i^-)) 
	&= \big(T_{w_0}^{-1}\otimes T_{w_0}^{-1}  \circ \Ya[i] \big) (T_{w_0}v_i^-) \\
	&= \big(T_{w_0}^{-1}\otimes T_{w_0}^{-1}  \circ \Ya[i] \big) (v_i^+) \\
	&= \sum_{\substack {I\cap J = \emptyset \\ I\cup J  = \lbrace i+1, \ldots, n\rbrace}} q^{J} T_{w_0}^{-1}(x_I)\otimes T_{w_0}^{-1}(x_J) \, .
\end{align*}
Now, Lemma \ref{tw0lemma} gives that 
$T_{w_0}^{-1}(x_{\emptyset}) = x_{\lbrace 1, \ldots, n\rbrace}$ 
and $T_{w_0}^{-1}(x_{\lbrace i+1, \ldots, n\rbrace}) = q^{-\lbrace i+1, \ldots, n\rbrace}x_{\lbrace 1, \ldots, i\rbrace}$, 
while $T_{w_0}^{-1}x_I\in \C(q^{\frac{1}{2}}) \cdot x_{I^c}$ for all $I$,
so
\begin{equation}
	\label{computeRmatrix2}
\begin{multlined}
\big((T_{w_0}^{-1}\otimes T_{w_0}^{-1}) \circ \Delta(T_{w_0}) \big) (\Ya[i](v_i^-)) = 
	q^{-\lbrace i+1, \ldots, n\rbrace}q^{\lbrace i+1, \ldots, n\rbrace}
		x_{\lbrace 1,\ldots, n\rbrace}\otimes x_{\lbrace 1, \ldots, i\rbrace} \\
		+ \sum_{\substack{I', J' \\ I' \ne\lbrace 1, \ldots, n\rbrace}} \zeta_{I',J'} x_{I'} \otimes x_{J'} \, ,
\end{multlined}
\end{equation}
for some $\zeta_{I',J'}\in \C(q^{\frac{1}{2}})$. The result then follows by comparing the
coefficients of $x_{\lbrace 1,\ldots, n\rbrace}\otimes x_{\lbrace 1, \ldots, i\rbrace}$ in
\eqref{computeRmatrix1} and \eqref{computeRmatrix2}.
\end{proof}

\begin{example}
	\label{ex:Y1}
We record the action of $\Ya[1]$ on the basis for $V_1$ from Proposition \ref{prop:V1basis}.
For $\ell \in \{1,\ldots,n\}$, 
let $\sigma_\ell \colon \lbrace 1\rbrace^c \longrightarrow \lbrace \ell \rbrace^c$ 
be the unique order preserving bijection, 
e.g.~if $n=4$ then $\sigma_3(2) = 1, \sigma_3(3) = 2$, and $\sigma_3(4) = 4$.
We then have
\[
\begin{gathered}
\Ya[1](a_\ell) = \sum_{I\subset \lbrace 2, \ldots, n\rbrace} q^Ix_{\sigma_\ell(\{2,\ldots,n\} \smallsetminus I)}\ot x_{\sigma_\ell(I)}
\ , \quad
\Ya[1](b_\ell) = \sum_{I\subset \lbrace 2, \ldots, n\rbrace} 
	q^Ix_{\sigma_\ell(\{2,\ldots,n\} \smallsetminus I)\cup \lbrace \ell \rbrace}\ot x_{\sigma_\ell(I) \cup \lbrace \ell \rbrace} \\
\Ya[1](u) = \sum_{I\subset \lbrace 2, \ldots, n\rbrace} q^Ix_{\sigma_n(\{2,\ldots,n\} \smallsetminus I)\cup\lbrace n\rbrace}\ot x_{\sigma_n(I)} 
	+ \sum_{I\subset \lbrace 2, \ldots, n\rbrace} q^{-1}\cdot 
		q^Ix_{\sigma_n(\{2,\ldots,n\} \smallsetminus I)}\ot x_{\sigma_n(I)\cup \lbrace n\rbrace}  \, .
\end{gathered}
\]
\end{example}

Next, we compute the composition of the braiding \eqref{eq:braiding} 
with the morphisms $\Ya[i]$ from Proposition \ref{prop:Y}. 
Below, we will use this result to give a formula for the braiding on $S \otimes S$ 
in terms of (morphisms built from) the morphisms $\Ya[i]$.

\begin{lem}
	\label{trivalenttwist}
For $0 \leq i \leq n-1$, we have
\[
\RM_{S, S}\circ \Ya[i] = q^{-\left(\frac{n-2i}{2} \right)} q^{-\lbrace i+1, \ldots, n\rbrace} \Ya[i] \, .
\]
\end{lem}
\begin{proof}
The decomposition \eqref{eq:StSdecomp} implies that
\[
\dim \big( \Hom_{U_q(\mathfrak{so}_{2n+1})}(V_i , S\otimes S) \big) =1
\]
so $\RM_{S, S}\circ \Ya[i]$ is necessarily a scalar multiple of $\Ya[i]$. 
We identify this scalar by computing the value of each on the highest weight vector $v_i^+ \in V_i$.
Proposition \ref{prop:Y} gives that
\[
\Ya[i](v_i^+) = x_{\{i+1,\ldots,n\}} \otimes x_{\emptyset}
	+ \sum_{\substack {I \cap J = \emptyset, J \neq \emptyset \\ I \cup J  
		= \lbrace i+1, \ldots, n\rbrace}} q^{J} x_{I} \otimes x_J
\]
so Proposition \ref{prop:R=xy} gives that
\[
\big((T_{w_0}^{-1}\otimes T_{w_0}^{-1}) \circ \Delta(T_{w_0}) \big)(\Ya[i](v_i^+)) =
	x_{\{i+1,\ldots,n\}} \otimes x_{\emptyset} 
		+ \sum_{\substack{I', J' \\ J' \ne \emptyset}} \xi_{I',J'} x_{I'} \otimes x_{J'} \, ,
\]
for some $\xi_{I',J'}\in \C(q^{\frac{1}{2}})$. Theorem \ref{thm:KR} and equation \eqref{eq:braiding} then give that the coefficient of 
$x_{\emptyset} \ot x_{\lbrace i+1, \ldots, n\rbrace}$ in $\RM_{S, S}\circ \Ya[i](v_i^+)$ is 
\[
q^{(\mathrm{wt}(-),\mathrm{wt}(-))}(x_{\lbrace i+1, \ldots, n\rbrace}\ot x_{\emptyset}) 
	= q^{(\sum_{\ell=1}^n \epsilon_\ell/2, \sum_{\ell=1}^i\epsilon_\ell/2 - \sum_{\ell=i+1}^n\epsilon_\ell/2)} 
		= q^{-\left(\frac{n-2i}{2} \right)} \, .
\]
The result now follows from observing that the coefficient of 
$x_{\emptyset}\ot x_{\lbrace i+1, \ldots, n\rbrace}$ in $\Ya[i](v_i^+)$ is $q^{\lbrace i+1, \ldots, n\rbrace}$. 
\end{proof}

\begin{rem}
Since $V_0$ is the monoidal identity, we will later draw the map $\Ya[0]$ as a cup, so that $R_{S,S} \circ \Ya[0]$ is a curl.  
In this case we see that Lemma \ref{trivalenttwist} recovers the (negative) twist coefficient 
for the braiding from Remark \ref{rem:framing}. 
In particular, we have
\begin{equation}\label{E:spin-twist}
q^{-\frac{n}{2}}q^{-\{1, \dots, n\}} 
	= (-1)^{\binom{n+1}{2}}q^{-\frac{n(2n+1)}{2}}
	= (-1)^{(2\varpi_n, \rho^{\vee})}q^{-(\varpi_n, \varpi_n + 2\rho)},
\end{equation}
where for the type $B_n$ root system:
\begin{equation} 2 \rho = \sum_{i=1}^{n}(2(n-i)+1) \epsilon_i, \quad
\rho^{\vee} = \frac{1}{2} \sum_{i=1}^{n}(n-i+1) \epsilon_i, \quad
\varpi_n = \frac{1}{2} \sum_{i=1}^{n} \epsilon_i.
\end{equation}
\end{rem}

Next, note that the fundamental $U_q(\son)$-representations $V_1,\ldots,V_{n-1},V_n=S$ are all self-dual, 
since, for $1 \leq i \leq n$, the highest weight of $V_i^\ast$ is $-w_0(\varpi_i) = -(-\varpi_i) = \varpi_i$.
We hence can consider the compositions
\begin{equation}
	\label{eq:cupcap1}
V_0 = \C(q^{\frac{1}{2}}) \to V_i \otimes V_i^\ast \xrightarrow{\cong} V_i \otimes V_i
\quad \text{and} \quad
V_i \otimes V_i \xrightarrow{\cong} V_i^\ast \otimes V_i \to \C(q^{\frac{1}{2}}) = V_0
\end{equation}
of these isomorphisms with the canonical coevaluation and evaluation maps.
We now aim to give an explicit description of these morphisms.

For each $1 \leq i \leq n$,
fix a basis $\mathbb{B}_i$ for $V_i$ that contains our distinguished highest and lowest weight vectors $v_i^+$ and $v_i^-$.
Given $v \in \mathbb{B}_i$, we let $v^\ast$ denote the corresponding vector in the dual basis $\mathbb{B}_i^\ast$ of $V_i^\ast$.
The assignment $v_i^+\mapsto (v_i^-)^\ast$ then determines an isomorphism 
$\varphi_i \colon V_i \xrightarrow{\cong} V_i^\ast$.

\begin{example}
The isomorphisms $\varphi_1 \colon V_1 \to V_1^\ast$ and $\varphi_n \colon S \to S^\ast$
have the following explicit descriptions:
\begin{equation}\label{varphi1}
\varphi_1(a_i) = (-q^{-2})^{i-1}b_i^\ast \, , \quad  
\varphi_1(u) = (-q^{-2})^n[2]u^\ast  \, , \quad 
\varphi_1(b_i) = -(-q^{-2})^{2n-i}a_i^\ast \, ,
\end{equation}
and
\begin{equation}\label{varphin}
\varphi_n(x_I)= q^{-I}x_{I^c}^\ast \, .
\end{equation}
\end{example}

We now describe the morphisms in \eqref{eq:cupcap1}.

\begin{prop}\label{prop:cups-caps}
For each $1\leq i \leq n$,
there is a unique $U_q(\son)$-module homomorphism
\[
\cu[i] \colon \C(q^{\frac{1}{2}}) \longrightarrow V_i \otimes V_i
\]
such that
\begin{equation}
	\label{eq:cupi}
\cu[i](1) = \sum_{v \in \mathbb{B}_i} v \otimes \varphi_i^{-1}(v^\ast) \, ,
\end{equation}
and there is a unique $U_q(\son)$-module homomorphism
\[
\ca[i] \colon V_i\otimes V_i \longrightarrow \C(q^{\frac{1}{2}})
\] 
such that
\[
\ca[i](v_i^+\otimes v_i^-) = 1 \, .
\]
Consequently, these maps agree with those in \eqref{eq:cupcap1}.
\end{prop}
\begin{proof}
Our choice of Hopf algebra structure $(U_q(\son);\Delta,\mathbf{S},\epsilon)$ ensures that the canonical 
linear maps $\coev \colon \C(q^{\frac{1}{2}}) \longrightarrow V_i\otimes V_i^\ast$ 
and $\ev \colon V_i^\ast \otimes V_i \longrightarrow \C(q^{\frac{1}{2}})$ intertwine the action of $U_q(\son)$. 
Thus, $\cu[i] := (\id \otimes \varphi_i^{-1}) \circ \coev$ and 
$\ca[i] := \ev \circ (\varphi_i \otimes \id)$ are $U_q(\son)$-module homomorphisms
and are characterized by the indicated formulae.
\end{proof}

We will write the morphisms in Proposition \ref{prop:cups-caps} using 
the graphical calculus for monoidal categories as follows:
\[
\cu[i]= \begin{tikzpicture}[scale =.5, smallnodes, anchorbase]
\draw[very thick, black] (1,2.25) node[above=-2pt]{$i$} to [out=270,in=0] (.5,1.5) to [out=180,in=270] (0,2.25) node[above=-2pt]{$i$};
\end{tikzpicture} \quad , \qquad 
\ca[i]=\begin{tikzpicture}[scale =.5, smallnodes, anchorbase, yscale=-1]
\draw[very thick, black] (1,2.25) node[below=-1pt]{$i$} to [out=270,in=0] (.5,1.5) to [out=180,in=270] (0,2.25) node[below=-1pt]{$i$};
\end{tikzpicture} \, .
\]
By construction, these morphism satisfy the relations:
\begin{equation}
	\label{eq:snake}
\begin{tikzpicture}[scale=.2, anchorbase]
	\draw[very thick] (-2,-4) node[below=-2pt]{\scs$i$} to (-2,1) to [out=90,in=180] (-1,2) to [out=0,in=90] (0,1)
		to (0,-1) to [out=270,in=180] (1,-2) to [out=0,in=270] (2,-1) to (2,4);
\end{tikzpicture}
\;\; = \;\;
\begin{tikzpicture}[scale=.2, anchorbase]
	\draw[very thick, black] (0,-4) node[below=-2pt]{\scs$i$} -- (0,4);
\end{tikzpicture}
\;\; = \;\;
\begin{tikzpicture}[scale=.2, anchorbase,xscale=-1]
	\draw[very thick] (-2,-4) node[below=-2pt]{\scs$i$} to (-2,1) to [out=90,in=180] (-1,2) to [out=0,in=90] (0,1)
		to (0,-1) to [out=270,in=180] (1,-2) to [out=0,in=270] (2,-1) to (2,4);
\end{tikzpicture}
\end{equation}

\begin{rem}\label{R:coherent-self-duality}
The full monoidal subcategory of $U_q(\son)$-modules generated by the self-dual representations $V_i$ for $i=0, 1, \ldots, n$ 
is a strict pivotal category, where we use the (non-standard) pivotal structure from \cite{ST}. 
With this choice of pivotal structure, all Frobenius-Schur indicators are $+1$, and therefore morphisms in the category can be described by 
an unoriented graphical calculus, see~\cite{Sel1, Sel2}. 
It is possible to show that our choice of cup and cap morphisms ($\cu[i]$ and $\ca[i]$, respectively)
agree with any other choice of cups and caps, up to simultaneous rescaling. 
See e.g.~\cite[Remark 5.2]{BERT} for further details.
\end{rem}

\begin{example}
By \eqref{varphin}, we have 
\begin{equation}\label{spincupcapformulas}
\ca[n](x_I\ot x_{J}) = q^{-I}\delta_{J, I^c} \, , \quad 
\cu[n](1) = \sum_{I\subset \lbrace 1, \ldots, n\rbrace}q^{I}x_{I^c}\ot x_{I},
\end{equation}
and 
\begin{equation}
	\label{eq:circlen}
\begin{tikzpicture}[scale =.5, anchorbase]
	\draw[very thick] (0,0) node[left,xshift=-5pt]{\scs$n$} circle (.5);
\end{tikzpicture} :=
\ca[n]\circ \cu[n] (1) 
= \sum_{I \subset \{1,\ldots,n\}} \frac{q^I}{q^{I^c}} 
= (-1)^{\binom{n+1}{2}}\prod_{i=1}^n(q^{2i-1}+ q^{1-2i})
\end{equation}
Note also that that $\cu[n] = \Ya[0]$. 
\end{example}

Establishing the last equality in \eqref{eq:circlen} is a straightforward exercise, 
as is the following generalization.

\begin{lem}
	\label{lem:dvalue}
If $d_i:= \prod_{\ell=1}^i\dfrac{[4\ell-2]}{[2\ell-1]} = \prod_{\ell=1}^i (q^{2\ell-1}+ q^{-2\ell+ 1})$,
then 
\[
\sum_{\substack {I\cup J = \lbrace i+1, \ldots , n\rbrace \\ I\cap J 
	= \emptyset}}q^{J}q^{-I} = (-1)^{\binom{n-i+1}{2}}d_{n-i} \, . 
\]
\qed
\end{lem}

We denote the morphisms $\Ya[i] \colon V_i \to S \otimes S$ 
from Proposition \ref{prop:Y} in our graphical calculus as follows:
\[
\begin{tikzpicture}[scale =.5, smallnodes,anchorbase]
	\draw[very thick, gray] (0,0) node[above=-2pt]{$S$} to [out=270,in=150] (.5,-.75);
	\draw[very thick, gray] (1,0) node[above=-2pt]{$S$} to [out=270,in=30] (.5,-.75);
	\draw[very thick, black] (.5,-.75) to (.5,-1.5) node[below=-2pt]{$i$};
\end{tikzpicture}
:= \Ya[i] \, .
\]
Here, and in the following, we use the color {\color{gray} gray} to graphically denote the spin representation.
Composing these morphisms with (tensor products of) the cap and cup morphisms from Proposition \ref{prop:cups-caps} 
yields the following morphisms in $\Rep(U_q(\son))$.

\begin{defn}
	\label{defn:Pa}
For $0 \leq i \leq n-1$,
let $\Pa[i] \colon S\otimes S\longrightarrow V_i$ be the $U_q(\son)$-module morphisms given 
in terms of the graphical calculus as follows:
\begin{equation}
	\label{eq:Pa}
\Pa[i]=\begin{tikzpicture}[scale =.5, smallnodes,anchorbase,rotate=180]
	\draw[very thick, gray] (0,0) node[below]{$S$} to [out=270,in=150] (.5,-.75);
	\draw[very thick, gray] (1,0) node[below]{$S$} to [out=270,in=30] (.5,-.75);
	\draw[very thick, black] (.5,-.75) to (.5,-1.5) node[above=-2pt]{$i$};
\end{tikzpicture}
:=
\begin{tikzpicture}[scale =.35, smallnodes,anchorbase]
	\draw[very thick, gray] (3,-2.5) node[below]{$S$} to (3,0) to [out=90,in=0] (1.5,1.5) to [out=180,in=90] (0,0) to [out=270,in=150] (.5,-.75);
	\draw[very thick, gray] (2,-2.5) node[below]{$S$} to (2,0) to [out=90,in=0] (1.5,.5) to [out=180,in=90] (1,0) to [out=270,in=30] (.5,-.75);
	\draw[very thick, black] (.5,-.75) to (.5,-1.5) to [out=270,in=0] (-.125,-2) to [out=180,in=270] (-.75,-1.5) to (-.75,2) node[above=-2pt]{$i$};
\end{tikzpicture} \, .
\end{equation}
\end{defn}

In particular, we have that $\Pa[0] = \ca[n]$.

\begin{lem}\label{bigon}
The following equalities hold between morphisms in $\Rep(U_q(\son))$:
\begin{equation}\label{E:left-twist-vs-right-twist}
\begin{tikzpicture}[scale =.35, smallnodes,anchorbase]
	\draw[very thick, gray] (-3,2.5) node[above=-2pt]{$S$} to (-3,0) to [out=270,in=180] (-1.5,-1.5) to [out=0,in=270] (0,0) to [out=90,in=330] (-.5,.75);
	\draw[very thick, gray] (-2,2.5) node[above=-2pt]{$S$} to (-2,0) to [out=270,in=180] (-1.5,-.5) to [out=0,in=270] (-1,0) to [out=90,in=210] (-.5,.75);
	\draw[very thick, black] (-.5,.75) to (-.5,1.5) to [out=90,in=180] (.125,2) to [out=0,in=90] (.75,1.5) to (.75,-2) node[below=-2pt]{$i$};
\end{tikzpicture}
=
\begin{tikzpicture}[scale =.5, smallnodes,anchorbase,rotate=180]
	\draw[very thick, gray] (0,0) node[above=-2pt]{$S$} to [out=90,in=210] (.5,.75);
	\draw[very thick, gray] (1,0) node[above=-2pt]{$S$} to [out=90,in=330] (.5,.75);
	\draw[very thick, black] (.5,.75) to (.5,1.5) node[below=-2pt]{$i$};
\end{tikzpicture}
=
\begin{tikzpicture}[scale =.35, smallnodes,anchorbase]
	\draw[very thick, gray] (3,2.5) node[above=-2pt]{$S$} to (3,0) to [out=270,in=0] (1.5,-1.5) to [out=180,in=270] (0,0) to [out=90,in=210] (.5,.75);
	\draw[very thick, gray] (2,2.5) node[above=-2pt]{$S$} to (2,0) to [out=270,in=0] (1.5,-.5) to [out=180,in=270] (1,0) to [out=90,in=330] (.5,.75);
	\draw[very thick, black] (.5,.75) to (.5,1.5) to [out=90,in=0] (-.125,2) to [out=180,in=90] (-.75,1.5) to (-.75,-2) node[below=-2pt]{$i$};
\end{tikzpicture} \, ,
\end{equation}
\begin{equation}\label{bigonequation}
\begin{tikzpicture}[scale=.175,tinynodes, anchorbase]
	\draw [very thick] (0,.75) to (0,2.5) node[above,yshift=-3pt]{$n{-}i$};
	\draw [very thick,gray] (0,-2.75) to [out=30,in=330] node[right,xshift=-2pt]{$S$} (0,.75);
	\draw [very thick,gray] (0,-2.75) to [out=150,in=210] node[left,xshift=2pt]{$S$} (0,.75);
	\draw [very thick] (0,-4.5) node[below,yshift=2pt]{$n{-}i$} to (0,-2.75);
\end{tikzpicture} :=
\Pa[n-i]\circ \Ya[n-i] = (-1)^{\binom{i+1}{2}}d_i \cdot \id_{V_{n{-}i}}
= (-1)^{\binom{i+1}{2}} d_i
\begin{tikzpicture}[scale=.175, tinynodes, anchorbase]
	\draw [very thick] (0,-4.5) node[below,yshift=2pt]{$n{-}i$} to (0,2.5);
\end{tikzpicture} \, ,
\end{equation}
and
\begin{equation}\label{bigoniszero}
\begin{tikzpicture}[scale=.175,tinynodes, anchorbase]
	\draw [very thick] (0,.75) to (0,2.5) node[above,yshift=-3pt]{$n-i$};
	\draw [very thick,gray] (0,-2.75) to [out=30,in=330] node[right,xshift=-2pt]{$S$} (0,.75);
	\draw [very thick,gray] (0,-2.75) to [out=150,in=210] node[left,xshift=2pt]{$S$} (0,.75);
	\draw [very thick] (0,-4.5) node[below,yshift=2pt]{$n-j$} to (0,-2.75);
\end{tikzpicture} :=
\Pa[n-i]\circ \Ya[n-j] = 0 \, , \qquad \text{for $i\ne j$.}
\end{equation}
\end{lem}

\begin{proof}
Equation \eqref{E:left-twist-vs-right-twist} is a consequence of our graphical calculus describing a pivotal category, 
see Remark \ref{R:coherent-self-duality}. When $i\ne j$,  $V(\varpi_i)$ and $V(\varpi_j)$ are non-isomorphic irreducible representations, 
so equation \eqref{bigoniszero} follows from Schur's Lemma.

To verify equation \eqref{bigonequation}, 
note that, since $\Hom_{U_q(\son)} \big( V_i \otimes V_i, \C(q^{\frac{1}{2}}) \big)$ is $1$-dimensional,
there is some scalar $\chi \in \C(q^{\frac{1}{2}})$ so that
\[
\begin{tikzpicture}[scale =.35, smallnodes,anchorbase,rotate=180]
	\draw[very thick, gray] (2.5,.75) to [out=330,in=90] (3,0) to [out=270,in=0] (1.5,-1.5) 
		node[above=-2pt]{$S$} to [out=180,in=270] 
			(0,0) to [out=90,in=210] (.5,.75);
	\draw[very thick, gray] (2.5,.75) to [out=210,in=90] (2,0) to [out=270,in=0] (1.5,-.5) 
		node[above=-2pt]{$S$} to [out=180,in=270]
			(1,0) to [out=90,in=330] (.5,.75);
	\draw[very thick, black] (.5,.75) to (.5,1.5) node[below=-2pt]{$i$};
	\draw[very thick, black] (2.5,.75) to (2.5,1.5) node[below=-2pt]{$i$};
\end{tikzpicture}
=:
\ca[n]\circ (\id_S\ot \ca[n]\ot \id_S)\circ (\Ya[i]\ot \Ya[i]) = \chi \cdot \ca[i] \, .
\]
Since $\ca[n](v_i^+\ot v_i^-)= 1$, it follows that 
\begin{align*}
\chi &=\ca[n]\circ (\id_S\ot \ca[n]\ot \id_S)\circ (\Ya[i]\ot \Ya[i])(v_i^+\ot v_i^-) \\
&= \sum_{\substack{I_1\cup J_1 = I_2\cup J_2 = \lbrace i+1, \ldots , n\rbrace \\ I_1\cap J_1 = I_2\cap J_2 = \emptyset}}q^{J_1}q^{J_2}\ca[n]\circ (\id_S\ot \ca[n]\ot \id_S)\left(x_{I_1}\ot x_{J_1}\ot x_{J_2^c} \ot x_{I_2^c}\right) \\
&= \sum_{\substack{I_1\cup J_1 = I_2\cup J_2 = \lbrace i+1, \ldots , n\rbrace \\ I_1\cap J_1 = I_2\cap J_2 = \emptyset}}q^{J_1}q^{J_2}\ca[n](x_{J_1}\ot x_{J_2^c})\ca[n](x_{I_1}\ot x_{I_2^c}) \\
&= \sum_{\substack{I_1\cup J_1 = I_2\cup J_2 = \lbrace i+1, \ldots , n\rbrace \\ I_1\cap J_1 = I_2\cap J_2 = \emptyset}}q^{J_1}q^{J_2}q^{-J_1}\delta_{J_1, J_2}q^{-I_1}\delta_{I_1, I_2} \\
&= \sum_{\substack {I\cup J = \lbrace i+1, \ldots , n\rbrace \\ I\cap J = \emptyset}}q^{J}q^{-I} 
=(-1)^{\binom{n-i+1}{2}}d_{n-i}
\end{align*}
with the last equality holding by Lemma \ref{lem:dvalue}.
The result then follows from Definition \ref{defn:Pa} using \eqref{eq:snake}.
\end{proof}

\begin{prop}
	\label{prop:Ibasis}
Set $\II^{(i)}:= \Ya[i]\circ \Pa[i]$, for $i=0, 1, \ldots, n-1$.
For $i,j\in \lbrace 0, \ldots, n-1\rbrace$, we have that
\begin{equation}\label{bigonschur}
\II^{(i)}\II^{(j)}= (-1)^{\binom{n-i+1}{2}}d_{n-i}\delta_{i,j}\cdot \II^{(i)}.
\end{equation}
If we further let $\II^{(n)}:= \id_{S \otimes S}$,
then $\{ \II^{(i)} \}_{i=0}^{n}$ is a basis for $\End_{U_q(\son)}(S \otimes S)$.
\end{prop}

\begin{proof}
Equation \eqref{bigonschur} follows from equation \eqref{bigonequation} and equation \eqref{bigoniszero}. 
The result then follows from the decomposition \eqref{eq:StSdecomp} and another application of Schur's lemma.
\end{proof}

We conclude this section with a description of the braiding $R_{S,S} \in \End_{U_q(\son)}(S \otimes S)$ 
in terms of the basis from Proposition \ref{prop:Ibasis}.

\begin{prop}
	\label{prop:braidI}
For $1 \leq i \leq n$, set
\begin{equation}
	\label{eq:braidcoeff}
b_{n-i} = \frac{q^{\frac{n}{2}}}{d_i} \left((q^{-2})^{\binom{i+1}{2}} - (-1)^{\binom{i+1}{2}}\right) \, .
\end{equation}
The braiding on $S \otimes S$ is given by
\[
\RM_{S, S}= q^{\frac{n}{2}}\id_{S\ot S} + \sum_{i=1}^n b_{n-i} \II^{(n-i)} \, .
\]
\end{prop}
\begin{proof}
By Proposition \ref{prop:Ibasis}, there exist scalars $b_j \in \C(q^{\frac{1}{2}})$ for $0 \leq j \leq n$
so that 
\[
\RM_{S, S}= b_n \id_{S\ot S} + \sum_{i=1}^n b_{n-i} \II^{(n-i)} \, .
\]
Plugging the highest weight vector $x_\emptyset \otimes x_\emptyset \in S \otimes S$ 
into $R_{S,S}$ and using Proposition \ref{prop:R=xy} gives that $b_n= q^{\frac{n}{2}}$. 
For the remaining coefficients, Lemmata \ref{trivalenttwist} and \ref{bigon} give
\[
q^{\frac{n-2i}{2}}q^{-\lbrace n-i+1, \ldots, n\rbrace} \Ya[n-i]
=\RM_{S, S} \circ \Ya[n-i] = q^{\frac{n}{2}} \Ya[n-i] + (-1)^{\binom{i+1}{2}} d_i b_{n-i} \Ya[n-i] \, ,
\]
thus
\[
b_{n-i} = \frac{(-1)^{\binom{i+1}{2}}q^{\frac{n}{2}}}{d_i}\left(q^{-i}q^{-\lbrace n-i+1, \ldots, n\rbrace} - 1\right) 
	= \frac{q^{\frac{n}{2}}}{d_i}\left((q^{-2})^{\binom{i+1}{2}} - (-1)^{\binom{i+1}{2}}\right) \, . \qedhere
\]
\end{proof}

\subsection{$\hh$- and $\xx$-morphisms}

In this section, we consider two additional bases $\{\hh^{(i)}\}_{i=0}^n$ and $\{\xx^{(i)}\}_{i=0}^n$
for the algebra $\End_{U_q(\son)}(S \otimes S$).
The first is obtained in a straightforward way: 
diagrammatically, it is the rotation of the basis from Proposition \ref{prop:Ibasis}. 
The latter basis is a novel construction of this paper and is more subtle. 
Its construction was motivated by our categorified considerations appearing later, 
and is the ``canonical basis'' for this algebra, in a certain sense.

To begin, define $\hh^{(i)}$ using graphical calculus as follows:
\begin{equation}
	\label{eq:Hi}
\hh^{(i)} =
\begin{tikzpicture}[scale=.4, rotate=90, tinynodes, anchorbase]
	\draw[very thick,gray] (-1,0) node[below,yshift=2pt,xshift=2pt]{$S$} to (0,1);
	\draw[very thick,gray] (1,0) node[above,yshift=-4pt,xshift=2pt]{$S$} to (0,1);
	\draw[very thick,gray] (0,2.5) to (-1,3.5) node[below,yshift=2pt,xshift=-2pt]{$S$};
	\draw[very thick,gray] (0,2.5) to (1,3.5) node[above,yshift=-4pt,xshift=-2pt]{$S$};
	\draw[very thick] (0,1) to node[below,yshift=2pt]{$i$} (0,2.5);
\end{tikzpicture}
:=
\begin{tikzpicture}[scale=.35, tinynodes, anchorbase]
	\draw[very thick,gray] (-1,-1) node[below=-2pt]{$S$} to [out=90,in=210] (0,1);
	\draw[very thick,gray] (2,4.5) node[above=-2pt]{$S$} to [out=270,in=0] (1.25,.25) to [out=180,in=330] (0,1);
	\draw[very thick,gray] (0,2.5) to [out=150,in=0] (-1.25,3.25) to [out=180,in=90] (-2,-1) node[below=-2pt]{$S$};
	\draw[very thick,gray] (0,2.5) to [out=30,in=270] (1,4.5) node[above=-2pt]{$S$};
	\draw[very thick] (0,1) to node[right,xshift=-2pt]{$i$} (0,2.5);
\end{tikzpicture}
\end{equation}
We begin by computing the values of the ``triangles'':
\[
\begin{tikzpicture}[scale=.3,smallnodes,anchorbase,rotate=180]
	\draw[very thick] (-1,0) to node[above=-1pt]{$1$} (1,0);
	\draw[very thick,gray] (-1,0) to (0,1.732);
	\draw[very thick,gray] (1,0) to (0,1.732);
	\draw[very thick] (0,1.732) to (0,3.232) node[below=-1pt]{$i$};
	\draw[very thick,gray] (-2.3,-.75) node[above=-2pt,xshift=2pt]{$S$} to (-1,0);
	\draw[very thick,gray] (2.3,-.75) node[above=-2pt,xshift=-2pt]{$S$} to (1,0);
\end{tikzpicture}
:= \hh^{(1)} \circ \Ya[i] \, .
\]

\begin{lem}\label{SS1triangle}
For $0 \leq i \leq n-1$,
\[
\hh^{(1)}\circ \Ya[i] = (-1)^{n-i}\frac{[2(n-i)+1]}{[2]} \Ya[i]
\]
\end{lem}

\begin{proof}
Again, the decomposition \ref{eq:StSdecomp} implies that
there is some scalar $\chi$ so that $\hh^{(1)}\circ \Ya[i]= \chi \Ya[i]$. 
Let 
\[
\pi_{\lbrace i+1, \ldots, n\rbrace\ot \emptyset} \in \End_{\C(q^{1/2})}(S \otimes S)
\]
denote the \emph{linear operator} which projects to $x_{\lbrace i+1, \ldots, n\rbrace}\ot x_{\emptyset}$ 
with respect to the basis $\lbrace x_I\ot x_J\rbrace_{I,J \subset \{1,\ldots,n\}}$ of $S \otimes S$.
(These are not $U_q(\son)$-module homomorphisms.)
It follows from \eqref{highestwti} that
\[
\chi \cdot q^{\emptyset}\cdot x_{\lbrace i+1, \ldots, n\rbrace}\ot x_{\emptyset}
	= \pi_{\lbrace i+1, \ldots, n\rbrace\ot \emptyset} \circ \chi \Ya[i] (v_i^+)
	= \pi_{\lbrace i+1, \ldots, n\rbrace \ot \emptyset}\circ \hh^{(1)}\circ \Ya[i](v_i^+) \, .
\]
Combining \eqref{eq:Hi} with \eqref{eq:Pa} gives
\[
\hh^{(1)}\circ \Ya[i]= (\ca[n]\ot \id_{S\otimes S}\ot \ca[n])\circ (\id_S\ot\Ya[1]\ot\Ya[1]\ot \id_S)\circ (\id_S\ot \cu[1] \ot\id_S)\circ \Ya[i] \, ,
\]
so \eqref{highestwti} and \eqref{eq:cupi} imply that
$\pi_{ \lbrace i+1, \ldots, n\rbrace \ot \emptyset}\circ \hh^{(1)}\circ \Ya[i](v_i^+)$ is equal to
\begin{equation}
	\label{eq:HYmess}
\sum_{\substack{I\subset\lbrace i+1, \ldots, n\rbrace \\ v\in \lbrace a_i, u, b_j\rbrace}} 
	q^I\pi_{\lbrace i+1, \ldots, n\rbrace\ot \emptyset}\circ(\ca[n]\ot \id_{S\ot S}\ot \ca[n])
		(x_{\lbrace i+1, \ldots, n\rbrace \smallsetminus I}\ot \Ya[1](v)\ot \Ya[1](\varphi_1^{-1}(v^\ast))\ot x_I) \, .
\end{equation}
We now work to simplify this expression.

For $K \subset \{1,\ldots,n\}$, similarly let $\pi_{K} \in \End_{\C(q^{1/2})}(S)$ denote 
the linear operator projecting to $x_K$ with respect to the basis $\lbrace x_I\rbrace_{I \subset \{1,\ldots,n\}}$. 
Let $\ell \in \lbrace 1, \ldots, n\rbrace$ and $I \subset \lbrace i+1, \ldots, n\rbrace$.
Using Example \ref{ex:Y1}, along with equation \eqref{spincupcapformulas},
we see that
\[
\pi_{\emptyset}\circ (\id_S\ot \ca[n])(\Ya[1](b_\ell) \ot x_I) = 0 \, ,
\]
and, unless $I = \{ \ell\}$, we also have
\[
\pi_{\emptyset}\circ (\id_S\ot \ca[n])(\Ya[1](a_\ell) \ot x_I) = 0 \, .
\]
Similarly, one finds that
\[
\pi_{\emptyset}\circ (\id_S\ot \ca[n])(\Ya[1](u) \ot x_I) \ne 0
\]
if and only if $I = \emptyset$ and 
\[
\pi_{\lbrace i+1, \ldots, n\rbrace} \circ (\ca[n]\ot \id_S)(x_I\ot \Ya[1](u)) \ne 0
\]
if and only if $I = \lbrace i+1, \ldots, n\rbrace$. 

These computations, together with equation \eqref{varphi1}, imply that \eqref{eq:HYmess} simplifies to
\begin{multline*}
\sum_{\ell=i+1}^n (-1)^{\ell-1} (q^2)^{\ell-1} q^{\lbrace \ell\rbrace}
	\pi_{\lbrace i+1, \ldots, n\rbrace\ot \emptyset} \circ (\ca[n] \ot \id_{S\ot S} \ot \ca[n]) 
		(x_{\lbrace i+1, \ldots, n\rbrace - \lbrace \ell \rbrace} \ot \Ya[1](b_\ell) \ot \Ya[1](a_\ell) \ot x_{\lbrace \ell \rbrace}) \\
+\frac{(-1)^n(q^2)^nq^{\emptyset}}{[2]}\pi_{\lbrace i+1, \ldots, n\rbrace\ot \emptyset}
	\circ(\ca[n]\ot \id_{S\ot S}\ot \ca[n])(x_{\lbrace i+1, \ldots, n\rbrace}\ot \Ya[1](u)\ot \Ya[1](u)\ot x_{\emptyset}).
\end{multline*}
Finally, using our explicit formulae for $\Ya[1]$ and $\ca[n]$, we conclude that
\begin{align*}
\chi &= \sum_{\ell=i+1}^n (-1)^{\ell-1}(q^2)^{\ell-1}q^{\lbrace \ell\rbrace}q^{\lbrace 2, \ldots, n\rbrace}q^{\lbrace i+2, \ldots, n\rbrace}
		q^{-\lbrace i+1\ldots, \hat{\ell}, \ldots, n\rbrace}q^{-\lbrace 1, \ldots, \hat{\ell}, \ldots, n\rbrace} \\
&\qquad + (-1)^n\frac{(q^2)^nq^{\emptyset}q^{-1}q^{\lbrace i+2, \ldots, n\rbrace}q^{-1}
	q^{\lbrace 2, \ldots, n\rbrace}q^{-\lbrace i+1, \ldots, n\rbrace}q^{-\lbrace 1, \ldots, n\rbrace}}{[2]}\\
&= \frac{q^{\lbrace 2, \ldots, n\rbrace}q^{\lbrace i+2, \ldots, n\rbrace}}{q^{\lbrace i+1, \ldots, n\rbrace}q^{\lbrace 1, \ldots, n\rbrace}}
	\left(\sum_{\ell=i+1}^n(-1)^{\ell-1}(q^2)^{\ell-1}q^{\lbrace \ell\rbrace}q^{\lbrace \ell\rbrace}q^{\lbrace \ell\rbrace} + (-1)^n \frac{(q^2)^nq^{-2}}{[2]}\right) \\
&= (-1)^n q^{-\lbrace 1, i+1 \rbrace} \left( \sum_{\ell=i+1}^n q^{6n-4\ell+1}+ \frac{q^{2n-2}}{[2]} \right) \\
&= \frac{(-1)^{n+i}}{[2]} \left( \sum_{\ell=i+1}^n (q^{2n+2i-4\ell+4}+ q^{2n+2i-4\ell+2})+ q^{-2n+2i} \right) \\
&= (-1)^{n-i} \frac{[2(n-i)+1]}{[2]} \, .  \qedhere
\end{align*}
\end{proof}

Using this, we now express the morphism $\hh^{(1)}$ in terms of the basis from Proposition \ref{prop:Ibasis}. 
Recall the devil's product from Definition \ref{def:devil}, 
and the notation $\II^{(i)}:= \Ya[i]\circ \Pa[i]$ and $\II^{(n)}= \id_{S\otimes S}$ introduced in Proposition \ref{prop:Ibasis},
both will be used extensively for the remainder of this section.

\begin{lem}\label{lem:H1reln}
We have that
\[
\hh^{(1)}= \dfrac{1}{[2]}\id_{S\ot S} + \sum_{i=1}^{n} (-1)^{\binom{i}{2}}\dfrac{``[i][i+1]"}{d_i} \II^{(n-i)} \, .
\]
\end{lem}
\begin{proof}
By Proposition \ref{prop:Ibasis}, there are scalars $\chi, \lambda_{0}, \ldots, \lambda_{n-1}\in \C(q^{\frac{1}{2}})$ such that
\begin{equation}\label{H1unknowns}
\hh^{(1)}= \chi \id_{S\ot S} + \sum_{\ell=0}^{n-1} \lambda_{\ell} \II^{(\ell)}.
\end{equation}
Since $\II^{(\ell)}(x_{\emptyset}\ot x_{\emptyset}) = 0$ for $0 \leq \ell \leq n-1$, 
a similar argument to the one given in the proof of Lemma \ref{SS1triangle} implies that $\chi = 1/[2]$. 
Next, composing \eqref{H1unknowns} with $\Pa[n-i]$ 
and applying Lemma \ref{SS1triangle},
we obtain
\[
(-1)^k\frac{[2(n-(n-i)) + 1]}{[2]} = \frac{1}{[2]} + \lambda_{n-i}(-1)^{\binom{i+1}{2}}d_i
\]
and so
\[
\lambda_{n-i} = (-1)^{\binom{i+1}{2}} (-1)^i \left( \frac{[2i+1]+(-1)^{i+1}}{[2]d_i} \right) \, .
\]
The result then follows from $m=i$, $n=i+1$ case of the identity
\begin{equation}
	\label{eq:2devil}
[2]``[m][n]" = [n+m] + (-1)^{m-1} [n-m]
\end{equation}
(which holds for $m \leq n$)
and the identity
\[
(-1)^{\binom{i+1}{2}} (-1)^i = (-1)^{\binom{i}{2}} \, . \qedhere
\]
\end{proof}

We now introduce the distinguished elements of $\End_{U_q(\son)}(S \otimes S)$ 
appearing in Theorem \ref{thm:X}.

\begin{defn}
	\label{D:xxk-defn}
Let 
\begin{equation}
	\label{eq:x1defn}
\xx:= \xx^{(1)} := \hh^{(1)}- \dfrac{1}{[2]} \II^{(n)} = \II^{(n-1)} + \sum_{\ell=2}^{n}(-1)^{\binom{\ell}{2}} \dfrac{``[\ell][\ell+1]"}{d_\ell} \II^{(n-\ell)} \, .
\end{equation}
More generally, let $\xx^{(0)}:=\id_{S\otimes S}$ and for $0 \leq i \leq n-1$, set
\begin{equation}
	\label{eq:xidefn}
\xx^{(i+1)} := \dfrac{(-1)^i}{``[i+1]^2"}\left(\xx^{(i)}\xx - (-1)^{i} ``[i][i+1]"\xx^{(i)}\right) \, .
\end{equation}
\end{defn}

By design we have that
\begin{equation}
	\label{E:xxk-recursion}
\xx^{(i)}\xx = (-1)^i``[i+1]^2"\xx^{(i+1)} + (-1)^i``[i][i+1]"\xx^{(i)}.
\end{equation}
Further, \eqref{eq:xidefn} implies that for $1 \leq i \leq n$, 
\begin{equation}
	\label{E:xxk-expanded}
\xx^{(i)} = (-1)^{\binom{i}{2}}\dfrac{(\xx+(-1)^i``[i-1][i]")\cdot (\xx+(-1)^{i-1}``[i-2][i-1]")\cdots (\xx+(-1)``[0][1]")}{``[i]^2"\cdot ``[i-1]^2"\cdots ``[1]^2"} \, .
\end{equation}
We now give an explicit unitriangular change of basis from $\{\xx^{(i)}\}_{i=0}^n$ to $\{ \II^{(n-i)} \}_{i=0}^n$, 
generalizing \eqref{eq:x1defn}.

\begin{prop}\label{itoxchangeofbasis}
For $1 \leq i \leq n$,
\begin{equation}
	\label{eq:ItoX}
\xx^{(i)} = \sum_{\ell=i}^n \prescript{n-i}{}\lambda_{n-\ell} \II^{(n-\ell)}
\end{equation}
where, for $i \leq \ell \leq n$,
\begin{equation}\label{lambdacoeff}
\prescript{n-i}{}\lambda_{n-\ell}:= \dfrac{(-1)^{\binom{\ell-i+1}{2}}}{d_\ell} \cdot \prod_{t=1}^i\dfrac{``[\ell+1- t][\ell+t]"}{``[t]^2"} \, .
\end{equation}
Further, $\prescript{n-i}{}\lambda_{n-i} = 1$ and thus $\{ \xx^{(i)} \}_{i=0}^n$ is a basis for $\End_{U_q(\son)}(S\otimes S)$.
\end{prop}
\begin{proof}
Note that $\prescript{n-1}{}\lambda_{n-\ell} = (-1)^{\binom{\ell}{2}}\frac{``[\ell][\ell+1]"}{d_\ell}$ 
so $\prescript{n-1}{}\lambda_{n-1}=1$ and thus \eqref{eq:x1defn} implies the claim when $i=1$. 
We now compute inductively that
{\allowdisplaybreaks
\begin{align*}
\xx^{(i+1)} &\stackrel{{\eqref{eq:xidefn}}}{=} \dfrac{(-1)^i}{``[i+1]^2"}\left(\xx^{(i)}\xx - (-1)^{i} ``[i][i+1]"\xx^{(i)}\right) \\
&\stackrel{{\eqref{eq:ItoX}}}{=} \dfrac{(-1)^i}{``[i+1]^2"} 
	\Big( \sum_{\ell=i}^n \prescript{n-i}{}\lambda_{n-\ell} \II^{(n-\ell)} \Big)
		\Big( \Big( \sum_{\ell=1}^n(-1)^{\binom{\ell}{2}}\frac{``[\ell][\ell+1]"}{d_\ell} \II^{(n-\ell)} \Big) 
			- (-1)^{i} ``[i][i+1]" \Big) \\
&\stackrel{{\eqref{bigonschur}}}{=} \dfrac{(-1)^i}{``[i+1]^2"} \sum_{\ell=i}^n \prescript{n-i}{}\lambda_{n-\ell} 
	\big( (-1)^\ell ``[\ell][\ell+1]" - (-1)^{i} ``[i][i+1]" \big) \II^{(n-\ell)} \\
&= \dfrac{(-1)^i}{``[i+1]^2"} \sum_{\ell=i+1}^n \prescript{n-i}{}\lambda_{n-\ell} 
	\big( (-1)^\ell ``[\ell][\ell+1]" - (-1)^{i} ``[i][i+1]" \big) \II^{(n-\ell)} \\
&= \dfrac{1}{``[i+1]^2"} \sum_{s=1}^{n-i} \prescript{n-i}{}\lambda_{n-i-s} 
	\big( (-1)^{s} ``[i+s][i+s+1]" - ``[i][i+1]" \big) \II^{(n-i-s)} \\
&= \dfrac{1}{``[i+1]^2"} \sum_{s=1}^{n-i} \prescript{n-i}{}\lambda_{n-i-s} 
	\Big( \sum_{j=0}^{i+s-1} (-1)^{j-s} [2i-2(j-s)] - \sum_{j=0}^{i-1} (-1)^j [2i-2j] \Big) \II^{(n-i-s)} \\
&= \dfrac{1}{``[i+1]^2"} \sum_{s=1}^{n-i} \prescript{n-i}{}\lambda_{n-i-s} 
	\Big( (-1)^s \sum_{j=0}^{s-1} (-1)^{j} [2i+s+1+s-2j-1] \Big) \II^{(n-i-s)} \\
&= \dfrac{1}{``[i+1]^2"} \sum_{s=1}^{n-i} \prescript{n-i}{}\lambda_{n-i-s}  (-1)^s ``[s][2i+s+1]" \II^{(n-i-s)} \\
&= \sum_{\ell=i+1}^{n} \prescript{n-i}{}\lambda_{n-\ell}  (-1)^{\ell-i} \frac{``[\ell-i][\ell+i+1]" }{``[i+1]^2"} \II^{(n-\ell)} \\
&= \sum_{\ell=i+1}^{n}  \dfrac{(-1)^{\binom{\ell-i+1}{2}}}{d_\ell} \cdot \prod_{t=1}^i\dfrac{``[\ell+1- t][\ell+t]"}{``[t]^2"} 
	(-1)^{\ell-i} \frac{``[\ell-i][\ell+i+1]" }{``[i+1]^2"} \II^{(n-\ell)} \\
&= \sum_{\ell=i+1}^{n}  \dfrac{(-1)^{\binom{\ell-(i+1)-1}{2}}}{d_\ell} \cdot \prod_{t=1}^{i+1} 
	\dfrac{``[\ell+1- t][\ell+t]"}{``[t]^2"} \II^{(n-\ell)}
\end{align*}
}
which establishes \eqref{eq:ItoX}.

Next, multiplying both numerator and denominator of \eqref{lambdacoeff} by $[2]^i$ 
(in the $\ell = i$ case) and using \eqref{eq:2devil}, we compute
{\allowdisplaybreaks
\begin{align*}
\prescript{n-i}{}\lambda_{n-i} &= \dfrac{1}{d_i}\prod_{t=1}^i \dfrac{[2i+1]+(-1)^{i-t}[2t-1]}{[2t]} \\
	&=  \dfrac{1}{d_i \prod_{t=1}^i [2t]} \prod_{\substack{t=1 \\ t \equiv i \, \ourmod 2}}^i \big([2i+1]+[2t-1]\big)
		\prod_{\substack{t=1 \\ t \equiv i+1 \, \ourmod 2}}^i \big([2i+1]-[2t-1]\big) \\
&= \frac{\Big( \frac{[2i][2]}{[1]}\frac{[2i-2][6]}{[3]}\frac{[2i-4][10]}{[5]}\frac{[2i-6][14]}{[7]}\cdots\Big)
	\Big(\frac{[4i-2][2]}{[2i-1]}\frac{[4i-6][4]}{[2i-3]}\frac{[4i-10][6]}{[2i-5]}\frac{[4i-14][8]}{[2i-7]}\cdots\Big)}
		{d_i \prod_{t=1}^i[2t]} \\
&=\dfrac{1}{d_i}\prod_{t=1}^i \dfrac{[4t-2]}{[2t-1]} = 1 \, .
\end{align*}}

Finally, recall that Proposition \ref{prop:Ibasis} gives that $\{ \II^{(i)} \}_{i=0}^n$ is a basis for $\End_{U_q(\son)}(S\otimes S)$. 
Since there is a unitriangular matrix relating this basis and $\{ \xx^{(i)} \}_{i=0}^n$, the latter is also a basis.
\end{proof}

\begin{rem}
	\label{rem:Xtoohigh}
In Definition \ref{D:xxk-defn}, we defined $\xx^{(i)}$ for $0 \leq i \leq n$;
however, we could have used the same recursion to define $\xx^{(i)}$ for all $i \ge 0$. 
In fact, nothing is gained, since for all $i > n$, we find that $\xx^{(i)}=0$.
To see this, it suffices to show that $\xx^{(n+1)}=0$. 
For this, Proposition \ref{itoxchangeofbasis} gives that $\xx^{(n)}= \II^{(0)}$ and thus
\[
\xx^{(n)} \xx = \II^{(0)} \xx \stackrel{\eqref{eq:x1defn},\eqref{bigonequation}}{=}
(-1)^{\binom{n}{2}}(-1)^{\binom{n+1}{2}}``[n][n+1]"\II^{(0)}
	= (-1)^{n}``[n][n+1]" \xx^{(n)}.
\]
and therefore $\xx^{(n+1)} := \xx^{(n)}\xx - (-1)^{n}``[n][n+1]"\xx^{(n)}= 0$.
In particular, this implies that when $i=n+1$, 
we can interpret the right-hand side of \eqref{E:xxk-expanded}
as the minimal polynomial of $\xx$. 
\end{rem}

We now arrive at the main result of this section, 
the description given in Theorem \ref{thm:X} 
of the braiding on $\End_{U_q(\son)}(S \otimes S)$ in terms of the basis $\{ \xx^{(i)} \}_{i=0}^n$.
(Note that we have already established the rest of this theorem.)
For this, we first need a technical result whose proof we relegate to the appendix.

\begin{lem}\label{L:rho-recurrence}
Fix $\ell \ge 1$, and set $\rho_{\ell+1}^{(\ell)}:=1$. For $1 \le t \le \ell$ consider the recurrence relation
\begin{equation}\label{E:rho-recurrence}
\rho_{t}^{(\ell)}:= (-1)^{\binom{\ell+2-t}{2}} + q^{-1}\dfrac{``[\ell+1-t][\ell+t]"}{``[t]^2"}\rho_{t+1}^{(\ell)} \, .
\end{equation}
Then $\rho_{1}^{(\ell)} = (q^{-2})^{\binom{\ell+1}{2}}$.
\end{lem}

\begin{proof}
See Appendix \ref{S:combinatorial-identities}.
\end{proof}

\begin{prop}\label{P:decat-spin-rickard}
The braiding $R_{S,S} \in \End_{U_q(\son)}(S \otimes S)$ is given in the basis $\{ \xx^{(i)} \}_{i=0}^n$
as:
\begin{equation}
	\label{eq:Xbraiding}
\RM_{S, S} = q^{\frac{n}{2}} \sum_{i=0}^n q^{-i}\xx^{(i)} \, .
\end{equation}
\end{prop}
\begin{proof}
Proposition \ref{itoxchangeofbasis} gives that
\begin{align*}
q^{\frac{n}{2}} \sum_{i=0}^n q^{-i}\xx^{(i)} 
&= q^{\frac{n}{2}} \id_{S\ot S} + q^{\frac{n}{2}} \sum_{i=1}^n q^{-i} \sum_{\ell=i}^n  \prescript{n-i}{}\lambda_{n-\ell} \II^{(n-\ell)} \\
&= q^{\frac{n}{2}}\id_{S\ot S} +  \sum_{\ell=1}^n q^{\frac{n}{2}} \left( \sum_{i=1}^\ell q^{-i} (\prescript{n-i}{}\lambda_{n-\ell}) \right) \II^{(n-\ell)} \, .
\end{align*}
The result then follows from Proposition \ref{prop:braidI} 
and the following calculation employing Lemma \ref{L:rho-recurrence}:
\begin{align*}
q^{\frac{n}{2}} \sum_{i=1}^\ell q^{-i} (\prescript{n-i}{}\lambda_{n-\ell})
&= \frac{q^{\frac{n}{2}}}{d_\ell} \sum_{i=1}^\ell q^{-i}(-1)^{\binom{\ell-i+1}{2}} 
	\prod_{t=1}^i \frac{``[\ell+1-t][\ell+t]"}{``[t]^2"} \\
&= \frac{q^{\frac{n}{2}}}{d_\ell} q^{-1} \frac{``[\ell][\ell+1]"}{``[1]^2"} \rho_{2}^{(\ell)} \\
&= \frac{q^{\frac{n}{2}}}{d_\ell} \left(\rho_{1}^{(\ell)} - (-1)^{\binom{\ell+1}{2}}\right) \\
&= \frac{q^{\frac{n}{2}}}{d_\ell} \left( (q^{-2})^{\binom{\ell+1}{2}} - (-1)^{\binom{\ell+1}{2}} \right)
\stackrel{\eqref{eq:braidcoeff}}{=} b_{n-\ell} \, . \qedhere
\end{align*}
\end{proof}

\begin{rem}
	\label{rem:inverse}
By Proposition \ref{itoxchangeofbasis}, 
the inverse braiding $\RM_{S,S}^{-1} \in \End_{U_q(\son)}(S \otimes S)$ 
can also be expressed as a $\C(q^{\frac{1}{2}})$-linear combination 
of the elements $\{ \xx^{(i)} \}_{i=0}^n$.
It is possible to show 
(either by analogous arguments to those used thus far, 
or as a consequence of Theorem \ref{thm:Ctaubraid} below) that
\begin{equation}
	\label{eq:invR}
\RM_{S, S}^{-1} = q^{-\frac{n}{2}} \sum_{i=0}^n q^{i}\xx^{(i)} \, .
\end{equation}

Note that \eqref{E:xxk-recursion} allows for any product of the 
$\{ \xx^{(i)} \}_{i=0}^n$ to be expanded as a linear combination of these elements.
Since they are linearly independent, 
this implies that \eqref{E:xxk-recursion} in fact suffices to confirm the formula \eqref{eq:invR}.
We will use this latter fact below.
\end{rem}

\subsection{Further relations in $\Rep\big(U_q(\son)\big)$}
	\label{ss:Webson}

In this section, we establish relations between various morphisms in the 
subcategory of $\Rep(U_q(\son))$ monoidally generated by the fundamental representations $S$ and $V_1=V(\varpi_1)$.
Our purpose is two-fold. 
First, these relations will be used later for our decategorification results.
Second, this lays the groundwork for giving a generators-and-relations presentation for the 
fundamental subcategory of $\Rep(U_q(\son))$, akin to that obtained for $\sln$ in \cite{CKM} and $\spn$ in \cite{BERT}, 
which we will pursue in future work.

The relations established here will be most conveniently described using the graphical language for monoidal categories, 
and we will use our conventions for diagrams established above (e.g.~gray strands correspond to $S$), but with one simplification:
since we only consider $V_k$ when $k=1$, we will drop the label $1$ from our previous diagrammatics and let 
unlabelled black strands correspond to $V_1$. 
We also abbreviate
\begin{equation}
	\label{eq:H1}
\hh := \hh^{(1)} = 
\begin{tikzpicture}[scale=.3, rotate=90, tinynodes, anchorbase]
	\draw[very thick,gray] (-1,0) to (0,1);
	\draw[very thick,gray] (1,0) to (0,1);
	\draw[very thick,gray] (0,2.5) to (-1,3.5);
	\draw[very thick,gray] (0,2.5) to (1,3.5);
	\draw[very thick] (0,1) to (0,2.5);
\end{tikzpicture}
=:
\begin{tikzpicture}[scale=.3, rotate=90, tinynodes, anchorbase]
	\draw[very thick,gray] (-1,1) to (1,1);
	\draw[very thick,gray] (-1,2.5) to (1,2.5);
	\draw[very thick] (0,1) to (0,2.5);
\end{tikzpicture} \, .
\end{equation}

\begin{example}\label{E:VSS-trivalent-twist}
In our present conventions, the $i=1$ case of Lemma \ref{trivalenttwist} is the relation
\[
\begin{tikzpicture}[scale =.35, smallnodes,anchorbase, rotate=180]
	\draw[very thick] (.5,.75) to (.5,1.5);
	\draw[very thick,gray] (1,-1.5) to [out=90,in=270] (0,0) to [out=90,in=210] (.5,.75);
	\begin{scope}
		\clip (0,-1.25) rectangle (1,-.25);
		\draw[overcross] (0,-1.5) to [out=90,in=270] (1,0) to [out=90,in=330] (.5,.75);
	\end{scope}
	\draw[very thick,gray] (0,-1.5) to [out=90,in=270] (1,0) to [out=90,in=330] (.5,.75);
\end{tikzpicture}
= q^{\frac{-n+2}{2}}q^{-\{2, \dots, n\}}
\begin{tikzpicture}[scale =.35, smallnodes,anchorbase, rotate=180]
	\draw[very thick] (.5,.625) to (.5,1.5);
	\draw[very thick,gray] (0,-1) to [out=90,in=210] (.5,.625);
	\draw[very thick,gray] (1,-1) to [out=90,in=330] (.5,.625);
\end{tikzpicture}
\]
which immediately implies that
\[
\begin{tikzpicture}[scale =.35, smallnodes,anchorbase, rotate=180,xscale=-1]
	\draw[very thick] (.5,.75) to (.5,1.5);
	\draw[very thick,gray] (1,-1.5) to [out=90,in=270] (0,0) to [out=90,in=210] (.5,.75);
	\begin{scope}
		\clip (0,-1.25) rectangle (1,-.25);
		\draw[overcross] (0,-1.5) to [out=90,in=270] (1,0) to [out=90,in=330] (.5,.75);
	\end{scope}
	\draw[very thick,gray] (0,-1.5) to [out=90,in=270] (1,0) to [out=90,in=330] (.5,.75);
\end{tikzpicture}
= q^{\frac{n-2}{2}}q^{\{2, \dots, n\}}
\begin{tikzpicture}[scale =.35, smallnodes,anchorbase, rotate=180]
	\draw[very thick] (.5,.625) to (.5,1.5);
	\draw[very thick,gray] (0,-1) to [out=90,in=210] (.5,.625);
	\draw[very thick,gray] (1,-1) to [out=90,in=330] (.5,.625);
\end{tikzpicture} \, .
\]
\end{example}

\begin{example}
	\label{eg:devilscircle}
Similarly, in our present color-coding, \eqref{eq:circlen} reads as:
\[
\begin{tikzpicture}[scale =.5, anchorbase]
	\draw[very thick,gray] (0,0) circle (.5);
\end{tikzpicture} 
	=(-1)^{\binom{n+1}{2}}\prod_{i=1}^n(q^{2i-1}+ q^{1-2i})
\]
Observe that the $i=\ell=n$ case of the equality $\prescript{n-i}{}\lambda_{n-i} = 1$
established in Proposition \ref{itoxchangeofbasis} gives that this latter quantity equals
\[
(-1)^{n+{n \choose 2}} \prod_{t=1}^{n} \frac{``[n+1-t][n+t]"}{``[t]^2"} \, .
\]
\end{example}

\begin{lem}\label{L:basic-relns-SSV-trivalent-1}
The relations
\[
\begin{tikzpicture}[scale=.175,tinynodes, anchorbase]
	\draw [very thick, gray] (0,.75) to (0,2.5) ;
	\draw [very thick] (0,-2.75) to [out=30,in=330]  (0,.75);
	\draw [very thick, gray] (0,-2.75) to [out=150,in=210]  (0,.75);
	\draw [very thick, gray] (0,-4.5) to (0,-2.75);
\end{tikzpicture}
= 
(-1)^n\frac{[2n+1]}{[2]}
\begin{tikzpicture}[scale=.175, tinynodes, anchorbase]
	\draw [very thick, gray] (0,-4.5) to (0,2.5);
\end{tikzpicture}
\qquad \text{and} \qquad
\begin{tikzpicture}[scale=.175,smallnodes,anchorbase]
	\draw[very thick] (-1,0) to (1,0);
	\draw[very thick, gray] (-1,0) to (0,1.732);
	\draw[very thick, gray] (1,0) to (0,1.732);
	\draw[very thick] (0,1.732) to (0,3.232);
	\draw[very thick, gray] (-2.3,-.75) to (-1,0);
	\draw[very thick, gray] (2.3,-.75) to (1,0);
\end{tikzpicture}
=
(-1)^{n-1}\frac{[2n-1]}{[2]}
\begin{tikzpicture}[scale =.5, smallnodes,anchorbase]
	\draw[very thick, gray] (0,0) to [out=90,in=210] (.5,.75);
	\draw[very thick, gray] (1,0) to [out=90,in=330] (.5,.75);
	\draw[very thick] (.5,.75) to (.5,1.5);
\end{tikzpicture}.
\]
hold in $\Rep(U_q(\son))$.
\end{lem}
\begin{proof}
The first equality follows from Remark \ref{spincupcapformulas}, Lemma \ref{bigon}, and \cite[Equation 1.1a]{BodWu}. 
We leave it to the reader to fill in the details, using the outline of the proof in \cite[Lemma 4.1]{BodWu}. 
The second equality is the $i=1$ case of Lemma \ref{SS1triangle}.
\end{proof}

Several of the remaining proofs in this section will use the fact that ``fork-slide'' relations hold 
in this graphical language, e.g.
\[
\begin{tikzpicture}[rotate=180,scale=.35,smallnodes,anchorbase]
	\draw[very thick,gray] (.5,-1) to [out=90,in=270] (-1,2);
	\draw[overcross] (-.5,-1) to [out=90,in=270] (.5,1);
	\draw[very thick] (-.5,-1) to [out=90,in=270] (.5,1);
	\draw[very thick,gray] (.5,1) to [out=30,in=270] (1,2);
	\draw[very thick,gray] (.5,1) to [out=150,in=270] (0,2);
\end{tikzpicture}
=
\begin{tikzpicture}[rotate=180,scale=.35,smallnodes,anchorbase]
	\draw[very thick,gray] (.5,-1) to [out=90,in=270] (-1,2);
	\draw[overcross] (-.5,-.25) to [out=30,in=270] (1,2);
	\draw[overcross] (-.5,-.25) to [out=150,in=270] (0,2);
	\draw[very thick] (-.5,-1) to (-.5,-.25);
	\draw[very thick,gray] (-.5,-.25) to [out=30,in=270] (1,2);
	\draw[very thick,gray] (-.5,-.25) to [out=150,in=270] (0,2);
\end{tikzpicture} \, .
\]
Such relations are consequences of the naturality of the braiding.

\begin{lem}\label{L:basic-relns-SSV-trivalent-2}
The relations
\[
\begin{tikzpicture}[scale =.35, smallnodes,anchorbase, rotate=180]
	\draw[very thick,gray] (.5,.75) to (.5,1.5);
	\draw[very thick] (1,-1.5) to [out=90,in=270] (0,0) to [out=90,in=210] (.5,.75);
	\begin{scope}
		\clip (0,-1.25) rectangle (1,-.25);
		\draw[overcross] (0,-1.5) to [out=90,in=270] (1,0) to [out=90,in=330] (.5,.75);
	\end{scope}
	\draw[very thick,gray] (0,-1.5) to [out=90,in=270] (1,0) to [out=90,in=330] (.5,.75);
\end{tikzpicture}
= (-1)^n q^{-2n}
\begin{tikzpicture}[scale =.35, smallnodes,anchorbase, rotate=180]
	\draw[very thick,gray] (.5,.625) to (.5,1.5);
	\draw[very thick,gray] (0,-1) to [out=90,in=210] (.5,.625);
	\draw[very thick] (1,-1) to [out=90,in=330] (.5,.625);
\end{tikzpicture}
\qquad \text{and} \qquad 
\begin{tikzpicture}[scale =.35, smallnodes,anchorbase, rotate=180,xscale=-1]
	\draw[very thick,gray] (.5,.75) to (.5,1.5);
	\draw[very thick] (1,-1.5) to [out=90,in=270] (0,0) to [out=90,in=210] (.5,.75);
	\begin{scope}
		\clip (0,-1.25) rectangle (1,-.25);
		\draw[overcross] (0,-1.5) to [out=90,in=270] (1,0) to [out=90,in=330] (.5,.75);
	\end{scope}
	\draw[very thick, gray] (0,-1.5) to [out=90,in=270] (1,0) to [out=90,in=330] (.5,.75);
\end{tikzpicture}
= (-1)^n q^{2n}
\begin{tikzpicture}[scale =.35, smallnodes,anchorbase, rotate=180]
	\draw[very thick,gray] (.5,.625) to (.5,1.5);
	\draw[very thick] (0,-1) to [out=90,in=210] (.5,.625);
	\draw[very thick, gray] (1,-1) to [out=90,in=330] (.5,.625);
\end{tikzpicture}
\]
hold in $\Rep(U_q(\son))$.
\end{lem}
\begin{proof}
The first relation follows by applying a fork-slide, 
resolving a curl using the $i=0$ case of Lemma \ref{trivalenttwist},
and then using Example \ref{E:VSS-trivalent-twist}.
The second is an immediate consequence of the first.
\end{proof}

\begin{lem}\label{L:SV-VS-basis}
The morphisms
\begin{equation}\label{E:SV-VS-basis}
\begin{tikzpicture}[scale=.3, rotate=90, tinynodes, anchorbase]
	\draw[very thick, gray] (-1,0) to (0,1);
	\draw[very thick] (1,0) to (0,1);
	\draw[very thick] (0,2.5) to (-1,3.5);
	\draw[very thick, gray] (0,2.5) to (1,3.5);
	\draw[very thick, gray] (0,1) to (0,2.5);
\end{tikzpicture}
\qquad \text{and} \qquad 
\begin{tikzpicture}[scale=.3, tinynodes, anchorbase]
	\draw[very thick] (-1,0) to (0,1);
	\draw[very thick, gray] (1,0) to (0,1);
	\draw[very thick, gray] (0,2.5) to (-1,3.5);
	\draw[very thick] (0,2.5) to (1,3.5);
	\draw[very thick, gray] (0,1) to (0,2.5);
\end{tikzpicture}
\end{equation}
are a basis of $\Hom_{U_q(\son)}(V_1\otimes S, S\otimes V_1)$.
The vertical reflections of these diagrams give a basis for $\Hom_{U_q(\son)}(S\otimes V_1,V_1\otimes S)$.
\end{lem}
\begin{proof}
Since $V_1\otimes S \cong V(\varpi_1{+}\varpi_n) \oplus S \cong S \otimes V_1$, 
it suffices to show linear independence. 
Suppose there were $x,y \in \C(q^{\frac{1}{2}})$ such that $x$ times the first morphism in \eqref{E:SV-VS-basis}
plus $y$ times the second morphism is zero.
Using Lemma \ref{L:basic-relns-SSV-trivalent-1} to evaluate this relation in two different ways by composing 
respectively on the bottom and the side with trivalent vertices,
we obtain the system of linear equations
\[
(-1)^{n-1}\frac{[2n-1]}{[2]}x + (-1)^n\frac{[2n+1]}{[2]}y = 0 \quad \text{and} \quad (-1)^n\frac{[2n+1]}{[2]}x + (-1)^{n-1}\frac{[2n-1]}{[2]} y= 0 \, .
\] 
This system has unique solution $x=0=y$. 
The argument for their vertical reflections is identical.
\end{proof}

\begin{prop}\label{P:SV-braiding}
The braiding $R_{S,V_{1}} \in \Hom_{U_q(\son)}(S \otimes V_1,V_1 \otimes S)$ and its inverse $R_{S,V_{1}}^{-1}$ are given by
\[
\begin{tikzpicture}[scale=.4, anchorbase]
	\draw[very thick] (1,0) to [out=90,in=270] (0,1.5);
	\draw[overcross] (0,0) to [out=90,in=270] (1,1.5);
	\draw[very thick, gray] (0,0) to [out=90,in=270] (1,1.5);
\end{tikzpicture}
= 
q
\begin{tikzpicture}[scale=.3, rotate=90, tinynodes, anchorbase]
	\draw[very thick] (-1,0) to (0,1);
	\draw[very thick, gray] (1,0) to (0,1);
	\draw[very thick, gray] (0,2.5) to (-1,3.5);
	\draw[very thick] (0,2.5) to (1,3.5);
	\draw[very thick, gray] (0,1) to (0,2.5);
\end{tikzpicture}
+q^{-1}
\begin{tikzpicture}[scale=.3, tinynodes, anchorbase]
	\draw[very thick, gray] (-1,0) to (0,1);
	\draw[very thick] (1,0) to (0,1);
	\draw[very thick] (0,2.5) to (-1,3.5);
	\draw[very thick, gray] (0,2.5) to (1,3.5);
	\draw[very thick, gray] (0,1) to (0,2.5);
\end{tikzpicture}
\qquad \text{and} \qquad 
\begin{tikzpicture}[scale=.4, anchorbase,xscale=-1]
	\draw[very thick] (1,0) to [out=90,in=270] (0,1.5);
	\draw[overcross] (0,0) to [out=90,in=270] (1,1.5);
	\draw[very thick, gray] (0,0) to [out=90,in=270] (1,1.5);
\end{tikzpicture}
= 
q^{-1}
\begin{tikzpicture}[scale=.3, rotate=90, tinynodes, anchorbase,xscale=-1]
	\draw[very thick] (-1,0) to (0,1);
	\draw[very thick, gray] (1,0) to (0,1);
	\draw[very thick, gray] (0,2.5) to (-1,3.5);
	\draw[very thick] (0,2.5) to (1,3.5);
	\draw[very thick, gray] (0,1) to (0,2.5);
\end{tikzpicture}
+q
\begin{tikzpicture}[scale=.3, tinynodes, anchorbase,xscale=-1]
	\draw[very thick, gray] (-1,0) to (0,1);
	\draw[very thick] (1,0) to (0,1);
	\draw[very thick] (0,2.5) to (-1,3.5);
	\draw[very thick, gray] (0,2.5) to (1,3.5);
	\draw[very thick, gray] (0,1) to (0,2.5);
\end{tikzpicture}
\]
\end{prop}
\begin{proof}
The second equality follows from the first using pivotality 
(e.g.~by rotating all diagrams by $90^\circ$), so it suffices to establish the first.
By Lemma \ref{L:SV-VS-basis}, 
there exist $x,y \in \C(q^{\frac{1}{2}})$ such that the braiding on $S \otimes V_1$
is $x$ times the first diagram in \eqref{E:SV-VS-basis} plus $y$ times the second. 
Attaching a trivalent vertex to this equality of morphisms on the bottom and right hand side respectively
and applying Lemma \ref{L:basic-relns-SSV-trivalent-2} gives the equations
\begin{align*}
(-1)^n q^{-2n} &= (-1)^{n-1}\frac{[2n-1]}{[2]}x + (-1)^n\frac{[2n+1]}{[2]}y \\ 
(-1)^n q^{2n} &= (-1)^n\frac{[2n+1]}{[2]}x + (-1)^{n-1}\frac{[2n-1]}{[2]}y
\end{align*}
respectively.
It is easy to verify that $x= q$ and $y=q^{-1}$ gives the unique solution.
\end{proof}

It is convenient to rewrite Proposition \ref{P:SV-braiding} as:
\begin{equation}
	\label{E:rung-swap}
\begin{tikzpicture}[scale=.3, tinynodes, anchorbase]
	\draw[very thick, gray] (-1,0) to (0,1);
	\draw[very thick] (1,0) to (0,1);
	\draw[very thick] (0,2.5) to (-1,3.5);
	\draw[very thick, gray] (0,2.5) to (1,3.5);
	\draw[very thick, gray] (0,1) to (0,2.5);
\end{tikzpicture}
= 
-q^{2}
\begin{tikzpicture}[scale=.3, rotate=90, tinynodes, anchorbase]
	\draw[very thick] (-1,0) to (0,1);
	\draw[very thick, gray] (1,0) to (0,1);
	\draw[very thick, gray] (0,2.5) to (-1,3.5);
	\draw[very thick] (0,2.5) to (1,3.5);
	\draw[very thick, gray] (0,1) to (0,2.5);
\end{tikzpicture}
+
q
\begin{tikzpicture}[scale=.4, anchorbase]
	\draw[very thick] (1,0) to [out=90,in=270] (0,1.5);
	\draw[overcross] (0,0) to [out=90,in=270] (1,1.5);
	\draw[very thick, gray] (0,0) to [out=90,in=270] (1,1.5);
\end{tikzpicture}
\qquad \text{and} \qquad
\begin{tikzpicture}[scale=.3, tinynodes, anchorbase]
	\draw[very thick] (-1,0) to (0,1);
	\draw[very thick, gray] (1,0) to (0,1);
	\draw[very thick, gray] (0,2.5) to (-1,3.5);
	\draw[very thick] (0,2.5) to (1,3.5);
	\draw[very thick, gray] (0,1) to (0,2.5);
\end{tikzpicture}
= 
-q^{-2}
\begin{tikzpicture}[scale=.3, rotate=90, tinynodes, anchorbase]
	\draw[very thick, gray] (-1,0) to (0,1);
	\draw[very thick] (1,0) to (0,1);
	\draw[very thick] (0,2.5) to (-1,3.5);
	\draw[very thick, gray] (0,2.5) to (1,3.5);
	\draw[very thick, gray] (0,1) to (0,2.5);
\end{tikzpicture}
+
q^{-1}
\begin{tikzpicture}[scale=.4, anchorbase]
	\draw[very thick] (0,0) to [out=90,in=270] (1,1.5);
	\draw[overcross] (1,0) to [out=90,in=270] (0,1.5);
	\draw[very thick, gray] (1,0) to [out=90,in=270] (0,1.5);
\end{tikzpicture} \, .
\end{equation}

\begin{cor}\label{C:Rstkn-Reln}
The relations
\[
\begin{tikzpicture}[scale=.4, anchorbase]
	\draw[very thick] (1,0) to [out=90,in=270] (0,1.5);
	\draw[overcross] (0,0) to [out=90,in=270] (1,1.5);
	\draw[very thick] (0,0) to [out=90,in=270] (1,1.5);
	\draw[very thick, gray] (1,0) to (0,0);
	\draw[very thick, gray] (1,0) to (1,-1);
	\draw[very thick, gray] (0,0) to (0,-1);
\end{tikzpicture}
= -q^{-2} \;
\begin{tikzpicture}[scale=.4, anchorbase]
	\draw[very thick] (1,0) to (1,1);
	\draw[very thick] (0,0) to (0,1);
	\draw[very thick, gray] (1,0) to (0,0);
	\draw[very thick, gray] (1,0) to (1,-1);
	\draw[very thick, gray] (0,0) to (0,-1);
\end{tikzpicture}
+ (-1)^nq^{-2n-1}
\begin{tikzpicture}[scale=.4, tinynodes, anchorbase]
	\draw[very thick, gray] (0,0) to [out=90,in=180] (.5,.625) 
		to [out=0,in=90] (1,0);
	\draw[very thick] (0,2) to [out=270,in=180] (.5,1.375)
		to [out=0,in=270] (1,2);
\end{tikzpicture}
\qquad \text{and} \qquad 
\begin{tikzpicture}[scale=.4, anchorbase]
	\draw[very thick] (0,0) to [out=90,in=270] (1,1.5);
	\draw[overcross] (1,0) to [out=90,in=270] (0,1.5);
	\draw[very thick] (1,0) to [out=90,in=270] (0,1.5);
	\draw[very thick, gray] (1,0) to (0,0);
	\draw[very thick, gray] (1,0) to (1,-1);
	\draw[very thick, gray] (0,0) to (0,-1);
\end{tikzpicture}
= -q^{2} \;
\begin{tikzpicture}[scale=.4, anchorbase]
	\draw[very thick] (1,0) to (1,1);
	\draw[very thick] (0,0) to (0,1);
	\draw[very thick, gray] (1,0) to (0,0);
	\draw[very thick, gray] (1,0) to (1,-1);
	\draw[very thick, gray] (0,0) to (0,-1);
\end{tikzpicture}
+ (-1)^nq^{2n+1}
\begin{tikzpicture}[scale=.4, tinynodes, anchorbase]
	\draw[very thick, gray] (0,0) to [out=90,in=180] (.5,.625) 
		to [out=0,in=90] (1,0);
	\draw[very thick] (0,2) to [out=270,in=180] (.5,1.375)
		to [out=0,in=270] (1,2);
\end{tikzpicture}
\]
hold in $\Hom_{U_q(\son)}(S \otimes S, V_1 \otimes V_1)$.
\end{cor}
\begin{proof}
We prove the first equality, which implies the second. 
Applying a fork-slide, then Lemma \ref{L:basic-relns-SSV-trivalent-2}, we find
\[
\begin{tikzpicture}[scale=.4, anchorbase]
	\draw[very thick] (1,0) to [out=90,in=270] (0,1.5);
	\draw[overcross] (0,0) to [out=90,in=270] (1,1.5);
	\draw[very thick] (0,0) to [out=90,in=270] (1,1.5);
	\draw[very thick, gray] (1,0) to (0,0);
	\draw[very thick, gray] (1,0) to (1,-1);
	\draw[very thick, gray] (0,0) to (0,-1);
\end{tikzpicture}
= (-1)^nq^{-2n}
\begin{tikzpicture}[scale=.3,smallnodes,anchorbase]
	\draw[very thick] (1,-1) to [out=150,in=270] 
		(0,2); 
	\draw[overcross] (0,-2) 
		to [out=90,in=210] (1,1);
	\draw[very thick,gray] (0,-2) 
		to [out=90,in=210] (1,1);
	\draw[very thick] (1,1) to (1,2);
	\draw[very thick,gray] (1,-2) to (1,-1); 
	\draw[very thick,gray] (1,-1) to [out=30,in=330] (1,1); 
\end{tikzpicture} \, .
\]
We then compute
\[
(-1)^n q^{-2n}
\begin{tikzpicture}[scale=.3,smallnodes,anchorbase]
	\draw[very thick] (1,-1) to [out=150,in=270] 
		(0,2); 
	\draw[overcross] (0,-2) 
		to [out=90,in=210] (1,1);
	\draw[very thick,gray] (0,-2) 
		to [out=90,in=210] (1,1);
	\draw[very thick] (1,1) to (1,2);
	\draw[very thick,gray] (1,-2) to (1,-1); 
	\draw[very thick,gray] (1,-1) to [out=30,in=330] (1,1); 
\end{tikzpicture}
=
(-1)^n q^{-2n}
\begin{tikzpicture}[scale=.3,smallnodes,anchorbase]
	\draw[very thick] (1.5,-.75) to [out=210,in=270] 
		(0,2); 
	\draw[overcross] (0,-2) 
		to [out=90,in=150] (1.5,.75);
	\draw[very thick,gray] (0,-2) 
		to [out=90,in=150] (1.5,.75);
	\draw[very thick] (1.5,.75) to [out=30,in=270] (2,2);
	\draw[very thick,gray] (2,-2) to [out=90,in=330] (1.5,-.75); 
	\draw[very thick,gray] (1.5,-.75) to (1.5,.75); 
\end{tikzpicture}
\stackrel{\eqref{E:rung-swap}}{=}
(-1)^n q^{-2n}
\left(
-q^{-2}
\begin{tikzpicture}[scale=.3,smallnodes,anchorbase]
	\draw[very thick] (-1,1) to [out=270,in=180] (0,-.5) to [out=0,in=240] (.75,0); 
	\draw[very thick,overcross] (-1,-1) to [out=90,in=180] (0,.5) to [out=0,in=120] (.75,0); 
	\draw[very thick,gray] (-1,-1) to [out=90,in=180] (0,.5) to [out=0,in=120] (.75,0); 
	\draw[very thick, gray] (.75,0) to (1.5,0);
	\draw[very thick] (2,1) to [out=270,in=60] (1.5,0); 
	\draw[very thick,gray] (2,-1) to [out=90,in=300] (1.5,0); 
\end{tikzpicture}
+
q^{-1}
\begin{tikzpicture}[scale=.3,smallnodes,anchorbase]
	\draw[very thick] (0,1) to [out=270,in=180] (1,-.5) to [out=0,in=270] (2,1);
	\draw[overcross] (0,-1) to [out=90,in=180] (1,.5) to [out=0,in=90] (2,-1);
	\draw[very thick,gray] (0,-1) to [out=90,in=180] (1,.5) to [out=0,in=90] (2,-1);
\end{tikzpicture}
\right)
\]
and the result follows from Lemma \ref{L:basic-relns-SSV-trivalent-2} 
and invertibility of the braiding.
\end{proof}

\begin{lem}\label{L:H-satisfies-i-quantum-gp-Serre}
The relation
\[
\begin{tikzpicture}[scale=.4, anchorbase]
	\draw[very thick] (0,1.5) to (1,1.5);
	\draw[very thick] (0,.75) to (1,.75);
	\draw[very thick] (1,2.25) to (2,2.25);
	\draw[very thick, gray] (0,0) to (0,3);
	\draw[very thick, gray] (1,0) to (1,3);
	\draw[very thick, gray] (2,0) to (2,3);
\end{tikzpicture}
+ 
\begin{tikzpicture}[scale=.4, anchorbase]
	\draw[very thick] (0,1.5) to (1,1.5);
	\draw[very thick] (0,2.25) to (1,2.25);
	\draw[very thick] (1,.75) to (2,.75);
	\draw[very thick, gray] (0,0) to (0,3);
	\draw[very thick, gray] (1,0) to (1,3);
	\draw[very thick, gray] (2,0) to (2,3);
\end{tikzpicture}
= 
-[2]_{q^2}
\begin{tikzpicture}[scale=.4, anchorbase]
	\draw[very thick, gray] (0,0) to (0,3);
	\draw[very thick, gray] (1,0) to (1,3);
	\draw[very thick, gray] (2,0) to (2,3);
	\draw[very thick] (0,.75) to (1,.75);
	\draw[very thick] (1,1.5) to (2,1.5);
	\draw[very thick] (0,2.25) to (1,2.25);
\end{tikzpicture}
+
\begin{tikzpicture}[scale=.4, anchorbase]
	\draw[very thick, gray] (0,0) to (0,3);
	\draw[very thick, gray] (1,0) to (1,3);
	\draw[very thick, gray] (2,0) to (2,3);
	\draw[very thick] (1,1.5) to (2,1.5);
\end{tikzpicture}
\]
holds in $\End_{U_q(\son)}(S^{\otimes 3})$.
\end{lem}
\begin{proof}
The relations \eqref{E:rung-swap} imply
\begin{equation}
	\label{eq:ladderswap}
\begin{tikzpicture}[scale=.4, anchorbase]
	\draw[very thick] (0,1.5) to (2,1.5);
	\draw[very thick, gray] (0,0) to (0,3);
	\draw[overcross] (1,0) to (1,3);
	\draw[very thick, gray] (1,0) to (1,3);
	\draw[very thick, gray] (2,0) to (2,3);
\end{tikzpicture}
=
q \;
\begin{tikzpicture}[scale=.4, anchorbase]
	\draw[very thick] (0,.875) to (1,.875);
	\draw[very thick] (1,2.125) to (2,2.125);
	\draw[very thick, gray] (0,0) to (0,3);
	\draw[very thick, gray] (1,0) to (1,3);
	\draw[very thick, gray] (2,0) to (2,3);
\end{tikzpicture}
+
q^{-1} \;
\begin{tikzpicture}[scale=.4, anchorbase]
	\draw[very thick] (1,.875) to (2,.875);
	\draw[very thick] (0,2.125) to (1,2.125);
	\draw[very thick, gray] (0,0) to (0,3);
	\draw[very thick, gray] (1,0) to (1,3);
	\draw[very thick, gray] (2,0) to (2,3);
\end{tikzpicture}
\end{equation}
which gives
\[
\begin{tikzpicture}[scale=.4, anchorbase]
	\draw[very thick] (0,1.5) to (1,1.5);
	\draw[very thick] (0,.75) to (1,.75);
	\draw[very thick] (1,2.25) to (2,2.25);
	\draw[very thick, gray] (0,0) to (0,3);
	\draw[very thick, gray] (1,0) to (1,3);
	\draw[very thick, gray] (2,0) to (2,3);
\end{tikzpicture}
+ 
\begin{tikzpicture}[scale=.4, anchorbase]
	\draw[very thick] (0,1.5) to (1,1.5);
	\draw[very thick] (0,2.25) to (1,2.25);
	\draw[very thick] (1,.75) to (2,.75);
	\draw[very thick, gray] (0,0) to (0,3);
	\draw[very thick, gray] (1,0) to (1,3);
	\draw[very thick, gray] (2,0) to (2,3);
\end{tikzpicture}
=
-[2]_{q^2}
\begin{tikzpicture}[scale=.4, anchorbase]
	\draw[very thick, gray] (0,0) to (0,3);
	\draw[very thick, gray] (1,0) to (1,3);
	\draw[very thick, gray] (2,0) to (2,3);
	\draw[very thick] (0,.75) to (1,.75);
	\draw[very thick] (1,1.5) to (2,1.5);
	\draw[very thick] (0,2.25) to (1,2.25);
\end{tikzpicture}
+q^{-1} \;
\begin{tikzpicture}[scale=.4, anchorbase]
	\draw[very thick] (0,1.5) to (2,1.5);
	\draw[overcross] (1,0) to (1,3);
	\draw[very thick, gray] (1,0) to (1,3);
	\draw[very thick] (0,.75) to (1,.75);
	\draw[very thick, gray] (0,0) to (0,3);
	\draw[very thick, gray] (2,0) to (2,3);
\end{tikzpicture}
+q \:
\begin{tikzpicture}[scale=.4, anchorbase]
	\draw[very thick] (0,1.5) to (2,1.5);
	\draw[overcross] (1,0) to (1,3);
	\draw[very thick, gray] (1,0) to (1,3);
	\draw[very thick] (0,2.25) to (1,2.25);
	\draw[very thick, gray] (0,0) to (0,3);
	\draw[very thick, gray] (2,0) to (2,3);
\end{tikzpicture} \, .
\]
Applying a forkslide, Corollary \ref{C:Rstkn-Reln}, 
and Lemma \ref{L:basic-relns-SSV-trivalent-2}, 
we simplify the last diagram as
\[
\begin{tikzpicture}[scale=.4, anchorbase]
	\draw[very thick] (0,1.5) to (2,1.5);
	\draw[overcross] (1,0) to (1,3);
	\draw[very thick, gray] (1,0) to (1,3);
	\draw[very thick] (0,2.25) to (1,2.25);
	\draw[very thick, gray] (0,0) to (0,3);
	\draw[very thick, gray] (2,0) to (2,3);
\end{tikzpicture} 
= 
-q^{-2}
\begin{tikzpicture}[scale=.4, anchorbase]
	\draw[very thick] (0,1.5) to (2,1.5);
	\draw[overcross] (1,0) to (1,3);
	\draw[very thick, gray] (1,0) to (1,3);
	\draw[very thick] (0,.75) to (1,.75);
	\draw[very thick, gray] (0,0) to (0,3);
	\draw[very thick, gray] (2,0) to (2,3);
\end{tikzpicture} 
+
(-1)^n q^{-2n-1} (-1)^n q^{2n} \;
\begin{tikzpicture}[scale=.4, anchorbase]
	\draw[very thick, gray] (0,0) to (0,3);
	\draw[very thick, gray] (1,0) to (1,3);
	\draw[very thick, gray] (2,0) to (2,3);
	\draw[very thick] (1,1.5) to (2,1.5);
\end{tikzpicture}
\]
and the result follows.
\end{proof}

It is natural to consider Lemma \ref{L:H-satisfies-i-quantum-gp-Serre} in terms of the morphisms
$\{\xx^{(i)}\}_{i=0}^n$ from Definition \ref{D:xxk-defn}.
For $m \geq 3$ and $1 \leq r \leq m-1$, let
\[
\hh_r := \id_{S^{\otimes r-1}} \otimes \hh \otimes \id_{S^{\otimes m-r-1}}
	\in \End_{U_q(\son)}(S^{\otimes m})
\]
and
\[
\xx_r^{(i)} := \id_{S^{\otimes r-1}} \otimes \xx^{(i)} \otimes \id_{S^{\otimes m-r-1}}
	\in \End_{U_q(\son)}(S^{\otimes m}) \, .
\]
As above, we abbreviate $\xx_r := \xx_r^{(1)}$ and by definition $\xx_r^{(0)} = \id_{S^{\otimes m}}$.
When $|r-s| \geq 2$, these morphisms satisfy the far-commutativity relation:
\[
\hh_r \hh_s = \hh_s \hh_r
\qquad \text{and}\qquad
\xx_r^{(i)} \xx_s^{(j)} = \xx_s^{(j)} \xx_r^{(i)} \, .
\]
In this notation, 
Lemma \ref{L:H-satisfies-i-quantum-gp-Serre} is the relation
\begin{equation}
	\label{E:H-satisfies-i-quantum-gp-Serre}
\hh_2 \hh_1 \hh_1 + \hh_1 \hh_1 \hh_2 = 
	-[2]_{q^2} \hh_1 \hh_2 \hh_1 + \hh_2
\end{equation}
which we now express in terms of the $\xx_i$.

\begin{prop}[devil's Serre relation]
	\label{prop:iSerre}
The relations 
\begin{equation}
	\label{eq:iSerre}
\xx_{i} \xx_{i\pm1} \xx_{i} 
	= \xx_{i}^{(2)} \xx_{i\pm1} + \xx_{i\pm1} \xx_{i}^{(2)} + [2] \xx_{i}^{(2)} + \xx_{i}
\end{equation}
hold in $\End_{U_q(\son)}(S^{\otimes m})$ 
whenever $\min(i,i\pm1) \geq 1$ and $\max(i,i\pm1) \leq m$.
\end{prop}
\begin{proof}
It suffices to consider the $m=3$ case, 
and we will establish the relation
\[
\xx_1\xx_2\xx_1 = \xx_1^{(2)}\xx_2 + \xx_2\xx_1^{(2)} + [2]\xx_1^{(2)} + \xx_1
\]
since the other relation is obtained by conjugating with appropriate braids
(or by making a similar computation).
Recall from Definition \ref{D:xxk-defn}, that
\[
\xx = \hh- \frac{1}{[2]} \qquad \text{and} \qquad \xx^{(2)}= \frac{-1}{[2]_{q^2}}\xx \big( \xx+ [2] \big)
\]
so
\begin{equation}\label{E:decat-serre-H^2}
\xx^{(2)}= \frac{-1}{[2]_{q^2}} \Big( \hh- \frac{1}{[2]} \Big) \Big( \hh + \frac{[3]}{[2]} \Big) 
\qquad \text{and} \qquad 
\frac{-1}{[2]_{q^2}}\hh^2 = \xx^{(2)} + \frac{1}{[2]}\hh - \frac{[3]}{[2]^2[2]_{q^2}} \, .
\end{equation}
Directly computing, we find
\begin{align*}
\xx_1\xx_2\xx_1 
&= \hh_1\hh_2\hh_1-\frac{1}{[2]}\big(\hh_1\hh_2 + \hh_1^2 + \hh_2\hh_1\big) + \frac{1}{[2]^2}\big(2\hh_1 + \hh_2\big) - \frac{1}{[2]^3} \\
&\stackrel{\eqref{E:H-satisfies-i-quantum-gp-Serre}}{=}
	\frac{-1}{[2]_{q^2}}\big( \hh_2\hh_1^2 + \hh_1^2\hh_2 - \hh_2 \big)
		-\frac{1}{[2]}\big(\hh_1\hh_2 + \hh_1^2 + \hh_2\hh_1\big) 
			+ \frac{1}{[2]^2}\big(2\hh_1 + \hh_2\big) - \frac{1}{[2]^3} \, .
\end{align*}
Similarly,
\[
\xx_1^{(2)}\xx_2 + \xx_2\xx_1^{(2)} \stackrel{\eqref{E:decat-serre-H^2}}{=}  
	\frac{-1}{[2]_{q^2}} \left( \Big(\hh_1- \frac{1}{[2]} \Big) \Big( \hh_1 + \frac{[3]}{[2]} \Big) \Big( \hh_2- \frac{1}{[2]} \Big) 
		+ \Big( \hh_2- \frac{1}{[2]} \Big) \Big( \hh_1- \frac{1}{[2]} \Big) \Big( \hh_1 + \frac{[3]}{[2]} \Big) \right)
 \]
 which expands as
 \[
 \frac{-1}{[2]_{q^2}}\bigg(\hh_1^2\hh_2 + \hh_2\hh_1^2 -\frac{2}{[2]}\hh_1^2 + \frac{[2]_{q^2}}{[2]}\hh_1\hh_2 
 	+ \frac{[2]_{q^2}}{[2]}\hh_2\hh_1 + 2\frac{[2]_{q^2}}{[2]^2}\hh_1 - 2\frac{[3]}{[2]^2}\hh_2 + 2\frac{[3]}{[2]^3}\bigg) \, .
\]
Thus, we see that $\xx_1\xx_2\xx_1 -  \xx_1^{(2)}\xx_2 - \xx_2\xx_1^{(2)}$ equals
\[
-\frac{[2]}{[2]_{q^2}} \hh_1^2 + \frac{1}{[2][2]_{q^2}} \stackrel{\eqref{E:decat-serre-H^2}}{=} [2]\xx_1^{(2)} + \hh_1 -\frac{[3]}{[2][2]_{q^2}} 
+ \frac{1}{[2][2]_{q^2}} =  [2]\xx_1^{(2)} + \xx_1 \, . \qedhere
\]
\end{proof}

We now establish the connection discussed above in Section \ref{ss:iquantum}
between $\iota$quantum groups 
and our elements $\xx^{(i)} \in \End_{U_q(\son)}(S\otimes S)$.

\begin{thm}\label{T:iqWeyl=typeBbraid}
The surjective $\C(q)$-algebra homomorphism 
$U'_{-q^2}(\som) \rightarrow \End_{U_q(\son)}(S^{\otimes m})$ 
from Wenzl's Theorem \ref{thm:Wenzl} is such that
\[
\nx_i^{(k)} \mapsto \xx_i^{(k)} \quad \text{and} \quad 
\sum_{k= 0}^{n} q^{-k} \nx_i^{(k)} \mapsto 
	q^{-\frac{n}{2}} \id_S^{\otimes i-1} \otimes R_{S,S} \otimes \id_S^{\otimes m-i-1} \, .
\]
\end{thm}
\begin{proof}
Recall from Definition \ref{def:idpb} that 
$\nx_i \in U'_{-q^2}(\som)$ are defined by $\nx_i := b_i - \frac{1}{[2]}$ 
where the $b_i$ are the standard generators of $U'_{-q^2}(\som)$ in Definition \ref{def:GK}.
Lemma \ref{L:H-satisfies-i-quantum-gp-Serre} then implies that the assignment
$\nx_i \mapsto \xx_i$ indeed defines an algebra homomorphism 
$U'_{-q^2}(\som)\rightarrow \End_{U_q(\son)}(S^{\otimes m})$.
Comparing Definition \ref{def:idpb} and \eqref{eq:xidefn}, 
we see that this sends $\nx_i^{(k)} \mapsto \xx_i^{(k)}$. 
In Proposition \ref{P:CequalsH1}, we show that our endomorphism
$\hh \in \End_{U_q(\son)}(S^{\otimes 2})$ agrees with Wenzl's endomorphism $C$, 
hence this is precisely Wenzl's homomorphism from Theorem \ref{thm:Wenzl}.
The remaining claim follows from Proposition \ref{P:decat-spin-rickard}.
\end{proof}

%
\subsection{Spin link polynomials}
	\label{ss:SLP}
%

We now characterize the spin-colored link polynomials $P_{\son}(\mathcal{L}_{\beta}^S)$ 
in terms of the elements $\xx_r^{(i)}$. 
Indeed, by Proposition \ref{P:decat-spin-rickard}, 
the endomorphism $R(\beta,S) \in \End_{U_q(\son)}(S^{\otimes m})$ 
assigned to a braid $\beta \in \Br_m$
can be written as a linear combination of $\{\xx_r^{(i)}\}_{1 \leq r \leq m-1}$
which, for each $r$, satisfy \eqref{eq:xidefn}, i.e.
\begin{equation}
	\label{eq:xiquad}
\xx_r^{(i+1)} := \dfrac{(-1)^i}{``[i+1]^2"}\left(\xx_r^{(i)}\xx_r - (-1)^{i} ``[i][i+1]"\xx_r^{(i)}\right) \, .
\end{equation}
We now show that, modulo a pair of conjectures that we have only verified 
in low rank ($n=1,2,3$), 
the link invariant $P_{\son}(\mathcal{L}_{\beta}^S)$ is characterized by 
the braiding formula \eqref{eq:Xbraiding}, the relations \eqref{eq:xidefn} and \eqref{eq:iSerre}, 
and a compatibility between 
the quantum traces on 
$\End_{U_q(\son)}(S^{\otimes m})$ and $\End_{U_q(\son)}(S^{\otimes m-1})$.
We will use this characterization for our decategorification results in 
\S \ref{ss:decat}.

To begin, we conjecture that relations \eqref{eq:xidefn} and \eqref{eq:iSerre} 
suffice to establish more-general versions of Proposition \ref{prop:iSerre}.

\begin{conj}
	\label{conj:R3}
Let $A_m^n$ be a unital $\C(q)$-algebra and
let $\{\xx_r^{(i)}\}_{\substack{1 \leq r \leq m-1 \\ 0 \leq i \leq n}} \subset A_m^n$ 
be a collection of elements
satisfying $\xx_r^{(0)} = 1$ and 
equations \eqref{eq:xidefn} and \eqref{eq:iSerre}. 
Then, given $1 \leq a,b,c \leq n$, there is a relation of the form
\[
\xx_{i}^{(a)} \xx_{i\pm1}^{(b)} \xx_{i}^{(c)} = 
	\xi \cdot \xx_{i\pm1}^{(a')} \xx_{i}^{(b')} \xx_{i\pm1}^{(c')} + \mathrm{LOT}_{a,b,c}
\]
in $A_m^n$, where $\xi \in \C(q)$, $1 \leq a',b',c' \leq n$, 
and $\mathrm{LOT}_{a,b,c}$ is a linear combination of terms of the form 
$\xx_{i}^{(k)} \xx_{i\pm1}^{(\ell)}$ and $\xx_{i\pm1}^{(\ell)} \xx_{i}^{(k)}$.
\end{conj}

\begin{prop}\label{prop:R3}
Conjecture \ref{conj:R3} holds when $n=1,2,3$.
\end{prop}

\begin{proof}
This is immediate from \eqref{eq:iSerre} when $n=1$, 
since in that case $\xx_{r}^{(a)}=0$ when $a \geq 2$.
We leave it as a (sometimes challenging) exercise to confirm that 
the following relations in $A_m^3$
can be derived using only \eqref{eq:xidefn} and \eqref{eq:iSerre}:
\begin{equation}
	\label{eq:k1l}
\xx_{i}^{(a)} \xx_{i\pm1}^{(1)} \xx_{i}^{(c)}	
\left\{
\begin{aligned}
\xx_{i}^{(2)} \xx_{i\pm1} \xx_{i} 
	&= \xx_{i\pm1} \xx_{i}^{(3)} - ``[2]^2" \xx_{i}^{(3)} \xx_{i\pm1} - [2]  \xx_{i}^{(2)} \xx_{i\pm1}
		- ``[2]^2" \xx_{i}^{(2)} - ``[2] [3]" \xx_{i}^{(3)} \\
\xx_{i}^{(2)} \xx_{i\pm1} \xx_{i}^{(2)} 
	&= ``[2] [3]" \big( \xx_{i\pm1} \xx_{i}^{(3)} + \xx_{i}^{(3)} \xx_{i\pm1} + \xx_{i}^{(2)} \big) 
		- (q^6 - 2q^4 - 2q^{-4} + q^{-6}) \xx_{i}^{(3)} \\ 
\xx_{i}^{(3)} \xx_{i\pm1} \xx_{i} 
	&= ``[2] [3]" \xx_{i}^{(3)} \xx_{i\pm1} + ``[3]^2" \xx_{i}^{(3)} \\
\xx_{i}^{(3)} \xx_{i\pm1} \xx_{i}^{(2)} 
	&= (q^4 + q^2 + q^{-2} + q^{-4}) \xx_{i}^{(3)} \xx_{i\pm1} 
		+ (q^7 - q^5 - q^{-5} + q^{-7}) \xx_{i}^{(3)} \\
\xx_{i}^{(3)} \xx_{i\pm1} \xx_{i}^{(3)} 
	&= -(q^8 + q^2 + q^{-2} + q^{-8}) \xx_{i}^{(3)} \\ 
	\end{aligned} \right.
	\end{equation}
\begin{equation}
	\label{eq:k2l}
\xx_{i}^{(a)} \xx_{i\pm1}^{(2)} \xx_{i}^{(c)}	
\left\{
\begin{aligned}
\xx_{i} \xx_{i\pm1}^{(2)} \xx_{i} 
	&= \xx_{i\pm1} \xx_{i}^{(2)} \xx_{i\pm1}  \\
\xx_{i}^{(2)} \xx_{i\pm1}^{(2)} \xx_{i} 
	&= \xx_{i\pm1} \xx_{i}^{(3)} \xx_{i\pm1} + \xx_{i}^{(3)} \xx_{i\pm1}^{(2)} 
		+ [2] \xx_{i}^{(3)} \xx_{i\pm1} + \xx_{i}^{(2)} \xx_{i\pm1}  \\
\xx_{i}^{(2)} \xx_{i\pm1}^{(2)} \xx_{i}^{(2)} 
	&= -[2] \xx_{i\pm1} \xx_{i}^{(3)} \xx_{i\pm1} + \xx_{i}^{(2)}
		- ``[2]^2" \big( \xx_{i\pm1} \xx_{i}^{(3)} + \xx_{i}^{(3)} \xx_{i\pm1} \big) 
			- ``[2][3]" \xx_{i}^{(3)} \\ 
\xx_{i}^{(3)} \xx_{i\pm1}^{(2)} \xx_{i} 
	&= -``[2]^2" \xx_{i}^{(3)} \xx_{i\pm1} - [2] \xx_{i}^{(3)} \xx_{i\pm1}^{(2)} \\
\xx_{i}^{(3)} \xx_{i\pm1}^{(2)} \xx_{i}^{(2)} 
	&= ``[2][3]" \xx_{i}^{(3)} \xx_{i\pm1} + ``[3]^2" \xx_{i}^{(3)} \\
\xx_{i}^{(3)} \xx_{i\pm1}^{(2)} \xx_{i}^{(3)} 
	&= -``[3][4]" \xx_{i}^{(3)} \\ 
	\end{aligned} \right.
	\end{equation}
\begin{equation}
	\label{eq:k3l}
\xx_{i}^{(a)} \xx_{i\pm1}^{(3)} \xx_{i}^{(c)}	
\left\{
\begin{aligned}
\xx_{i} \xx_{i\pm1}^{(3)} \xx_{i} 
	&= \xx_{i\pm1}^{(2)} \xx_{i}^{(3)} \xx_{i\pm1}^{(2)} \\
\xx_{i}^{(2)} \xx_{i\pm1}^{(3)} \xx_{i} 
	&= \xx_{i\pm1} \xx_{i}^{(3)} \xx_{i\pm1}^{(2)} \\
\xx_{i}^{(3)} \xx_{i\pm1}^{(3)} \xx_{i} 
	&= \xx_{i}^{(3)} \xx_{i\pm1}^{(2)} \\
\xx_{i}^{(3)} \xx_{i\pm1}^{(3)} \xx_{i}^{(2)} 
	&=  \xx_{i}^{(3)} \xx_{i\pm1} \\
\xx_{i}^{(3)} \xx_{i\pm1}^{(3)} \xx_{i}^{(3)} 
	&=  \xx_{i}^{(3)} \\ 
	\end{aligned} \right.
	\end{equation}
as well as those obtained from these by reversing the ordering of every monomial.
The relations for $A_m^2$ are obtained from these by setting $\xx_r^{(3)}=0$.
\end{proof}

The next ingredient in our characterization of $P_{\son}(\mathcal{L}_{\beta}^S)$ 
is a relation between the quantum traces on 
$\End_{U_q(\son)}(S^{\otimes m})$ and $\End_{U_q(\son)}(S^{\otimes m-1})$.
Recall that, in the present setting, 
the quantum traces are the $\C(q^{\frac{1}{2}})$-linear maps
\[
\trq \colon \End_{U_q(\son)}(S^{\otimes m}) \to \End_{U_q(\son)} \big(\C(q^{\frac{1}{2}}) \big) 
									\cong \C(q^{\frac{1}{2}})
\]
that are most easily described on $\mathsf{W} \in \End_{U_q(\son)}(S^{\otimes m})$ 
in the graphical calculus for $\Rep\big(U_q(\son)\big)$ by:
\[
\begin{tikzpicture}[anchorbase]
		\draw[very thick, gray] (-.25,-.5) to (-.25,.5);
		\draw[very thick, gray] (.25,-.5) to (.25,.5);
		\fill[white] (-.375,-.25) rectangle (.375,.25); 
		\draw[thick] (-.375,-.25) rectangle (.375,.25); 
		\node at (0,0) {$\mathsf{W}$};
		\node at (0,.375) {$\mydots$};
		\node at (0,-.375) {$\mydots$};
		\end{tikzpicture}
\xmapsto{\trq}
\begin{tikzpicture}[anchorbase]
		\draw[very thick, gray] (-.25,-.5) to (-.25,.5) to [out=90,in=180] (.5,1.25) to 
			[out=0,in=90] (1.25,.5) to (1.25,-.5) to [out=270,in=0] (.5,-1.25) to 
				[out=180,in=270] (-.25,-.5);
		\draw[very thick, gray] (.25,-.5) to (.25,.5) to [out=90,in=180] (.5,.75) to 
			[out=0,in=90] (.75,.5) to (.75,-.5) to [out=270,in=0] (.5,-.75) to 
				[out=180,in=270] (.25,-.5);
		\fill[white] (-.375,-.25) rectangle (.375,.25); 
		\draw[thick] (-.375,-.25) rectangle (.375,.25); 
		\node at (0,0) {$\mathsf{W}$};
		\node at (0,.375) {$\mydots$};
		\node at (0,-.375) {$\mydots$};
		\node at (1,0) {$\mydots$};
		\end{tikzpicture} \, .
\]
Note that $\trq$ is trace-like with respect to composition in $\End_{U_q(\son)}(S^{\otimes m})$, 
i.e.~$\trq(\mathsf{W}_1 \circ \mathsf{W}_2) = \trq(\mathsf{W}_2 \circ \mathsf{W}_1)$.

\begin{conj}
	\label{conj:TrX}
Let $\iota \colon \End_{U_q(\son)}(S^{\otimes m-1}) \to \End_{U_q(\son)}(S^{\otimes m})$ 
be the inclusion given by $\iota(\mathsf{W}) := \mathsf{W} \otimes \id_S$.
If $0 \leq k \leq n$ and $\mathsf{W} \in \End_{U_q(\son)}(S^{\otimes m-1})$, then
\begin{equation}
	\label{eq:TrX}
\trq \big( \xx_{m-1}^{(k)} \circ \iota(\mathsf{W}) \big)
	= (-1)^{n(k+1)+{n-k \choose 2}} \prod_{t=1}^{n-k} \frac{``[n+1-t][n+t]"}{``[t]^2"}
		\trq(\mathsf{W}) \, .
\end{equation}
\end{conj}

\begin{prop}
	\label{prop:TrX}
Conjecture \ref{conj:TrX} holds when $k=0,n$. It holds 
for all $0 \leq k \leq n$ when $n=1,2,3$.
\end{prop}
\begin{proof}
Graphically, \eqref{eq:TrX} is the equality
\[
\begin{tikzpicture}[anchorbase]
		\draw[very thick, gray] (-.375,-.5) to (-.375,1.25) to [out=90,in=180] (.875,2.5) to 
			[out=0,in=90] (2.125,1.25) to (2.125,-.5) to [out=270,in=0] (.875,-1.75) to 
				[out=180,in=270] (-.375,-.5);
		\draw[very thick, gray] (.125,-.5) to (.125,1.25) to [out=90,in=180] (.875,2) to 
			[out=0,in=90] (1.625,1.25) to (1.625,-.5) to [out=270,in=0] (.875,-1.25) to 
				[out=180,in=270] (.125,-.5);
		\draw[very thick, gray] (.375,-.5) to (.375,1.25) to [out=90,in=180] (.875,1.75) to 
			[out=0,in=90] (1.375,1.25) to (1.375,-.5) to [out=270,in=0] (.875,-1) to 
				[out=180,in=270] (.375,-.5);		
		\draw[very thick, gray] (.625,-.5) to (.625,1.25) to [out=90,in=180] (.875,1.5) to 
			[out=0,in=90] (1.125,1.25) to (1.125,-.5) to [out=270,in=0] (.875,-.75) to 
				[out=180,in=270] (.625,-.5);		
		\fill[white] (-.5,-.25) rectangle (.5,.25); 
		\draw[thick] (-.5,-.25) rectangle (.5,.25);
		\node at (0,0) {$\mathsf{W}$};
		\node at (-.125,.375) {$\mydots$};
		\node at (-.125,-.375) {$\mydots$};
		\fill[white] (.25,.5) rectangle (.75,1); 
		\draw[thick] (.25,.5) rectangle (.75,1); 
		\node at (.51875,.75) {\tiny$\xx^{(k)}$};
		\node at (1.875,.5) {$\mydots$};
		\end{tikzpicture}
= (-1)^{n(k+1)+{n-k \choose 2}} \prod_{t=1}^{n-k} \frac{``[n+1-t][n+t]"}{``[t]^2"}
\begin{tikzpicture}[anchorbase]
		\draw[very thick, gray] (-.25,-.5) to (-.25,.5) to [out=90,in=180] (.5,1.25) to 
			[out=0,in=90] (1.25,.5) to (1.25,-.5) to [out=270,in=0] (.5,-1.25) to 
				[out=180,in=270] (-.25,-.5);
		\draw[very thick, gray] (.25,-.5) to (.25,.5) to [out=90,in=180] (.5,.75) to 
			[out=0,in=90] (.75,.5) to (.75,-.5) to [out=270,in=0] (.5,-.75) to 
				[out=180,in=270] (.25,-.5);
		\fill[white] (-.375,-.25) rectangle (.375,.25); 
		\draw[thick] (-.375,-.25) rectangle (.375,.25); 
		\node at (0,0) {$\mathsf{W}$};
		\node at (0,.375) {$\mydots$};
		\node at (0,-.375) {$\mydots$};
		\node at (1,0) {$\mydots$};
		\end{tikzpicture} \, .
	\]
Since $\xx^{(n)}= \II^{(0)} = \cu[n] \circ \ca[n]$, 
this makes clear that the $k=n$ case follows from \eqref{eq:snake}, 
provided we interpret the product on the right-hand side of \eqref{eq:TrX} as equaling $1$.
Similarly, since $\xx^{(0)} = \id_{S\otimes S}$, 
we see that this reduces the $k=0$ case to the assertion
\[
(-1)^{n+{n \choose 2}} \prod_{t=1}^{n} \frac{``[n+1-t][n+t]"}{``[t]^2"}
=
\begin{tikzpicture}[scale =.5, anchorbase]
	\draw[very thick,gray] (0,0) circle (.5);
\end{tikzpicture} 
\]
which holds by Example \ref{eg:devilscircle}.
The remaining low-rank cases ($n=2$, $k=1$ and $n=3$, $k=1,2$) 
follow from direct computations using the graphical relations 
in $\Rep\big(U_q(\son)\big)$ established in \S \ref{ss:Webson}.
\end{proof}

The astute reader will notice that the product in \eqref{eq:TrX} 
appears in the $i=n$ case of \eqref{lambdacoeff}. 
Indeed, Conjecture \ref{conj:TrX}
is a consequence of Proposition \ref{itoxchangeofbasis} and the following 
more-elemental conjecture (which is obvious when $k=0$ or $n$).

\begin{conj}
	\label{conj:Xrot}
For $0 \leq k \leq n$, the following holds in $\End_{U_q(\son)}(S \otimes S)$:
\[
\begin{tikzpicture}[anchorbase]
		\draw[very thick, gray] (-.75,.625) to (-.75,-.25) to [out=270,in=180]
			(-.5,-.5) to [out=0,in=270] (-.25,-.25) to (-.25,.625);
		\draw[very thick, gray] (.25,-.625) to (.25,.25) to [out=90,in=180] (.5,.5) to 
			[out=0,in=90] (.75,.25) to (.75,-.625);
		\fill[white] (-.375,-.25) rectangle (.375,.25); 
		\draw[thick] (-.375,-.25) rectangle (.375,.25); 
		\node at (0,0) {$\xx^{(k)}$};
		\end{tikzpicture}
=
\begin{tikzpicture}[anchorbase]
		\draw[very thick, gray] (-.25,-.5) to (-.25,.5);
		\draw[very thick, gray] (.25,-.5) to (.25,.5);
		\fill[white] (-.5,-.25) rectangle (.5,.25); 
		\draw[thick] (-.5,-.25) rectangle (.5,.25); 
		\node at (0,0) {\scs$\xx^{(n-k)}$};
		\end{tikzpicture} \, .
\]
\end{conj}

Returning to the task at hand, we finally
arrive at our desired characterization of the spin link polynomials.

\begin{thm}
	\label{thm:characterization}
Fix $n \geq 1$ and suppose that for $m \geq 1$ there are 
unital $\C(q^{\frac{1}{2}})$-algebras $A_m^n$ such that:
\begin{enumerate}
	\item[A1.] (each) $A_m^n$ contain elements 
	$\{\xx_r^{(i)}\}_{\substack{1 \leq r \leq m-1 \\ 0 \leq i \leq n}}$ 
	satisfying $\xx_r^{(0)} = 1$ and
	equations \eqref{eq:xiquad} and \eqref{eq:iSerre},
	\item[A2.] there are $\C(q^{\frac{1}{2}})$-linear maps $\iota \colon A_{m-1}^n \to A_m^n$ 
	that send $\xx_r^{(i)} \mapsto \xx_r^{(i)}$ for $1\leq r\leq m-2$
	and trace-like $\C(q^{\frac{1}{2}})$-linear maps $T_m^n \colon A_m^n \to \C(q^{\frac{1}{2}})$ 
	that satisfy \eqref{eq:TrX},
	\item[A3.] $A_1^n \cong \C(q^{\frac{1}{2}})$ and 
	$T_1^n(1) = (-1)^{\binom{n+1}{2}}\prod_{i=1}^n(q^{2i-1}+ q^{1-2i})$, and
	\item[A4.] the assignment 
	$\bgen_i \mapsto q^{\frac{n}{2}} \sum_{\ell=0}^n q^{-\ell}\xx_i^{(\ell)}$ determines 
	braid group representations $R_m^n \colon \Br_m \to A_m^n$. 
\end{enumerate}
Assuming that
\begin{enumerate}
	\item[$(\ast)$] $n=1,2,3$, or (more generally) 
	that Conjectures \ref{conj:R3} and \ref{conj:TrX} hold
\end{enumerate}
we have that $T_m^n\big(R_m^n(\beta)\big) = P_{\son}(\mathcal{L}_{\beta}^S)$ for any braid $\beta$.
\end{thm}

\begin{proof}
First observe that, under the assumptions $(\ast)$,
Definition \ref{D:xxk-defn} and
Propositions 
\ref{prop:iSerre},
\ref{prop:TrX}, 
and \ref{P:decat-spin-rickard} 
imply that
the remaining hypotheses A1.--A4.~hold 
when $A_m^n = \End_{U_q(\son)}(S^{\otimes m})$ 
and $T_m^n = \trq$.
Since $P_{\son}$ is defined in terms of $\trq$ and $\End_{U_q(\son)}(S^{\otimes m})$, 
the result follows once we show that the hypotheses suffice to compute 
$T_m^n\big(R_m^n(\beta)\big)$.

Remark \ref{rem:inverse} implies that the braid group representation is 
necessarily given on the inverses of the Artin generators as
\[
R_m^n(\bgen_i^{-1}) = q^{-\frac{n}{2}} \sum_{\ell=0}^n q^{\ell}\xx_i^{(\ell)} \, .
\]
Given $\beta \in \Br_m$, the element $R_m^n(\beta) \in A_m^n$ 
is thus a $\C(q^{\frac{1}{2}})$-linear combination of 
words in the elements $\{\xx_r^{(i)}\}_{\substack{1 \leq r \leq m-1 \\ 0 \leq i \leq n}}$, 
so we need only show that the hypotheses suffice to compute $T_m^n(\mathsf{W})$ 
for any such word $\mathsf{W} \in A_m^n$.
For this we argue inductively on $m$, with the base case $m=1$ holding by Hypothesis A3.

Now, let $\mathsf{W} = \xx_{r_1}^{(i_{1})} \cdots \xx_{r_k}^{(i_{k})}$ be a word in the elements 
$\{\xx_r^{(i)}\}_{\substack{1 \leq r \leq m-1 \\ 0 \leq i \leq n}} \subset A_m^n$. 
With $m$ fixed, we now argue, by induction on the length $k$ of the word $\mathsf{W}$,
that we can compute $T_m^n(\mathsf{W})$. For the base case $k=1$, 
we observe that a length-one word necessarily takes the form
$\mathsf{W}=\xx_{m-1}^{(p)} \iota(\mathsf{W}')$ 
with $0 \leq p \leq n$ and $\mathsf{W}'$ a word in 
$\{\xx_r^{(i)}\}_{\substack{1 \leq r \leq m-2 \\ 0 \leq i \leq n}} \subset A_{m-1}^n$.
Hypothesis A2.~then gives that
\[
T_m^n(\mathsf{W}) = 
(-1)^{n(k+1)+{n-k \choose 2}} \prod_{t=1}^{n-k} \frac{``[n+1-t][n+t]"}{``[t]^2"}
		T_{m-1}^n(\mathsf{W}')
\]
which in turn can be computed using the inductive hypothesis (with respect to $m$).
Now suppose that $k \geq 2$. 
Given $\mathsf{W} = \xx_{r_1}^{(i_{1})} \cdots \xx_{r_k}^{(i_{k})}$, 
consider the corresponding word $s_{r_1} \cdots s_{r_k}$ in the standard Coxeter generators 
of the symmetric group $\SG_m$. 
By standard results on (reduced) expressions in the symmetric group \cite[Theorem 3.3.1]{BjBr}, 
it is possible to use a sequence of the moves
\begin{equation}
	\label{eq:SGrels}
s_r s_r \to \emptyset \quad , \quad s_r s_{r\pm1} s_r \leftrightarrow s_{r\pm1} s_r s_{r\pm1}
\end{equation}
to pass from the word 
$s_{r_1} \cdots s_{r_k}$ to a (reduced) word in $\SG_m$ 
wherein the generator $s_{m-1}$ appears only once.
Equation \eqref{eq:xiquad} implies that
\[
\xx_{r}^{(i)} \xx_{r}^{(j)} \in \operatorname{span}_{\C(q^{\frac{1}{2}})} \{\xx_{r}^{(i)}\}_{i=0}^n \,
\]
which, together with Proposition \ref{prop:R3} and Conjecture \ref{conj:R3}, 
implies that $\eqref{eq:SGrels}$ holds in $A_m^n$ up to scalars, 
modulo words of shorter length.
It follows that we can use \eqref{eq:xiquad} and \ref{conj:R3} to obtain
\[
\mathsf{W} = \xi \cdot \iota(\mathsf{W}') \xx_{m-1}^{(p)} \iota(\mathsf{W}'')
	+ \sum_{j=1}^M  \xi_j \cdot \mathsf{W}_j
\]
for $\xi, \xi_j \in \C(q^{\frac{1}{2}})$ and 
where $\mathsf{W}_j$ are words of length strictly less than $k$. 
By induction (on $k$), we can compute $T_m^n(\mathsf{W}_j)$, 
while Hypothesis A2.~allows us to compute
\[
\begin{aligned}
T_m^n \big(\iota(\mathsf{W}') \xx_{m-1}^{(p)} \iota(\mathsf{W}'') \big)
&= T_m^n \big(\xx_{m-1}^{(p)} \iota(\mathsf{W}'')\iota(\mathsf{W}') \big) \\
&= (-1)^{n(k+1)+{n-k \choose 2}} \prod_{t=1}^{n-k} \frac{``[n+1-t][n+t]"}{``[t]^2"}
		T_{m-1}^n(\mathsf{W}''\mathsf{W}')
\end{aligned}
\]
and the result follows (by induction on $m$).
\end{proof}

\begin{example}
The Temperley--Lieb algebras at circle value $-(q+q^{-1})$ satisfy the hypotheses of 
Theorem \ref{thm:characterization} when $n=1$.
This recovers the known fact that the Jones polynomial equals the spin-colored 
$U_q(\son[3])$ link invariant. 
(In fact, the Templerley--Lieb algebras are simply equal to $\End_{U_q(\son[3])}(S^{\otimes m})$.)
\end{example}

We will later use Theorem \ref{thm:characterization} to 
show\footnote{Assuming Conjectures \ref{conj:R3}, \ref{conj:TrX}, and \ref{conj:TrX2} when $n>3$.} 
that the link homology theories defined in Section \ref{s:SLH} 
categorify the spin link polynomials $P_{\son}(\mathcal{L}_{\beta}^S)$.

\begin{rem}
According to Reshetikhin \cite[Proposition 9.3]{Resh-unpub2}, 
the relation from Corollary \ref{C:Rstkn-Reln} 
(combined with the circle relations and the ribbon category relations) 
are sufficient to evaluate any closed braided graph, where each vertex is a grey-grey-black trivalent vertex. 
However, in the present paper, we do not categorify webs with edges colored by the vector representation.
Thus, in order to connect the spin colored $\son$ link polynomial with our categorification, 
we require Theorem \ref{thm:characterization}. 
\end{rem}

%
\section{Background on categorified quantum groups}
	\label{s:BCQG}
%

In this section, we review the categorification of quantum groups in type $A$. 
We consider quantum $\glm$, rather than $\slm$, 
since it is the categorification of the former that is most relevant to our approach to link invariants.

\subsection{The idempotent form of the quantum group}
	\label{ss:idempotent}

Categorified quantum $\glm$ does not actually categorify $U_q(\glm)$;
rather, it  categorifies the following close relative which was first studied in \cite{BLM,Lus4}.

\begin{defn}
	\label{def:Udotgl}
The \emph{idempotent quantum group} 
$\dU(\glm)$ is the (non-unital) $\C(q)$-algebra 
generated by mutually orthogonal idempotents $\one_{\wt}$ for $\wt \in \Z^{m}$
and elements $\one_{\wt+ \alpha_i} e_i \one_{\wt}$ and $\one_{\wt- \alpha_i} f_i \one_{\wt}$
for $1 \leq i \leq m$.
Set $e_i^{(r)} := \frac{1}{[r]!} e_i^r$ and $f_i^{(r)} := \frac{1}{[r]!} f_i^r$, 
then the relations are as follows:
\begin{itemize}
\item $e_i^{(r)}f_i^{(s)}\one_{\wt} = \sum_t {\alpha_i^{\vee}(\wt) + r-s \brack t}f_i^{(s-t)}e_i^{(r-t)}\one_{\wt}$,
\item for $i \ne j$, $e_i^{(r)}f_j^{(s)}\one_{\wt} = f_j^{(s)}e_i^{(r)}\one_{\wt}$,
\item for $|i-j|=1$, $e_i e_j e_i\one_{\wt} = (e_i^{(2)} e_j + e_j e_i^{(2)})\one_{\wt}$ and
		$f_i f_j f_i\one_{\wt} = (f_i^{(2)} f_j + f_j f_i^{(2)})\one_{\wt}$,
\item for $|i-j|>1$,  $e_i^{(r)} e_j^{(s)}\one_{\wt} = e_j^{(s)}e_i^{(r)}\one_{\wt}$ and
		$f_i^{(r)} f_j^{(s)}\one_{\wt} = f_j^{(s)} f_i^{(r)}\one_{\wt}$, and
\item $e_i^{(s)} e_i^{(r)}\one_{\wt} = {r+s \brack r}e_i^{(r+s)}\one_{\wt}$ and
		$f_i^{(s)} f_i^{(r)}\one_{\wt} = {r+s \brack r} f_i^{(r+s)}\one_{\wt}$.
\end{itemize}

The \emph{integral idempotent quantum group} 
$\ZdU(\glm)$ is the (non-unital) $\Z[q^{\pm}]$-subalgebra of $\dU(\glm)$
generated by mutually orthogonal idempotents $\one_{\wt}$ for $\wt \in \Z^{m}$
and elements $\one_{\wt+ r\alpha_i} e_i^{(r)} \one_{\wt}$ and $\one_{\wt- \alpha_i} f_i^{(r)} \one_{\wt}$
for $1 \leq i \leq m$ and $r\ge 1$. \end{defn}

Note that the (non-integral) idempotent quantum group is recovered from the integral version 
as $\dU(\glm) := \C(q) \otimes_{\Z[q^{\pm}]} \ZdU(\glm)$.

\begin{rem}\label{R:dU-is-category}
Since $\dU(\glm)$ is an algebra equipped with a system of mutually orthogonal idempotents 
indexed by $\wt \in \Z^m$, we can consider it as a category wherein the objects are $\wt \in \Z^m$ 
and $\Hom_{\dU(\glm)}(\wt,\wtt)$ consists of elements of the form $\one_{\wtt} x \one_\wt$.
\end{rem}

\subsection{The categorified quantum group}
	\label{ss:CQG}
	
For the rest of this chapter, let $\K$ be an integral domain.
We will explicitly assume $\K$ is a field when we state results about the Grothendieck group.

We next define our version of the categorified quantum group $\UU_q(\glm)$.
The original definition of $\UU_q(\glm)$ given in \cite{MSV2} by Mackaay--Sto\v{s}i\'{c}--Vaz (MSV) is 
a (signed) analogue of the 
Khovanov--Lauda--Rouquier \cite{KL3,Rou2} categorified quantum group $\UU_q(\slm)$ wherein 
the $\slm$ weight lattice (the set of objects in $\UU_q(\slm)$) 
is replaced by the $\glm$ weight lattice.
Our version is related, but not identical, to that used by MSV, 
and we now clarify the differences for the expert.  
Recall that bubbles are certain 2-morphisms 
living in the endomorphism algebra of an identity $1$-morphism.
The two differences are as follows:
\begin{itemize}
\item 
In MSV, bubbles may be identified with symmetric functions in an alphabet of variables. 
From the perspective of the $\gln$ foam $2$-category in \cite{QR1}, 
which is related to $\UU_q(\glm)$ via categorical skew Howe duality, 
one should consider symmetric functions in $m$ alphabets, 
and these bubbles should be viewed as symmetric functions in a \emph{difference} of two adjacent alphabets.
We extend scalars in the endomorphism algebras of all identity $1$-morphisms accordingly. 
This modification has previously appeared (in a slightly different guise) in work of Webster \cite{WebSchur}.

\item We work with a different orientation for the $\slm$ Dynkin diagram, 
and hence different scalars associated to degree-zero bubbles. \end{itemize}

\begin{rem} \label{rmk:Laudaparams} 
Experts may be familiar with Lauda's paper \cite{LaudaParameters}, 
which parametrizes all choices of scalars one might use for the degree-zero bubbles. 
In the conventions thereof, for each $\glm$ weight $\wt = (a_1, \ldots, a_m) \in \Z^m$ we set
\begin{subequations}
\begin{equation} 
		\label{eq:ourbubbles}
	c_{i,\wt}^- = (-1)^{{a_i}} \, , \quad c_{i,\wt}^+ = (-1)^{{a_i}-1}. \end{equation}
Via \cite[Equations (2.1), (2.2)]{LaudaParameters}, 
this determines the parameters
\begin{equation} t_{i,i}=-1 \, , \quad t_{i,i-1}=-1 \, , \quad t_{i,i+1}=1 \, , \quad t_{i,k}=1 \text{ otherwise.} \end{equation}
\end{subequations}
Note that \eqref{eq:ourbubbles} are the values of the anti-clockwise and clockwise degree-zero bubbles, respectively.
\end{rem}

Before proceeding, we note that
neither of the aforementioned differences affect any salient properties of $\UU_q(\glm)$, 
and the following categorification result of Khovanov--Lauda still holds.

\begin{thm}[{\cite{KL3}}]
	\label{thm:KL}
Suppose $\K$ is a field. 
There is an isomorphism of $\C(q)$-algebras
\[
\C(q) \otimes_{\Z[q^\pm]} \Kzero{\Kar(\UU_q(\glm))} \cong \dU(\glm) \, . 
\]
\qed
\end{thm}

We now give our definition of $\UU_q(\glm)$.

\begin{defn}
	\label{def:CQG}
Let $m \geq 1$. The categorified quantum group $\UU_q(\glm)$ is the 
$\Z$-additive closure of the $\Z$-graded $\K$-linear $2$-category given as follows.
\begin{itemize}[leftmargin=*]
\item \textbf{Objects} are elements $\wt \in \Z^m$.
\item \textbf{$1$-morphisms} are generated by 
\[
\EE_i \one_{\wt} \colon \wt \to \wt + \ee_i \, , \quad
\FF_i \one_{\wt}  \colon \wt \to \wt - \ee_i,
\]
for $1 \le i \le m-1$. Here $\ee_i = (0,\ldots,1,-1,\ldots,0)$.
\item \textbf{$2$-morphisms} are $\K$-linear combinations of ``string diagrams'' that are generated via 
horizontal and vertical composition by the generators:
\begin{equation}
	\label{eq:CQGgens}
\begin{gathered}
\begin{tikzpicture}[anchorbase,scale=1]
\draw[thick,->] (0,0) node[below=-1pt]{\scs $i$} to node[black]{$\bullet$} (0,1);
\node at (.5,.75){$\wt$};
\node at (-.75,.75){$\wt{+}\ee_i$};
\end{tikzpicture}
\in \End^2(\EE_i \one_{\wt})
\, , \quad
\begin{tikzpicture}[anchorbase,scale=1]
\draw[thick,->] (0,0) node[below=-1pt]{\scs $i$} to [out=90,in=270] (.5,1);
\draw[thick,->] (.5,0) node[below=-1pt]{\scs $j$} to [out=90,in=270] (0,1);
\node at (.875,.75){$\wt$};
\node at (-1,.75){$\wt{+}\ee_i{+}\ee_j$};
\end{tikzpicture}
\in \Hom^{- i \cdot j}(\EE_i \EE_j \one_\wt, \EE_j \EE_i \one_\wt) \, , \\
\begin{tikzpicture}[anchorbase,scale=1]
\draw[thick,<-] (-.25,0) to [out=90,in=180] (0,.5) 
	to [out=0,in=90] (.25,0) node[below=-1pt]{\scs $i$};
\node at (.5,.5){$\wt$};
\end{tikzpicture}
\in \Hom^{1+a_i-a_{i+1}}(\FF_i \EE_i \one_\wt, \one_\wt)
\, , \quad
\begin{tikzpicture}[anchorbase,scale=1]
\draw[thick,->] (-.25,0) node[below=-1pt]{\scs $i$} to [out=90,in=180] (0,.5) 
	to [out=0,in=90] (.25,0);
\node at (.5,.5){$\wt$};
\end{tikzpicture}
\in \Hom^{1-a_i+a_{i+1}}(\EE_i \FF_i \one_\wt, \one_\wt) \, , \\
\begin{tikzpicture}[anchorbase,scale=1]
\draw[thick,->] (-.25,0) node[above=-1pt]{\scs $i$} to [out=270,in=180] (0,-.5) 
	to [out=0,in=270] (.25,0);
\node at (.5,-.5){$\wt$};
\end{tikzpicture}
\in \Hom^{1+a_i-a_{i+1}}(\one_\wt , \FF_i \EE_i \one_\wt)
\, , \quad
\begin{tikzpicture}[anchorbase,scale=1]
\draw[thick,<-] (-.25,0) to [out=270,in=180] (0,-.5) 
	to [out=0,in=270] (.25,0) node[above=-1pt]{\scs $i$};
\node at (.5,-.5){$\wt$};
\end{tikzpicture}
\in \Hom^{1-a_i+a_{i+1}}(\one_\wt , \EE_i \FF_i \one_\wt)
\end{gathered}
\end{equation}
and for each $f \in \Sym(\X_1| \cdots |\X_m)$, a \emph{new bubble} generator in each weight:
\begin{equation}
	\label{eq:newbub}
\begin{tikzpicture}[anchorbase,scale=1]
\node at (0,0){\NB{f}};
\node at (.5,.25){$\wt$};
\end{tikzpicture}
\in \End^{2 \deg(f)} (\one_\wt) \, .
\end{equation}
Here, the alphabets $\X_1,\ldots,\X_m$ are viewed as formal.
That is, for each $1 \leq i \leq m$, $\Sym(\X_i)$ is the ring of symmetric functions, 
and $\Sym(\X_1| \cdots |\X_m) \cong \bigotimes_{i=1}^m \Sym(\X_i)$.
\end{itemize}
These $2$-morphisms are subject to the following local relations:
\begin{enumerate}[leftmargin=*]
\item \textbf{Adjunction and cyclicity:} diagrams that are related by planar 
isotopy (rel. boundary) are equal. This is guaranteed by the following relations:
\begin{enumerate}
\item The cap and cup morphisms in \eqref{eq:CQGgens}
are the units and counits for (graded) biadjunctions 
between $\EE_i \one_\wt$ and $\FF_i \one_{\wt+\ee_i}$.
\item The dot and crossing morphisms in \eqref{eq:CQGgens} 
are cyclic with respect to this biadjoint structure, 
meaning that the two ways of building a dotted downward strand are equal, 
as are the two ways of building a downward crossing.
\end{enumerate}

\item \textbf{Dot slide:}
\begin{equation} \label{dotslide}
\begin{tikzpicture}[anchorbase,scale=1]
\draw[thick,->] (0,0) node[below=-1pt]{\scs $i$} to [out=90,in=270] (.5,1);
\draw[thick,->] (.5,0) node[below=-1pt]{\scs $j$} to [out=90,in=270] 
	node[pos=.7,yshift=-.5pt]{$\bullet$} (0,1);
\node at (.875,.75){$\wt$};
\end{tikzpicture}
-
\begin{tikzpicture}[anchorbase,scale=1]
\draw[thick,->] (0,0) node[below=-1pt]{\scs $i$} to [out=90,in=270] (.5,1);
\draw[thick,->] (.5,0) node[below=-1pt]{\scs $j$} to [out=90,in=270] 
	node[pos=.3,yshift=-.5pt]{$\bullet$} (0,1);
\node at (.875,.75){$\wt$};
\end{tikzpicture}
=
\begin{tikzpicture}[anchorbase,scale=1]
\draw[thick,->] (0,0) node[below=-1pt]{\scs $i$} to [out=90,in=270] 
	node[pos=.3,yshift=-.5pt]{$\bullet$}(.5,1);
\draw[thick,->] (.5,0) node[below=-1pt]{\scs $j$} to [out=90,in=270] (0,1);
\node at (.875,.75){$\wt$};
\end{tikzpicture}
-
\begin{tikzpicture}[anchorbase,scale=1]
\draw[thick,->] (0,0) node[below=-1pt]{\scs $i$} to [out=90,in=270] 
	node[pos=.7,yshift=-.5pt]{$\bullet$}(.5,1);
\draw[thick,->] (.5,0) node[below=-1pt]{\scs $j$} to [out=90,in=270] (0,1);
\node at (.875,.75){$\wt$};
\end{tikzpicture}
=
\begin{cases}
\begin{tikzpicture}[anchorbase]
\draw[thick,->] (0,0) node[below=-1pt]{\scs $i$} to (0,1);
\draw[thick,->] (.375,0) node[below=-1pt]{\scs $i$} to (.375,1);
\node at (.75,.75){$\wt$};
\end{tikzpicture} & \text{if } i=j \\ \\
0 & \text{else}
\end{cases}
\end{equation}

\item \textbf{Quadratic KLR relation:}
\begin{equation} \label{eq:quadKLR}
\begin{tikzpicture}[anchorbase,yscale=1]
\draw[thick,->] (0,0) node[below=-1pt]{\scs $i$} to [out=90,in=270] (.5,.75) 
	to [out=90,in=270] (0,1.5);
\draw[thick,->] (.5,0) node[below=-1pt]{\scs $j$} to [out=90,in=270] (0,.75) 
	to [out=90,in=270] (.5,1.5);
\node at (.75,1.25){$\wt$};
\end{tikzpicture}
=
\begin{cases}
0 & \text{if } i=j \\ \\
(j-i)
\left(
\begin{tikzpicture}[anchorbase]
\draw[thick,->] (0,0) node[below=-1pt]{\scs $i$} to 
	node{$\bullet$} (0,1);
\draw[thick,->] (.375,0) node[below=-1pt]{\scs $j$} to (.375,1);
\node at (.75,.75){$\wt$};
\end{tikzpicture}
-
\begin{tikzpicture}[anchorbase]
\draw[thick,->] (0,0) node[below=-1pt]{\scs $i$} to (0,1);
\draw[thick,->] (.375,0) node[below=-1pt]{\scs $j$} to 
	node {$\bullet$} (.375,1);
\node at (.75,.75){$\wt$};
\end{tikzpicture} 
\right) & \text{if } i \cdot j = -1 \\
\begin{tikzpicture}[anchorbase]
\draw[thick,->] (0,0) node[below=-1pt]{\scs $i$} to (0,1);
\draw[thick,->] (.375,0) node[below=-1pt]{\scs $j$} to (.375,1);
\node at (.75,.75){$\wt$};
\end{tikzpicture} & \text{if } i \cdot j = 0 \\
\end{cases}
\end{equation}

\item \textbf{Cubic KLR relation:}
\begin{equation} \label{eq:cubicKLR}
\begin{tikzpicture}[anchorbase,yscale=.75]
\draw[thick,->] (0,0) node[below=-1pt]{\scs $i$} to [out=90,in=270] (1,2);
\draw[thick,->] (.5,0) node[below=-1pt]{\scs $j$} to [out=90,in=270] (0,1) 
	to [out=90,in=270] (.5,2);
\draw[thick,->] (1,0) node[below=-1pt]{\scs $k$} to [out=90,in=270] (0,2);
\node at (1.25,1.5){$\wt$};
\end{tikzpicture}
-
\begin{tikzpicture}[anchorbase,yscale=.75]
\draw[thick,->] (0,0) node[below=-1pt]{\scs $i$} to [out=90,in=270] (1,2);
\draw[thick,->] (.5,0) node[below=-1pt]{\scs $j$} to [out=90,in=270] (1,1) 
	to [out=90,in=270] (.5,2);
\draw[thick,->] (1,0) node[below=-1pt]{\scs $k$} to [out=90,in=270] (0,2);
\node at (1.25,1.5){$\wt$};
\end{tikzpicture}
=
\begin{cases}
(j-i)
\begin{tikzpicture}[anchorbase,yscale=.75]
\draw[thick,->] (0,0) node[below=-1pt]{\scs $i$} to [out=90,in=270] (0,2);
\draw[thick,->] (.5,0) node[below=-1pt]{\scs $j$} to (.5,2);
\draw[thick,->] (1,0) node[below=-1pt]{\scs $i$} to [out=90,in=270] (1,2);
\node at (1.25,1.5){$\wt$};
\end{tikzpicture} & \text{if } k=i \text{ and } i \cdot j = -1 \\ \\
0 & \text{else}
\end{cases}
\end{equation}

\item \textbf{New bubble relations:}
\begin{enumerate}
\item \textbf{Algebra compatibility:}
\[
\begin{tikzpicture}[anchorbase,scale=1]
\node at (-.5,0){\NB{f}};
\node at (0,0){\NB{g}};
\node at (.5,.25){$\wt$};
\end{tikzpicture} 
=
\begin{tikzpicture}[anchorbase,scale=1]
\node at (0,0){\NB{fg}};
\node at (.5,.25){$\wt$};
\end{tikzpicture},
\qquad
\begin{tikzpicture}[anchorbase,scale=1]
\node at (0,0){\NB{f}};
\node at (.5,.25){$\wt$};
\end{tikzpicture}
+
\begin{tikzpicture}[anchorbase,scale=1]
\node at (0,0){\NB{g}};
\node at (.5,.25){$\wt$};
\end{tikzpicture}
=
\begin{tikzpicture}[anchorbase,scale=1]
\node at (0,0){\NB{f+g}};
\node at (.5,.25){$\wt$};
\end{tikzpicture}
\]

\item \textbf{Bubbles to new bubbles:}
Let us define spaded bubble notation:
\begin{equation} \label{realbubdef}
\begin{tikzpicture}[anchorbase,scale=1.25]
\draw[thick,->] (.25,0) arc[start angle=0,end angle=360,radius=.25] 
	node[pos=.5,left=-1pt]{\scs$i$} node[pos=.75,black]{$\bullet$} 
		node[pos=.75,black,below]{\scs$\spadesuit{+}r$};
\node at (.375,.375){$\wt$};
\end{tikzpicture}
:=
\begin{tikzpicture}[anchorbase,scale=1.25]
\draw[thick,->] (.25,0) arc[start angle=0,end angle=360,radius=.25] 
	node[pos=.5,left=-1pt]{\scs$i$} node[pos=.75,black]{$\bullet$} 
		node[pos=.75,black,below]{\scs$r{-}1{-}a_i{+}a_{i+1}$};
\node at (.375,.375){$\wt$};
\end{tikzpicture}
\in \End^{2r}(\one_\wt) \, , \quad
\begin{tikzpicture}[anchorbase,scale=1.25]
\draw[thick,<-] (.25,0) arc[start angle=0,end angle=360,radius=.25] 
	node[pos=.5,left=-1pt]{\scs$i$} node[pos=.75,black]{$\bullet$} 
		node[pos=.75,black,below]{\scs$\spadesuit{+}r$};
\node at (.375,.375){$\wt$};
\end{tikzpicture}
:=
\begin{tikzpicture}[anchorbase,scale=1.25]
\draw[thick,<-] (.25,0) arc[start angle=0,end angle=360,radius=.25] 
	node[pos=.5,left=-1pt]{\scs$i$} node[pos=.75,black]{$\bullet$} 
		node[pos=.75,black,below]{\scs$r{-}1{+}a_i{-}a_{i+1}$};
\node at (.375,.375){$\wt$};
\end{tikzpicture}
\in \End^{2r}(\one_\wt).
\end{equation}
When the number of dots is a non-negative integer, this picture is called a \emph{real bubble}, 
and represents a genuine composition of a cup, a cap, and a number of dots. 
When the number of dots is a negative integer, both sides of \eqref{realbubdef} 
are formal symbols, often called \emph{fake bubbles}. 
The following equation \eqref{newbub} serves two purposes: 
it gives a definition of fake bubbles as genuine morphisms, 
and it provides a relation between real bubbles and symmetric functions: 

\begin{equation} \label{newbub}
\begin{tikzpicture}[anchorbase,scale=1.25]
\draw[thick,->] (.25,0) arc[start angle=0,end angle=360,radius=.25] 
	node[pos=.5,left=-1pt]{\scs$i$} node[pos=.75,black]{$\bullet$} 
		node[pos=.75,black,below]{\scs$\spadesuit{+}r$};
\node at (.375,.375){$\wt$};
\end{tikzpicture}
=
(-1)^{a_i}
\begin{tikzpicture}[anchorbase,scale=1]
\node at (0,0){\NB{h_r(\X_{i+1} - \X_i)}};
\node at (1.5,.25){$\wt$};
\end{tikzpicture}
\, , \quad
\begin{tikzpicture}[anchorbase,scale=1.25]
\draw[thick,<-] (.25,0) arc[start angle=0,end angle=360,radius=.25] 
	node[pos=.5,left=-1pt]{\scs$i$} node[pos=.75,black]{$\bullet$} 
		node[pos=.75,black,below]{\scs$\spadesuit{+}r$};
\node at (.375,.375){$\wt$};
\end{tikzpicture}
=
(-1)^{{a_i}-1}
\begin{tikzpicture}[anchorbase,scale=1]
\node at (0,0){\NB{h_r(\X_{i} - \X_{i+1})}};
\node at (1.5,.25){$\wt$};
\end{tikzpicture}
\, .
\end{equation}
Thus when $r < 0$, a spaded bubble with decoration $\spadesuit{+}r$ is zero.

\item \textbf{Dots to new bubbles:}
\begin{equation} \label{dotstonewbubbles}
\begin{tikzpicture}[anchorbase,yscale=1]
\draw[thick,->] (0,0) node[below=-1pt]{\scs $i$} to [out=90,in=270] node{$\bullet$} (0,1.5);
\end{tikzpicture}
=
\begin{tikzpicture}[anchorbase,yscale=1]
\draw[thick,->] (0,0) node[below=-1pt]{\scs $i$} to [out=90,in=270] (0,1.5);
\node at (-.75,.75){\NB{e_1(\X_i)}};
\end{tikzpicture}
-
\begin{tikzpicture}[anchorbase,yscale=1]
\draw[thick,->] (0,0) node[below=-1pt]{\scs $i$} to [out=90,in=270] (0,1.5);
\node at (.75,.75){\NB{e_1(\X_i)}};
\end{tikzpicture}
=
\begin{tikzpicture}[anchorbase,yscale=1]
\draw[thick,->] (0,0) node[below=-1pt]{\scs $i$} to [out=90,in=270] (0,1.5);
\node at (.875,.75){\NB{e_1(\X_{i+1})}};
\end{tikzpicture}
-
\begin{tikzpicture}[anchorbase,yscale=1]
\draw[thick,->] (0,0) node[below=-1pt]{\scs $i$} to [out=90,in=270] (0,1.5);
\node at (-.875,.75){\NB{e_1(\X_{i+1})}};
\end{tikzpicture}
\end{equation}
\end{enumerate}

\item \textbf{Extended $\sln[2]$ relations:}
\begin{gather} \label{extendedreln1}
\begin{tikzpicture}[anchorbase,yscale=1]
\draw[thick,->] (0,0) node[below=-1pt]{\scs $i$} to [out=90,in=270] (0,1.5);
\draw[thick,<-] (.5,0) node[below=-1pt]{\scs $i$} to [out=90,in=270] (.5,1.5);
\node at (.75,1.25){$\wt$};
\end{tikzpicture}
=
\begin{tikzpicture}[anchorbase,yscale=1]
\draw[thick,->] (0,0) node[below=-1pt]{\scs $i$} to [out=90,in=270] (.5,.75) 
	to [out=90,in=270] (0,1.5);
\draw[thick,<-] (.5,0) node[below=-1pt]{\scs $i$} to [out=90,in=270] (0,.75) 
	to [out=90,in=270] (.5,1.5);
\node at (.75,1.25){$\wt$};
\end{tikzpicture}
-
\displaystyle \sum_{\substack{p+q+r = \\ a_i -a_{i+1}-1}}
\begin{tikzpicture}[anchorbase,yscale=1]
\draw[thick,->] (0,0) node[below=-1pt]{\scs $i$} to [out=90,in=180] (.25,.5) 
	to [out=0,in=90] node{$\bullet$} node[above,xshift=1pt]{\scs$q$} (.5,0);
\draw[thick,->] (1.5,.75) arc[start angle=0,end angle=360,radius=.25] 
	node[pos=.5,left=-1pt]{\scs$i$} node[pos=.75,black]{$\bullet$} 
		node[pos=.75,black,below]{\scs$\spadesuit{+}r$};
\draw[thick,<-] (0,1.5) to [out=270,in=180] (.25,1) 
	to [out=0,in=270] node{$\bullet$} node[below,xshift=1pt]{\scs$p$} 
		(.5,1.5) node[above=-1pt]{\scs $i$};
\node at (1.5,1.25){$\wt$};
\end{tikzpicture} \\
\label{extendedreln2}
\begin{tikzpicture}[anchorbase,yscale=1]
\draw[thick,<-] (0,0) node[below=-1pt]{\scs $i$} to [out=90,in=270] (0,1.5);
\draw[thick,->] (.5,0) node[below=-1pt]{\scs $i$} to [out=90,in=270] (.5,1.5);
\node at (.75,1.25){$\wt$};
\end{tikzpicture}
=
\begin{tikzpicture}[anchorbase,yscale=1]
\draw[thick,<-] (0,0) node[below=-1pt]{\scs $i$} to [out=90,in=270] (.5,.75) 
	to [out=90,in=270] (0,1.5);
\draw[thick,->] (.5,0) node[below=-1pt]{\scs $i$} to [out=90,in=270] (0,.75) 
	to [out=90,in=270] (.5,1.5);
\node at (.75,1.25){$\wt$};
\end{tikzpicture}
-
\displaystyle \sum_{\substack{p+q+r = \\ -a_i+a_{i+1}-1}}
\begin{tikzpicture}[anchorbase,yscale=1]
\draw[thick,<-] (0,0) to [out=90,in=180] (.25,.5) to [out=0,in=90] 
	node{$\bullet$} node[above,xshift=1pt]{\scs$q$} (.5,0)
		node[below=-1pt]{\scs $i$};
\draw[thick,<-] (1.5,.75) arc[start angle=0,end angle=360,radius=.25] 
	node[pos=.5,left=-1pt]{\scs$i$} node[pos=.75,black]{$\bullet$} 
		node[pos=.75,black,below]{\scs$\spadesuit{+}r$};
\draw[thick,->] (0,1.5) node[above=-1pt]{\scs $i$} to [out=270,in=180] (.25,1) 
	to [out=0,in=270] node{$\bullet$} node[below,xshift=1pt]{\scs$p$} (.5,1.5);
\node at (1.5,1.25){$\wt$};
\end{tikzpicture}
\end{gather}
(By convention, the summations on the right-hand side are zero when this requires $p+q+r < 0$.)
\item \textbf{Mixed $\EE,\FF$ relations:} for $i \neq j$, 
\begin{equation} \label{eq:mixedef}
\begin{tikzpicture}[anchorbase,yscale=1]
\draw[thick,->] (0,0) node[below=-1pt]{\scs $i$} to [out=90,in=270] (0,1.5);
\draw[thick,<-] (.5,0) node[below=-1pt]{\scs $j$} to [out=90,in=270] (.5,1.5);
\node at (.75,1.25){$\wt$};
\end{tikzpicture}
=
\begin{tikzpicture}[anchorbase,yscale=1]
\draw[thick,->] (0,0) node[below=-1pt]{\scs $i$} to [out=90,in=270] (.5,.75) 
	to [out=90,in=270] (0,1.5);
\draw[thick,<-] (.5,0) node[below=-1pt]{\scs $j$} to [out=90,in=270] (0,.75) 
	to [out=90,in=270] (.5,1.5);
\node at (.75,1.25){$\wt$};
\end{tikzpicture}
\, , \quad
\begin{tikzpicture}[anchorbase,yscale=1]
\draw[thick,<-] (0,0) node[below=-1pt]{\scs $i$} to [out=90,in=270] (0,1.5);
\draw[thick,->] (.5,0) node[below=-1pt]{\scs $j$} to [out=90,in=270] (.5,1.5);
\node at (.75,1.25){$\wt$};
\end{tikzpicture}
=
\begin{tikzpicture}[anchorbase,yscale=1]
\draw[thick,<-] (0,0) node[below=-1pt]{\scs $i$} to [out=90,in=270] (.5,.75) 
	to [out=90,in=270] (0,1.5);
\draw[thick,->] (.5,0) node[below=-1pt]{\scs $j$} to [out=90,in=270] (0,.75) 
	to [out=90,in=270] (.5,1.5);
\node at (.75,1.25){$\wt$};
\end{tikzpicture}
\end{equation}
\end{enumerate}
\end{defn}

\begin{rem}
We collect several easy consequences of the new bubble relations. 
Most important are the standard \textbf{Bubble relations}, 
which are signed versions of the relations originally appearing in \cite{Lau1}.
\begin{enumerate}
\item \textbf{Non-positive degree relations:}
\[
\begin{tikzpicture}[anchorbase,scale=1.25]
\draw[thick,->] (.25,0) arc[start angle=0,end angle=360,radius=.25] 
	node[pos=.5,left=-1pt]{\scs$i$} node[pos=.75,black]{$\bullet$} 
		node[pos=.75,black,below]{\scs$\spadesuit{+}r$};
\node at (.375,.375){$\wt$};
\end{tikzpicture}
=
\begin{cases}
0 & \text{if } r < 0 \\
(-1)^{a_{i}} & \text{if } r = 0
\end{cases} \, , \quad
\begin{tikzpicture}[anchorbase,scale=1.25]
\draw[thick,<-] (.25,0) arc[start angle=0,end angle=360,radius=.25] 
	node[pos=.5,left=-1pt]{\scs$i$} node[pos=.75,black]{$\bullet$} 
		node[pos=.75,black,below]{\scs$\spadesuit{+}r$};
\node at (.375,.375){$\wt$};
\end{tikzpicture}
=
\begin{cases}
0 & \text{if } r < 0 \\
(-1)^{a_{i}-1} & \text{if } r = 0
\end{cases}
\]
Compare with Remark \ref{rmk:Laudaparams}.

\item \textbf{Infinite Grassmannian relation:} the equality
\[
\left(
\sum_{r=0}^\infty 
\begin{tikzpicture}[anchorbase,scale=1.25]
\draw[thick,->] (.25,0) arc[start angle=0,end angle=360,radius=.25] 
	node[pos=.5,left=-1pt]{\scs$i$} node[pos=.75,black]{$\bullet$} 
		node[pos=.75,black,below]{\scs$\spadesuit{+}r$};
\node at (.375,.375){$\wt$};
\end{tikzpicture}
t^r \right)
\left(
\sum_{r=0}^\infty 
\begin{tikzpicture}[anchorbase,scale=1.25]
\draw[thick,<-] (.25,0) arc[start angle=0,end angle=360,radius=.25] 
	node[pos=.5,left=-1pt]{\scs$i$} node[pos=.75,black]{$\bullet$} 
		node[pos=.75,black,below]{\scs$\spadesuit{+}r$};
\node at (.375,.375){$\wt$};
\end{tikzpicture}
t^r \right) = -1
\]
holds in the formal power series ring $\End(\one_\wt)[[t]]$.
\end{enumerate}

Lastly, 
\textbf{symmetric generators slide:} 
let $f \in \Sym(\X_1| \cdots | \X_i+\X_{i+1} | \cdots | \X_m)$, then
\begin{equation}
	\label{eq:symslide}
\begin{tikzpicture}[anchorbase,yscale=1]
\draw[thick,->] (0,0) node[below=-1pt]{\scs $i$} to [out=90,in=270] (0,1.5);
\node at (.5,.75){\NB{f}};
\node at (1,1.25){$\wt$};
\end{tikzpicture}
=
\begin{tikzpicture}[anchorbase,yscale=1]
\draw[thick,->] (0,0) node[below=-1pt]{\scs $i$} to [out=90,in=270] (0,1.5);
\node at (-.5,.75){\NB{f}};
\node at (.5,1.25){$\wt$};
\end{tikzpicture} \, .
\end{equation}
\end{rem}

\begin{rem}
	\label{rem:gl1}
When $m=1$, there are no generating $1$-morphisms $\EE_i \one_{\wt}$ or $\FF_i \one_{\wt}$,
so all $1$-morphisms in $\UU_q(\gln[1])$ are (direct sums of shifts of) identity $1$-morphisms $\one_{\wt}$.
Hence, $\End(\one_{\wt}) \cong \Sym(\X_1)$ is the $\Z$-graded $\K$-algebra of new bubbles \eqref{eq:newbub}.

By convention, $\UU_q(\gln[0])$ is a $2$-category with one object 
(the zero weight in a zero-dimensional lattice). 
Again all $1$-morphisms are direct sums of shifts of the identity $1$-morphism $\one_0$,
but now $\End(\one_{0}) = \K$. 
\end{rem}

\begin{rem}
Note that each strand in a string diagram in $\UU_q(\glm)$
carries a label $i \in \{1,\ldots,m-1\}$, i.e.~a label by a node of 
the $\glm$ Dynkin diagram.
We will refer to these labels as \emph{colors}, 
since later we will depict them using colored strands. 
See Convention \ref{conv:colors}.
\end{rem}

\begin{rem}
	\label{rem:monoidal2cat}
The $2$-category $\coprod_{m \geq 0} \UU_q(\glm)$ is monoidal, 
via the \emph{external tensor product} 
\[
\boxtimes \colon \UU_q(\gln[m_1]) \times \UU_q(\gln[m_2]) \to \UU_q(\gln[m_1+m_2])
\]
which is given on objects by 
\[
\big( (a_1,\ldots,a_{m_1}) , (b_1,\ldots,b_{m_2}) \big) \mapsto (a_1,\ldots,a_{m_1}, b_1,\ldots,b_{m_2})
\]
and on $1$- and $2$-morphisms by ``raising'' the colors on the second factor by $m_{1}$ 
and then concatenating the $1$- and $2$-morphisms. 
For example, 
$\boxtimes \colon \UU_q(\gln[2]) \times \UU_q(\gln[3]) \to \UU_q(\gln[5])$
sends
\[
\big( \EE_1 \one_{(a,b)} , \FF_1 \EE_2 \one_{(c,d,e)} \big) \mapsto \EE_1 \FF_3 \EE_4 \one_{(a,b,c,d,e)}
\quad , \quad
\Bigg(
\begin{tikzpicture}[anchorbase,scale=.75]
\draw[thick,->] (0,0) node[below=-1pt]{\scs $1$} to node[black]{$\bullet$} (0,1);
\node at (.625,.75){\scs$(a,b)$};
\end{tikzpicture}
\ , \
\begin{tikzpicture}[anchorbase,scale=.75]
\draw[thick,<-] (0,0) node[below=-1pt]{\scs $1$} to [out=90,in=270] (.5,1);
\draw[thick,->] (.5,0) node[below=-1pt]{\scs $2$} to [out=90,in=270] (0,1);
\node at (1.25,.75){\scs$(c,d,e)$};
\end{tikzpicture}
\Bigg)
\mapsto
\begin{tikzpicture}[anchorbase,scale=.75]
\draw[thick,->] (0,0) node[below=-1pt]{\scs $1$} to node[black]{$\bullet$} (0,1);
\draw[thick,<-] (.5,0) node[below=-1pt]{\scs $3$} to [out=90,in=270] (1,1);
\draw[thick,->] (1,0) node[below=-1pt]{\scs $4$} to [out=90,in=270] (.5,1);
\node at (2.25,.75){\scs$(a,b,c,d,e)$};
\end{tikzpicture} \, .
\]
\end{rem}

\subsection{Digression on parameters} \label{ssec:rescaling}

The parameters governing the relations in quantum groups have evolved over time. 
We have chosen our sign conventions in the definition of $\UU_q(\glm)$ above carefully, 
to match signs appearing in foams and singular Soergel bimodules, for future ease of use. 
However, we will need to use relations derived by Lauda \cite{Lau1} 
and Khovanov--Lauda--Mackaay--Sto\v{s}i\'{c} (KLMS) \cite{KLMS}, 
which were provided in a version of the $2$-category defined with different sign conventions. 
In \cite[Section 3.1]{LaudaParameters}, Lauda constructs a functor $\crazyF$ 
that is an equivalence from one particular version of the categorified quantum group to any other version. 
By taking the source of $\crazyF$ to be the version used in \cite{Lau1} and \cite{KLMS} 
and the target of $\crazyF$ to be our version $\UU_q(\glm)$, 
we are able to translate relations in \cite{KLMS} into relations in $\UU_q(\glm)$.

\begin{rem}
In Remark \ref{rmk:Laudaparams} we discussed the parameters used for $\UU_q(\glm)$. 
To apply \cite[Section 3.1]{LaudaParameters}, 
one also needs to know the parameters for the version used by \cite{Lau1} and \cite{KLMS}. 
In the language of \cite[Section 3.1]{LaudaParameters} this is $c_{i,\wt} = 1$.
\end{rem}

Lauda's functor $\crazyF$ rescales the cups and caps by certain scalars 
which depend in a complicated way on the ambient weight. 
Fortunately, when $\crazyF$ is applied to a bubble, which has a cap paired with a cup in the same ambient weight, 
the complicated individual scaling factors multiply to an easy overall scaling factor. 
Indeed, an $i$-colored clockwise bubble in region $\wt$ is rescaled by $(-1)^{{a_i}-1}$. 
This is as it must be, since the degree zero clockwise bubble in \cite{KLMS} is equal to the scalar $1$ 
(times an identity map), while the degree zero clockwise bubble in $\UU_q(\glm)$ is the scalar $(-1)^{{a_i}-1}$. 
Moreover, whenever any $i$-colored clockwise cup and clockwise cap appear in a region of the
same weight $\wt$, even if they are not closed into a bubble, they will contribute a scaling factor of $(-1)^{{a_i}-1}$. 
An example of such a pair appears in the RHS of \eqref{extendedreln1}.

Similarly, a pair of an $i$-colored counterclockwise cup and counterclockwise cap in weight $\wt$ 
will contribute $(-1)^{{a_i}}$ to the scaling factor. 
For example, the reader can confirm that the bubble slide relations from \cite[Proposition 5.7]{Lau1} 
pick up a sign when translated to a relation in $\UU_q(\glm)$, giving
\begin{equation}
	\label{eq:bubslide1}
\begin{tikzpicture}[anchorbase,yscale=1]
\draw[thick,<-] (.5,0) node[below=-1pt]{\scs $i$} to [out=90,in=270] (.5,1.5);
\draw[thick,->] (1.5,.75) arc[start angle=0,end angle=360,radius=.25] 
	node[pos=.5,left=-1pt]{\scs$i$} node[pos=.75,black]{$\bullet$} 
		node[pos=.75,black,below]{\scs$\spadesuit{+}r$};
\node at (1.75,1.25){$\wt$};
\end{tikzpicture}
=
- \sum_{k=0}^2 
(-1)^k \binom{2}{k} \;
\begin{tikzpicture}[anchorbase]
\draw[thick,<-] (.5,0) node[below=-1pt]{\scs $i$} to [out=90,in=270] node[pos=.5]{$\bullet$} 
	node[pos=.5,right=-1pt]{\scs$k$} (.5,1.5);
\draw[thick,->] (0,.75) arc[start angle=0,end angle=360,radius=.25] 
	node[pos=.5,left=-1pt]{\scs$i$} node[pos=.75,black]{$\bullet$} 
		node[pos=.75,black,below]{\scs$\spadesuit{+}r{-}k$};
\node at (1.25,1.25){$\wt$};
\end{tikzpicture} \, .
\end{equation}
Meanwhile, the usual curl relations \cite[Proposition 5.4]{Lau1} hold verbatim, e.g.~
\begin{equation}
	\label{eq:curl1}
\begin{tikzpicture}[anchorbase]
\draw[thick,->] (0,-.75) node[below=-1pt]{\scs $i$} to [out=90,in=180] (.5,.25) to [out=0,in=90] (.75,0) 
	to [out=270,in=0] (.5,-.25) to [out=180,in=270] (0,.75);
\node at (.75,.5){$\wt$};
\end{tikzpicture}
=
- \sum_{p+q= a_{i+1} - a_i}
\begin{tikzpicture}[anchorbase]
\draw[thick,->] (0,-.75) node[below=-1pt]{\scs $i$} to node[pos=.5]{$\bullet$} node[pos=.5,right=-1pt]{\scs$p$} (0,.75);
\draw[thick,<-] (1.25,0) arc[start angle=0,end angle=360,radius=.25] 
	node[pos=.5,left=-1pt]{\scs$i$} node[pos=.75,black]{$\bullet$} 
		node[pos=.75,black,below]{\scs$\spadesuit{+}q$};
\node at (1.75,.5){$\wt$};
\end{tikzpicture} \, .
\end{equation}

While upwards-oriented crossings are not rescaled by $\crazyF$, 
sideways crossings do pick up a scaling factor, since they are defined as compositions of 
upward crossings with caps and cups. 
Again, the general scaling factor on a single sideways crossing is complicated,
but the overall scaling factor on a pair of uni-colored sideways crossings (one facing right, one facing left) 
in the same region is $-1$, independent of the weight $\wt$. 
For example, this explains the coefficient of the first term on the RHS of \eqref{extendedreln1}, 
which differs from the relation in \cite{Lau1} by a sign.

\subsection{Thick calculus}
	\label{ss:thickcalc}

The $2$-category $\UU_q(\glm)$ is not Karoubian, 
e.g.~the idempotent $2$-morphism
\begin{equation} \label{idempotentEE}
\begin{tikzpicture}[anchorbase,scale=1]
\draw[thick,->] (0,0) node[below=-1pt]{\scs $i$} to [out=90,in=270] (.5,1);
\draw[thick,->] (.5,0) node[below=-1pt]{\scs $i$} to [out=90,in=270] 
	node[pos=.75,yshift=-1pt]{$\bullet$} (0,1);
\node at (.875,.75){$\wt$};
\end{tikzpicture}
\in \Hom^{0}(\EE_i \EE_i\one_\wt)
\end{equation}
does not have an image $1$-morphism. 
In \cite{KLMS}, KLMS construct a diagrammatic category equivalent to the Karoubi envelope of $\UU_q(\gln[2])$, 
which is commonly called the \emph{thick calculus}, which we now recall.
As above, we use a different sign convention from \cite{KLMS} in this paper; 
see Remark \ref{rmk:rescalingredux} which discusses the relevant rescaling.

\begin{defn} 
	\label{def:Thickgl2}
Set $m=2$ and let $\EE := \EE_1$ and $\FF := \FF_1$. 
The $2$-category $\U_q(\gln[2])$ is obtained from $\UU_q(\gln[2])$ 
by adjoining new \emph{divided power} $1$-morphisms
\[
\EE^{(k)} \one_{\wt} \colon \wt \to \wt + k \ee_1 \, , \quad
\FF^{(k)} \one_{\wt}  \colon \wt \to \wt - k \ee_1 
\]
(where $\EE^{(1)} \one_\wt = \EE \one_\wt$ and $\FF^{(1)} \one_\wt = \FF \one_\wt$), together with
new \emph{merge/split} $2$-morphisms
\begin{equation}
	\label{eq:MS}
\begin{gathered}
\begin{tikzpicture}[anchorbase,yscale=-1]
\draw[ultra thick,<-] (0,0) node[above=-2pt]{\scs$k{+}\ell$} to [out=90,in=270] (0,.375);
\draw[ultra thick] (0,.375) to [out=150,in=270] (-.25,1) node[below=-2pt]{\scs$k$};
\draw[ultra thick] (0,.375) to [out=30,in=270] (.25,1) node[below=-2pt]{\scs$\ell$};
\node at (.5,.25){\scs$\wt$};
\end{tikzpicture}
\in \Hom^{-k \ell}(\EE^{(k)} \EE^{(\ell)} \one_\wt,\EE^{(k+ \ell)} \one_\wt) \, , \quad
\begin{tikzpicture}[anchorbase,yscale=-1]
\draw[ultra thick] (0,0) node[above=-2pt]{\scs$k{+}\ell$} to [out=90,in=270] (0,.375);
\draw[ultra thick,->] (0,.375) to [out=150,in=270] (-.25,1) node[below=-2pt]{\scs$k$};
\draw[ultra thick,->] (0,.375) to [out=30,in=270] (.25,1) node[below=-2pt]{\scs$\ell$};
\node at (.5,.25){\scs$\wt$};
\end{tikzpicture}
\in \Hom^{-k \ell}(\FF^{(k)} \FF^{(\ell)} \one_\wt,\FF^{(k+ \ell)} \one_\wt) \\
\begin{tikzpicture}[anchorbase,scale=1]
\draw[ultra thick] (0,0) node[below=-2pt]{\scs$k{+}\ell$} to [out=90,in=270] (0,.375);
\draw[ultra thick,->] (0,.375) to [out=150,in=270] (-.25,1) node[above=-2pt]{\scs$k$};
\draw[ultra thick,->] (0,.375) to [out=30,in=270] (.25,1) node[above=-2pt]{\scs$\ell$};
\node at (.5,.75){\scs$\wt$};
\end{tikzpicture}
\in \Hom^{-k \ell}(\EE^{(k+ \ell)} \one_\wt, \EE^{(k)} \EE^{(\ell)} \one_\wt) \, , \quad
\begin{tikzpicture}[anchorbase,scale=1]
\draw[ultra thick,<-] (0,0) node[below=-2pt]{\scs$k{+}\ell$} to [out=90,in=270] (0,.375);
\draw[ultra thick] (0,.375) to [out=150,in=270] (-.25,1) node[above=-2pt]{\scs$k$};
\draw[ultra thick] (0,.375) to [out=30,in=270] (.25,1) node[above=-2pt]{\scs$\ell$};
\node at (.5,.75){\scs$\wt$};
\end{tikzpicture}
\in \Hom^{-k \ell}(\FF^{(k+ \ell)} \one_\wt, \FF^{(k)} \FF^{(\ell)} \one_\wt) \, .
\end{gathered}
\end{equation}
Here, the labels on the strands in the string diagram correspond to their \emph{thickness}, 
which is the index $k$ in $\EE^{(k)}$, rather than their Dynkin label.
(There is a unique Dynkin label when $m=2$, thus no need to record it.)
By convention, any (diagrammatically) ``thin'' strand has thickness $1$; 
see \eqref{eq:explode} below.

These $2$-morphisms satisfy the oriented (co)associativity relations:
\begin{equation}
	\label{eq:assoc}
\begin{tikzpicture}[scale=.2, xscale=-1,tinynodes, anchorbase]
	\draw [ultra thick] (-1,-1) node[below,yshift=2pt]{$m$} to [out=90,in=210] (0,.75);
	\draw [ultra thick] (1,-1) node[below,yshift=2pt]{$\ell$} to [out=90,in=330] (0,.75);
	\draw [ultra thick] (3,-1) node[below,yshift=2pt]{$k$} to [out=90,in=330] (1,2.5);
	\draw [ultra thick, directed=.5] (0,.75) to [out=90,in=210] (1,2.5);
	\draw [ultra thick, ->] (1,2.5) to (1,4.25) node[above,yshift=-3pt]{$k{+}\ell{+}m$};
\end{tikzpicture}
=
\begin{tikzpicture}[scale=.2, tinynodes, anchorbase]
	\draw [ultra thick] (-1,-1) node[below,yshift=2pt]{$k$} to [out=90,in=210] (0,.75);
	\draw [ultra thick] (1,-1) node[below,yshift=2pt]{$\ell$} to [out=90,in=330] (0,.75);
	\draw [ultra thick] (3,-1) node[below,yshift=2pt]{$m$} to [out=90,in=330] (1,2.5);
	\draw [ultra thick, directed=.5] (0,.75) to [out=90,in=210] (1,2.5);
	\draw [ultra thick, ->] (1,2.5) to (1,4.25) node[above,yshift=-3pt]{$k{+}\ell{+}m$};
\end{tikzpicture}
\, , \quad
\begin{tikzpicture}[scale=.2, xscale=-1,tinynodes, anchorbase,yscale=-1]
	\draw [ultra thick,<-] (-1,-1) node[above=-2pt]{$k$} to [out=90,in=210] (0,.75);
	\draw [ultra thick,<-] (1,-1) node[above=-2pt]{$\ell$} to [out=90,in=330] (0,.75);
	\draw [ultra thick,<-] (3,-1) node[above=-2pt]{$m$} to [out=90,in=330] (1,2.5);
	\draw [ultra thick, rdirected=.5] (0,.75) to [out=90,in=210] (1,2.5);
	\draw [ultra thick] (1,2.5) to (1,4.25) node[below=-2pt]{$k{+}\ell{+}m$};
\end{tikzpicture}
=
\begin{tikzpicture}[scale=.2, tinynodes, anchorbase,yscale=-1]
	\draw [ultra thick,<-] (-1,-1) node[above=-2pt]{$k$} to [out=90,in=210] (0,.75);
	\draw [ultra thick,<-] (1,-1) node[above=-2pt]{$\ell$} to [out=90,in=330] (0,.75);
	\draw [ultra thick,<-] (3,-1) node[above=-2pt]{$m$} to [out=90,in=330] (1,2.5);
	\draw [ultra thick, rdirected=.5] (0,.75) to [out=90,in=210] (1,2.5);
	\draw [ultra thick] (1,2.5) to (1,4.25) node[below=-2pt]{$k{+}\ell{+}m$};
\end{tikzpicture} \, .
\end{equation}
As a consequence, there is a unique $2$-morphism $\EE^{(k)} \to \EE^k$ built from splits called the \emph{full split},
denoted
\[
\begin{tikzpicture}[anchorbase,scale=1]
\draw[ultra thick,directed=.5] (0,-.25) node[below=-2pt]{\scs$k$} to (0,.25);
\draw[thick,->] (0,.25) to [out=180,in=270] (-.5,.75);
\draw[thick,->] (0,.25) to [out=150,in=270] (-.25,.75);
\node at (0,.5) {$\mydots$};
\draw[thick,->] (0,.25) to [out=30,in=270] (.25,.75);
\draw[thick,->] (0,.25) to [out=0,in=270] (.5,.75);
\end{tikzpicture} \, .
\]
Similarly, there is a unique $2$-morphism $\EE^k \to \EE^{(k)}$ built from merges, 
called the \emph{full merge}, that is denoted analogously.
The full merge/split morphisms satisfy the following relations:
\begin{equation}
	\label{eq:MStoHT}
\begin{tikzpicture}[anchorbase,scale=1]
\draw[ultra thick,directed=.5] (0,-.25) to node[right=-1pt]{\scs$k$} (0,.25);
\draw[thick,->] (0,.25) to [out=180,in=270] (-.5,.75);
\draw[thick,->] (0,.25) to [out=150,in=270] (-.25,.75);
\node at (0,.5) {$\mydots$};
\draw[thick,->] (0,.25) to [out=30,in=270] (.25,.75);
\draw[thick,->] (0,.25) to [out=0,in=270] (.5,.75);
\node at (.625,.25){\scs$\wt$};
\draw[thick] (-.5,-.75) to [out=90,in=180] (0,-.25);
\draw[thick] (-.25,-.75) to [out=90,in=210] (0,-.25);
\node at (0,-.5) {$\mydots$};
\draw[thick] (.25,-.75) to [out=90,in=330] (0,-.25);
\draw[thick] (.5,-.75) to [out=90,in=0] (0,-.25);
\end{tikzpicture}
=
\begin{tikzpicture}[anchorbase,scale=1]
\draw[thick,->] (-.375,-.75) to (-.375,.75);
\draw[thick,->] (.375,-.75) to (.375,.75);
\node at (0,.5) {$\mydots$};
\node at (0,-.5) {$\mydots$};
\path[fill=white] (-.5,-.25) rectangle (.5,.25);
\draw[very thick] (-.5,-.25) rectangle (.5,.25);
\node at (0,0) {\small$\HT_k$};
\node at (.75,.25){\scs$\wt$};
\end{tikzpicture}
\, , \quad
\begin{tikzpicture}[anchorbase,scale=1]
\draw[ultra thick,rdirected=.5] (0,-.25) to node[right=-1pt]{\scs$k$} (0,.25);
\draw[thick] (0,.25) to [out=180,in=270] (-.5,.75);
\draw[thick] (0,.25) to [out=150,in=270] (-.25,.75);
\node at (0,.5) {$\mydots$};
\draw[thick] (0,.25) to [out=30,in=270] (.25,.75);
\draw[thick] (0,.25) to [out=0,in=270] (.5,.75);
\node at (.625,.25){\scs$\wt$};
\draw[thick,<-] (-.5,-.75) to [out=90,in=180] (0,-.25);
\draw[thick,<-] (-.25,-.75) to [out=90,in=210] (0,-.25);
\node at (0,-.5) {$\mydots$};
\draw[thick,<-] (.25,-.75) to [out=90,in=330] (0,-.25);
\draw[thick,<-] (.5,-.75) to [out=90,in=0] (0,-.25);
\end{tikzpicture}
=
\begin{tikzpicture}[anchorbase,scale=1]
\draw[thick,<-] (-.375,-.75) to (-.375,.75);
\draw[thick,<-] (.375,-.75) to (.375,.75);
\node at (0,.5) {$\mydots$};
\node at (0,-.5) {$\mydots$};
\path[fill=white] (-.5,-.25) rectangle (.5,.25);
\draw[very thick] (-.5,-.25) rectangle (.5,.25);
\node at (0,0) {\small$\HT_k$};
\node at (.75,.25){\scs$\wt$};
\end{tikzpicture}
\end{equation}
and
\begin{equation}
	\label{eq:explode}
\begin{tikzpicture}[anchorbase,scale=1]
\draw[ultra thick,directed=.6] (0,-.75) node[below=-2pt]{\scs$k$} to (0,-.375);
\draw[thick] (0,-.375) to [out=180,in=270] node[pos=.95]{$\bullet$}  (-1.125,0) 
	node[left=-1pt]{\scs$k{-}1$} to [out=90,in=180] (0,.375);
\draw[thick] (0,-.375) to [out=150,in=270] node[pos=.85]{$\bullet$} (-.25,0) 
	node[left=-1pt]{\scs$k{-}2$} to [out=90,in=210] (0,.375);
\node at (0,0) {$\mydots$};
\draw[thick] (0,-.375) to [out=30,in=270] node[pos=.85]{$\bullet$} (.25,0) to [out=90,in=330] (0,.375);
\draw[thick] (0,-.375) to [out=0,in=270] (.5,0) to [out=90,in=0] (0,.375);
\draw[ultra thick,directed=.6] (0,.375) to (0,.75) node[above=-2pt]{\scs$k$} ;
\node at (.625,.5){\scs$\wt$};
\end{tikzpicture}
=
\begin{tikzpicture}[anchorbase,scale=1]
\draw[ultra thick,->] (0,-.75) node[below=-2pt]{\scs$k$} to (0,.75) node[white,above=-2pt]{\scs$k$};
\node at (.375,.5){\scs$\wt$};
\end{tikzpicture} 
\, , \quad
\begin{tikzpicture}[anchorbase,scale=1,rotate=180]
\draw[ultra thick,directed=.6] (0,-.75) node[above=-2pt]{\scs$k$} to (0,-.375);
\draw[thick] (0,-.375) to [out=180,in=270] node[pos=.95]{$\bullet$}  (-1.125,0) 
	node[right=-1pt]{\scs$k{-}1$} to [out=90,in=180] (0,.375);
\draw[thick] (0,-.375) to [out=150,in=270] node[pos=.85]{$\bullet$} (-.25,0) 
	node[right=-1pt]{\scs$k{-}2$} to [out=90,in=210] (0,.375);
\node at (0,0) {$\mydots$};
\draw[thick] (0,-.375) to [out=30,in=270] node[pos=.85]{$\bullet$} (.25,0) to [out=90,in=330] (0,.375);
\draw[thick] (0,-.375) to [out=0,in=270] (.5,0) to [out=90,in=0] (0,.375);
\draw[ultra thick,directed=.6] (0,.375) to (0,.75) node[below=-2pt]{\scs$k$} ;
\node at (-.625,-.625){\scs$\wt$};
\end{tikzpicture}
=
\begin{tikzpicture}[anchorbase,scale=1]
\draw[ultra thick,<-] (0,-.75) node[below=-2pt]{\scs$k$} to (0,.75) node[white,above=-2pt]{\scs$k$};
\node at (.375,.5){\scs$\wt$};
\end{tikzpicture} 
\end{equation}
Here $\HT_k$ denotes any string diagram 
depicting a reduced word for the half-twist permutation 
(i.e.~the longest element of $\SG_k$).
\end{defn}

Our definition above is a presentation of $\U_q(\gln[2])$ by generators and relations. 
In \cite{KLMS} they give many more relations than these, and do not discuss which relations
are truly needed in a presentation. 
In the following lengthy remark, 
we sketch a proof for why our definition above agrees with theirs, 
i.e.~why the relations above imply the more complicated
relations in \cite{KLMS}.

\begin{rem} \label{rmk:presentKar}
Consider an idempotent $2$-endomorphism $e \in \End(X)$ 
of a $1$-morphism $X$ in a $2$-category $\Cat$.
Let $\Cat(e)$ be the partial idempotent completion, which adjoins the image $\Im(e)$ of $e$ as a new object. 
It is straightforward to extend a presentation of $\Cat$ to obtain a presentation of $\Cat(e)$.
One need only adjoin two new $2$-morphisms $\pi_e \colon X \to \Im(e)$ and $\iota_e \colon \Im(e) \to X$, 
which satisfy two new relations:
\begin{equation} 
	\label{eq:KarEnv} 
\pi_e \circ \iota_e = \id_{\Im(e)} \, , \quad \iota_e \circ \pi_e = e \, . 
\end{equation}
Although there may be many new relations in $\Cat(e)$, 
all may be obtained as a consequence of in \eqref{eq:KarEnv} and existing relations in $\Cat$.
Further, if the idempotent $e$ could already be split in $\Cat$ as factoring through an object $Y$, 
then $\Im(e)$ and $Y$ will be isomorphic in $\Cat(e)$, 
and there will be an equivalence of categories $\Cat \cong \Cat(e)$.

In the Karoubi envelope of $\UU_q(\gln[2])$,
the object $\EE^k$ splits into $k!$ indecomposable summands, 
all of which happen to be isomorphic up to grading shift; 
see e.g.~\cite[Section 9.2]{Lau1}. 
The paper \cite{KLMS} proceeds based on the observation that
the partial idempotent completion which adds the images of all these idempotents 
is equivalent to the partial idempotent completion which adds just one of them. 
In fact, it does not adjoin the image of any idempotents in the decomposition of $\EE^k$, 
but rather adjoins the split and merge maps that factor through a single object $\EE^{(k)}$, 
which is isomorphic to each of these images, up to shift.

A straightforward (if perhaps less motivated) way to proceed is to 
note that one of the idempotents in the decomposition of $\EE^k$ factors in a nice way, 
as a composition $e = f \circ g$ where $g \circ e = g$. 
Specifically, we can take
$g = \HT_k$ and $f = x_1^{k-1} x_2^{k-2} \ldots$, where $x_i$ represents a dot on the $i$-th strand; 
see e.g.~\eqref{idempotentEE}.
For such $e, f, g$, the morphism $g \in \End(X)$ can be factored in $\Cat(e)$ 
as $g = \iota'_e \circ \pi_e$ 
by computing
\[
g = g \circ e = g \circ \iota_e \circ \pi_e
\]
and setting $\iota'_e := g \circ \iota_e$.
Since one can recover $\iota_e$ as 
\[
\iota_e = e \circ \iota_e = f \circ g \circ \iota_e = f \circ \iota'_e \, ,
\]
this provides a new presentation of $\Cat(e)$ 
wherein the generator $\iota_e$ is replaced by $\iota'_e$ and the relations \eqref{eq:KarEnv} 
are replaced with
\begin{equation}
	\label{eq:KarEnv2}
	\pi_e \circ f \circ \iota_e' = \id_{\Im(e)} \, , \quad \iota_e' \circ \pi_e = g \, .
\end{equation}
The relations \eqref{eq:MStoHT} and \eqref{eq:explode} are exactly those in \eqref{eq:KarEnv2};
this justifies our presentation for $\U_q(\gln[2])$. \end{rem}

\begin{conv}\label{Conv:thickcapscups}
\emph{Thick caps/cups} and \emph{thick crossings} are defined in 
$\U_q(\gln[2])$ using ``thin'' caps/cups and crossings via \eqref{eq:explode}. 
For example,
\begin{equation}\label{eqn:thickcapdefn}
\begin{tikzpicture}[anchorbase,scale=1]
\draw[ultra thick,black,<-] (-.25,0) to [out=90,in=180] (0,.5) 
	to [out=0,in=90] (.25,0) node[below=-2pt]{\scs$k$};
\node at (.5,.5){$\wt$};
\end{tikzpicture}
:=
\begin{tikzpicture}[anchorbase,scale=1]
\draw[ultra thick,black,->] (-.5,-1) to (-.5,-1.5);
\draw[thick,black] (-.5,-1) to [out=30,in=150] node[black]{$\bullet$} 
	node[black,below=-2pt]{\tiny$k{-}1$} (.5,-1);
\draw[thick,black] (-.5,-1) to [out=150,in=180] (0,-.5) node[black]{$\bullet$} 
	node[black,below=-2pt]{\tiny$k{-}2$} to [out=0,in=30] (.5,-1);
\draw[thick,black] (-.5,-1) to [out=180,in=180] (0,0) 
	node[rotate=90,xshift=-6pt,black]{$\mydots$} to [out=0,in=0] (.5,-1);
\draw[ultra thick,black] (.5,-1) to (.5,-1.5) node[below=-2pt]{\scs$k$};
\end{tikzpicture}
\, , \quad
\begin{tikzpicture}[anchorbase,scale=1]
\draw[ultra thick,black,->] (-.25,0) node[below=-2pt]{\scs $k$} to [out=90,in=180] (0,.5) 
	to [out=0,in=90] (.25,0);
\node at (.5,.5){$\wt$};
\end{tikzpicture}
:=
\begin{tikzpicture}[anchorbase,scale=1]
\draw[ultra thick,black] (-.5,-1) to (-.5,-1.5) node[below=-2pt]{\scs$k$};
\draw[thick,black] (-.5,-1) to [out=30,in=150] (.5,-1);
\draw[thick,black] (-.5,-1) to [out=150,in=180] (0,-.5) node[black]{$\bullet$} 
	node[black,above=-2pt]{\tiny$k{-}2$} node[rotate=90,xshift=-6pt,black]{$\mysdots$}
		to [out=0,in=30] (.5,-1);
\draw[thick,black] (-.5,-1) to [out=180,in=180] (0,0) node[black]{$\bullet$} 
	node[black,above=-2pt]{\tiny$k{-}1$} to [out=0,in=0] (.5,-1);
\draw[ultra thick,black,->] (.5,-1) to (.5,-1.5);
\end{tikzpicture}
\end{equation}
and
\begin{equation}\label{eqn:thickcrossingdefn}
\begin{tikzpicture}[anchorbase,scale=1]
\draw[ultra thick,black,->] (0,0) node[below=-2pt]{\scs$k$} to [out=90,in=270] (.5,1);
\draw[ultra thick,->] (.5,0) node[below=-2pt]{\scs$\ell$} to [out=90,in=270] (0,1);
\node at (.875,.75){$\wt$};
\end{tikzpicture}
:=
\begin{tikzpicture}[anchorbase,scale=1]
\draw[ultra thick,->] (-.75,1) node[below,xshift=1pt]{$\mydots$} to (-.75,1.25);
\draw[ultra thick,black,->] (.75,1) node[below,xshift=-1pt,black]{$\mydots$} to (.75,1.25);
\draw[ultra thick,black] (-.75,-1) to (-.75,-1.25) node[below=-2pt]{\scs$k$};
\draw[ultra thick] (.75,-1) to (.75,-1.25) node[below=-2pt]{\scs$\ell$};
\draw[thick,black] (-.75,-1) to [out=150,in=210] node[pos=.85,black]{$\bullet$} 
	node[pos=.85,above,black]{\tiny$k{-}1$} (.75,1);
\draw[thick,black] (-.75,-1) to [out=30,in=330] node[pos=.85,black]{$\bullet$} (.75,1);
\draw[thick,black] (-.75,-1) to [out=0,in=270] (1.25,.5) to [out=90,in=0] (.75,1);
\draw[thick] (.75,-1) to [out=150,in=210] node[pos=.75]{$\bullet$} 
	node[pos=.85,below,xshift=10pt]{\tiny$\ell{-}2$} (-.75,1);
\draw[thick] (.75,-1) to [out=30,in=330] (-.75,1);
\draw[thick] (.75,-1) to [out=180,in=270] (-1.25,.5) node{$\bullet$} 
	node[left]{\tiny$\ell{-}1$}
	to [out=90,in=180] (-.75,1);
\node at (1.5,.75){$\wt$};
\end{tikzpicture} \, .
\end{equation}
These $2$-morphisms interact with merge/split morphisms in a straightforward manner,
e.g.
\begin{equation}
	\label{eq:otherthick}
\begin{tikzpicture}[anchorbase,yscale=-1]
\draw[ultra thick,black,<-] (1,1) node[below=-2pt]{\scs$k{+}\ell$} to (1,.375) to [out=270,in=0] (.5,-.125) 
	to [out=180,in=270] (0,.375);
\draw[ultra thick,black] (0,.375) to [out=150,in=270] (-.25,1) node[below=-2pt]{\scs$k$};
\draw[ultra thick,black] (0,.375) to [out=30,in=270] (.25,1) node[below=-2pt]{\scs$\ell$};
\node at (1.25,0){$\wt$};
\end{tikzpicture}
=
\begin{tikzpicture}[anchorbase,scale=1]
\draw[ultra thick,black,<-] (0,0) node[below=-2pt]{\scs$k{+}\ell$} to [out=90,in=270] (0,.375);
\draw[ultra thick,black] (0,.375) to [out=150,in=270] (-.25,.75) to [out=90,in=0] (-.5,1) to [out=180,in=90]
	(-.75,.75) to (-.75,0) node[below=-2pt]{\scs$\ell$};
\draw[ultra thick,black] (0,.375) to [out=30,in=270] (.25,.75) to [out=90,in=0] (-.5,1.5) to [out=180,in=90]
	(-1.25,.75) to (-1.25,0) node[below=-2pt]{\scs$k$};
\node at (.5,1.125){$\wt$};
\end{tikzpicture}
\, , \quad
\begin{tikzpicture}[anchorbase,scale=1]
\draw[ultra thick,black] (-.25,-.5) node[below=-2pt]{\scs$k_1$} to [out=90,in=210] (0,0);
\draw[ultra thick,black] (.25,-.5) node[below=-2pt]{\scs$k_2$} to [out=90,in=330] (0,0);
\draw[ultra thick,black,->] (0,0) to [out=90,in=270] (.5,1);
\draw[ultra thick,->] (.75,-.5) node[below=-2pt]{\scs$\ell$} to [out=90,in=270] (0,1);
\node at (.875,.75){$\wt$};
\end{tikzpicture}
=
\begin{tikzpicture}[anchorbase,scale=1]
\draw[ultra thick,black] (-.25,-.5) node[below=-2pt]{\scs$k_1$} to [out=90,in=210] (.5,.625);
\draw[ultra thick,black] (.25,-.5) node[below=-2pt]{\scs$k_2$} to [out=90,in=330] (.5,.625);
\draw[ultra thick,black,->] (.5,.625) to [out=90,in=270] (.5,1);
\draw[ultra thick,->] (.75,-.5) node[below=-2pt]{\scs$\ell$} to [out=90,in=270] (0,1);
\node at (.875,.75){$\wt$};
\end{tikzpicture} \, .
\end{equation}
Further, crossings satisfy the relations
\begin{equation}
	\label{eq:bypass}
\begin{tikzpicture}[anchorbase,scale=1.25]
\draw[ultra thick,black,->] (.5,.25) to (.5,.5) node[above=-2pt]{\scs$k{+}j$};
\draw[ultra thick,black,directed=.3] (0,-.5) node[below=-2pt]{\scs$k$} to [out=90,in=210] (.5,.25);
\draw[ultra thick,black,directed=.6] (.5,-.25) to [out=30,in=330] node[right=-1pt]{\scs$j$} (.5,.25);
\draw[ultra thick,black,->] (.5,-.25) to [out=150,in=270] (0,.5) node[above=-2pt]{\scs$\ell$};
\draw[ultra thick,black,directed=.75] (.5,-.5) node[below=-2pt]{\scs$\ell{+}j$} to (.5,-.25);
\node at (1,.25){$\wt$};
\end{tikzpicture}
=
\begin{tikzpicture}[anchorbase,scale=1.25]
\draw[ultra thick,black,->] (.25,.125) to [out=150,in=270] (0,.5) node[above=-2pt]{\scs$\ell$};
\draw[ultra thick,black,->] (.25,.125) to [out=30,in=270] (.5,.5) node[above=-2pt]{\scs$k{+}j$};
\draw[ultra thick,black,directed=.7] (.25,-.125) to (.25,.125);
\draw[ultra thick,black,directed=.6] (0,-.5) node[below=-2pt]{\scs$k$} to [out=90,in=210] (.25,-.125);
\draw[ultra thick,black,directed=.6] (.5,-.5) node[below=-2pt]{\scs$\ell{+}j$} to [out=90,in=330] (.25,-.125);
\node at (.75,.25){$\wt$};
\end{tikzpicture}
\quad , \quad
\begin{tikzpicture}[anchorbase,scale=1.25,xscale=-1]
\draw[ultra thick,black,->] (.5,.25) to (.5,.5) node[above=-2pt]{\scs$\ell{+}j$};
\draw[ultra thick,black,directed=.3] (0,-.5) node[below=-2pt]{\scs$\ell$} to [out=90,in=210] (.5,.25);
\draw[ultra thick,black,directed=.6] (.5,-.25) to [out=30,in=330] node[left=-1pt]{\scs$j$} (.5,.25);
\draw[ultra thick,black,->] (.5,-.25) to [out=150,in=270] (0,.5) node[above=-2pt]{\scs$k$};
\draw[ultra thick,black,directed=.75] (.5,-.5) node[below=-2pt]{\scs$k{+}j$} to (.5,-.25);
\node at (-.25,.25){$\wt$};
\end{tikzpicture}
=
\begin{tikzpicture}[anchorbase,scale=1.25,xscale=-1]
\draw[ultra thick,black,->] (.25,.125) to [out=150,in=270] (0,.5) node[above=-2pt]{\scs$k$};
\draw[ultra thick,black,->] (.25,.125) to [out=30,in=270] (.5,.5) node[above=-2pt]{\scs$\ell{+}j$};
\draw[ultra thick,black,directed=.7] (.25,-.125) to (.25,.125);
\draw[ultra thick,black,directed=.6] (0,-.5) node[below=-2pt]{\scs$\ell$} to [out=90,in=210] (.25,-.125);
\draw[ultra thick,black,directed=.6] (.5,-.5) node[below=-2pt]{\scs$k{+}j$} to [out=90,in=330] (.25,-.125);
\node at (-.25,.25){$\wt$};
\end{tikzpicture} \, .
\end{equation}
\end{conv}

\begin{rem} \label{rmk:rescalingredux} 
Continuing the discussion of \S\ref{ssec:rescaling}, we describe how the functor $\crazyF$ 
extends to the thick calculus. 
Any pair of a clockwise cup and clockwise cap in the same weight $\wt$, with label $i$ and thickness $k$, 
will be rescaled by $(-1)^{k(a_i-1) + \binom{k}{2}}$. 
For counterclockwise cap/cup pairs, the scaling factor is $(-1)^{ka_i + \binom{k}{2}}$. 
A pair of sideways crossings with thickness $k$ and $\ell$ (in the same weight $\wt$) 
will be rescaled by $(-1)^{k \ell}$. 
\end{rem}

While the relations in Definition \ref{def:Thickgl2} 
are the only ones needed in a presentation of $\U_q(\gln[2])$, 
many other useful relations, 
which are (difficult) consequences of these, 
are provided in \cite{KLMS}. 
The rescaling in Remark \ref{rmk:rescalingredux} 
allows us to translate these relations to our conventions, 
and we now record those which are most important to us. 
In what follows, $P(k)$ denotes the set of partitions with at most $k$ rows, 
and $P(k,\ell)$ denotes the set of partitions fitting in a $k \times \ell$ rectangle. 
Given $\parti \in P(k,\ell)$, 
$\parti^c \in P(k,\ell)$ denotes its complement (in a $k \times \ell$ rectangle), 
$\bar{\parti} \in P(\ell,k)$ denotes the transpose partition, 
and $\hat{\parti} := \overline{\parti^c} \in P(\ell,k)$; see \cite[Page 14]{KLMS}.
If $\parti= (\parti, \ldots, \parti_k)\in P(k)$, then we write $|\parti|:= \sum_{i=1}^k\parti_i$. 

\begin{prop}[{\cite{KLMS}}]
Let $\parti \in P(k)$ be a partition and set
\begin{equation}
	\label{eq:SymDec}
\begin{tikzpicture}[anchorbase,scale=1]
\draw[ultra thick,->] (0,-.75) node[below=-2pt]{\scs$k$} to node{$\CQGbox{\mathfrak{s}_\parti}$}
	(0,.75) node[white,above=-2pt]{\scs$k$};
\node at (.25,.5){$\wt$};
\end{tikzpicture} 
:=
\begin{tikzpicture}[anchorbase,scale=1]
\draw[ultra thick,directed=.6] (0,-.75) node[below=-2pt]{\scs$k$} to (0,-.375);
\draw[thick] (0,-.375) to [out=180,in=270] node[pos=.95]{$\bullet$}  (-1.125,0) 
	node[left=-1pt,yshift=2pt]{\scs$k{-}1$} node[left=-1pt,yshift=-4pt]{\scs$+\parti_1$} 
		to [out=90,in=180] (0,.375);
\draw[thick] (0,-.375) to [out=150,in=270] node[pos=.85]{$\bullet$} (-.25,0) 
	node[left=-1pt,yshift=2pt]{\scs$k{-}2$} node[left=-1pt,yshift=-4pt]{\scs$+\parti_2$} 
		to [out=90,in=210] (0,.375);
\node at (0,0) {$\mydots$};
\draw[thick] (0,-.375) to [out=30,in=270] node[pos=.85]{$\bullet$} (.25,0) 
	node[right=-1pt,yshift=2pt]{\scs$1{+}$} node[right=-1pt,yshift=-4pt]{\scs$\parti_{k-1}$} 
		to [out=90,in=330] (0,.375);
\draw[thick] (0,-.375) to [out=0,in=270] node[pos=.95]{$\bullet$} (1.25,0) 
	node[right=-1pt]{\scs$\parti_k$} to [out=90,in=0] (0,.375);
\draw[ultra thick,directed=.6] (0,.375) to (0,.75) node[above=-2pt]{\scs$k$} ;
\node at (1,.625){$\wt$};
\end{tikzpicture}
\end{equation}
The following (as well as their $180^\circ$ rotations) hold in $\U_q(\gln[2])$.
\begin{itemize}
\item The assignment sending a Schur polynomial 
$\mathfrak{s}_{\parti}(x_1,\ldots,x_k)$ to the element in 
\eqref{eq:SymDec} determines a $\K$-algebra isomorphism
$\K[x_1,\ldots,x_k]^{\SG_k} \xrightarrow{\cong} \End(\EE^{(k)} \one_\wt)$.
In particular,
\begin{equation}
\begin{tikzpicture}[anchorbase,scale=1]
\draw[ultra thick,->] (0,-.75) node[below=-2pt]{\scs$k$} to 
	node[pos=.3]{$\CQGbox{\mathfrak{s}_\mu}$} node[pos=.7]{$\CQGbox{\mathfrak{s}_\lambda}$}
	(0,.75) node[white,above=-2pt]{\scs$k$};
\node at (.5,.5){$\wt$};
\end{tikzpicture} 
=
\sum_{\nu \in P(k)}
c_{\lambda,\mu}^{\nu}
\begin{tikzpicture}[anchorbase,scale=1]
\draw[ultra thick,->] (0,-.75) node[below=-2pt]{\scs$k$} to node{$\CQGbox{\mathfrak{s}_\nu}$}
	(0,.75) node[white,above=-2pt]{\scs$k$};
\node at (.375,.5){$\wt$};
\end{tikzpicture},
\end{equation}
where $c_{\lambda,\mu}^{\nu}$ is a Littlewood-Richardson coefficient.
\item Decorations migrate according to the coproduct for symmetric functions, e.g.
\begin{equation}
	\label{eq:migrate}
\begin{tikzpicture}[anchorbase,xscale=1.25]
\draw[ultra thick] (0,0) node[below=-2pt]{\scs$k{+}\ell$} to [out=90,in=270] node[pos=.5]{\tiny$\CQGsbox{\Schur[\nu]}$} (0,.5);
\draw[ultra thick,->] (0,.5) to [out=150,in=270] (-.25,1.125) node[above=-2pt]{\scs$k$};
\draw[ultra thick,->] (0,.5) to [out=30,in=270] (.25,1.125) node[above=-2pt]{\scs$\ell$};
\node at (.5,.75){\scs$\wt$};
\end{tikzpicture}
=
\sum
c_{\lambda,\mu}^{\nu}
\begin{tikzpicture}[anchorbase,xscale=1.25]
\draw[ultra thick] (0,0) node[below=-2pt]{\scs$k{+}\ell$} to [out=90,in=270] (0,.375);
\draw[ultra thick,->] (0,.375) to [out=150,in=270] node[pos=.5]{\tiny$\CQGsbox{\Schur}$} (-.25,1.125) node[above=-2pt]{\scs$k$};
\draw[ultra thick,->] (0,.375) to [out=30,in=270] node[pos=.5]{\tiny$\CQGsbox{\Schur[\mu]}$} (.25,1.125) node[above=-2pt]{\scs$\ell$};
\node at (.625,.75){\scs$\wt$};
\end{tikzpicture}
\end{equation}

\item For $\parti \in P(k,\ell)$ and  $\mu \in P(\ell,k)$,
\begin{equation}
	\label{eq:CQGdigon}
\begin{tikzpicture}[anchorbase,scale=1]
\draw[ultra thick] (0,-.75) to (0,-.375);
\draw[ultra thick] (0,-.375) to [out=150,in=270] (-.375,0) node{\scs$\CQGbox{\mathfrak{s}_{\parti}}$}
	to [out=90,in=210] (0,.375);
\draw[ultra thick] (0,-.375) to [out=30,in=270] (.375,0) node{\scs$\CQGbox{\mathfrak{s}_{\mu}}$}
	to [out=90,in=330] (0,.375);
\draw[ultra thick,->] (0,.375) to (0,.75);
\end{tikzpicture}
=
\begin{cases}
(-1)^{|\hat{\parti}|}
\begin{tikzpicture}[anchorbase,scale=1]
\draw[ultra thick,->] (0,-.5) to (0,.5);
\end{tikzpicture} & \text{if } \mu = \hat{\parti} \\ \\
0 & \text{else.}
\end{cases}
\end{equation}

\item There is an idempotent decomposition of the identity
\begin{equation}
	\label{eq:EEdecomp}
\sum_{\parti \in P(k,\ell)}
(-1)^{|\hat{\parti}|}
\begin{tikzpicture}[anchorbase,scale=1.25]
\draw[ultra thick] (-.25,-.75) node[below=-2pt]{\scs$k$} to [out=90,in=210] (0,-.125);
\draw[ultra thick] (.25,-.75) node[below=-2pt]{\scs$\ell$} to [out=90,in=330] 
	node{\scs$\CQGbox{\mathfrak{s}_{\hat{\parti}}}$} (0,-.125);
\draw[ultra thick] (0,-.125) to [out=90,in=270] (0,.125);
\draw[ultra thick,->] (0,.125) to [out=150,in=270] 
	node{\scs$\CQGbox{\mathfrak{s}_{\parti}}$} (-.25,.75) node[above=-2pt]{\scs$k$};
\draw[ultra thick,->] (0,.125) to [out=30,in=270] (.25,.75) node[above=-2pt]{\scs$\ell$};
\node at (.5,.5){$\wt$};
\end{tikzpicture}
=
\begin{tikzpicture}[anchorbase,scale=1.25]
\draw[ultra thick,->] (-.25,-.75) node[below=-2pt]{\scs$k$} to 
	(-.25,.75) node[white, above=-2pt]{\scs$k$};
\draw[ultra thick,->] (.25,-.75) node[below=-2pt]{\scs$\ell$} to 
	(.25,.75) node[white, above=-2pt]{\scs$\ell$};
\node at (.5,.5){$\wt$};
\end{tikzpicture}
\end{equation}
Together with \eqref{eq:CQGdigon}, 
this implies that 
$\EE^{(k)} \EE^{(\ell)} \one_{\wt} \cong \bigoplus_{k+\ell \brack \ell} \EE^{(k+\ell)} \one_{\wt}$
and that
$\FF^{(k)} \FF^{(\ell)} \one_{\wt} \cong \bigoplus_{k+\ell \brack \ell} \FF^{(k+\ell)} \one_{\wt}$.
\item For any $b_1, \ldots, b_k\in \Z_{\ge 0}$ and any $0\le r < k$
\begin{equation}\label{eqn:sgnfromdemazure}
\begin{tikzpicture}[anchorbase,scale=1]
\draw[ultra thick,directed=.6] (0,-.75) node[below=-2pt]{\scs$k$} to (0,-.375);
\draw[thick] (0,-.375) to [out=180,in=270] node[pos=.95]{$\bullet$}  (-1.125,0) 
	node[left=-1pt,yshift=2pt]{\scs$b_1$}
		to [out=90,in=180] (0,.375);
\draw[thick] (0,-.375) to [out=150,in=270] node[pos=.85]{$\bullet$} (-.25,0) 
	node[left=-1pt,yshift=2pt]{\scs$b_r$}
		to [out=90,in=210] (0,.375);
\node[yshift=-2pt] at (-.625,-.125) {$\mydots$};
\node[yshift=-2pt] at (.625,-.125) {$\mydots$};
\draw[thick] (0,-.375) to [out=30,in=270] node[pos=.85]{$\bullet$} (.25,0) 
	node[right=-1pt,yshift=2pt]{\scs$b_{r+1}$}
		to [out=90,in=330] (0,.375);
\draw[thick] (0,-.375) to [out=0,in=270] node[pos=.95]{$\bullet$} (1.25,0) 
	node[right=-1pt]{\scs$b_k$} to [out=90,in=0] (0,.375);
\draw[ultra thick,->] (0,.375) to (0,.75);
\node at (1,.625){$\wt$};
\end{tikzpicture}
= -1 \cdot
\begin{tikzpicture}[anchorbase,scale=1]
\draw[ultra thick,directed=.6] (0,-.75) node[below=-2pt]{\scs$k$} to (0,-.375);
\draw[thick] (0,-.375) to [out=180,in=270] node[pos=.95]{$\bullet$}  (-1.25,0) 
	node[left=-1pt,yshift=2pt]{\scs$b_1$}
		to [out=90,in=180] (0,.375);
\draw[thick] (0,-.375) to [out=150,in=270] node[pos=.85]{$\bullet$} (-.25,0) 
	node[left=-1pt,yshift=2pt]{\scs$b_{r+1}$}
		to [out=90,in=210] (0,.375);
\node[yshift=-2pt] at (-.625,-.125) {$\mydots$};
\node[yshift=-2pt] at (.625,-.125) {$\mydots$};
\draw[thick] (0,-.375) to [out=30,in=270] node[pos=.85]{$\bullet$} (.25,0) 
	node[right=-1pt,yshift=2pt]{\scs$b_r$}
		to [out=90,in=330] (0,.375);
\draw[thick] (0,-.375) to [out=0,in=270] node[pos=.95]{$\bullet$} (1.125,0) 
	node[right=-1pt]{\scs$b_k$} to [out=90,in=0] (0,.375);
\draw[ultra thick,->] (0,.375) to (0,.75);
\node at (1,.625){$\wt$};
\end{tikzpicture} \, .
\end{equation}

\item Generalizing \eqref{eq:curl1}, the thick curl relation:
\begin{equation}
	\label{eq:thickcurl}
\begin{aligned}
\begin{tikzpicture}[anchorbase]
\draw[ultra thick,->] (0,-.75) node[below=-2pt]{\scs $k$} to [out=90,in=180] (.5,.25) to [out=0,in=90] (.75,0) 
	node{\tiny$\CQGsbox{\Schur[\nu]}$} to [out=270,in=0] (.5,-.25) to [out=180,in=270] (0,.75);
\node at (.75,.5){\scs$(a_1,a_2)$};
\end{tikzpicture}
&=
(-1)^{k^2} \sum_{\parti,\mu}
c_{\parti,\mu}^{\nu-a_1+a_2}
\begin{tikzpicture}[anchorbase]
\draw[ultra thick,->] (0,-.75) node[below=-2pt]{\scs $k$} to node[pos=.5]{\tiny$\CQGsbox{\Schur}$} (0,.75);
\draw[ultra thick,<-] (1.375,0) arc[start angle=0,end angle=360,radius=.375] 
	node[pos=.5]{\tiny$\CQGsbox{\mathfrak{s}_\mu^\spadesuit}$};
\node at (1.75,.5){\scs$(a_1,a_2)$};
\end{tikzpicture} \\
&=
(-1)^{k(k+a_i-1)} \sum_{\parti,\mu}
c_{\parti,\mu}^{\nu-a_1+a_2}
\begin{tikzpicture}[anchorbase]
\draw[ultra thick,->] (0,-.75) node[below=-2pt]{\scs $k$} to node[pos=.5]{\tiny$\CQGsbox{\Schur}$} (0,.75);
\node at (1.25,0){$\NB{\Schur[\mu](\X_1 - \X_2)}$}; 
\node at (1.75,.5){\scs$(a_1,a_2)$};
\end{tikzpicture} 
\, ,
\end{aligned}
\end{equation}
holds.
Here, the notation
$\nu \pm r$ stands for the partition $(\nu_1\pm r,\ldots,\nu_k \pm r)$
and
$\Schur[\mu]^\spadesuit=\Schur[\mu-r]$ 
for $r$ such that the resulting bubble has degree $2|\mu|$.

\item There are decompositions into indecomposable $1$-morphisms:
\begin{equation}
	\label{eq:EF}
\EE^{(k)}\FF^{(\ell)} \one_{(a_1,a_2)} \cong \bigoplus_{j=0}^{\min(k,\ell)} \bigoplus_{k-\ell+a_1-a_2 \brack j} 
	\FF^{(\ell-j)} \EE^{(k-j)} \one_{(a_1,a_2)} \quad  \text{if } k - \ell +a_1-a_2 \geq 0
\end{equation}
and
\begin{equation}
	\label{eq:FE}
\FF^{(\ell)}\EE^{(k)} \one_{(a_1,a_2)} \cong \bigoplus_{j=0}^{\min(k,\ell)} \bigoplus_{k-\ell-a_1+ a_2 \brack j} 
	\EE^{(k-j)} \FF^{(\ell-j)} \one_{(a_1,a_2)} \quad  \text{if } \ell-k-a_1+ a_2 \geq 0 
\end{equation}
given via the \emph{Sto\v{s}i\'{c} formulae:}
\begin{subequations}
\label{eq:Stosic}
\begin{equation}
	\label{eq:StosicEF}
\begin{tikzpicture}[anchorbase,scale=1.25]
\draw[ultra thick,->] (-.25,-.75) node[below=-2pt]{\scs$k$} to (-.25,.75);
\draw[ultra thick,<-] (.25,-.75) node[below=-2pt]{\scs$\ell$} to (.25,.75);
\node at (.75,.5){\scs$(a_1,a_2)$};
\end{tikzpicture}
=
\sum_{j=0}^{\min(k,\ell)}
\sum_{\substack{\parti,\mu,\nu \in P(j) \\ y \in P(j,k-j) \\ z \in P(j,\ell-j)}}
(-1)^{\frac{j(j+1)}{2}+|y|+|z|+(k+\ell)j}
c_{\parti,\mu,\nu,y,z}^{N_j}
\begin{tikzpicture}[anchorbase,scale=1]
\draw[ultra thick,->] (-.75,1) to (-.75,1.5) node[above=-2pt]{\scs$k$};
\draw[ultra thick,rdirected=.5] (.75,1) to (.75,1.5) node[above=-2pt]{\scs$\ell$};
\draw[ultra thick,rdirected=.25] (-.75,1) to [out=330,in=210] 
	node[pos=.7]{\tiny$\CQGbox{\mathfrak{s}_{\parti}}$} node[pos=.3,above]{\scs$j$} (.75,1);
\draw[ultra thick,directed=.55] (-.75,-2) to [out=150,in=270] (-1,-1.25) node{\tiny$\CQGbox{\mathfrak{s}_{\bar{y}}}$} 
	to [out=90,in=270] (.375,0) to [out=90,in=270] (-1,.75) to [out=90,in=210] (-.75,1);
\draw[ultra thick,rdirected=.55] (.75,-2) to [out=30,in=270] (1,-1.25) node{\tiny$\CQGbox{\mathfrak{s}_{\bar{z}}}$} 
	to [out=90,in=270] (-.375,0) to [out=90,in=270] (1,.75) to [out=90,in=330] (.75,1);
\node at (-.75,0) {\scs$\ell{-}j$};
\node at (.75,0) {\scs$k{-}j$};
\draw[ultra thick,directed=.7] (-.75,-2) to [out=30,in=150] 
	node[pos=.3]{\tiny$\CQGbox{\mathfrak{s}_{\mu}}$} node[pos=.7,below]{\scs$j$} (.75,-2);
\draw[ultra thick,directed=.5] (0,-1.125) circle (.375);
\draw node at (.375,-1.125) {\tiny$\CQGbox{\mathfrak{s}_\nu^\spadesuit}$};
\draw node at (-.5,-1.25) {\scs$j$};
\draw[ultra thick,directed=.5] (-.75,-2.5) node[below=-2pt]{\scs$k$} to (-.75,-2);
\draw[ultra thick,<-] (.75,-2.5) node[below=-2pt]{\scs$\ell$} to (.75,-2);
\node at (1.625,1){\scs$(a_1,a_2)$};
\end{tikzpicture}
\end{equation}
\begin{equation}
	\label{eq:StosicFE}
\begin{tikzpicture}[anchorbase,scale=1.25,xscale=-1]
\draw[ultra thick,->] (-.25,-.75) node[below=-2pt]{\scs$k$} to (-.25,.75);
\draw[ultra thick,<-] (.25,-.75) node[below=-2pt]{\scs$\ell$} to (.25,.75);
\node at (-.75,.5){\scs$(a_1,a_2)$};
\end{tikzpicture}
=
\sum_{j=0}^{\min(k,\ell)}
\sum_{\substack{\parti,\mu,\nu \in P(j) \\ y \in P(j,k-j) \\ z \in P(j,\ell-j)}}
(-1)^{\frac{j(j+1)}{2}+|y|+|z|}
c_{\parti,\mu,\nu,y,z}^{M_j}
\begin{tikzpicture}[anchorbase,scale=1,xscale=-1]
\draw[ultra thick,->] (-.75,1) to (-.75,1.5) node[above=-2pt]{\scs$k$};
\draw[ultra thick,rdirected=.5] (.75,1) to (.75,1.5) node[above=-2pt]{\scs$\ell$};
\draw[ultra thick,rdirected=.25] (-.75,1) to [out=330,in=210] 
	node[pos=.7]{\tiny$\CQGbox{\mathfrak{s}_{\parti}}$} node[pos=.3,above]{\scs$j$} (.75,1);
\draw[ultra thick,directed=.55] (-.75,-2) to [out=150,in=270] (-1,-1.25) node{\tiny$\CQGbox{\mathfrak{s}_{\bar{y}}}$} 
	to [out=90,in=270] (.375,0) to [out=90,in=270] (-1,.75) to [out=90,in=210] (-.75,1);
\draw[ultra thick,rdirected=.55] (.75,-2) to [out=30,in=270] (1,-1.25) node{\tiny$\CQGbox{\mathfrak{s}_{\bar{z}}}$} 
	to [out=90,in=270] (-.375,0) to [out=90,in=270] (1,.75) to [out=90,in=330] (.75,1);
\node at (-.75,0) {\scs$\ell{-}j$};
\node at (.75,0) {\scs$k{-}j$};
\draw[ultra thick,directed=.7] (-.75,-2) to [out=30,in=150] 
	node[pos=.3]{\tiny$\CQGbox{\mathfrak{s}_{\mu}}$} node[pos=.7,below]{\scs$j$} (.75,-2);
\draw[ultra thick,directed=.5] (0,-1.125) circle (.375);
\draw node at (.375,-1.125) {\tiny$\CQGbox{\Schur[\nu]^\spadesuit}$};
\draw node at (-.5,-1.25) {\scs$j$};
\draw[ultra thick,directed=.5] (-.75,-2.5) node[below=-2pt]{\scs$k$} to (-.75,-2);
\draw[ultra thick,<-] (.75,-2.5) node[below=-2pt]{\scs$\ell$} to (.75,-2);
\node at (-1.625,1){\scs$(a_1,a_2)$};
\end{tikzpicture}
\end{equation}
\end{subequations}
Here, $N_j := (a_1- a_2 + k - \ell - j)^j$, $M_j := (a_2- a_1 + \ell - k - j)^j$, and 
$c_{\parti,\mu,\nu,y,z}^{\rho}$ is an iterated Littlewood-Richardson coefficient, 
defined by the equality
$\Schur \Schur[\mu] \Schur[\nu] \Schur[y] \Schur[z] = \sum c_{\parti,\mu,\nu,y,z}^{\rho} \Schur[\rho]$.
Observe that the signs in \eqref{eq:Stosic} differ from those in \cite{KLMS}; 
see Remark \ref{rmk:rescalingredux}.

Moreover, the $1$-morphisms
\[
\begin{cases}
\FF^{(\ell)}\EE^{(k)} \one_{(a_1,a_2)} & \text{if } \ell - k \leq a_1 - a_2 \\
\EE^{(k)}\FF^{(\ell)} \one_{(a_1,a_2)} & \text{if } k - \ell \leq a_2 - a_1 \\
\end{cases}
\]
(that appear on the right-hand side of \eqref{eq:EF} and \eqref{eq:FE}) 
constitute all of the indecomposable $1$-morphisms in $\cUU_q(\gln[2])$, 
up to grading shift. \qed
\end{itemize}
\end{prop}

Now we pass from $\UU_q(\gln[2])$ to the setting of $\UU_q(\glm)$ for general $m$.
When $m \ge 3$, 
there is no combinatorial description of the Karoubi envelope of $\UU_q(\glm)$,
and, indeed, the canonical basis for $\dU(\glm)$, 
which the indecomposable 1-morphisms in $\Kar(\UU_q(\glm))$ 
categorify when working in characteristic zero \cite{Web4}, 
is not even known explicitly. 
However, using the $m=2$ case, 
it is possible to describe a partial idempotent completion $\U_q(\glm)$ of $\UU_q(\glm)$ 
that contains divided power morphisms
\[
\EE_i^{(k)} \one_{\wt} \colon \wt \to \wt + k \ee_i \, , \quad
\FF_i^{(k)} \one_{\wt}  \colon \wt \to \wt - k \ee_i 
\]
for each $1 \leq i \leq m-1$. See e.g.~\cite[Definition 2.2]{QR1}.

\begin{defn}
	\label{def:thickCQG}
Let $m \geq 1$. The \emph{thick categorified quantum group} $\U_q(\glm)$ is the 
$\Z$-additive closure of the $2$-category given as follows.
\begin{itemize}[leftmargin=*]
\item Objects are elements $\wt \in \Z^m$.
\item $1$-morphisms are generated by
\[
\EE_i^{(k)} \one_{\wt} \colon \wt \to \wt + k \ee_i \, , \quad
\FF_i^{(k)} \one_{\wt}  \colon \wt \to \wt - k \ee_i
\]
for $1 \leq i \leq m-1$, $k \in \Z_+$, and $\ee_i = (0,\ldots,1,-1,\ldots,0)$.
\item $2$-morphisms are $\K$-linear combinations of ``string diagrams'' that are generated via 
horizontal and vertical composition by those in \eqref{eq:CQGgens} and for each 
Dynkin node $i \in \{1,\ldots,m-1\}$ merge/split generators \eqref{eq:MS}, 
modulo the relations in Definition \ref{def:CQG} and the relations \eqref{eq:MStoHT} 
and \eqref{eq:explode}.
\end{itemize}
\end{defn}

\begin{prop} The category $\U_q(\glm)$ is equivalent to the partial idempotent completion of $\UU_q(\glm)$ 
which adjoins all direct summands of $\EE_i^k \one_\wt$ and $\FF_i^k \one_\wt$ for $1 \le i \le m-1$ and $k \ge 1$.
\end{prop}

\begin{proof}[Proof (Sketch)] 
The discussion in Remark \ref{rmk:presentKar} implies that one can present partial idempotent 
completions by adding one idempotent at a time, independently of other idempotents. 
Meanwhile, \cite{KLMS} proves that $\U_q(\gln[2])$ 
is equivalent to the partial idempotent completion which adds all summands of $\EE^k \one_\wt$ and $\FF^k \one_\wt$. 
Thus using the construction of \cite{KLMS} independently for each $\EE_i \one_\wt$ and $\FF_i \one_\wt$, 
one has added all direct summands of $\EE_i^k \one_\wt$ and $\FF_i^k \one_\wt$. 
\end{proof}

As in Remark \ref{rem:monoidal2cat}, we will view $\coprod_{m \geq 0} \cUU_q(\glm)$
as a monoidal $2$-category under the analogously defined external tensor product $\boxtimes$.

\begin{conv}
	\label{conv:colors}
Note that the strands in the string diagrams for $\U_q(\glm)$ require 
both a Dynkin label $i \in \{1,\ldots,m-1\}$ and 
a thickness label $k \in \Z_+$, e.g.~we have
\[
\begin{tikzpicture}[anchorbase,scale=1]
\draw[ultra thick] (0,0) node[below=-2pt]{\scs$i,k{+}\ell$} to [out=90,in=270] (0,.375);
\draw[ultra thick,->] (0,.375) to [out=150,in=270] (-.25,1) node[above=-2pt,xshift=-2pt]{\scs$i,k$};
\draw[ultra thick,->] (0,.375) to [out=30,in=270] (.25,1) node[above=-2pt,xshift=2pt]{\scs$i,\ell$};
\node at (.625,.75){$\wt$};
\end{tikzpicture}
\in \Hom^{-k \ell}(\EE_i^{(k+ \ell)} \one_\wt, \EE_i^{(k)} \EE_i^{(\ell)} \one_\wt)
\]
Instead of including both labels, 
we will denote the Dynkin label by coloring the relevant strands according to the 
following coloring of the $\glm$ Dynkin diagram:
\[
\begin{tikzpicture}[anchorbase,scale=1]
\draw[thick] (-1.75,0) to (-1.5,0);
\node at (-1.25,0) {$\mydots$};
\draw[thick] (-1,0) to (1,0);
\node at (1.25,0) {$\mydots$};
\draw[thick] (1.5,0) to (1.75,0);
%
\node at (-1.75,0) {\large$\bullet$};
\node at (-1.75,-.25) {\small$1$};
\node[red] at (-.75,0) {$\bullet$};
\node[red] at (-.75,-.25) {\small$i{-}1$};
\node[green] at (0,0) {\large$\bullet$};
\node[green] at (0,-.25) {\small$i$};
\node[blue] at (.75,0) {$\bullet$};
\node[blue] at (.75,-.25) {\small$i{+}1$};
\node at (1.75,0) {$\bullet$};
\node at (1.75,-.25) {\small$m{-}1$};
\end{tikzpicture} \; \; .
\]
To elaborate, we typically use {\color{green} green} as our arbitrary index $i$ in the Dynkin diagram. 
It is always assumed that {\color{red} red} is one less than {\color{green} green} 
and {\color{blue} blue} is one more than {\color{green} green}. 
Meanwhile, 
\textbf{black} represents an arbitrary index
(possibly equal to {\color{red} red}, {\color{green} green}, {\color{blue} blue}, 
or distant from these colors).
\end{conv}

\begin{example} 
Pairing Convention \ref{conv:colors} with
the conventions discussed in Remark \ref{rmk:Laudaparams}, 
we have $t_{\grb,\grb} = -1$, $t_{\redb,\grb} = +1$, $t_{\grb,\redb} = -1$, etc.
\end{example}

\begin{conv}
Convention \ref{Conv:thickcapscups} explains how to define thick caps/cups and thick crossings for $\U_q(\gln[2])$. 
We use the same recipe to define thick caps/cups and (unicolored) thick crossings for $\U_q(\glm)$. 
We similarly define thick crossings for different colors as 
\begin{equation}\label{eqn:multicolorthickcrossingdefn}
\begin{tikzpicture}[anchorbase,scale=1]
\draw[ultra thick,green,->] (0,0) node[below=-2pt]{\scs$k$} to [out=90,in=270] (.5,1);
\draw[ultra thick,->] (.5,0) node[below=-2pt]{\scs$\ell$} to [out=90,in=270] (0,1);
\node at (.875,.75){$\wt$};
\end{tikzpicture}
:=
\begin{tikzpicture}[anchorbase,scale=1]
\draw[ultra thick,->] (-.75,1) node[below,xshift=1pt]{$\mydots$} to (-.75,1.25);
\draw[ultra thick,green,->] (.75,1) node[below,xshift=-1pt,black]{$\mydots$} to (.75,1.25);
\draw[ultra thick,green] (-.75,-1) to (-.75,-1.25) node[below=-2pt]{\scs$k$};
\draw[ultra thick] (.75,-1) to (.75,-1.25) node[below=-2pt]{\scs$\ell$};
\draw[thick,green] (-.75,-1) to [out=150,in=210] node[pos=.85,black]{$\bullet$} 
	node[pos=.85,above,black]{\tiny$k{-}1$} (.75,1);
\draw[thick,green] (-.75,-1) to [out=30,in=330] node[pos=.85,black]{$\bullet$} (.75,1);
\draw[thick,green] (-.75,-1) to [out=0,in=270] (1.25,.5) to [out=90,in=0] (.75,1);
\draw[thick] (.75,-1) to [out=150,in=210] node[pos=.75]{$\bullet$} 
	node[pos=.85,below,xshift=10pt]{\tiny$\ell{-}2$} (-.75,1);
\draw[thick] (.75,-1) to [out=30,in=330] (-.75,1);
\draw[thick] (.75,-1) to [out=180,in=270] (-1.25,.5) node{$\bullet$} 
	node[left]{\tiny$\ell{-}1$}
	to [out=90,in=180] (-.75,1);
\node at (1.5,.75){$\wt$};
\end{tikzpicture} \, .
\end{equation}
These thick crossings still satisfy fork-slide relations with merge/split morphisms 
as in Convention \ref{Conv:thickcapscups}.
For example, 
\begin{equation}
	\label{eq:multicolorotherthick}
\begin{tikzpicture}[anchorbase,scale=1]
\draw[ultra thick,green] (-.25,-.5) node[below=-2pt]{\scs$k_1$} to [out=90,in=210] (0,0);
\draw[ultra thick,green] (.25,-.5) node[below=-2pt]{\scs$k_2$} to [out=90,in=330] (0,0);
\draw[ultra thick,green,->] (0,0) to [out=90,in=270] (.5,1);
\draw[ultra thick,->] (.75,-.5) node[below=-2pt]{\scs$\ell$} to [out=90,in=270] (0,1);
\node at (.875,.75){$\wt$};
\end{tikzpicture}
=
\begin{tikzpicture}[anchorbase,scale=1]
\draw[ultra thick,green] (-.25,-.5) node[below=-2pt]{\scs$k_1$} to [out=90,in=210] (.5,.625);
\draw[ultra thick,green] (.25,-.5) node[below=-2pt]{\scs$k_2$} to [out=90,in=330] (.5,.625);
\draw[ultra thick,green,->] (.5,.625) to [out=90,in=270] (.5,1);
\draw[ultra thick,->] (.75,-.5) node[below=-2pt]{\scs$\ell$} to [out=90,in=270] (0,1);
\node at (.875,.75){$\wt$};
\end{tikzpicture} \, ,
\end{equation}
which follows using \eqref{eq:cubicKLR}, 
in the case where $i$ and $j$ are green and $k$ is black.
\end{conv}

%
\section{Involutions on (part of) the categorified quantum group}\label{s:inv}
%

Recall from Section \ref{ss:inmoredepth} that we seek an involution 
of the $2$-category $\U_q(\glm)$, in order to carry out the equivariantization procedure 
sketched there (and defined concretely below in \S\ref{sec:equiv}).
For the sake of this paper, we work with the following definition, 
rather than the more-general notion in which the stated equality of functors is replaced 
by a natural isomorphism.

\begin{defn}\label{D:2-involution}
A \emph{(strict) involution} of a $\ak$-linear $2$-category $\Cat$ is a a $\ak$-linear $2$-functor 
$\sigma \colon \Cat\rightarrow \Cat$ such that $\sigma\circ \sigma = \id_{\Cat}$. 
\end{defn}

Now, fix $n \ge 1$ and $m \ge 1$.
Let $\wtn$ denote the $\glm$ weight $(n,n,\ldots,n)$, 
thus $2\wtn = (2n,2n,\ldots,2n)$.
To begin, we seek an involution $\tau$ of $\UU_q(\glm)$ that
\begin{itemize}
	\item is given on $\glm$ weights by sending $\wt \mapsto 2\wtn - \wt$,
	\item is determined on $1$-morphisms by $\EE_i \one_\wt \mapsto \FF_i \one_{2\wtn-\wt}$, and
	\item acts on new bubble $2$-morphisms by 
		$\mathfrak{s}_{\lambda}(\X_i)\mapsto \mathfrak{s}_{\lambda^{t}}(\X_i)$.
\end{itemize}

The condition on objects and $1$-morphisms guarantee that $\tau$ maps 
the indecomposable $1$-morphism $\FE_i:=\FF_i\EE_i\one_{\wtn}$ to $\EE_i\FF_i\one_{\wtn}\cong \FE_i$. 
As outlined in Section \ref{ss:introinter}, an equivariant structure on $\FE_i$ gives a natural candidate 
for a $1$-morphism categorifying the distinguished elements $\xx_i \in \End_{U_q(\son)}(S^{\otimes m})$.
The condition on $2$-morphisms is present in order for this involution to be compatible with 
the involution on $H^\ast(\Gr{n}{2n})$ from Section \ref{ss:nutshell}. 

As it turns out, these requirements 
(together with integrality assumptions)
essentially determine $\tau$ on the generating morphisms of $\UU_q(\glm)$, 
up to signs. 
For example, taken together with equation \eqref{dotstonewbubbles}, 
they imply that the dot endomorphism in $\End^2(\EE_i \one_\wt)$ must be sent to 
minus the dot endomorphism in $\End^2(\FF_i \one_{2\wtn - \wt})$.
In \S \ref{ss:symmetries}, 
we find conditions on the requisite signs that guarantee that 
the above recipe for $\tau$ is a well-defined automorphism of $\UU_q(\glm)$. 
Next, in \S \ref{ss:inv?} we investigate further conditions on these signs which 
would imply that $\tau$ is an involution. 
As it turns out, 
when $m \geq 3$ it is surprisingly difficult to satisfy these conditions!
(See Theorem \ref{thm:nonaivesolution} for the precise statement.) 
Fortunately, as we show in \S \ref{ss:invonpart},
we \emph{are} able to construct a symmetry $\tau$ of order $4$ 
for all $m \geq 1$ which restricts to an involution on an appropriate $2$-subcategory 
that suffices for our present considerations.
To conclude this section, 
in \S \ref{ss:dependence-on-n} we discuss the dependence on our chosen value of $n$
and in \S \ref{ss:extended-to-thick} we extend $\tau$ to thick calculus.

For the remainder of this section, 
$\K$ is permitted to be any integral domain for which $2 \neq 0$.

\subsection{A family of symmetries}
	\label{ss:symmetries}

We now define a family of autoequivalences of categorified $\glm$, 
one for each\footnote{Actually, in Definition \ref{def:symmetry}, we do so for 
$n \geq 0$; we use the $n=0$ case in \S \ref{ss:dependence-on-n}.
In fact, our definition works for all $n \in \Z$, 
but those for $n < 0$ are not of use for us.} 
$n \geq 1$.
By slight abuse of notation, we will denote all of these
autoequivalences by $\tau$.

\begin{rem}\label{rem:notKL}
Each of these symmetries will send $\EE_i$ to $\FF_i$, 
which is similar to the first automorphism of $\UU_q(\slm)$ defined by 
Khovanov--Lauda in \cite[\S 3.3.2]{KL3}.
It would be easy to confuse our automorphism with theirs, 
as both act by rescaling the generating $2$-morphisms and then reversing 
the orientation on string diagrams.
However, these automorphisms are distinct: 
as discussed above, ours rescales the ``dot'' endomorphism by $-1$ and
does not rescale the uni-colored crossing, while the opposite is true for 
the Khovanov--Lauda symmetry.
If $2=0$ in $\K$ (a case not permitted by our assumptions on $\K$ above), 
then $1= -1$ and the symmetry we define below will agree with Khovanov--Lauda's automorphism.
\end{rem}

Given the action on $2$-morphisms, 
our symmetry is covariant for both $1$-morphism and $2$-morphism composition.
Moving forward, we will say that
$\tau$ is \emph{covariant} to mean that it is covariant for $2$-morphism composition, 
and \emph{monoidal} to mean that it is covariant for $1$-morphism composition.

\begin{defn}
	\label{def:symmetry}
Fix $n \geq 0$. For each $i \in \{1, \ldots, m-1\}$, 
let $r_i$, $r_i'$, $l_i$, and $l_i'$ be functions from the $\glm$ weight lattice $\Z^m$ to $\Z/2$. 
For each pair $i, j \in \{1, \ldots, m-1\}$, 
let $v_{i,j}$  be a function from $\Z^m$ to $\Z/2$.
Associated to this data, define the following map $\tau$ on the generating data of $\UU_q(\glm)$. 
In the next theorem, we will state precise conditions on $r_i$, $r_i'$, $l_i$, $l_i'$, and $v_{i,j}$ 
which imply that $\tau$ extends to a $2$-functor $\tau \colon \UU_q(\glm) \to \UU_q(\glm)$.
\begin{itemize}
\item \textbf{Objects:} $\wtn + \wt \xmapsto{\tau} \wtn - \wt$ (i.e.~$\wt \xmapsto{\tau} 2\wtn-\wt$).
	
\item \textbf{$1$-morphisms:} $\EE_i \one_{\wtn+\wt} \xmapsto{\tau} \FF_i \one_{\wtn-\wt}$.

\item \textbf{$2$-morphisms:} 
\begin{equation}
\begin{gathered} \label{eq:inv}
\begin{tikzpicture}[anchorbase,scale=1]
\draw[thick,green,->] (0,0) to node[black]{$\bullet$} (0,1);
\node at (.5,.75){\scs$\wtn{+}\wt$};
\end{tikzpicture}
\xmapsto{\tau}
- \begin{tikzpicture}[anchorbase,scale=1]
\draw[thick,green,<-] (0,0) to node[black]{$\bullet$} (0,1);
\node at (.5,.75){\scs$\wtn{-}\wt$};
\end{tikzpicture}
\, , \quad
\begin{tikzpicture}[anchorbase,scale=1]
\node at (0,0){\NB{e_r(\X_i)}};
\node at (.875,.375){\scs$\wtn{+}\wt$};
\end{tikzpicture}
\xmapsto{\tau}
\begin{tikzpicture}[anchorbase,scale=1]
\node at (0,0){\NB{h_r(\X_i)}};
\node at (.875,.375){\scs$\wtn{-}\wt$};
\end{tikzpicture}
\, , \\
\begin{tikzpicture}[anchorbase,scale=1]
\draw[thick,green,->] (0,0) to [out=90,in=270] (.5,1);
\draw[thick,->] (.5,0) to [out=90,in=270] (0,1);
\node at (.875,.75){\scs$\wtn{+}\wt$};
\end{tikzpicture}
\xmapsto{\tau}
(-1)^{v_{\grb\bullet}(\wt)}
\begin{tikzpicture}[anchorbase,scale=1]
\draw[thick,green,<-] (0,0) to [out=90,in=270] (.5,1);
\draw[thick,<-] (.5,0) to [out=90,in=270] (0,1);
\node at (.875,.75){\scs$\wtn{-}\wt$};
\end{tikzpicture}
\, , \\
\begin{tikzpicture}[anchorbase,scale=1]
\draw[thick,green,->] (-.25,0) to [out=90,in=180] (0,.5) 
	to [out=0,in=90] (.25,0);
\node at (.625,.5){\scs$\wtn{+}\wt$};
\end{tikzpicture}
\xmapsto{\tau}
(-1)^{r_{\grb}(\wt)}
\begin{tikzpicture}[anchorbase,scale=1]
\draw[thick,green,<-] (-.25,0) to [out=90,in=180] (0,.5) 
	to [out=0,in=90] (.25,0);
\node at (.625,.5){\scs$\wtn{-}\wt$};
\end{tikzpicture}
\, , \quad
\begin{tikzpicture}[anchorbase,scale=1]
\draw[thick,green,->] (-.25,0) to [out=270,in=180] (0,-.5) 
	to [out=0,in=270] (.25,0);
\node at (.625,-.5){\scs$\wtn{+}\wt$};
\end{tikzpicture}
\xmapsto{\tau}
(-1)^{r'_{\grb}(\wt)}
\begin{tikzpicture}[anchorbase,scale=1]
\draw[thick,green,<-] (-.25,0) to [out=270,in=180] (0,-.5) 
	to [out=0,in=270] (.25,0);
\node at (.625,-.5){\scs$\wtn{-}\wt$};
\end{tikzpicture}
\, , \\
\begin{tikzpicture}[anchorbase,scale=1]
\draw[thick,green,<-] (-.25,0) to [out=90,in=180] (0,.5) 
	to [out=0,in=90] (.25,0);
\node at (.625,.5){\scs$\wtn{+}\wt$};
\end{tikzpicture}
\xmapsto{\tau}
(-1)^{l_{\grb}(\wt)}
\begin{tikzpicture}[anchorbase,scale=1]
\draw[thick,green,->] (-.25,0) to [out=90,in=180] (0,.5) 
	to [out=0,in=90] (.25,0);
\node at (.625,.5){\scs$\wtn{-}\wt$};
\end{tikzpicture}
\, , \quad
\begin{tikzpicture}[anchorbase,scale=1]
\draw[thick,green,<-] (-.25,0) to [out=270,in=180] (0,-.5) 
	to [out=0,in=270] (.25,0);
\node at (.625,-.5){\scs$\wtn{+}\wt$};
\end{tikzpicture}
\xmapsto{\tau}
(-1)^{l'_{\grb}(\wt)}
\begin{tikzpicture}[anchorbase,scale=1]
\draw[thick,green,->] (-.25,0) to [out=270,in=180] (0,-.5) 
	to [out=0,in=270] (.25,0);
\node at (.625,-.5){\scs$\wtn{-}\wt$};
\end{tikzpicture} \, .
\end{gathered}
\end{equation}
\end{itemize}
\end{defn}

\begin{rem}\label{R:tau-depends-on-n}
We emphasize that $\tau$ depends on the fixed integer $n>0$ appearing in $\wtn$, 
although our notation suppresses this.
For further discussion regarding dependence on $n$, in some special cases, see Section \ref{ss:dependence-on-n}. 
\end{rem}

\begin{thm}
	\label{thm:symmetry}
The map $\tau$ defined above extends to a unique 
(invertible, monoidal, covariant, graded, $\K$-linear) $2$-functor $\tau \colon \UU_q(\glm) \to \UU_q(\glm)$ 
if and only if the following equalities hold (in $\Z/2$) for all $\wt$ in the weight lattice, 
and for all $i, j, k \in \{1, \ldots, m-1\}$:
\begin{subequations} \label{tauconditions}
\begin{gather} 
\label{primesign} 
r_i'(\wt) \equiv r_i(\wt + \ee_i) \, , \quad l_i'(\wt) \equiv l_i(\wt - \ee_i). \\
\label{bubblesign} 
r_i(\wt) + l_i(\wt - \ee_i) \equiv a_i - a_{i+1} = \langle \ee_i^\vee , \wt \rangle \, , \\	
\label{samecrossingsign} 
v_{i,i}(\wt) \equiv 0 \, , \\
\label{diffcrossingsign} 
v_{i,j}(\wt) + v_{j,i}(\wt) \equiv i \cdot j = \langle \ee_i^\vee, \ee_j \rangle \, \\
 \label{R3sign} 
v_{i,j}(\wt) + v_{i,k}(\wt) + v_{j,k}(\wt) + v_{i,j}(\wt + \ee_k) + v_{i,k}(\wt + \ee_j) + v_{j,k}(\wt + \ee_i) \equiv 0 \, . \end{gather}
\end{subequations}
\end{thm}

Here, and for the remainder of this section, 
we use $\equiv$ to denote equality modulo $2$. 
Note that $i \cdot j$ in \eqref{diffcrossingsign} is odd when $i$ and $j$ are adjacent 
in the Dynkin diagram, and even otherwise.

In Remark \ref{rem:togetsymmetry} below 
we demonstrate that functions satisfying \eqref{tauconditions} 
do indeed exist, and in Theorem \ref{thm:involution}
we fix a choice (with additional desired properties) 
that we use for the duration.
For easy reference, 
whenever \eqref{tauconditions} holds, 
the value of $\tau$ on non-generating morphisms 
is as follows:

\begin{equation}
	\label{eq:tauonSchur}
\begin{tikzpicture}[anchorbase,scale=1]
\node at (0,0){\NB{\Schur(\X_i)}};
\node at (.875,.375){\scs$\wtn{+}\wt$};
\end{tikzpicture}
\xmapsto{\tau}
\begin{tikzpicture}[anchorbase,scale=1]
\node at (0,0){\NB{\Schur[\parti^t](\X_i)}};
\node at (.875,.375){\scs$\wtn{-}\wt$};
\end{tikzpicture} \, 
\end{equation}
\begin{equation}
	\label{eq:downdot}
\begin{tikzpicture}[anchorbase,scale=1]
\draw[thick,green,<-] (0,0) to node[black]{$\bullet$} (0,1);
\node at (.5,.75){\scs$\wtn{+}\wt$};
\end{tikzpicture}
\xmapsto{\tau}
- \begin{tikzpicture}[anchorbase,scale=1]
\draw[thick,green,->] (0,0) to node[black]{$\bullet$} (0,1);
\node at (.5,.75){\scs$\wtn{-}\wt$};
\end{tikzpicture}
\end{equation}
\begin{align*}
\begin{tikzpicture}[anchorbase,scale=1]
\draw[thick,green,->] (0,0) to [out=90,in=270] (.5,1);
\draw[thick,<-] (.5,0) to [out=90,in=270] (0,1);
\node at (1,.75){\scs$\wtn{+}\wt$};
\end{tikzpicture}
&\xmapsto{\tau}
(-1)^{v_{\bullet \grb}(\wt - \ee_{\bullet}) + r_{\bullet}(\wt) + r_{\bullet}(\wt + \ee_{\grb})}
\begin{tikzpicture}[anchorbase,scale=1]
\draw[thick,green,<-] (0,0) to [out=90,in=270] (.5,1);
\draw[thick,->] (.5,0) to [out=90,in=270] (0,1);
\node at (1,.75){\scs$\wtn{-}\wt$};
\end{tikzpicture} \\
\begin{tikzpicture}[anchorbase,scale=1]
\draw[thick,green,<-] (0,0) to [out=90,in=270] (.5,1);
\draw[thick,->] (.5,0) to [out=90,in=270] (0,1);
\node at (1,.75){\scs$\wtn{+}\wt$};
\end{tikzpicture}
&\xmapsto{\tau}
(-1)^{v_{\bullet  \grb}(\wt - \ee_{\grb}) + l_{\grb}(\wt - \ee_{\grb}) + l_{\grb}(\wt - \ee_{\grb} + \ee_{\bullet})}
\begin{tikzpicture}[anchorbase,scale=1]
\draw[thick,green,->] (0,0) to [out=90,in=270] (.5,1);
\draw[thick,<-] (.5,0) to [out=90,in=270] (0,1);
\node at (1,.75){\scs$\wtn{-}\wt$};
\end{tikzpicture}
\end{align*}
and
\begin{equation} 
	\label{eq:downwardcross}
\begin{tikzpicture}[anchorbase,scale=1]
\draw[thick,green,<-] (0,0) to [out=90,in=270] (.5,1);
\draw[thick,<-] (.5,0) to [out=90,in=270] (0,1);
\node at (.875,.75){\scs$\wtn{+}\wt$};
\end{tikzpicture}
\xmapsto{\tau}
(-1)^{v_{\grb\bullet}'(\wt)} 
\begin{tikzpicture}[anchorbase,scale=1]
\draw[thick,green,->] (0,0) to [out=90,in=270] (.5,1);
\draw[thick,->] (.5,0) to [out=90,in=270] (0,1);
\node at (.875,.75){\scs$\wtn{-}\wt$};
\end{tikzpicture},
\end{equation}
where
\begin{equation} \label{primev} 
v_{i,j}'(\wt) \equiv v_{i,j}(\wt - \ee_i - \ee_j) + r_i(\wt) + r_j(\wt) + r_i(\wt - \ee_j) + r_j(\wt - \ee_i) \, .
\end{equation}
Equation \ref{eq:tauonSchur} follows from noting that the standard involution of symmetric functions, 
defined by sending elementary symmetric functions to complete symmetric functions as in equation \eqref{eq:inv}, 
sends the Schur function associated to a partition to the Schur function associated to the transpose partition. 
It is routine to use equation \eqref{eq:inv}, which defines $\tau$ on generators, 
along with the defining relations of $\UU_q(\glm)$, 
to verify equations \eqref{eq:downdot} and \eqref{eq:downwardcross}

\begin{rem} 
We view \eqref{primesign} (and \eqref{primev}) as expressing the primed variables 
in terms of the unprimed ones, 
and the remaining equations as conditions on the unprimed variables. 
Note that \eqref{samecrossingsign} also pairs with \eqref{primev} 
to imply that $v_{i,i}'(\wt) \equiv 0$.  \end{rem}

\begin{proof}[Proof (of Theorem \ref{thm:symmetry})]
We need to check under what conditions $\tau$ preserves each relation in $\UU_q(\glm)$, 
and match these conditions to the conditions in \eqref{tauconditions}. 
First we will derive each of the conditions above by checking a particular relation:
\begin{itemize}
\item The biadjunction relation gives \eqref{primesign}.
\item Evaluation of real bubbles gives \eqref{bubblesign}.
\item Cyclicity of the crossing gives \eqref{primev}, and the formula on sideways crossings.
\item The dot slide relation gives \eqref{samecrossingsign}.
\item The quadratic KLR relation gives \eqref{diffcrossingsign}.
\item The cubic KLR relation gives \eqref{R3sign}.
\end{itemize}
Then, we confirm that the remaining relations are satisfied under the conditions \eqref{tauconditions}. 
The reader wishing to skip the computations will not miss anything important by skipping ahead. \\

\begin{itemize}[leftmargin=*]

\item \textbf{Adjunction and cyclicity of dots:}

Consider the action of $\tau$ on the (downward, right) biadjunction relation:
\begin{equation} \label{eq:biadjcomp}
\left(
\begin{tikzpicture}[anchorbase,scale=1]
\draw[thick,green,<-] (0,0) to (0,1); 
\node at (.5,.75){\scs$\wtn{+}\wt$};
\end{tikzpicture}
=
\begin{tikzpicture}[anchorbase,scale=1]
\draw[thick,green,->] (-.5,.75) to (-.5,0) to [out=270,in=180] (-.25,-.5) to [out=0,in=270] 
	(0,0) to [out=90,in=180] (.25,.5) to [out=0,in=90] (.5,0) to (.5,-.75);
\node at (.75,.5){\scs$\wtn{+}\wt$};
\end{tikzpicture}
\right)
\xmapsto{\tau}
\left(
\begin{tikzpicture}[anchorbase,scale=1]
\draw[thick,green,->] (0,0) to (0,1); 
\node at (.5,.75){\scs$\wtn{-}\wt$};
\end{tikzpicture}
\stackrel{?}{=}
(-1)^{r_{\grb}(\wt)} (-1)^{r'_{\grb}(\wt-\ee_{\grb})} \,
\begin{tikzpicture}[anchorbase,scale=1]
\draw[thick,green,<-] (-.5,.75) to (-.5,0) to [out=270,in=180] (-.25,-.5) to [out=0,in=270] 
	(0,0) to [out=90,in=180] (.25,.5) to [out=0,in=90] (.5,0) to (.5,-.75);
\node at (.875,.5){\scs$\wtn{-}\wt$};
\end{tikzpicture}
\right) \, .
\end{equation}
This relation is preserved by $\tau$ if and only if $r_{\grb}(\wt) \equiv r'_{\grb}(\wt - \ee_{\grb})$. 
Checking the same relation with orientations reversed gives $l_{\grb}(\wt) \equiv l'_{\grb}(\wt + \ee_{\grb})$.
These two equations are equivalent to \eqref{primesign}, and that equation also gives the
other versions of the biadjunction relation (downward-left and upward-right).

There is a similar isotopy relation which states that the two ways of defining the dot endomorphism of $\FF_i \one_\wt$ as the 
left and right mates of the generating dot endomorphism of $\EE_i \one_\wt$ agree.
This follows similarly to \eqref{eq:biadjcomp}, 
with an extra factor of $-1$ (coming from the dot) on both sides of the equation. 
As a consequence, we see that \eqref{primesign} implies equation \eqref{eq:downdot} above. \\

\item \textbf{Bubbles to new bubbles:}

Before continuing, we discuss the action of $\tau$ on bubbles.
Recall from the discussion following \eqref{realbubdef} that there
are two kinds of bubbles: real and fake. 
Real bubbles are genuine compositions of a cup, a cap, and dots; 
in the formula \eqref{realbubdef}, the number of dots is a non-negative integer. 
Thus, $\tau$ acts on real bubbles based on its action on the cup, cap, and dot, e.g.
\begin{equation} \label{eq:realbubbles} 
\begin{tikzpicture}[anchorbase,scale=1.25]
\draw[thick,green,<-] (.25,0) arc[start angle=0,end angle=360,radius=.25] 
	node[pos=.75,black]{$\bullet$} node[pos=.75,black,below]{\scs$c$};
\node at (.375,.375){\scs$\wtn{+}\wt$};
\end{tikzpicture}
\xmapsto{\tau}
(-1)^{r_{\grb}(\wt) + l'_{\grb}(\wt) + c} \
\begin{tikzpicture}[anchorbase,scale=1.25]
\draw[thick,green,->] (.25,0) arc[start angle=0,end angle=360,radius=.25] 
	node[pos=.75,black]{$\bullet$} node[pos=.75,black,below]{\scs$c$};
\node at (.5,.375){\scs$\wtn{-}\wt$};
\end{tikzpicture} \, .
\end{equation}
On the other hand, equation \eqref{newbub} shows that real bubbles
are also equal to a new bubble symmetric function, which, by \eqref{eq:inv}, 
is acted on by the involution $\tau$ via $e_r \mapsto h_r$.
These actions must agree for $\tau$ to be well-defined, 
the consequences of which we now explore.

Consider a real clockwise bubble of degree $2r \geq 0$, 
then $c = a_{i+1} - a_i - 1 + r \geq 0$.
Equation \eqref{newbub} gives that the real clockwise bubble is equal to $(-1)^{n+a_i-1}h_r(\X_{i} - \X_{i+1})$, 
which is sent by $\tau$ to $(-1)^{n+a_i - 1} e_r(\X_{i} - \X_{i+1})$. 
Meanwhile, equation \eqref{eq:realbubbles} shows that a real clockwise bubble is also sent by $\tau$ 
to $(-1)^{r_{\grb}(\wt) + l'_{\grb}(\wt) + c}$ times a counterclockwise bubble, 
and the counterclockwise bubble is equal to $(-1)^{n-a_i} h_r(\X_{i+1} - \X_i)$. 
Since $e_r(\X) = (-1)^r h_r(-\X)$, the two actions of $\tau$ agree provided
\[
(-1)^{n+a_i - 1} = (-1)^{n-a_i} (-1)^r (-1)^{r_{\grb}(\wt) + l'_{\grb}(\wt) + a_{i+1} - a_i - 1 + r}.
\]
Cancelling terms, this gives
$r_{\grb}(\wt) + l'_{\grb}(\wt) + a_{i+1} - a_i \equiv 0$, 
which is precisely \eqref{bubblesign}, 
after applying \eqref{primesign} to replace $l'_{\grb}(\wt)$ with $l_{\grb}(\wt - \ee_i)$.

The corresponding computation for counterclockwise bubbles is left to the reader. 
Given \eqref{primesign}, the two actions of $\tau$ again agree if and only if \eqref{bubblesign} holds. 
To see this most effortlessly, one should study a bubble in a region labeled $\wtn + (\wt - \ee_i)$.

Lastly, recall that fake bubbles (where the number of dots in \eqref{realbubdef} is a negative number) 
are defined via \eqref{newbub} as formal symmetric functions, 
and $\tau$ acts on them accordingly. 
Since fake bubbles are not genuine compositions of a cap, cup, and dots, 
there is no requirement for $\tau$ to act on them ``as though they were compositions'', 
e.g.~by the sign $(-1)^{r_{\grb}(\wt) + l'_{\grb}(\wt) + c}$ (where $c$ is a negative number) 
for fake clockwise green bubbles. 
However, fake bubbles do indeed satisfy this formula: \eqref{eq:realbubbles} still holds when $c < 0$.
The same computation as above serves as the proof, 
where now one uses \eqref{bubblesign} rather than deriving \eqref{bubblesign}. \\

\item \textbf{Cyclicity of crossings:}

Now consider the isotopy relation in $\UU_q(\glm)$ 
which states that both the left and right mates under adjunction of the upward crossing agree,
i.e.~that a downward crossing may be defined as either mate.
For this purpose we temporarily assume that \eqref{eq:downwardcross} 
is the definition of $\tau$ on the downward crossing, 
and check both compatibilities.

We compute:
\begin{equation}
\begin{tikzpicture}[anchorbase,scale=1]
\draw[thick,green,<-] (0,0) to [out=90,in=270] (.5,1);
\draw[thick,<-] (.5,0) to [out=90,in=270] (0,1);
\node at (1,.75){\scs$\wtn{+}\wt$};
\end{tikzpicture}
=
\begin{tikzpicture}[anchorbase,scale=1]
\draw[thick,green,->] (-.5,1.75) to (-.5,0) \arc{180}{360}{.25} to [out=90,in=270] (.5,1) \arc{180}{0}{.25} to (1,-.75);
\draw[thick,->] (-1,1.75) to (-1,0) \arc{180}{360}{.75} to [out=90,in=270] (0,1) \arc{180}{0}{.75} to (1.5,-.75);
\node at (2,.75){\scs$\wtn{+}\wt$};
\end{tikzpicture}
\xmapsto{\tau}
(-1)^{\substack{v_{\grb\bullet}(\wt-\ee_{\bullet}-\ee_{\grb}) + r_\bullet(\wt) + r_{\grb}(\wt-\ee_{\bullet})\\
	+ r'_{\bullet}(\wt-\ee_{\grb}-\ee_{\bullet}) + r'_{\grb}(\wt-\ee_{\grb}) }}
\begin{tikzpicture}[anchorbase,scale=1]
\draw[thick,green,<-] (-.5,1.75) to (-.5,0) \arc{180}{360}{.25} to [out=90,in=270] (.5,1) \arc{180}{0}{.25} to (1,-.75);
\draw[thick,<-] (-1,1.75) to (-1,0) \arc{180}{360}{.75} to [out=90,in=270] (0,1) \arc{180}{0}{.75} to (1.5,-.75);
\node at (2,.75){\scs$\wtn{-}\wt$};
\end{tikzpicture}
\end{equation}
Thus, for consistency with \eqref{eq:inv}, we need
\begin{equation}
	\label{eq:crossinv}
 v_{\grb\bullet}'(\wt) \equiv 
 	v_{\grb\bullet}(\wt-\ee_{\bullet}-\ee_{\grb}) + r_\bullet(\wt) + r_{\grb}(\wt-\ee_{\bullet}) 
 		+ r'_{\bullet}(\wt-\ee_{\grb}-\ee_{\bullet}) + r'_{\grb}(\wt-\ee_{\grb}) \, .
\end{equation}
Using \eqref{primesign}, one obtains \eqref{primev}, as desired. 
Note that when $\bullet = \grb$, many terms cancel and \eqref{eq:crossinv} (or \eqref{primev}) simply becomes
$v_{\grb\grb}'(\wt) \equiv v_{\grb\grb}(\wt-2\ee_{\grb})$.

For the other mate, we compute:
\begin{equation}
\begin{tikzpicture}[anchorbase,scale=1]
\draw[thick,green,<-] (0,0) to [out=90,in=270] (.5,1);
\draw[thick,<-] (.5,0) to [out=90,in=270] (0,1);
\node at (1,.75){\scs$\wtn{+}\wt$};
\end{tikzpicture}
=
\begin{tikzpicture}[anchorbase,scale=1,xscale=-1]
\draw[thick,->] (-.5,1.75) to (-.5,0) \arc{180}{360}{.25} to [out=90,in=270] (.5,1) \arc{180}{0}{.25} to (1,-.75);
\draw[thick,green,->] (-1,1.75) to (-1,0) \arc{180}{360}{.75} to [out=90,in=270] (0,1) \arc{180}{0}{.75} to (1.5,-.75);
\node at (-1.5,.75){\scs$\wtn{+}\wt$};
\end{tikzpicture}
\xmapsto{\tau}
(-1)^{\substack{v_{\grb\bullet}(\wt-\ee_{\grb}-\ee_{\bullet})  + l'_{\grb}(\wt) + l'_{\bullet}(\wt-\ee_{\grb}) \\
+ l_{\grb}(\wt-\ee_{\grb}-\ee_{\bullet}) + l_{\bullet}(\wt-\ee_{\bullet})
}}
\begin{tikzpicture}[anchorbase,scale=1,xscale=-1]
\draw[thick,<-] (-.5,1.75) to (-.5,0) \arc{180}{360}{.25} to [out=90,in=270] (.5,1) \arc{180}{0}{.25} to (1,-.75);
\draw[thick,green,<-] (-1,1.75) to (-1,0) \arc{180}{360}{.75} to [out=90,in=270] (0,1) \arc{180}{0}{.75} to (1.5,-.75);
\node at (-1.5,.75){\scs$\wtn{-}\wt$};
\end{tikzpicture} \, ,
\end{equation}
which requires
\[
v_{\grb\bullet}'(\wt) \equiv 
	v_{\grb\bullet}(\wt-\ee_{\grb}-\ee_{\bullet}) + l_{\grb}(\wt-\ee_{\grb}-\ee_{\bullet}) 
		+ l_{\bullet}(\wt-\ee_{\bullet}) + l'_{\grb}(\wt) + l'_{\bullet}(\wt-\ee_{\grb}) \, .
\]
Substituting \eqref{eq:crossinv} (which we now assume) for the left-hand side 
and applying \eqref{primesign} gives the condition
\[
r_\bullet(\wt) + r_{\grb}(\wt-\ee_{\bullet}) + r_{\bullet}(\wt-\ee_{\grb}) + r_{\grb}(\wt) + 
l_{\grb}(\wt-\ee_{\grb}-\ee_{\bullet}) + l_{\bullet}(\wt-\ee_{\bullet}) + l_{\grb}(\wt - \ee_{\grb}) + l_{\bullet}(\wt-\ee_{\grb} - \ee_{\bullet}) \equiv 0.
\]
Using \eqref{bubblesign} four times, this becomes
\[
\langle \ee_{\bullet}^\vee, \wt \rangle + \langle \ee_{\grb}^\vee, \wt - \ee_{\bullet} \rangle 
	+ \langle \ee_{\bullet}^\vee, \wt - \ee_{\grb} \rangle + \langle \ee_{\grb}^\vee, \wt \rangle
\equiv 0
\]
which holds by the symmetry 
$\langle \ee_{\grb}^{\vee}, \ee_{\bullet} \rangle = \langle \ee_{\bullet}^{\vee}, \ee_{\grb} \rangle$
of the type $A$ Cartan matrix. 
Hence, \eqref{eq:crossinv} is the only requirement, 
so no additional relations are needed for cyclicity of crossings.

Before studying the remaining relations, let us record how $\tau$ acts on the rightward sideways crossing:
\begin{equation} 
	\label{rightwardcross} 
\begin{multlined}
\begin{tikzpicture}[anchorbase,scale=1]
\draw[thick,green,->] (0,0) to [out=90,in=270] (.5,1);
\draw[thick,<-] (.5,0) to [out=90,in=270] (0,1);
\node at (1,.75){\scs$\wtn{+}\wt$};
\end{tikzpicture}	
=
\begin{tikzpicture}[anchorbase,scale=1]
\draw[thick,->] (-.5,1.5) to (-.5,0) to [out=270,in=180] (-.25,-.25) to [out=0,in=270] (0,0) 
	to [out=90,in=270] (.5,1) to [out=90,in=180] (.75,1.25) to [out=0,in=90] (1,1) to (1,-.5);
\draw[thick,green,->] (.5,-.5) to (.5,0) to [out=90,in=270] (0,1) to (0,1.5);
\node at (1.5,.75){\scs$\wtn{+}\wt$};
\end{tikzpicture}
\xmapsto{\tau} (-1)^{v_{\bullet \grb}(\wt - \ee_{\bullet}) + r_{\bullet}(\wt) + r'_{\bullet}(\wt - \ee_{\bullet} + \ee_{\grb})}
\begin{tikzpicture}[anchorbase,scale=1]
\draw[thick,green,<-] (0,0) to [out=90,in=270] (.5,1);
\draw[thick,->] (.5,0) to [out=90,in=270] (0,1);
\node at (1,.75){\scs$\wtn{-}\wt$};
\end{tikzpicture} \\
\stackrel{\eqref{primesign}}{=} (-1)^{v_{\bullet \grb}(\wt - \ee_{\bullet}) + r_{\bullet}(\wt) + r_{\bullet}(\wt + \ee_{\grb})}
\begin{tikzpicture}[anchorbase,scale=1]
\draw[thick,green,<-] (0,0) to [out=90,in=270] (.5,1);
\draw[thick,->] (.5,0) to [out=90,in=270] (0,1);
\node at (1,.75){\scs$\wtn{-}\wt$};
\end{tikzpicture}
\end{multlined}
\end{equation}
and the leftward sideways crossing:
\begin{equation} 
	\label{leftwardcross} 
\begin{multlined}
\begin{tikzpicture}[anchorbase,scale=1]
\draw[thick,green,<-] (0,0) to [out=90,in=270] (.5,1);
\draw[thick,->] (.5,0) to [out=90,in=270] (0,1);
\node at (1,.75){\scs$\wtn{+}\wt$};
\end{tikzpicture}	
=
\begin{tikzpicture}[anchorbase,scale=1,xscale=-1]
\draw[thick,green,->] (-.5,1.5) to (-.5,0) to [out=270,in=180] (-.25,-.25) to [out=0,in=270] (0,0) 
	to [out=90,in=270] (.5,1) to [out=90,in=180] (.75,1.25) to [out=0,in=90] (1,1) to (1,-.5);
\draw[thick,->] (.5,-.5) to (.5,0) to [out=90,in=270] (0,1) to (0,1.5);
\node at (-1,.75){\scs$\wtn{+}\wt$};
\end{tikzpicture}
\xmapsto{\tau} (-1)^{v_{\bullet \grb}(\wt - \ee_{\grb}) + l'_{\grb}(\wt) + l_{\grb}(\wt - \ee_{\grb} + \ee_{\bullet})}
\begin{tikzpicture}[anchorbase,scale=1]
\draw[thick,green,->] (0,0) to [out=90,in=270] (.5,1);
\draw[thick,<-] (.5,0) to [out=90,in=270] (0,1);
\node at (1,.75){\scs$\wtn{+}\wt$};
\end{tikzpicture} \\
\stackrel{\eqref{primesign}}{=} (-1)^{v_{\bullet  \grb}(\wt - \ee_{\grb}) + l_{\grb}(\wt - \ee_{\grb}) + l_{\grb}(\wt - \ee_{\grb} + \ee_{\bullet})}
\begin{tikzpicture}[anchorbase,scale=1]
\draw[thick,green,->] (0,0) to [out=90,in=270] (.5,1);
\draw[thick,<-] (.5,0) to [out=90,in=270] (0,1);
\node at (1,.75){\scs$\wtn{+}\wt$};
\end{tikzpicture} \, .
\end{multlined}
\end{equation}

\item \textbf{Dot slide:}

Now consider the dot slide relation \eqref{dotslide} when $i = j = \grb$.
We compute that
\[
\left(
\begin{tikzpicture}[anchorbase,scale=1]
\draw[thick,green,->] (0,0) to [out=90,in=270] (.5,1);
\draw[thick,green,->] (.5,0) to [out=90,in=270] 
	node[pos=.7,yshift=-.5pt,black]{$\bullet$} (0,1);
\node at (.875,.75){\scs$\wtn{+}\wt$};
\end{tikzpicture}
-
\begin{tikzpicture}[anchorbase,scale=1]
\draw[thick,green,->] (0,0) to [out=90,in=270] (.5,1);
\draw[thick,green,->] (.5,0) to [out=90,in=270] 
	node[pos=.3,yshift=-.5pt,black]{$\bullet$} (0,1);
\node at (.875,.75){\scs$\wtn{+}\wt$};
\end{tikzpicture}
\right)
\mapsto
(-1)^{v_{\grb \grb}(\wt) + 1}
\left(
\begin{tikzpicture}[anchorbase,scale=1]
\draw[thick,green,<-] (0,0) to [out=90,in=270] (.5,1);
\draw[thick,green,<-] (.5,0) to [out=90,in=270] 
	node[pos=.7,yshift=.5pt,black]{$\bullet$} (0,1);
\node at (.875,.75){\scs$\wtn{-}\wt$};
\end{tikzpicture}
-
\begin{tikzpicture}[anchorbase,scale=1]
\draw[thick,green,<-] (0,0) to [out=90,in=270] (.5,1);
\draw[thick,green,<-] (.5,0) to [out=90,in=270] 
	node[pos=.3,yshift=.5pt,black]{$\bullet$} (0,1);
\node at (.875,.75){\scs$\wtn{-}\wt$};
\end{tikzpicture}
\right)
\stackrel{\eqref{dotslide}}{=}
(-1)^{v_{\grb \grb}(\wt)}
\begin{tikzpicture}[anchorbase,scale=1]
\draw[thick,green,<-] (0,0) to [out=90,in=270] (0,1);
\draw[thick,green,<-] (.5,0) to [out=90,in=270] (.5,1);
\node at (.875,.75){\scs$\wtn{-}\wt$};
\end{tikzpicture}
\]
and
\[
\left(
\begin{tikzpicture}[anchorbase,scale=1]
\draw[thick,green,->] (0,0) to [out=90,in=270] 
	node[pos=.3,yshift=-.5pt,black]{$\bullet$}(.5,1);
\draw[thick,green,->] (.5,0) to [out=90,in=270] (0,1);
\node at (.875,.75){\scs$\wtn{+}\wt$};
\end{tikzpicture}
-
\begin{tikzpicture}[anchorbase,scale=1]
\draw[thick,green,->] (0,0) to [out=90,in=270] 
	node[pos=.7,yshift=-.5pt,black]{$\bullet$}(.5,1);
\draw[thick,green,->] (.5,0) to [out=90,in=270] (0,1);
\node at (.875,.75){\scs$\wtn{+}\wt$};
\end{tikzpicture}
\right)
\mapsto
(-1)^{v_{\grb \grb}(\wt) + 1}
\left(
\begin{tikzpicture}[anchorbase,scale=1]
\draw[thick,green,<-] (0,0) to [out=90,in=270] 
	node[pos=.3,yshift=.5pt,black]{$\bullet$}(.5,1);
\draw[thick,green,<-] (.5,0) to [out=90,in=270] (0,1);
\node at (.875,.75){\scs$\wtn{-}\wt$};
\end{tikzpicture}
-
\begin{tikzpicture}[anchorbase,scale=1]
\draw[thick,green,<-] (0,0) to [out=90,in=270] 
	node[pos=.7,yshift=.5pt,black]{$\bullet$}(.5,1);
\draw[thick,green,<-] (.5,0) to [out=90,in=270] (0,1);
\node at (.875,.75){\scs$\wtn{-}\wt$};
\end{tikzpicture}
\right)
\stackrel{\eqref{dotslide}}{=}
(-1)^{v_{\grb \grb}(\wt)}
\begin{tikzpicture}[anchorbase,scale=1]
\draw[thick,green,<-] (0,0) to [out=90,in=270] (0,1);
\draw[thick,green,<-] (.5,0) to [out=90,in=270] (.5,1);
\node at (.875,.75){\scs$\wtn{-}\wt$};
\end{tikzpicture}
\]
However, the (far) right-hand side of \eqref{dotslide} is sent by $\tau$ 
to the identity map of $\FF_{\grb} \FF_{\grb} \one_{n-\wt}$ with no sign. 
For this relation to be preserved by $\tau$, we therefore need $v_{\grb \grb}(\wt) \equiv 0$, 
which is \eqref{samecrossingsign}.

Finally, it is easy to confirm that the dot slide relation when $i \ne j$ induces no constraints,
essentially because $-1 \cdot 0 = 0$. \\

\item \textbf{Quadratic KLR:}

The $i=j$ case of the quadratic KLR relation \eqref{eq:quadKLR} is obvious.
In the case where $i \cdot j = 0$, 
the left-hand side picks up the sign $(-1)^{v_{i,j}(\wt) + v_{j,i}(\wt)}$ under $\tau$, 
and the right-hand side picks up no sign. 
Therefore, this relation is preserved by $\tau$ if and only if $v_{i,j}(\wt) \equiv v_{j,i}(\wt)$ when $i \cdot j = 0$, 
which gives one case of \eqref{diffcrossingsign}.

Lastly, suppose that $i \cdot j = -1$. 
Under $\tau$, the left-hand side picks up the sign $(-1)^{v_{i,j}(\wt) + v_{j,i}(\wt)}$ 
and the right-hand side picks up the sign $(-1)$ from the dot. 
This relation is preserved by $\tau$ if and only if $v_{i,j}(\wt) \equiv v_{j,i}(\wt) + 1$ when $i \cdot j = -1$, 
which gives the other case of \eqref{diffcrossingsign}.
In more detail, suppose that $j=i+1$ (the $j=i-1$ case is similar).
Following Convention \ref{conv:colors}, 
we compute
\[
\left(
\begin{tikzpicture}[anchorbase,yscale=1,scale=.75]
\draw[thick,green,->] (0,0) to [out=90,in=270] (.5,.75) to [out=90,in=270] (0,1.5);
\draw[thick,blue,->] (.5,0) to [out=90,in=270] (0,.75) to [out=90,in=270] (.5,1.5);
\node at (1,1.25){\scs$\wtn{+}\wt$};
\end{tikzpicture}
=
\begin{tikzpicture}[anchorbase]
\draw[thick,green,->] (0,0) to node[black]{$\bullet$} (0,1);
\draw[thick,blue,->] (.375,0) to (.375,1);
\node at (.75,.75){\scs$\wtn{+}\wt$};
\end{tikzpicture}
-
\begin{tikzpicture}[anchorbase]
\draw[thick,green,->] (0,0) to (0,1);
\draw[thick,blue,->] (.375,0) to node[black]{$\bullet$} (.375,1);
\node at (.75,.75){\scs$\wtn{+}\wt$};
\end{tikzpicture} 
\right)
\xmapsto{\tau}
\left(
(-1)^{v_{\grb \blb}(\wt) + v_{\blb \grb}(\wt)}
\begin{tikzpicture}[anchorbase,yscale=1,scale=.75]
\draw[thick,green,<-] (0,0) to [out=90,in=270] (.5,.75) to [out=90,in=270] (0,1.5);
\draw[thick,blue,<-] (.5,0) to [out=90,in=270] (0,.75) to [out=90,in=270] (.5,1.5);
\node at (1,1.25){\scs$\wtn{-}\wt$};
\end{tikzpicture}
=
-
\begin{tikzpicture}[anchorbase]
\draw[thick,green,<-] (0,0) to node[black]{$\bullet$} (0,1);
\draw[thick,blue,<-] (.375,0) to (.375,1);
\node at (.875,.75){\scs$\wtn{-}\wt$};
\end{tikzpicture}
+
\begin{tikzpicture}[anchorbase]
\draw[thick,green,<-] (0,0) to (0,1);
\draw[thick,blue,<-] (.375,0) to node[black]{$\bullet$} (.375,1);
\node at (.875,.75){\scs$\wtn{-}\wt$};
\end{tikzpicture} 
\right) \, .
\]
Since
\[
-
\begin{tikzpicture}[anchorbase]
\draw[thick,green,<-] (0,0) to node[black]{$\bullet$} (0,1);
\draw[thick,blue,<-] (.375,0) to (.375,1);
\node at (.875,.75){\scs$\wtn{-}\wt$};
\end{tikzpicture}
+
\begin{tikzpicture}[anchorbase]
\draw[thick,green,<-] (0,0) to (0,1);
\draw[thick,blue,<-] (.375,0) to node[black]{$\bullet$} (.375,1);
\node at (.875,.75){\scs$\wtn{-}\wt$};
\end{tikzpicture} 
=
-
\begin{tikzpicture}[anchorbase,yscale=1,scale=.75]
\draw[thick,green,<-] (0,0) to [out=90,in=270] (.5,.75) to [out=90,in=270] (0,1.5);
\draw[thick,blue,<-] (.5,0) to [out=90,in=270] (0,.75) to [out=90,in=270] (.5,1.5);
\node at (1,1.25){\scs$\wtn{-}\wt$};
\end{tikzpicture} \, ,
\]
this requires the $i\cdot j=-1$ case of \eqref{diffcrossingsign}. \\

\item \textbf{Cubic KLR:}

First, consider the $k=i$ and $i \cdot j = -1$ case of the cubic KLR relation \eqref{eq:cubicKLR}.
When $j=i-1$ (following Convention \ref{conv:colors}, as usual), 
we have
\begin{multline*}
\left(
\begin{tikzpicture}[anchorbase,yscale=.75]
\draw[thick,->,green] (0,0) to [out=90,in=270] (1,2);
\draw[thick,->,red] (.5,0) to [out=90,in=270] (0,1) 
	to [out=90,in=270] (.5,2);
\draw[thick,->,green] (1,0) to [out=90,in=270] (0,2);
\node at (1.375,1.5){\scs$\wtn{+}\wt$};
\end{tikzpicture}
-
\begin{tikzpicture}[anchorbase,yscale=.75]
\draw[thick,->,green] (0,0) to [out=90,in=270] (1,2);
\draw[thick,->,red] (.5,0) to [out=90,in=270] (1,1) 
	to [out=90,in=270] (.5,2);
\draw[thick,->,green] (1,0) to [out=90,in=270] (0,2);
\node at (1.375,1.5){\scs$\wtn{+}\wt$};
\end{tikzpicture} =
-
\begin{tikzpicture}[anchorbase,yscale=.75]
\draw[thick,->,green] (0,0) to [out=90,in=270] (0,2);
\draw[thick,->,red] (.5,0) to (.5,2);
\draw[thick,->,green] (1,0) to [out=90,in=270] (1,2);
\node at (1.375,1.5){\scs$\wtn{+}\wt$};
\end{tikzpicture}
\right) \mapsto \\
\left(
(-1)^{\substack{v_{\grb \grb}(\wt) + v_{\redb \grb}(\wt+\ee_{\grb}) \\ 
	+ v_{\grb \redb}(\wt+\ee_{\grb})}}
\begin{tikzpicture}[anchorbase,yscale=.75]
\draw[thick,<-,green] (0,0) to [out=90,in=270] (1,2);
\draw[thick,<-,red] (.5,0) to [out=90,in=270] (0,1) 
	to [out=90,in=270] (.5,2);
\draw[thick,<-,green] (1,0) to [out=90,in=270] (0,2);
\node at (1.375,1.5){\scs$\wtn{-}\wt$};
\end{tikzpicture}
-
(-1)^{\substack{v_{\grb \grb}(\wt+\ee_{\redb}) + v_{\redb \grb}(\wt) \\ 
	+ v_{\grb \redb}(\wt)}}
\begin{tikzpicture}[anchorbase,yscale=.75]
\draw[thick,<-,green] (0,0) to [out=90,in=270] (1,2);
\draw[thick,<-,red] (.5,0) to [out=90,in=270] (1,1) 
	to [out=90,in=270] (.5,2);
\draw[thick,<-,green] (1,0) to [out=90,in=270] (0,2);
\node at (1.375,1.5){\scs$\wtn{-}\wt$};
\end{tikzpicture} =
-
\begin{tikzpicture}[anchorbase,yscale=.75]
\draw[thick,<-,green] (0,0) to [out=90,in=270] (0,2);
\draw[thick,<-,red] (.5,0) to (.5,2);
\draw[thick,<-,green] (1,0) to [out=90,in=270] (1,2);
\node at (1.375,1.5){\scs$\wtn{-}\wt$};
\end{tikzpicture}
\right)
\end{multline*}
which, by \eqref{samecrossingsign} and \eqref{diffcrossingsign}, gives
\[
-
\begin{tikzpicture}[anchorbase,yscale=.75]
\draw[thick,<-,green] (0,0) to [out=90,in=270] (1,2);
\draw[thick,<-,red] (.5,0) to [out=90,in=270] (0,1) 
	to [out=90,in=270] (.5,2);
\draw[thick,<-,green] (1,0) to [out=90,in=270] (0,2);
\node at (1.375,1.5){\scs $n{-}\wt$};
\end{tikzpicture}
+
\begin{tikzpicture}[anchorbase,yscale=.75]
\draw[thick,<-,green] (0,0) to [out=90,in=270] (1,2);
\draw[thick,<-,red] (.5,0) to [out=90,in=270] (1,1) 
	to [out=90,in=270] (.5,2);
\draw[thick,<-,green] (1,0) to [out=90,in=270] (0,2);
\node at (1.375,1.5){\scs$\wtn{-}\wt$};
\end{tikzpicture} =
-
\begin{tikzpicture}[anchorbase,yscale=.75]
\draw[thick,<-,green] (0,0) to [out=90,in=270] (0,2);
\draw[thick,<-,red] (.5,0) to (.5,2);
\draw[thick,<-,green] (1,0) to [out=90,in=270] (1,2);
\node at (1.375,1.5){\scs$\wtn{-}\wt$};
\end{tikzpicture} \, .
\]
This is exactly \eqref{eq:cubicKLR}, so no new constraints are imposed.
The $j=i+1$ case is nearly identical.

For all remaining cases, we need only verify that
\[
(-1)^{v_{i,k}(\wt) + v_{i,j}(\wt+\ee_{k}) + v_{j,k}(\wt+\ee_i)}
=
(-1)^{v_{i,k}(\wt+\ee_j) + v_{i,j}(\wt) + v_{j,k}(\wt)} \, ,
\]
i.e.~that
\[
v_{i,j}(\wt) + v_{j,k}(\wt) + v_{i,k}(\wt)
\equiv
v_{i,j}(\wt+\ee_{k}) + v_{j,k}(\wt+\ee_{i}) + v_{i,k}(\wt+\ee_{j}).
\]
This is exactly \eqref{R3sign}.
Note that if $i,j,k$ are not all distinct, 
this condition is already a consequence of \eqref{samecrossingsign} and \eqref{diffcrossingsign}
(e.g.~this relation was already implicit in our check of the $j=i-1$, $k=i$ case above).
\end{itemize}

Having established the necessity of \eqref{tauconditions}, 
we now complete the check that they are sufficient.

\begin{itemize}[leftmargin=*]
	
\item \textbf{Dots to new bubbles:}
If we rotate \eqref{dotstonewbubbles} by 180 degrees, we get
\begin{equation} \label{dotstonewbubblesrot}
\begin{tikzpicture}[anchorbase,yscale=1]
\draw[thick,<-] (0,0) node[below=-1pt]{\scs $i$} to [out=90,in=270] node{$\bullet$} (0,1.5);
\end{tikzpicture}
=
-
\begin{tikzpicture}[anchorbase,yscale=1]
\draw[thick,<-] (0,0) node[below=-1pt]{\scs $i$} to [out=90,in=270] (0,1.5);
\node at (-.75,.75){\NB{e_1(\X_i)}};
\end{tikzpicture}
+
\begin{tikzpicture}[anchorbase,yscale=1]
\draw[thick,<-] (0,0) node[below=-1pt]{\scs $i$} to [out=90,in=270] (0,1.5);
\node at (.75,.75){\NB{e_1(\X_i)}};
\end{tikzpicture}
=
-
\begin{tikzpicture}[anchorbase,yscale=1]
\draw[thick,<-] (0,0) node[below=-1pt]{\scs $i$} to [out=90,in=270] (0,1.5);
\node at (.875,.75){\NB{e_1(\X_{i+1})}};
\end{tikzpicture}
+
\begin{tikzpicture}[anchorbase,yscale=1]
\draw[thick,<-] (0,0) node[below=-1pt]{\scs $i$} to [out=90,in=270] (0,1.5);
\node at (-.875,.75){\NB{e_1(\X_{i+1})}};
\end{tikzpicture}.
\end{equation}
Meanwhile, if we apply $\tau$ to \eqref{dotstonewbubbles}, 
we get the same relation as \eqref{dotstonewbubblesrot} 
(albeit in a different ambient weight, and with both sides mutliplied by $-1$). 
Here, we use that $\tau(e_1) = h_1 = e_1$.

\item \textbf{Extended $\sln[2]$ relations:}

Consider the two relations \eqref{extendedreln1} and \eqref{extendedreln2}, 
with ambient weight $\wtn + \wt$ on the right. 
Applying $\tau$ appears to swap the two relations, 
though it changes the ambient weight to $\wtn - \wt$ and may produce signs on each diagram. 
The change in ambient weight ensures that the two summations range over the same set of $\{p, q, r\}$. 
No signs are produced on the left-hand side, so we need only check that no overall $-1$ sign 
is introduced in each diagram 
appearing on the right-hand side.

We confirm this for \eqref{extendedreln2}, since the case of \eqref{extendedreln1} is similar.
In the first diagram on the right-hand side of \eqref{extendedreln2}, 
using \eqref{rightwardcross} and \eqref{leftwardcross}, 
the two sideways crossings produce a total sign with exponent
$2 v_{i,i}(\wt - \ee_i) + r_i(\wt) + r_i(\wt + \ee_i) + l_i(\wt) + l_i(\wt - \ee_i)$.
Applying \eqref{bubblesign}, the result is equivalent modulo 2 to
$\langle \ee_i^\vee, \wt \rangle + \langle \ee_i^\vee, \wt + \ee_i \rangle$.
Since $\langle \ee_i^\vee, \ee_i \rangle = 2$, this exponent is even, as desired.

In each term of the summation on the right-hand side, 
the overall sign produced by $\tau$ is 
\[
(-1)^{p+q} (-1)^{r_i(\wt) + r'_i(\wt) + l_i(\wt)+l'_i(\wt)} (-1)^{a_{i+1} - a_i - 1 + r} \, . 
\]
Here we've e.g.~used \eqref{eq:realbubbles} with $c = a_{i+1} - a_i - 1 + r$. 
Since $p+q+r = a_{i+1} - a_i - 1$, 
the overall exponent is equivalent $\ourmod 2$ to
\begin{align*}
r_i(\wt) + r'_i(\wt) + l_i(\wt)+l'_i(\wt) 
	&\stackrel{\eqref{primesign}}{\equiv} r_i(\wt) + r_i(\wt+\ee_i) + l_i(\wt)+l_i(\wt-\ee_i) \\
	&\stackrel{\eqref{bubblesign}}{\equiv} \langle \ee_i^\vee , \wt \rangle + \langle \ee_i^\vee , \wt + \ee_i \rangle
	\equiv 2 \langle \ee_i^\vee , \wt \rangle + \langle \ee_i^\vee , \ee_i \rangle \equiv 0
\end{align*}
as desired. \\

\item \textbf{Mixed $\EE,\FF$ relations:}

Consider the first relation in \eqref{eq:mixedef}, with ambient weight $\wtn + \wt$ on the right. 
The identity map is sent by $\tau$ to an identity map, with no sign. 
Once again, using \eqref{rightwardcross} and \eqref{leftwardcross}, 
the overall sign on the right-hand side has exponent
\begin{equation} 
v_{j,i}(\wt - \ee_j) + r_{j}(\wt) + r_{j}(\wt + \ee_{i}) 
	+ v_{i,j}(\wt - \ee_j) + l_{j}(\wt - \ee_{j}) + l_{j}(\wt - \ee_{j} + \ee_{i})  \, . 
\end{equation}
Applying \eqref{bubblesign} and \eqref{diffcrossingsign}, 
this is equivalent $\ourmod 2$ to
\[
\langle \ee_j^\vee , \ee_i \rangle +\langle \ee_j^\vee , \wt \rangle + \langle \ee_j^\vee , \wt + \ee_i \rangle \equiv 0
\]
as desired.
\end{itemize}
This concludes the proof of Theorem \ref{thm:symmetry}.
\end{proof}

\begin{rem} \label{rem:togetsymmetry}
Let us discuss what it takes to define functions $r_i, l_i, r'_i, l'_i, v_{i,j} \colon \Z^m \to \Z/2$ 
for $i,j \in \{1,\ldots,m-1\}$ that satisfy \eqref{tauconditions}. 
Since \eqref{tauconditions} has no conditions relating $v_{i,j}$ to the remaining functions, 
we can discuss the choice of $v_{i,j}$ separately from the others. 
However, we emphasize that in subsequent sections, 
we impose further conditions on $\tau$ such that $v_{i,j}$ will cease to be decoupled from the remaining functions.

Observe that one can use \eqref{primesign} to define $r'_i$ and $l'_i$ in terms of $r_i$ and $l_i$ 
and can then use \eqref{bubblesign} to define $l_i$ in terms of $r_i$.
We thus can choose each $r_i$ freely, and extrapolate the values of $r'_i, l_i, l'_i$.
Similarly, \eqref{samecrossingsign} and \eqref{diffcrossingsign} can be used to define $v_{i,j}$ for $i \ge j$ 
in terms of $v_{i,j}$ for $i < j$.
Next, \eqref{samecrossingsign} implies that \eqref{diffcrossingsign} holds when $i = j$, 
and together \eqref{samecrossingsign} and \eqref{diffcrossingsign} imply that \eqref{R3sign} holds 
whenever $\{i,j,k\}$ are not distinct.
Lastly, if \eqref{R3sign} holds for $\{i,j,k\}$ then it also holds for any permutation of $\{i,j,k\}$.

Thus, it suffices to find $v_{i,j}$ for $i < j$ that satisfy \eqref{R3sign}.
This can be accomplished in a number of ways, e.g.~we can let each such $v_{i,j} \colon \Z^m \to \Z/2$ be constant.
\end{rem}

\subsection{Is it an involution?} \label{ss:inv?}

Remark \ref{rem:togetsymmetry} shows it is easy to find functions 
that determine a $2$-automorphism $\tau$ as in Theorem \ref{thm:symmetry}.
However, we desire more, namely that $\tau$ be an involution.
Our next result shows that any choice gets us close.

\begin{cor} 
	\label{cor:order4}
When the conditions \eqref{tauconditions} are satisfied 
(so that $\tau$ is well-defined), 
we have $\tau^4 = \id_{\UU_q(\glm)}$. \end{cor}
	
\begin{proof} 
It is clear that $\tau^2$ acts as the identity on all objects and $1$-morphisms. 
Moreover, $\tau^2$ sends each diagram to itself, up to multiplication by some sign. 
Therefore, $\tau^4$ is the identity on $2$-morphisms. \end{proof}

Hence, any such $\tau$ has order $2$ or $4$ (since $\tau$ is clearly not the identity). 
Additional requirements must be imposed on the functions determining $\tau$ 
to guarantee that it has order $2$.
	
\begin{lem} 
	\label{lem:inv?}
Assume that \eqref{tauconditions} is satisfied. 
Then, $\tau$ is an involution on $\UU_q(\glm)$ if and only if the following hold in $\Z/2$ for all weights $\wt$:
\begin{subequations} \label{tauinvolution} 
\begin{equation} \label{lfromr} r_{\grb}(\wt) + l_{\grb}(-\wt) \equiv 0, \end{equation}
\begin{equation} \label{bubblesignalt} r_i(\wt) + r_i(\ee_i - \wt) \equiv \langle \ee_i^\vee, \wt \rangle. \end{equation}
\begin{equation} \label{vfortauinvolution} v_{\grb \bullet}(\wt-\ee_{\grb}-\ee_{\bullet}) + v_{\grb \bullet}(-\wt) + r_{\grb}(\wt) + r_{\bullet}(\wt) + r_{\bullet}(\wt-\ee_{\grb}) + r_{\grb}(\wt - \ee_{\bullet}) \equiv 0. \end{equation}
\end{subequations}
\end{lem}

\begin{proof} By a straightforward examination, 
$\tau^2$ sends each of the generating $2$-morphisms of $\UU_q(\glm)$ to a signed multiple of itself,
e.g.
\[
\begin{tikzpicture}[anchorbase,scale=1]
\draw[thick,green,->] (-.25,0) to [out=90,in=180] (0,.5) 
	to [out=0,in=90] (.25,0);
\node at (.625,.5){\scs$\wtn{+}\wt$};
\end{tikzpicture}
\xmapsto{\tau}
(-1)^{r_{\grb}(\wt)}
\begin{tikzpicture}[anchorbase,scale=1]
\draw[thick,green,<-] (-.25,0) to [out=90,in=180] (0,.5) 
	to [out=0,in=90] (.25,0);
\node at (.625,.5){\scs$\wtn{-}\wt$}; 
\end{tikzpicture}
\xmapsto{\tau}
(-1)^{r_{\grb}(\wt)+l_{\grb}(-\wt)}
\begin{tikzpicture}[anchorbase,scale=1]
\draw[thick,green,->] (-.25,0) to [out=90,in=180] (0,.5) 
	to [out=0,in=90] (.25,0);
\node at (.625,.5){\scs$\wtn{+}\wt$};
\end{tikzpicture} \, .
\]
Similar computations for the other generating $2$-morphisms give the requirements:
\begin{subequations} \label{tauinvolutionnaive} 
\begin{equation} \label{naiverl} r_{\grb}(\wt) + l_{\grb}(-\wt) \equiv 0, \end{equation}
\begin{equation} \label{naiveprime} r'_{\grb}(\wt) + l'_{\grb}(-\wt) \equiv 0, \end{equation}
\begin{equation} \label{naivev} v'_{\grb \bullet}(\wt) + v_{\grb \bullet}(-\wt) \equiv 0 \, . \end{equation}
\end{subequations}
(We also note that the automorphism of symmetric functions sending 
$e_r \mapsto h_r$ is indeed a well-known involution.)

Thus, we need only match \eqref{tauinvolutionnaive} to \eqref{tauinvolution}. 
First, \eqref{naiverl} is just a reprinting of \eqref{lfromr}. 
Using \eqref{primesign}, \eqref{naiveprime} is equivalent to
\[r_{\grb}(\wt + \ee_{\grb}) + l_{\grb}(-\wt - \ee_{\grb}) \equiv 0\, , \]
which is equivalent to \eqref{naiverl}. 
If we view \eqref{lfromr} as defining the $l$ function in terms of the $r$ function, 
then rewriting \eqref{bubblesign} so that it represents a condition on the $r$ function gives \eqref{bubblesignalt}. 
Finally, \eqref{naivev} is equivalent to \eqref{vfortauinvolution}, using \eqref{primev}.
\end{proof}

We now ask: does such a $\tau$ of order $2$ exist? 
That is, is there any choice of functions for which both
\eqref{tauconditions} holds (so $\tau$ is well-defined) \textbf{and} 
\eqref{tauinvolution} holds (so $\tau$ is an involution)? 
We now show that ``na\"{i}ve'' solutions to these formulae do not exist.
Here, by na\"{i}ve, we refer to two additional constraints: 
that $v_{i,j}(\wt)$ is a constant function and that $r_i(\wt)$ is a function ``defined locally'', 
in a sense explained below. 
The choice of $\tau$ that we use in applications, 
described in Theorem \ref{thm:involution}, is indeed na\"{i}ve in this sense, so it will not be an involution
(though as mentioned above, it fortunately restricts to an involution on a $2$-subcategory 
that is sufficient for our applications).

\begin{rem} Most of the relations in \eqref{tauconditions} and
\eqref{tauinvolution} can be used to define $r'$, $l$, $l'$, or $v'$ in terms of $r$ and $v$. 
As noted in Remark \ref{rem:togetsymmetry}, 
we need only define $v_{i,j}$ for all $i < j$ and check \eqref{R3sign} 
for $i < j < k$ in order to define $v_{i,j}$ for all $i, j$ such that 
\eqref{samecrossingsign}, \eqref{diffcrossingsign}, and \eqref{R3sign} all hold. 
Thus, it suffices to define $r_i$ and $v_{i,j}$ for all $i<j$. 
These functions must satisfy \eqref{bubblesignalt} and 
\eqref{vfortauinvolution} and \eqref{R3sign} for $i < j < k$. \end{rem}

It is impossible for \eqref{bubblesignalt} to hold when $r_i(\wt)$ is a constant function, 
so one must abandon the possibility that the sign on the cups and caps is independent of the ambient weight. 
However, it might be desirable for $v_{i,j}(\wt)$ to be a constant function, 
allowing the sign on crossings to be independent of the ambient weight and thus only dependent on the coloring. 
In \eqref{vfortauinvolution} and \eqref{R3sign}, each $v_{i,j}$ appears an even number of times. 
If each $v_{i,j}$ is a constant function, then \eqref{R3sign} holds, and \eqref{vfortauinvolution} is replaced by
\[ r_i(\wt) + r_j(\wt) + r_i(\wt - \ee_j) + r_j(\wt - \ee_i) \equiv 0. \]
Applying \eqref{bubblesignalt}, we derive the equivalent formula
\begin{equation} \label{rrsign} 
r_i(\wt) + r_j(\wt) + r_i(\ee_i + \ee_j - \wt) + r_j(\ee_i + \ee_j - \wt) 
	= \langle \ee_i^\vee + \ee_j^\vee, \wt \rangle,\end{equation}
which bears a comforting similarity to \eqref{bubblesignalt}, but for the function $r_i + r_j$.

A second property one might desire is that the formula for $\tau$ acting on $\UU_q(\glm)$ 
is compatible with a formula for $\tau$ on $\UU_q(\mathfrak{gl}_2)$, 
with regards to the standard inclusions of $\mathfrak{gl}_2$ into $\glm$. 
For $\mathfrak{gl}_2$, there is a single function $r \colon \Z^2 \to \Z/2$, and it satisfies
\begin{equation} \label{globalrsign} 
	r(a,b) + r(1 - a, -1-b) \equiv a + b \end{equation} 
by \eqref{bubblesignalt}. 
We thus say that the family of functions $\{r_i\}$ is \emph{defined locally} 
if there exists a function $r \colon \Z^2 \to \Z/2$ satisfying \eqref{globalrsign}, for which
\begin{equation} r_i(a_1, \ldots, a_i, a_{i+1}, \ldots, a_m) = r(a_i, a_{i+1}) \, . \end{equation}

\begin{thm} \label{thm:nonaivesolution} 
If $m \ge 3$, then there is no solution to \eqref{tauconditions} and \eqref{tauinvolution} 
for which $v_{i,j}$ is a constant function for all $i, j$, and $\{r_i\}$ is defined locally. \end{thm}

\begin{proof} 
Following the discussion above, regardless of the value of $v_{i,j}$ for $i < j$, 
we need only define a function $r$ satisfying \eqref{globalrsign} for which \eqref{rrsign} holds. 
Without loss of generality, 
we can assume that $r(0,0) = 0$, since adding a constant function to $r$ will not affect whether $r$ is a solution.

Assume that $m=3$, $i=1$, $j=2$ and let $\wt = (a,b,c)$;
more generally, we can choose $j = i+1$ and study three consecutive coordinates of a 
general $\glm$ weight for $m \ge 3$. 
We then see that \eqref{rrsign} is equivalent to
\[ r(a,b) + r(b,c) + r(1-a,-b) + r(-b, -1-c) \equiv a + c \, , \]
while \eqref{globalrsign} gives
\[ r(1-a,-b) + r(a,b-1) \equiv a+b+1 \quad \text{and} \quad r(-b,-1-c) + r(b+1,c) \equiv b+c+1 \, . \] 
Thus
\[ r(a,b) + r(b,c) + r(1-a,-b) + r(-b, -1-c) \equiv r(a,b) + r(a,b-1) + r(b,c) + r(b+1,c) + a + c \]
so we conclude that
\begin{equation}\label{eq:theproblem}
r(a,b) + r(a,b-1) \equiv r(b,c) + r(b+1,c) \end{equation}
in $\Z/2$.

Let $\epsilon_{a,b} := r(a,b) + r(a,b-1)$, then \eqref{eq:theproblem} shows that this
element of $\Z/2$ is equal to $r(b,c) + r(b+1,c)$ for any $c$, 
implying that $\epsilon_{a,b}$ is independent of $a$. 
We thus write it as $\epsilon_b$, i.e.~
\[
r(a,b) + r(a,b-1) \equiv \epsilon_b  \quad \text{and} \quad r(b,c) + r(b+1,c) \equiv \epsilon_b \]
for all $a,c \in \Z$.
	
Recalling our assumption that $r(0,0) = 0$ we compute
\[ r(0,-1) = \epsilon_0 \quad \text{and} \quad r(1,0) = \epsilon_0 \, , \]
but this contradicts \eqref{globalrsign}, which requires that
$r(1,0) + r(0,-1) \equiv 1$.
\end{proof}

It is still an open question, and a surprisingly thorny one, 
whether there is any (non-na\"{i}ve) choice of functions for which $\tau$ is well-defined and an involution.

\begin{rem}
It is interesting to replace the function $v_{ij}(\wt)$ with
\[ w_{ij}(\wt) := v_{ij}(-\wt) + r_i(\wt) + r_j(\wt) \, . \]
After doing so, 
\eqref{vfortauinvolution} can be rephrased (using \eqref{bubblesignalt}) as
\begin{equation} w_{ij}(\wt) + w_{ij}(\ee_i + \ee_j - \wt) \equiv \langle \ee_i^\vee + \ee_j^\vee, \wt \rangle \, , \end{equation}
which bears a comforting similarity to \eqref{bubblesignalt}. 
One can now construct functions $q_r$ and $q_w$ (depending on indices $i < j < k$) with the symmetry
\begin{equation} q(\wt) + q(\ee_i + \ee_j + \ee_k - \wt) \equiv 0 \end{equation}
by the formulas
\[ q_w(\wt) := w_{ij}(\wt) + w_{jk}(\wt) + w_{ik}(\wt) + w_{ij}(\ee_i + \ee_j + \ee_k - \wt) 
	+ w_{jk}(\ee_i + \ee_j + \ee_k - \wt) + w_{ik}(\ee_i + \ee_j + \ee_k - \wt) \, , \]
and
\[ q_r(\wt) := r_i(\wt - \ee_k) + r_i(\wt - \ee_j) + r_j(\wt - \ee_i) 
	+ r_j(\wt - \ee_k) + r_k(\wt - \ee_i) + r_k(\wt - \ee_j) \, . \]
The final equation \eqref{R3sign} can then be reformulated as
\begin{equation} q_w \equiv q_r \, . \end{equation}
\end{rem}

\subsection{An involution on part of $\UU_q(\glm)$}\label{ss:invonpart}

Although we are unable 
to find $r_{i}, v_{i,j} \colon \Z^m \to \Z/2$ 
such that the $2$-automorphism $\tau$ from Theorem \ref{thm:symmetry} is an involution, 
we now show that it is possible to find such functions so that $\tau$
restricts to an involution on a certain full $2$-subcategory of $\UU_q(\glm)$. 
Moreover, the functions we construct are such that
$\{r_i\}$ is defined locally and $v_{i,j}$ is a constant function for all $i,j$;
c.f.~Theorem \ref{thm:nonaivesolution}.

\begin{defn}\label{D:FE-generated-subcat}
Fix a $\glm$ weight $\wt \in \Z^m$. 
Let $\XX_q(\glm) \one_\wt$ denote the full $2$-subcategory of $\UU_q(\glm)$ generated 
by the $1$-endomorphisms $\EE_i\FF_i \one_\wt$ and $\FF_i\EE_i \one_\wt$, and their grading shifts. 
Let $\XX_q(\glm)$ be the full $2$-subcategory containing all $\XX_q(\glm) \one_\wt$ as $\wt$ varies.
\end{defn}

Since $\EE_i \FF_i \one_\wt$ and $\FF_i\EE_i \one_\wt$ are $1$-endomorphisms 
(mapping from $\wt$ to $\wt$), 
$\XX_q(\glm) \one_\wt = \one_\wt \XX_q(\glm) \one_\wt$ is a monoidal category. 
There is no interaction between $\XX_q(\glm) \one_{\wt}$ and $\XX_q(\glm) \one_{\wt'}$ for $\wt \ne \wt'$.

\begin{thm}
	\label{thm:wheninvolution}
If $\tau$ is well-defined, 
then it restricts to an involution on $\XX_q(\glm)$ if and only if $l_i(\wt) \equiv r_i(-\wt)$ for all $\wt$.
\end{thm}

\begin{proof}
If $\tau$ is well-defined (i.e.~\eqref{tauconditions} is satisfied), then 
$\tau^2$ is the identity on new bubble generators and, since $v_{i,i} = 0$, 
it is also the identity on same-colored crossings. 
As observed in the proof of Lemma \ref{lem:inv?}, 
the condition $l_i(\wt) \equiv r_i(-\wt)$ is equivalent to the condition that $\tau^2$ acts as the identity on 
cap/cup generators in weight $\wtn+\wt$. 
This, in particular, establishes the ``only if'' assertion of the statement, 
since $\XX_q(\glm)$ contains all such $2$-morphisms. 

We thus assume that the condition $l_i(\wt) \equiv r_i(-\wt)$ holds for all $\wt \in \Z^m$.
It remains to consider the action of $\tau^2$ on crossings of differently colored strands.
Here, $\tau^2$ introduces rescalings by hard-to-control signs; 
the key observation is that we can track these contributions for $2$-morphisms lying in $\XX_q(\glm)$.

To begin,
consider the full subcategory $\FECatprime(\glm)$ of $\XX_q(\glm)$ 
generated by objects of the form $\FF_i \EE_i \one_{\wt}$. 
We claim that if $\tau^2$ is the identity on $\FECatprime(\glm)$, then it is the identity on all of $\XX_q(\glm)$. 
This is a consequence of \eqref{extendedreln1}, 
which describes the identity of $\EE_i \FF_i \one_{\wt}$ as a sum of diagrams which
factor through $\FF_i \EE_i \one_{\wt}$ and $\one_{\wt}$. 
More precisely, using \eqref{extendedreln2}, any morphism in $\XX_q(\glm)$ factors as a 
composition $f \circ g \circ h$ where $g$ is a morphism in $\FECatprime(\glm)$ and $f$ and $h$ 
are built from cups, caps, dots, and same-colored crossings. 
Since $f$ and $h$ are fixed by $\tau^2$, we deduce the claim.

For the rest of the proof, we focus on $\FECatprime(\glm) \one_{\wt}$ for some fixed weight $\wt$. 
Let $D$ be a diagram which represents a morphism in $\FECatprime(\glm) \one_{\wt}$, 
and consider the diagram $D'$ obtained from $D$ by replacing all same-colored crossings with identity maps:
\[
\begin{tikzpicture}[anchorbase,scale=1]
\draw[thick,->] (0,0) to [out=90,in=270] (.5,1);
\draw[thick,->] (.5,0) to [out=90,in=270] (0,1);
\end{tikzpicture}
\leadsto
\begin{tikzpicture}[anchorbase,scale=1]
\draw[thick,->] (0,0) to (0,1);
\draw[thick,->] (.5,0) to (.5,1);
\end{tikzpicture}
\]
The morphisms $D$ and $D'$ are completely unrelated; 
nonetheless, if $\tau^2$ fixes $D'$ then it fixes $D$. 
Hence, it suffices to show that $\tau^2$ acts as the identity on diagrams
with no same-colored crossings. 
Similarly, we can assume that our diagrams have no dots or new bubbles
($\tau^2$ will act as the identity on any diagram if and only if it acts as the identity on the diagram 
with all dots and new bubbles removed).
Since \cite[Proposition 3.11]{KL3} shows that $\Hom$-spaces in $\UU_q(\glm)$ 
(and hence in $\FECatprime(\glm) \one_{\wt}$)
are spanned by diagrams where the only closed components are bubbles, 
\eqref{newbub} implies that we may further assume that $D$ has no closed components.
 
After this simplification, 
the diagrams which remain can be viewed as a union of 
transversely intersecting colored $1$-manifolds, as in the diagram $\mathcal{M}$ here:
\begin{equation} \label{bigexample}
\mathcal{M} =
\begin{tikzpicture}[anchorbase,scale=.5]
\draw[thick,red,<-] (4,0) to [out=90,in=270] (0,4);
\draw[thick,red,<-] (1,4) to [out=270,in=270] (6,4);
\draw[thick,red,->] (5,0) to [out=90,in=270] (7,4);
\draw[thick,green,->] (2,4) to [out=270,in=90] (0,0);
\draw[thick,green,->] (1,0) to [out=90,in=90] (6,0);
\draw[thick,green,<-] (3,4) to [out=270,in=90] (7,0);
\draw[thick,blue,->] (4,4) to [out=270,in=90] (2,1) to (2,0);
\draw[thick,blue,->] (3,0) to [out=90,in=180] (5.5,.5) to [out=0,in=90] (8,0);
\draw[thick,blue,<-] (5,4) to [out=270,in=90] (9,0);
\end{tikzpicture} \, .
\end{equation}
We call such diagrams \emph{(colored oriented) matching diagram}. 
For such a diagram $\mathcal{M}$, we let $T(\mathcal{M}) \in \Z/2$ denote 
the exponent of the sign obtained when acting by $\tau^2$.

While not entirely necessary for the proof, the following visualization trick helps to clarify the situation. 
Take a matching diagram and label (i.e.~shade) the regions with subsets of the Dynkin vertices 
$\{1,2, \ldots, m-1\}$, akin to a Venn diagram. 
Shade the rightmost region with the empty set (i.e.~white), and, for each color $i$, 
the condition that a region contains $i$ alternates across each $i$-colored strand. 
An example of a shaded matching diagram with two colors is:
\begin{equation} \label{eq:twocolorexample}
\mathcal{M} =
\raisebox{-0.5\height}{\includegraphics{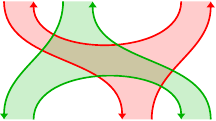}}
\, .
\end{equation}
The central region of this picture is shaded with both red and green. 
(We encourage the reader to shade \eqref{bigexample}.)

By virtue of the fact that our boundary is an object in $\FECatprime(\glm)$, 
the shading will always satisfy the property that, 
when standing on an $i$-colored strand and facing along the orientation, 
the $i$-shaded region is on your left. 
Moreover, the shading exactly records the weight of that region: 
it is $\wt + \sum_i \alpha_i$, 
where the sum is over those $i$ present in the shading.

Now we continue our simplification process, ignoring the shading for a moment. 
Remember that there are no same-colored crossings, 
so all crossings are transverse intersections of different-colored 1-manifolds.

\noindent{\textbf{Claim}:} $T(\mathcal{M})$ is invariant under the type II and III graph Reidemeister moves:
\begin{equation}
	\label{eq:gRM}
\mathrm{gRII} :
\begin{tikzpicture} [scale=.35,anchorbase]
	\draw[thick] (1,0) to [out=90,in=270] (0,1.5) to [out=90,in=270] (1,3);
	\draw[thick] (0,0) to [out=90,in=270] (1,1.5) to [out=90,in=270] (0,3);
\end{tikzpicture}
\sim
\begin{tikzpicture}[scale=.35, anchorbase]
	\draw[very thick] (1,0) to (1,3);
	\draw[very thick] (0,0) to (0,3);
\end{tikzpicture}
\quad , \quad
\mathrm{gRIII} :
\begin{tikzpicture}[scale=.2, tinynodes, anchorbase]
	\draw[thick] (-2,-3) to [out=90,in=270] (2,3);
	\draw[thick] (0,-3) to [out=90,in=270] (-2,0) to [out=90,in=270] (0,3);
	\draw[thick] (2,-3) to [out=90,in=270] (-2,3);	
\end{tikzpicture}
\sim
\begin{tikzpicture}[scale=.2, tinynodes, anchorbase]
	\draw[thick] (-2,-3) to [out=90,in=270] (2,3);
	\draw[thick] (0,-3) to [out=90,in=270] (2,0) to [out=90,in=270] (0,3);
	\draw[thick] (2,-3) to [out=90,in=270] (-2,3);	
\end{tikzpicture} \, .
\end{equation}

For gRII, there are four cases for the left-hand side:
\[
\begin{tikzpicture} [scale=.5,anchorbase]
	\draw[thick,->] (1,0) to [out=90,in=270] (0,1.5) to [out=90,in=270] (1,3);
	\draw[thick,green,->] (0,0) to [out=90,in=270] (1,1.5) to [out=90,in=270] (0,3);
\end{tikzpicture}
\quad , \quad
\begin{tikzpicture} [scale=.5,anchorbase, rotate=180]
	\draw[thick,green,->] (1,0) to [out=90,in=270] (0,1.5) to [out=90,in=270] (1,3);
	\draw[thick,->] (0,0) to [out=90,in=270] (1,1.5) to [out=90,in=270] (0,3);
\end{tikzpicture}
\quad , \quad
\begin{tikzpicture} [scale=.5,anchorbase]
	\draw[thick,<-] (1,0) to [out=90,in=270] (0,1.5) to [out=90,in=270] (1,3);
	\draw[thick,green,->] (0,0) to [out=90,in=270] (1,1.5) to [out=90,in=270] (0,3);
\end{tikzpicture}
\quad , \quad
\begin{tikzpicture} [scale=.5,anchorbase]
	\draw[thick,->] (1,0) to [out=90,in=270] (0,1.5) to [out=90,in=270] (1,3);
	\draw[thick,green,<-] (0,0) to [out=90,in=270] (1,1.5) to [out=90,in=270] (0,3);
\end{tikzpicture}
\]
(here, \textbf{black} denotes a color distinct from {\color{green} green}). 
However, we have already checked that $\tau$ is a $2$-functor, and hence preserves the relations of $\UU_q(\glm)$. 
Each diagram above is one side of a relation such as \eqref{eq:mixedef} or \eqref{eq:quadKLR}, 
and the other side of the relation is fixed by $\tau^2$. 
Thus the contribution of both sides of gRII to $T(\mathcal{M})$ must be $0 \ \ourmod 2$.

For gRIII, it suffices to check the braid-like orientation
\[
\begin{tikzpicture}[scale=.2, tinynodes, anchorbase]
	\draw[very thick,->] (-2,-3) to [out=90,in=270] (2,3);
	\draw[very thick,->] (0,-3) to [out=90,in=270] (-2,0) to [out=90,in=270] (0,3);
	\draw[very thick,->] (2,-3) to [out=90,in=270] (-2,3);	
\end{tikzpicture}
\sim
\begin{tikzpicture}[scale=.2, tinynodes, anchorbase]
	\draw[very thick,->] (-2,-3) to [out=90,in=270] (2,3);
	\draw[very thick,->] (0,-3) to [out=90,in=270] (2,0) to [out=90,in=270] (0,3);
	\draw[very thick,->] (2,-3) to [out=90,in=270] (-2,3);	
\end{tikzpicture}
\]
since the others can be obtained\footnote{See e.g.~\cite{Polyak}, which accomplishes the more difficult 
task of showing that only five oriented link Reidemeister moves suffice to obtain all others. 
Our task here is easier since we work with graphs, and have already established all versions of gRII.} 
from this using gRII moves. 
Again, $\tau^2$ is a $2$-functor, so it preserves the relation \eqref{eq:cubicKLR}.
Thus, it must have the same sign on both sides of gRIII, which concludes the proof of the claim.
\hfill $\blacktriangle$ \\

Any two matching diagrams with the same underlying matching 
are related by the moves \eqref{eq:gRM}, hence have the same value of $T(\mathcal{M})$.
We now pass to different matching diagram $\mathcal{M}'$ 
having the same underlying matching as our given diagram $\mathcal{M}$, 
for which the value $T(\mathcal{M}')$ is easy to compute.

Pick the last color appearing in $\mathcal{M}$ 
and use gRII and gRIII to ``pull all instances of this color'' to the far right.
For example, 
if we start with \eqref{bigexample} then this color is blue 
and the result is
\[
\mathcal{M} =
\begin{tikzpicture}[anchorbase,scale=.5]
\draw[thick,red,<-] (4,0) to [out=90,in=270] (0,4);
\draw[thick,red,<-] (1,4) to [out=270,in=270] (6,4);
\draw[thick,red,->] (5,0) to [out=90,in=270] (7,4);
\draw[thick,green,->] (2,4) to [out=270,in=90] (0,0);
\draw[thick,green,->] (1,0) to [out=90,in=90] (6,0);
\draw[thick,green,<-] (3,4) to [out=270,in=90] (7,0);
\draw[thick,blue,->] (4,4) to [out=270,in=90] (2,1) to (2,0);
\draw[thick,blue,->] (3,0) to [out=90,in=180] (5.5,.5) to [out=0,in=90] (8,0);
\draw[thick,blue,<-] (5,4) to [out=270,in=90] (9,0);
\end{tikzpicture}
\; \sim \; 
\mathcal{M}' :=
\begin{tikzpicture}[anchorbase,scale=.5]
\draw[thick,red,<-] (4,0) to (3.5,2);
\draw[thick,red] (0,4) to (3.5,2);
\draw[thick,red,<-] (1,4) to (3.5,2); 
\draw[thick,red] (6,4) to (3.5,2);
\draw[thick,red,->] (5,0) to (3.5,2);
\draw[thick,red] (7,4) to (3.5,2);
\draw[thick,green] (2,4) to (3.5,2);
\draw[thick,green,->] (0,0) to (3.5,2);
\draw[thick,green] (1,0) to (3.5,2);
\draw[thick,green,->] (6,0) to (3.5,2);
\draw[thick,green,<-] (3,4) to (3.5,2);
\draw[thick,green] (7,0) to (3.5,2);
\draw[dashed,fill=white] (3.5,2) circle (1);
\node at (3.5,2) {$\mathcal{M}_{rg}$};
\draw[thick,blue,->] (4,4) to [out=270,in=120] (6.5,1.375) to [out=300,in=90] (2,0);
\draw[thick,blue,->] (3,0) to [out=90,in=180] (5.5,.5) to [out=0,in=90] (8,0);
\draw[thick,blue,<-] (5,4) to [out=270,in=90] (9,0);
\end{tikzpicture}
\]
for 
\[
\mathcal{M}_{rg} =
\begin{tikzpicture}[anchorbase,scale=.5]
\draw[thick,red,<-] (4,0) to [out=90,in=270] (0,4);
\draw[thick,red,<-] (1,4) to [out=270,in=270] (6,4);
\draw[thick,red,->] (5,0) to [out=90,in=270] (7,4);
\draw[thick,green,->] (2,4) to [out=270,in=90] (0,0);
\draw[thick,green,->] (1,0) to [out=90,in=90] (6,0);
\draw[thick,green,<-] (3,4) to [out=270,in=90] (7,0);
\draw[dashed] (3.5,2) ellipse (3 and 1.5);
\end{tikzpicture}
=
\begin{tikzpicture}[anchorbase,scale=.5]
\draw[thick,red,<-] (4,0) to (3.5,2);
\draw[thick,red] (0,4) to (3.5,2);
\draw[thick,red,<-] (1,4) to (3.5,2); 
\draw[thick,red] (6,4) to (3.5,2);
\draw[thick,red,->] (5,0) to (3.5,2);
\draw[thick,red] (7,4) to (3.5,2);
\draw[thick,green] (2,4) to (3.5,2);
\draw[thick,green,->] (0,0) to (3.5,2);
\draw[thick,green] (1,0) to (3.5,2);
\draw[thick,green,->] (6,0) to (3.5,2);
\draw[thick,green,<-] (3,4) to (3.5,2);
\draw[thick,green] (7,0) to (3.5,2);
\draw[dashed,fill=white] (3.5,2) circle (1);
\node at (3.5,2) {$\mathcal{M}_{rg}$};
\end{tikzpicture} \, .
\]
For the remainder of this proof, we will use ``blue'' to mean this final color.
In general, the subdiagram $\mathcal{M}_{rg}$ 
containing the other colors will be an arbitrary colored matching diagram (without blue), 
and the subdiagram to the right of that can be an arbitrary matching diagram (with only blue).

The meeting of ``blue'' and ``non-blue'' strands occur only in a proscribed fashion.
Consider the shading on $\mathcal{M}'$:
\[\mathcal{M}' =
\raisebox{-0.5\height}{\includegraphics{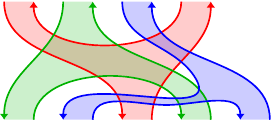}}
\, .
\]
The only intersections of blue strands and non-blue strands occur when
blue-shaded strips cross over non-blue strips:
\[
\raisebox{-0.5\height}{\includegraphics{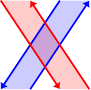}}
\]
Said another way, each region with blue in its shading is either pure blue, or blue and one other color, 
and in the latter case the region is a square. We call these \emph{doubly-shaded squares}.

The advantage of having done this manipulation is that each crossing in a doubly-shaded square 
is adjacent to a white-colored region, i.e.~a region in weight $\wt$. 
Hence, in each doubly-shaded square there are two crossings of the form:
\begin{equation} \label{thisone}
\raisebox{-0.5\height}{\includegraphics{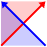}}
\end{equation}
and two crossings of the form:
\begin{equation} \label{thatone}
\raisebox{-0.5\height}{\includegraphics{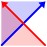}}
\, .
\end{equation}
The signs induced by $\tau^2$ on the two instances of \eqref{thisone} are the same, and therefore cancel out. 
The same goes for \eqref{thatone}.

Thus $\tau^2$ fixes $\mathcal{M}'$ if and only if it fixes $\mathcal{M}_{rg}$, 
a colored matching with one fewer color. 
By induction on the number of colors, $\tau^2$ fixes all
diagrams representing morphisms in $\FECatprime(\glm)$.
\end{proof}

\begin{rem} 
If we directly consider the original diagram $\mathcal{M}$ 
from \eqref{bigexample}, there would be red-blue crossings in the green-shaded region.
The action of $\tau^2$ on this diagram can intoduce different signs on these crossings than the ones
in the white region, which makes it harder to argue that all signs cancel out. 
However, our proof shows that they must. \end{rem}

We now fix our preferred choice of involution on $\XX_q(\glm)$, which we use henceforth.

\begin{thm}
	\label{thm:involution}
If 
\begin{equation}
	\label{eq:ourfns}
r_i(\wt) := \begin{cases} 	a_i & \text{if } \langle \ee_i^\vee, \wt \rangle \equiv 1  \\
							0 & \text{else} \end{cases}
\, , \quad
l_i(\wt) \equiv r_i(-\wt)
\, , \quad
v_{i,j}(\wt) := \begin{cases} 	1 & \text{if } j = i+1 \\
						0 & \text{else} \end{cases}
\end{equation}
then $\tau$ restricts to an involution on $\XX_q(\glm)$.
\end{thm}

\begin{proof} 
Recall that $\langle \ee_i^\vee, \wt \rangle = a_i - a_{i+1}$.
It is straightforward to check that these functions satisfy \eqref{tauconditions}, 
where we let $r_i'(\wt)$ and $l_i'(\wt)$ be defined by \eqref{primesign}. 
Since $l_i(\wt) \equiv r_i(-\wt)$, 
$\tau$ restricts to an involution on $\XX_q(\glm)$ by Theorem \ref{thm:wheninvolution}. \end{proof}

\subsection{Dependence on $n$}\label{ss:dependence-on-n}

Although the notation obscures it, 
the assignments in Definition \ref{def:symmetry} depend on $n$, 
e.g.~via the action of $\tau$ on objects. 
On the other hand, the formulae in Theorem \ref{thm:involution} giving 
our preferred choice of $\tau$ do not depend on $n$. 
In this section, we clarify the dependence on $n$ for 
our preferred $\tau$.

\begin{nota}
For this section (and in various other places when we want to emphasize the pertinent value of $n$),
we will write $\tau_n$ to denote the preferred $\tau$ which acts on objects by $\wt\mapsto 2\wtn-\wt$. 
\end{nota}

Note that $\tau_0$ sends $\wt$ to $-\wt$, while $\tau_n$ sends $\wtn + \wt$ to $\wtn - \wt$. 
Suppose there were an automorphism $\sh_n$ of $\UU_q(\glm)$ which mapped $\wt$ to $\wt+\wtn$.
Then, by conjugating $\tau_0$ by $\sh_n$, we obtain an automorphism which would act as
\[
\wtn+\wt\xmapsto{\sh_n^{-1}} \wt \xmapsto{\tau_0} -\wt\xmapsto{\sh_n} \wtn-\wt. 
\]
Ideally, this automorphism would be identified with $\tau_n$. 

We now set out to define automorphisms 
$\sh_n \colon \UU_q(\glm)\rightarrow \UU_q(\glm)$ 
that send $\wt \mapsto \wt+\wtn$ and satisfy
$\sh_n\circ \tau_0 = \tau_n\circ \sh_n$.

\begin{prop}
	\label{prop:sh-n}
The following assignments determine a $2$-functor $\sh_1 \colon \UU_q(\glm) \to \UU_q(\glm)$.
\begin{itemize}
\item \textbf{Objects:} $\wt \xmapsto{\sh_1} \wt+ \wtone$.
	
\item \textbf{$1$-morphisms:} 
$\EE_i \one_{\wt} \xmapsto{\sh_1} \EE_i \one_{\wt+ \wtone}$
and
$\FF_i \one_{\wt} \xmapsto{\sh_1} \FF_i \one_{\wt+ \wtone}$.

\item \textbf{$2$-morphisms:} 
\begin{equation}
\begin{gathered} \label{eq:sh-1}
\begin{tikzpicture}[anchorbase,scale=1]
\draw[thick,green,->] (0,0) to node[black]{$\bullet$} (0,1);
\node at (.5,.75){\scs$\wt$};
\end{tikzpicture}
\xmapsto{\sh_1}
\begin{tikzpicture}[anchorbase,scale=1]
\draw[thick,green,->] (0,0) to node[black]{$\bullet$} (0,1);
\node at (.5,.75){\scs$\wt{+}\wtone$};
\end{tikzpicture}
\, , \quad
\begin{tikzpicture}[anchorbase,scale=1]
\node at (0,0){\NB{e_r(\X_i)}};
\node at (.875,.375){\scs$\wt$};
\end{tikzpicture}
\xmapsto{\sh_1}
\begin{tikzpicture}[anchorbase,scale=1]
\node at (0,0){\NB{e_r(\X_i)}};
\node at (.875,.375){\scs$\wt{+}\wtone$};
\end{tikzpicture}
\, , \\
\begin{tikzpicture}[anchorbase,scale=1]
\draw[thick,green,->] (0,0) to [out=90,in=270] (.5,1);
\draw[thick,->] (.5,0) to [out=90,in=270] (0,1);
\node at (.875,.75){\scs$\wt$};
\end{tikzpicture}
\xmapsto{\sh_1}
\begin{tikzpicture}[anchorbase,scale=1]
\draw[thick,green,->] (0,0) to [out=90,in=270] (.5,1);
\draw[thick,->] (.5,0) to [out=90,in=270] (0,1);
\node at (.875,.75){\scs$\wt{+}\wtone$};
\end{tikzpicture}
\, , \\
\begin{tikzpicture}[anchorbase,scale=1]
\draw[thick,green,->] (-.25,0) to [out=90,in=180] (0,.5) 
	to [out=0,in=90] (.25,0);
\node at (.625,.5){\scs$\wt$};
\end{tikzpicture}
\xmapsto{\sh_1}
(-1)^{\rho_{\grb}(\wt)}
\begin{tikzpicture}[anchorbase,scale=1]
\draw[thick,green,->] (-.25,0) to [out=90,in=180] (0,.5) 
	to [out=0,in=90] (.25,0);
\node at (.625,.5){\scs$\wt{+}\wtone$};
\end{tikzpicture}
\, , \quad
\begin{tikzpicture}[anchorbase,scale=1]
\draw[thick,green,->] (-.25,0) to [out=270,in=180] (0,-.5) 
	to [out=0,in=270] (.25,0);
\node at (.625,-.5){\scs$\wt$};
\end{tikzpicture}
\xmapsto{\sh_1}
(-1)^{\rho'_{\grb}(\wt)}
\begin{tikzpicture}[anchorbase,scale=1]
\draw[thick,green,->] (-.25,0) to [out=270,in=180] (0,-.5) 
	to [out=0,in=270] (.25,0);
\node at (.625,-.5){\scs$\wt{+}\wtone$};
\end{tikzpicture}
\, , \\
\begin{tikzpicture}[anchorbase,scale=1]
\draw[thick,green,<-] (-.25,0) to [out=90,in=180] (0,.5) 
	to [out=0,in=90] (.25,0);
\node at (.625,.5){\scs$\wt$};
\end{tikzpicture}
\xmapsto{\sh_1}
(-1)^{\lambda_{\grb}(\wt)}
\begin{tikzpicture}[anchorbase,scale=1]
\draw[thick,green,<-] (-.25,0) to [out=90,in=180] (0,.5) 
	to [out=0,in=90] (.25,0);
\node at (.625,.5){\scs$\wt{+}\wtone$};
\end{tikzpicture}
\, , \quad
\begin{tikzpicture}[anchorbase,scale=1]
\draw[thick,green,<-] (-.25,0) to [out=270,in=180] (0,-.5) 
	to [out=0,in=270] (.25,0);
\node at (.625,-.5){\scs$\wt$};
\end{tikzpicture}
\xmapsto{\tau}
(-1)^{\lambda'_{\grb}(\wt)}
\begin{tikzpicture}[anchorbase,scale=1]
\draw[thick,green,<-] (-.25,0) to [out=270,in=180] (0,-.5) 
	to [out=0,in=270] (.25,0);
\node at (.625,-.5){\scs$\wt{+}\wtone$};
\end{tikzpicture} \, ,
\end{gathered}
\end{equation}
where
\[
\rho_{\grb}(\wt) = a_i \, , \quad \lambda_{\grb}(\wt) = \rho_{\grb}(\wt+ \alpha_{\grb}) + 1 \, , \quad 
\rho'_{\grb}(\wt) =  \lambda_{\grb}(\wt)+1 \, , \quad \text{and} \quad \lambda_{\grb}'(\wt) = \rho_{\grb}(\wt) + 1 \, .
\]
\end{itemize}
\end{prop}
\begin{proof}
A straightforward check, similar to the proof of Theorem \ref{thm:symmetry}. 
Note that preserving equation \eqref{newbub} imposes the requirements 
\[
\lambda_{\grb}(\wt)+\rho'_{\grb}(\wt)+1 \equiv 0 
\quad \text{and} \quad 
\lambda'_{\grb}(\wt) + \rho_{\grb}(\wt) +1 \equiv 0 \, ,
\]
while preserving the biadjunction relation (c.f~equation \eqref{eq:biadjcomp}) 
imposes the requirement
\[
\lambda_{\grb}(\wt+\alpha_{\grb}) \equiv 1+ \rho_{\grb}(\wt) \, . \qedhere
\]
\end{proof}

\begin{defn}
Let $\sh_n:=\sh_1\circ \sh_1\circ \dots \circ \sh_1$ be $\sh_1$ composed with itself $n$ times.
\end{defn} 

Note that $\sh_n$ acts on cup and cap generators of $\UU_q(\glm)$ as follows:
\begin{equation}
\begin{gathered} \label{eq:sh-n}
\begin{tikzpicture}[anchorbase,scale=1]
\draw[thick,green,->] (-.25,0) to [out=90,in=180] (0,.5) 
	to [out=0,in=90] (.25,0);
\node at (.625,.5){\scs$\wt$};
\end{tikzpicture}
\xmapsto{\sh_n}
(-1)^{\sum_{k=0}^{n-1}\rho_{\grb}(\wt+\wtk)}
\begin{tikzpicture}[anchorbase,scale=1]
\draw[thick,green,->] (-.25,0) to [out=90,in=180] (0,.5) 
	to [out=0,in=90] (.25,0);
\node at (.625,.5){\scs$\wt{+}\wtn$};
\end{tikzpicture}
\, , \quad
\begin{tikzpicture}[anchorbase,scale=1]
\draw[thick,green,->] (-.25,0) to [out=270,in=180] (0,-.5) 
	to [out=0,in=270] (.25,0);
\node at (.625,-.5){\scs$\wt$};
\end{tikzpicture}
\xmapsto{\sh_n}
(-1)^{\sum_{k=0}^{n-1}\rho'_{\grb}(\wt+\wtk)}
\begin{tikzpicture}[anchorbase,scale=1]
\draw[thick,green,->] (-.25,0) to [out=270,in=180] (0,-.5) 
	to [out=0,in=270] (.25,0);
\node at (.625,-.5){\scs$\wt{+}\wtn$};
\end{tikzpicture}
\, , \\
\begin{tikzpicture}[anchorbase,scale=1]
\draw[thick,green,<-] (-.25,0) to [out=90,in=180] (0,.5) 
	to [out=0,in=90] (.25,0);
\node at (.625,.5){\scs$\wt$};
\end{tikzpicture}
\xmapsto{\sh_n}
(-1)^{\sum_{k=0}^{n-1}\lambda_{\grb}(\wt+\wtk)}
\begin{tikzpicture}[anchorbase,scale=1]
\draw[thick,green,<-] (-.25,0) to [out=90,in=180] (0,.5) 
	to [out=0,in=90] (.25,0);
\node at (.625,.5){\scs$\wt{+}\wtn$};
\end{tikzpicture}
\, , \quad
\begin{tikzpicture}[anchorbase,scale=1]
\draw[thick,green,<-] (-.25,0) to [out=270,in=180] (0,-.5) 
	to [out=0,in=270] (.25,0);
\node at (.625,-.5){\scs$\wt$};
\end{tikzpicture}
\xmapsto{\sh_n}
(-1)^{\sum_{k=0}^{n-1}\lambda'_{\grb}(\wt+\wtk)}
\begin{tikzpicture}[anchorbase,scale=1]
\draw[thick,green,<-] (-.25,0) to [out=270,in=180] (0,-.5) 
	to [out=0,in=270] (.25,0);
\node at (.625,-.5){\scs$\wt{+}\wtn$};
\end{tikzpicture} \, .
\end{gathered}
\end{equation}
Here, we let $\mathbf{k} := (k,\ldots,k) \in \Z^m$ for $k \geq 0$.

\begin{prop}\label{P:sh-n-conjugation-is-tau-n}
We have 
$\tau_n = \sh_n\circ \tau_0 \circ \sh_{n}^{-1}$. 
\end{prop}

\begin{proof}
A direct computation using \eqref{eq:inv} and \eqref{eq:sh-n}.
Since our preferred $\tau$ from Theorem \eqref{thm:involution} uses constant functions for the $v_{ij}$'s, 
the claim follows from checking that $\tau_n\circ \sh_n = \sh_n\circ \tau_0$ on cups and caps. 
We verify this for rightward caps, leaving the other three cases to the reader. To this end, observe that
\[
\begin{tikzpicture}[anchorbase,scale=1]
\draw[thick,green,->] (-.25,0) to [out=90,in=180] (0,.5) 
	to [out=0,in=90] (.25,0);
\node at (.625,.5){\scs$\wt$};
\end{tikzpicture}
\xmapsto{\sh_n}
(-1)^{\sum_{k=0}^{n-1}\rho_{\grb}(\wt+\wtk)}
\begin{tikzpicture}[anchorbase,scale=1]
\draw[thick,green,->] (-.25,0) to [out=90,in=180] (0,.5) 
	to [out=0,in=90] (.25,0);
\node at (.625,.5){\scs$\wt{+}\wtn$};
\end{tikzpicture}
\xmapsto{\tau_n}
(-1)^{r_{\grb}(\wt)}(-1)^{\sum_{k=0}^{n-1}\rho_{\grb}(\wt+\wtk)}
\begin{tikzpicture}[anchorbase,scale=1]
\draw[thick,green,<-] (-.25,0) to [out=90,in=180] (0,.5) 
	to [out=0,in=90] (.25,0);
\node at (.625,.5){\scs$\wtn{-}\wt$};
\end{tikzpicture}
\]
and
\[
\begin{tikzpicture}[anchorbase,scale=1]
\draw[thick,green,->] (-.25,0) to [out=90,in=180] (0,.5) 
	to [out=0,in=90] (.25,0);
\node at (.625,.5){\scs$\wt$};
\end{tikzpicture}
\xmapsto{\tau_0}
(-1)^{r_{\grb}(\wt)}
\begin{tikzpicture}[anchorbase,scale=1]
\draw[thick,green,<-] (-.25,0) to [out=90,in=180] (0,.5) 
	to [out=0,in=90] (.25,0);
\node at (.625,.5){\scs${-}\wt$};
\end{tikzpicture}
\xmapsto{\sh_n}
(-1)^{\sum_{k=0}^{n-1}\lambda_{\grb}(-\wt+\wtk)}(-1)^{r_{\grb}(\wt)}
\begin{tikzpicture}[anchorbase,scale=1]
\draw[thick,green,<-] (-.25,0) to [out=90,in=180] (0,.5) 
	to [out=0,in=90] (.25,0);
\node at (.625,.5){\scs$\wtn{-}\wt$};
\end{tikzpicture} \, .
\]
Using equation \eqref{eq:sh-1}, we compute (modulo $2$) that
\[
\sum_{k=0}^{n-1}\lambda_{\grb}(-\wt+\wtk) 
= \sum_{k=0}^{n-1} \rho_{\grb}(-\wt + \wtk+ \alpha_{\grb}) + 1 
= \sum_{k=1}^{n-1}-a_i+k+1+1 
\equiv \sum_{k=1}^{n-1}a_i+k 
= \sum_{k=0}^{n-1}\rho_{\grb}(\wt + \wtk) \, ,
\]
so $\tau_n\circ \sh_n = \sh_n \circ \tau_0$ on rightward caps. 
\end{proof}

\begin{rem}
Proposition \ref{P:sh-n-conjugation-is-tau-n} 
shows that all of the automorphisms $\tau_n$ are conjugate.
Nevertheless,
in Section \ref{ss:schur-FE} we will consider involutions induced by $\tau_n$ on 
certain quotients of $\UU_q(\glm)$ that depend on a specific value of $n$. 
Hence, we will make use of each $\tau_n$.
\end{rem}

\subsection{Extending to thick calculus}\label{ss:extended-to-thick}

We now observe that our automorphism $\tau$ extends to the thick calculus 
of Section \ref{ss:thickcalc} and record its action there.

\begin{prop}\label{prop:tauactsonthicker}
Suppose that $\tau$ satisfies the conditions of Theorem \ref{thm:symmetry}. 
There is a unique extension of $\tau$ to $\cUU_q(\glm)$ such that
$\mathcal{E}_i^{(k)}\one_{n+\wt}\mapsto \mathcal{F}_i^{(k)}\one_{n-\wt}$ and 
\begin{equation}
	\label{eq:tauonMS}
\begin{tikzpicture}[anchorbase,yscale=-1]
\draw[ultra thick,green,<-] (0,0) node[above=-2pt]{\scs$k{+}\ell$} to [out=90,in=270] (0,.375);
\draw[ultra thick,green] (0,.375) to [out=150,in=270] (-.25,1) node[below=-2pt]{\scs$k$};
\draw[ultra thick,green] (0,.375) to [out=30,in=270] (.25,1) node[below=-2pt]{\scs$\ell$};
\node at (.625,.25){\scs$\wtn{+}\wt$};
\end{tikzpicture}
\xmapsto{\tau}
\begin{tikzpicture}[anchorbase,yscale=-1]
\draw[ultra thick,green] (0,0) node[above=-2pt]{\scs$k{+}\ell$} to [out=90,in=270] (0,.375);
\draw[ultra thick,green,->] (0,.375) to [out=150,in=270] (-.25,1) node[below=-2pt]{\scs$k$};
\draw[ultra thick,green,->] (0,.375) to [out=30,in=270] (.25,1) node[below=-2pt]{\scs$\ell$};
\node at (.625,.25){\scs$\wtn{-}\wt$};
\end{tikzpicture} 
\xmapsto{\tau} 
\begin{tikzpicture}[anchorbase,yscale=-1]
\draw[ultra thick,green,<-] (0,0) node[above=-2pt]{\scs$k{+}\ell$} to [out=90,in=270] (0,.375);
\draw[ultra thick,green] (0,.375) to [out=150,in=270] (-.25,1) node[below=-2pt]{\scs$k$};
\draw[ultra thick,green] (0,.375) to [out=30,in=270] (.25,1) node[below=-2pt]{\scs$\ell$};
\node at (.625,.25){\scs$\wtn{+}\wt$};
\end{tikzpicture}
\, , \quad 
\begin{tikzpicture}[anchorbase,scale=1]
\draw[ultra thick,green] (0,0) node[below=-2pt]{\scs$k{+}\ell$} to [out=90,in=270] (0,.375);
\draw[ultra thick,green,->] (0,.375) to [out=150,in=270] (-.25,1) node[above=-2pt]{\scs$k$};
\draw[ultra thick,green,->] (0,.375) to [out=30,in=270] (.25,1) node[above=-2pt]{\scs$\ell$};
\node at (.625,.75){\scs$\wtn{+}\wt$};
\end{tikzpicture}
\xmapsto{\tau} \begin{tikzpicture}[anchorbase,scale=1]
\draw[ultra thick,green,<-] (0,0) node[below=-2pt]{\scs$k{+}\ell$} to [out=90,in=270] (0,.375);
\draw[ultra thick,green] (0,.375) to [out=150,in=270] (-.25,1) node[above=-2pt]{\scs$k$};
\draw[ultra thick,green] (0,.375) to [out=30,in=270] (.25,1) node[above=-2pt]{\scs$\ell$};
\node at (.625,.75){\scs$\wtn{-}\wt$};
\end{tikzpicture} 
\xmapsto{\tau} 
\begin{tikzpicture}[anchorbase,scale=1]
\draw[ultra thick,green] (0,0) node[below=-2pt]{\scs$k{+}\ell$} to [out=90,in=270] (0,.375);
\draw[ultra thick,green,->] (0,.375) to [out=150,in=270] (-.25,1) node[above=-2pt]{\scs$k$};
\draw[ultra thick,green,->] (0,.375) to [out=30,in=270] (.25,1) node[above=-2pt]{\scs$\ell$};
\node at (.625,.75){\scs$\wtn{+}\wt$};
\end{tikzpicture}.
\end{equation}
Moreover, under this extension
\begin{equation}
\begin{gathered}
\begin{tikzpicture}[anchorbase,scale=1]
\draw[ultra thick,green,->] (-.25,0) node[below=-2pt]{\scs $k$} to [out=90,in=180] (0,.5) 
	to [out=0,in=90] (.25,0);
\node at (.625,.5){\scs$\wtn{+}\wt$};
\end{tikzpicture}
\mapsto (-1)^{\sum_{s=0}^{k-1}r_{\grb}(\wt - s \alpha_{\grb})}
\begin{tikzpicture}[anchorbase,scale=1]
\draw[ultra thick,green,<-] (-.25,0) to [out=90,in=180] (0,.5) 
	to [out=0,in=90] (.25,0) node[below=-2pt]{\scs$k$};
\node at (.625,.5){\scs$\wtn{-}\wt$};
\end{tikzpicture}
\, , \quad
\begin{tikzpicture}[anchorbase,scale=1]
\draw[ultra thick,green,<-] (-.25,0) to [out=90,in=180] (0,.5) 
	to [out=0,in=90] (.25,0) node[below=-2pt]{\scs$k$};
\node at (.625,.5){\scs$\wtn{+}\wt$};
\end{tikzpicture}
\mapsto (-1)^{\sum_{s=0}^{k-1}\ell_{\grb}(\wt + s \alpha_{\grb})}
\begin{tikzpicture}[anchorbase,scale=1]
\draw[ultra thick,green,->] (-.25,0) node[below=-2pt]{\scs $k$} to [out=90,in=180] (0,.5) 
	to [out=0,in=90] (.25,0);
\node at (.625,.5){\scs$\wtn{-}\wt$};
\end{tikzpicture}
\, , \\
\begin{tikzpicture}[anchorbase,scale=1, yscale=-1]
\draw[ultra thick,green,->] (-.25,0) node[below=-12pt]{\scs $k$} to [out=90,in=180] (0,.5) 
	to [out=0,in=90] (.25,0);
\node at (.625,.5){\scs$\wtn{+}\wt$};
\end{tikzpicture}
\mapsto (-1)^{\sum_{s=0}^{k-1}r'_{\grb}(\wt + s \alpha_{\grb})}
\begin{tikzpicture}[anchorbase,yscale=-1]
\draw[ultra thick,green,<-] (-.25,0) to [out=90,in=180] (0,.5) 
	to [out=0,in=90] (.25,0) node[below=-12pt]{\scs$k$};
\node at (.625,.5){\scs$\wtn{-}\wt$};
\end{tikzpicture}
\, , \quad
\begin{tikzpicture}[anchorbase,yscale=-1]
\draw[ultra thick,green,<-] (-.25,0) to [out=90,in=180] (0,.5) 
	to [out=0,in=90] (.25,0) node[below=-12pt]{\scs$k$};
\node at (.625,.5){\scs$\wtn{+}\wt$};
\end{tikzpicture}
\mapsto (-1)^{\sum_{s=0}^{k-1}\ell'_{\grb}(\wt - s \alpha_{\grb})}
\begin{tikzpicture}[anchorbase,yscale=-1]
\draw[ultra thick,green,->] (-.25,0) node[below=-12pt]{\scs $k$} to [out=90,in=180] (0,.5) 
	to [out=0,in=90] (.25,0);
\node at (.625,.5){\scs$\wtn{-}\wt$};
\end{tikzpicture}
\, ,
\end{gathered}
\end{equation}
\begin{equation}
\begin{tikzpicture}[anchorbase,scale=1]
\draw[ultra thick,green,->] (0,0) node[below=-2pt]{\scs$k$} to [out=90,in=270] (.5,1);
\draw[ultra thick,->] (.5,0) node[below=-2pt]{\scs$\ell$} to [out=90,in=270] (0,1);
\node at (.875,.75){\scs$\wtn{+}\wt$};
\end{tikzpicture}
\mapsto (-1)^{\sum_{\substack{0\le s\le k-1 \\ 0 \le t\le \ell-1}}v_{\grb \bullet}(\wt+s\alpha_{\grb} + t\alpha_{\bullet})}
\begin{tikzpicture}[anchorbase,scale=1]
\draw[ultra thick,green,<-] (0,0) node[below=-2pt]{\scs$k$} to [out=90,in=270] (.5,1);
\draw[ultra thick,<-] (.5,0) node[below=-2pt]{\scs$\ell$} to [out=90,in=270] (0,1);
\node at (.875,.75){\scs$\wtn{-}\wt$};
\end{tikzpicture} \, ,
\end{equation}
and
\begin{equation}\label{eqn:tauonthickdot}
\begin{tikzpicture}[anchorbase,scale=1]
\draw[ultra thick,green,->] (0,-.75) node[below=-2pt]{\scs$k$} to node[black]{$\CQGbox{\mathfrak{s}_\parti}$}
	(0,.75) node[white,above=-2pt]{\scs$k$};
\node at (.5,.5){\scs$\wtn{+}\wt$};
\end{tikzpicture}
\mapsto (-1)^{|\parti|}
\begin{tikzpicture}[anchorbase,scale=1]
\draw[ultra thick,green,<-] (0,-.75) node[below=-2pt]{\scs$k$} to node[black]{$\CQGbox{\mathfrak{s}_\parti}$}
	(0,.75) node[white,above=-2pt]{\scs$k$};
\node at (.5,.5){\scs$\wtn{-}\wt$};
\end{tikzpicture} \, . 
\end{equation}
\end{prop}

\begin{proof}
Since $\tau$ preserves the relations of $\UU_q(\glm)$, 
it follows from Definition \ref{def:thickCQG} that the extension of $\tau$ to $\cUU_q(\glm)$ 
is well-defined if and only if it preserves the relations in \eqref{eq:MStoHT} and \eqref{eq:explode}. 
The former follows by \eqref{samecrossingsign}, 
i.e.~since $\tau$ introduces no signs on thin uni-colored crossings, regardless of ambient weight.
The latter is a consequence of $\tau$ acting on dots by $-1$ 
(i.e.~the first relation in \eqref{eq:inv}) and the formula:
\begin{equation}\label{eqn:rearrangedots}
\begin{tikzpicture}[anchorbase,scale=1]
\draw[ultra thick,directed=.6] (0,-.75) node[below=-2pt]{\scs$k$} to (0,-.375);
\draw[thick] (0,-.375) to [out=180,in=270] node[pos=.95]{$\bullet$}  (-1.125,0) 
	node[left=-1pt]{\scs$k{-}1$} to [out=90,in=180] (0,.375);
\draw[thick] (0,-.375) to [out=150,in=270] node[pos=.85]{$\bullet$} (-.25,0) 
	node[left=-1pt]{\scs$k{-}2$} to [out=90,in=210] (0,.375);
\node at (0,0) {$\mydots$};
\draw[thick] (0,-.375) to [out=30,in=270] node[pos=.85]{$\bullet$} (.25,0) to [out=90,in=330] (0,.375);
\draw[thick] (0,-.375) to [out=0,in=270] (.5,0) to [out=90,in=0] (0,.375);
\draw[ultra thick,directed=.6] (0,.375) to (0,.75) node[above=-2pt]{\scs$k$} ;
\node at (.625,.5){\scs$\wt$};
\end{tikzpicture}
= (-1)^{\binom{k}{2}} \,
\begin{tikzpicture}[anchorbase,scale=1, xscale=-1]
\draw[ultra thick,directed=.6] (0,-.75) node[below=-2pt]{\scs$k$} to (0,-.375);
\draw[thick] (0,-.375) to [out=180,in=270] node[pos=.95]{$\bullet$}  (-1.125,0) 
	node[right=-1pt]{\scs$k{-}1$} to [out=90,in=180] (0,.375);
\draw[thick] (0,-.375) to [out=150,in=270] node[pos=.85]{$\bullet$} (-.25,0) 
	node[right=-1pt]{\scs$k{-}2$} to [out=90,in=210] (0,.375);
\node at (0,0) {$\mydots$};
\draw[thick] (0,-.375) to [out=30,in=270] node[pos=.85]{$\bullet$} (.25,0) to [out=90,in=330] (0,.375);
\draw[thick] (0,-.375) to [out=0,in=270] (.5,0) to [out=90,in=0] (0,.375);
\draw[ultra thick,directed=.6] (0,.375) to (0,.75) node[above=-2pt]{\scs$k$} ;
\node at (-.625,.5){\scs$\wt$};
\end{tikzpicture}
\end{equation}
which is a consequence of \eqref{eqn:sgnfromdemazure}.

The action of $\tau$ on thick caps follows from applying $\tau$, as defined in \eqref{eq:inv}, 
to the right hand side of \eqref{eqn:thickcapdefn} then using \eqref{eqn:rearrangedots}. 
The action of $\tau$ on thick cups is analogous. 
The action of $\tau$ on thick crossings (with one or two colors) 
follows from applying $\tau$ to the right hand side of 
\eqref{eqn:thickcrossingdefn} or \eqref{eqn:multicolorthickcrossingdefn}, 
then using \eqref{eqn:rearrangedots} in conjunction with \eqref{eq:otherthick}, in the case of one color crossings, 
and \eqref{eq:multicolorotherthick}, in the case of two color crossings. 
Finally, the action of $\tau$ on $\Schur \in \End(\EE_i^{(k)}\one_{\wt})$ 
comes from applying $\tau$ to the right hand side of \eqref{eq:SymDec}. \end{proof}

\begin{cor}
	\label{cor:ourthicktau}
Our preferred $\tau$ from Theorem \ref{thm:involution}, 
acts on thick cap, cup, and upward crossing $2$-morphisms as follows:
\begin{equation}
	\label{eq:tauonthickcap}
\begin{tikzpicture}[anchorbase,scale=.875]
\draw[ultra thick,green,->] (-.25,0) node[below=-2pt]{\scs $k$} to [out=90,in=180] (0,.5) 
	to [out=0,in=90] (.25,0);
\node at (.375,.75){\tiny$\wtn{+}\wt$};
\end{tikzpicture}
\xmapsto{\tau}
\begin{cases}
\begin{tikzpicture}[anchorbase,scale=.875]
\draw[ultra thick,green,<-] (-.25,0) to [out=90,in=180] (0,.5) 
	to [out=0,in=90] (.25,0) node[below=-2pt]{\scs$k$};
\node at (.375,.75){\tiny$\wtn{-}\wt$};
\end{tikzpicture} 
& \text{if } \langle \alpha_{\grb}^{\vee} , \wt \rangle \equiv 0 \\
(-1)^{k\cdot a_{\grb} - \binom{k}{2}}
\begin{tikzpicture}[anchorbase,scale=.875,xshift=-5pt]
\draw[ultra thick,green,<-] (-.25,0) to [out=90,in=180] (0,.5) 
	to [out=0,in=90] (.25,0) node[below=-2pt]{\scs$k$};
\node at (.375,.75){\tiny$\wtn{-}\wt$};
\end{tikzpicture} 
& \text{if } \langle \alpha_{\grb}^{\vee} , \wt \rangle \equiv 1
\end{cases}
\, , \quad
\begin{tikzpicture}[anchorbase,scale=.875]
\draw[ultra thick,green,<-] (-.25,0) node[below=-2pt]{\scs $k$} to [out=90,in=180] (0,.5) 
	to [out=0,in=90] (.25,0);
\node at (.375,.75){\tiny$\wtn{+}\wt$};
\end{tikzpicture}
\xmapsto{\tau}
\begin{cases}
\begin{tikzpicture}[anchorbase,scale=.875]
\draw[ultra thick,green,<-] (-.25,0) to [out=90,in=180] (0,.5) 
	to [out=0,in=90] (.25,0) node[below=-2pt]{\scs$k$};
\node at (.375,.75){\tiny$\wtn{-}\wt$};
\end{tikzpicture} 
& \text{if } \langle \alpha_{\grb}^{\vee} , \wt \rangle \equiv 0 \\
(-1)^{-k\cdot a_{\grb} - \binom{k}{2}}
\begin{tikzpicture}[anchorbase,scale=.875,xshift=-5pt]
\draw[ultra thick,green,<-] (-.25,0) to [out=90,in=180] (0,.5) 
	to [out=0,in=90] (.25,0) node[below=-2pt]{\scs$k$};
\node at (.375,.75){\tiny$\wtn{-}\wt$};
\end{tikzpicture} 
& \text{if } \langle \alpha_{\grb}^{\vee} , \wt \rangle \equiv 1
\end{cases}
\end{equation}
\begin{equation}
	\label{eq:tauonthickcup}
\begin{tikzpicture}[anchorbase,yscale=-1,scale=.875]
\draw[ultra thick,green,->] (-.25,0) node[below=-12pt]{\scs $k$} to [out=90,in=180] (0,.5) 
	to [out=0,in=90] (.25,0);
\node at (.375,.75){\tiny$\wtn{+}\wt$};
\end{tikzpicture}
\xmapsto{\tau}
\begin{cases}
\begin{tikzpicture}[anchorbase,yscale=-1,scale=.875]
\draw[ultra thick,green,<-] (-.25,0) to [out=90,in=180] (0,.5) 
	to [out=0,in=90] (.25,0) node[below=-12pt]{\scs$k$};
\node at (.375,.75){\tiny$\wtn{-}\wt$};
\end{tikzpicture} 
& \text{if } \langle \alpha_{\grb}^{\vee} , \wt \rangle \equiv 0 \\
(-1)^{k\cdot a_{\grb} + \binom{k+1}{2}}
\begin{tikzpicture}[anchorbase,yscale=-1,scale=.875]
\draw[ultra thick,green,<-] (-.25,0) to [out=90,in=180] (0,.5) 
	to [out=0,in=90] (.25,0) node[below=-12pt]{\scs$k$};
\node at (.375,.75){\tiny$\wtn{-}\wt$};
\end{tikzpicture} 
& \text{if } \langle \alpha_{\grb}^{\vee} , \wt \rangle \equiv 1
\end{cases}
\, , \quad
\begin{tikzpicture}[anchorbase,yscale=-1,scale=.875]
\draw[ultra thick,green,<-] (-.25,0) node[below=-12pt]{\scs $k$} to [out=90,in=180] (0,.5) 
	to [out=0,in=90] (.25,0);
\node at (.375,.75){\tiny$\wtn{+}\wt$};
\end{tikzpicture}
\xmapsto{\tau}
\begin{cases}
\begin{tikzpicture}[anchorbase,yscale=-1,scale=.875]
\draw[ultra thick,green,->] (-.25,0) to [out=90,in=180] (0,.5) 
	to [out=0,in=90] (.25,0) node[below=-12pt]{\scs$k$};
\node at (.375,.75){\tiny$\wtn{-}\wt$};
\end{tikzpicture} 
& \text{if } \langle \alpha_{\grb}^{\vee} , \wt \rangle \equiv 0 \\
(-1)^{-k\cdot a_{\grb} + \binom{k+1}{2}}
\begin{tikzpicture}[anchorbase,yscale=-1,scale=.875]
\draw[ultra thick,green,->] (-.25,0) to [out=90,in=180] (0,.5) 
	to [out=0,in=90] (.25,0) node[below=-12pt]{\scs$k$};
\node at (.375,.75){\tiny$\wtn{-}\wt$};
\end{tikzpicture} 
& \text{if } \langle \alpha_{\grb}^{\vee} , \wt \rangle \equiv 1
\end{cases}
\end{equation}
\begin{equation}
	\label{eq:tauonthickcrossing}
\begin{tikzpicture}[anchorbase,scale=1]
\draw[ultra thick,green,->] (0,0) node[below=-2pt]{\scs$k$} to [out=90,in=270] (.5,1);
\draw[ultra thick,->] (.5,0) node[below=-2pt]{\scs$\ell$} to [out=90,in=270] (0,1);
\node at (.875,.75){\scs$\wtn{+}\wt$};
\end{tikzpicture}
\mapsto \begin{cases}
(-1)^{k\ell}
\begin{tikzpicture}[anchorbase,scale=1]
\draw[ultra thick,green,<-] (0,0) node[below=-2pt]{\scs$k$} to [out=90,in=270] (.5,1);
\draw[ultra thick,blue,<-] (.5,0) node[below=-2pt]{\scs$\ell$} to [out=90,in=270] (0,1);
\node at (.875,.75){\scs$\wtn{-}\wt$};
\end{tikzpicture} \, & \text { if } \bullet = \blb, \\
\begin{tikzpicture}[anchorbase,scale=1]
\draw[ultra thick,green,<-] (0,0) node[below=-2pt]{\scs$k$} to [out=90,in=270] (.5,1);
\draw[ultra thick,<-] (.5,0) node[below=-2pt]{\scs$\ell$} to [out=90,in=270] (0,1);
\node at (.875,.75){\scs$\wtn{-}\wt$};
\end{tikzpicture} \, & \text{ otherwise.}
\end{cases}
\end{equation}	
(Recall our color Convention \ref{conv:colors}.) 
Consequently, $\tau^2$ acts as the identity on all uni-colored thick diagrams. \qed
\end{cor}

Theorem \ref{thm:involution} also extends to thick calculus, 
after introducing the relevant $2$-categories. 
Note that for any weight $\wt$, 
we have $\EE_i^{(k)}\FF_i^{(k)}\one_{\wt} = \one_{\wt}\EE_i^{(k)}\FF_i^{(k)}\one_{\wt}$. 

\begin{defn}\label{D:cXX-2cat}
Fix a weight $\wt$. 
Let $\cXX_q(\glm) \one_\wt$ denote the full $2$-subcategory of $\cUU_q(\glm)$ generated 
by the $1$-endomorphisms $\EE_i^{(k)}\FF_i^{(k)}\one_{\wt}$ and $\FF_i^{(k)}\EE_i^{(k)}\one_{\wt}$, 
and their grading shifts. 
Let $\cXX_q(\glm)$ be the full $2$-subcategory containing all $\cXX_q(\glm) \one_\wt$ as $\wt$ varies.
\end{defn}

As before, each $\cXX_q(\glm) \one_{\wt}$ is a monoidal category 
and there is no interaction between $\cXX_q(\glm) \one_{\wt}$ and $\cXX_q(\glm) \one_{\wt'}$ 
for $\wt \ne \wt'$.

\begin{cor}
	\label{cor:thickinvolution}
Let $\tau$ be as in Theorem \ref{thm:involution}. 
The extension of $\tau$ to $\cUU_q(\glm)$ restricts to an involution on $\cXX_q(\glm)$. 
\end{cor}

\begin{proof}
Fix $\wt \in \Z^m$ and consider the monoidal category $\cXX_q(\glm) \one_{\wt}$, 
which is generated by $\EE_i^{(k)} \FF_i^{(k)} \one_{\wt}$ and $\FF_i^{(k)}\EE_i^{(k)} \one_{\wt}$. 
Note that if $a_i - a_{i+1} > 0$, 
then we can use the Sto\v{s}i\'{c} formula \eqref{eq:StosicEF} to write the identity morphism 
of $\EE_i^{(k)} \FF_i^{(k)} \one_{\wt}$ as a linear combination of morphisms 
that factor through the objects $\FF_i^{(\ell)} \EE_i^{(\ell)} \one_{\wt}$ for $0 \leq \ell \leq k$. 
(The objects being factored through cannot be rewritten in the same way.)
Similarly, if $a_i - a_{i+1} < 0$ then we can use \eqref{eq:StosicFE} 
to write the identity morphism of $\FF_i^{(k)} \EE_i^{(k)} \one_{\wt}$
as a linear combination of morphisms that factor through $\EE_i^{(\ell)} \FF_i^{(\ell)} \one_{\wt}$ for $0 \leq \ell \leq k$.
As observed in Corollary \ref{cor:ourthicktau}, 
$\tau^2$ acts as the identity on all uni-colored diagrams,
hence on tensor products of uni-colored diagrams. 
Since the diagrams in the Sto\v{s}i\'{c} formula are uni-colored,
we see that $\tau^2$ will equal the identity if and only if it 
acts as the identity the full subcategory generated 
by $\EE_i^{(k)} \FF_i^{(k)} \one_{\wt}$ for $a_i - a_{i+1} \leq 0$ 
and $\FF_i^{(k)} \EE_i^{(k)} \one_{\wt}$ for $a_i - a_{i+1} \geq 0$. 

Next, we observe that when $a_i - a_{i+1} \leq 0$ 
the identity morphism of $\EE_i^{(k)} \FF_i^{(k)} \one_{\wt}$
factors through the object $(\EE_i \FF_i \one_{\wt})^{\otimes k}$. 
This follows by observing that, in this case, 
$\EE_i^{(k)} \FF_i^{(k)} \one_{\wt}$ is a summand of $(\EE_i \FF_i \one_{\wt})^{\otimes k}$
(which can be confirmed using Theorem \ref{thm:KL}
by expressing $(e_i f_i \one_{\wt})^{k} \in \ZdU(\glm)$ in terms of the 
canonical basis for this $\sln[2]$ string).
Alternatively, this can be done explicitly, e.g.~when $a_i - a_{i+1} \leq 0$ 
and $k=3$, we have
\[
\begin{tikzpicture}[anchorbase,scale=.75]
\draw[ultra thick,green,->] (.5,1) to (.5,1.5) node[above=-2pt]{\scs$3$};
\draw[ultra thick,green] (2,1) to (2,1.5) node[above=-2pt]{\scs$3$};
	\draw[thick,green,directed=.5] (.5,-1) to [out=150,in=270] (0,0) to [out=90,in=210] 
		node[black,pos=.5]{$\bullet$} node[black,pos=.5,left=-1pt]{\scs$2$} (.5,1);
	\draw[thick,green,directed=.5] (.5,-1) to [out=90,in=270] (1,0) 
		to [out=90,in=270] node[black,pos=.75]{$\bullet$} (.5,1);
	\draw[thick,green,directed=.5] (.5,-1) to [out=30,in=270] (2,0) to [out=90,in=330] (.5,1);
	\draw[thick,green,rdirected=.5] (2,-1) to [out=150,in=270] (.5,0) to [out=90,in=210] (2,1);
	\draw[thick,green,rdirected=.5] (2,-1) to [out=90,in=270] 
		node[black,pos=.25]{$\bullet$} (1.5,0) to [out=90,in=270] (2,1);
	\draw[thick,green,rdirected=.5] (2,-1) to [out=30,in=270] 
		node[black,pos=.5]{$\bullet$} node[black,pos=.5,right=-1pt]{\scs$2$}
		(2.5,0) to [out=90,in=330] (2,1);
\draw[ultra thick,green] (.5,-1) to (.5,-1.5) node[below=-2pt]{\scs$3$};
\draw[ultra thick,green,->] (2,-1) to (2,-1.5) node[below=-2pt]{\scs$3$};
\node at (2.75,1){\small$\wt$};
\end{tikzpicture}
=
\begin{tikzpicture}[anchorbase,scale=.75]
\draw[ultra thick,green,->] (0,-1) to (0,1) node[above=-2pt]{\scs$3$};
\draw[ultra thick,green,->] (1,1) to (1,-1) node[below=-2pt]{\scs$3$};
\node at (1.375,.5){\small$\wt$};
\end{tikzpicture}
\]
and similar formulae (for all $k\geq 2$) can be deduced from \eqref{extendedreln1}.
Similarly, when $a_i - a_{i+1} \geq 0$ 
the identity morphism of $\FF_i^{(k)} \EE_i^{(k)} \one_{\wt}$
factors through the object $(\FF_i \EE_i \one_{\wt})^{\otimes k}$. 
Again, since $\tau^2$ is the identity on tensor products of uni-colored diagrams, 
we see that $\tau^2$ is the identity if and only if it acts as the identity on the full subcategory 
generated by $(\EE_i \FF_i \one_{\wt})^{\otimes k}$ for $a_i - a_{i+1} \leq 0$
and $(\FF_i \EE_i \one_{\wt})^{\otimes k}$ for $a_i - a_{i+1} \geq 0$.
This holds by Theorem \ref{thm:involution}.
\end{proof}

%
\section{Background on decompositions and equivariant categories} \label{sec:equiv}
%

There are two main goals in this section. 
First, we study direct sum decompositions of objects into indecomposables via composition pairings. 
Second, in Section \ref{subsec:equivariantization}, 
we recall the definition and elementary structure theory of equivariant categories. 
We combine these ideas to discuss the classification of
indecomposable objects in equivariant categories. 
This material is well-known to experts, and mostly adapted from the overview given in \cite{EliasFolding}.

Some basic results we will repeatedly quote in the sequel are the following. \begin{itemize}
\item Lemma \ref{lem:basisforV}, 
which gives a basis for certain multiplicity spaces in the presence of known direct sum decompositions.
\item Proposition \ref{prop:SummandViaEV}, 
which states that one can decompose objects in equivariant categories by computing 
eigenspaces of a certain involution $\sigast$ acting on multiplicity spaces.
\item Lemma \ref{L:Ben's-linear-algebra-lemma}, 
which gives an efficient technique for computing the action of $\sigast$ on multiplicity spaces.
\end{itemize}
Lemma \ref{L:Ben's-linear-algebra-lemma} did not previously appear in \cite{EliasFolding}.

We work with $\K$-linear categories over a commutative ring $\K$. 
Until \S \ref{subsec:equivariantizationmixed},
there are very few restrictions we need to place on the base ring $\K$ for the results above to hold, 
so long as one studies objects whose (degree zero) morphism spaces are well-understood as $\K$-modules. 
However, in \cite{EliasFolding} and elsewhere in the literature, 
such results are typically proven under the assumption that $\K$ is an algebraically closed field. 
Where it streamlines the exposition, we do assume that $\K$ is an algebraically closed field.
Supplementary material dealing with more general $\K$ can be found in Appendix \ref{SS:equivariantsupplement}.

\subsection{Composition pairings and endopositive objects} 
	\label{LIP}

We now discuss direct sum decompositions, and tools for studying them in various contexts. 
Throughout, we are motivated by the case of $\K$-linear categories when $\K$ is a field,
but we state results that hold over general base rings. 

\begin{defn}[{\cite[Definition 11.64]{soergelbook}}] \label{def:compositionpairing} 
Let $\K$ be a commutative ring and let $X$ and $Y$ be objects
in a graded additive $\K$-linear category $\Cat$. 
The \emph{na\"{i}ve composition pairing} of $Y$ at $X$ is the $\K$-bilinear pairing
\[ 
\beta_{X,Y} \colon \Hom(Y, X)\times \Hom(X, Y)\longrightarrow \End(X) 
\]
defined by
$\beta_{X,Y}(f,g) = f \circ g$.
\end{defn}

The na\"{i}ve composition pairing is valued in $\End(X)$ rather than in $\K$, 
but we will soon have cause to discuss $\K$-valued bilinear pairings as well.

\begin{defn} \label{defn:dualsets} 
Given sets of morphisms $\{p_1, \ldots, p_d\} \in \Hom(Y, X)$ 
and $\{\iota_1, \ldots, \iota_d\} \in \Hom(X, Y)$, 
we call them \emph{dual sets} with respect to the na\"{i}ve composition pairing $\beta_{X,Y}$ if
\begin{equation} \label{dualsets} \beta_{X,Y}(p_j,\iota_\ell) = \delta_{j\ell} \id_X \, . \end{equation}
Similarly, given $\K$-modules $V$ and $V'$ and a $\K$-bilinear pairing $\beta \colon V \times V' \to \K$, 
sets $\{p_1, \ldots, p_d\} \in V$ and $\{\iota_1, \ldots, \iota_d\} \in V'$ are called \emph{dual sets} if
\[ \beta(p_j, \iota_\ell) = \delta_{j \ell} \, . \]
In either context, 
the \emph{na\"ive rank} of the pairing is the maximal integer $d$ such that one can find dual sets of size $d$. 
\end{defn}

The following lemma is tautological, following from the definition of direct sums and summands.

\begin{lem} \label{lem:dualsets} 
The na\"{i}ve rank of the composition pairing is the multiplicity of $X$ as a direct summand of $Y$. \end{lem}

\begin{proof} 
Dual sets give (orthogonal) projection and inclusion maps for $d$ copies of $X$ appearing as summands within $Y$. 
\end{proof}

\begin{rem} \label{rem:findrankd} 
For a $\K$-valued bilinear pairing $\beta$, 
if one finds elements $\{p_1, \ldots, p_d\}$ and $\{\iota_1, \ldots, \iota_d\}$ for which the $d \times d$ matrix
with entries $\beta(p_j,\iota_\ell)$ has invertible determinant in $\K$, 
the na\"{i}ve rank of the pairing is at least $d$. 
This follows using standard techniques in linear algebra (e.g.~Cramer's rule) to find
$\K$-linear combinations $\{p_1', \ldots, p_d'\}$ of the original elements $\{p_j\}$ 
which are dual to $\{\iota_\ell\}$. \end{rem}

One drawback of the na\"{i}ve composition pairing is that it is valued in the $\K$-algebra $\End(X)$.  
When $\End(X)$ is a (graded) local ring with (homogeneous) 
maximal ideal $J(X)$, then $\End(X)/J(X)$ is a division algebra. 
In this case we can use techniques of linear algebra over division algebras.

\begin{defn}[{\cite[Definition 11.70]{soergelbook}}]\label{def:localcompositionpairing}
Let $\K$ be a commutative ring and let $X$ and $Y$ be objects
in a graded additive $\K$-linear category $\Cat$. 
Suppose that $\End(X)$ is a (graded) local ring with (homogeneous) maximal ideal $J(X)$.
The \emph{local composition pairing},
which we also denote $\beta_{X,Y}$, is defined in the same way as the na\"{i}ve composition pairing
except that one takes the image of $f \circ g$ 
in the division algebra $\End(X)/J(X)$.
\end{defn}

We point out the following stronger result in this case. 

\begin{lem}[{\cite[Corollary 11.71]{soergelbook}}] \label{lem:EMTWlemma}
Let $X$ have (graded) local endomorphism algebra. 
The (graded) rank\footnote{One can make sense of the rank of a bilinear pairing valued in a division ring.} 
of the local composition pairing $\beta_{X, Y}$
is equal to the (graded) multiplicity of $X$ as a summand of $Y$. \qed
\end{lem}

\begin{rem} 
This lemma says that we need not find true dual sets to find summands of $Y$, 
but need only find sets which are dual modulo the maximal ideal of $\End(X)$. 
This is analogous to classical results involving idempotent lifting 
modulo the Jacobson radical \cite[Proposition 11.69]{soergelbook}. 
Those results require the category to be linear over a complete ring $\K$; 
however, as shown in loc.~cit., 
these assumptions can be omitted when the idempotent is factored 
as a composition of inclusion and projection maps. 
\end{rem}

Rather than working modulo the maximal ideal (or when this is not necessarily possible), 
we can instead focus on situations where the composition pairing is valued in $\K$ rather than in $\End(X)$. 
We thus introduce the following notion.

\begin{defn} 
Let $\K$ be a commutative ring and $\Cat$ be a graded additive $\K$-linear category. 
An object $X \in \Cat$ will be called \emph{endopositive} if the graded $\K$-algebra $\End(X)$ 
is supported in degrees $\ge 0$ and $\End^0(X) = \K \cdot \id_X$. \end{defn}

If $\K$ is a field, then any endopositive object $X$ is indecomposable in $\Cat$, 
and even in $\Kar(\Cat)$, since the endomorphism ring of $X$ is graded local. 
In general, $\K$ need not be a domain, or local, so $\End(X)$ might have nontrivial idempotents. 
The reader should view endopositive objects as a stand-in for indecomposable objects 
when working in $\K$-linear categories over general commutative rings. 
In particular, endopositive objects become indecomposable after base change to a field.

For any endopositive object, computing composition pairings one degree at a time gives 
the following $\K$-valued pairing.

\begin{defn} 
Let $\Cat$ be a graded additive $\K$-linear category 
over a commutative ring $\K$ and let $X$ be an endopositive object. 
The \emph{($\K$-valued) graded composition pairing of $Y$ at $\qsh^k X$} 
is the following restriction of the na\"{i}ve composition pairing:
\begin{equation}
\beta^k_{X, Y} \colon \Hom^{-k}(Y, X)\times \Hom^{k}(X, Y) 
	\longrightarrow \End^0(X) = \K \cdot \id_{X} \, .
\end{equation} \end{defn}

\begin{cor} \label{cor:gLCP} 
Let $X$ be endopositive. 
The na\"{i}ve rank of $\beta_{X, Y}^k$ is equal to the multiplicity of $\qsh^k X$ as a summand of $Y$. \qed \end{cor}

\begin{proof} This follows from Lemma \ref{lem:dualsets}. \end{proof}

\begin{rem} \label{rem:mixedInd}
If $X$ is endopositive and $\K$ is a field, then $\End(X)$ is graded local, 
the graded Jacobson radical $J(X)$ is equal to $\End^{>0}(X)$, 
and the inclusion of degree zero endomorphisms induces an isomorphism 
$\K\cdot \id_{X}= \End^0(X)\cong \End(X)/J(X)$. 
In particular, $\beta^k_{X,Y}$ is a graded slice of the local composition pairing 
as in Definition \ref{def:localcompositionpairing}. \end{rem}

Informally, one can view the na\"{i}ve rank of $\beta_{X,Y}^k$ as 
the ``number of independent inclusion maps'' from $\qsh^k X$ to $Y$; 
however, in general there is no canonical subspace of $\Hom(X,Y)$ 
of this dimension spanned by such inclusion maps. 
We now define a canonical quotient of $\Hom(X,Y)$ 
(having this dimension, when $\K$ is a field), 
which can be viewed as the analogue of a
multiplicity space.

\begin{defn}
Let the right and left kernels of $\beta^k_{X, Y}$ be the $\K$-linear subspaces
\begin{align*}
\rker(\beta^k_{X, Y}) 
&:= \{ g \in \Hom^k(X, Y) \mid \beta^k_{X, Y}(f, g)= 0 
	\text{ for all } f\in \Hom^{-k}(Y, X) \} \\
\lker(\beta^k_{X, Y}) 
&:= \{ f \in \Hom^{-k}(Y, X) \mid \beta^k_{X, Y}(f, g)= 0 
	\text{ for all } g\in \Hom^{k}(X, Y) \}
\end{align*}
and set
\[
V^k(X, Y) := \Hom^k(X, Y) \big/ \rker(\beta^k_{X, Y}) \, , \quad
V^{-k}(Y, X) := \Hom^{-k}(Y, X) \big/ \lker(\beta^k_{X, Y}) \, .
\]
The induced pairing
\begin{equation}
\LIP^k_{X, Y} \colon V^{-k}(Y, X)\times V^k(X, Y) \to \K
\end{equation}
is the \emph{non-degenerate graded composition pairing}.
\end{defn}

\begin{rem} 
The notation $V^k(X,Y)$ and $V^{-k}(Y,X)$ does not treat $X$ and $Y$ symmetrically, 
so one must know from context that we consider the composition pairing of $Y$ at $X$, and not vice versa. 
In practice, this distinction will be obvious, as $X$ will be the object which is endopositive. \end{rem}

The following is clear.

\begin{lem} \label{lem:gLIP} 
Let $X$ be endopositive. 
The na\"{i}ve rank of the graded composition pairing (of $Y$ at $\qsh^k X$)
equals the na\"{i}ve rank of the non-degenerate graded composition pairing. 
If $\K$ is a field, this rank is equal to $\dim V^k(X,Y)$. \qed 
\end{lem}

We also make the following observations.

\begin{lem} \label{lem:Vadditive}
Let $X$ be an endopositive object in $\Cat$, and $Y_1, Y_2$ be arbitrary objects. 
The isomorphism $\Hom(X,Y_1 \oplus Y_2) \cong \Hom(X,Y_1) \oplus \Hom(X,Y_2)$ induces an isomorphism 
$V^k(X,Y_1 \oplus Y_2) \cong V^k(X,Y_1) \oplus V^k(X,Y_2)$ for each $k \in \Z$. 
Similarly, $V^{-k}(Y_1 \oplus Y_2,X) \cong V^{-k}(Y_1,X) \oplus V^{-k}(Y_2,X)$. \end{lem}

\begin{proof} It is straightforward to verify that the natural isomorphism of $\Hom$-spaces induces an isomorphism
	\[ \lker(\beta^k_{X,Y_1 \oplus Y_2}) \cong \lker(\beta^k_{X,Y_1}) \oplus \lker(\beta^k_{X,Y_2}). \]
This implies the first statement, and the second is similar. \end{proof}

\begin{nota} For endopositive $X$, 
let $V(X,Y) := \bigoplus_k V^k(X,Y)$ and $V(Y,X) := \bigoplus_k V^{-k}(Y,X)$. \end{nota}

\begin{cor} \label{cor:basisforVwhensumofX} 
Let $X$ be endopositive.
If $Y \cong \bigoplus_{j=1}^d \qsh^{k_j} X$ for some shifts $k_j \in \Z$,
then $V(Y,X)$ (resp.~$V(X,Y)$) is a free graded $\K$-module 
with graded rank $\sum_{j=1}^d q^{k_j}$ (resp.~$\sum_{j=1}^d q^{-k_j}$). 
If one chooses a decomposition as above with projection maps $\{p_j\}$ and inclusion maps $\{\iota_j\}$, 
then $\{p_j\} \subset \Hom(Y,X)$ descends to a basis for $V(Y,X)$ 
and $\{\iota_j\} \subset \Hom(X,Y)$ descends to a basis for $V(X,Y)$. \end{cor}

\begin{proof} 
It is immediate to verify when $Y = X$ that $V(Y,X)$ and $V(X,Y)$ are free of rank $1$ over $\K$, 
spanned by the identity map. 
The general result for $Y = \bigoplus_j \qsh^{k_j} X$ then follows from Lemma \ref{lem:Vadditive}. 
Note that under the isomorphism 
\[ 
\Hom(X,Y) \cong \Hom(X,\bigoplus_j \qsh^{k_j} X) \cong \bigoplus_j \qsh^{k_j} \End(X) 
\] 
the identity maps of $X$ on the right-hand side become the chosen inclusion maps on the left-hand side. \end{proof}

\begin{example}
When pondering Corollary \ref{cor:basisforVwhensumofX}, 
it is important to realize that the space $\Hom(X,Y)$ is typically much larger than $V(X,Y)$. 
Suppose that $X$ is endopositive and $Y \cong \qsh X \oplus \qsh^{-1} X$.
Then, $\Hom^{-1}(X,Y)$ is free over $\K$ of rank $1$ spanned by $\iota_{-1}$
and is isomorphic to $V^{-1}(X,Y)$. 
However, $\Hom^1(X,Y)$ is spanned by $\iota_1$ together with $\iota_{-1} \circ f$ for $f \in \End^2(X)$. 
Since $X$ is endopositive, any term of the form $\iota_{-1} \circ f$ lies in $\rker(\beta^1_{X,Y})$.
Thus, as guaranteed by Corollary \ref{cor:basisforVwhensumofX}, 
we see that $V^{1}(X,Y)$ is free of rank $1$, spanned by the image of $\iota_1$.
\end{example}

This observation will be significant later when we consider equivariance:
an automorphism of $\Cat$ will induce an automorphism of $V^k(X,Y)$, 
but need not preserve the span elements in $\Hom^k(X,Y)$ that descend to a basis. 
There are a number of tools for studying the 
action of an automorphism on multiplicity spaces.
The first is the following, which allows us to work directly with $V(X,Y)$
by identify elements in the left and right kernels.

\begin{lem} \label{lem:kernelisjacobson}
If $X$ is endopositive, then
\[
\rker(X,-) := \bigoplus_k \rker(\beta^k_{X, -})
\quad 
\text{and}
\quad
\lker(-,X) := \bigoplus_k \lker(\beta^k_{X, -})
\]
are left and right ideals in $\Cat$, respectively.
Further, any positive degree endomorphism in $\End(X)$ is in $\rker(X,-)\cap\lker(-,X)$.
\end{lem}

\begin{proof} 
The statement about ideals is true more generally for any (endopositive) object $X$.
To see that $\rker(X,-)$ is a left ideal, we must show that given $g \in \rker(X,Y) \subset \Hom(X,Y)$ 
and any morphism $f \colon Y \to Z$, the composition $f \circ g$ is in $\rker(X,Z)$. 
Clearly, it suffices to consider homogenous $g$ and $f$.
Thus, the claim is that if $g \in \rker(\beta^k_{X, Y})$ 
and $f \in \Hom^{\ell}(Y,Z)$, then $f \circ g \in \rker(\beta^{k+\ell}_{X, Y})$, 
which is immediate from the definition of $\rker(\beta^k_{X, Y})$.
The argument for $\lker(-,X)$ is analogous.
Finally, if $k > 0$ then $\Hom^{-k}(X,X)$ is zero, 
which implies the final result.
\end{proof}

If $(\Cat,\otimes,\one)$ is a monoidal category, 
we can identify further elements in the kernel of $\beta^k_{X,Y}$
using the actions of the endomorphism algebra $\End(\one)$ 
on homomorphism spaces in $\Cat$ via tensor product on the left and right.

\begin{lem}
	\label{lem:EndKer}
Let $X$ be an endopositive object in a monoidal category $\Cat$.
For $\ell > 0$ and any object $Y$, the left or right action by $\End^{\ell}(\one)$ 
sends $\Hom^{k-\ell}(X, Y)$ to the right kernel of $\beta^{k}_{X, Y}$. 
\end{lem}

\begin{proof}
This follows immediately from Lemma \ref{lem:kernelisjacobson}, 
since 
\[ \id_{X} \cdot \End^{\ell}(\one) \subset \End^{\ell}(X) 
\supset \End^{\ell}(\one) \cdot \id_{X} \, .\qedhere \]
\end{proof}

Finally, we record a method for working
with $V(X,Y)$ indirectly.

\begin{rem} \label{rem:howtodealwithVk}
Let $X$ be endopositive. 
Suppose we are given maps $\{\iota_1, \ldots, \iota_d\} \subset \Hom^k(X,Y)$ which are ``candidate inclusions'' 
and maps $\{p_1, \ldots, p_d\} \subset \Hom^{-k}(Y,X)$ which are ``candidate projections.''
If we compute the pairing matrix $\LIP^k_{X,Y}(p_j,\iota_\ell)$ and the determinant is invertible in $\K$, 
then these maps descend to linearly independent sets in $V^k(X,Y)$ and $V^{-k}(Y,X)$, respectively. 
If we know somehow that $\{\iota_1, \ldots, \iota_d\}$ descend to a spanning set of $V^k(X,Y)$ 
(e.g.~using Corollary \ref{cor:basisforVwhensumofX}, or because we work over a field and $d = \dim V^k(X,Y)$), 
then we obtain bases of $V^k(X,Y)$ and $V^{-k}(Y,X)$. 
In this case, an arbitrary morphism $f \in \Hom^k(X,Y)$ need not be in the span of $\{\iota_1, \ldots, \iota_d\}$, 
but its image in $V^k(X,Y)$ is in (the image of) this span. 
By non-degeneracy, $f = \sum_{j=1}^d a_j \iota_j$ in $V^k(X,Y)$ 
if and only if they have the same pairing against all $p_j$. 
Thus, one can compute the image of $f$ in $V^k(X,Y)$ 
by computing all the pairings $\beta^k(p_j,f)$ 
and using elementary linear algebra.
\end{rem}

\subsection{Endopositive families of objects}
	\label{ss:PGC}
	
We now focus on categories with distinguished collections of endopositive objects.

\begin{defn} \label{defn:endoposfamily} 
Let $\K$ be a commutative ring and 
let $\BC$ be a set indexing a family of distinguished objects $\{ X_b \}_{b \in \BC}$ 
in a graded additive $\K$-linear category $\Cat$.
We call $\{X_b\}_{b \in \BC}$ an \emph{endopositive family} of objects 
provided each $X_b$ is endopositive and $\Hom(X_b,X_{b'})$ 
is concentrated in strictly positive degrees when $b \ne b'$. 
We let $\CatBC$ denote the full subcategory of $\Cat$ whose objects are 
isomorphic to direct sums of shifts of objects in the corresponding endopositive family. \end{defn}

Given an endopositive family,
if $b \ne b'$, then $X_b$ is not isomorphic to any grading shift of $X_{b'}$, 
as such an isomorphism would require a morphism of non-positive degree 
(either $X_b \to X_{b'}$ or $X_{b'} \to X_b$). 
Similarly, $X_b$ is not isomorphic to a nonzero grading shift of itself.
Note also that the condition of being an endopositive family is preserved by base change.

We informally summarize the condition that $\{ X_b\}$ is an endopositive family with the equation:
\begin{equation} \label{homdim} 
\grdim_{\ak} \Hom(X_{b}, X_{b'})\in \delta_{b, b'} + q \Z_{\ge 0}[[q]] \, .
\end{equation}
This uses several abuses of notation! 
We have not assumed that $\Hom(X_b,X_{b'})$ is a free module over $\K$, except in degree $0$ when $b = b'$. 
Also, $\K$ is not necessarily a field, yet for simplicity we choose to write $\grdim_{\K}$ 
to denote the graded rank of a $\K$-module. 
Nonetheless, \eqref{homdim} encapsulates the idea of an endopositive family, 
and is mathematically accurate whenever $\K$ is a field.

\begin{example} \label{ex:CQG2}
For any commutative ring $\K$, 
the monoidal category $\one_{\wtn} \cXX_q(\gln[2]) \one_{\wtn}$ 
has an endopositive family $\{ \FF^{(k)} \EE^{(k)} \one_{\wtn} \}_{k \geq 0}$. 
Note that $\EE^{(k)} \FF^{(k)} \one_{\wtn} \cong \FF^{(k)} \EE^{(k)} \one_{\wtn}$, 
by \eqref{eq:EF} and \eqref{eq:FE}. 
See \cite[Proposition 5.15]{KLMS} for a description of a (positively graded) basis for morphism spaces 
(the summary in \cite[Example 2.21]{ELau} may also be helpful). 
In this example, 
every object in $\one_{\wtn} \cXX_q(\gln[2]) \one_{\wtn}$ 
is a direct sum of shifts of distinguished objects. \end{example}

\begin{rem} \label{rem:karoubianoverfield}
Let $\K$ be a field.
As in Remark \ref{rem:mixedInd}, 
the endopositive objects $X_b$ are indecomposable with graded local endomorphism rings. 
Moreover, since every object in $\CatBC$ 
admits a decomposition into objects with graded local endomorphism rings, 
one concludes that the category $\CatBC$ is Karoubian and graded Krull--Schmidt, 
see e.g.~\cite[Theorem 11.50]{soergelbook}. 
The set $\BC$ therefore indexes the indecomposable objects of $\CatBC$, 
up to isomorphism and grading shift. \end{rem}

When $\K$ is a field, Corollary \ref{cor:gLCP} and Lemma \ref{lem:gLIP} 
show that if $\{X_b\}_{b \in \BC}$ is an endopositive familiy and $Y \in \CatBC$, 
then $Y\cong \bigoplus_{\substack{k\in \Z \\ b \in \BC}} \qsh^k X_b^{\oplus d_{b,k}}$ 
for $d_{b, k}= \dim_{\K} V^k(X_b, Y)$.
More generally, we have the following.

\begin{lem} \label{lem:basisforV}
Let $\{X_b\}_{b \in \BC}$ be an endopositive family in $\Cat$, over any commutative ring $\K$, 
and let $Y$ be an object of $\CatBC$. 
Fix any decomposition
\[Y \cong \bigoplus_{b \in \BC} \bigoplus_{j=1}^{d_b} \qsh^{k_{b,j}} X_b \, , \]
i.e.~fix inclusion maps $\{p_{b,j}\}_{j=1}^{d_b}$ and $\{\iota_{b,j}\}_{j=1}^{d_b}$ for the decomposition above. 
Then, for each $b \in \BC$, 
the space $V(Y,X_b)$ is free over $\K$ with graded dimension $\sum_{j=1}^{d_b} q^{k_{b,j}}$ 
and the elements $\{p_{b,j}\}$ descend to a basis. 
Similarly, $V(X_b,Y)$ is free with graded dimension $\sum_{j=1}^{d_b} q^{-k_{b,j}}$ 
and basis $\{\iota_{b,j}\}$. \end{lem}

\begin{proof} 
This follows from Lemma \ref{lem:Vadditive}, similar to the proof of Corollary \ref{cor:basisforVwhensumofX}, 
after noting that $V(X_b,X_{b'})$ is free over $\K$ of graded rank $\delta_{b,b'}$, 
spanned by the identity map of $X_b$.
\end{proof}

Finally, we record a supplement to Lemma \ref{lem:kernelisjacobson}
in the presence of an endopositive family.
When paired with that result, this will guarantee that,
when examining the graded composition pairing at $X_b$,
any morphism which factors through $X_{b'}$ for $b' \ne b$ is automatically in the kernel.

\begin{lem} \label{lem:btob'inker}
Let If $\{X_b\}_{b \in \BC}$ be an endopositive family.
If $b \ne b' \in \BC$, then any morphism in $\Hom(X_b,X_{b'})$ 
is in $\rker(X_b,-) \cap \lker(-,X_{b'})$. 
\end{lem}
\begin{proof}
If $b \ne b'$, \eqref{homdim} implies that for each $k \in \Z$, 
at least one of $\Hom^{-k}(X_{b'}, X_b)$ or $\Hom^{k}(X_b, X_{b'})$ is identically zero. 
\end{proof}

\subsection{Mixed categories}
	\label{ss:mixed}

In the literature, endopositive families are typically studied in the context of mixed categories, 
which also possess a duality functor.

\begin{defn}[c.f.~{\cite[Definition 3.9]{EliasFolding}}] \label{def:mixed} 
Let $\K$ be a commutative ring, and $\Cat$ be a graded additive $\K$-linear category. 
We call $\Cat$ a \emph{positively graded category} if 
\begin{itemize} 
	\item all morphism spaces in $\Cat$ are free $\mathbb{Z}$-graded 
		$\K$-modules of finite rank in each degree, and
	\item $\Cat$ is equipped with an endopositive family $\{ X_b \}_{b \in \BC}$ 
		such that $\Cat = \CatBC$.
\end{itemize}
Suppose further that $\Cat$ is equipped with a $\K$-linear 
functor $\mathbb{D} \colon \Cat \to \Cat^{\mathrm{op}}$ such that 
$\mathbb{D}^2 = \one_\Cat$ and $\mathbb{D}(\qsh^k X) = \qsh^{-k} \mathbb{D}(X)$. 
Then, we call such a (positively graded) $\Cat$ \emph{mixed} provided 
for each $b \in \BC$, there is some $b^* \in \BC$ such that $\mathbb{D}(X_b)\cong X_{b^*}$. 
If $b = b^*$ for all $b \in \BC$, then we say $\Cat$ is \emph{self-dual mixed}. \end{defn}

As usual, we extend these notions to $2$-categories via their $\Hom$-categories.

Note that Remark \ref{rem:karoubianoverfield} implies that 
a positively graded category over a field $\K$ is Karoubian and graded Krull--Schmidt, 
and $\BC$ indexes the indecomposable objects up to isomorphism and grading shift.
	
\begin{example}
	\label{ex:CQGmixed}
When $\K$ is a field of characteristic zero, 
the $2$-category $\Kar(\UU_q(\glm))$ is self-dual mixed, by \cite{Web4}. 
\end{example}

\begin{rem} With apologies to Webster, the definition of mixed in Definition \ref{def:mixed}
(which is adapted from \cite[Definition 3.9]{EliasFolding})
differs from the definition of mixed in \cite[Definitions 1.2 and 1.11]{Web4}. 
The primary feature of both definitions is \eqref{homdim}.  
Webster's definition assumes the category is self-dual.
When $\K$ is a field, our definition coincides in the self-dual case with the special case of Webster's 
where the orthodox and canonical bases coincide; see \cite[Corollary 1.14]{Web4}.
\end{rem}

We briefly note several advantages of a self-dual mixed category. 
The first is that it gives a reason to prefer $X_b$ over $\qsh^k X_b$, 
as only when $k=0$ is $\qsh^k X_b$ self-dual.
The second is that any automorphism of the category which commutes with shifts and duality 
must preserve the self-dual distinguished objects $X_b$. 
The third (which we do not use in this paper) is that if $X_b$ and $Y$ are both self-dual objects, 
then one can identify $\Hom(X_b,Y)$ with $\Hom(Y,X_b)$ using $\mathbb{D}$. 
This transforms the composition pairing into a bilinear form on $\Hom(X_b,Y)$.

\subsection{Equivariant categories} 
	\label{subsec:equivariantization}

In the theory of quantum groups and Hecke algebras, 
there are certain algebras with fixed bases such that the structure constants 
for multiplication have the form $P-Q$ for $P, Q\in \Z_{\ge 0}[q^{\pm}]$. 
As an immediately relevant example, recall \eqref{eq:unknots} which stated that
\begin{equation}
P_{\sln[4]}\left( 
\begin{tikzpicture}[scale =.5, anchorbase]
	\draw[very thick] (0,0) node[right=5pt,yshift=5pt]{\scs$\Lambda^2$} circle (.5);
\end{tikzpicture} \right) 
= [5]+[1]
\quad \text{and} \quad
P_{\son[5]}\left( 
\begin{tikzpicture}[scale =.5, anchorbase]
	\draw[very thick] (0,0) node[right=5pt,yshift=5pt]{\scs$S$} circle (.5);
\end{tikzpicture} \right) 
= [5]-[1] \, .
\end{equation}
The latter value appears as a structure constant in $\End_{U_q(\son[5])}(S\otimes S)$, 
when computing $\xx^{(2)} \cdot \xx^{(2)}$, 
so centralizer algebras for type $B$ quantum groups 
are examples of such algebras.

Lusztig intuited that such structure constants may be realized
by the trace of an involution acting on a graded vector space, 
making $P$ (resp. $Q$) the graded dimension of the $+1$ (resp. $-1$) eigenspace. 
See e.g.~\cite[Remark 14.4.14]{Lus4}.
We now briefly summarize the general exposition in \cite{EliasFolding} 
concerning how such algebras may be realized as the weighted Grothendieck rings 
of certain $G$-equivariant categories. 
We then give precise definitions only for the group $G=\Z/2$, 
since that is the case that concerns us in the present work.
Throughout this section, $\K$ can be an arbitrary commutative ring, 
although some of the definitions will be trivial when $\K$ has characteristic two.

Let $\Cat$ be a (graded) additive $\K$-linear category with a strict action of a group $G$. 
We can then form the \emph{$G$-equivariantization} of $\Cat$, denoted $\Cat^G$. 
The objects of $\Cat^G$ are pairs $(X,\varphi)$, 
where $X$ is an object of $\Cat$ such that $g \cdot X \cong X$ for all $g \in G$, 
and $\varphi$ is a family of isomorphisms $\varphi_g \colon X \to g \cdot X$ 
that satisfy a natural compatibility constraint. 

The category $\Cat^G$ has a strict action of the dual group $G^* := \Hom(G,\K^{\times})$. 
The action of a homomorphism $\xi \colon G \to \K^\times$ will rescale $\varphi_g$ by $\xi(g)$. 
That is,
\[ 
\xi \cdot (X,\varphi) = (X,\xi \cdot \varphi) 
\, , \quad  (\xi \cdot \varphi)_g = \xi(g) \varphi_g \, . 
\]

Given an element $g \in G$, 
we can form the \emph{$g$-weighted Grothendieck group} of $\Cat^G$, 
denoted $K_{0}^{g}(\Cat^G)$. 
This is the quotient of the usual (additive) Grothendieck group, 
base changed to $\K$ (or any subring containing the images of all $\xi \in G^*$), 
by the relation $[\xi \cdot (X,\varphi)] = \xi(g) [(X,\varphi)]$ for all $\xi \in G^*$. 
For more details see \cite[Section 3.1]{EliasFolding}.

In the case that $\Cat$ is monoidal and/or is equipped with a duality functor 
and the $G$-action preserves these structures, 
$\Cat^G$ inherits these structures and the action of $G^*$ respects them. 
In this way, $K_{0}^{g}(\Cat^G)$ inherits the structure of a ring and/or inherits the bar involution 
(the anti-linear action of duality on the Grothendieck group).

We now restrict our attention to $G = \Z/2$.

\begin{defn}[c.f.~{\cite[Definition 3.1]{EliasFolding}}]
A \emph{strict action} of $\Z/2$ on an additive $\ak$-linear category $\Cat$ is 
an additive $\ak$-linear autoequivalence $\sigma \colon \Cat \rightarrow \Cat$ 
and a natural isomorphism $\mathsf{s} \colon \sigma \circ \sigma\xrightarrow{\cong} \one_{\Cat}$ such 
that\footnote{Here, $\otimes$ denotes horizontal composition of natural transformations; 
this is the tensor product on the monoidal category of endofunctors of $\Cat$.} 
$\mathsf{s} \otimes \id_{\sigma} = \id_{\sigma} \otimes \mathsf{s}$ 
as natural transformations $\sigma \circ \sigma \circ \sigma \to \sigma$.
\end{defn}

\begin{rem}
From the group $\Z/2$ (here written multiplicatively), 
one obtains a monoidal category $\Ztwogroupoid$, 
often called the $2$-groupoid of $\Z/2$. 
It has objects $\pm 1$, tensor product given by the group operation, 
and the only morphisms are identities. 
Equivalent to the definition above, 
a \emph{strict action} of $\Z/2$ on $\Cat$ 
is a strong\footnote{A monoidal functor always comes with data of coherence maps;
a strong monoidal functor refers to when these maps are isomorphisms 
and a strict monoidal functor refers to when these isomorphisms are equalities.} 
(but not necessarily strict) monoidal functor 
$F \colon \Ztwogroupoid \rightarrow \End(\Cat)$ for which $F(+1)=\one_{\Cat}$ 
and $F(-1) =: \sigma$.
\end{rem}

\begin{rem}
	\label{rem:involution}
We call $\sigma$ an \emph{involution} if $\sigma \circ \sigma = \one_{\Cat}$ is an equality of functors. 
In this case, we can use $\mathsf{s} =\id_{\one_{\Cat}}$ to obtain a strict action of $\Z/2$ on $\Cat$. 
\end{rem}

\textbf{From here on, 
we assume that $\sigma$ is an involution and that} $\mathsf{s} = \id_{\one_{\Cat}}$.

\begin{rem}
Consider an involution $\sigma$ on a $2$-category in the sense of Definition \ref{D:2-involution}. 
If $\wt$ is an object fixed by $\sigma$, then $\sigma$ induces a monoidal functor 
on the endomorphism category of $\wt$. 
This functor will be an involution in the sense of Remark \ref{rem:involution}. 
\end{rem}

\begin{defn}[{\cite[Definition 3.5]{EliasFolding}}]\label{def:equivariant}
Let $\Cat$ be a category and let $\sigma \colon \Cat \rightarrow \Cat$ be an involution. 
An \emph{equivariant object} is a pair $(X, \varphi_X \colon X \xrightarrow{\cong} \sigma X)$ 
such that $\sigma(\varphi_X)\circ \varphi_X = \id_X$. 
A \emph{morphism of equivariant objects} from $(X, \varphi_X)$ to $(Y, \varphi_Y)$ is a morphism 
$f\in \Hom(X,Y)$ such that $\varphi_Y\circ f= \sigma(f)\circ \varphi_X$. 
Equivariant objects and morphisms of equivariant objects form the \emph{equivariant category} $\ECat$.
\end{defn}

Applying $\sigma$ to the equation $\sigma(\varphi_X) \circ \varphi_X = \id_X$,
we see that $\varphi_X\circ \sigma(\varphi_X) = \id_{\sigma(X)}$. 
Thus, the condition determining an equivariant object 
can be rewritten simply as $\sigma(\varphi_X) = \varphi_X^{-1}$.
We will refer to the choice of isomorphism $\varphi_X \colon X \to \sigma(X)$ 
satisfying this condition as an \emph{equivariant structure} on $X$. 
We say that $X$ is \emph{equivariantizable} if some equivariant structure exists.

We now establish a perspective on the $\Hom$-spaces in $\ECat$ that will prove useful.

\begin{prop}
	\label{prop:sigast}
Let $\sigma$ be an involution of a $\K$-linear category $\Cat$ and let $(X, \varphi_X)$ and $(Y, \varphi_Y)$ be objects in $\ECat$. 
For $f\in \Hom(X, Y)$, the formula
\begin{equation}
	\label{eq:sigastdef}
\sigast f = \sigma(\varphi_Y)\circ \sigma(f)\circ \varphi_X
\end{equation}
determines a $\K$-linear action of the group $\langle \sigast\rangle \cong \Z/2$ on $\Hom(X, Y)$. 
This action respects composition, i.e.~
if further $(Z, \varphi_Z)$ is an object in $\ECat$ and $g \in \Hom(Y,Z)$, 
then 
\begin{equation}
	\label{eq:sigasthomo}
\sigast (g\circ f) = (\sigast g) \circ (\sigast f) \, .
\end{equation}
\end{prop}

\begin{proof}
We have that $\sigma(\varphi_X) = \varphi_X^{-1}$ and $\sigma(\varphi_Y) = \varphi_Y^{-1}$.
We compute
\begin{align*}
\sigast(\sigast f) = \sigast(\sigma(\varphi_Y)\circ \sigma(f)\circ \varphi_X) 
			&= \sigma(\varphi_Y) \circ \sigma \big( \sigma(\varphi_Y)\circ \sigma(f)\circ \varphi_X \big) \circ \varphi_X \\
			&= \varphi_Y^{-1} \circ \varphi_Y \circ f \circ \varphi_X^{-1} \circ \varphi_X = f \, ,
\end{align*}
which shows that $\sigast$ generates a $\Z/2$-action. 
Since composition and $\sigma$ are $\K$-linear, $\sigast f$ is linear in $f$. 
We leave the reader to check compatibility with composition.
\end{proof}

We \textbf{strongly emphasize} that the definition of $\sigast$ on $\Hom(X,Y)$ 
depends on the choice of isomorphisms $\varphi_X$ and $\varphi_Y$,
i.e.~on the choice of equivariant objects $(X, \varphi_X)$ and $(Y, \varphi_Y)$.

\begin{cor}
	\label{cor:HomAsInv}
Let $(X, \varphi_X)$ and $(Y, \varphi_Y)$ be equivariant objects and
let $\Hom(X, Y)^{\sigma}$ denote the $(\sigast)$-invariants in $\Hom(X,Y)$.
Then, $\Hom_{\ECat}\big( (X, \varphi_X), (Y, \varphi_Y) \big) = \Hom(X, Y)^{\sigma}$. 
\end{cor}

\begin{proof}
Since $\sigma(\varphi_Y)= \varphi_Y^{-1}$, 
we see that $f= \sigast f$ if and only if $\varphi_Y \circ f = \sigma(f)\circ \varphi_X$. 
\end{proof}

\begin{rem}
	\label{rem:indECat}
Corollary \ref{cor:HomAsInv} immediately shows that 
if $(X,\varphi_X)$ is decomposable in $\ECat$, 
then $X \in \Cat$ is decomposable. 
Indeed, any idempotents in $\End_{\ECat}\big((X,\varphi)\big) = \End(X)^\sigma$ 
giving a direct sum decomposition in $\ECat$ also give such a decomposition in $\Cat$.
The converse of this statement need not be true; see Proposition \ref{prop:indECat} below.
Note also that if $X$ is endopositive, 
then Corollary \ref{cor:HomAsInv} implies that $(X,\varphi_X)$ is endopositive.
\end{rem}

Next, we review the weighted notion of Grothendieck group that can be applied to the category $\ECat$.
If $(X, \varphi_X)$ is an object in $\ECat$, 
then $\id_X = \sigma(\varphi_X) \circ \varphi_X = \sigma(-\varphi_X) \circ (-\varphi_X)$.
Thus, $(X, -\varphi_X)$ is an object of $\ECat$ as well. 
For $f \in \Hom(X,Y)$, it is easy to verify that
\[
f \in \Hom_{\ECat}\big( (X, \varphi_X) , (Y, \varphi_Y) \big) 
\iff f \in \Hom_{\ECat}\big( (X, -\varphi_X) , (Y, -\varphi_Y) \big) \, .
\]
This motivates the following.

\begin{defn}[{\cite[Definition 3.6]{EliasFolding}}]
Let $\sgn \colon \ECat \rightarrow \ECat$ be the involutive functor given on 
objects by $\sgn \big( (X, \varphi_X) \big) = (X, -\varphi_X)$ 
and on morphisms by $\sgn(f)= f$. 
\end{defn}

In the context of the discussion at the start of this section,
one should view $\sgn$ as the generator of the dual group $(\Z/2)^*$.
We now introduce the weighted Grothendieck group of $\ECat$ 
as a quotient of the usual additive Grothendieck group $\Kzero{\ECat}$.
Classes in the latter are denoted $[(X,\varphi_X)]$.

\begin{defn}[{\cite[Definition 3.7]{EliasFolding}}]
	\label{def:weightedG}
The $\sigma$-\emph{weighted Grothendieck group} of $\ECat$, 
denoted $\wKzero{\ECat}$, 
is the quotient of the Grothendieck group $\Kzero{\ECat} \ot_{\Z} \Z[\frac{1}{2}]$
by the relation 
\[
[(X,-\varphi_X)] \equiv -[(X, \varphi_X)] \, .
\] 
\end{defn}

For an object $(X,\varphi_X) \in \ECat$, we continue to write $[(X,\varphi_X)]$ 
for its class in the ordinary Grothendieck group $\Kzero{\ECat}$
and denote by $[(X, \varphi_X)]_{\sigma}$ the image of this class in $\wKzero{\ECat}$.

\begin{rem}
We change the base of the weighted Grothendieck group 
from $\Z$ to $\Z[\frac{1}{2}]$ to eliminate $2$-torsion; 
as a consequence, if $(X,\varphi_X) \cong (X,-\varphi_X)$ in $\ECat$, 
then $[(X,\varphi_X)]_{\sigma} = 0$.
\end{rem}

Finally, 
we precisely record properties of $\Cat$ inherited by $\ECat$ 
under the assumption that $\sigma$ is compatible with them. 
The proof is straightforward.

\begin{prop}
	\label{prop:invMD}
Let $\Cat$ be a $\ak$-linear additive category with an involution $\sigma$.
\begin{itemize}
	\item If $\Cat$ is monoidal and $\sigma$ is a monoidal functor, 
		then $(X, \varphi_X)\otimes (Y, \varphi_Y) := (X\otimes Y, \varphi_X\otimes \varphi_Y)$ 
		defines a monoidal structure on $\ECat$. Tensor product of morphisms is given as in $\Cat$.
	\item If $\Cat$ has a contravariant duality functor $\mathbb{D}$ as in Definition \ref{def:mixed} 
		such that $\mathbb{D} \circ \sigma = \sigma \circ \mathbb{D}$, 
		then $\mathbb{D}(X, \varphi):= (\mathbb{D}(X), \mathbb{D}(\sigma(\varphi)))$ 
		defines a contravariant duality on $\ECat$.
		The duality is given on morphisms as in $\Cat$.  \qed
\end{itemize}
\end{prop}

\subsection{Indecomposable objects in equivariant categories} 
	\label{subsec:equivariantizationmixed}
	
Now, assume that $\K$ is an integral domain.
We aim to characterize the indecomposable objects in $\ECat$.
Observe that if $X \in \Cat$ is an equivariantizable object, then $X \cong \sigma(X)$.
However, we warn the reader that, in general, the converse does not hold 
without additional assumptions. 
In particular, the theory greatly simplifies under Hypothesis \ref{hypo:assumeme} below, 
and simplifies further under the assumption that
$\K$ is an algebraically closed field of characteristic not equal to $2$. 
The case when $\K$ is not algebraically closed is discussed in Appendix \ref{SS:equivariantsupplement}.

\begin{nota} \label{nota:mixedsigma}
Let $\K$ be an integral domain where $2$ is invertible 
and let $\Cat$ be a graded additive $\K$-linear category equipped with an involution $\sigma$ 
which commutes with the shift autoequivalence $\qsh$. 
Let $\{X_b\}_{b \in \BC}$ be an endopositive family which is preserved by $\sigma$, 
up to isomorphism. 
Define an action of $\sigma$ on $\BC$ by $\sigma (X_b) \cong X_{\sigma(b)}$. 
Since $\sigma(X_{\sigma(b)}) \cong X_b$, 
there is a decomposition
\[
\BC = \BC^{\mathrm{fix}} \sqcup \BC^{\mathrm{free}}
\]
where $\sigma(b) = b$ for all $b\in \BC^{\mathrm{fix}}$ and $\sigma(b) \ne b$ for all $b\in \BC^{\mathrm{free}}$.
\end{nota}

In other words, $\BC^{\mathrm{fix}}$ is the set of fixed points in $\BC$ under the $\sigma$-action 
and $\BC^{\mathrm{free}}$ is the union of all orbits of size two.

\begin{rem} If $\Cat$ is self-dual mixed and $\sigma$ commutes with $\mathbb{D}$, 
then $\sigma$ preserves the set of self-dual indecomposable objects.
Hence, it automatically preserves $\{X_b\}_{b \in \BC}$, 
up to isomorphism. \end{rem}

For the remainder of this section, 
\textbf{we assume the following hypothesis.}

\begin{hypothesis} \label{hypo:assumeme} 
Assume we are in the situation of Notation \ref{nota:mixedsigma}
and that $X_b$ is equivariantizable for each $b \in \BC^{\mathrm{fix}}$.
\end{hypothesis}

When $\K$ is an algebraically closed field with $2 \ne 0$, 
the second part of this hypothesis is superfluous, 
since it is straightforward to show that
$X_b \cong \sigma(X_b)$ implies that $X_b$ is equivariantizable. 
The results cited below from \cite{EliasFolding} are sketched there 
under the assumption that $\K$ is an algebraically closed field where $2 \ne 0$. 
We prove them more generally under Hypothesis \ref{hypo:assumeme} in Appendix \ref{SS:equivariantsupplement}.

\begin{lem}[{\cite[Proposition 3.10]{EliasFolding}}]
	\label{lem:phib}
For all $b\in \BC^{\mathrm{fix}}$, 
there is an isomorphism $\varphi_b \colon X_b \xrightarrow{\cong} \sigma(X_b)$ 
such that $\sigma(\varphi_b) \circ \varphi_b= \id_{X_b}$. 
For all orbits $\{ b, \sigma(b)\} \subset \BC^{\mathrm{free}}$, 
there is an isomorphism 
\[
\psi_b \colon X_b \oplus X_{\sigma(b)} \xrightarrow{\cong} \sigma(X_b\oplus X_{\sigma(b)})
\] 
such that $\sigma(\psi_b)\circ \psi_b= \id$.
As equivariant objects, $(X_b, \varphi_b) \not \cong (X_b, -\varphi_b)$ for $b \in \BC^{\mathrm{fix}}$ while
\begin{equation} \label{notfixedwhathappens}
(X_b \oplus X_{\sigma(b)}, \psi_b) \cong (X_b \oplus X_{\sigma(b)},-\psi_b)
\end{equation}
for $b\in \BC^{\mathrm{free}}$. \qed
\end{lem}

\begin{prop}[{\cite[Proposition 3.10]{EliasFolding}}]
	\label{prop:indECat}
Fix isomorphisms $\varphi_b$ as in Lemma \ref{lem:phib} for all $b \in \BC^{\mathrm{fix}}$.
If $(X, \varphi_X)$ is an indecomposable object in $\ECat$ for which $X \in \CatBC$, 
then exactly one of the following holds.
\begin{itemize}
\item $(X,\varphi_X) \cong \qsh^k (X_b, \varphi_b)$ for some $b\in \BC^{\mathrm{fix}}$ and $k \in \Z$,
\item $(X,\varphi_X) \cong \qsh^k (X_b, -\varphi_b)$ for some $b\in \BC^{\mathrm{fix}}$ and $k \in \Z$, or
\item $(X,\varphi_X) \cong \qsh^k (X_b\oplus X_{\sigma(b)}, \psi_b) \cong \qsh^k (X_b\oplus X_{\sigma(b)},-\psi_b)$ 
	for some $b\in \BC^{\mathrm{free}}$ and $k \in \Z$. \qed
\end{itemize}
\end{prop}

It is relatively straightforward to compute the morphism spaces between the 
self-dual indecomposable equivariant objects $(X_b,\pm \varphi_b)$ and $(X_b\oplus X_{\sigma(b)}, \psi_b)$. 

\begin{cor}
	\label{cor:ECatsdm}
In $\ECat$, 
the objects $\{(X_b,\varphi_b), (X_b,-\varphi_b)\}_{b \in \BC^{\mathrm{fix}}} 
	\cup \{(X_b \oplus X_{\sigma(b)},\psi_b)\}_{b \in \BC^{\mathrm{free}}}$ 
form an endopositive family. 
If $\Cat$ is self-dual mixed, then $\ECat$ is self-dual mixed. \end{cor}

\begin{rem}
For general $G$, the equivariant category of a self-dual mixed category is mixed, but not necessarily self-dual. 
The equivariant category is self-dual mixed if and only if all representations of $G$ are self dual, 
as is the case for $G= \Z/2$. 
\end{rem}

\begin{nota} 
	\label{not:ECatBC}
Using a slight abuse of notation, 
the full subcategory of $\ECat$ whose objects are direct sums of shifts of the endopositive family in 
Corollary \ref{cor:ECatsdm} will be denoted $\ECatBC$. \end{nota}

Next, we discuss an effective algorithm to take an equivariant object $(Y,\varphi_Y)$ in $\ECat$
(and not necessarily in $\ECatBC$) 
and find all indecomposable summands of the form in Corollary \ref{cor:ECatsdm}. 
Recall that if $X$ and $Y$ are equipped with equivariant structures, 
then $\Hom(X,Y)$ admits an action of $\sigast$ which depends on the choice of those structures. 
When $X = X_b$ for $b \in \BC^{\mathrm{fix}}$, we always assume that the choice of 
equivariant structure for this action is $\varphi_b$ (rather than $-\varphi_b$), 
and we denote the corresponding action as $\sigbast$. 

We leave it to the reader to verify that $\sigbast$ preserves the kernels 
$\rker(\beta^k_{X_b, Y})$ and $\lker(\beta^k_{X_b, Y})$ of the graded composition pairing. 
Therefore, $\sigbast$ descends to an action on $V^k(X_b,Y)$ and $V^{-k}(Y,X_b)$ for all $k \in \Z$. 
It is not hard to show that any element of the $+1$-eigenspace of $\sigbast$ 
is indeed a valid inclusion map making $(X_b,\varphi_b)$ a direct summand of $(Y,\varphi_Y)$, 
while any element of the $-1$-eigenspace gives a direct summand of the form $(X_b, -\varphi_b)$.

\begin{prop}[{\cite[Claim 3.16]{EliasFolding}}]
	\label{prop:SummandViaEV}
Let $(Y, \varphi_Y)$ be an object in $\ECat$.
\begin{enumerate}
\item For $b\in \BC^{\mathrm{free}}$, 
	the multiplicity of $\qsh^k (X_b\oplus X_{\sigma(b)}, \psi_b)$ in $(Y,\varphi_Y)$ 
	is equal to the multiplicity of $\qsh^k X_b$ in $Y$.

\item For $b\in \BC^{\mathrm{fix}}$, 
	the multiplicity of $\qsh^k (X_b,\varphi_b)$ in $(Y,\varphi_Y)$ 
	equals the dimension of the $+1$-eigenspace of $\sigbast$ in $V^k(X_b, Y)$. 
	
\item For $b\in \BC^{\mathrm{fix}}$, 
	the multiplicity of $\qsh^k (X_b,-\varphi_b)$ in $(Y,\varphi_Y)$ 
	equals the dimension of the {$-1$-eigenspace} of $\sigbast$ in $V^k(X_b, Y)$. 
\end{enumerate}
\end{prop}

The above results allow for the computation of the weighted Grothendieck group of $\ECatBC$.
Observe that \eqref{notfixedwhathappens} implies that $[(X_b\oplus X_{\sigma(b)}, \psi_b)]_{\sigma}=0$ 
whenever $b \in \BC^{\mathrm{free}}$, since in this case
\[
[(X_b\oplus X_{\sigma(b)}, \psi_b)]_{\sigma} 
= [(X_b\oplus X_{\sigma(b)}, -\psi_b)]_{\sigma} = -[(X_b\oplus X_{\sigma(b)}, \psi_b)]_{\sigma} \, .
\]

\begin{prop}[{\cite[Prop. 3.17 and Cor. 3.19]{EliasFolding}}]
	\label{prop:computeequivGgp}
The $\sigma$-weighted Grothendieck group $\wKzero{\ECatBC}$ has a basis
\[
\big\{ [(X_b, \varphi_b)]_{\sigma} \big\}_{b\in \BC^{\mathrm{fix}}}
\]
in bijection with $\BC^{\mathrm{fix}}$. 
Moreover, if $(Y, \varphi_Y)$ is any object in $\ECatBC$, then 
\begin{equation}
	\label{eq:KzeroDecomp}
[(Y, \varphi_Y)]_{\sigma} 
= \sum_{k\in \Z}\sum_{b\in \BC^{\mathrm{fix}}} \tr(\sigbast |_{V^k(X_b, Y)}) q^k [(X_b, \varphi_b)]_{\sigma} \, . 
\end{equation}
\qed
\end{prop}

In light of \eqref{eq:KzeroDecomp},
we conclude this section with remarks on the practicalities of computing the trace of $\sigbast$. 
It is often easy to compute $\sigbast$ on $\Hom^k(X_b,Y)$, 
but computing the action on $V^k(X_b,Y)$ can be difficult. 
In \cite{EliasFolding}, the action of $\sigbast$ on $V^k(X_b,Y)$ could be computed directly, 
since one could find a subset $\{\iota_1, \ldots, \iota_d\} \subset \Hom^k(X_b,Y)$ 
that is permuted by $\sigbast$ and which descends to a basis of $V^k(X_b,Y)$.
See e.g.~\cite[Sections 5.1, 5.4]{EliasFolding}.
 
In our setting, finding such a ``basis of inclusion maps'' which is permuted by $\sigbast$ is significantly more difficult.
In some instances, we are able to overcome this difficulty by computing directly, 
e.g.~the computations in \S \ref{s:Divpower}.
In others, e.g.~ in \S \ref{s:Serre}, we instead compute the action of $\sigbast$ on
$V^k(X_b,Y)$ using the techniques of Remark \ref{rem:howtodealwithVk}.
The following result records this method explicitly. 

\begin{lem}\label{L:Ben's-linear-algebra-lemma}
Suppose that $V^k(X_b,Y)$ is free over $\K$ of rank $d$ 
and choose $\{\iota_1, \ldots, \iota_d\} \subset \Hom^k(X_b,Y)$ 
and $\{p_1, \ldots, p_d\} \subset \Hom^{-k}(Y,X_b)$ 
which descend to bases of $V^k(X_b,Y)$ and $V^{-k}(Y,X_b)$ respectively.
Assume also that the pairing matrix $C$ between these bases is invertible over $\K$. 
If $S$ is the pairing matrix between $\{\sigbast \iota_1, \ldots, \sigbast \iota_d\}$ and $\{p_1, \ldots, p_d\}$, 
then $C^{-1}S$ is the matrix encoding the action of $\sigbast$ on $V^k(X_b, Y)$, 
in the basis induced by $\{\iota_1, \ldots, \iota_d\}$.
\end{lem}
\begin{proof}
Let $x\mapsto \overline{x}$ denote the quotient map $\Hom^k(X_b,Y) \twoheadrightarrow V^k(X_b,Y)$, 
so $\{\overline{\iota_1}, \dots, \overline{\iota_d}\}$ is a basis for $V^k(X_b, Y)$. 
The action of $\sigbast$ on $V^{k}(X_b, Y)$ is determined by the matrix $\Sigma$ 
which describes how to write $\sigbast \overline{\iota_r}$ 
in terms of the basis $\{\overline{\iota_1}, \dots, \overline{\iota_d}\}$. 
It is easy to see that the pairing matrix $S$ is equal to 
the composition of $\Sigma$ and the pairing matrix $C$, i.e.~$S= C\cdot\Sigma$. 
Thus, $\Sigma = C^{-1}S$. 
\end{proof}

%
\section{An equivariant category from the categorified quantum group}
	\label{s:ECQG}
%

In this section we assume that $\K$ is an integral domain and $2$ is invertible in $\K$.

\subsection{The $\FF \EE \one_{\wtn}$ equivariant category}
	\label{ss:FE}	

We assume now, and for the duration, 
that we have chosen the functions in Theorem \ref{thm:involution} and fixed $n \geq 1$.
We thus have the involution $\tau \colon \cXX_q(\glm) \to \cXX_q(\glm)$ 
from Corollary \ref{cor:thickinvolution}.
Since $\tau$ preserves the object $\one_{\wtn}$, 
it preserves the endomorphism category of $\one_{\wtn}$,
which is the following monoidal category.

\begin{defn}
\label{def:BFoam}
The monoidal category $\BFoam$ is $\one_{\wtn}\cXX_q(\glm)\one_{\wtn}$. 
Equivalently, it is the full $2$-subcategory of $\cUU_q(\glm)$ generated by the $1$-morphisms
$\EE_i^{(k)}\FF_i^{(k)}\one_{\wtn}$ and $\FF_i^{(k)}\EE_i^{(k)}\one_{\wtn}$
for $i\in \{1, \ldots, m-1\}$ and $k\in \Z_{\ge 0}$.
\end{defn}

Note that both $\tau$ and $\BFoam$ depend on $n$, 
though, in light of the results from \S\ref{ss:dependence-on-n}, the dependence is not significant.
Hence, the absence of this dependence in our notation.
We now aim to apply the techniques of \S \ref{subsec:equivariantizationmixed} 
to study the corresponding equivariant category $\EBFoam$.
Since the involution $\tau$ from Corollary \ref{cor:thickinvolution} gives a monoidal involution of $\BFoam$, 
Proposition \ref{prop:invMD} shows that $\EBFoam$ is monoidal. 

\begin{rem}
	\label{rem:BFoamSDM}
Example \ref{ex:CQGmixed} gives that $\Kar(\UU_q(\glm))$ is self-dual mixed when $\K$ is a field of characteristic zero, 
thus the same is true for the full $2$-subcategory $\Kar(\BFoam)$. 
By Corollary \ref{cor:ECatsdm}, its $\tau$-equivariant category is also self-dual mixed.
\end{rem}

We begin by exploring equivariant objects.
Note that
\begin{equation}
	\label{eq:easyStosic}
\begin{tikzpicture}[anchorbase,yscale=1]
\draw[ultra thick, green,->] (0,0) node[below,yshift=2pt] {\scs$k$} to (0,1.5);
\draw[ultra thick, green,<-] (.5,0) node[below,yshift=2pt] {\scs$k$} to (.5,1.5);
\node at (.875,1.25){$\wtn$};
\end{tikzpicture}
=
\begin{tikzpicture}[anchorbase,yscale=1]
\draw[ultra thick, green,->] (0,0) node[below,yshift=2pt] {\scs$k$} to [out=90,in=270] (.5,.75) 
	to [out=90,in=270] (0,1.5);
\draw[ultra thick, green,<-] (.5,0) node[below,yshift=2pt] {\scs$k$} to [out=90,in=270] (0,.75) 
	to [out=90,in=270] (.5,1.5);
\node at (.875,1.25){$\wtn$};
\end{tikzpicture}
\, , \quad
\begin{tikzpicture}[anchorbase,yscale=1]
\draw[ultra thick, green,<-] (0,0) node[below,yshift=2pt] {\scs$k$} to (0,1.5);
\draw[ultra thick, green,->] (.5,0) node[below,yshift=2pt] {\scs$k$} to (.5,1.5);
\node at (.875,1.25){$\wtn$};
\end{tikzpicture}
=
\begin{tikzpicture}[anchorbase,yscale=1]
\draw[ultra thick, green,<-] (0,0) node[below,yshift=2pt] {\scs$k$} to [out=90,in=270] (.5,.75) 
	to [out=90,in=270] (0,1.5);
\draw[ultra thick, green,->] (.5,0) node[below,yshift=2pt] {\scs$k$} to [out=90,in=270] (0,.75) 
	to [out=90,in=270] (.5,1.5);
\node at (.875,1.25){$\wtn$};
\end{tikzpicture}
\end{equation}
which are special cases of the Sto\v{s}i\'{c} formulae \eqref{eq:Stosic} when $k=\ell$. Consequently 
\begin{equation}
	\label{eq:EF=FE}
\EE_i^{(k)} \FF_i^{(k)} \one_{\wtn} \cong \FF_i^{(k)}\EE_i^{(k)} \one_{\wtn}.
\end{equation}
Moreover, these $1$-morphisms are indecomposable 
(see \eqref{eq:EF} and \eqref{eq:FE}).

\begin{rem}\label{rem:FiEiep}
More generally, for fixed $i$ and over any commutative ring $\K$, 
the sets of objects $\{\FF_i^{(k)} \EE_i^{(k)} \one_{\wtn}\}_{k \ge 0}$ and $\{\EE_i^{(k)} \FF_i^{(k)} \one_{\wtn}\}_{k \ge 0}$ 
each form an endopositive family.
This generalizes Example \ref{ex:CQG2}.
To see this, note that the size of the morphism space 
e.g.~between $\FF_i^{(k)} \EE_i^{(k)} \one_{\wtn}$ and $\FF_i^{(l)} \EE_i^{(l)} \one_{\wtn}$ 
depends on the ambient context. 
That they form an endopositive family within $\BFoam[2]$ was shown explicitly in \cite[Proposition 5.15]{KLMS}. 
For $m > 2$, 
the morphism space in $\BFoam$ is larger than in $\BFoam[2]$, 
since, at the very least, there are additional (positive degree) new bubble generators in 
$\End_{\BFoam}(\one_\wtn) \cong \Sym(\X_1| \cdots |\X_m)$.
In fact, this is the only difference: as shown in \cite[Proposition 3.11]{KL3}, 
if the color $j$ does not appear on the boundary of a diagram, 
then, as a module over $\End_{\BFoam}(\one_\wtn)$, 
there is a set of generators for the $\Hom$-space that does not contain the color $j$.
\end{rem}

\begin{rem} \label{rmk:homformula}
Further, it is possible to describe a basis for the $\Hom$-spaces in $\BFoam$.
There is a certain sesquilinear form on the quantum group $\dU(\glm)$
which gives an upper bound on the size of morphism spaces in $\UU_q(\glm)$; 
see \cite[Corollary 3.14]{KL3}. 
In the proof of \cite[Proposition 3.11]{KL3}, Khovanov--Lauda give diagrammatic arguments 
to find a spanning set for morphism spaces. 
By construction,
the size of this spanning set agrees with a corresponding value of the sesquilinear form \cite[Proposition 3.12]{KL3}. 

Khovanov--Lauda call the category \emph{nondegenerate} \cite[Definition 3.15]{KL3} if their spanning set is actually a basis, 
in which case the sesquilinear form precisely controls the size of morphism spaces. 
We refer to this as the \emph{Khovanov--Lauda Hom formula}. 
They use an action of the categorified quantum group on cohomology of partial flag varieties to prove nondegeneracy 
when $\K$ is a field \cite[Theorem 1.3]{KL3}. 

Upon inspection, the diagrammatic spanning argument is valid over $\Z$. 
Since linearly independent elements in a $\mathbb{Q}$-vector space are necessarily linearly 
independent in the $\Z$-module that they generate, 
nondegeneracy over $\Z$ (and hence over any commutative ring), 
follows from nondegeneracy over $\Q$.
\end{rem}

The following objects and morphisms will be of fundamental importance moving forward.

\begin{nota}
Let $\FEk_i :=\FF_{i}^{(k)}\EE_{i}^{(k)}\one_{\wtn}$, viewed as an object in the monoidal category $\BFoam$,
and denote
\[
\swcl_{k} :=
\begin{tikzpicture}[anchorbase,scale=.75]
\draw[ultra thick,green,<-] (0,0) node[below=-2pt]{\scs$k$} to [out=90,in=270] (.5,1);
\draw[ultra thick,green,->] (.5,0) node[below=-2pt]{\scs$k$} to [out=90,in=270] (0,1);
\node at (.875,.75){\scs$\wtn$};
\end{tikzpicture} 
\, , \quad
\swcr_{k} :=
\begin{tikzpicture}[anchorbase,scale=.75]
\draw[ultra thick,green,->] (0,0) node[below=-2pt]{\scs$k$} to [out=90,in=270] (.5,1);
\draw[ultra thick,green,<-] (.5,0) node[below=-2pt]{\scs$k$} to [out=90,in=270] (0,1);
\node at (.875,.75){\scs$\wtn$};
\end{tikzpicture} \, .
\]
(Here, as usual, we follow Convention \ref{conv:colors}, i.e.~$\swcl_k$ and $\swcr_k$ are $i$-colored.) 
We will also denote $\FEk_i$ by $\FEk_{\grb}$
and abbreviate $\FE_i := \FE_i^{(1)}$.
\end{nota}

\begin{prop}
	\label{prop:eX}
For each $k \geq 0$,
the pair $(\FEk_{\grb},\swcl_k)$ is an object in $\EBFoam$. 
The family $\{(\FEk_{\grb},\swcl_k)\}_{k \geq 0}$ is endopositive.
\end{prop}

\begin{proof}
We have $\swcl_k \in \Hom_{\BFoam}\big(\FEk_\grb, \tau(\FEk_\grb) \big)$
and a computation using
equations \eqref{eq:tauonthickcap}, \eqref{eq:tauonthickcup}, and \eqref{eq:tauonthickcrossing}
shows that $\tau(\swcl_k) = \swcr_k$. 
We then compute
\[
\tau(\swcl_k) \circ \swcl_k = \swcr_k \circ \swcl_k \stackrel{\eqref{eq:easyStosic}}{=} \id_{\FEk_\grb}
\]
so $(\FEk_\grb, \swcl_k)$ is an equivariant object, per Definition \ref{def:equivariant}. 
The assertion that the family of all such objects is endopositive follows from
Remarks \ref{rem:indECat} and \ref{rem:FiEiep}.
\end{proof}

\begin{rem}
The monoidal unit in $\BFoam$ is $\one_{\wtn}$ 
and the monoidal unit of $\EBFoam$ is the pair $(\one_{\wtn}, \id_{\one_{\wtn}})$. 
In the notation above, this is the same as $(\FE_{\grb}^{(0)}, \swcl_{0})$ for any color $i$. 
\end{rem}

In the following sections,
we study the monoidal structure of $\EBFoam$ and establish certain tensor product decompositions 
via the techniques in \S\ref{sec:equiv}. 
We use the notation established there.

\subsection{Divided power relation for $\FE_i$}
	\label{s:Divpower}

Our main goal in this subsection is to show that relation \eqref{E:xxk-recursion} 
holds for $(\FE_{\grb}^{(k)}, \swcl_k)$ in the weighted Grothendieck group of $\EBFoam$.

\begin{lem}
The sets
\[
\left\{
\begin{tikzpicture}[anchorbase,scale=.75]
\draw[ultra thick,green,<-] (-.25,-.5) to [out=90,in=270] (-.75,.5) to (-.75,1.5) node[above=-2pt]{\scs$k$};
\draw[ultra thick,green,->] (-.25,1) to (-.25,1.5) node[above=-2pt]{\scs$k$};
\draw[thick,green,rdirected=.75] (-.25,1) to [out=330,in=270] node[black,pos=.5]{\scs$\bullet$} node[black,pos=.5,below=-1pt]{\scs$p$} (.25,1) to (.25,1.5);
\draw[ultra thick,green,directed=.55] (.25,0) to [out=150,in=270] (-.5,.75) to [out=90,in=210] (-.25,1);
\draw[thick,green,->] (.25,0) to [out=30,in=270] node[black,pos=.5]{\scs$\bullet$} node[black,pos=.5,right=-2pt]{\scs$q$} (.75,1.5);
\draw[ultra thick,green,directed=.55] (.25,-.5) node[below=-2pt]{\scs$k$} to (.25,0);
\node at (1.25,1){\scs$\wtn$};
\end{tikzpicture}
\right\}_{\substack{0\leq p \leq k \\ 0 \leq q \leq k-1}}
\quad \text{and} \qquad
\left\{
\begin{tikzpicture}[anchorbase,scale=.75]
\draw[ultra thick,green,rdirected=.35] (-.25,0) to [out=150,in=270] 
	(-.75,1.5) node[above=-2pt]{\scs$k$};
\draw[thick,green,rdirected=.75] (-.25,0) to [out=30,in=270] node[black,pos=.5]{\scs$\bullet$} node[black,pos=.5,right=-2pt]{\scs$p$} (.25,1.5);
\draw[ultra thick,green,->] (.25,0) to [out=150,in=270] (-.25,1.5) node[above=-2pt]{\scs$k$};
\draw[thick,green,->] (.25,0) to [out=30,in=270] node[black,pos=.5]{\scs$\bullet$} node[black,pos=.5,right=-2pt]{\scs$q$} (.75,1.5);
\draw[ultra thick,green,rdirected=.55] (-.25,-.5) node[below=-2pt,xshift=-5pt]{\scs$k{+}1$} to (-.25,0);
\draw[ultra thick,green,directed=.55] (.25,-.5) node[below=-2pt,xshift=5pt]{\scs$k{+}1$} to (.25,0);
\node at (1.25,1){\scs$\wtn$};
\end{tikzpicture}
\right\}_{0 \leq p, q \leq k}
\]
descend to bases for $V(\FEk_\grb, \FEk_\grb \FE_\grb)$ and $V(\FEkk_\grb , \FEk_\grb \FE_\grb)$, respectively.
\end{lem}
\begin{proof}
Immediate from \eqref{eq:EEdecomp} and \eqref{eq:Stosic}, 
together with Lemma \ref{lem:basisforV}.
\end{proof}

We now consider the $\tauast$-action from Proposition \ref{prop:sigast} 
on $V(\FEk_\grb, \FEk_\grb \FE_\grb)$ and $V(\FEkk_\grb , \FEk_\grb \FE_\grb)$,
with respect to the isomorphisms $\FEk_\grb \xrightarrow{\swcl_k} \tau(\FEk_\grb)$ and  
$\FEk_\grb \FE_\grb \xrightarrow{\swcl_k \otimes \swcl_1} \tau(\FEk_\grb \FE_\grb)$.

\begin{prop}
	\label{prop:tauast2strand}
Considered as elements in $V(\FEk_\grb, \FEk_\grb \FE_\grb)$ and $V(\FEkk_\grb , \FEk_\grb \FE_\grb)$, 
we have that
\[
\tauast 
 \, .
\]
\end{prop}

\begin{proof}
\begingroup
\allowdisplaybreaks
We compute, modulo terms that are in the kernel of the graded composition pairing. 
Each time we identify a term in the kernel, we depict it in {\color{gray} gray} 
and then drop it in subsequent steps.
\begin{equation}
	\label{eq:tauast1}
\begin{aligned}
\tauast 
%
}
\right) \, .
\end{aligned}
\end{equation}
Here, we see that the gray terms are in 
the kernel of the graded composition pairing 
using Lemma \ref{lem:kernelisjacobson},
since they are compositions of positive degree endomorphisms of the endopositive object 
$\FE_i^{(k)}$ with dotted cup maps (tensor the identity).
 
We next simplify the following subdiagram of the first term:
\begin{align}
\label{eq:simp}
\nonumber
%
\end{align}
Returning to the computation in \eqref{eq:tauast1}, this gives
\begin{align*}
\tauast 
%
\right) \, .
\end{align*}
The last gray terms in this computation are again in the kernel of the graded composition pairing 
since they are compositions with positive degree endomorphisms of the endopositive object $\FE_i^{(k)}$.
The earlier gray terms are in the kernel by Lemma \ref{lem:btob'inker} and Remark \ref{rem:FiEiep}, 
since they factor through $\FE_i^{(\ell)}$ for $\ell \ne k$. 
We henceforth use both arguments tacitly when marking morphisms gray.

By Lemma \ref{lem:EndKer}, all terms in the final summation are in the kernel of the graded composition pairing unless
$j=u$, thus we find
\begin{align*}
\tauast 

\right) \, .
\end{align*}

Next, we compute
\begin{align}
\nonumber
\tauast
%
\right) \, .
\end{align}
Using equation \eqref{eq:tauRickdiff} below to simplify the last diagram in \eqref{eq:tauast2} gives
\[
\tauast
%
 \, .
\qedhere
\]
\endgroup
\end{proof}

\begin{cor}
	\label{cor:traceonV}
We have
\[
\tr(\tauast |_{V(\FEk_\grb, \FEk_\grb \FE_\grb)}) 
	= (-1)^{n+k+1} \sum_{j=0}^{k-1} (-1)^{j} [2(k-j)]
		= (-1)^{n+k+1} ``[k][k+1]"
\]
and
\[
\tr(\tauast |_{V(\FEkk_\grb , \FEk_\grb \FE_\grb)}) 
	= \sum_{j=0}^{k} (-1)^j [2(k-j)+1]
		= ``[k+1]^2" \, .
\]
\end{cor}
\begin{proof}
Using Proposition \ref{prop:tauast2strand}, we compute
\begin{align*}
\tr(\tauast |_{V(\FEk_\grb, \FEk_\grb \FE_\grb)}) 
&= \sum_{\substack{0 \leq p \leq k \\ 0 \leq q \leq k-1}} 
	(-1)^{p+q+n+k+1} q^{2(p+q-k)+1} 
		(\delta_{p,q} - \delta_{p,k}) \\
&= (-1)^{n+k+1} q^{1-2k} \sum_{d=0}^{2k-1} 
	(-1)^d q^{2d} \sum_{\substack{p+q=d \\ 0 \leq p \leq k \\ 0 \leq q \leq k-1}} 
		(\delta_{p,q} - \delta_{p,k}) \, .
\end{align*}
Now, 
\[
\sum_{\substack{p+q=d \\ 0 \leq p \leq k \\ 0 \leq q \leq k-1}} \delta_{p,q} = 
\begin{cases} 	1 & d \text{ is even and } d \leq 2k-2 \\
			0 & \text{else}
	\end{cases}
\quad \text{and} \quad
\sum_{\substack{p+q=d \\ 0 \leq p \leq k \\ 0 \leq q \leq k-1}} \delta_{p,k} 
= 
\begin{cases} 	1 & k \leq d \leq 2k-1 \\
			0 & \text{else}
	\end{cases} \, ,
\]
so
\[
\tr(\tauast |_{V(\FEk_\grb, \FEk_\grb \FE_\grb)}) 
=
(-1)^{n+k+1} q^{1-2k} 
\left( 
\sum_{r=0}^{k-1} q^{4r} 
- \sum_{s=k}^{2k-1}(-1)^s q^{2s} 
\right) = (-1)^{n+k+1} \sum_{j=0}^{k-1} (-1)^{j} [2(k-j)] \, .
\]

For the second statement, we have
\begin{align*}
\tr(\tauast |_{V(\FEkk_\grb , \FEk_\grb \FE_\grb)})
= \sum_{0 \leq p,q \leq k} 
	(-1)^{p+q} q^{2(p+q-k)} \delta_{p,q}
&= \sum_{0 \leq p \leq k} q^{4p-2k} \\
&= \sum_{j=0}^{k} (-1)^j [2(k-j)+1] \, . \qedhere
\end{align*}
\end{proof}

\begin{thm}\label{T:div-powers-in-twisted-K0}
The following equality holds in $\tauKzero{\EBFoam}$:
\[
[(\FEk_{\grb}, \swcl_{k})]_{\tau}\cdot [(\FE_{\grb}, \swcl_{1})]_{\tau} 
	= (-1)^{n+k+1} ``[k][k+1]" 
		[(\FEk_{\grb},\swcl_{k})]_{\tau}
		+ ``[k+1]^2" 
			[(\FEkk_{\grb},\swcl_{k+1})]_{\tau} \, .
\]
\end{thm}

\begin{proof}
This follows from Proposition \ref{prop:computeequivGgp} and Corollary \ref{cor:traceonV}. 
\end{proof}

This result implies Theorem \ref{thm:introKzero2} from the introduction.

\begin{cor}
	\label{cor:K0(X2)}
Let $\K$ be a an integral domain in which $2$ is invertible.
The assignment 
$\nx^{(k)} \mapsto [(\FE^{(k)}, (-1)^{{n+1 \choose 2}+{n-k+1 \choose 2}+k} \swcl_{k})]_{\tau}$
determines an isomorphism
\[
U'_{-q^2}(\som[2])
\xrightarrow{\cong} \C(q) \otimes_{\Z[q^\pm]} \tauKzero{\EBFoam[2]} \, .
\]
\end{cor}
\begin{proof}
Set 
$\mathbf{x}^{(k)} := [(\FE^{(k)}, (-1)^{{n+1 \choose 2}+{n-k+1 \choose 2}+k} \swcl_{k})]_{\tau}$.
Theorem \ref{T:div-powers-in-twisted-K0} gives that
\begin{multline*}
(-1)^{2{n+1 \choose 2}+{n-k+1 \choose 2}+{n \choose 2}+k+1} \mathbf{x}^{(k)} \mathbf{x}^{(1)}
	= (-1)^{n+2k+1+{n+1 \choose 2}+{n-k+1 \choose 2}} ``[k][k+1]" \mathbf{x}^{(k)} \\
		+ (-1)^{{n+1 \choose 2}+{n-k \choose 2}+k+1} ``[k+1]^2" \mathbf{x}^{(k+1)} \, ,
\end{multline*}
i.e.~that 
\[
\mathbf{x}^{(k)} \mathbf{x}^{(1)}
	= (-1)^{k} ``[k][k+1]" \mathbf{x}^{(k)} + (-1)^{k} ``[k+1]^2" \mathbf{x}^{(k+1)} \, .
\]
This is exactly the relation that defines $\nx^{(k)} \in U'_{-q^2}(\som[2])$ 
and $\{ \nx^{(k)}\}_{k \geq 0}$ are a $\C(q)$-basis for $U'_{-q^2}(\som[2])$,
so $\nx^{(k)} \mapsto \mathbf{x}^{(k)}$ defines a $\C(q)$-algebra homomorphism
$U'_{-q^2}(\som[2]) \to \C(q) \otimes_{\Z[q^\pm]} \tauKzero{\EBFoam[2]}$.
Since Hypothesis \ref{hypo:assumeme} holds for $\BFoam$, 
Proposition \ref{prop:computeequivGgp} implies that
this homomorphism is an isomorphism.
\end{proof}


\subsection{Devil's Serre relation for $\FE_i$ and $\FE_j$}
	\label{s:Serre}

We now show that a categorification of the devil's Serre relation for $U'_{-q^2}(\som)$
holds for $\FE_i$ and $\FE_{i \pm 1}$.
Precisely, we show that the relation \eqref{eq:iSerre} holds 
for the classes of $(\FE_i, \swcl_1)$ and $(\FE_{i\pm 1}, \swclp_1)$ 
in the weighted Grothendieck group of $\EBFoam$. 

We begin by recording some isomorphisms in $\U_q(\glm)$. 
For each $1\le i \le m-1$, 
equations \eqref{extendedreln1}, \eqref{extendedreln2}, and \eqref{eq:StosicFE} 
give the following isomorphisms with indicated projection maps:
\begin{equation}\label{eq:sl2relns-for-serre1}
\begin{pmatrix}
\begin{tikzpicture} [scale=.35,anchorbase]
	\draw[thick,green,<-] (1,0) to [out=90,in=270] (0,1.5);
	\draw[thick,green,->] (0,0) to [out=90,in=270] (1,1.5);
\end{tikzpicture}
&
\begin{tikzpicture} [scale=.35,anchorbase]
	\draw[thick,green,->] (0,0) to [out=90,in=180] (.5,.75) to [out=0,in=90] (1,0);
\end{tikzpicture}
\end{pmatrix}^T \colon
\EE_i\FF_i\one_{\wt} \cong \FF_i\EE_i\one_{\wt}\oplus \one_{\wt} \, , \quad \text{if} \ \alpha_i^{\vee}(\wt) = 1 \, ,
\end{equation}
\begin{equation}\label{eq:sl2relns-for-serre2}
\begin{pmatrix}
\begin{tikzpicture} [scale=.35,anchorbase,xscale=-1]
	\draw[thick,green,<-] (1,0) to [out=90,in=270] (0,1.5);
	\draw[thick,green,->] (0,0) to [out=90,in=270] (1,1.5);
\end{tikzpicture}
&
\begin{tikzpicture} [scale=.35,anchorbase,xscale=-1]
	\draw[thick,green,->] (0,0) to [out=90,in=180] (.5,.75) to [out=0,in=90] (1,0);
\end{tikzpicture}
\end{pmatrix}^T \colon
\FF_i\EE_i\one_{\wt} \cong \EE_i\FF_i\one_{\wt}\oplus \one_{\wt} \, , \quad \text{if} \ \alpha_i^{\vee}(\wt) = -1 \, ,
\end{equation}
\begin{equation}\label{eq:sl2relns-for-serre3}
\begin{pmatrix}
\begin{tikzpicture} [scale=.35,anchorbase,xscale=-1]
	\draw[thick,green,<-] (1,0) to [out=90,in=270] (0,1.5);
	\draw[thick,green,->] (0,0) to [out=90,in=270] (1,1.5);
\end{tikzpicture}
&
\begin{tikzpicture} [scale=.35,anchorbase,xscale=-1]
	\draw[thick,green,->] (0,0) to [out=90,in=180] (.5,.75) to [out=0,in=90] (1,0);
\end{tikzpicture}
&
\begin{tikzpicture} [scale=.35,anchorbase,xscale=-1]
	\draw[thick,green,->] (0,0) to [out=90,in=180] (.5,.75) 
		node[black]{\scs$\bullet$} to [out=0,in=90] (1,0);
\end{tikzpicture}
\end{pmatrix}^T \colon
\FF_i\EE_i\one_{\wt} \cong \EE_i\FF_i\one_{\wt}\oplus [2]\one_{\wt} \, , \quad \text{if} \ \alpha_i^{\vee}(\wt)= -2 \, ,
\end{equation}
and
\begin{equation}\label{eq:sl2relns-for-serre4}
\begin{pmatrix}
\begin{tikzpicture} [scale=.35,anchorbase,xscale=-1]
	\draw[ultra thick,green,<-] (1,0) to [out=90,in=270] (0,1.5);
	\draw[ultra thick,green,->] (0,0) to [out=90,in=270] (1,1.5);
\end{tikzpicture}
&
\begin{tikzpicture}[anchorbase,scale=.75]
\draw[thick,green,rdirected=.35] (-.25,.5) to [out=150,in=270] (.25,1.25);
\draw[thick,green,<-] (-.25,1.25) to [out=270,in=30] (.25,.5);
\draw[thick,green,rdirected=.6] (-.25,.5) to [out=30,in=150] (.25,.5);
\draw[ultra thick,green,rdirected=.5] (.25,.5) to [out=270,in=90] (.25,.125);
\draw[ultra thick,green,->] (-.25,.5) to [out=270,in=90] (-.25,.125);
\end{tikzpicture}
\end{pmatrix}^T \colon
\FF_i^{(2)}\EE_i^{(2)} \cong \EE_i^{(2)}\FF_i^{(2)}\oplus \EE_i\FF_i, \quad \text{if} \ \alpha_i^{\vee}(\wt) = -1 \, .
\end{equation}
Further, for $j\ne i$ and all $\wt$,
we have
\begin{equation}\label{eq:sl3relns-for-serre1}
\begin{tikzpicture} [scale=.35,anchorbase]
	\draw[ultra thick,green,->] (0,0) to [out=90,in=270] (1,1.5);
	\draw[ultra thick,purple,<-] (1,0) to [out=90,in=270] (0,1.5);
\end{tikzpicture} \colon
\EE_i^{(k)}\FF_j^{(l)}\one_{\wt} \cong \FF_j^{(l)}\EE_i^{(k)}\one_{\wt}.
\end{equation}
which follows as a consequence of 
\eqref{eq:mixedef}, \eqref{dotslide}, \eqref{eqn:multicolorthickcrossingdefn}, and \eqref{eq:multicolorotherthick}.

\begin{conv}
Recall our convention for colors in Convention \ref{conv:colors} 
is that strands corresponding to the $i$th Dynkin node are {\color{green}green}, 
while strands corresponding to the $i-1$ and $i+1$ Dynkin nodes are colored 
{\color{red}red} and {\color{blue}blue}, respectively. 
In this subsection, we will color $j\in \{i\pm 1\}$ using {\color{purple}purple} 
--- the color obtained from combining blue and red.
Further, for the remainder of this subsection, 
we use the convention that all thick strands have thickness $2$.
(As always, thin strands are assumed to have thickness $1$.)
\end{conv}

Suppose that $j=i\pm 1$, then 
\cite[Theorem 3]{Stosic} implies that, for all $\wt$, 
there are isomorphisms with indicated projection maps:
\begin{equation}\label{eq:sl3relns-for-serre2}
\begin{aligned}
\begin{pmatrix}
\begin{tikzpicture}[anchorbase,scale=.25]
\draw[ultra thick,green,->] (-.5,.5) to (-.5,2);
\draw[thick,green] (-1,-1) to [out=90,in=210] (-.5,.5);
\draw[thick,green] (1,-1) to [out=90,in=330] (-.5,.5);
\draw[thick,purple,->] (0,-1) to [out=90,in=270] (.5,2);
\end{tikzpicture}
&
\begin{tikzpicture}[anchorbase,scale=.25,xscale=-1]
\draw[ultra thick,green,->] (-.5,.5) to (-.5,2);
\draw[thick,green] (-1,-1) to [out=90,in=210] (-.5,.5);
\draw[thick,green] (1,-1) to [out=90,in=330] (-.5,.5);
\draw[thick,purple,->] (0,-1) to [out=90,in=270] (.5,2);
\end{tikzpicture}
\end{pmatrix}^T \colon
\EE_i\EE_j\EE_i\one_{\wt} &\cong \EE_i^{(2)}\EE_j\one_{\wt} \oplus \EE_j\EE_i^{(2)}\one_{\wt} \\
\begin{pmatrix}
\begin{tikzpicture}[anchorbase,scale=.25]
\draw[ultra thick,green] (-.5,.5) to (-.5,2);
\draw[thick,green,<-] (-1,-1) to [out=90,in=210] (-.5,.5);
\draw[thick,green,<-] (1,-1) to [out=90,in=330] (-.5,.5);
\draw[thick,purple,<-] (0,-1) to [out=90,in=270] (.5,2);
\end{tikzpicture}
&
\begin{tikzpicture}[anchorbase,scale=.25,xscale=-1]
\draw[ultra thick,green] (-.5,.5) to (-.5,2);
\draw[thick,green,<-] (-1,-1) to [out=90,in=210] (-.5,.5);
\draw[thick,green,<-] (1,-1) to [out=90,in=330] (-.5,.5);
\draw[thick,purple,<-] (0,-1) to [out=90,in=270] (.5,2);
\end{tikzpicture}
\end{pmatrix}^T \colon
\FF_i\FF_j\FF_i\one_{\wt} &\cong \FF_i^{(2)}\FF_j\one_{\wt} \oplus \FF_j\FF_i^{(2)}\one_{\wt} \, .
\end{aligned}
\end{equation}

Our aim is to decompose 
\[
(\FE_i, \swcl_1) \otimes (\FE_{j}, \swclp_1) \otimes (\FE_i,\swcl_1) 
= (\FE_i\FE_j\FE_i, \swcl_1\otimes\swclp_1\otimes\swcl_1) 
\]
in $\Kar(\EBFoam)$. The first step is to decompose $\FE_i\FE_j\FE_i$ in $\U_q(\glm)$.

\begin{lem}\label{lem:X1X2X1decomposition}
Let $j=i\pm1$. 
There is an isomorphism
\begin{equation}\label{eq:X1X2X1decomposition}
\FE_i\FE_j\FE_i\cong \FE_i^{(2)}\FE_j\oplus \FE_j\FE_i^{(2)}\oplus [2]\FE_i^{(2)}\oplus \FE_i \oplus A\oplus B
\end{equation}
in $\U_q(\glm)$,
where $A = A_1 \oplus A_2$ and $B = B_1 \oplus B_2$ for 
$A_1 :=  \FF_j\EE_i\FF_i\EE_j\one_{\wtn}$, $B_1 := \FF_i^{(2)}\EE_j\FF_j\EE_i^{(2)}\one_{\wtn}$,
$\tau(A_1) =: A_2$, and $\tau(B_1) =: B_2$.
Moreover, the set
\[\{\FE_i^{(2)}\FE_j, \FE_j\FE_i^{(2)}, \FE_i^{(2)}, \FE_i,  A_1, A_2, B_1, B_2\}\]
of summands appearing in \eqref{eq:X1X2X1decomposition} is an endopositive family. \end{lem}

\begin{proof}
We have the following chain of isomorphisms:
\begin{align*}
\FF_i\EE_i\FF_j\EE_j\FF_i\EE_i\one_{\wtn}
	&\stackrel{{\eqref{eq:sl3relns-for-serre1}}}{\cong} 
		\FF_i\FF_j\EE_i\FF_i\EE_j\EE_i\one_{\wtn} \\
	&\stackrel{{\eqref{eq:sl2relns-for-serre1}}}{\cong} 
		\FF_i\FF_j\FF_i\EE_i\EE_j\EE_i\one_{\wtn}\oplus \FF_i\FF_j\EE_j\EE_i\one_{\wtn}  \\
	&\!\!\!\!\!\! \stackrel{{\eqref{eq:sl3relns-for-serre2},\eqref{eq:sl2relns-for-serre2}}}{\cong} 
		\FF_i^{(2)}\FF_j\EE_i^{(2)}\EE_j\one_{\wtn} \oplus \FF_i^{(2)}\FF_j\EE_j\EE_i^{(2)}\one_{\wtn} \oplus 
			\FF_j\FF_i^{(2)}\EE_i^{(2)}\EE_j\one_{\wtn}\oplus \FF_j\FF_i^{(2)}\EE_j\EE_i^{(2)}\one_{\wtn}  \\
		&\qquad \qquad \oplus \FF_i\EE_j\FF_j\EE_i\one_{\wtn} \oplus \FF_i\EE_i\one_{\wtn} \\
	&\!\!\!\!\!\!\!\!\!\!\! \stackrel{{\eqref{eq:sl2relns-for-serre3},\eqref{eq:sl2relns-for-serre4},
	\eqref{eq:sl3relns-for-serre1}}}{\cong} 
		\FF_i^{(2)}\EE_i^{(2)}\FF_j\EE_j\one_{\wtn} \oplus \FF_i^{(2)}\EE_j\FF_j\EE_i^{(2)}\one_{\wtn}\oplus 	
			[2]\FF_i^{(2)}\EE_i^{(2)}\one_{\wtn} \oplus \FF_j\EE_i^{(2)}\FF_i^{(2)}\EE_j\one_{\wtn} \\
		&\qquad \qquad \qquad \oplus \FF_j\EE_i\FF_i\EE_j\one_{\wtn}\oplus 
			\FF_j\EE_j\FF_i^{(2)}\EE_i^{(2)}\one_{\wtn} \oplus  \EE_j\FF_i\EE_i\FF_j\one_{\wtn} 
				\oplus \FF_i\EE_i\one_{\wtn} \\
	&\stackrel{{\eqref{eq:sl3relns-for-serre1}}}{\cong}
		\FF_i^{(2)}\EE_i^{(2)}\FF_j\EE_j\one_{\wtn} \oplus \FF_j\EE_j\FF_i^{(2)}\EE_i^{(2)}\one_{\wtn} 
			\oplus [2]\FF_i^{(2)}\EE_i^{(2)}\one_{\wtn} \oplus \FF_i\EE_i\one_{\wtn}  \\
	&\qquad \qquad \oplus \FF_j\EE_i\FF_i\EE_j\one_{\wtn} \oplus \EE_j\FF_i\EE_i\FF_j\one_{\wtn}
		\oplus \FF_i^{(2)}\EE_j\FF_j\EE_i^{(2)}\one_{\wtn} \oplus \FF_j\EE_i^{(2)}\FF_i^{(2)}\EE_j\one_{\wtn} \, .
\end{align*}
This yields \eqref{eq:X1X2X1decomposition},
since $\tau(\FF_j\EE_i\FF_i\EE_j\one_{\wtn})= \EE_j\FF_i\EE_i\FF_j\one_{\wtn}$ 
and $\tau(\FF_i^{(2)}\EE_j\FF_j\EE_i^{(2)}\one_{\wtn})= \FF_j\EE_i^{(2)}\FF_i^{(2)}\EE_j\one_{\wtn}$.

That the summands appearing form an endopositive family is a straightforward but tedious verification 
using the Khovanov--Lauda Hom formula; see Remark \ref{rmk:homformula}. 
It can also be checked directly, as we now demonstrate for a pair of relevant 
$\Hom$-spaces. (The others can be treated similarly.)

The proof of \cite[Proposition 3.11]{KL3} gives that the following diagrams:
\[
\begin{tikzpicture}[anchorbase,scale=.25]
\draw[thick,green,->] (-1,-2) to (-1,2) ;
\draw[thick,green,<-] (1,-2) to (1,2) ;
\draw[thick,purple,<-] (-2,-2) to (-2,2) ;
\draw[thick,purple,->] (2,-2) to (2,2) ;
\end{tikzpicture}
\qquad
\begin{tikzpicture}[anchorbase,scale=.25]
\draw[thick,green,<-] (-1,2) .. controls (-.5,1) and (.5,1) .. (1,2) ;
\draw[thick,green,->] (-1,-2) .. controls (-.5,-1) and (.5,-1) .. (1,-2) ;
\draw[thick,purple,<-] (-2,-2) to (-2,2) ;
\draw[thick,purple,->] (2,-2) to (2,2) ;
\end{tikzpicture}
\qquad
\begin{tikzpicture}[anchorbase,scale=.25]
\draw[thick,green,->] (-1,-2) to (-1,2) ;
\draw[thick,green,<-] (1,-2) to (1,2) ;
\draw[thick,purple,->] (-2,2) .. controls (-1.5,0) and (1.5,0) .. (2,2) ;
\draw[thick,purple,<-] (-2,-2) .. controls (-1.5,0) and (1.5,0) .. (2,-2) ;
\end{tikzpicture}
\qquad
\begin{tikzpicture}[anchorbase,scale=.25]
\draw[thick,green,<-] (-1,2) .. controls (-.5,1) and (.5,1) .. (1,2) ;
\draw[thick,green,->] (-1,-2) .. controls (-.5,-1) and (.5,-1) .. (1,-2) ;
\draw[thick,purple,->] (-2,2) .. controls (-1.5,0) and (1.5,0) .. (2,2) ;
\draw[thick,purple,<-] (-2,-2) .. controls (-1.5,0) and (1.5,0) .. (2,-2) ;
\end{tikzpicture} \, ,
\]
possibly adorned with
dots on strands and positive degree new-bubbles in the far right region, 
span $\End_{\UU_q(\glm)}(\FF_j\EE_i\FF_i\EE_j\one_{\wtn})$.
From \eqref{eq:CQGgens}, we compute that the degrees of these diagrams 
are $0, +4, +4$, and $+6$ respectively. 
It follows that $\End_{\UU_q(\glm)}(\FF_j\EE_i\FF_i\EE_j\one_{\wtn})$ is 
one-dimensional in degree zero and all other endomorphisms are in strictly positive degree, 
hence $\FF_j\EE_i\FF_i\EE_j\one_{\wtn}$ is an endopositive object.

The analogous diagrams for $\Hom_{\UU_q(\glm)}(\FF_j\EE_i\FF_i\EE_j\one_{\wtn}, \EE_j\FF_i\EE_i\FF_j\one_{\wtn})$ are as follows.
\[
\begin{tikzpicture}[anchorbase,scale=.25]
\draw[thick,green,->] (-1,-2) .. controls (2,1) .. (1,2) ;
\draw[thick,green,<-] (1,-2) .. controls (-2,1) .. (-1,2) ;
\draw[thick,purple,<-] (-2,-2) to (2,2) ;
\draw[thick,purple,->] (2,-2) to (-2,2) ;
\end{tikzpicture}
\qquad
\begin{tikzpicture}[anchorbase,scale=.25]
\draw[thick,green,->] (-1,2) .. controls (-.5,1) and (.5,1) .. (1,2) ;
\draw[thick,green,->] (-1,-2) .. controls (-.5,-1) and (.5,-1) .. (1,-2) ;
\draw[thick,purple,<-] (-2,-2) to (2,2) ;
\draw[thick,purple,->] (2,-2) to (-2,2) ;
\end{tikzpicture}
\qquad
\begin{tikzpicture}[anchorbase,scale=.25]
\draw[thick,green,->] (-1,-2) to (1,2) ;
\draw[thick,green,<-] (1,-2) to (-1,2) ;
\draw[thick,purple,<-] (-2,2) .. controls (-1.5,0) and (1.5,0) .. (2,2) ;
\draw[thick,purple,<-] (-2,-2) .. controls (-1.5,0) and (1.5,0) .. (2,-2) ;
\end{tikzpicture}
\qquad
\begin{tikzpicture}[anchorbase,scale=.25]
\draw[thick,green,<-] (-1,2) .. controls (-.5,1) and (.5,1) .. (1,2) ;
\draw[thick,green,->] (-1,-2) .. controls (-.5,-1) and (.5,-1) .. (1,-2) ;
\draw[thick,purple,<-] (-2,2) .. controls (-1.5,0) and (1.5,0) .. (2,2) ;
\draw[thick,purple,->] (-2,-2) .. controls (-1.5,0) and (1.5,0) .. (2,-2) ;
\end{tikzpicture}
\]
Again, \eqref{eq:CQGgens} gives that the degrees are $+2, +4, +4$, and $+6$. 
Thus, $\Hom_{\UU_q(\glm)}(\FF_j\EE_i\FF_i\EE_j\one_{\wtn}, \EE_j\FF_i\EE_i\FF_j\one_{\wtn})$ 
can only be non-zero in strictly positive degrees. 
\end{proof}

We now proceed to compute the class of $(\FE_i\FE_j\FE_i, \swcl_1\otimes\swclp_1\otimes\swcl_1)$ 
in the weighted Grothendieck group, 
using the technique from Remark \ref{rem:howtodealwithVk}, 
as implemented in Lemma \ref{L:Ben's-linear-algebra-lemma}.
This will be accomplished via a sequence of Lemmata which 
establish and study bases for the relevant factors in the domain of the 
non-degenerate graded composition pairing.
For the rest of this subsection, 
all diagrams are now assumed to have the $\glm$ weight $\wtn$ on the far right.

\begin{nota}
When applying the relation \eqref{eq:quadKLR} for strands labelled by $i$ and $j=i\pm 1$, 
we often deal with scalars of the form $\pm (i-j)$. For $j=i\pm 1$, we write
\[
\epsilon_{ij}:=j-i \quad \text{and} \quad \epsilon_{ji}:=i-j \, .
\]
We also write $\varphi_{ji}:=(-1)^{a_j}$, when $\wt=\wtn+\alpha_i$. 
By equation \eqref{newbub}, this is the value of a counterclockwise $j$-colored 
degree-zero bubble in weight $\wtn+\alpha_i$. 
Note that $\epsilon_{ij}\varphi_{ji}=(-1)^{n-1}$ and $\epsilon_{ji}\varphi_{ji}=(-1)^n$ for $j=i\pm 1$.
\end{nota}

\begin{lem}
	\label{lem:prin}
The maps 
\[
\pr:=\begin{tikzpicture}[anchorbase,scale=.25]
\draw[thick,green,->] (-3,1) to (-3,-2) ;
\draw[thick,green,->] (3,-2) to (3,1) ;
\draw[thick,green,->] (-2,-2) .. controls (-1.25,.25) and (1.25,.25) .. (2,-2) ;
\draw[thick,purple,->] (1,-2) .. controls (.75,-.5) and (-.75,-.5) .. (-1,-2) ;
\end{tikzpicture} 
\quad \text{and} \quad
\inc:=\begin{tikzpicture}[anchorbase,scale=.25]
\draw[thick,green,<-] (-3,-1) to (-3,2) ;
\draw[thick,green,<-] (3,2) to (3,-1) ;
\draw[thick,green,<-] (-2,2) .. controls (-1.25,-.25) and (1.25,-.25) .. (2,2) ;
\draw[thick,purple,<-] (1,2) .. controls (.75,.5) and (-.75,.5) .. (-1,2) ;
\end{tikzpicture}
\]
descend to bases for $V(\FE_i\FE_j\FE_i, \FE_i)$ and $V(\FE_i, \FE_i\FE_j\FE_i)$, 
respectively. 
\end{lem}
\begin{proof}
We use Lemma \ref{lem:basisforV}. 
Tracing through the isomorphisms in the proof of Lemma \ref{lem:X1X2X1decomposition}, 
we find the projection map $\FE_i\FE_j\FE_i\rightarrow \FE_i$ is
\[
\begin{tikzpicture}[anchorbase,scale=.25]
\draw[thick,green,->] (-3,2) to (-3,-2) ;
\draw[thick,green,->] (3,-2) to (3,2) ;
\draw[thick,green,->] (-2,-2) .. controls (-1.25,.25) and (1.25,.25) .. (2,-2) ;
\draw[thick, purple,->] (1,-2) to [out=90,in=270] (2,0) to [out=90,in=90]
	(-2,0) to [out=270,in=90] (-1,-2);
\end{tikzpicture} 
\stackrel{{\eqref{eq:mixedef}}}{=} \pr \, .
\]
The statement for the map $\inc$ can be obtained similarly, 
or follows from the computation of $\pr \circ \inc$ in the proof of Lemma \ref{L:tau-star-Serre1} below.
\end{proof}

\begin{lem}\label{L:tau-star-Serre1}
The equality 
$\tauast \inc = \inc$ 
holds in $V(\FE_i, \FE_i\FE_j\FE_i)$. 
\end{lem}
\begin{proof}
First, we compute 
\[
\pr \circ \inc = 
\begin{tikzpicture}[anchorbase,scale=.25]
\draw[thick,green,->] (-3,2) to (-3,-2) ;
\draw[thick,green,->] (3,-2) to (3,2) ;
\draw[thick, purple, ->] (1,0) arc (0:360:1);
\draw[thick, green,<-] (2,0) arc (0:360:2);
\end{tikzpicture}
\stackrel{\eqref{eq:mixedef}}{=}
\begin{tikzpicture}[anchorbase,scale=.25]
\draw[thick,green,->] (-3,2) to (-3,-2) ;
\draw[thick,green,->] (3,-2) to (3,2) ;
\draw[thick, purple, ->] (.5,0) arc (0:360:1);
\draw[thick, green,<-] (1.5,0) arc (0:360:1);
\end{tikzpicture}
\stackrel{\eqref{eq:quadKLR}}{=}
\epsilon_{ji}\cdot 
\begin{tikzpicture}[anchorbase,scale=.25]
\draw[thick,green,->] (-3,2) to (-3,-2) ;
\draw[thick,green,->] (3,-2) to (3,2) ;
\draw[thick, purple, ->] (-.5,0) arc (0:360:1);
\draw[thick, green,<-] (2.5,0) arc (0:360:1);
\filldraw [black] (-.5,0) circle (5pt);
\end{tikzpicture}
-\epsilon_{ji}\cdot 
\begin{tikzpicture}[anchorbase,scale=.25]
\draw[thick,green,->] (-3,2) to (-3,-2) ;
\draw[thick,green,->] (3,-2) to (3,2) ;
\draw[thick, purple, ->] (-.5,0) arc (0:360:1);
\draw[thick, green,<-] (2.5,0) arc (0:360:1);
\filldraw [black] (.5,0) circle (5pt);
\end{tikzpicture}
\quad.
\]
The weight to the far right is $\wtn$, 
so the clockwise undotted $i$-colored bubble in weight $\wtn+\alpha_i$
has degree $2(1-\alpha_i^{\vee}(\wtn+\alpha_i))= -2$ and thus equals zero. 
Hence,
\[
\pr\circ \inc 
=
-\epsilon_{ji}\cdot \begin{tikzpicture}[anchorbase,scale=.25]
\draw[thick,green,->] (-3,2) to (-3,-2) ;
\draw[thick,green,->] (3,-2) to (3,2) ;
\draw[thick, purple, ->] (-.5,0) arc (0:360:1);
\draw[thick, green,<-] (2.5,0) arc (0:360:1);
\filldraw [black] (.5,0) circle (5pt);
\end{tikzpicture}
\stackrel{\eqref{realbubdef},\eqref{newbub}}{=}
-\epsilon_{ji}\varphi_{ji}(-1)^{n} \id_{\FE_i} = -\id_{\FE_i}.
\]

Next, we compute $\pr\circ (\tauast \inc)$. 
Since
\[
\tau(\inc) = 
(-1)^{\ell_i'(\alpha_i) + r_j'(0)}
\begin{tikzpicture}[anchorbase,scale=.25]
\draw[thick,green,->] (-3,-1) to (-3,2) ;
\draw[thick,green,->] (3,2) to (3,-1) ;
\draw[thick,green,->] (-2,2) .. controls (-1.25,-.25) and (1.25,-.25) .. (2,2) ;
\draw[thick,purple,->] (1,2) .. controls (.75,.5) and (-.75,.5) .. (-1,2) ;
\end{tikzpicture}
\stackrel{{\eqref{tauconditions},\eqref{eq:ourfns}}}{=}
\begin{tikzpicture}[anchorbase,scale=.25]
\draw[thick,green,->] (-3,-1) to (-3,2) ;
\draw[thick,green,->] (3,2) to (3,-1) ;
\draw[thick,green,->] (-2,2) .. controls (-1.25,-.25) and (1.25,-.25) .. (2,2) ;
\draw[thick,purple,->] (1,2) .. controls (.75,.5) and (-.75,.5) .. (-1,2) ;
\end{tikzpicture}
\]
we have
\begin{align*}
\pr\circ (\tauast \inc) &= 
\begin{tikzpicture}[anchorbase,scale=.2]
\draw[thick,green] (-3,-4) to (3,-1.5) ;
\draw[thick,green] (3,-4) to (-3,-1.5) ;
\draw[thick,green] (-3,-1.5) to (-3,1) ;
\draw[thick,green] (3,-1.5) to (3,1) ;
\draw[thick,green] (-3,1) to (-2,2) ;
\draw[thick,green] (3,1) to (2,2) ;
\draw[thick,green] (2,1) to (3,2) ;
\draw[thick,green] (-2,1) to (-3,2) ;
\draw[thick,green] (-3,2) to (-3,4) ;
\draw[thick,green] (3,2) to (3,4) ;
\draw[thick, green,<-] (2,2) arc (0:180:2);
\draw[thick, green,<-] (2,1) arc (0:-180:2);
\draw[thick,purple] (1,1) to (-1,2) ;
\draw[thick,purple] (-1,1) to (1,2) ;
\draw[thick, purple,->] (1,2) arc (0:180:1);
\draw[thick, purple] (1,1) arc (0:-180:1);
\end{tikzpicture}
\stackrel{\eqref{eq:curl1}, \eqref{realbubdef}}{=}
 -\sum_{\substack{p+q=0 \\ p\ge 0}}
 \begin{tikzpicture}[anchorbase,scale=.2]
\draw[thick,green] (-3,-4) to (3,-1.5) ;
\draw[thick,green] (3,-4) to (-3,-1.5) ;
\draw[thick,green] (-3,-1.5) to (-3,1) ;
\draw[thick,green] (3,-1.5) to (3,1) ;
\draw[thick,green] (-3,1) to (-2,2) ;
\draw[thick,green] (3,1) to (2,2) ;
\draw[thick,green] (2,1) to (3,2) ;
\draw[thick,green] (-2,1) to (-3,2) ;
\draw[thick,green] (-3,2) to (-3,4) ;
\draw[thick,green] (3,2) to (3,4) ;
\draw[thick, green,<-] (2,2) arc (0:180:2);
\draw[thick, green,<-] (2,1) arc (0:-180:2);
\draw[thick, purple,<-] (.75,.5) arc (0:360:.75);
\draw[thick, purple,<-] (.75,2.5) arc (0:-360:.75);
\filldraw [black] (-.75,.5) circle (5pt) node[anchor=east]{\tiny{$q{+}\spadesuit$}} ;
\filldraw [black] (-.75,2.5) circle (5pt) node[anchor=east]{\scs{$p$}} ;
\end{tikzpicture}
\stackrel{\eqref{newbub}}{=}
-(-1)^{n-1}
 \begin{tikzpicture}[anchorbase,scale=.2]
\draw[thick,green] (-3,-4) to (3,-1.5) ;
\draw[thick,green] (3,-4) to (-3,-1.5) ;
\draw[thick,green] (-3,-1.5) to (-3,1) ;
\draw[thick,green] (3,-1.5) to (3,1) ;
\draw[thick,green] (-3,1) to (-2,2) ;
\draw[thick,green] (3,1) to (2,2) ;
\draw[thick,green] (2,1) to (3,2) ;
\draw[thick,green] (-2,1) to (-3,2) ;
\draw[thick,green] (-3,2) to (-3,4) ;
\draw[thick,green] (3,2) to (3,4) ;
\draw[thick, green,<-] (2,2) arc (0:180:2);
\draw[thick, green,<-] (2,1) arc (0:-180:2);
\draw[thick, purple,->] (1,1.5) arc (0:360:1);
\end{tikzpicture} \\
&\stackrel{\eqref{eq:mixedef}}{=}
(-1)^n
 \begin{tikzpicture}[anchorbase,scale=.2]
\draw[thick,green] (-3,-4) to (3,-1.5) ;
\draw[thick,green] (3,-4) to (-3,-1.5) ;
\draw[thick,green] (-3,-1.5) to (-3,1) ;
\draw[thick,green] (3,-1.5) to (3,1) ;
\draw[thick,green] (-3,1) to (-2,2) ;
\draw[thick,green] (3,1) to (2,2) ;
\draw[thick,green] (2,1) to (3,2) ;
\draw[thick,green] (-2,1) to (-3,2) ;
\draw[thick,green] (-3,2) to (-3,4) ;
\draw[thick,green] (3,2) to (3,4) ;
\draw[thick, green,<-] (2,2) arc (0:180:2);
\draw[thick, green,<-] (2,1) arc (0:-180:2);
\draw[thick, purple,->] (1,3.5) arc (0:360:1);
\end{tikzpicture}
\stackrel{\eqref{eq:quadKLR}}{=}
(-1)^n\epsilon_{ji}
 \begin{tikzpicture}[anchorbase,scale=.2]
\draw[thick,green] (-3,-4) to (3,-1.5) ;
\draw[thick,green] (3,-4) to (-3,-1.5) ;
\draw[thick,green] (-3,-1.5) to (-3,1) ;
\draw[thick,green] (3,-1.5) to (3,1) ;
\draw[thick,green] (-3,1) to (-2,2) ;
\draw[thick,green] (3,1) to (2,2) ;
\draw[thick,green] (2,1) to (3,2) ;
\draw[thick,green] (-2,1) to (-3,2) ;
\draw[thick,green] (-3,2) to (-3,4) ;
\draw[thick,green] (3,2) to (3,4) ;
\draw[thick, green,<-] (2,2) to (-2,2) ;
\draw[thick, green,<-] (2,1) arc (0:-180:2);
\draw[thick, purple,->] (1,3.5) arc (0:360:1);
\filldraw [black] (0,2.5) circle (5pt) ;
\end{tikzpicture}
-(-1)^n\epsilon_{ji}
 \begin{tikzpicture}[anchorbase,scale=.2]
\draw[thick,green] (-3,-4) to (3,-1.5) ;
\draw[thick,green] (3,-4) to (-3,-1.5) ;
\draw[thick,green] (-3,-1.5) to (-3,1) ;
\draw[thick,green] (3,-1.5) to (3,1) ;
\draw[thick,green] (-3,1) to (-2,2) ;
\draw[thick,green] (3,1) to (2,2) ;
\draw[thick,green] (2,1) to (3,2) ;
\draw[thick,green] (-2,1) to (-3,2) ;
\draw[thick,green] (-3,2) to (-3,4) ;
\draw[thick,green] (3,2) to (3,4) ;
\draw[thick, green,<-] (2,2) to (-2,2) ;
\draw[thick, green,<-] (2,1) arc (0:-180:2);
\draw[thick, purple,->] (1,3.5) arc (0:360:1);
\filldraw [black] (0,2) circle (5pt) ;
\end{tikzpicture} \\
&\stackrel{\eqref{eq:quadKLR},\eqref{dotslide}}{=}
(-1)^{n+1}\epsilon_{ji}
\begin{tikzpicture}[anchorbase,scale=.2]
\draw[thick,green,<-] (-3,-4) to (3,-1.5) ;
\draw[thick,green] (3,-4) to (-3,-1.5) ;
\draw[thick,green] (-3,-1.5) to (3,2) ;
\draw[thick,green] (3,-1.5) to (-3,2) ;
\draw[thick,green] (-3,2) to (-3,4) ;
\draw[thick,green,->] (3,2) to (3,4) ;
\draw[thick, purple,->] (1,3.5) arc (0:360:1);
\end{tikzpicture}
\stackrel{\eqref{extendedreln1},\eqref{newbub}}{=}
(-1)^{n+1}\epsilon_{ji}\varphi_{ji}
\begin{tikzpicture}[anchorbase,scale=.2]
\draw[thick,green,->] (-3,-4) to (-3,4) ;
\draw[thick,green,<-] (3,-4) to (3,4) ;
\end{tikzpicture} = -\id_{\FE_i} \, .
\end{align*}
It follows from Lemma \ref{L:Ben's-linear-algebra-lemma} that $\tauast \inc = \inc$.
\end{proof}

\begin{lem}
	\label{lem:prinplusminus}
The maps 
\[
\pr_{-1}:=
\begin{tikzpicture}[anchorbase,scale=.2]
\draw[thick,green] (3,-3) to (2,1) ;
\draw[thick,green,<-] (-3,-3) to (-2,1) ;
\draw[thick,green,<-] (2,-3) to (-2,1) ;
\draw[thick,green] (-2,-3) to (2,1) ;
\draw[ultra thick,green,->] (2,1) to (2,3) ;
\draw[ultra thick,green] (-2,1) to (-2,3) ;
\draw[thick, purple,->] (1,-3) arc (0:180:1);
\end{tikzpicture}
\quad \text{and} \quad 
\pr_{+1} :=
\begin{tikzpicture}[anchorbase,scale=.2]
\draw[thick,green] (3,-3) to (2,1) ;
\draw[thick,green,<-] (-3,-3) to (-2,1) ;
\draw[thick,green,<-] (2,-3) to (-2,1) ;
\draw[thick,green] (-2,-3) to (2,1) ;
\draw[ultra thick,green,->] (2,1) to (2,3) ;
\draw[ultra thick,green] (-2,1) to (-2,3) ;
\draw[thick, purple,->] (1,-3) arc (0:180:1) ;
\node at (0,-2) {\scs$\bullet$};
\end{tikzpicture}
\]
descend to a basis for $V^{-1}(\FE_i\FE_j\FE_i, \FE_i^{(2)})$ 
and $V^{+1}(\FE_i\FE_j\FE_i, \FE^{(2)})$, respectively. 
Flipping each of $\pr_{-1}$ and $\pr_{+1}$ upside down and reversing orientation
yields maps $\inc_{-1}$ and $\inc_{+1}$ that descend to a basis for 
$V^{-1}(\FE_i^{(2)}, \FE_i\FE_j\FE_i)$ and $V^{+1}(\FE^{(2)}, \FE_i\FE_j\FE_i)$, respectively. 
\end{lem}
\begin{proof}
We again use Lemma \ref{lem:basisforV}. 
The projections realized by the isomorphisms 
in the proof of Lemma \ref{lem:X1X2X1decomposition} are
\[
\begin{tikzpicture}[anchorbase,scale=.2]
\draw[thick,green] (3,-3) to (2,1) ;
\draw[thick,green,<-] (-3,-3) to (-2,1) ;
\draw[thick,green,<-] (2,-3) to (-2,1) ;
\draw[thick,green] (-2,-3) to (2,1) ;
\draw[ultra thick,green,->] (2,1) to (2,3) ;
\draw[ultra thick,green] (-2,1) to (-2,3) ;
\draw[thick,purple] (1,-3) to (2,-1) ;
\draw[thick,purple,<-] (-1,-3) to (-2,-1) ;
\draw[thick, purple] (2,-1) arc (0:180:2);
\end{tikzpicture}
\quad \text{and} \quad 
\begin{tikzpicture}[anchorbase,scale=.2]
\draw[thick,green] (3,-3) to (2,1) ;
\draw[thick,green,<-] (-3,-3) to (-2,1) ;
\draw[thick,green,<-] (2,-3) to (-2,1) ;
\draw[thick,green] (-2,-3) to (2,1) ;
\draw[ultra thick,green,->] (2,1) to (2,3) ;
\draw[ultra thick,green] (-2,1) to (-2,3) ;
\draw[thick,purple] (1,-3) to (2,-1) ;
\draw[thick,purple,<-] (-1,-3) to (-2,-1) ;
\draw[thick, purple] (2,-1) arc (0:180:2);
\node at (0,1) {\scs$\bullet$};
\end{tikzpicture} \quad .
\]
Applying \eqref{dotslide} to the second diagram, 
and then using \eqref{eq:cubicKLR} and \eqref{eq:quadKLR}, 
we obtain the indicated maps. 
The argument for $\inc_{\pm1}$ is similar, 
or follows from the computation that 
$\pr_{\pm1}\circ \inc_{\mp1}= (-1)^n\id_{\FE^{(2)}}$ 
given in the proof of Lemma \ref{L:tau-star-Serre2} below.
\end{proof}

\begin{lem}\label{L:tau-star-Serre2}
The equalities
$\tauast \inc_{\pm 1}=(-1)^n\inc_{\pm 1}$
hold in $V^{+1}(\FE^{(2)}, \FE_i\FE_j\FE_i)$. 
\end{lem}

\begin{proof}
We begin by computing
\begin{align*}
\pr_{+1}\circ \inc_{-1}&= 
\begin{tikzpicture}[anchorbase,scale=.2]
\draw[thick,green] (4,0) to (2,4) ;
\draw[thick,green] (-4,0) to (-2,4) ;
\draw[thick,green] (2,0) to (-2,4) ;
\draw[thick,green] (-2,0) to (2,4) ;
\draw[ultra thick,green,->] (2,4) to (2,6) ;
\draw[ultra thick,green] (-2,4) to (-2,6) ;
\draw[thick, purple, ->] (1,0) arc (0:360:1);
\node at (0,1) {\scs$\bullet$};
\draw[thick,green] (4,0) to (2,-4) ;
\draw[thick,green] (-4,0) to (-2,-4) ;
\draw[thick,green] (2,0) to (-2,-4) ;
\draw[thick,green] (-2,0) to (2,-4) ;
\draw[ultra thick,green] (2,-4) to (2,-6) ;
\draw[ultra thick,green,->] (-2,-4) to (-2,-6) ;
\end{tikzpicture}
\stackrel{\eqref{eq:mixedef}, \eqref{dotslide}}{=}
\begin{tikzpicture}[anchorbase,scale=.2]
\draw[thick,green] (4,0) to (2,4) ;
\draw[thick,green] (-4,0) to (-2,4) ;
\draw[thick,green] (2,0) to (-2,4) ;
\draw[thick,green] (-2,0) to (2,4) ;
\draw[ultra thick,green,->] (2,4) to (2,6) ;
\draw[ultra thick,green] (-2,4) to (-2,6) ;
\draw[thick, purple, ->] (-1,0) arc (0:360:1);
\filldraw [black] (-2,-1) circle (6pt) ;
\draw[thick,green] (4,0) to (2,-4) ;
\draw[thick,green] (-4,0) to (-2,-4) ;
\draw[thick,green] (2,0) to (-2,-4) ;
\draw[thick,green] (-2,0) to (2,-4) ;
\draw[ultra thick,green] (2,-4) to (2,-6) ;
\draw[ultra thick,green,->] (-2,-4) to (-2,-6) ;
\end{tikzpicture}
\stackrel{\eqref{eq:quadKLR}}{=}
\epsilon_{ji}
\begin{tikzpicture}[anchorbase,scale=.2]
\draw[thick,green] (4,0) to (2,4) ;
\draw[thick,green] (-4,0) to (-2,4) ;
\draw[thick,green] (2,0) to (-2,4) ;
\draw[thick,green] (0,0) to (2,4) ;
\draw[ultra thick,green,->] (2,4) to (2,6) ;
\draw[ultra thick,green] (-2,4) to (-2,6) ;
\draw[thick, purple, ->] (-1,0) arc (0:360:1);
\filldraw [black] (-2,-1) circle (6pt) ;
\draw[thick,green] (4,0) to (2,-4) ;
\draw[thick,green] (-4,0) to (-2,-4) ;
\draw[thick,green] (2,0) to (-2,-4) ;
\draw[thick,green] (0,0) to (2,-4) ;
\draw[ultra thick,green] (2,-4) to (2,-6) ;
\draw[ultra thick,green,->] (-2,-4) to (-2,-6) ;
\filldraw [black] (-2,1) circle (6pt) ;
\end{tikzpicture}
-\epsilon_{ji}
\begin{tikzpicture}[anchorbase,scale=.2]
\draw[thick,green] (4,0) to (2,4) ;
\draw[thick,green] (-4,0) to (-2,4) ;
\draw[thick,green] (2,0) to (-2,4) ;
\draw[thick,green] (0,0) to (2,4) ;
\draw[ultra thick,green,->] (2,4) to (2,6) ;
\draw[ultra thick,green] (-2,4) to (-2,6) ;
\draw[thick, purple, ->] (-1,0) arc (0:360:1);
\filldraw [black] (-2,-1) circle (6pt) ;
\draw[thick,green] (4,0) to (2,-4) ;
\draw[thick,green] (-4,0) to (-2,-4) ;
\draw[thick,green] (2,0) to (-2,-4) ;
\draw[thick,green] (0,0) to (2,-4) ;
\draw[ultra thick,green] (2,-4) to (2,-6) ;
\draw[ultra thick,green,->] (-2,-4) to (-2,-6) ;
\filldraw [black] (0,0) circle (6pt) ;
\end{tikzpicture} \\
&\stackrel{\eqref{extendedreln2}, \eqref{dotslide}}{=}
\epsilon_{ji}
\begin{tikzpicture}[anchorbase,scale=.2]
\draw[thick,green] (5,0) to (3,4) ;
\draw[thick,green] (-5,0) to (-3,4) ;
\draw[thick,green] (-1,0) to (-3,4) ;
\draw[thick,green] (1,0) to (3,4) ;
\draw[ultra thick,green,->] (3,4) to (3,6) ;
\draw[ultra thick,green] (-3,4) to (-3,6) ;
\draw[thick, purple, ->] (-2,0) arc (0:360:1);
\filldraw [black] (-3,-1) circle (6pt) ;
\filldraw [black] (-3,1) circle (6pt) ;
\draw[thick,green] (5,0) to (3,-4) ;
\draw[thick,green] (-5,0) to (-3,-4) ;
\draw[thick,green] (-1,0) to (-3,-4) ;
\draw[thick,green] (1,0) to (3,-4) ;
\draw[ultra thick,green] (3,-4) to (3,-6) ;
\draw[ultra thick,green,->] (-3,-4) to (-3,-6) ;
\end{tikzpicture}
-\epsilon_{ji}
\begin{tikzpicture}[anchorbase,scale=.2]
\draw[thick,green] (5,0) to (3,4) ;
\draw[thick,green] (-5,0) to (-3,4) ;
\draw[thick,green] (-1,0) to (-3,4) ;
\draw[thick,green] (1,0) to (3,4) ;
\draw[ultra thick,green,->] (3,4) to (3,6) ;
\draw[ultra thick,green] (-3,4) to (-3,6) ;
\draw[thick, purple, ->] (-2,0) arc (0:360:1);
\filldraw [black] (-3,-1) circle (6pt) ;
\draw[thick,green] (5,0) to (3,-4) ;
\draw[thick,green] (-5,0) to (-3,-4) ;
\draw[thick,green] (-1,0) to (-3,-4) ;
\draw[thick,green] (1,0) to (3,-4) ;
\draw[ultra thick,green] (3,-4) to (3,-6) ;
\draw[ultra thick,green,->] (-3,-4) to (-3,-6) ;
\filldraw [black] (1,0) circle (6pt) ;
\end{tikzpicture}
+\epsilon_{ji}
\begin{tikzpicture}[anchorbase,scale=.2]
\draw[thick,green] (5,0) to (3,4) ;
\draw[thick,green] (-5,0) to (-3,4) ;
\draw[ultra thick,green,->] (3,4) to (3,6) ;
\draw[ultra thick,green] (-3,4) to (-3,6) ;
\draw[thick, purple, ->] (-2,0) arc (0:360:1);
\filldraw [black] (-3,-1) circle (6pt) ;
\draw[thick,green] (5,0) to (3,-4) ;
\draw[thick,green] (-5,0) to (-3,-4) ;
\draw[ultra thick,green] (3,-4) to (3,-6) ;
\draw[ultra thick,green,->] (-3,-4) to (-3,-6) ;
\draw[thick,green] (-3,4) to (3,4) ;
\draw[thick,green] (-3,-4) to (1,-1) ;
\draw[thick,green] (3,-4) to (-1,-1) ;
\draw[thick, green, <-] (1,-1) arc (0:180:1);
\end{tikzpicture} \\
&\stackrel{{\eqref{eq:mixedef}, \eqref{dotslide},\eqref{eq:curl1}}}{=}
-\epsilon_{ji}
\begin{tikzpicture}[anchorbase,scale=.2]
\draw[thick,green] (-5,0) to (-3,4) ;
\draw[thick,green] (-1,0) to (-3,4) ;
\draw[ultra thick,green,->] (3,-6) to (3,6) ;
\draw[ultra thick,green] (-3,4) to (-3,6) ;
\draw[thick, purple, ->] (0,0) arc (0:360:1);
\filldraw [black] (-1,-1) circle (6pt) ;
\draw[thick,green] (-5,0) to (-3,-4) ;
\draw[thick,green] (-1,0) to (-3,-4) ;
\draw[ultra thick,green,->] (-3,-4) to (-3,-6) ;
\end{tikzpicture}
\stackrel{\eqref{eq:quadKLR}}{=}
-\epsilon_{ji}^2
\begin{tikzpicture}[anchorbase,scale=.2]
\draw[thick,green] (-5,0) to (-3,4) ;
\draw[thick,green] (-1,0) to (-3,4) ;
\draw[ultra thick,green,->] (3,-6) to (3,6) ;
\draw[ultra thick,green] (-3,4) to (-3,6) ;
\draw[thick, purple, ->] (2,0) arc (0:360:1);
\filldraw [black] (1,1) circle (6pt) ;
\filldraw [black] (1,-1) circle (6pt) ;
\draw[thick,green] (-5,0) to (-3,-4) ;
\draw[thick,green] (-1,0) to (-3,-4) ;
\draw[ultra thick,green,->] (-3,-4) to (-3,-6) ;
\end{tikzpicture}
+\epsilon_{ji}^2
\begin{tikzpicture}[anchorbase,scale=.2]
\draw[thick,green] (-5,0) to (-3,4) ;
\draw[thick,green] (-1,0) to (-3,4) ;
\draw[ultra thick,green,->] (3,-6) to (3,6) ;
\draw[ultra thick,green] (-3,4) to (-3,6) ;
\draw[thick, purple, ->] (2,0) arc (0:360:1);
\draw[thick,green] (-5,0) to (-3,-4) ;
\draw[thick,green] (-1,0) to (-3,-4) ;
\draw[ultra thick,green,->] (-3,-4) to (-3,-6) ;
\filldraw [black] (1,-1) circle (6pt) ;
\filldraw [black] (-1,0) circle (6pt) ;
\end{tikzpicture} \\
&\stackrel{\eqref{eq:quadKLR}}{=}
\begin{tikzpicture}[anchorbase,scale=.2]
\draw[ultra thick,green,->] (3,-6) to (3,6) ;
\draw[ultra thick,green,<-] (-3,-6) to (-3,6) ;
\draw[thick, purple, ->] (1,0) arc (0:360:1);
\filldraw [black] (0,-1) circle (6pt) ;
\end{tikzpicture}
\stackrel{\eqref{newbub}}{=}(-1)^n\id_{\FE^{(2)}}.
\end{align*} 
Note, for use below,
that the same computation shows that $\pr_{-1}\circ \inc_{+1}= (-1)^n\id_{\FE^{(2)}}$. 

It follows from equations \eqref{eq:inv} and \eqref{eq:tauonMS} and Theorem \ref{thm:involution} 
that $\tau(\inc_{-1})$ is obtained from $\inc_{-1}$ by reversing orientation and multiplying by $+1$.
Now, we compute 
\begin{align*}
\pr_{+1}\circ (\tauast \inc_{-1}) &= 
\begin{tikzpicture}[anchorbase,scale=.2]
\draw[thick,green] (-3,-5) to (5,1) ;
\draw[thick,green] (3,-5) to (-5,1) ;
\draw[thick,green] (-3,-5) to (-5,-1) ;
\draw[thick,green] (3,-5) to (5,-1) ;
\draw[thick,green] (-3,5) to (5,-1) ;
\draw[thick,green] (3,5) to (-5,-1) ;
\draw[thick,green] (-3,5) to (-5,1) ;
\draw[thick,green] (3,5) to (5,1) ;
\draw[thick,purple] (-1,-1) to (1,1) ;
\draw[thick,purple] (-1,1) to (1,-1) ;
\draw[thick, purple, ->] (1,1) arc (0:180:1);
\draw[thick, purple] (1,-1) arc (0:-180:1);
\filldraw [black] (0,2) circle (6pt) ;
\draw[ultra thick,green] (3,-5) to (3,-6) ;
\draw[ultra thick,green,<-] (-3,-5) to (-3,-6) ;
\draw[ultra thick,green,->] (3,5) to (3,6) ;
\draw[ultra thick,green,<-] (-3,5) to (-3,6) ;
\draw[ultra thick,green,->] (3,-6) to (-3,-8) ;
\draw[ultra thick,green] (-3,-6) to (3,-8) ;
\end{tikzpicture}
\stackrel{\eqref{eq:curl1}}{=}
-
\begin{tikzpicture}[anchorbase,scale=.2]
\draw[thick,green] (-3,-5) to (5,1) ;
\draw[thick,green] (3,-5) to (-5,1) ;
\draw[thick,green] (-3,-5) to (-5,-1) ;
\draw[thick,green] (3,-5) to (5,-1) ;
\draw[thick,green] (-3,5) to (5,-1) ;
\draw[thick,green] (3,5) to (-5,-1) ;
\draw[thick,green] (-3,5) to (-5,1) ;
\draw[thick,green] (3,5) to (5,1) ;
\draw[thick, purple, ->] (1,1.2) arc (0:360:1);
\draw[thick, purple,->] (1,-1.2) arc (0:-360:1);
\filldraw [black] (0,2.2) circle (6pt) ;
\filldraw [black] (0,-.2) circle (6pt) node[anchor=east]{\tiny{${\spadesuit+0}$}} ;
\draw[ultra thick,green] (3,-5) to (3,-6) ;
\draw[ultra thick,green,<-] (-3,-5) to (-3,-6) ;
\draw[ultra thick,green,->] (3,5) to (3,6) ;
\draw[ultra thick,green,<-] (-3,5) to (-3,6) ;
\draw[ultra thick,green,->] (3,-6) to (-3,-8) ;
\draw[ultra thick,green] (-3,-6) to (3,-8) ;
\end{tikzpicture}
\stackrel{\eqref{newbub},\eqref{eq:mixedef},\eqref{dotslide}}{=}
-(-1)^{n-1}
\begin{tikzpicture}[anchorbase,scale=.2]
\draw[thick,green] (-3,-5) to (5,1) ;
\draw[thick,green] (3,-5) to (-5,1) ;
\draw[thick,green] (-3,-5) to (-5,-1) ;
\draw[thick,green] (3,-5) to (5,-1) ;
\draw[thick,green] (-3,5) to (5,-1) ;
\draw[thick,green] (3,5) to (-5,-1) ;
\draw[thick,green] (-3,5) to (-5,1) ;
\draw[thick,green] (3,5) to (5,1) ;
\draw[thick, purple, ->] (-1,1) arc (0:360:1);
\filldraw [black] (-2,2) circle (6pt) ;
\draw[ultra thick,green] (3,-5) to (3,-6) ;
\draw[ultra thick,green,<-] (-3,-5) to (-3,-6) ;
\draw[ultra thick,green,->] (3,5) to (3,6) ;
\draw[ultra thick,green,<-] (-3,5) to (-3,6) ;
\draw[ultra thick,green,->] (3,-6) to (-3,-8) ;
\draw[ultra thick,green] (-3,-6) to (3,-8) ;
\end{tikzpicture} \\
&\stackrel{\eqref{eq:quadKLR}}{=}
(-1)^n\epsilon_{ji}
\begin{tikzpicture}[anchorbase,scale=.2]
\draw[thick,green] (-3,-5) to (5,1) ;
\draw[thick,green] (3,-5) to (-5,1) ;
\draw[thick,green] (-3,-5) to (-5,-1) ;
\draw[thick,green] (3,-5) to (5,-1) ;
\draw[thick,green] (-3,5) to (5,-1) ;
\draw[thick,green] (3,5) to (-5,-1) ;
\draw[thick,green] (-3,5) to (-5,1) ;
\draw[thick,green] (3,5) to (5,1) ;
\draw[thick, purple, ->] (-1.5,2.75) arc (0:360:1);
\filldraw [black] (-2.5,3.75) circle (6pt);
\filldraw [black] (-2.5,1.75) circle (6pt);
\draw[ultra thick,green] (3,-5) to (3,-6) ;
\draw[ultra thick,green,<-] (-3,-5) to (-3,-6) ;
\draw[ultra thick,green,->] (3,5) to (3,6) ;
\draw[ultra thick,green,<-] (-3,5) to (-3,6) ;
\draw[ultra thick,green,->] (3,-6) to (-3,-8) ;
\draw[ultra thick,green] (-3,-6) to (3,-8) ;
\end{tikzpicture}
-(-1)^n\epsilon_{ji}
\begin{tikzpicture}[anchorbase,scale=.2]
\draw[thick,green] (-3,-5) to (5,1) ;
\draw[thick,green] (3,-5) to (-5,1) ;
\draw[thick,green] (-3,-5) to (-5,-1) ;
\draw[thick,green] (3,-5) to (5,-1) ;
\draw[thick,green] (-3,5) to (5,-1) ;
\draw[thick,green] (3,5) to (-5,-1) ;
\draw[thick,green] (-3,5) to (-5,1) ;
\draw[thick,green] (3,5) to (5,1) ;
\draw[thick, purple, ->] (-1.5,2.75) arc (0:360:1);
\filldraw [black] (-2.5,3.75) circle (6pt) ;
\draw[ultra thick,green] (3,-5) to (3,-6) ;
\draw[ultra thick,green,<-] (-3,-5) to (-3,-6) ;
\draw[ultra thick,green,->] (3,5) to (3,6) ;
\draw[ultra thick,green,<-] (-3,5) to (-3,6) ;
\draw[ultra thick,green,->] (3,-6) to (-3,-8) ;
\draw[ultra thick,green] (-3,-6) to (3,-8) ;
\filldraw [black] (-1,2) circle (6pt) ;
\end{tikzpicture}
\stackrel{\eqref{eq:otherthick},\eqref{eq:quadKLR},\eqref{dotslide}}{=}
(-1)^{n+1}\epsilon_{ji}
\begin{tikzpicture}[anchorbase,scale=.2]
\draw[thick,green] (3,-5) to (-5,1) ;
\draw[thick,green] (3,-5) to (5,-1) ;
\draw[thick,green] (-3,5) to (5,-1) ;
\draw[thick,green] (-3,5) to (-5,1) ;
\draw[thick, purple, ->] (-1.5,2.75) arc (0:360:1);
\filldraw [black] (-2.5,3.75) circle (6pt) ;
\draw[ultra thick,green] (3,-5) to (3,-6) ;
\draw[ultra thick,green] (-3,-5) to (-3,-6) ;
\draw[ultra thick,green,->] (3,5) to (3,6) ;
\draw[ultra thick,green,<-] (-3,5) to (-3,6) ;
\draw[ultra thick,green,->] (3,-6) to (-3,-8) ;
\draw[ultra thick,green] (-3,-6) to (3,-8) ;
\draw[ultra thick,green] (-3,-5) to (3,5) ;
\end{tikzpicture} \\
&\stackrel{\eqref{eq:otherthick},\eqref{eq:StosicFE}}{=}
(-1)^{n+1}\epsilon_{ji}
\begin{tikzpicture}[anchorbase,scale=.2]
\draw[thick,green] (-3,5) to (-1,1) ;
\draw[thick,green] (-3,5) to (-5,1) ;
\draw[thick,green] (-5,1) to (-3,-3) ;
\draw[thick,green] (-1,1) to (-3,-3) ;
\draw[thick, purple, ->] (-2,1) arc (0:360:1);
\filldraw [black] (-3,2) circle (6pt) ;
\draw[ultra thick,green,->] (3,-6) to (3,6) ;
\draw[ultra thick,green,->] (-3,6) to (-3,5) ;
\draw[ultra thick,green,->] (-3,-3) to (-3,-6) ;
\end{tikzpicture}
\stackrel{\eqref{eq:mixedef},\eqref{dotslide}}{=}
(-1)^{n+1}\epsilon_{ji}
\begin{tikzpicture}[anchorbase,scale=.2]
\draw[thick,green] (-3,5) to (-1,1) ;
\draw[thick,green] (-3,5) to (-5,1) ;
\draw[thick,green] (-5,1) to (-3,-3) ;
\draw[thick,green] (-1,1) to (-3,-3) ;
\draw[thick, purple, ->] (0,1) arc (0:360:1);
\filldraw [black] (-1,2) circle (6pt) ;
\draw[ultra thick,green,->] (3,-6) to (3,6) ;
\draw[ultra thick,green,->] (-3,6) to (-3,5) ;
\draw[ultra thick,green,->] (-3,-3) to (-3,-6) ;
\end{tikzpicture} \\
&\stackrel{\eqref{eq:quadKLR},\eqref{dotslide}}{=}
-(-1)^{n+1}\epsilon_{ji}^2
\begin{tikzpicture}[anchorbase,scale=.2]
\draw[thick, purple, ->] (1,1) arc (0:360:1);
\filldraw [black] (0,2) circle (6pt) ;
\draw[ultra thick,green,->] (3,-6) to (3,6) ;
\draw[ultra thick,green,->] (-3,6) to (-3,-6) ;
\end{tikzpicture}
\stackrel{\eqref{newbub}}{=}
-(-1)^{n+1}(-1)^n
\begin{tikzpicture}[anchorbase,scale=.2]
\draw[ultra thick,green,->] (3,-6) to (3,6) ;
\draw[ultra thick,green,->] (-3,6) to (-3,-6) ;
\end{tikzpicture}
=\id_{\FE^{(2)}} \, .
\end{align*}
Lemma \ref{L:Ben's-linear-algebra-lemma} then implies that $\tauast \inc_{-1} = (-1)^n\inc_{-1}$.

The argument for $\tauast \inc_{+1}$ is similar.
We pair the equality $\pr_{-1}\circ \inc_{+1} = (-1)^n\id_{\FE^{(2)}}$
established above with the computation of $\pr_{-1}\circ (\tauast \inc_{+1})$
and apply Lemma \ref{L:Ben's-linear-algebra-lemma}.
The computation of $\pr_{-1}\circ (\tauast \inc_{+1})$ is exactly the same as for $\pr_{+1}\circ (\tauast \inc_{-1})$ 
except that a $-1$ factor appears when computing $\tau(\inc_{+1})$ due to the presence of the dot, 
and we see
\[
\begin{tikzpicture}[anchorbase,scale=.2]
\draw[thick,purple] (-1,-1) to (1,1) ;
\draw[thick,purple] (-1,1) to (1,-1) ;
\draw[thick, purple, ->] (1,1) arc (0:180:1);
\draw[thick, purple] (1,-1) arc (0:-180:1);
\filldraw [black] (0,-2) circle (6pt) ;
\end{tikzpicture}
\qquad \text{instead of} \qquad 
\begin{tikzpicture}[anchorbase,scale=.2]
\draw[thick,purple] (-1,-1) to (1,1) ;
\draw[thick,purple] (-1,1) to (1,-1) ;
\draw[thick, purple, ->] (1,1) arc (0:180:1);
\draw[thick, purple] (1,-1) arc (0:-180:1);
\filldraw [black] (0,2) circle (6pt) ;
\end{tikzpicture} \, . 
\]
Regardless, we find that $\pr_{-1}\circ (\tauast \inc_{+1}) = \id_{\FE^{(2)}}$
so $\tauast \inc_{+1}=(-1)^n\inc_{+1}$.
\end{proof}

Similar arguments to the proofs of Lemmata \ref{lem:prinplusminus} 
and \ref{lem:prin} give the following.

\begin{lem}\label{lem:FE2FE}
The maps
\[
\pr_{\FE^{(2)}_i\FE_j}:=
\begin{tikzpicture}[anchorbase,scale=.25]
\draw[ultra thick,green] (-1.5,.5) to (-1.5,2);
\draw[thick,green,<-] (-2.5,-2) to [out=90,in=210] (-1.5,.5);
\draw[thick,green,<-] (1.5,-2) to [out=90,in=330] (-1.5,.5);
\draw[ultra thick,green,->] (-.5,.5) to (-.5,2);
\draw[thick,green] (-1.5,-2) to [out=90,in=210] (-.5,.5);
\draw[thick,green] (2.5,-2) to [out=90,in=330] (-.5,.5);
\draw[thick,purple,<-] (-.5,-2) to [out=90,in=270] (.5,2);
\draw[thick,purple,->] (.5,-2) to [out=90,in=270] (1.5,2);
\end{tikzpicture}
\quad \text{and} \quad
\pr_{\FE_j\FE^{(2)}_i}:=
\begin{tikzpicture}[anchorbase,scale=.25,xscale=-1]
\draw[ultra thick,green,->] (-1.5,.5) to (-1.5,2);
\draw[thick,green] (-2.5,-2) to [out=90,in=210] (-1.5,.5);
\draw[thick,green] (1.5,-2) to [out=90,in=330] (-1.5,.5);
\draw[ultra thick,green] (-.5,.5) to (-.5,2);
\draw[thick,green,<-] (-1.5,-2) to [out=90,in=210] (-.5,.5);
\draw[thick,green,<-] (2.5,-2) to [out=90,in=330] (-.5,.5);
\draw[thick,purple,->] (-.5,-2) to [out=90,in=270] (.5,2);
\draw[thick,purple,<-] (.5,-2) to [out=90,in=270] (1.5,2);
\end{tikzpicture}
\]
descend to a bases for 
$V(\FE_i\FE_j\FE_i, \FE^{(2)}_i\FE_j)$ and $V(\FE_i\FE_j\FE_i,\FE_j\FE^{(2)}_i)$. 
Flipping these diagrams upside down and reversing orientation yields maps, 
denoted $\inc_{\FE^{(2)}_i\FE_j}$ and $\inc_{\FE_j\FE^{(2)}_i}$, 
which descends to bases for 
$V(\FE^{(2)}_i\FE_j, \FE_i\FE_j\FE_i)$ and $V(\FE_j\FE^{(2)}_i, \FE_i\FE_j\FE_i)$. \qed
\end{lem}

\begin{lem}\label{L:tau-star-Serre3}
The equalities
\[
\tauast \inc_{\FE^{(2)}_i\FE_j} = -\inc_{\FE^{(2)}_i\FE_j}
\quad \text{and} \quad
\tauast \inc_{\FE_j\FE^{(2)}_i} = -\inc_{\FE_j\FE^{(2)}_i}
\]
hold in $V(\FE^{(2)}_i\FE_j, \FE_i\FE_j\FE_i)$ and $V(\FE_j\FE^{(2)}_i, \FE_i\FE_j\FE_i)$, 
respectively.
\end{lem}
\begin{proof}
First, we compute
\begin{align*}
\pr_{\FE^{(2)}_i\FE_j}\circ \inc_{\FE^{(2)}_i\FE_j} &= \begin{tikzpicture}[anchorbase,scale=.2]
\draw[thick,green] (-3,0) to (-2,5) ;
\draw[thick,green] (-2,0) to (-1,5) ;
\draw[thick,green] (2,0) to (-2,5) ;
\draw[thick,green] (3,0) to (-1,5) ;
\draw[thick,purple] (-1,0) to (1,6) ;
\draw[thick,purple,->] (0,0) to (2,6) ;
\draw[ultra thick,green] (-2,5) to (-2,6) ;
\draw[ultra thick,green,->] (-1,5) to (-1,6) ;
\draw[thick,green] (-3,0) to (-2,-5) ;
\draw[thick,green] (-2,0) to (-1,-5) ;
\draw[thick,green] (2,0) to (-2,-5) ;
\draw[thick,green] (3,0) to (-1,-5) ;
\draw[thick,purple,->] (-1,0) to (1,-6) ;
\draw[thick,purple] (0,0) to (2,-6) ;
\draw[ultra thick,green,->] (-2,-5) to (-2,-6) ;
\draw[ultra thick,green] (-1,-5) to (-1,-6) ;
\end{tikzpicture}
\stackrel{\eqref{eq:mixedef},\eqref{eq:quadKLR},\eqref{dotslide}}{=}
\epsilon_{ij}\epsilon_{ji}
\begin{tikzpicture}[anchorbase,scale=.2]
\draw[thick,purple,<-] (2,-6) to (2,6) ;
\draw[thick,purple,->] (3,-6) to (3,6) ;
\draw[ultra thick,green] (0,-6) to (0,-1) ;
\draw[ultra thick,green,->] (0,1) to (0,6) ;
\draw[ultra thick,green,<-] (-1,-6) to (-1,-1) ;
\draw[ultra thick,green] (-1,1) to (-1,6) ;
\draw[thick, green] (1,0) arc (0:360:1);
\draw[thick, green] (0,0) arc (0:360:1);
\filldraw [black] (1,0) circle (6pt) ;
\filldraw [black] (2,0) circle (6pt) ;
\end{tikzpicture}
-\epsilon_{ij}\epsilon_{ji}
\begin{tikzpicture}[anchorbase,scale=.2]
\draw[thick,purple,<-] (2,-6) to (2,6) ;
\draw[thick,purple,->] (3,-6) to (3,6) ;
\draw[ultra thick,green] (0,-6) to (0,-1) ;
\draw[ultra thick,green,->] (0,1) to (0,6) ;
\draw[ultra thick,green,<-] (-1,-6) to (-1,-1) ;
\draw[ultra thick,green] (-1,1) to (-1,6) ;
\draw[thick, green] (1,0) arc (0:360:1);
\draw[thick, green] (0,0) arc (0:360:1);
\filldraw [black] (0,0) circle (6pt) ;
\filldraw [black] (1,0) circle (6pt) ;
\end{tikzpicture}
-\epsilon_{ij}\epsilon_{ji}
\begin{tikzpicture}[anchorbase,scale=.2]
\draw[thick,purple,<-] (2,-6) to (2,6) ;
\draw[thick,purple,->] (3,-6) to (3,6) ;
\draw[ultra thick,green] (0,-6) to (0,-1) ;
\draw[ultra thick,green,->] (0,1) to (0,6) ;
\draw[ultra thick,green,<-] (-1,-6) to (-1,-1) ;
\draw[ultra thick,green] (-1,1) to (-1,6) ;
\draw[thick, green] (1,0) arc (0:360:1);
\draw[thick, green] (0,0) arc (0:360:1);
\filldraw [black] (2,0) circle (6pt) ;
\filldraw [black] (3,0) circle (6pt) ;
\end{tikzpicture}
+\epsilon_{ij}\epsilon_{ji}
\begin{tikzpicture}[anchorbase,scale=.2]
\draw[thick,purple,<-] (2,-6) to (2,6) ;
\draw[thick,purple,->] (3,-6) to (3,6) ;
\draw[ultra thick,green] (0,-6) to (0,-1) ;
\draw[ultra thick,green,->] (0,1) to (0,6) ;
\draw[ultra thick,green,<-] (-1,-6) to (-1,-1) ;
\draw[ultra thick,green] (-1,1) to (-1,6) ;
\draw[thick, green] (1,0) arc (0:360:1);
\draw[thick, green] (0,0) arc (0:360:1);
\filldraw [black] (0,0) circle (6pt) ;
\filldraw [black] (3,0) circle (6pt) ;
\end{tikzpicture} \\
&=
\begin{tikzpicture}[anchorbase,scale=.2]
\draw[thick,purple,<-] (2,-6) to (2,6) ;
\draw[thick,purple,->] (3,-6) to (3,6) ;
\draw[ultra thick,green] (0,-6) to (0,-1) ;
\draw[ultra thick,green,->] (0,1) to (0,6) ;
\draw[ultra thick,green,<-] (-1,-6) to (-1,-1) ;
\draw[ultra thick,green] (-1,1) to (-1,6) ;
\draw[thick, green] (1,0) arc (0:360:1);
\draw[thick, green] (0,0) arc (0:360:1);
\filldraw [black] (0,0) circle (6pt) ;
\filldraw [black] (1,0) circle (6pt) ;
\end{tikzpicture}
\stackrel{\eqref{dotslide},\eqref{extendedreln2}}{=}
\begin{tikzpicture}[anchorbase,scale=.2]
\draw[thick,purple,<-] (4,-6) to (4,6) ;
\draw[thick,purple,->] (5,-6) to (5,6) ;
\draw[ultra thick,green] (-1,-6) to (-1,-1) ;
\draw[ultra thick,green,->] (-1,1) to (-1,6) ;
\draw[ultra thick,green,<-] (2,-6) to (2,-1) ;
\draw[ultra thick,green] (2,1) to (2,6) ;
\draw[thick, green] (3,0) arc (0:360:1);
\draw[thick, green] (0,0) arc (0:360:1);
\filldraw [black] (0,0) circle (6pt) ;
\filldraw [black] (3,0) circle (6pt) ;
\end{tikzpicture}
\stackrel{\eqref{eq:explode}}{=}
-
\begin{tikzpicture}[anchorbase,scale=.2]
\draw[thick,purple,<-] (2,-6) to (2,6) ;
\draw[thick,purple,->] (3,-6) to (3,6) ;
\draw[ultra thick,green,->] (-1,-6) to (-1,6) ;
\draw[ultra thick,green,<-] (0,-6) to (0,6) ;
\end{tikzpicture} \, .
\end{align*}
In the penultimate step, 
we have omitted terms 
which are are quickly seen to be zero using Lemma \ref{lem:btob'inker},
as they contain endomorphisms of $\FE_i^{(2)}$ that factor through $\FE_i$.

Using equation \eqref{eq:inv}, 
we see that $\tau(\inc_{\FE_i^{(2)}\FE_j})$ 
is obtained from $\inc_{\FE_i^{(2)}\FE_j}$ by reversing orientation and 
multipliplying by
\[
\begin{multlined}
(-1)^{v_{ij}(0) + v_{ji}(0) + r_j(\alpha_j) + r_j(\alpha_i+\alpha_j) + v_{ji}(0) + \ell_i(0) + \ell_i(\alpha_j) 
	+ v_{ij}'(\alpha_i + \alpha_j) + v_{ii}(0) + \ell_i(0) + \ell_i(\alpha_i)} \quad \\
\quad \stackrel{\eqref{primev}}{=}(-1)^{\ell_i(\alpha_j)+ r_i(\alpha_i + \alpha_j) + r_i(\alpha_i) + \ell_i(\alpha_i)} 
\stackrel{\eqref{eq:ourfns}}{=} (-1)^{r_i(-\alpha_j) + r_i(\alpha_i + \alpha_j) + r_i(\alpha_i)+ r_i(-\alpha_i)} \, .
\end{multlined}
\]
It follows from Theorem \ref{thm:involution} that this is equal to 
\[
(-1)^{r_i(-\alpha_j) + r_i(\alpha_i + \alpha_j)} = \begin{cases} (-1)^{1+ 0}\quad \text{if} \quad j=i-1 \\
(-1)^{0+1}\quad \text{if} \quad j=i+1,
\end{cases}
\]
which, in either case, is equal to $-1$. 
Thus, 
\begin{align*}
-1\cdot \pr_{\FE_i^{(2)}\FE_j}\circ \tauast \inc_{\FE_i^{(2)}\FE_j} &= \begin{tikzpicture}[anchorbase,scale=.2]
\draw[thick,green] (-3,1) to (-2,5) ;
\draw[thick,green] (-2,1) to (-1,5) ;
\draw[thick,green] (2,1) to (-2,5) ;
\draw[thick,green] (3,1) to (-1,5) ;
\draw[thick,purple] (-1,1) to (1,6) ;
\draw[thick,purple,->] (0,1) to (2,6) ;
\draw[ultra thick,green] (-2,5) to (-2,6) ;
\draw[ultra thick,green,->] (-1,5) to (-1,6) ;
\draw[thick,green] (-3,-1) to (-2,-5) ;
\draw[thick,green] (-2,-1) to (-1,-5) ;
\draw[thick,green] (2,-1) to (-2,-5) ;
\draw[thick,green] (3,-1) to (-1,-5) ;
\draw[thick,purple] (-1,-1) to (1,-6) ;
\draw[thick,purple] (0,-1) to (2,-6) ;
\draw[ultra thick,green] (-2,-5) to (-2,-6) ;
\draw[ultra thick,green] (-1,-5) to (-1,-6) ;
\draw[ultra thick,green] (-2,-6) to (-1,-8) ;
\draw[ultra thick,green,->] (-1,-6) to (-2,-8) ;
\draw[thick,purple] (-1,-1) to (0,1) ;
\draw[thick,purple] (0,-1) to (-1,1) ;
\draw[thick,purple] (2,-8) to (1,-6) ;
\draw[thick,purple,<-] (1,-8) to (2,-6) ;
\draw[thick,green] (3,1) to (2,-1) ;
\draw[thick,green] (3,-1) to (2,1) ;
\draw[thick,green] (-3,1) to (-2,-1) ;
\draw[thick,green] (-3,-1) to (-2,1) ;
\end{tikzpicture} 
\stackrel{\eqref{eq:cubicKLR},\eqref{eq:mixedef}}{=}
\begin{tikzpicture}[anchorbase,scale=.2]
\draw[thick,green] (-3,1) to (-2,5) ;
\draw[thick,green] (-2,1) to (-1,5) ;
\draw[thick,green] (2,1) .. controls (-1,1) ..  (-2,5) ;
\draw[thick,green] (3,1) .. controls (2,4) ..  (-1,5) ;
\draw[thick,purple,->] (-1,0) to (2,6) ;
\draw[thick,purple] (0,0) to (1,6) ;
\draw[ultra thick,green] (-2,5) to (-2,6) ;
\draw[ultra thick,green,->] (-1,5) to (-1,6) ;
\draw[thick,green] (-3,-1) to (-2,-5) ;
\draw[thick,green] (-2,-1) to (-1,-5) ;
\draw[thick,green] (2,-1) to (-2,-5) ;
\draw[thick,green] (3,-1) to (-1,-5) ;
\draw[thick,purple] (-1,0) to (1,-6) ;
\draw[thick,purple] (0,0) to (2,-6) ;
\draw[ultra thick,green] (-2,-5) to (-2,-6) ;
\draw[ultra thick,green] (-1,-5) to (-1,-6) ;
\draw[ultra thick,green] (-2,-6) to (-1,-8) ;
\draw[ultra thick,green,->] (-1,-6) to (-2,-8) ;
\draw[thick,purple] (2,-8) to (1,-6) ;
\draw[thick,purple,<-] (1,-8) to (2,-6) ;
\draw[thick,green] (3,1) to (2,-1) ;
\draw[thick,green] (3,-1) to (2,1) ;
\draw[thick,green] (-3,1) to (-2,-1) ;
\draw[thick,green] (-3,-1) to (-2,1) ;
\end{tikzpicture} 
+\epsilon_{ji}
\begin{tikzpicture}[anchorbase,scale=.2]
\draw[thick,green] (-3,1) to (-2,5) ;
\draw[thick,green] (-2,1) to (-1,5) ;
\draw[thick,green] (2,1) to (-2,5) ;
\draw[thick,green] (3,1) to (-1,5) ;
\draw[ultra thick,green] (-2,5) to (-2,6) ;
\draw[ultra thick,green,->] (-1,5) to (-1,6) ;
\draw[thick,green] (-3,-1) to (-2,-5) ;
\draw[thick,green] (-2,-1) to (-1,-5) ;
\draw[thick,green] (2,-1) to (-2,-5) ;
\draw[thick,green] (3,-1) to (-1,-5) ;
\draw[thick,purple] (-1,-1) to (1,-6) ;
\draw[thick,purple] (0,-1) to (2,-6) ;
\draw[ultra thick,green] (-2,-5) to (-2,-6) ;
\draw[ultra thick,green] (-1,-5) to (-1,-6) ;
\draw[ultra thick,green] (-2,-6) to (-1,-8) ;
\draw[ultra thick,green,->] (-1,-6) to (-2,-8) ;
\draw[thick,purple] (2,-8) to (1,-6) ;
\draw[thick,purple,<-] (1,-8) to (2,-6) ;
\draw[thick,green] (3,1) to (2,-1) ;
\draw[thick,green] (3,-1) to (2,1) ;
\draw[thick,green] (-3,1) to (-2,-1) ;
\draw[thick,green] (-3,-1) to (-2,1) ;
\draw[thick, purple,<-] (2,6) arc (0:-180:1);
\draw[thick, purple] (0,-1) arc (0:180:.5);
\end{tikzpicture}
\stackrel{\eqref{eq:cubicKLR}}{\equiv}
\begin{tikzpicture}[anchorbase,scale=.2]
\draw[thick,green] (-3,1) to (-2,5) ;
\draw[thick,green] (-2,1) to (-1,5) ;
\draw[thick,green] (2,1) .. controls (-1,1) ..  (-2,5) ;
\draw[thick,green] (3,1) .. controls (2,4) ..  (-1,5) ;
\draw[ultra thick,green] (-2,5) to (-2,6) ;
\draw[ultra thick,green,->] (-1,5) to (-1,6) ;
\draw[thick,green] (-3,-1) to (-2,-5) ;
\draw[thick,green] (-2,-1) to (-1,-5) ;
\draw[thick,green] (2,-1) to (-2,-5) ;
\draw[thick,green] (3,-1) to (-1,-5) ;
\draw[thick,purple] (-1,0) to (1,-6) ;
\draw[thick,purple] (1,-2.5) to (2,-6) ;
\draw[thick,purple] (1,-2.5) to (2.5,-1) ;
\draw[thick,purple] (2.5,-1) .. controls (3,0) ..  (2.5,1) ;
\draw[ultra thick,green] (-2,-5) to (-2,-6) ;
\draw[ultra thick,green] (-1,-5) to (-1,-6) ;
\draw[ultra thick,green] (-2,-6) to (-1,-8) ;
\draw[ultra thick,green,->] (-1,-6) to (-2,-8) ;
\draw[thick,purple] (2,-8) to (1,-6) ;
\draw[thick,purple,<-] (1,-8) to (2,-6) ;
\draw[thick,purple,->] (-1,0) to (2,6) ;
\draw[thick,purple] (2.5,1) .. controls (1,2) ..  (0,6) ;
\draw[thick,green] (3,1) to (2,-1) ;
\draw[thick,green] (3,-1) to (2,1) ;
\draw[thick,green] (-3,1) to (-2,-1) ;
\draw[thick,green] (-3,-1) to (-2,1) ;
\end{tikzpicture} 
\stackrel{\eqref{eq:cubicKLR}}{=}
\begin{tikzpicture}[anchorbase,scale=.2]
\draw[thick,green] (-3,1) to (-2,5) ;
\draw[thick,green] (-2,1) to (-1,5) ;
\draw[thick,green] (3,-1) .. controls (-1,1) ..  (-2,5) ;
\draw[thick,green] (3,1) .. controls (2,4) ..  (-1,5) ;
\draw[ultra thick,green] (-2,5) to (-2,6) ;
\draw[ultra thick,green,->] (-1,5) to (-1,6) ;
\draw[thick,green] (-3,-1) to (-2,-5) ;
\draw[thick,green] (-2,-1) to (-1,-5) ;
\draw[thick,green] (2,-1) to (-2,-5) ;
\draw[thick,green] (3,-1) to (-1,-5) ;
\draw[thick,purple] (-1,0) to (1,-6) ;
\draw[thick,purple] (1,-2.5) to (2,-6) ;
\draw[thick,purple] (1,-2.5) to (2.5,-1) ;
\draw[thick,purple] (2.5,-1) to (4.5,2) ;
\draw[ultra thick,green] (-2,-5) to (-2,-6) ;
\draw[ultra thick,green] (-1,-5) to (-1,-6) ;
\draw[ultra thick,green] (-2,-6) to (-1,-8) ;
\draw[ultra thick,green,->] (-1,-6) to (-2,-8) ;
\draw[thick,purple] (2,-8) to (1,-6) ;
\draw[thick,purple,<-] (1,-8) to (2,-6) ;
\draw[thick,purple,->] (-1,0) .. controls (4,1) ..  (2,6) ;
\draw[thick,purple] (4.5,2) .. controls (0,4) ..  (0,6) ;
\draw[thick,green] (3,1) to (2,-1) ;
\draw[thick,green] (-3,1) to (-2,-1) ;
\draw[thick,green] (-3,-1) to (-2,1) ;
\end{tikzpicture} 
\\
&\stackrel{\eqref{eq:mixedef},\eqref{eq:cubicKLR}}{=}
\begin{tikzpicture}[anchorbase,scale=.2]
\draw[thick,green] (-3,1) to (-2,5) ;
\draw[thick,green] (-2,1) to (-1,5) ;
\draw[thick,green] (3,-1) .. controls (-1,1) ..  (-2,5) ;
\draw[thick,green] (-2,-5) .. controls (2,0) ..  (-1,5) ;
\draw[ultra thick,green] (-2,5) to (-2,6) ;
\draw[ultra thick,green,->] (-1,5) to (-1,6) ;
\draw[thick,green] (-3,-1) to (-2,-5) ;
\draw[thick,green] (-2,-1) to (-1,-5) ;
\draw[thick,green] (3,-1) to (-1,-5) ;
\draw[thick,purple] (1,-2.5) to (2,-6) ;
\draw[thick,purple] (1,-2.5) to (2.5,-1) ;
\draw[thick,purple] (2.5,-1) to (4.5,2) ;
\draw[ultra thick,green] (-2,-5) to (-2,-6) ;
\draw[ultra thick,green] (-1,-5) to (-1,-6) ;
\draw[ultra thick,green] (-2,-6) to (-1,-8) ;
\draw[ultra thick,green,->] (-1,-6) to (-2,-8) ;
\draw[thick,purple] (2,-8) to (1,-6) ;
\draw[thick,purple,<-] (1,-8) to (2,-6) ;
\draw[thick,purple,->] (1.5,-1.5) .. controls (2,1) ..  (2,6) ;
\draw[thick,purple] (1.5,-1.5) .. controls (0,-3) .. (1,-6) ;
\draw[thick,purple] (4.5,2) to (0,6) ;
\draw[thick,green] (-3,1) to (-2,-1) ;
\draw[thick,green] (-3,-1) to (-2,1) ;
\end{tikzpicture} 
\stackrel{\eqref{eq:quadKLR},\eqref{dotslide}}{\equiv}
\begin{tikzpicture}[anchorbase,scale=.2]
\draw[thick,green] (-3,1) to (-2,5) ;
\draw[thick,green] (-2,1) to (-1,5) ;
\draw[thick,green] (-1,-5) .. controls (2,-1) ..  (-2,5) ;
\draw[thick,green] (-2,-5) .. controls (2,1) ..  (-1,5) ;
\draw[ultra thick,green] (-2,5) to (-2,6) ;
\draw[ultra thick,green,->] (-1,5) to (-1,6) ;
\draw[thick,green] (-3,-1) to (-2,-5) ;
\draw[thick,green] (-2,-1) to (-1,-5) ;
\draw[ultra thick,green] (-2,-5) to (-2,-6) ;
\draw[ultra thick,green] (-1,-5) to (-1,-6) ;
\draw[ultra thick,green] (-2,-6) to (-1,-8) ;
\draw[ultra thick,green,->] (-1,-6) to (-2,-8) ;
\draw[thick,purple,<-] (1,-8) .. controls (4,0) .. (1,6) ;
\draw[thick,purple,->] (2,-8) .. controls (2,0) ..  (2,6) ;
\draw[thick,green] (-3,1) to (-2,-1) ;
\draw[thick,green] (-3,-1) to (-2,1) ;
\filldraw [black] (1,2) circle (6pt) ;
\filldraw [black] (1,-2) circle (6pt) ;
\end{tikzpicture}
\stackrel{\eqref{eq:otherthick},\eqref{eq:explode},\eqref{eqn:sgnfromdemazure},\eqref{extendedreln2}}{=}
-
\begin{tikzpicture}[anchorbase,scale=.2]
\draw[thick,purple,<-] (1,-8) to (1,6) ;
\draw[thick,purple,->] (2,-8) to  (2,6) ;
\draw[ultra thick,green,<-] (-2,-8) to (-2,6) ;
\draw[ultra thick,green,->] (-1,-8) to (-1,6) ;
\end{tikzpicture} \, .
\end{align*}
It then follows from Lemma \ref{L:Ben's-linear-algebra-lemma} 
that $\tauast \inc_{\FE_i^{(2)}\FE_j} = - \inc_{\FE_i^{(2)}\FE_j}$. 

Similar computations (which we leave as an exercise) 
show that $\tauast \inc_{\FE_j\FE^{(2)}_i} = -\inc_{\FE_j\FE^{(2)}_i}$. 
\end{proof}

\begin{thm}\label{T:iserre-in-twisted-K0}
The following equality holds in $\tauKzero{\EBFoam}$:
\begin{multline*}
[(\FE_i, \swcl_1)]_{\tau} \cdot [(\FE_{i \pm 1}, \swclp_1)]_{\tau} \cdot
	[(\FE_i, \swcl_1)]_{\tau} =
[(\FE_i, \swcl_1)]_{\tau} + (-1)^n[2] [(\FE_i^{(2)}, \swcl_2)]_{\tau} \\
	- [(\FE^{(2)}_i,\swcl_2)]_{\tau}  [(\FE_{i\pm 1}, \swclp_1)]_{\tau}
		- [(\FE_{i\pm 1}, \swclp_1)]_{\tau} [(\FE^{(2)}_i,\swcl_2)]_{\tau} \, .
\end{multline*}
\end{thm}

\begin{proof}
Using Proposition \ref{prop:computeequivGgp}, 
the claim follows from Lemmata 
\ref{lem:X1X2X1decomposition}, 
\ref{L:tau-star-Serre1},
\ref{L:tau-star-Serre2}, 
and \ref{L:tau-star-Serre3}.
\end{proof}

Combining Theorems \ref{T:div-powers-in-twisted-K0} and \ref{T:iserre-in-twisted-K0}, 
we arrive at the following result, 
which establishes the existence of the $\C(q)$-algebra homomorphism appearing 
in Conjecture \ref{conj:introKzero2} from the introduction.

\begin{cor}
	\label{cor:K0(Xm)}
The assignment
$\nx^{(k)}_i \mapsto [(\FE^{(k)}_i, (-1)^{{n+1 \choose 2}+{n-k+1 \choose 2}+k} \swcl_{k})]_{\tau}$
determines a $\C(q)$-algebra homomorphism $U'_{-q^2}(\som[2]) \rightarrow \tauKzero{\EBnFoam}$.
\end{cor}

\begin{proof}
Set 
$\mathbf{x}_i^{(k)} := [(\FE^{(k)}_i, (-1)^{{n+1 \choose 2}+{n-k+1 \choose 2}+k} \swcl_{k})]_{\tau}$
and $\mathbf{x}_i := \mathbf{x}_i^{(1)}$.
In light of Corollary \ref{cor:K0(X2)} and Proposition \ref{prop:iSerreintro},
it suffices to show that
\begin{equation}
	\label{eq:iSerrecat1}
\mathbf{x}_i \mathbf{x}_\ell = \mathbf{x}_\ell \mathbf{x}_i \quad \text{for } |i-\ell| >1
\end{equation}
and
\begin{equation}
	\label{eq:iSerrecat2}
\mathbf{x}_{i} \mathbf{x}_{i\pm1} \mathbf{x}_{i} 
	= \mathbf{x}_{i}^{(2)} \mathbf{x}_{i\pm1} + \mathbf{x}_{i\pm1} \mathbf{x}_{i}^{(2)} 
		+ [2] \mathbf{x}_{i}^{(2)} + \mathbf{x}_{i} \, .
\end{equation}

Equation \eqref{eq:iSerrecat1} follows from the isomorphism
\[
\begin{tikzpicture} [scale=.25,anchorbase]
	\draw[thick,green,<-] (0,0) to [out=90,in=270] (2,3);
	\draw[thick,green,->] (1,0) to [out=90,in=270] (3,3);
	\draw[thick,<-] (2,0) to [out=90,in=270] (0,3);
	\draw[thick,->] (3,0) to [out=90,in=270] (1,3);
\end{tikzpicture} \colon
[(\FE_i, \swcl_{1})]_{\tau} \otimes [(\FE_\ell, \mathrm{LC}_1)]_{\tau} 
\cong
[(\FE_\ell, \mathrm{LC}_1)]_{\tau} \otimes [(\FE_i, \swcl_{1})]_{\tau} \, .
\]
Finally, Theorem \ref{T:iserre-in-twisted-K0} gives that
\begin{multline*}
(-1)^{{n+1 \choose 2}+{n \choose 2}+1} \mathbf{x}_{i} \mathbf{x}_{i\pm1} \mathbf{x}_{i} 
= (-1)^{{n+1 \choose 2}+{n \choose 2}+1} \mathbf{x}_{i}  
	+ (-1)^{n+{n+1 \choose 2}+{n-1 \choose 2}+2} [2] \mathbf{x}_{i}^{(2)} \\
		- (-1)^{{n-1 \choose 2}+2+{n \choose 2}+1} 
			(\mathbf{x}_{i}^{(2)} \mathbf{x}_{i\pm1} + \mathbf{x}_{i\pm1} \mathbf{x}_{i}^{(2)} )
\end{multline*}
which simplifies to \eqref{eq:iSerrecat2}.
\end{proof}

%
\section{Background on link homology}
	\label{s:LH}
%

In this section, we review background material on link homology in type $A$.
We also discuss certain representation-theoretic results which will allow us to 
deduce invariance results for our type $B$ link homologies (defined below in Section \ref{s:SLH}) 
from their type $A$ counterparts.

\subsection{Type $A$ link polynomials via Howe duality}
	\label{ss:Alinkpoly}

In type $A$, work of 
Cautis--Kamnitzer--Licata \cite{CKL} and Cautis--Kamnitzer--Morrison \cite{CKM} 
shows that the quantum $\sln$ link polynomials $P_{\sln}(\mathcal{L}_{\beta}^{\vec{\lambda}})$ 
can be computed using an 
auxiliary quantum group $U_q(\glm)$ associated to the (non-simple) Lie algebra $\glm$. 
This approach proceeds through (a quantization of) the Howe duality between $\glm$ and $\gln$
that we now briefly recall.
Consider the vector space $\Lambda(\C^m \otimes \C^N)$, 
which admits actions of $\glm$ and $\gln$ that generate each others commutant. 
The weight space decomposition for the $\glm$ action (in degree $k$) is given by
\begin{equation}
	\label{eq:skewHowe}
\Lambda^k(\C^m \otimes \C^N) \cong 
\Lambda^k(\underbrace{\C^N \oplus \cdots \oplus \C^N}_{m}) \cong
\bigoplus_{\sum a_i = k} \Lambda^{a_1}(\C^N) \otimes \cdots \otimes \Lambda^{a_m}(\C^N) \, .
\end{equation}
Most importantly for us, the (symmetric) braiding on the $\sln$ modules appearing as summands in the 
right-hand side of \eqref{eq:skewHowe} admits a description in terms of the Weyl group 
action associated with the action of $\glm$ on the left-hand side.
In the quantized setting, this remains true: the braiding on $U_q(\sln)$ modules admits a 
description in terms of the quantum Weyl group action for $U_q(\glm)$.

To formulate quantum skew Howe duality in the form most useful for our considerations, 
we use the idempotent form  $\dU(\glm)$ of the quantum group from Section \ref{ss:idempotent}. 
Recall from Remark \ref{R:dU-is-category} that $\dU(\glm)$ is a category with objects $\wt \in \Z^m$.

\begin{thm}[{\cite[Theorem 4.3]{CKL}, \cite[Theorem 6.2.1]{CKM}}]
	\label{thm:qWG}
There is a full functor 
\[
\SH_m^N \colon \dU(\glm) \to \Rep \big(U_q(\sln) \big)
\]
given on objects by
\begin{equation}
	\label{eq:SH}
\SH_m^N(\wt) = \Lambda^{a_1}(\C^N) \otimes \cdots \otimes \Lambda^{a_m}(\C^N) =: \Lambda^{\wt}(\C^N) \, .
\end{equation}
If we define the \emph{renormalized quantum Weyl group} elements of $\dU(\glm)$ via
\begin{equation}
	\label{eq:qWG}
\QW_i^\pm \one_{\wt} := \sum_{s\ge 0}(-q)^{\pm (s- a_{i+1})} f_i^{(\alpha_i^{\vee}(\wt) + s)} e_i^{(s)}\one_{\wt} \, ,
\end{equation}
then 
\begin{equation}
	\label{eq:SHR}
\SH_m^N( \QW_i^\pm\one_{\wt} ) 
	= (-q^{1/N})^{\mp a_i a_{i+1}} R^{\mp1}_{\Lambda^{a_i}(\C^N),\Lambda^{a_{i+1}}(\C^N)} \, . \qed
\end{equation}
\end{thm}

By convention, $\Lambda^a(\C^N) = 0$ if $a<0$ or $a > N$, so \eqref{eq:qWG} sends $\wt \in \Z^m$ to zero if any 
$a_i < 0$ or if any $a_i > N$. It follows that the functor $\SH_m^N$ factors through the following quotients of $\dU(\glm)$.

\begin{defn}
	\label{def:qSchur}
The \emph{integral $q$-Schur algebra} $\ZdS(\glm)$ is the quotient of $\ZdU(\glm)$ by the ideal generated by 
all weight idempotents $\one_{\wt}$ such that some $a_i< 0$. 
The \emph{integral $N$-bounded quotient} of $\ZdS(\glm)$, denoted $\ZdSn(\glm)$, 
is the quotient of $\ZdS(\glm)$ by the ideal generated by the $\one_{\wt}$ such that some $a_i> N$.
\end{defn}

As in the case of the quantum group, there are non-integral versions of the above
\[
\dS(\glm) := \C(q) \otimes_{\Z[q^{\pm}]} \ZdS(\glm)
\, , \quad
\dSn(\glm) := \C(q) \otimes_{\Z[q^{\pm}]} \ZdSn(\glm)
\]
which can also be defined as quotients of $\dU(\glm)$.
The following now refines the first statement of Theorem \ref{thm:qWG}.

\begin{thm}[{\cite[Theorem 4.4.1]{CKM}}]
	\label{thm:SH}
The functor $\dSn(\glm) \to \Rep \big( U_q(\sln) \big)$ is fully faithful. \qed
\end{thm}

Theorems \ref{thm:qWG} and \ref{thm:SH} now provide a description of the $U_q(\sln)$ link 
polynomial purely in terms of the $N$-bounded quotient of the $q$-Schur algebra.

\begin{prop}
	\label{prop:qSlinkpoly}
There is a unique $\C(q)$-valued bilinear form $(-,-)_N$ on $\dSn(\glm)$ such that 
\begin{enumerate}
\item $(\one_{\wtt} e_i x \one_{\wt}, \one_{\wtt} y \one_{\wt})_N 
	= (\one_{\wtt- \alpha_i} x \one_{\wt}, \one_{\wtt- \alpha_i} f_i y \one_{\wt})_N$,
\item $(\one_{\wtt} f_i x \one_{\wt}, \one_{\wtt} y \one_{\wt})_N 
	= (\one_{\wtt+\alpha_i} x \one_{\wt}, \one_{\wtt+\alpha_i} e_i y \one_{\wt})_N$,
\item $(\one_{\wtt} x e_i\one_{\wt}, \one_{\wtt} y \one_{\wt})_N 
	= (\one_{\wtt} x \one_{\wt+ \alpha_i}, \one_{\wtt} y f_i\one_{\wt+\alpha_i})_N$, 
\item $(\one_{\wtt} x f_i\one_{\wt}, \one_{\wtt} y \one_{\wt})_N
	= (\one_{\wtt} x \one_{\wt- \alpha_i}, \one_{\wtt} y e_i\one_{\wt+\alpha_i})_N$,
\item $(\one_{\wtt} x \one_{\wt}, \one_{\wtt'}y\one_{\wt'})_N = 0$, unless $\wtt= \wtt'$ and $\wt = \wt'$, and
\item $(\one_{\wt}, \one_{\wt})_N = \prod_{i=1}^m{N \brack a_i}$.
\end{enumerate}
Given $x \in \one_{\wt}\dSn(\glm)\one_{\wt}$, this bilinear form satisfies $\trq(\SH_m^N(x))= (\one_{\wt}, x)_N$. 
Consequently, if $\beta = \beta_{i_1}^{\epsilon_1}\cdots \beta_{i_d}^{\epsilon_d} \in \brgroup$
is an $\wt$-balanced braid, then 
\begin{equation}
	\label{eq:SHlp}
\bP_{\sln}(\mathcal{L}_\beta^{\wt}) = 
(-q^{1/N})^{\epsilon(\bgen,\wt)}
\big( \one_{\wt}, \QW_{i_1}^{\epsilon_1} \cdots \QW_{i_d}^{\epsilon_d} \one_{\wt} \big)_N \, .
\end{equation}
Here, $\mathcal{L}_\beta^{\wt}$ denotes the coloring of the braid closure 
$\mathcal{L}_\beta$ determined by $\wt$, 
and $\epsilon(\bgen,\wt)$ is\footnote{As the notation suggests, 
this can be viewed as a colored analogue of the exponent sum of $\beta$.} 
the sum over the crossings in $\beta \one_\wt$ 
of $\pm a b$, where $a$ and $b$ are the labels coloring the strands of the crossing 
and the sign is given by the sign of the crossing.
\end{prop}

\begin{proof}
Properties (1) -- (6) suffice to compute the value of such a bilinear form on $\dSn(\glm)$, 
e.g.~using the triangular decomposition for $\dU(\glm)$ 
or the annular evaluation algorithm from \cite[Theorems 1.2 and 3.2]{QRS}.
It thus suffices to show that such a form exists, and satisfies $\trq(\SH_m^N(x))= (\one_{\wt}, x)_N$.

In fact, we can essentially use the latter as the definition. 
Consider the $\C(q)$-linear anti-algebra involution of $\dU(\glm)$ given by
$\one_{\wtt} e_i \one_{\wt} \mapsto \one_{\wt} f_i \one_{\wtt}$.
This involution descends to $\dSn(\glm)$ and we denote the image of 
$x \in \dSn(\glm)$ under this involution by $\bar{x}$.
For $x,y \in \dSn(\glm)$, set
\[
(x,y)_N := \trq(\SH_m^N(\bar{x} y))
\]
which clearly satisfies (1) and (2).
Properties (3) and (4) hold since $\trq$ is trace-like, 
and (5) then follows from the mutual orthogonality of the weight idempotents $\one_\wt$.
Lastly, (6) holds since $\trq(\SH_m^N(\one_\wt))$ computes the quantum $\sln$ 
invariant of the $\wt$-colored $m$-component unlink, which equals $\prod_{i=1}^m{N \brack a_i}$.
\end{proof}

\begin{rem}
Note that, in light of \eqref{eq:SHR}, 
it is the link polynomials $\bP_{\sln}(\mathcal{L})$ from Remark \ref{rem:Rconv} 
that appear in \eqref{eq:SHlp}.
\end{rem}

\subsection{Rickard complexes and colored $\sln$ link homology}
	\label{ss:Rickard}

We next discuss the categorification of the Howe duality approach to  
$\sln$ link polynomials from \S \ref{ss:Alinkpoly}.
The construction we present is entirely parallel to the decategorified story:
first, one assigns to each braid $\beta \in \Br_m$ an invariant living in (the homotopy category of)
a categorified analogue of the Schur algebra for $\glm$. 
Passing to an $N$-bounded quotient and applying a trace-like functor yields a complex of graded vector spaces, 
whose homology is a link invariant which categorifies the $\sln$ link polynomial.

To begin in detail, we have categorical analogues of the $q$-Schur algebra and its $N$-bounded quotient.

\begin{defn}
	\label{def:CSQ}
Let $\SSS_q(\glm)$ (respectively $\cSSS_q(\glm)$) be the quotient 
of $\UU_q(\glm)$ (respectively $\cUU_q(\glm)$) given 
by evaluating\footnote{A consequence is that this kills all weights $\wt$ such that $a_i <0$ for some $1 \leq i \leq m$, 
since, by convention, symmetric polynomials in alphabets of negative cardinality are the zero ring.}
all formal alphabets $\X_i$ in weight $\wt$ to alphabets of cardinality equal to $a_i$.

Let $\SSSn_q(\glm)$ (respectively $\cSSSn_q(\glm)$) be the (further) 
quotient by the ideal $\mathcal{J}^{\leq N}$ generated by all new bubble morphisms
$\begin{tikzpicture}[anchorbase,scale=1]
\node at (0,0){\NB{\Schur(\X_i)}};
\node at (.75,.125){\scs$\wt$};
\end{tikzpicture}$
wherein the partition $\lambda$ does not fit inside an $a_i \times (N-a_i)$ box.
\end{defn}

\begin{rem}
	\label{rem:size}
One might instead expect $\SSSn_q(\glm)$ and $\cSSSn_q(\glm)$ to be the quotient by $\glm$ weights 
whose entries do not lie strictly between $0$ and $N$.
Such weights are indeed killed in $\SSSn_q(\glm)$ and $\cSSSn_q(\glm)$ 
since $a \times (N-a)$ is not a box in this case (thus we quotient by $\id_{\one_\wt} = \Schur[\emptyset]$).
Our definition is a further quotient which ensures that $\Hom$-spaces in $\cSSSn_q(\glm)$ have the correct size for the computation of 
(undeformed) $\sln$ link homology.
\end{rem}

\begin{rem}
Let $\K$ be a field of characteristic zero. 
As explained in \cite{WebSchur}, the $2$-category $\Kar(\SSS_q(\glm))$ is self-dual mixed; 
see Example \ref{ex:CQGmixed}.
It follows from \cite[Lemma 1.15]{Web4} that $\Kar(\SSSn_q(\glm))$ is self-dual mixed.
\end{rem}

We call the $2$-categories $\SSS_q(\glm)$ and $\SSSn_q(\glm)$,
as well as their (partially) Karoubi completed 
variants,
the \emph{categorified Schur quotient} and the \emph{categorified $N$-bounded Schur quotient}, 
respectively. The appropriateness of this terminology is given by the following.

\begin{thm}
	\label{thm:SchurCategorification}
Let $\K$ be a field. There are $\Z[q^{\pm}]$-algebra isomorphisms
\[
\Kzero{\Kar(\SSS_q(\glm))} \cong \ZdS(\glm)
\, , \quad
\Kzero{\Kar(\SSSn_q(\glm))} \cong \ZdSn(\glm) \, .
\]
\end{thm}

\begin{proof}
The statement for the (categorified) Schur algebra is the main result of \cite{MSV2}. 
The techniques therein can be used to establish the result for the $N$-bounded quotient.
\end{proof}

Before continuing, we 
make precise the final claim from Remark \ref{rem:size}.

\begin{prop}
	\label{prop:size}
Let $\K$ be a commutative ring and 
let $\one_{\wt} X \one_{\wt}$ be a $1$-endomorphism in $\cSSSn_q(\glm)$.
Then, $\Hom_{\cSSSn_q(\glm)}(\one_\wt, X)$ is a free $\K$-module and
if $x \in \dSn(\glm)$ denotes the 
class of $X$ in $\dSn(\glm)$, then
\begin{equation}
	\label{eq:dimformula}
\qdim \big( \qsh^{\sum_{i=1}^m a_i(a_i-N)} \Hom_{\cSSSn_q(\glm)}(\one_\wt, X) \big) 
= (\one_\wt, x)_N \, .
\end{equation}
\end{prop}

\begin{proof}
This is a consequence of the ``annular evaluation algorithm'' from \cite[Theorems 1.2 and 3.2]{QRS}, 
which, in particular, shows that the same algorithm can be used to compute both sides of 
\eqref{eq:dimformula}.

In some detail, first note that it suffices to establish the result when $\K = \Z$.
Further, since
\[
\Hom_{\cSSSn_q(\glm)}(\one_\wt, \qsh^k X) \cong \qsh^k \Hom_{\cSSSn_q(\glm)}(\one_\wt, X)
\]
and
\[
\Hom_{\cSSSn_q(\glm)}(\one_\wt, X_1 \oplus X_2 ) \cong 
\Hom_{\cSSSn_q(\glm)}(\one_\wt, X_1) \oplus \Hom_{\cSSSn_q(\glm)}(\one_\wt, X_2) \, ,
\]
we can assume that $X$ is given as a word in the $1$-morphisms $\EE_i^{(k)}$ and $\FF_j^{(\ell)}$.
By using the isomorphisms
\[
\EE_i^k \cong [k]! \EE_i^{(k)} \quad \text{and} \quad \FF_i^k \cong [k]! \FF_i^{(k)} \, ,
\]
the general case is a straightforward consequence of the case when $X$ is a word in $\EE_i$ and $\FF_j$, 
so we make this additional assumption.

The algorithm uses\footnote{The shift $\qsh^{\sum_{i=1}^m a_i(a_i-N)}$ 
here ensures that the assignment 
$X \mapsto \qsh^{\sum_{i=1}^m a_i(a_i-N)} \Hom_{\cSSSn_q(\glm)}(\one_\wt, X)$ 
is indeed trace-like} 
the following:
\begin{subequations}
	\label{eq:algtools}
\begin{equation}
\qsh^{\sum_{i=1}^m a_i(a_i-N)} \Hom_{\cSSSn_q(\glm)}(\one_\wt, X_1 \one_{\mathbf{b}} X_2) 
		\cong \qsh^{\sum_{i=1}^m b_i(b_i-N)} 
			\Hom_{\cSSSn_q(\glm)}(\one_{\mathbf{b}},X_2 \one_{\wt} X_1) \, 
\end{equation}
\begin{multline}
\Hom_{\cSSSn_q(\glm)}(\one_\wt,X_1 \EE_i \FF_i \one_{\mathbf{b}} X_2) 
		\cong \Hom_{\cSSSn_q(\glm)}(\one_\wt, X_1 \FF_i \EE_i \one_{\mathbf{b}} X_2) \\
			\oplus [b_i {-} b_{i+1}] \Hom_{\cSSSn_q(\glm)}(\one_\wt,X_1 X_2) 
				\quad \text{if} \ b_i {-} b_{i+1} = \langle \ee_i^\vee , \mathbf{b}\rangle \geq 0 \, ,
\end{multline}
\begin{equation}			
\Hom_{\cSSSn_q(\glm)}(\one_\wt, X_1 \EE_i \FF_i \one_{\mathbf{b}} X_2 ) 
		\stackrel{\oplus}{\subset} 
			\Hom_{\cSSSn_q(\glm)}(\one_\wt, X_1 \FF_i \EE_i \one_{\mathbf{b}} X_2)
				\quad \text{if} \ b_i {-} b_{i+1} = \langle \ee_i^\vee , \mathbf{b}\rangle < 0 \, ,
\end{equation}
\end{subequations}
all of which hold integrally, 
to show that $\Hom_{\cSSSn_q(\glm)}(\one_\wt, X)$ is isomorphic 
to a summand of a graded $\Z$-module 
of the form
\[
\bigoplus_{l} \qsh^{k_l + \sum_{i=1}^m a_{l,i}(a_{l,i}-N)} \End_{\cSSSn_q(\glm)}(\one_{\wt_l}) \, .
\]
The algorithm proceeds by inducting on both the length of the word $X$, 
and the minimal weights (with respect to the 
standard partial order on the $\glm$ weight lattice) 
which appear in $X$.
In doing so, it is crucial that we work in the ($N$-bounded) 
Schur quotient, which implies that $\one_\wt \cong 0$ for sufficiently large $\wt$.

Now, \cite[Proposition 3.11]{KL3}, equation \eqref{newbub}, and 
Definition \ref{def:CSQ} give a surjective ring homomorphism
\begin{equation}
	\label{eq:End1=Gr}
H^\ast(\Gr{a_{l,1}}{N}) \otimes \cdots \otimes H^\ast(\Gr{a_{l,m}}{N})
\twoheadrightarrow
\End_{\cSSSn_q(\glm)}(\one_{\wt_{l}}) \, .
\end{equation}
This map must be an isomorphism, 
since, after base change to a field, 
the dimension of the latter can be bounded below 
by $\prod_{j=1}^m q^{a_{l,j}(N-a_{l,j})} {N \brack a_{l,j}}$ 
e.g.~using the $2$-representations of $\cUU_q(\glm)$ 
from \cite{MY} or \cite{Cautis} (which factor through $\cSSSn_q(\glm)$).
Consequently, each 
$\qsh^{\sum_{i=1}^m a_{l,i}(a_{l,i}-N)} \End_{\cSSSn_q(\glm)}(\one_{\wt_{l}})$ 
is a free $\Z$-module 
(of graded rank $\prod_{i=1}^m {N \brack a_{l,i}}$)
so $\Hom_{\cSSSn_q(\glm)}(\one_\wt, X)$ is a summand of a free graded $\Z$-module, 
thus is a free graded $\Z$-module.

Further, this procedure gives a recipe for computing the left-hand side of \eqref{eq:dimformula}.
Indeed, the equations \eqref{eq:algtools} give
\begin{subequations}
	\label{eq:algtools2}
\begin{equation}
\qdim(
\qsh^{\sum_{i=1}^m a_i(a_i-N)} \Hom_{\cSSSn_q(\glm)}(\one_\wt, X_1 \one_{\mathbf{b}} X_2) )
		= \qdim( \qsh^{\sum_{i=1}^m b_i(b_i-N)} 
			\Hom_{\cSSSn_q(\glm)}(\one_{\mathbf{b}},X_2 \one_{\wt} X_1) ) \, 
\end{equation}
\begin{multline}
\qdim(
\Hom_{\cSSSn_q(\glm)}(\one_\wt,X_1 \EE_i \FF_i \one_{\mathbf{b}} X_2) )
	= \Hom_{\cSSSn_q(\glm)}(\one_\wt, X_1 \FF_i \EE_i \one_{\mathbf{b}} X_2) \\
			+ [b_i {-} b_{i+1}] \qdim(\Hom_{\cSSSn_q(\glm)}(\one_\wt,X_1 X_2) )
				\quad \text{if} \ b_i {-} b_{i+1} = \langle \ee_i^\vee , \mathbf{b}\rangle \geq 0 \, ,
\end{multline}
\begin{multline}		
\qdim (\Hom_{\cSSSn_q(\glm)}(\one_\wt, X_1 \EE_i \FF_i \one_{\mathbf{b}} X_2 ) )
		= \qdim (\Hom_{\cSSSn_q(\glm)}(\one_\wt, X_1 \FF_i \EE_i \one_{\mathbf{b}} X_2) ) \\
			- [b_{i+1} {-} b_{i}] \qdim(\Hom_{\cSSSn_q(\glm)}(\one_\wt,X_1 X_2) )
				\quad \text{if} \ b_i {-} b_{i+1} = \langle \ee_i^\vee , \mathbf{b}\rangle < 0 \, ,
\end{multline}
\end{subequations}
which therefore can be used to write
\[
\qdim \big( \qsh^{\sum_{i=1}^m a_i(a_i-N)} \Hom_{\cSSSn_q(\glm)}(\one_\wt, X) \big) 
= \sum_{l=1}^{L_X} r_l q^{d_l} \cdot \qdim 
\big( \qsh^{\sum_{i=1}^m a_{l,i} (a_{l,i}-N)} \End_{\cSSSn_q(\glm)}(\one_{\wt_{l}}) \big) \, .
\]

However, the equations \eqref{eq:algtools2} hold with each instance of 
\[
\qdim(
\qsh^{\sum_{i=1}^m a_i(a_i-N)} \Hom_{\cSSSn_q(\glm)}(\one_\wt, - ) )
\]
replaced by $(\one_{\wt},-)_N$, so the same recipe gives that
\[
(\one_{\wt}, x)_N
= \sum_{l=1}^{L_X} r_l q^{d_l} (\one_{\wt_{l}} , \one_{\wt_{l}} )_N \, .
\]
Since
\[
(\one_{\wt_{l}} , \one_{\wt_{l}} )_N
	= \prod_{i=1}^m {N \brack a_{l,i}}
	= \qdim \big( \qsh^{\sum_{i=1}^m a_{l,i} (a_{l,i}-N)} 
		\End_{\cSSSn_q(\glm)}(\one_{\wt_{l}}) \big) \, ,
\]
the result follows.
\end{proof}

We next discuss the complexes categorifying the quantum Weyl group elements from Definition \ref{eq:qWG} 
defined in \cite{CK,CKL,Cautis}  by Cautis, Kamnitzer, and Licata 
(following pioneering work of Chuang and Rouquier \cite{CR,Rou3}).
Adapted to our present setting of the categorified Schur quotient, 
these complexes are as follows.

\begin{defn}
	\label{def:Rickard}
For $a,b \geq 0$, the \emph{$2$-strand Rickard complex} is the chain complex
\begin{align*}
C_{a,b} &:= \big( \cdots \xrightarrow{d} \qsh^{-k} \tsh^k \FF^{(a-k)} \EE^{(b-k)} \one_{a,b} 
			\xrightarrow{d} \qsh^{-k-1} \tsh^{k+1} \FF^{(a-k-1)} \EE^{(b-k-1)} \one_{a,b} \xrightarrow{d} \cdots \big) \\
		&= \Big( {\textstyle \bigoplus_{k=0}^{\min(a,b)} \qsh^{-k} \tsh^k \FF^{(a-k)} \EE^{(b-k)} \one_{a,b}} \, , \, d \Big) 
			\in \dgCat[{\cSSS_q(\gln[2])}]
\end{align*}
with differential
\begin{equation}
	\label{eq:RickDiff}
d=
\begin{tikzpicture}[anchorbase,scale=1]
\draw[ultra thick,green,rdirected=.65] (-.375,.5) to [out=150,in=270] (-.5,1.125) node[above=-2pt,xshift=-3pt]{\scs$a{-}k{-}1$};
\draw[ultra thick,green,<-] (.5,1.125) node[above=-2pt,xshift=3pt]{\scs$b{-}k{-}1$} to [out=270,in=30] (.375,.5);
\draw[thick,green,rdirected=.6] (-.375,.5) to [out=30,in=150] (.375,.5);
\draw[ultra thick,green,rdirected=.55] (.375,.5) to [out=270,in=90] (.375,0) node[below=-2pt,xshift=1pt]{\scs$b{-}k$};
\draw[ultra thick,green,->] (-.375,.5) to [out=270,in=90] (-.375,0) node[below=-2pt,xshift=-1pt]{\scs$a{-}k$};
\node at (1.125,.75){\scs$(a,b)$};
\end{tikzpicture} \, .
\end{equation}
More generally, for $1 \leq i \leq m-1$, the \emph{$i^{th}$ Rickard complex} is the chain complex
\[
C(\beta_i) \one_\wt := \one_{(a_1,\ldots,a_{i-1})} \boxtimes C_{a_i,a_{i+1}} \boxtimes \one_{(a_{i+2},\ldots,a_m)}
\in \dgCat[{\cSSS_q(\glm)}] \, .
\]
\end{defn}

The Rickard complex $C(\beta_i) \one_\wt$ determines complexes in various quotients of $\dgCat[{\cSSS_q(\gln[2])}]$, 
(the homotopy category, the $N$-bounded quotient, etc.)
which will we denote using the same notation.
The following is a consequence of \cite{CK,CKL,ETW};
see the proof of \cite[Proposition 2.25]{HRW1} for further details.

\begin{thm}
	\label{thm:Rickard}
Let $\K$ be a field. 
The complex $C_{a,b}$ is invertible in $\hCat[{\cSSS_q(\gln[2])}]$, 
with homotopy inverse the complex
\begin{align*}
C^\vee_{b,a} &:= \big( \cdots \xrightarrow{d^\vee} \qsh^{k} \tsh^{-k} \FF^{(b-k)} \EE^{(a-k)} \one_{b,a} 
			\xrightarrow{d^\vee} \qsh^{k-1} \tsh^{-k+1} \FF^{(b-k+1)} \EE^{(a-k+1)} \one_{b,a} 
				\xrightarrow{d^\vee} \cdots \big) \, ,
	\\	&= \Big( {\textstyle \bigoplus_{k=0}^{\min(a,b)} \qsh^{k} \tsh^{-k} \FF^{(b-k)} \EE^{(a-k)} \one_{b,a}} \, , \, d^\vee \Big) 
\end{align*}
i.e.~$C^\vee_{b,a} \hComp C_{a,b} \simeq \one_{(a,b)} \simeq C_{b,a} \hComp C^\vee_{a,b}$. 
Here,
\begin{equation}
	\label{eq:dualRickDiff}
d^\vee = 
\begin{tikzpicture}[anchorbase,scale=1,yscale=-1]
\draw[ultra thick,green,->] (-.375,.5) to [out=150,in=270] (-.5,1.125) node[below=-2pt]{\scs$b{-}k$};
\draw[ultra thick,green,directed=.45] (.5,1.125) node[below=-2pt]{\scs$a{-}k$} to [out=270,in=30] (.375,.5);
\draw[thick,green,directed=.6] (-.375,.5) to [out=30,in=150] (.375,.5);
\draw[ultra thick,green,->] (.375,.5) to [out=270,in=90] (.375,0) node[above=-2pt,xshift=5pt]{\scs$a{-}k{+}1$};
\draw[ultra thick,green,rdirected=.5] (-.375,.5) to [out=270,in=90] (-.375,0) node[above=-2pt,xshift=-5pt]{\scs$b{-}k{+}1$};
\node at (1,.25){\scs$(b,a)$};
\end{tikzpicture} \, .
\end{equation}
Set $C(\beta_i^{-1}) \one_\wt 
:= \one_{(a_1,\ldots,a_{i-1})} \boxtimes C^\vee_{a_i,a_{i+1}} \boxtimes \one_{(a_{i+2},\ldots,a_m)}$.
Given a braid word $\beta_{i_1}^{\epsilon_1} \cdots \beta_{i_r}^{\epsilon_r}$, let
\[
C(\beta_{i_1}^{\epsilon_1} \cdots \beta_{i_r}^{\epsilon_r}) \one_\wt := 
C(\beta_{i_1}^{\epsilon_1}) \cdots (\beta_{i_r}^{\epsilon_r}) \one_\wt \, ,
\]
then these complexes satisfy the (colored) braid relations in $\hCat[{\cSSS_q(\glm)}]$ 
up to canonical homotopy equivalence. \qed
\end{thm}

For each $\glm$ weight $\wt$ with $a_i \geq 0$ for $1 \leq i \leq m$,
Theorem \ref{thm:Rickard} canonically assigns
a complex $C(\beta)\one_\wt \in \hCat[\cSSS_q(\glm)]$ to each braid $\beta \in \Br_m$.
For each $N \geq 1$, we therefore obtain complexes 
$C(\beta)\one_\wt \in \hCat[\cSSSn_q(\glm)]$. 
(This complex is zero if $a_i \geq N$ for some $1 \leq i \leq m$.)
Following \cite[Section 6]{QR2}, we now recall a procedure that recovers the 
(colored) $\sln$ link homologies. 

The functor
\begin{equation}
	\label{eq:HomTrace}
\Hom_{\cSSSn_q(\glm)}\big(\one_\wt , - \big) \colon \one_\wt \cSSSn_q(\glm) \one_\wt \to \Vect^\Z_\K
\end{equation}
induces a functor on homotopy categories $\hCat[ \one_\wt \cSSSn_q(\glm) \one_\wt] \to \hCat[ \Vect^\Z_\K ]$, 
that we denote similarly. 
Since $\hCat[ \Vect^\Z_\K ] \cong \dCat[ \Vect^\Z_\K ] \cong \Vect^{\Z \times \Z}_\K$, 
the functor \eqref{eq:HomTrace} assigns a bi-graded vector space to
pairs $(\beta, \wt) \in \Br_m \times \Zge^m$ such that $\beta$ is $\wt$-balanced, 
meaning $C(\beta) \one_\wt = \one_\wt C(\beta) \one_\wt$.
In down-to-earth terms, this is simply the homology of the complex 
$\Hom_{\cSSSn_q(\glm)}\big(\one_\wt , \one_\wt C(\beta) \one_\wt \big)$ of bi-graded vector spaces.

The following is a repackaging of results from \cite{QR1,QR2,QRS} 
in the setup and conventions of the present paper;
see also \cite{Cautis}. 
As above, if $(\beta, \wt) \in \Br_m \times \Zge^m$ is a pair such that 
$C(\beta) \one_\wt = \one_\wt C(\beta) \one_\wt$, 
then the corresponding colored braid closure is denoted by $\mathcal{L}_\beta^{\wt}$.

\begin{thm}
	\label{thm:KhR}
Let $\K$ be a field.
If $\wt \in \Zge^m$ and $\bgen \in \Br_m$ is a braid such that 
$C(\beta) \one_\wt = \one_\wt C(\beta) \one_\wt$,
then the complex 
\begin{equation}
	\label{eq:KhRdef}
\llbracket \one_\wt \beta \one_\wt \rrbracket_N :=
\qsh^{\sum_{i=1}^m a_i(a_i-N)} \Hom_{\cSSSn_q(\glm)}\big(\one_\wt , \one_\wt C(\beta) \one_\wt \big) \in \hCat[ \Vect^\Z_\K ]
\end{equation}
is an invariant of the framed colored link $\mathcal{L}_\beta^{\wt}$.
Denote the homology of the complex $\llbracket \one_\wt \beta \one_\wt \rrbracket_N$ 
by $H_{\sln}(\mathcal{L}_\beta^{\wt})$.
Then, we have
\begin{equation}
	\label{eq:slNdecat}
\begin{aligned}
\bP_{\sln}(\mathcal{L}_\beta^{\wt}) 
	&= (-1)^{\sum_{i=1}^m a_i(N-a_i)} (-q^{1/N})^{\epsilon(\bgen,\wt)}
		\dim_{\qsh,\tsh} \big( H^{i}_{\sln}(\mathcal{L}_\beta^{\wt}) \big) \big\rvert_{t=-1} \\
	&= (-1)^{\sum_{i=1}^m a_i(N-a_i)} (-q^{1/N})^{\epsilon(\bgen,\wt)}
	\sum_{i} (-1)^i \dim_\qsh \big( H^{i}_{\sln}(\mathcal{L}_\beta^{\wt}) \big)
\end{aligned}
\end{equation}
i.e.~$H_{\sln}(\mathcal{L}_\beta^{\wt})$ categorifies 
(a multiple of) the colored $\sln$ link polynomial. \qed
\end{thm}

\begin{rem}
It follows from \cite[Theorem 4.12]{QR1} and \cite[Section 6]{QR2} that
the invariant $H_{\sln}(\mathcal{L}_\beta^{\wt})$ agrees with 
(colored) $\sln$ Khovanov--Rozansky link homology \cite{KhR,Wu,Yon}, 
up to a choice of conventions 
(for positive/negative crossings and grading shifts).
Our conventions for the Rickard complexes in Definition \ref{def:Rickard}
imply that if $\bgen \in \Br_m$, 
$\wt = (a_1,\ldots,a_m)$, and $\wt' = (a_1,\ldots,a_m,a_m)$, then
\[
\llbracket \one_{\wt'} \bgen_m^{\pm} \beta \one_{\wt'} \rrbracket_N 
\simeq \qsh^{\mp a_m(N-a_m+1)} \tsh^{\pm a_m} \llbracket \one_{\wt} \beta \one_{\wt} \rrbracket_N \, .
\]
In other words, 
a positive $a$-colored stabilization yields a shift of $\qsh^{-a(N-a+1)} \tsh^a$, 
while a negative $a$-colored stabilization yields the opposite shift of $\qsh^{a(N-a+1)} \tsh^{-a}$.
\end{rem}

\begin{rem}
The precise decategorification result given in \eqref{eq:slNdecat} follows from 
\eqref{eq:SHR} and Remark \ref{rem:Rconv}.
The astute reader will note that the first factor of $\pm1$ in \eqref{eq:slNdecat}
would not appear for the $\sln$ link polynomials defined in the conventions of \cite{CKM}.
They work with a different pivotal structure on $\Rep(U_q(\sln))$
than the one implicit in \S \ref{ss:QG};
in their conventions, 
the twist coefficients \eqref{eq:curlsforframing} lack the $\pm1$ factors.

The term $(-1)^{\sum_{i=1}^m a_i(N-a_i)}$ in \eqref{eq:slNdecat}
can be accounted for as follows.
If we only multiply by $(-q^{1/N})^{\pm a_i a_{i+1}}$ for each $\pm$-crossing, 
the result is the $\sln$ polynomial of the mirror link, 
computed as in \cite{CKM}. 
(The mirror since their $\pm$-crossing is our $\mp$-crossing.)
However, we (implicitly) work with the pivotal structure studied in \cite{ST}
while \cite{CKM} work with the standard pivotal structure 
(which has positive circle values and no signs in the twist coefficients).
To account for this discrepancy, 
we need to multiply by $(-1)^{a(N-a)}$ for each $a$-colored cap/cup pair in a link diagram.
For closures of braids $\bgen \in \Br_m$, there are $m$ such pairs.
\end{rem}

\subsection{An elaboration on Rickard canonicity}
	\label{ss:RickCan}

In order to define our type $B$ link homology, 
we will need to establish an equivariant analogue of Theorem \ref{thm:Rickard}. 
Among other things, Theorem \ref{thm:Rickard} states that when
two braid words represent the two braids, 
then their Rickard complexes are \emph{canonically} homotopy equivalent; 
we call this \emph{Rickard canonicity}. 
In this section, we establish an effective method for understanding this canonical homotopy equivalence. 
We focus on the special case when $N = 2n$, 
which is the case relevant to our type $B$ link homology.

Our method will be to study homotopy equivalences after applying a 2-functor from $\SSStn_q(\glm)$ 
to the following simplified setting.

\begin{defn}
Let $\Catn$ be the monoidal category of finitely generated free graded modules 
over $H^\ast(Gr_n(\mathbb{C}^{2n}))$. 
We treat this as a $2$-category with one object, 
denoted (by abuse of notation) as $\wtn$.
\end{defn}

We record some properties of this $2$-category, which are mostly tautological.
The identity $1$-morphism $\one_{\wtn}$ in $\Catn$, 
i.e.~the left regular module $H^\ast(Gr_n(\mathbb{C}^{2n}))$, 
is the only indecomposable $1$-morphism in $\Catn$, 
up to grading shift and isomorphism. 
Left multiplication gives an identification
$\End_{\Catn}(\one_\wtn) = H^\ast(Gr_n(\mathbb{C}^{2n}))$, 
and the latter may be identified with
the quotient of  the ring $\Sym(\X)$ of symmetric functions in one alphabet $\X$
by Schur polynomials indexed by partitions which do not fit in an $n\times n$ box.

\begin{prop}\label{P:there-is-Gamma}
The following assignments determine a well-defined lax 
$2$-functor
\[
\Gamma \colon \Kar(\SSStn_q(\glm)) \rightarrow\Catn \, .
\]
which is full on $2$-morphisms.
All objects of $\Kar(\SSStn_q(\glm))$ are sent to the unique object of $\Catn$. 
The identity $1$-morphisms $\one_{\wt}$ are sent to the zero module when $\wt \ne \wtn$, 
and $\one_{\wtn} \mapsto \one_{\wtn}$. 
All other generating $1$-morphisms $\EE_i$ and $\FF_i$ are sent to the zero module, 
and thus most generating $2$-morphisms are sent to zero. 
The new bubbles $h_r(\X_i)$ are sent to $h_r(\X)$ for all $r \ge 0$ and all $1 \le i \le m$.
\end{prop}

\begin{proof} 
Using the generators and relations for $\UU_q(\glm)$ in Definition \ref{def:CQG}, 
it is straightforward to check that the indicated assignments yield
a lax $2$-functor from $\UU_q(\glm)$ to $\Catn$. 
(This $2$-functor is lax since when $\wt \neq \wtn$, the structure map 
$H^\ast(Gr_n(\mathbb{C}^{2n})) = \one_{\Gamma(\wt)} \to \Gamma(\one_\wt) = 0$
is not an isomorphism.)

To elaborate, both sides of most relations go immediately to zero, 
since they involve the $1$-morphisms $\EE_i$ or $\FF_i$. 
The only relation with any subtlety is \eqref{newbub}.
Recall that \eqref{newbub} is not a relation for fake bubbles (merely a notational convenience), 
while for real bubbles it is a relation equating a real bubble with a new bubble. 
It is crucial that all real bubbles in the region $\wtn$ have strictly positive degree, 
so in this case the left-hand side of \eqref{newbub} is sent to zero, 
while the right-hand side $h_r(\X_i - \X_{i+1})$ is sent to $h_r(\X - \X)$, 
which is zero when $r > 0$.

The $2$-functor 
factors through $\SSStn_q(\glm)$ by Definition \ref{def:CSQ}
and is full on $2$-morphisms because all $h_r(\X)$ are in the image.
Since projective modules for $H^\ast(Gr_n(\mathbb{C}^{2n}))$ are free, 
$\Catn$ is Karoubian, and 
therefore there is an induced $2$-functor
$\Kar(\SSStn_q(\glm))\rightarrow\Catn$. \end{proof}

\begin{rem}\label{R:Gamma-on-FE}
All indecomposable $1$-morphisms in $\Kar(\SSStn_q(\glm))$
except (shifts of) $\one_{\wtn}$ are sent by $\Gamma$ to zero. In particular, 
\[
\Gamma(\FF_i^{(k)}\EE^{(k)}_i\one_{\wtn}) = 0, \qquad \text{for $k > 0$ and $i=1, \dots, m-1$.}
\]
Note that, by definition, the $1$-morphism $\one_{\wtn}$ survives in the quotient.

In fact, for any weight $\wt$ with $a_i \ne a_{i+1}$ for some $i$, 
that $\one_{\wt}$ is sent to zero is already 
implied by the fact that $\EE_i$ and $\FF_i$ are sent to zero. 
Indeed, in such weights there is at least one real bubble of degree zero, 
and consequently $\id_{\one_{\wt}} = h_0(\X_i - \X_{i+1})$ will be sent to zero.
\end{rem}

We will use the symbol $\Gamma$ to refer to several related $2$-functors. 
We need not work with the full Karoubi envelope, 
but can restrict $\Gamma$ to the partial idempotent completion $\cSSStn_q(\glm)$:
\[
\Gamma \colon \cSSStn_q(\glm)\rightarrow \Catn.
\]
Since $\Gamma$ is additive, it induces a $2$-functor between dg categories of chain complexes:
\begin{equation}
	\label{eq:complexquotient}
\Gamma \colon \dgCat[\cSSStn_q(\glm)] \to \dgCat[\Catn]
\end{equation}
and hence also between homotopy categories.
The targets of such functors provide a simplified setting to study Rickard complexes.

\begin{lem} 
	\label{lem:Gamma(C)}
The Rickard complexes $C(\bgen_{i}) \one_\wtn$ and $C(\bgen_{i}^{-1}) \one_\wtn$
are sent by $\Gamma$ to the complexes $\qsh^{-n}\tsh^{n}\one_{\wtn}$ and $\qsh^{n}\tsh^{-n}\one_{\wtn}$ 
(with zero differential), respectively. 
More generally, if $\bword = \bgen_{i_1}^{\epsilon_1} \cdots \bgen_{i_r}^{\epsilon_r}$ is a braid word
and $\epsilon(\bword) = \epsilon_1 + \cdots + \epsilon_r$ is the braid exponent, 
then the complex 
$C(\bword) \one_\wtn = C(\bgen_{i_1}^{\epsilon_1}) \one_\wtn \cdots C(\bgen_{i_r}^{\epsilon_r}) \one_\wtn$ 
maps under \eqref{eq:complexquotient} to the complex $(\qsh^{-n}\tsh^{n})^{\epsilon(\bword)}\one_{\wtn}$ 
supported in a single degree. 
\end{lem}

\begin{proof} 
The first statement is an obvious consequence of Remark \ref{R:Gamma-on-FE} 
and the second follows from $\Gamma$ being a $2$-functor.
\end{proof}

The following lemma is our effective method to understand morphisms 
between Rickard complexes in the homotopy category.

\begin{lem}
	\label{L:detect-equivalence-cell-quotient}
Let $\K$ be a field and let
$\bword$ and $\bword'$ be two braid words for the same braid. 
The functor $\Gamma$ from \eqref{eq:complexquotient} induces an isomorphism
\begin{equation}
	\label{eq:quotientiso}
\Hom_{\hCat[\cSSStn_q(\glm)]}^{0} \big( C(\bword) \one_\wtn, C(\bword') \one_\wtn \big) 
	\xrightarrow{\cong} \Hom_{\hCat[\Catn]}^{0} 
	\big( \Gamma(C(\bword) \one_\wtn), \Gamma(C(\bword') \one_\wtn) \big) = \K \cdot \id
\end{equation}
(on the space of $\qsh$-degree zero morphisms). 
Further, a $\qsh$-degree zero 
chain map $f \colon C(\bword) \one_\wtn \to C(\bword') \one_\wtn$ in $\dgCat[\cSSStn_q(\glm)]$
will be a homotopy equivalence if and only if $\Gamma(f)$ is an isomorphism, 
and $f$ will be nulhomotopic if and only if $\Gamma(f) = 0$. 
\end{lem}

\begin{rem}
Since $\K$ is a field, it is immediate from the proof of Lemma \ref{L:detect-equivalence-cell-quotient} 
that $f$ is a homotopy equivalence if and only if $\Gamma(f)\ne 0$. 
However, the following proof of the above lemma is designed to work over $\Z$, 
once invertibility of Rickard complexes is established over $\Z$.
\end{rem}

\begin{proof} 
Let $\bword$ and $\bword'$ be braid words for the same braid.
Then, we necessarily have $\epsilon(\bword)=\epsilon(\bword')$,
so Lemma \ref{lem:Gamma(C)} gives that
$\Gamma(C(\bword) \one_\wtn) = \Gamma(C(\bword') \one_\wtn)$ 
is a complex supported in one degree, a shift of the free module $\one_{\wtn}$. 
Thus, the space
\[
\Hom_{\hCat[\Catn]}^{0} 
	\big( \Gamma(C(\bword) \one_\wtn), \Gamma(C(\bword') \one_\wtn) \big)
\]
of $\qsh$-degree zero morphisms in the homotopy category
consists solely of scalar multiples of the identity map. 
Since $\Gamma(C(\bword) \one_\wtn)$ and $\Gamma(C(\bword') \one_\wtn)$ 
are one-term complexes supported in the same homological degree, 
the morphism space in the homotopy category agrees with the space 
\[
\Hom_{\ChCat[\Catn]}^{0}
	\big( \Gamma(C(\bword) \one_\wtn), \Gamma(C(\bword') \one_\wtn) \big)
\]
of $\qsh$-degree zero chain maps
(which further agrees with the space 
of all $\qsh$-degree zero maps in $\dgCat[\Catn]$).
To summarize, 
\begin{equation}
	\label{eq:summary=}
\Hom_{\hCat[\Catn]}^{0} 
	\big( \Gamma(C(\bword) \one_\wtn), \Gamma(C(\bword') \one_\wtn) \big)
	= \K \cdot \id
	= \Hom_{\dgCat[\Catn]}^{\qsh\text{-deg} = 0} 
	\big( \Gamma(C(\bword) \one_\wtn), \Gamma(C(\bword') \one_\wtn) \big) \, ,
\end{equation}
i.e.~elements in the latter are homotopy equivalences if and only if they are isomorphisms 
and are nulhomotopic if and only if they are zero.

By Theorem \ref{thm:Rickard}, 
$C(\bword) \one_\wtn$ and $C(\bword') \one_\wtn$ are homotopy equivalent, 
and any choice of homotopy equivalence will induce an isomorphism
\[
\Hom_{\hCat[\cSSStn_q(\glm)]}^{0}(C(\bword) \one_\wtn, C(\bword') \one_\wtn)
	\cong \End_{\hCat[\cSSStn_q(\glm)]}^{0}(C(\bword) \one_{\wtn}) \, .
\]
Theorem \ref{thm:Rickard} also gives that $C(\bword)$ is invertible, 
so the functor of tensoring with the inverse Rickard complex induces an isomorphism 
\[
\End_{\hCat[\cSSStn_q(\glm)]}^{0}(C(\bword) \one_{\wtn}) 
	\cong \End_{\hCat[\cSSStn_q(\glm)]}^{0}(\one_{\wtn}) = \K \cdot \id \, .
\]
The functor $\Gamma$ sends 
$\id_{C(\bword) \one_{\wtn}} \in \End_{\hCat[\cSSStn_q(\glm)]}^{0}(C(\bword) \one_{\wtn})$ 
to the identity map of $(\qsh^{-n}\tsh^{n})^{\epsilon(\bword)}\one_{\wtn}$ in $\Catn$.
Using the functoriality of $\Gamma$, this then implies that \eqref{eq:quotientiso} is an isomorphism.

Since $\Gamma$ is additive, it maps homotopy equivalences to homotopy equivalences 
and nullhomotopic maps to nullhomotopic maps.
The final statement then follows from \eqref{eq:summary=} 
and the isomorphism \eqref{eq:quotientiso}.
\end{proof}

We use this lemma to pick distinguished canonical isomorphisms, thus clarifying Rickard canonicity.

\begin{prop}\label{P:rick-can}
Let $\K$ be a field and $\bgen \in \Br_m$ be a braid.
Given braid words $\bword$ and $\bword'$ for $\bgen$, let
$f_{\bword',\bword} \colon C(\bword) \one_\wtn \to C(\bword') \one_\wtn$
be a homotopy equivalence for which 
$\Gamma(f_{\bword',\bword}) = \id_{(\qsh^{-n}\tsh^{n})^{\epsilon(\bword)}\one_{\wtn}}$. 
If $f$ is another homotopy equivalence satisfying this property, then $f \sim f_{\bword',\bword}$. 
Moreover, if $\bword''$ is another braid word for $\bgen$, 
then $f_{\bword'',\bword'} \circ f_{\bword',\bword} \sim f_{\bword'',\bword}$.
\end{prop}

\begin{proof}
It follows from Lemma \ref{L:detect-equivalence-cell-quotient} that a chain map 
between such Rickard complexes is nullhomotopic if and only if 
it is zero after applying $\Gamma$. 
Thus, two chain maps between such Rickard complexes are homotopic 
if and only if they agree after one applies $\Gamma$. 
From this fact, the proposition is obvious.
\end{proof}

\begin{rem}
We have focused here on the case $\wt = \wtn$ for simplicity, 
and since this is the relevant setting for our spin link homology defined below. 
One might be able to make similar arguments for ``Rickard canonicity'' in arbitrary weight, 
by considering different quotient functors analogous to $\Gamma$. 
\end{rem}

%
\section{Spin link homology}
	\label{s:SLH}
%

At last, in this section we define our spin link homology theory.
First, in \S \ref{ss:schur-FE} we introduce a monoidal category $\BnFoam$, 
defined as a quotient of $\BFoam$, which inherits the involution $\tau$.
The equivariant category $\EBnFoam$ will be the setting for type $B_n$ spin-colored link homology. 
Then, in \S \ref{ss:eRickard}, we define complexes $C^{\tau}(\bgen) \in \dgCat[\EBnFoam]$
that determine a categorical braid group action.
The latter fact is proved in \S \ref{ss:eBraidrel} using a general result about 
lifting homotopy equivalences to equivariant categories that is established in \S \ref{ss:lifting}. 
In \S \ref{ss:eHomology} we apply an appropriate representable functor to $C^{\tau}(\bgen)$
which yields a chain complex of super vector spaces
whose homology is a bigraded super vector space valued 
braid invariant that is further invariant under the Markov moves, up to isomorphism.
Hence, by Theorem \ref{thm:Markov}, we obtain an invariant of links valued 
in isomorphism classes of bigraded super vector spaces.
When $n=1,2,3$, we prove this invariant categorifies 
the spin-colored quantum $\son$ link 
polynomials\footnote{Since $\spn[4] \cong \son[5]$,
we have thus also categorified the quantum $\spn[4]$ link invariant
colored by the defining representation.},
and we conjecture that this holds for all $n>3$ as well.
We therefore denote this new link homology theory by $H_{\son,\mathcal{S}}(\mathcal{L}_\bgen)$.

\begin{rem}
Our construction bootstraps on the properties of the Rickard complexes 
from Theorem \ref{thm:Rickard}, 
namely that they are invertible and satisfy the braid relations.
Since the proof of these results in the literature assumes that $\K$ is a field, 
we will make this assumption when working with the Rickard complexes 
(essentially from Section \ref{ss:eRickard} onward).
Consequently, our link invariants will be defined over a field.
However, once some folklore results are carefully established, 
it should be possible to refine our proofs to construct our link invariants over 
$\K = \Z[\frac{1}{2}]$, and thus over any commutative ring in which $2$ is invertible.
\end{rem}

\subsection{$\FF^{(k)}\EE^{(k)}\one_{\wtn}$-generated Schur quotient}\label{ss:schur-FE}

In \S \ref{ss:Rickard}, we saw that $\sln$ link homology may be formulated entirely 
in the setting of the $N$-bounded Schur quotient $\cSSSn_q(\glm)$ 
of the categorified quantum group $\cUU_q(\glm)$.
We also saw in \S \ref{s:inv} that, in order to guarantee that $\tau$ is an involution, 
we must pass to $\cXX_q(\glm)$, the full $2$-subcategory of $\cUU_q(\glm)$ generated 
by the $1$-morphisms $\FF^{(k)}_i \EE^{(k)}_i\one_{\wt}$ and $\EE^{(k)}_i \FF^{(k)}_i\one_{\wt}$.
Finally, to find a monoidal category categorifying 
the relations in the centralizer algebras studied in \S \ref{s:typeB}, 
we had to work in \S \ref{s:ECQG} 
with the equivariantization of the full $2$-subcategory $\cXX_q(\glm)\one_\wtn$. 
To define our spin link homology, 
we will work in a setting that combines these desired features.  

First, we verify that the $2$-automorphism $\tau$ 
descends to an automorphism of the $N$-bounded Schur quotient from Definition \ref{def:CSQ}. 
As discussed back in \S \ref{ss:nutshell}, 
and in light of the action of $\tau$ on weights, the relevant value here is $N=2n$.

\begin{lem}\label{lem:tau-aut-cSSStn}
If $\tau=\tau_n$ satisfies the conditions of Theorem \ref{thm:symmetry}, 
then it preserves the ideal $\mathcal{J}^{\leq 2n}$ from Definition \ref{def:CSQ}. 
Consequently, it induces an automorphism of $\cSSStn_q(\glm)$.
\end{lem}

\begin{proof}
Recall that $\mathcal{J}^{\leq N}$ is generated 
by all new bubble morphisms $\Schur(\X_i)\in \End(\one_{\wtn+\wt})$ 
such that the partition $\lambda$ does not fit inside an $(n+a_i) \times (n-a_i)$ box.
(Note that here we consider ambient weight $\wtn+\wt$.)
It follows from equation \eqref{eq:tauonSchur} that if $\Schur(\X_i)\in \End(\one_{n+\wt})$ is such that $\parti$ is not contained in an $(n+a_i) \times (n-a_i)$ box, 
then $\tau(\Schur(\X_i)) = \Schur[\parti^t](\X_i) \in \End(\one_{n-\wt})$ for $\parti^t$ not contained in an 
$(n-a_i) \times (n+a_i)$ box, as desired. 
\end{proof}

Recall that if $\tau$ satisfies the conditions of Theorem \ref{thm:involution}, 
then $\tau$ restricts to an involution on $\cXX_q(\glm)$ by Corollary \ref{cor:thickinvolution}. 
We now define a Schur quotient category on which $\tau$ is an involution. 

\begin{defn}\label{D:cXXtn-2cat}
Let $\cXXtn_q(\glm)$ be the quotient of the $2$-category $\cXX_q(\glm)$ 
from Definition \ref{D:cXX-2cat} 
by the intersection of $\cXX_q(\glm)$ with the ideal $\mathcal{J}^{\leq 2n}$.
\end{defn}

\begin{cor}
	\label{cor:atlastinv}
The involution $\tau$ on $\cXX_q(\glm)$ from Corollary \ref{cor:thickinvolution}
induces an involution $\tau$ on $\cXXtn_q(\glm)$. \qed
	\end{cor}

Finally, we arrive at our setting for type $B_n$ spin link homology. 

\begin{defn}
	\label{def:BnFoam}
Let $\BnFoam$ be the monoidal category $\one_{\wtn}\cXXtn_q(\glm)\one_{\wtn}$.
\end{defn}

\begin{cor}
The involution $\tau$ restricts from $\cXXtn_q(\glm)$ to a 
monoidal involution of $\BnFoam$.
\end{cor}
\begin{proof}
In light of Corollary \ref{cor:atlastinv}, it remains to note that 
$\tau(\one_{\wtn}) = \one_{\wtn}$.
\end{proof}

\begin{rem}
The $2$-category $\cXXtn_q(\glm)$ can equivalently be described as the 
full $2$-subcategory of $\cSSStn_q(\glm)$ generated by all 
$\FF^{(k)}_i \EE^{(k)}_i\one_{\wt}$ and $\EE^{(k)}_i \FF^{(k)}_i\one_{\wt}$.
Similarly, $\BnFoam$ can be identified with the $2$-subcategory of 
$\cXXtn_q(\glm)$ generated by the $1$-morphisms
$\FE^{(k)}_i=\FF_i^{(k)}\EE_i^{(k)}\one_{\wtn}$ and $\EE_i^{(k)}\FF_i^{(k)}\one_{\wtn}$
for $i\in \{1, \ldots, m-1\}$ and $k\in \Z_{\ge 0}$.
Therefore, the following commutative diagram collects the various $2$-categories 
introduced thus far.
\[
\begin{tikzcd}
\BFoam \ar[r,hook] \ar[d, two heads] & \cXX_q(\glm) \ar[r,hook] \ar[d, two heads] 
	& \cUU_q(\glm) \ar[d, two heads] \\
\BnFoam \ar[r,hook] & \cXXtn_q(\glm) \ar[r,hook] & \cSSStn_q(\glm)
\end{tikzcd}
\]
Here, all arrows are full and the horizontal arrows are faithful.
\end{rem}

\begin{rem}
If $\K$ is a field of characteristic zero, then $\BnFoam$ is self-dual mixed.
Further, if $\K$ is also algebraically closed, then Hypothesis \ref{hypo:assumeme}
holds and Corollary \ref{cor:ECatsdm} implies that
$\EBnFoam$ is also self-dual mixed. 
\end{rem}

Before proceeding, 
we establish a decategorification result 
(Theorem \ref{thm:introKzero1})
from \S \ref{ss:introinter}.

\begin{lem}
	\label{lem:stillindecomp}
Let $\K$ be an integral domain.
The $1$-morphisms $\FEk_\grb \in \BnFoam$ are either indecomposable or zero, 
and they are non-zero if and only if $0\le k\le n$. 
\end{lem}
\begin{proof}
The first statement follows from fullness of the functor $\BFoam\rightarrow \BnFoam$. 
For the second statement, it suffices to argue that $\FEk_\grb$ is non-zero in $\cSSStn_q(\glm)$ 
if and only if $0\le k\le n$, which, in light of Theorem \ref{thm:SchurCategorification}, 
can be easily checked in the Grothendieck ring. 
\end{proof}

\begin{cor}
	\label{cor:XmnDecat}
Let $\K$ be an integral domain in which $2$ is invertible.
The direct sum decompositions from 
Theorems \ref{T:div-powers-in-twisted-K0} and \ref{T:iserre-in-twisted-K0} hold 
in $\EBnFoam$.
Hence, in the $m=2$ case,
there is an isomorphism of $\C(q)$-algebras
\[
\C(q) \ot_{\Z[q^\pm]} \tauKzero{\EBnFoam[2]} \xrightarrow{\cong} \End_{U_q(\son)}(S \ot S)
\]
that sends 
$[(\FE^{(k)}, (-1)^{{n+1 \choose 2}+{n-k+1 \choose 2}+k} \swcl_{k})]_{\tau} 
\mapsto \xx^{(k)}$.
\end{cor}
\begin{proof}
The first claim follows since the functor $\BFoam \to \BnFoam$ intertwines the involution 
$\tau=\tau_n$ and sends $\FEk_i \in \BFoam$ to  $\FEk_i \in \BnFoam$.
For the second, 
Lemma \ref{lem:stillindecomp} implies that $\{ \FEk \}_{0 \leq k \leq n}$ 
constitutes a complete set of indecomposable objects in $\BnFoam$, 
up to isomorphism and shift.
Hence, 
Proposition \ref{prop:computeequivGgp}
implies that the elements 
$[(\FE^{(k)}, (-1)^{{n+1 \choose 2}+{n-k+1 \choose 2}+k} \swcl_{k})]_{\tau}$ 
for $0 \leq k \leq n$ give a basis for $\tauKzero{\EBnFoam[2]}$.
Proposition \ref{itoxchangeofbasis} then implies that the 
indicated assignment is an isomorphism of $\C(q)$-vector spaces.
Finally, Corollary \ref{cor:K0(X2)} and \eqref{E:xxk-recursion} 
imply that this map is a morphism of algebras.
\end{proof}

\subsection{Equivariant Rickard complexes}
	\label{ss:eRickard}

From here through the end of Section \ref{s:SLH}, 
we \textbf{assume that} $\K$ \textbf{is a field in which} $2 \neq 0$.
(We will occasionally re\"{e}mphasize this hypothesis when stating some of our major results.)

We now promote the Rickard complexes from Definition \ref{def:Rickard}
to the equivariant category $\EBnFoam$.
For $1 \leq \grb \leq m-1$, recall that
\[
C(\bgen_{\grb}) \one_{\wtn} 
	= \big( \cdots \xrightarrow{d} \qsh^{-k} \tsh^k \FF_{\grb}^{(n-k)} \EE_{\grb}^{(n-k)} \one_{\wtn} 
		\xrightarrow{d} \qsh^{-k-1} \tsh^{k+1} \FF_{\grb}^{(n-k-1)} \EE_{\grb}^{(n-k-1)} \one_{\wtn} \xrightarrow{d} \cdots \big) \, .
\]
Since the chain objects in this complex are all of the form $\FF_{\grb}^{(\ell)}\EE_{\grb}^{(\ell)}\one_{\wtn}$, 
this actually defines a complex
\[
C(\beta_{\grb}) := \big( \cdots \xrightarrow{d} \qsh^{-k} \tsh^k \FE_{\grb}^{(n-k)} \xrightarrow{d} \qsh^{-k-1} \tsh^{k+1} \FE_{\grb}^{(n-k-1)} \xrightarrow{d} \cdots \big)
		\in \dgCat[\BnFoam].
\]
Similarly, the inverse complex
\[
C(\beta_{\grb}^{-1}) := \big( \cdots \xrightarrow{d^{\vee}} \qsh^{k} \tsh^{-k} \FE_{\grb}^{(n-k)} 
	\xrightarrow{d^{\vee}} \qsh^{k-1} \tsh^{-k+1} \FE_{\grb}^{(n-k+1)} \xrightarrow{d^{\vee}} \cdots \big) 
\]
lies in $\BnFoam$. Henceforth, we denote the components of the corresponding differentials 
appearing in \eqref{eq:RickDiff} and \eqref{eq:dualRickDiff}
as follows:
\[
d_k \colon \qsh\FE_{\grb}^{(k)}\rightarrow \FE_{\grb}^{(k-1)} 
\quad \text{and} \quad 
d_k^{\vee} \colon \qsh\FE_{\grb}^{(k)}\rightarrow \FE_{\grb}^{(k+1)} \, .
\]

\begin{lem}
	\label{lem:eRickdiff}
The differentials in the Rickard complexes induce the following 
degree-zero morphisms in $\EBnFoam$:
\[
d_k \colon ( \qsh \FE^{(k)}_{\grb}, \swcl_{k}) \rightarrow (\FE^{(k-1)}_{\grb}, (-1)^{n-k+1}\swcl_{k-1})
\]
and
\[
d_k^{\vee} \colon ( \qsh \FE^{(k)}_{\grb}, (-1)^{n+k} \swcl_{k})\rightarrow (\FE^{(k+1)}_{\grb}, \swcl_{k+1}).
\]
\end{lem} 
\begin{proof}
Using the definitions of $\tau$ and $d$, we have
\[
\tau(d_k)=
\begin{tikzpicture}[anchorbase,scale=1]
\draw[ultra thick,green,->] (-.375,.5) to [out=150,in=270] (-.5,1.125) node[above=-2pt]{\scs$k{-}1$};
\draw[ultra thick,green,directed=.45] (.5,1.125) node[above=-2pt]{\scs$k{-}1$} to [out=270,in=30] (.375,.5);
\draw[thick,green,directed=.6] (-.375,.5) to [out=30,in=150] (.375,.5);
\draw[ultra thick,green,->] (.375,.5) to [out=270,in=90] (.375,0) node[below=-2pt]{\scs$k$};
\draw[ultra thick,green,rdirected=.5] (-.375,.5) to [out=270,in=90] (-.375,0) node[below=-2pt]{\scs$k$};
\node at (1,.75){\scs$\wtn$};
\end{tikzpicture}
\quad \text{and} \quad 
\tau(d_k^{\vee}) = 
\begin{tikzpicture}[anchorbase,scale=1,rotate=180]
\draw[ultra thick,green,->] (-.375,.5) to [out=150,in=270] (-.5,1.125) node[below=-2pt]{\scs$k$};
\draw[ultra thick,green,directed=.45] (.5,1.125) node[below=-2pt]{\scs$k$} to [out=270,in=30] (.375,.5);
\draw[thick,green,directed=.6] (-.375,.5) to [out=30,in=150] (.375,.5);
\draw[ultra thick,green,->] (.375,.5) to [out=270,in=90] (.375,0) node[above=-2pt]{\scs$k{+}1$};
\draw[ultra thick,green,rdirected=.5] (-.375,.5) to [out=270,in=90] (-.375,0) node[above=-2pt]{\scs$k{+}1$};
\node at (-1,.25){\scs$\wtn$};
\end{tikzpicture} \, .
\]
It suffices to show that
\begin{equation}
	\label{eq:tauRicknec}
\tau(d_k)\circ\swcl_{k} = (-1)^{n-k+1} \swcl_{k-1}\circ d_k
\quad \text{and} \quad 
(-1)^{n+k} \tau(d_k^{\vee})\circ \swcl_{k} = \swcl_{k+1}\circ d_k^{\vee} \, .
\end{equation} 
For the former, we compute
\begin{multline}
	\label{eq:tauRickdiff}
\begin{tikzpicture}[anchorbase,scale=1]
\draw[ultra thick,green,->] (-.25,.5) to [out=150,in=270] (-.25,1) node[above=-2pt,xshift=-3pt]{\scs$k{-}1$};
\draw[ultra thick,green,directed=.45] (.25,1) node[above=-2pt,xshift=3pt]{\scs$k{-}1$} to [out=270,in=30] (.25,.5);
\draw[thick,green,directed=.6] (-.25,.5) to [out=30,in=150] (.25,.5);
\draw[ultra thick,green,->] (.25,.5) to [out=270,in=90] (-.25,-.25) node[below=-2pt]{\scs$k$};
\draw[ultra thick,green,rdirected=.3] (-.25,.5) to [out=270,in=90] (.25,-.25) node[below=-2pt]{\scs$k$};
\node at (.625,.75){\scs$\wtn$};
\end{tikzpicture}
\stackrel{\eqref{eq:bypass}}{=}
\begin{tikzpicture}[anchorbase,scale=1]
\draw[ultra thick,green,->] (-.25,.5) to [out=150,in=270] (-.25,1) node[above=-2pt,xshift=-3pt]{\scs$k{-}1$};
\draw[ultra thick,green,directed=.45] (.25,1) node[above=-2pt,xshift=3pt]{\scs$k{-}1$} to [out=270,in=30] (.25,.5);
\draw[thick,green,directed=.6] (-.25,.5) to [out=30,in=150] (.25,.5);
\draw[ultra thick, green, directed=.7] (.125,.125) to (-.125,.125);
\draw[ultra thick,green,->] (-.125,.125) to [out=240,in=90] (-.25,-.25) node[below=-2pt]{\scs$k$};
\draw[ultra thick,green,rdirected=.5] (-.25,.5) to [out=270,in=120] (-.125,.125);
\draw[ultra thick,green,directed=.6] (.25,.5) to [out=270,in=60] (.125,.125);
\draw[ultra thick,green,rdirected=.6] (.125,.125) to [out=300,in=90] (.25,-.25) node[below=-2pt]{\scs$k$};
\node at (.625,.75){\scs$\wtn$};
\end{tikzpicture}
\stackrel{\eqref{eq:assoc}}{=}
\begin{tikzpicture}[anchorbase,scale=1]
\draw[ultra thick,green,->] (-.25,.125) to [out=120,in=270] (-.25,1) node[above=-2pt,xshift=-3pt]{\scs$k{-}1$};
\draw[ultra thick,green,directed=.45] (.25,1) node[above=-2pt,xshift=3pt]{\scs$k{-}1$} to [out=270,in=60] (.25,.125);
\draw[thick,green,directed=.6] (-.125,.375) to [out=120,in=180] (-.0625,.625) to (.0625,.625) to [out=0,in=60] (.125,.375);
\draw[ultra thick, green, directed=.7] (.125,.375) to (-.125,.375);
\draw[ultra thick,green,rdirected=.6] (-.25,.125) to [out=0,in=240] (-.125,.375);
\draw[ultra thick,green,directed=.8] (.25,.125) to [out=180,in=300] (.125,.375);
\draw[ultra thick,green,->] (-.25,.125) to [out=240,in=90] (-.25,-.25) node[below=-2pt]{\scs$k$};
\draw[ultra thick,green,rdirected=.6] (.25,.125) to [out=300,in=90] (.25,-.25) node[below=-2pt]{\scs$k$};
\node at (.625,.75){\scs$\wtn$};
\end{tikzpicture}
\stackrel{\eqref{eq:bypass}}{=}
\begin{tikzpicture}[anchorbase,scale=1]
\draw[ultra thick,green,->] (-.375,.125) to [out=120,in=270] (-.25,1) node[above=-2pt,xshift=-3pt]{\scs$k{-}1$};
\draw[ultra thick,green,directed=.45] (.25,1) node[above=-2pt,xshift=3pt]{\scs$k{-}1$} to [out=270,in=60] (.375,.125);
\draw[thick,green,directed=.55] (.25,.125) to [out=120, in=270] (-.25,.375) to [out=90,in=180] (-.125,.5) 
	to (.125,.5) to [out=0,in=90] (.25,.375) to [out=270,in=60] (-.25,.125);
\draw[ultra thick, green, directed=.7] (.25,.125) to [out=240,in=300] (-.25,.125);
\draw[ultra thick,green] (-.25,.125) to (-.375,.125);
\draw[ultra thick,green] (.25,.125) to (.375,.125);
\draw[ultra thick,green,->] (-.375,.125) to [out=240,in=90] (-.25,-.25) node[below=-2pt]{\scs$k$};
\draw[ultra thick,green,rdirected=.6] (.375,.125) to [out=300,in=90] (.25,-.25) node[below=-2pt]{\scs$k$};
\node at (.625,.75){\scs$\wtn$};
\end{tikzpicture}
\stackrel{\eqref{eq:curl1}}{=}
- \sum_{p+q=2k-2}
\begin{tikzpicture}[anchorbase,scale=1]
\draw[ultra thick,green,->] (-.375,.125) to [out=120,in=270] (-.375,1.25) node[above=-2pt,xshift=-3pt]{\scs$k{-}1$};
\draw[ultra thick,green,directed=.45] (.375,1.25) node[above=-2pt,xshift=3pt]{\scs$k{-}1$} to [out=270,in=60] (.375,.125);
\draw[thick,green,directed=.75] (.25,.125) to [out=120,in=60] node[black,pos=.3]{\scs$\bullet$} 
	node[black,pos=.3,above=-2pt]{\tiny$p$} (-.25,.125);
\draw[ultra thick, green, directed=.7] (.25,.125) to [out=240,in=300] (-.25,.125);
\draw[ultra thick,green] (-.25,.125) to (-.375,.125);
\draw[ultra thick,green] (.25,.125) to (.375,.125);
\draw[thick,green,<-] (.175,.75) arc[start angle=0,end angle=360,radius=.175] 
	node[pos=.25,black]{\scs$\bullet$} node[pos=.25,black,above=-2pt]{\tiny$\spadesuit{+}q$};
\draw[ultra thick,green,->] (-.375,.125) to [out=240,in=90] (-.25,-.25) node[below=-2pt]{\scs$k$};
\draw[ultra thick,green,rdirected=.6] (.375,.125) to [out=300,in=90] (.25,-.25) node[below=-2pt]{\scs$k$};
\node at (.75,.875){\scs$\wtn$};
\end{tikzpicture} \\
\stackrel{\eqref{newbub},\eqref{eq:explode},\eqref{eqn:sgnfromdemazure}}{=}
(-1)^{n-k+1}
\begin{tikzpicture}[anchorbase,scale=1]
\draw[ultra thick, green, directed=.7] (.125,.125) to (-.125,.125);
\draw[ultra thick,green,->] (-.125,.125) to [out=240,in=90] (-.25,-.25) node[below=-2pt]{\scs$k$};
\draw[ultra thick,green,rdirected=.5] (-.25,.5) node[above=-2pt,xshift=-3pt]{\scs$k{-}1$} to [out=270,in=120] (-.125,.125);
\draw[ultra thick,green,directed=.6] (.25,.5) node[above=-2pt,xshift=3pt]{\scs$k{-}1$} to [out=270,in=60] (.125,.125);
\draw[ultra thick,green,rdirected=.6] (.125,.125) to [out=300,in=90] (.25,-.25) node[below=-2pt]{\scs$k$};
\node at (.625,.25){\scs$\wtn$};
\end{tikzpicture}
\stackrel{\eqref{eq:bypass}}{=}
(-1)^{n-k+1}
\begin{tikzpicture}[anchorbase,yscale=-1]
\draw[ultra thick,green,->] (-.25,.5) to [out=180,in=270] (-.25,1) node[below=-2pt]{\scs$k$};
\draw[ultra thick,green,directed=.45] (.25,1) node[below=-2pt]{\scs$k$} to [out=270,in=0] (.25,.5);
\draw[thick,green,rdirected=.6] (-.25,.5) to [out=60,in=120] (.25,.5);
\draw[ultra thick,green,->] (.25,.5) to [out=240,in=90] (-.25,-.25) node[above=-2pt,xshift=-3pt]{\scs$k{-}1$};
\draw[ultra thick,green,rdirected=.3] (-.25,.5) to [out=300,in=90] (.25,-.25) node[above=-2pt,xshift=3pt]{\scs$k{-}1$};
\node at (.625,.25){\scs$\wtn$};
\end{tikzpicture}
=
(-1)^{n-k+1}
\begin{tikzpicture}[anchorbase,scale=1]
\draw[ultra thick,green,rdirected=.35] (-.25,.5) to [out=150,in=270] (.25,1.25) node[above=-2pt,xshift=3pt]{\scs$k{-}1$};
\draw[ultra thick,green,<-] (-.25,1.25) node[above=-2pt,xshift=-3pt]{\scs$k{-}1$} to [out=270,in=30] (.25,.5);
\draw[thick,green,rdirected=.6] (-.25,.5) to [out=30,in=150] (.25,.5);
\draw[ultra thick,green,rdirected=.5] (.25,.5) to [out=270,in=90] (.25,.125) node[below=-2pt]{\scs$k$};
\draw[ultra thick,green,->] (-.25,.5) to [out=270,in=90] (-.25,.125) node[below=-2pt]{\scs$k$};
\node at (.625,.75){\scs$\wtn$};
\end{tikzpicture} \, .
\end{multline}
The second equality in \eqref{eq:tauRicknec} follows similarly. 
\end{proof}

\begin{rem}\label{R:tau-of-chainmap-is-chainmap}
The functor $\tau$ extends from $\BnFoam$ to the category of complexes $\dgCat[\BnFoam]$. 
Explicitly, given a complex 
$C:= \big( \cdots \xrightarrow{d_Y} \tsh^k Y_k\xrightarrow{d_Y} \tsh^{k+1} Y_{k+1} \xrightarrow{d_Y} \cdots \big)$
over $\BnFoam$,
we can apply $\tau$ to each term in the complex, obtaining the complex
\begin{equation}
	\label{eq:tauC}
\tau(C):= \big( \cdots \xrightarrow{\tau(d_Y)} \tsh^k \tau(Y_k) \xrightarrow{\tau(d_Y)} \tsh^{k+1} \tau(Y_{k+1}) 
	\xrightarrow{\tau(d_Y)} \cdots \big) \, .
\end{equation}
If $f\in \Hom_{\ChCat[\BnFoam]}(C, C')$ is a chain map, 
then so is $\tau(f) \colon \tau(C)\rightarrow \tau(C')$. 
\end{rem}

\begin{rem}
	\label{rem:LCchainmap}
Lemma \ref{lem:eRickdiff} can be read as saying that the collection of maps 
$\{ (-1)^{k+1 \choose 2}\swcl_{n-k}\}_{0\le k\le n}$
gives a chain map $C(\beta_{\grb})\rightarrow \tau(C(\beta_{\grb}))$,
i.e.~a map in $\ChCat[\BnFoam]$.
For this, the relevant computation is that
\begin{equation}
	\label{eq:eRicksigns}
{k+1 \choose 2} + {k+2 \choose 2} = (k+1)^2 \equiv k+1 \ \ourmod 2
\end{equation}
which, by \eqref{eq:tauRicknec}, ensures that
$\tau(d_{n-k})\circ (-1)^{k+1 \choose 2} \swcl_{n-k} = (-1)^{k+2 \choose 2} \swcl_{n-k-1}\circ d_{n-k}$.
\end{rem}

\begin{defn}
	\label{def:eRickard}
For $1\le \grb\le m-1$, the \emph{$\grb^{th}$ equivariant Rickard complex} is the chain complex
\begin{equation}
	\label{eq:eRickard+}
C^{\tau}(\beta_{\grb}) := \Big( \bigoplus_{k=0}^{n} \qsh^{-k} \tsh^k (\FE_{\grb}^{(n-k)}, 
	(-1)^{{n+1 \choose 2}+{k+1 \choose 2}} \swcl_{n-k}) \, , \, d \Big)
				\in \dgCat[\EBnFoam]
\end{equation}
and the \emph{$\grb^{th}$ equivariant inverse Rickard complex} is
\begin{equation}
	\label{eq:eRickard-}
C^{\tau}(\beta_{\grb}^{-1}) := \Big( \bigoplus_{k=0}^{n} \qsh^{k} \tsh^{-k} (\FE_{\grb}^{(n-k)}, 
	(-1)^{{n+1 \choose 2}+{k+1 \choose 2}} \swcl_{n-k}) \, , \, d^{\vee} \Big) 
		\in \dgCat[\EBnFoam] \, .
\end{equation}
For a braid word $\beta_{i_1}^{\epsilon_1}\dots \beta_{i_r}^{\epsilon_r}$, the \emph{equivariant Rickard complex} is 
\[
C^{\tau}(\beta_{i_1}^{\epsilon_1}\dots \beta_{i_r}^{\epsilon_r}) 
	:= C^{\tau}(\beta_{i_1}^{\epsilon_1}) \cdots C^{\tau}(\beta_{i_r}^{\epsilon_r}) \in \dgCat[\EBnFoam] \, .
\]
\end{defn}

Observe that \eqref{eq:eRickard+} indeed determines a complex in $\dgCat[\EBnFoam]$, 
since
\[
d \colon (\FE_{\grb}^{(n-k)}, (-1)^{{n+1 \choose 2}+{k+1 \choose 2}} \swcl_{n-k}) 
	\to (\FE_{\grb}^{(n-k-1)}, (-1)^{{n+1 \choose 2}+{k+2 \choose 2}} \swcl_{n-k-1})
\]
is a morphism in $\EBnFoam$ if and only if 
\[
d \colon (\FE_{\grb}^{(n-k)}, \swcl_{n-k}) 
	\to (\FE_{\grb}^{(n-k-1)}, (-1)^{{k+1 \choose 2} + {k+2 \choose 2}} \swcl_{n-k-1})
\]
is such a morphism. 
This holds by Lemma \ref{lem:eRickdiff}, 
using \eqref{eq:eRicksigns}.
Similarly, \eqref{eq:eRickard-} gives a complex in $\dgCat[\EBnFoam]$. 
Our conventions (e.g. the choice of global sign $(-1)^{{n+1 \choose 2}}$) are such that both complexes have the monoidal unit 
$(\FE_{\grb}^{(0)}, \swcl_{0}) = (\one_{\wtn}, \id_{\one_{\wtn}})$ 
appearing in the extremal homological degree.

Although we know that the complexes $C(\beta_{i_1}^{\epsilon_1}\dots \beta_{i_r}^{\epsilon_{r}})$ 
satisfy the braid relations in $\mathcal{K}(\BnFoam)$, 
at this point it is not obvious that the equivariant Rickard complexes satisfy the braid relations in $\mathcal{K}(\EBnFoam)$, 
or even that \eqref{eq:eRickard+} and \eqref{eq:eRickard-} are homotopy inverses.
In the following subsection, 
we take a detour to prove a general result that will be used to establish these facts.

\subsection{Lifting to equivariant homotopy equivalences}
	\label{ss:lifting}

For this section, let $\Cat$ be a $\K$-linear category equipped with an involution $\sigma$, 
as defined in Remark \ref{rem:involution}.
Let $\For \colon \ECat\longrightarrow \Cat$ denote the ``forgetful functor'' which sends an equivariant object $(X, \varphi)$ 
to the underlying object $X$ and an equivariant morphism $f \colon (X,\varphi_X) \to (Y,\varphi_Y)$ to itself, 
now viewed as $f\in \Hom(X,Y)$. This functor $\For$ is faithful.

Recall from \S \ref{D:cats-of-complexes}
that $\dgCat$ denotes the dg category of chain complexes over $\Cat$.
In particular, the $\Hom$-spaces in $\dgCat$ are themselves complexes.
If $X = \big( \bigoplus_{i \in \Z} \tsh^i (X_i, \varphi_i), d_X \big)$ is a complex in $\dgECat$, 
then, applying the forgetful functor $\For$, 
we obtain a complex $\For(X) = \big( \bigoplus_{i \in \Z} \tsh^i X_i , d_X \big)$ in $\dgCat$. 

\begin{lem}
	\label{lem:sigastchainmap}
Let $X =\big( \bigoplus_{i\in \Z} \tsh^i (X_i, \varphi_i), d_X \big)$ 
and $Y = \big( \bigoplus_{i\in \Z} \tsh^i (Y_i, \psi_i), d_Y \big)$ 
be complexes in $\dgECat$. 
The map
\[
\sigast \colon \Hom_{\dgCat}(\For(X), \For(Y)) \to \Hom_{\dgCat}(\For(X), \For(Y))
\]
(given by applying $\sigast$ degree-wise) is a chain map. 
That is, if $D$ denotes the differential on the Hom complexes, then $D \circ \sigast = \sigast \circ D$. 
\end{lem}

\begin{proof}
Since $\big( \bigoplus_{i\in \Z} \tsh^i (X_i, \varphi_i), d_X \big)$ 
and $Y = \big( \bigoplus_{i\in \Z} \tsh^i (Y_i, \psi_i), d_Y \big)$ 
are complexes in $\dgECat$, their differentials are equivariant morphisms.
Hence, Corollary \ref{cor:HomAsInv} implies that both $d_X$ and $d_Y$ are fixed by $\sigast$.
We thus compute
\begin{align*}
\sigast( D(f) ) 
&= \sigast (d_Y \circ f) - (-1)^{|f|} \sigast (f \circ d_X) \\
&\stackrel{\eqref{eq:sigasthomo}}{=} (\sigast d_Y) \circ (\sigast f) - (-1)^{|f|} (\sigast f) \circ \sigast(d_X) \\
&= d_Y \circ (\sigast f) - (-1)^{|f|} (\sigast f) \circ d_X \\
&= D(\sigast(f)) \, .  \qedhere
\end{align*}
\end{proof}

Next, given a morphism $f \in \Hom_{\dgECat}(X,Y)$, 
we can consider the corresponding morphism $\For(f) \in \Hom_{\dgCat}\big( \For(X),\For(Y) \big)$. 
Note that if $f$ is a chain map then so is $\For(f)$. 
Recall that the space of chain maps is denoted $\Hom_{\ChCat} \subset \Hom_{\dgCat}$.

\begin{prop}
	\label{prop:lift-equivariant-homotopy}
Let $X =\big( \bigoplus_{i\in \Z} \tsh^i (X_i, \varphi_i), d_X \big)$ 
and $Y = \big( \bigoplus_{i\in \Z} \tsh^i (Y_i, \psi_i), d_Y \big)$ 
be complexes in $\dgECat$ and let $f \in \Hom_{\ChECat}(X, Y)$.
If $\For(f) \in \Hom_{\ChCat}\big( \For(X), \For(Y) \big)$
is a homotopy equivalence, then $f$ is a homotopy equivalence.
\end{prop}

\begin{proof}
Suppose $\For(f)$ is a homotopy equivalence.  
Thus, there exists degree-zero $g \in \Hom_{\dgCat} \big( \For(Y), \For(X) \big)$ such that $D(g)=0$,
and $\tsh$-degree $-1$ morphisms 
$h_X \colon \For(X) \to \For(X)$ and $h_Y \colon \For(Y) \to \For(Y)$ in $\dgCat$, such that 
\[
\id_X - g \circ f= D(h_X) \quad \text{and} \quad \id_Y- f\circ g= D(h_Y) \, .
\]

Set
\[
G := \frac{1}{2}(g + \sigast g) \, , \quad H_X := \frac{1}{2}(h_X + \sigast h_X) \, , \quad H_Y := \frac{1}{2}(h_Y + \sigast h_Y)
\]
and note that $G$, $H_X$, and $H_Y$ are all morphisms in $\dgECat$. 
Using Lemma \ref{lem:sigastchainmap}, we compute
\[
D(G) = \frac{1}{2} D(g+\sigast g)
=\frac{1}{2}(D(g)+D(\sigast g))
= \frac{1}{2}(D(g) +\sigast(D(g))) = 0
\]
and that
\begin{align*}
\id_X-G\circ f & = \id_X - \frac{1}{2} (g + \sigast g)\circ f \\
&= \frac{1}{2} \big( (\id_X - g\circ f) + (\id_X - (\sigast g) \circ f) \big) \\
&\!\!\!\!\!\!\!\stackrel{{Cor. \ref{cor:HomAsInv}}}{=} \frac{1}{2} \big( (\id_X - g\circ f) 
	+ ((\sigast \id_X) - (\sigast g) \circ (\sigast f) ) \big) \\
&= \frac{1}{2}\big( (\id_X- g\circ f) + \sigast( \id_X -  g \circ f) \big)\\
&=\frac{1}{2} \big( D(h_X)+ \sigast (D(h_X)) \big) \\
&=\frac{1}{2} \big( D(h_X)+D(\sigast h_X) \big) \\
&=D\left(\frac{1}{2}(h_X+\sigast h_X)\right)\\
&=D(H_X) \, . 
\end{align*}
A similar computation shows that $\id_Y- f\circ G= D(H_Y)$. 
Thus, $f$, $G$, $H_X$, and $H_Y$ give the data of a homotopy equivalence
in $\dgECat$.
\end{proof}

\subsection{Braid relations for equivariant Rickard complexes}
	\label{ss:eBraidrel}

We now prove the analogue of Theorem \ref{thm:Rickard} 
for the equivariant Rickard complexes.
In doing so, we will make use of the
$2$-functor $\Gamma \colon \Kar(\SSStn_q(\glm))\rightarrow \Catn$
from \S \ref{ss:RickCan}.
Since $\BnFoam \hookrightarrow \Kar(\SSStn_q(\glm))$, 
we can consider the $2$-functor $\Gamma$ restricted to $\BnFoam$, 
which we still denote by $\Gamma$. 
Observe that 
\[
\Gamma \colon \BnFoam \rightarrow \Catn
\]
is a full and strongly monoidal, which follows from Proposition \ref{P:there-is-Gamma}
and the observation that the monoidal unit $\one_\wtn \in \BnFoam$ is, 
by definition, sent to the left regular representation of 
$H^\ast(Gr_n(\mathbb{C}^{2n}))$, which is the monoidal unit in $\Catn$.
Recall that all indecomposable objects in $\BnFoam$ other than (shifts of) $\one_\wtn$ 
are sent to zero by $\Gamma$.

Given any (graded) commutative ring $R$, we can consider the monoidal category of 
finitely generated free (graded) $R$-modules $\Cat_{R}$. 
(The category $\Catn$ is the special case when $R = H^\ast(Gr_n(\mathbb{C}^{2n}))$.)
A ring automorphism of $R$ determines an additive, monoidal endofunctor of $\Cat_{R}$, 
which is the identity on objects and is determined on morphisms by acting via the automorphism 
on $\End_{\Cat_R}(R) \cong R$.

\begin{defn}
Define a monoidal involution of $\Catn$, which abusing notation we denote $\tau$, 
to be the monoidal functor induced by 
the (graded) ring automorphism of $H^\ast(Gr_n(\mathbb{C}^{2n}))$ 
which acts on the basis of Schur polynomials by
$\Schur(\X)\mapsto \mathfrak{s}_{\lambda^{t}}(\X)$.
\end{defn}

\begin{lem} \label{lem:tauoncellquot} 
We have the following equality of monoidal functors: $\Gamma\circ \tau = \tau \circ \Gamma$.  
\end{lem}

\begin{proof} 
Immediate from the description of $\Gamma$ in Proposition \ref{P:there-is-Gamma} 
and the description of $\tau$ in Definition \ref{def:symmetry}.
\end{proof}

Now, let $\bword$ be a braid word and write the complex 
$C^{\tau}(\bword) \in \dgCat[\EBnFoam]$ as
\begin{equation}
	\label{eq:CTbword}
C^{\tau}(\bword) =: \left( \bigoplus_{j\in\Z} \tsh^j \big(C_j(\bword), \varphi_j(\bword) \big) , d_{\bword} \right) \, .
\end{equation}
If follows from Definition \ref{def:eRickard} that $C_j(\bword)$ is a direct sum of tensor products of 
shifts of objects of the form $\FE_i^{(k)}$, 
and that $\varphi_j(\bword)$ is a (diagonal) 
matrix whose entries are tensor products of the morphisms $\pm \swcl_k$.
Further, Remark \ref{rem:LCchainmap} implies that the collection of maps
$\big\{ \varphi_j(\bword) \big\}_{j \in \Z}$ gives a chain map 
$\varphi_{\bword} \colon C(\bword )\to \tau(C(\bword))$ in $\dgCat[\BnFoam]$. 

\begin{lem}
	\label{L:equivariant-structure-map-cell-quotient}
If $\bword$ is a braid word, 
then $\Gamma(\For(\varphi_{\bword})) = \id_{(\qsh^{-n}\tsh^{n})^{\epsilon(\bword)}\one_{\wtn}}$.
\end{lem}

\begin{proof}
Since $\Gamma$ and $\For$ are monoidal,
it suffices to establish the result for $\bword = \bgen_i^{\pm1}$.
In this case, both $C(\bgen_i^{\pm1})$ and $\tau(C(\bgen_i^{\pm1}))$ 
are mapped by $\Gamma$ to 
$\qsh^{\mp n}\tsh^{\pm n}\FE_i^{(0)}=\qsh^{\mp n}\tsh^{\pm n}\one_{\wtn}$, 
and the sole component of the chain map $\Gamma(\varphi_{\bgen_i^{\pm1}})$ is 
$(-1)^{{n+1 \choose 2} + {n+1 \choose 2}}\swcl_{0} = \id_{\one_{\wtn}}$. 
\end{proof}

\begin{lem}
	\label{L:canonical-homotopy-equivalence-cell-quotient}
Let $\bword$ and $\bword'$ be two braid words for the same braid.
If $f \in \Hom_{\ChCat[\BnFoam]}(C(\bword), C(\bword'))$ is a chain map such that 
$\Gamma(f) = \id_{(\qsh^{-n}\tsh^{n})^{\epsilon(\bword)}\one_{\wtn}}$,
then the same is true for $\tau(f)$.
\end{lem}
\begin{proof}
That $\tau(f)$ is a chain map follows from Remark \ref{R:tau-of-chainmap-is-chainmap}. 
The result then follows from Lemma \ref{lem:tauoncellquot}, 
since $\tau$ preserves the identity map of $\one_{\wtn}$ both before and after applying $\Gamma$.
\end{proof}

Pairing Lemmata \ref{L:detect-equivalence-cell-quotient}, 
\ref{L:equivariant-structure-map-cell-quotient}, and
\ref{L:canonical-homotopy-equivalence-cell-quotient}
with the results from \S \ref{ss:lifting}, 
we now establish that the complexes $C^{\tau}(\bgen_{i_1}^{\epsilon_1} \cdots \bgen_{i_r}^{\epsilon_r})$ 
canonically satisfy the braid relations in $\hCat[\EBnFoam]$.
Precisely, we have the following:

\begin{thm}
	\label{thm:Ctaubraid}
Let $\K$ be a field in which $2 \neq 0$.
In $\dgCat[\EBnFoam]$,
there are homotopy equivalences
\begin{equation}
	\label{eq:Ctaubraid}
C^{\tau}(\bgen_{i}\bgen_{i+1}\bgen_{i}) \simeq C^{\tau}(\bgen_{i+1}\bgen_{i}\bgen_{i+1})
\end{equation}
for $1 \leq i \leq m-2$ and
\begin{equation}
	\label{eq:Ctauinv}
C^{\tau}(\bgen_i)C^{\tau}(\bgen_i^{-1}) \simeq (\one_{\wtn}, \id_{\one_{\wtn}}) \simeq C^{\tau}(\bgen_i^{-1})C^{\tau}(\bgen_i)
\end{equation}
for $1 \leq i \leq m-1$. 
These homotopy equivalences can be chosen so that their image under 
$\Gamma \circ \For$ is the identity map of the appropriate shift of $\one_{\wtn}$. 
Consequently, given two braid words $\bword$ and $\bword'$ for the same braid, 
any two homotopy equivalences $C^{\tau}(\bword) \simeq C^{\tau}(\bword')$ 
that are given as compositions of \eqref{eq:Ctaubraid} and \eqref{eq:Ctauinv} agree 
in $\hCat[\EBnFoam]$.
\end{thm}

\begin{proof}
By Theorem \ref{thm:Rickard}, before accounting for equivariant structures,
there are homotopy equivalences 
\[
f_{i} \colon C(\beta_{i}\beta_{i+1}\beta_{i}) \xrightarrow{\simeq} C(\beta_{i+1}\beta_{i}\beta_{i+1})
\]
and
\[
g_{i} \colon \one_{\wtn} \xrightarrow{\simeq} C(\bgen_{i}) C(\bgen_{i}^{-1})
\, , \quad 
g_{i}' \colon \one_{\wtn} \xrightarrow{\simeq} C(\bgen_{i}^{-1}) C(\bgen_{i}) \, .
\] 
Moreover, using Lemma \ref{L:detect-equivalence-cell-quotient} we can choose
$f_{i}$, $g_{i}$, and $g_{i}'$ so that they each map to $\id_{\one_{\wtn}}$ 
under $\Gamma$. 

Combining Lemma \ref{L:equivariant-structure-map-cell-quotient} and 
Lemma \ref{L:canonical-homotopy-equivalence-cell-quotient}
with the definition of $\tauast$ from \eqref{eq:sigastdef},
we see that $\tauast f_{i}$, $\tauast g_{i}$, and $\tauast g_{i}'$ 
all also map to $\id_{\one_{\wtn}}$ under the quotient functor. 
It follows that the same is true for each of
\[
f_{i}^{\tau} := \frac{1}{2} \big( f_{i}+\tauast f_{i} \big)
\, , \quad
g_{i}^{\tau} := \frac{1}{2} \big( g_{i}+\tauast g_{i} \big)
\, , \quad
g_{i}'^{\tau} := \frac{1}{2} \big( g_{i}'+\tauast g_{i}' \big) \, .
\] 
Lemma \ref{L:detect-equivalence-cell-quotient} implies that each of 
$f_{i}^{\tau}$, $g_{i}^{\tau}$, and $g_{i}'^{\tau}$ 
are therefore homotopy equivalences in $\dgCat[\BnFoam]$.
Each of these maps is fixed by $\tauast$, 
so they give chain maps
\begin{equation}
	\label{eq:Ctaubraidmaps}
f_{i}^{\tau} \colon C^{\tau}(\beta_{i}\beta_{i+1}\beta_{i}) \to C^{\tau}(\beta_{i+1}\beta_{i}\beta_{i+1})
\, , \quad
g_{i}^{\tau} \colon \one_{\wtn} \to C^{\tau}(\bgen_{i}) C^{\tau}(\bgen_{i}^{-1})
\, , \quad
g_{i}'^{\tau} \colon \one_{\wtn} \to C^{\tau}(\bgen_{i}^{-1}) C^{\tau}(\bgen_{i})
\end{equation}
in $\dgCat[\EBnFoam]$.
Proposition \ref{prop:lift-equivariant-homotopy} then shows that each of the maps in 
\eqref{eq:Ctaubraidmaps} is a homotopy equivalence.

The exact same argument as in the proof of Lemma \ref{L:detect-equivalence-cell-quotient} implies that
\[
\Hom_{\hCat[\EBnFoam]} \big( C^{\tau}(\bword), C^{\tau}(\bword') \big) \cong \K
\]
when $\bword$ and $\bword'$ are words for the same braid.
Thus one recovers an equivariant version of Lemma \ref{L:detect-equivalence-cell-quotient}: 
if $f$ is a chain map in this morphism space, 
and $\Gamma(\For(f)) = c \cdot \id_{\one_{\wtn}}$ for $c \in \K$, 
then either $c \ne 0$ and $f$ is a homotopy equivalence, or $c=0$ and $f$ is nulhomotopic. 
Hence two chain maps $f, g$ in this morphism space are homotopic if and only if they agree 
after applying $\Gamma \circ \For$. 
The final statement of the theorem is an immediate consequence. \end{proof}

Theorem \ref{thm:Ctaubraid} shows that we can canonically associate 
a complex over $\EBnFoam$ to a braid $\bgen$ by choosing a braid word representative 
$\bword$ and considering $C^{\tau}(\bword)$.
We thus will slightly abuse notation and denote this complex 
as $C^{\tau}(\bgen)$ moving forward.

\subsection{Spin link homology}
	\label{ss:eHomology}

At last, we define our invariant.
We proceed analogously to \eqref{eq:KhRdef},
and then establish that we indeed obtain an invariant of framed links $\mathcal{L} \subset S^3$.
To simplify notation, we set
\[
(\one_{\wtn},\pm) := (\one_{\wtn},\pm \id_{\one_{\wtn}}) \in \EBnFoam \, .
\]
For the duration, when we consider the $\tauast$-action on $\Hom$-spaces of the form 
$\Hom_{\BnFoam} \big( \one_{\wtn} , X \big)$, it will be assumed with respect to the 
equivariant structure $+ \id_{\one_{\wtn}}$ on $\one_{\wtn}$
(and whatever equivariant structure $\varphi_X$ is currently being considered on $X$).

Recall from \S \ref{ss:conv} that $\sVect^\Z_\K$ denotes the category of $\Z$-graded 
super $\K$-vector spaces, and that we denote shift in super degree by $\ssh$.
We will implicitly view $\Vect^\Z_\K \subset \sVect^\Z_\K$ as the full subcategory 
of super vector spaces concentrated in super degree zero.

\begin{defn}
	\label{def:SLH}
Consider the additive functors $\RR^+,\RR^- \colon \EBnFoam \to \Vect^\Z_\K$ given by
\begin{equation}
	\label{eq:Hom1X}
\RR^+(x) :=
\Hom_{\EBnFoam}\big( (\one_{\wtn},+), x \big) \, , \quad
\RR^-(x) := \Hom_{\EBnFoam}\big( (\one_{\wtn},-) , x \big)
\end{equation}
and similarly denote the
induced functors $\RR^+, \RR^- \colon \hCat[\EBnFoam] \to \hCat[\Vect^\Z_\K]$ given by applying \eqref{eq:Hom1X} term-wise.
Let
\[
\SLC{\bgen}^+ 
	:= \qsh^{-mn^2} \RR^+ \big(C^{\tau}(\bgen) \big) \, , \quad
\SLC{\bgen}^- 
	:= \qsh^{-mn^2} \RR^- \big(C^{\tau}(\bgen) \big)
\]
and set
$\SLC{\bgen}
:= \SLC{\bgen}^+ \oplus \ssh \SLC{\bgen}^-
\in \hCat[\sVect^\Z_\K]$.
\end{defn}

It is immediate from Theorem \ref{thm:Ctaubraid} that 
$\SLC{\bgen} \in \hCat[\sVect^\Z_\K]$ is an invariant 
of braids $\bgen \in \brgroup$. 
We will establish that the assignment $\bgen \mapsto \SLC{\bgen}$ 
is further invariant under the Markov moves up to (appropriate) shift, 
and therefore is an invariant of the framed link $\mathcal{L}_\bgen$ given as the closure of $\bgen$.

Our approach is as follows.
Recall that the $\sln[2n]$ link invariant $\llbracket \bgen \one_{\wtn} \rrbracket_{2n}$ 
comes from applying the functor $\Hom(\one_{\wtn},-)$ to each term in the complex $C(\beta)$, 
yielding a complex whose chain group in degree $j$ is 
$\Hom_{\BnFoam} \big(\one_{\wtn} , C_j(\bgen) \big)$. 
Using Corollary \ref{cor:HomAsInv}, one has
\[
\Hom_{\EBnFoam} \big( (\one_{\wtn},+) ,  (C_j(\bgen),\varphi_j(\bgen)) \big) 
= \Hom_{\BnFoam} \big( \one_{\wtn} , C_j(\bgen) \big)^{\tau},
\]
the subspace of invariants under the $\tauast$-action. 
Thus $\SLC{\bgen}^+$ is a subcomplex of $\llbracket \bgen \one_{\wtn} \rrbracket_{2n}$, 
the subspace of $\tauast$-invariants. 
Similarly, $\SLC{\bgen}^-$ is a subspace of $\llbracket \bgen \one_{\wtn} \rrbracket_{2n}$, 
the subcomplex of $\tauast$-anti-invariants 
(i.e.~the isotypic component of the sign representation). 
Lemma \ref{lem:sigastchainmap} then implies that 
$\SLC{\bgen}^{\pm} \subset \llbracket \bgen \one_{\wtn} \rrbracket_{2n}$
are in fact subcomplexes.

By Theorem \ref{thm:KhR}, we know that $\llbracket \bgen \one_{\wtn} \rrbracket_{2n}$ satisfies the 
Markov moves, up to homotopy equivalence. 
We will observe that these homotopy equivalences restrict to $\SLC{\bgen}^{\pm}$,
and thus give Markov invariance for $\SLC{\bgen}$.

We begin with a result concerning group actions on chain complexes of vector spaces.

\begin{lem}
	\label{lem:GactC}
Let $G$ be a finite group, and suppose $|G|$ is invertible in $\K$. 
Let $(X,d_X), (Y,d_Y) \in \dgCat[\Vect_{\K}^\Z]$ be complexes of (graded) $\K$-vector spaces 
that admit a $\K$-linear action of $G$, 
i.e.~the chain groups $X_i$ and $Y_i$ are $G$-representations 
and the differentials $d_X$ and $d_Y$ commute with the $G$-action.
Suppose that $f \colon (X,d_X) \xrightarrow{\simeq} (Y,d_Y)$ is a homotopy equivalence such that 
$f|_{X_i} \colon X_i \to Y_i$ is a morphism of $G$-representations. 
Then, $f$ restricts to a homotopy equivalence between the $G$-invariant subcomplexes 
$(X^G,d_X)$ and $(Y^G,d_Y)$ of $(X,d_X)$ and $(Y,d_Y)$.
If $f$ is an isomorphism of chain complexes in $\ChCat[\Vect_{\K}^\Z]$, 
then the restriction is as well.
\end{lem}

\begin{proof}
Throughout, we let $Z = X$ or $Y$.
First, observe that $(Z^G,d_Z) \subset (Z,d_Z)$ is indeed a subcomplex, 
where 
\[
Z_{i}^G := \{z \in Z_i \mid g \cdot z = z \ \  \forall g \in G\} \, ,
\] 
since $d_Z \circ (g \cdot) = (g \cdot) \circ d_Z$ for all $g \in G$. 
Since $f|_{X_i}$ is assumed to be a morphism of $G$-representations, 
$f|_{X^G} \colon (X^G,d_X) \to (Y^G,d_Y)$ is a chain map,
so it remains to construct its homotopy inverse.

The argument is similar to the proof of Proposition \ref{prop:lift-equivariant-homotopy}.
Let $k \colon (Y,d_Y) \to (X,d_X)$ be the homotopy inverse to $f$, 
so there exist $h_X$ and $h_Y$ such that
\begin{equation}
	\label{eq:hedata}
\id_X - k \circ f= d_X \circ h_X + h_X \circ d_X 
\quad \text{and} \quad 
\id_Y- f\circ k = d_Y \circ h_Y + h_Y \circ d_Y \, .
\end{equation}
Set
\[
k^G := \frac{1}{|G|} \sum_{g \in G} (g\cdot) \circ k 
\, , \quad 
h_X^G := \frac{1}{|G|} \sum_{g \in G} (g\cdot) \circ h_X 
\, , \quad 
h_Y^G := \frac{1}{|G|} \sum_{g \in G} (g\cdot) \circ h_Y
\]
and note that $k^G \in \Hom_{\ChCat[\Vect_{\K}^{\Z}]}\big( (Y^G,d_Y) , (X^G,d_X) \big)$
and $h_Z^G \in \End_{\dgCat[\Vect_{\K}^{\Z}]}\big( (Z^G,d_Z) \big)$.
Averaging \eqref{eq:hedata} over $G$ 
(and using $f \circ (g \cdot) = (g \cdot) \circ f$ and $d_Z \circ (g \cdot) = (g \cdot) \circ d_Z$)
gives
\[
\left( \frac{1}{|G|} \sum_{g \in G} (g\cdot) \circ \id_X \right) - k^G \circ f= d_X \circ h_X^G + h_X^G \circ d_X 
\quad \text{and} \quad 
\left( \frac{1}{|G|} \sum_{g \in G} (g\cdot) \circ \id_Y \right) - f \circ k^G = d_Y \circ h_Y^G + h_Y^G \circ d_Y \, .
\]
Since $\left(\frac{1}{|G|} \sum_{g \in G} (g\cdot) \circ \id_Z \right)|_{Z^G} = \id_{Z^G}$, 
this shows that $f|_{X^G} \colon (X^G,d_X) \to (Y^G,d_Y)$ is a homotopy equivalence.

For the final statement, note that the statement that $f$ is an isomorphism
is equivalent to being able to choose maps 
so that \eqref{eq:hedata} is satisfied with $h_X = 0$ and $h_Y=0$. 
It then follows that $h_{X}^G = 0$ and $h_{Y}^G=0$,
so $f|_{X^G} \colon (X^G,d_X) \to (Y^G,d_Y)$ is 
an isomorphism of chain complexes as well.
\end{proof}

\begin{rem}
	\label{rem:GactC}
When $G = \Z/2$, 
Lemma \ref{lem:GactC} can also be applied to sign components $(X^\sgn,d_X)$ and $(Y^\sgn,d_Y)$, 
rather than trivial components $(X^G,d_X)$ and $(Y^G,d_Y)$. 
This is because, for an involutive linear map $\sigma$, 
the anti-invariants of $\sigma$ agree with the invariants of $-\sigma$. 
More generally though, 
an analogous argument to the proof of Lemma \ref{lem:GactC} 
(using other central idempotents in the group algebra) 
establishes homotopy equivalences (and isomorphisms) 
between subcomplexes of isotypic components of $(X,d_X)$ and $(Y,d_Y)$. \end{rem}

\begin{prop}
	\label{prop:SLHM1}
The complex $\SLC{\bgen}$ satisfies the first 
Markov move, up to isomorphism in $\ChCat[\sVect_{\K}^{\Z}]$. 
Explicitly, given braids $\bgen, \bgen' \in \brgroup$, 
there is an isomorphism
$\SLC{\bgen \bgen'}
\cong \SLC{\bgen' \bgen}$
of chain complexes of super $\K$-vector spaces.
\end{prop}

\begin{proof}
By Theorem \ref{thm:KhR}, there is a homotopy equivalence of chain complexes
\begin{equation}
	\label{eq:KhRM1}
f \colon
\llbracket \bgen \bgen' \rrbracket_{2n} \xrightarrow{\simeq} 
	\llbracket \bgen' \bgen \rrbracket_{2n} \, .
\end{equation}
In fact, this is an isomorphism of chain complexes, 
that we now explicitly describe.

In the notation of \eqref{eq:CTbword}, 
the term of the complex $\llbracket \bgen \bgen' \rrbracket_{2n}$ 
in homological degree $k$ takes the form
\[
\qsh^{-mn^2}
\bigoplus_{i+j=k} \Hom_{\BnFoam} \big( \one_{\wtn} , C_i(\bgen)C_j(\bgen') \big) \, .
\]
By definition, 
each of $C_i(\bgen)$ and $C_j(\bgen')$ can themselves be written as
\[
C_i(\bgen) = \bigoplus_{p=1}^{\ell_i} C_{i,p} \, , \quad
C_j(\bgen') = \bigoplus_{r=1}^{\ell'_i} C'_{j,r}
\]
where each $C_{i,p}$ and $C'_{j,q}$ is a tensor product of 
objects of the form $\FE_s^{(t)}$.
The isomorphism \eqref{eq:KhRM1} is then given summand-wise as follows:
\begin{equation}
	\label{eq:KhRM1map}
\begin{aligned}
\Hom_{\BnFoam} \big( \one_{\wtn} , C_{i,p}C'_{j,r} \big) &\xrightarrow{f}
\Hom_{\BnFoam} \big( \one_{\wtn} , C'_{j,r}C_{i,p} \big) \\
\begin{tikzpicture}[anchorbase]
	\fill[white] (-1,-.25) rectangle (1,.25);
	\draw[very thick] (-1,-.25) rectangle (1,.25);
	\node at (0,0) {$\xi$};
	\draw[very thick,rdirected=.5] (-.875,.25) to (-.875,.75);
	\draw[very thick,->] (-.75,.25) to (-.75,.75);
	\node at (-.5625,.5625) {$\mysdots$};
	\draw[very thick,rdirected=.5] (-.375,.25) to (-.375,.75);
	\draw[very thick,->] (-.25,.25) to (-.25,.75);
	\draw [thick, decorate, decoration = {brace}] (-1,.8125) to (-.125,.8125);
	\node at (-.5625,1.125) {$C_{i,p}$};
	\draw[very thick,rdirected=.5] (.25,.25) to (.25,.75);
	\draw[very thick,->] (.375,.25) to (.375,.75);
	\node at (.5625,.5625) {$\mysdots$};
	\draw[very thick,rdirected=.5] (.75,.25) to (.75,.75);
	\draw[very thick,->] (.875,.25) to (.875,.75);
	\draw [thick, decorate, decoration = {brace}] (.125,.8125) to (1,.8125);
	\node at (.5625,1.125) {$C'_{j,r}$};
	\node at (1.5,.25) {$\wtn$};
\end{tikzpicture}
&\mapsto
(-1)^{ij} \
\begin{tikzpicture}[anchorbase]
	\fill[white] (-1,-.25) rectangle (1,.25);
	\draw[very thick] (-1,-.25) rectangle (1,.25);
	\node at (0,0) {$\xi$};
	\draw[very thick,rdirected=.5] (-.875,.25) to [out=90,in=0] (-1,.5) 
		to [out=180,in=90] (-1.125,.25) to (-1.125,-.25) to [out=270,in=180] (.125,-.75) 
			to [out=0,in=270] (1.375,-.25) to (1.375,.75);
	\draw[very thick,->] (-.75,.25) to [out=90,in=0] (-1,.625) 
		to [out=180,in=90] (-1.25,.25) to (-1.25,-.25) to [out=270,in=180] (.125,-.875) 
			to [out=0,in=270] (1.5,-.25) to (1.5,.75);
	\node[rotate=270] at (.125,-1.0625) {$\mysdots$};
	\draw[very thick,rdirected=.5] (-.375,.25) to [out=90,in=0] (-1,1) 
		to [out=180,in=90] (-1.625,.25) to (-1.625,-.25) to [out=270,in=180] (.125,-1.25) 
			to [out=0,in=270] (1.875,-.25) to (1.875,.75);
	\draw[very thick,->] (-.25,.25) to [out=90,in=0] (-1,1.125) 
		to [out=180,in=90] (-1.75,.25) to (-1.75,-.25) to [out=270,in=180] (.125,-1.375) 
			to [out=0,in=270] (2,-.25) to (2,.75);
	\draw [thick, decorate, decoration = {brace}] (1.25,.8125) to (2.125,.8125);
	\node at (1.6875,1.125) {$C_{i,p}$};
	\draw[very thick,rdirected=.5] (.25,.25) to (.25,.75);
	\draw[very thick,->] (.375,.25) to (.375,.75);
	\node at (.5625,.5625) {$\mysdots$};
	\draw[very thick,rdirected=.5] (.75,.25) to (.75,.75);
	\draw[very thick,->] (.875,.25) to (.875,.75);
	\draw [thick, decorate, decoration = {brace}] (.125,.8125) to (1,.8125);
	\node at (.5625,1.125) {$C'_{j,r}$};
	\node at (2.5,.25) {$\wtn$};
\end{tikzpicture}
\end{aligned} \, .
\end{equation}
Here, the sign $(-1)^{ij}$ ensures that this indeed gives a chain map, 
since the differential on the complex $C_i(\bgen)C_j(\bgen')$ follows the usual 
Koszul sign rule for a tensor product of complexes.

As noted above, 
$\SLC{\bgen \bgen'}^{+}$ 
and $\SLC{\bgen' \bgen}^{+}$  
are precisely the $\Z/2$-invariant subcomplexes for the $\tauast$-actions on 
$\llbracket \bgen \bgen' \rrbracket_{2n}$ and $\llbracket \bgen' \bgen \rrbracket_{2n}$
(respectively), 
while $\SLC{\bgen \bgen'}^{-}$ 
and $\SLC{\bgen' \bgen}^{-}$ 
are the sign component subcomplexes.
By Lemma \ref{lem:GactC} and Remark \ref{rem:GactC}, 
it therefore suffices to show that the assignment 
\eqref{eq:KhRM1map} is a morphism of $\Z/2$-representations, 
i.e.~that $f \circ \tauast = \tauast \circ f$.
For this, we compute
\[
f \left(\tauast 
\begin{tikzpicture}[anchorbase]
	\fill[white] (-1,-.25) rectangle (1,.25);
	\draw[very thick] (-1,-.25) rectangle (1,.25);
	\node at (0,0) {$\xi$};
	\draw[very thick,rdirected=.5] (-.875,.25) to (-.875,.75);
	\draw[very thick,->] (-.75,.25) to (-.75,.75);
	\node at (-.5625,.5625) {$\mysdots$};
	\draw[very thick,rdirected=.5] (-.375,.25) to (-.375,.75);
	\draw[very thick,->] (-.25,.25) to (-.25,.75);
	\draw [thick, decorate, decoration = {brace}] (-1,.8125) to (-.125,.8125);
	\node at (-.5625,1.125) {$C_{i,p}$};
	\draw[very thick,rdirected=.5] (.25,.25) to (.25,.75);
	\draw[very thick,->] (.375,.25) to (.375,.75);
	\node at (.5625,.5625) {$\mysdots$};
	\draw[very thick,rdirected=.5] (.75,.25) to (.75,.75);
	\draw[very thick,->] (.875,.25) to (.875,.75);
	\draw [thick, decorate, decoration = {brace}] (.125,.8125) to (1,.8125);
	\node at (.5625,1.125) {$C'_{j,r}$};
	\node at (1.5,.25) {$\wtn$};
\end{tikzpicture} \right) =
f \left( \,
\begin{tikzpicture}[anchorbase]
	\fill[white] (-1,-.25) rectangle (1,.25);
	\draw[very thick] (-1,-.25) rectangle (1,.25);
	\node at (0,0) {$\tau(\xi)$};
	\draw[very thick,rdirected=.25] (-.75,.25) to [out=90,in=270] (-.875,.75);
	\draw[very thick,->] (-.875,.25) to [out=90,in=270] (-.75,.75);
	\node at (-.5625,.5625) {$\mysdots$};
	\draw[very thick,rdirected=.25] (-.25,.25) to [out=90,in=270] (-.375,.75);
	\draw[very thick,->] (-.375,.25) to [out=90,in=270] (-.25,.75);
	\draw [thick, decorate, decoration = {brace}] (-1,.8125) to (-.125,.8125);
	\node at (-.5625,1.125) {$C_{i,p}$};
	\draw[very thick,rdirected=.25] (.375,.25) to [out=90,in=270] (.25,.75);
	\draw[very thick,->] (.25,.25) to [out=90,in=270] (.375,.75);
	\node at (.5625,.5625) {$\mysdots$};
	\draw[very thick,rdirected=.25] (.875,.25) to [out=90,in=270] (.75,.75);
	\draw[very thick,->] (.75,.25) to [out=90,in=270] (.875,.75);
	\draw [thick, decorate, decoration = {brace}] (.125,.8125) to (1,.8125);
	\node at (.5625,1.125) {$C'_{j,r}$};
	\node at (1.5,.25) {$\wtn$};
\end{tikzpicture}
\right) = (-1)^{ij} \
\begin{tikzpicture}[anchorbase]
	\fill[white] (-1,-.25) rectangle (1,.25);
	\draw[very thick] (-1,-.25) rectangle (1,.25);
	\node at (0,0) {$\tau(\xi)$};
	\draw[very thick,rdirected=.5] (-.75,.25) to [out=90,in=0] (-1,.625) 
		to [out=180,in=90] (-1.125,.25) to (-1.125,-.25) to [out=270,in=180] (.125,-.75) 
			to [out=0,in=270] (1.375,-.25) to (1.375,.75);
	\draw[very thick,->] (-.875,.25) to [out=90,in=270] (-.75,.625) to [out=90,in=0] (-1,.875) 
		to [out=180,in=90] (-1.25,.25) to (-1.25,-.25) to [out=270,in=180] (.125,-.875) 
			to [out=0,in=270] (1.5,-.25) to (1.5,.75);
	\node[rotate=270] at (.125,-1.0625) {$\mysdots$};
	\draw[very thick,rdirected=.5] (-.25,.25) to [out=90,in=270] (-.375,.75) to [out=90,in=0] (-1,1.25) 
		to [out=180,in=90] (-1.625,.25) to (-1.625,-.25) to [out=270,in=180] (.125,-1.25) 
			to [out=0,in=270] (1.875,-.25) to (1.875,.75);
	\draw[very thick,->] (-.375,.25) to [out=90,in=270] (-.25,.75) to [out=90,in=0] (-1,1.375) 
		to [out=180,in=90] (-1.75,.25) to (-1.75,-.25) to [out=270,in=180] (.125,-1.375) 
			to [out=0,in=270] (2,-.25) to (2,.75);
	\draw [thick, decorate, decoration = {brace}] (1.25,.8125) to (2.125,.8125);
	\node at (1.6875,1.125) {$C_{i,p}$};
	\draw[very thick,rdirected=.25] (.375,.25) to [out=90,in=270] (.25,.75);
	\draw[very thick,->] (.25,.25) to [out=90,in=270] (.375,.75);
	\node at (.5625,.5625) {$\mysdots$};
	\draw[very thick,rdirected=.25] (.875,.25) to [out=90,in=270] (.75,.75);
	\draw[very thick,->] (.75,.25) to [out=90,in=270] (.875,.75);
	\draw [thick, decorate, decoration = {brace}] (.125,.8125) to (1,.8125);
	\node at (.5625,1.125) {$C'_{j,r}$};
	\node at (2.5,.25) {$\wtn$};
\end{tikzpicture},
\]
which agrees with the following calculation. 
\[
\tauast f \left( \,
\begin{tikzpicture}[anchorbase]
	\fill[white] (-1,-.25) rectangle (1,.25);
	\draw[very thick] (-1,-.25) rectangle (1,.25);
	\node at (0,0) {$\xi$};
	\draw[very thick,rdirected=.5] (-.875,.25) to (-.875,.75);
	\draw[very thick,->] (-.75,.25) to (-.75,.75);
	\node at (-.5625,.5625) {$\mysdots$};
	\draw[very thick,rdirected=.5] (-.375,.25) to (-.375,.75);
	\draw[very thick,->] (-.25,.25) to (-.25,.75);
	\draw [thick, decorate, decoration = {brace}] (-1,.8125) to (-.125,.8125);
	\node at (-.5625,1.125) {$C_{i,p}$};
	\draw[very thick,rdirected=.5] (.25,.25) to (.25,.75);
	\draw[very thick,->] (.375,.25) to (.375,.75);
	\node at (.5625,.5625) {$\mysdots$};
	\draw[very thick,rdirected=.5] (.75,.25) to (.75,.75);
	\draw[very thick,->] (.875,.25) to (.875,.75);
	\draw [thick, decorate, decoration = {brace}] (.125,.8125) to (1,.8125);
	\node at (.5625,1.125) {$C'_{j,r}$};
	\node at (1.5,.25) {$\wtn$};
\end{tikzpicture} \right) =
(-1)^{ij} \tauast \
\begin{tikzpicture}[anchorbase]
	\fill[white] (-1,-.25) rectangle (1,.25);
	\draw[very thick] (-1,-.25) rectangle (1,.25);
	\node at (0,0) {$\xi$};
	\draw[very thick,rdirected=.5] (-.875,.25) to [out=90,in=0] (-1,.5) 
		to [out=180,in=90] (-1.125,.25) to (-1.125,-.25) to [out=270,in=180] (.125,-.75) 
			to [out=0,in=270] (1.375,-.25) to (1.375,.75);
	\draw[very thick,->] (-.75,.25) to [out=90,in=0] (-1,.625) 
		to [out=180,in=90] (-1.25,.25) to (-1.25,-.25) to [out=270,in=180] (.125,-.875) 
			to [out=0,in=270] (1.5,-.25) to (1.5,.75);
	\node[rotate=270] at (.125,-1.0625) {$\mysdots$};
	\draw[very thick,rdirected=.5] (-.375,.25) to [out=90,in=0] (-1,1) 
		to [out=180,in=90] (-1.625,.25) to (-1.625,-.25) to [out=270,in=180] (.125,-1.25) 
			to [out=0,in=270] (1.875,-.25) to (1.875,.75);
	\draw[very thick,->] (-.25,.25) to [out=90,in=0] (-1,1.125) 
		to [out=180,in=90] (-1.75,.25) to (-1.75,-.25) to [out=270,in=180] (.125,-1.375) 
			to [out=0,in=270] (2,-.25) to (2,.75);
	\draw [thick, decorate, decoration = {brace}] (1.25,.8125) to (2.125,.8125);
	\node at (1.6875,1.125) {$C_{i,p}$};
	\draw[very thick,rdirected=.5] (.25,.25) to (.25,.75);
	\draw[very thick,->] (.375,.25) to (.375,.75);
	\node at (.5625,.5625) {$\mysdots$};
	\draw[very thick,rdirected=.5] (.75,.25) to (.75,.75);
	\draw[very thick,->] (.875,.25) to (.875,.75);
	\draw [thick, decorate, decoration = {brace}] (.125,.8125) to (1,.8125);
	\node at (.5625,1.125) {$C'_{j,r}$};
	\node at (2.5,.25) {$\wtn$};
\end{tikzpicture} 
= (-1)^{ij} \
\begin{tikzpicture}[anchorbase]
	\fill[white] (-1,-.25) rectangle (1,.25);
	\draw[very thick] (-1,-.25) rectangle (1,.25);
	\node at (0,0) {$\tau(\xi)$};
	\draw[very thick,->] (-.875,.25) to [out=90,in=0] (-1,.5) 
		to [out=180,in=90] (-1.125,.25) to (-1.125,-.25) to [out=270,in=180] (.125,-.75) 
			to [out=0,in=270] (1.375,-.25) to (1.375,.25) to [out=90,in=270] (1.5,.75);
	\draw[very thick,rdirected=.5] (-.75,.25) to [out=90,in=0] (-1,.625) 
		to [out=180,in=90] (-1.25,.25) to (-1.25,-.25) to [out=270,in=180] (.125,-.875) 
			to [out=0,in=270] (1.5,-.25) to (1.5,.25) to [out=90,in=270] (1.375,.75);
	\node[rotate=270] at (.125,-1.0625) {$\mysdots$};
	\draw[very thick,->] (-.375,.25) to [out=90,in=0] (-1,1) 
		to [out=180,in=90] (-1.625,.25) to (-1.625,-.25) to [out=270,in=180] (.125,-1.25) 
			to [out=0,in=270] (1.875,-.25) to (1.875,.25) to [out=90,in=270] (2,.75);
	\draw[very thick,rdirected=.5] (-.25,.25) to [out=90,in=0] (-1,1.125) 
		to [out=180,in=90] (-1.75,.25) to (-1.75,-.25) to [out=270,in=180] (.125,-1.375) 
			to [out=0,in=270] (2,-.25) to (2,.25) to [out=90,in=270] (1.875,.75);
	\draw [thick, decorate, decoration = {brace}] (1.25,.8125) to (2.125,.8125);
	\node at (1.6875,1.125) {$C_{i,p}$};
	\draw[very thick,rdirected=.25] (.375,.25) to [out=90,in=270] (.25,.75);
	\draw[very thick,->] (.25,.25) to [out=90,in=270] (.375,.75);
	\node at (.5625,.5625) {$\mysdots$};
	\draw[very thick,rdirected=.25] (.875,.25) to [out=90,in=270] (.75,.75);
	\draw[very thick,->] (.75,.25) to [out=90,in=270] (.875,.75);
	\draw [thick, decorate, decoration = {brace}] (.125,.8125) to (1,.8125);
	\node at (.5625,1.125) {$C'_{j,r}$};
	\node at (2.5,.25) {$\wtn$};
\end{tikzpicture}
\]
The last equality is most easily seen by applying 
\eqref{eq:tauonthickcap} and \eqref{eq:tauonthickcup}, 
which shows that, in the weights appearing in the computation, 
$\tau$ does not contribute any signs when acting on the caps/cups.
\end{proof}

Next, we establish invariance of $\SLC{\bgen}$ 
under the second Markov move, up to shift. 
In type $A$, Markov II invariance amounts to a homotopy equivalence
\begin{equation} \label{typeAmarkovII}
\llbracket \bgen \one_{\wtn} \rrbracket_{2n} \simeq 
\qsh^{\mp n(n+1)} \tsh^{\pm n}
\llbracket \bgen_m^{\pm 1} \bgen \one_{\wtn} \rrbracket_{2n}
\end{equation}
where $\bgen \in \brgroup$ (and\footnote{Recall that $\bgen_m^{\pm 1} \bgen \in \brgroup[m+1]$ 
is obtained by sending $\bgen \in \brgroup$ to its image after applying 
the standard inclusion $\brgroup \hookrightarrow \brgroup[m+1]$, 
and then multiplying with the new Artin generator or its inverse.} 
$\bgen_m^{\pm 1} \bgen \in \brgroup[m+1]$). 
Our approach, as in the proof of Proposition \ref{prop:SLHM1}, 
is to prove that the chain map realizing \eqref{typeAmarkovII} is a morphism 
of $\Z/2$-representations for the relevant $\tauast$-actions. 
Unfortunately, 
the homotopy equivalence \eqref{typeAmarkovII} is not explicitly given in the literature, 
but rather is typically proven indirectly; see e.g.~\cite[Lemma 14.7]{Wu}. 
However, using results of Wu \cite[Section 14.2]{Wu} 
concerning the ``topologically local'' nature of the known Markov II invariance, 
we are able to explicitly write down\footnote{The explicit form of this homotopy equivalence 
may be of independent interest, even in the setting of type $A$ link homologies.} 
the relevant homotopy equivalence and check the $\Z/2$-equivariance.

Let us explain the general outline in more detail. 
There are two essential ways in which $\llbracket \bgen \one_{\wtn} \rrbracket_{2n}$ 
a priori differs from $\llbracket \bgen_m^{\pm 1} \bgen \one_{\wtn} \rrbracket_{2n}$. 
The first is obvious: the complex $C(\bgen_m^{\pm 1} \bgen)$ is obtained from $C(\bgen)$ 
by tensoring with an extra factor $C(\bgen_m^{\pm 1})$. 
The second is more subtle: the computation of $\Hom$ spaces takes place in different categories
(associated to $m$ and $m+1$, respectively). 
In particular, the endomorphism ring of $\one_{\wtn}$ in $\BnFoam[m+1]$ 
has an extra alphabet's worth of symmetric functions $\Sym(\X_{m+1})$,
and therefore an extra tensor factor of $H^\ast(\Gr{n}{2n})$, 
compared to the endomorphism ring in $\BnFoam$.

In \cite[Lemma 14.8]{Wu}, Wu actually proves a stronger result than 
the homotopy equivalence \eqref{typeAmarkovII}.
His results imply that\footnote{This is the content, in our present setup, 
of the phrase ``knotted MOY graphs''
in Wu's Lemma 14.8.}
given any object $X \in \BnFoam$, 
there is a homotopy equivalence
\begin{equation} \label{objectmarkovII}
\llbracket \bgen_{m}^{\pm 1} X \rrbracket_{2n} 
\xrightarrow{\simeq} 
\qsh^{\mp n(n+1)} \tsh^{\pm n} \llbracket X \rrbracket_{2n}\, .
\end{equation}
(Here, we slightly abuse notation in setting
$\llbracket X \rrbracket_{2n} :=
\qsh^{-mn^2} \Hom_{\cSSStn_q(\glm)}(\one_{\wtn} , X)$ 
for $X \in \BnFoam$.)
Below, we explicitly construct a chain map for this homotopy equivalence.
By construction, our chain map is natural in $X$, 
hence these chain maps assemble into the desired homotopy equivalence 
$\llbracket \bgen_{m}^{\pm 1} C \rrbracket_{2n} 
\xrightarrow{\simeq} 
\qsh^{\mp n(n+1)} \tsh^{\pm n} \llbracket C \rrbracket_{2n}$ 
for any complex $C \in \dgCat[\BnFoam]$.

We now proceed with the argument in detail,
which uses
separate constructions of the chain map in \eqref{objectmarkovII} 
for $\bgen_m$ and $\bgen_m^{-1}$. 
We emphasize that our arguments do not constitute a re-proof of Markov II invariance for
$\sln[2n]$ link homology, as we rely on \cite{Wu} for the existence of the 
homotopy equivalence \eqref{objectmarkovII}.

\begin{prop}
	\label{prop:SLHM2+}
The complex $\SLC{\bgen}$ satisfies the 
positive second Markov move, up to homotopy equivalence and degree shift. 
Explicitly, given $\bgen \in \brgroup$, 
there is a homotopy equivalence
$\qsh^{-n(n+1)} \tsh^n \SLC{\bgen}
\simeq \SLC{\bgen_{m} \bgen}$
of chain complexes of super $\K$-vector spaces.
\end{prop}

\begin{proof}
As discussed above, given any object $X \in \BnFoam$, 
\cite[Lemma 14.8]{Wu} implies that there is a homotopy equivalence
\begin{equation}
	\label{eq:KhRM2w+}
f \colon \llbracket \bgen_{m} X \rrbracket_{2n} 
\xrightarrow{\simeq} 
\qsh^{-n(n+1)} \tsh^{n} \llbracket X \rrbracket_{2n}\, .
\end{equation}
We will describe the homotopy equivalence $f$ explicitly, 
in such a way that it is natural $X$. 

For this, consider
\begin{multline}
	\label{eq:M2+expanded}
\llbracket \bgen_{m} X \rrbracket_{2n}
=
\qsh^{-(m+1) n^2} \Big( 
\Hom_{\cSSStn_q(\gln[m+1])}(\one_{\wtn} ,\FF_m^{(n)} \EE_m^{(n)} X ) \to \cdots \\
	\cdots \to \qsh^{1-n} \tsh^{n-1} \Hom_{\cSSStn_q(\gln[m+1])}(\one_{\wtn} , \FF_m \EE_m X )
		\xrightarrow{\delta} \qsh^{-n} \tsh^{n} 
		\Hom_{\cSSStn_q(\gln[m+1])}(\one_{\wtn} ,X )
\Big) \,
\end{multline}
and (as indicated) let $\delta$ denote the final differential in this complex. 
Given that 
\[
\qsh^{-n(n+1)} \tsh^{n} \llbracket X \rrbracket_{2n} 
= \qsh^{-mn^2-n(n+1)}\tsh^n \Hom_{\cSSStn_q(\glm)}(\one_{\wtn} , X )
\]
is a chain complex of graded $\K$-vector spaces that is non-zero only in homological degree $n$, 
the homotopy equivalence \eqref{eq:KhRM2w+} can be described by finding 
a surjective, degree-zero, $\K$-linear map
\begin{equation}
	\label{eq:M2+map}
f \colon 
\Hom_{\cSSStn_q(\gln[m+1])}(\one_{\wtn} ,X )
\twoheadrightarrow
\Hom_{\cSSStn_q(\glm)}(\one_{\wtn} , X ) 
\end{equation}
such that $\Im(\delta) \subset \ker(f)$.
Indeed, by \eqref{eq:KhRM2w+} we know that
\[
\qsh^{-n} \tsh^{n} \Hom_{\cSSStn_q(\gln[m+1])}(\one_{\wtn} ,X ) {\big /} \Im(\delta) 
\cong \qsh^{-n} \tsh^{n} \Hom_{\cSSStn_q(\glm)}(\one_{\wtn} ,X )
\]
thus such an $f$ will necessarily be a quasi-isomorphism 
of bounded complexes of $\Z$-graded $\K$-vector spaces, 
hence a homotopy equivalence.

To construct $f$, we examine $\Hom_{\cSSStn_q(\gln[m+1])}(\one_{\wtn} ,X )$. 
We know by \cite[Proposition 3.11]{KL3} and \eqref{newbub} that 
$\Hom_{\cSSStn_q(\gln[m+1])}(\one_{\wtn} ,X )$ is spanned by morphisms of the form
\begin{equation}
	\label{eq:Hom1Xbasis}
\begin{tikzpicture}[anchorbase]
	\node at (-.75,.5){\NB{p(\X_{m+1})}};
	\fill[white] (.25,0) rectangle (.75,.5);
	\draw[very thick] (.25,0) rectangle (.75,.5);
	\node at (.5,.25) {$\xi$};
	\draw[very thick] (.5,.5) to (.5,1) node[above]{$X$};
	\node at (1,.75){\scs$\wtn$};
\end{tikzpicture}
\end{equation}
with $p \in \Sym(\X)$ and $\xi \in \Hom_{\cSSStn_q(\glm)}(\one_{\wtn} ,X )$. 
Here, any new bubbles in $\Sym(\X_1| \cdots |\X_m)$ 
have been included as part of $\xi$, 
which is why only the last alphabet $\X_{m+1}$ remains. 
In fact, we can winnow \eqref{eq:Hom1Xbasis} down to a basis for 
$\Hom_{\cSSStn_q(\gln[m+1])}(\one_{\wtn} ,X )$. 
The map
\begin{equation}
	\label{eq:Hom1Xiso}
H^\ast(\Gr{n}{2n}) \otimes_{\K} \Hom_{\cSSStn_q(\glm)}(\one_{\wtn} ,X ) 
\to \Hom_{\cSSStn_q(\gln[m+1])}(\one_{\wtn} ,X )
\end{equation}
sending $p(\X) \otimes \xi$ to the morphism in \eqref{eq:Hom1Xbasis}
is surjective, and Proposition \ref{prop:size} implies that
\begin{align*}
\qdim\big( \Hom_{\cSSStn_q(\gln[m+1])}(\one_{\wtn} ,X ) \big)
&= q^{n^2} {2n \brack n} \qdim \big( \Hom_{\cSSStn_q(\glm)}(\one_{\wtn} ,X ) \big) \\ 
&= \qdim \Big( 
H^\ast(\Gr{n}{2n}) \otimes_{\K} \Hom_{\cSSStn_q(\glm)}(\one_{\wtn} ,X ) 
\Big) \, .
\end{align*}
Thus, \eqref{eq:Hom1Xiso} is an isomorphism,
and we can consider a basis in which the $p(\X)$ range amongst 
Schur polynomials associated to partitions contained in an $n \times n$ box.

Now, we define $f$ by the formula
\begin{equation}
f \left(
\begin{tikzpicture}[anchorbase]
	\node at (-.75,.5){\NB{p(\X_{m+1})}};
	\fill[white] (.25,0) rectangle (.75,.5);
	\draw[very thick] (.25,0) rectangle (.75,.5);
	\node at (.5,.25) {$\xi$};
	\draw[very thick] (.5,.5) to (.5,1) node[above]{$X$};
	\node at (1,.75){\scs$\wtn$};
\end{tikzpicture}
\right)
:=
\begin{tikzpicture}[anchorbase]
	\node at (-.5,.5){\NB{p(\X_{m})}};
	\fill[white] (.25,0) rectangle (.75,.5);
	\draw[very thick] (.25,0) rectangle (.75,.5);
	\node at (.5,.25) {$\xi$};
	\draw[very thick] (.5,.5) to (.5,1) node[above]{$X$};
	\node at (1,.75){\scs$\wtn$};
\end{tikzpicture}
\in \Hom_{\cSSStn_q(\glm)}(\one_{\wtn} ,X ) \, ,
\end{equation}
which identifies the alphabets $\X_{m+1}$ and $\X_m$.
By the above discussion,
this is a well-defined, degree-zero, $\K$-linear map
$\Hom_{\cSSStn_q(\gln[m+1])}(\one_{\wtn} ,X ) \to \Hom_{\cSSStn_q(\glm)}(\one_{\wtn} , X )$
which is surjective and satisfies the desired relation $f \circ \tauast = \tauast \circ f$.

Next, we establish that $\Im(\delta) \subset \ker(f)$. 
For this, we first claim that 
$\Hom_{\cSSStn_q(\gln[m+1])}(\one_{\wtn} , \FF_m \EE_m X )$
is spanned by morphisms of the form
\begin{equation} \label{separatemagenta}
\begin{tikzpicture}[anchorbase]
	\draw[thick,magenta,<-] (0,1) to [out=270,in=0] (-.25,.5) node[black]{$\bullet$} node[black,below]{\scs$r$}
		to [out=180,in=270] (-.5,1);
	\fill[white] (.25,0) rectangle (.75,.5);
	\draw[very thick] (.25,0) rectangle (.75,.5);
	\node at (.5,.25) {$\eta$};
	\draw[very thick] (.5,.5) to (.5,1) node[above]{$X$};
	\node at (1,.75){\scs$\wtn$};
\end{tikzpicture}
\end{equation}
with $\eta \in \Hom_{\cSSStn_q(\gln[m+1])}(\one_{\wtn} , X )$ and $r \geq 0$.
Here, 
we use {\color{magenta} magenta} to depict the final Dynkin label $m$. 
The reason that morphisms of the form \eqref{separatemagenta} 
constitute a spanning set is that the color
{\color{magenta} magenta} does not appear within the object $X$, 
so the basis in \cite[Proposition 3.11]{KL3} 
will separate {\color{magenta} magenta} from the other colors.

The differential $\delta$ takes a morphism as in \eqref{separatemagenta} 
and post-composes with an anti-clockwise {\color{magenta} magenta} cap. 
Thus $\Im(\delta)$ is spanned by elements of the form
\[
\begin{tikzpicture}[anchorbase]
	\draw[magenta,thick,->] (0,.5) arc[start angle=0,end angle=360,radius=.25] 
		node[pos=.75,black]{$\bullet$} node[pos=.75,black,below]{\scs$r$};
	\fill[white] (.25,0) rectangle (.75,.5);
	\draw[very thick] (.25,0) rectangle (.75,.5);
	\node at (.5,.25) {$\eta$};
	\draw[very thick] (.5,.5) to (.5,1) node[above]{$X$};
	\node at (1,.75){\scs$\wtn$};
\end{tikzpicture}
\stackrel{\eqref{realbubdef}}{=}
\begin{tikzpicture}[anchorbase]
	\draw[magenta,thick,->] (0,.5) arc[start angle=0,end angle=360,radius=.25] 
		node[pos=.75,black]{$\bullet$} node[pos=.75,black,below,xshift=-5pt]{\scs$\spadesuit{+}r{+}1$};
	\fill[white] (.25,0) rectangle (.75,.5);
	\draw[very thick] (.25,0) rectangle (.75,.5);
	\node at (.5,.25) {$\eta$};
	\draw[very thick] (.5,.5) to (.5,1) node[above]{$X$};
	\node at (1,.75){\scs$\wtn$};
\end{tikzpicture}
\stackrel{\eqref{newbub}}{=}
(-1)^n
\begin{tikzpicture}[anchorbase]
	\node at (-1.25,.5){\NB{h_{r+1}(\X_{m+1} - \X_m)}};
	\fill[white] (.25,0) rectangle (.75,.5);
	\draw[very thick] (.25,0) rectangle (.75,.5);
	\node at (.5,.25) {$\eta$};
	\draw[very thick] (.5,.5) to (.5,1) node[above]{$X$};
	\node at (1,.75){\scs$\wtn$};
\end{tikzpicture}
\]
with $\eta \in \Hom_{\cSSStn_q(\gln[m+1])}(\one_{\wtn} , X )$ and $r \geq 0$. 
We thus compute
\[
f \left(
\begin{tikzpicture}[anchorbase]
	\node at (-1.25,.5){\NB{h_{r+1}(\X_{m+1} - \X_m)}};
	\fill[white] (.25,0) rectangle (.75,.5);
	\draw[very thick] (.25,0) rectangle (.75,.5);
	\node at (.5,.25) {$\eta$};
	\draw[very thick] (.5,.5) to (.5,1) node[above]{$X$};
	\node at (1,.75){\scs$\wtn$};
\end{tikzpicture}
\right) =
\begin{tikzpicture}[anchorbase]
	\node at (-1.25,.5){\NB{h_{r+1}(\X_{m} - \X_m)}};
\end{tikzpicture}
f \left( \
\begin{tikzpicture}[anchorbase]
	\fill[white] (.25,0) rectangle (.75,.5);
	\draw[very thick] (.25,0) rectangle (.75,.5);
	\node at (.5,.25) {$\eta$};
	\draw[very thick] (.5,.5) to (.5,1) node[above]{$X$};
	\node at (1,.75){\scs$\wtn$};
\end{tikzpicture}
\right) = 0 \, .
\]

Finally, we conclude as sketched above. 
The homotopy equivalence \eqref{eq:KhRM2w+} is natural in $X$, 
so by applying it to every chain object we obtain a 
homotopy equivalence
\[
f \colon \llbracket \bgen_{m} C_j(\bgen) \one_{\wtn} \rrbracket_{2n}
\xrightarrow{{\eqref{eq:KhRM2w+}}}
\qsh^{-n(n+1)} \tsh^{n} \llbracket C_j(\bgen) \one_{\wtn} \rrbracket_{2n}
\]
that satisfies $f \circ \tauast = \tauast \circ f$.
Thus, using Lemma \ref{lem:GactC} and Remark \ref{rem:GactC}
as in the proof of Proposition \ref{prop:SLHM1}, 
we deduce that $\qsh^{-n(n+1)} \tsh^n \SLC{\bgen} \simeq \SLC{\bgen_{m} \bgen}$.
\end{proof}

\begin{prop}
	\label{prop:SLHM2-}
The complex $\SLC{\bgen}$ satisfies the 
negative second Markov move, up to homotopy equivalence and degree shift. 
Explicitly, given $\bgen \in \brgroup$, 
there is a homotopy equivalence
$\qsh^{n(n+1)} \tsh^{-n} \SLC{\bgen}
\simeq \SLC{\bgen_{m}^{-1} \bgen}$
of chain complexes of super $\K$-vector spaces.
\end{prop}

\begin{proof}
The proof again follows that of Propositions \ref{prop:SLHM1} and \ref{prop:SLHM2+}.
Namely, we explicitly describe the homotopy equivalence
\begin{equation}
	\label{eq:KhRM2-map}
f \colon
\qsh^{n(n+1)} \tsh^{-n} \llbracket \bgen \rrbracket_{2n}
\xrightarrow{\simeq} \llbracket \bgen_{m}^{-1} \bgen \rrbracket_{2n} \, ,
\end{equation}
observe that $f \circ \tauast = \tauast \circ f$, 
and then deduce the result from 
Lemma \ref{lem:GactC} and Remark \ref{rem:GactC}. 
Again, the chain map $f$ in \eqref{eq:KhRM2-map} is induced 
from a map defined for each $X \in \BnFoam$ 
in \cite[Lemma 14.8]{Wu}
that is natural in $X$:
\begin{equation}
	\label{eq:KhRM2w-}
f \colon 
\qsh^{n(n+1)} \tsh^{-n} \llbracket X \rrbracket_{2n}
\xrightarrow{\simeq} \llbracket \bgen_{m}^{-1} X \rrbracket_{2n} \, .
\end{equation}

Now, we let $\delta$ denote the first differential in
\begin{multline}
	\label{eq:M2-expanded}
\llbracket \bgen_{m}^{-1} X \rrbracket_{2n}
=
\qsh^{-(m+1) n^2} \Big( 
\qsh^{n} \tsh^{-n} \Hom_{\cSSStn_q(\gln[m+1])}(\one_{\wtn} ,X ) 
	\xrightarrow{\delta} 
\qsh^{n-1} \tsh^{1-n} \Hom_{\cSSStn_q(\gln[m+1])}(\one_{\wtn} , \FF_m \EE_m X )
\to \\
\cdots \to \Hom_{\cSSStn_q(\gln[m+1])}(\one_{\wtn} ,\FF_m^{(n)} \EE_m^{(n)} X )
\Big).
\end{multline}
Since
\[
\qsh^{n(n+1)} \tsh^{-n} \llbracket X \rrbracket_{2n} 
= \qsh^{-mn^2+n(n+1)}\tsh^{-n} \Hom_{\cSSStn_q(\glm)}(\one_{\wtn} , X )
\]
is a chain complex of graded $\K$-vector spaces that is non-zero only in homological degree $-n$, 
the homotopy equivalence \eqref{eq:KhRM2w-} can be described by finding 
an injective, degree-zero, $\K$-linear map
\begin{equation}
	\label{eq:M2-map}
f \colon 
\qsh^{2n^2} \Hom_{\cSSStn_q(\glm)}(\one_{\wtn} , X ) 
\hookrightarrow
\Hom_{\cSSStn_q(\gln[m+1])}(\one_{\wtn} ,X )
\end{equation}
such that $\Im(f) \subset \ker(\delta)$.
Indeed, by \eqref{eq:KhRM2w-} we know that
\[
\ker \Big( \delta \colon \Hom_{\cSSStn_q(\gln[m+1])}(\one_{\wtn} ,X ) 
\to \qsh^{-1} \Hom_{\cSSStn_q(\gln[m+1])}(\one_{\wtn} , \FF_m \EE_m X )
\Big) \cong \qsh^{2n^2} \Hom_{\cSSStn_q(\glm)}(\one_{\wtn} , X ) 
\]
so such a map must be a homotopy equivalence.

Let $n^n$ denote the partition of $n^2$ 
whose Young diagram is an $n \times n$ box. 
Given $\xi \in \Hom_{\cSSStn_q(\glm)}(\one_{\wtn} ,X )$,
define $f$ in \eqref{eq:M2-map} by
\[
f \left( \
\begin{tikzpicture}[anchorbase]
	\fill[white] (.25,0) rectangle (.75,.5);
	\draw[very thick] (.25,0) rectangle (.75,.5);
	\node at (.5,.25) {$\xi$};
	\draw[very thick] (.5,.5) to (.5,1) node[above]{$X$};
	\node at (1,.75){\scs$\wtn$};
\end{tikzpicture}
\right)
:=
\begin{tikzpicture}[anchorbase]
	\node at (-1.25,.5){\NB{\Schur[n^n](\X_m{+}\X_{m+1})}};
	\fill[white] (.25,0) rectangle (.75,.5);
	\draw[very thick] (.25,0) rectangle (.75,.5);
	\node at (.5,.25) {$\xi$};
	\draw[very thick] (.5,.5) to (.5,1) node[above]{$X$};
	\node at (1,.75){\scs$\wtn$};
\end{tikzpicture} \, .
\]
Since the partition $n^n$ is self-transpose, 
we have $f \circ \tauast = \tauast \circ f$ by \eqref{eq:tauonSchur}.

The salient properties of $f$ will follow from the expansion
\begin{equation}
	\label{eq:Schurbox}
\Schur[n^{n}](\X_m{+}\X_{m+1}) = 
\sum_{\parti \subseteq n^n} \Schur[\parti^c](\X_m) \Schur(\X_{m+1}) \, .
\end{equation}
Being the cohomology ring of a compact orientable manifold, 
$H^\ast(\Gr{n}{2n})$ has a $\K$-linear functional $I$
which picks out the top cohomology class (corresponding to integration).
Explicitly, this map $I$ satisfies
\[ 
I(\Schur) = \begin{cases} 1 & \text{if } \parti = n^n \\ 0 & \text{else}. \end{cases} 
\]
Applying $I$ to elements in $H^\ast(\Gr{n}{2n}) \cong \Sym(\X_{m+1}) / \! \sim$
will take an element in 
$H^\ast(\Gr{n}{2n})^{\otimes 2} \cong \Sym(\X_m | \X_{m+1}) / \! \sim$
to an element of $H^\ast(\Gr{n}{2n}) \cong \Sym(\X_m) / \! \sim$. 

Applying this map to (the class of) $\Schur[n^n](\X_m + \X_{m+1})$ will yield $1$, 
by \eqref{eq:Schurbox}. 
Consequently, 
post-composing the isomorphism
\[
\Hom_{\cSSStn_q(\gln[m+1])}(\one_{\wtn} ,X ) 
\stackrel{\eqref{eq:Hom1Xiso}}{\cong}
H^\ast(\Gr{n}{2n}) \otimes_{\K} \Hom_{\cSSStn_q(\glm)}(\one_{\wtn} ,X ) 
\]
with the map 
$I \ot \id \colon H^\ast(\Gr{n}{2n}) \otimes_{\K} \Hom_{\cSSStn_q(\glm)}(\one_{\wtn} ,X ) 
\to \qsh^{2n^2} \Hom_{\cSSStn_q(\glm)}(\one_{\wtn} ,X )$,
one obtains a left inverse to $f$. 
Thus, $f$ is injective, as desired.

The map $\delta$ in \eqref{eq:M2-expanded} is given by 
post-composition with an anti-clockwise {\color{magenta} magenta} cup
(we again use {\color{magenta} magenta} to depict the final Dynkin label $m$).
We compute that
\begin{equation}
	\label{eq:deltaf}
\delta \circ f \left( \
\begin{tikzpicture}[anchorbase]
	\fill[white] (.25,0) rectangle (.75,.5);
	\draw[very thick] (.25,0) rectangle (.75,.5);
	\node at (.5,.25) {$\xi$};
	\draw[very thick] (.5,.5) to (.5,1) node[above]{$X$};
	\node at (1,.75){\scs$\wtn$};
\end{tikzpicture}
\right)
=
\begin{tikzpicture}[anchorbase]
	\draw[thick,magenta,<-] (-.5,1) to [out=270,in=0] (-.75,.5) to [out=180,in=270] (-1,1);
	\node at (-.25,-.75){\NB{\Schur[n^{n}](\X_m{+}\X_{m+1})}};
	\fill[white] (.25,0) rectangle (.75,.5);
	\draw[very thick] (.25,0) rectangle (.75,.5);
	\node at (.5,.25) {$\xi$};
	\draw[very thick] (.5,.5) to (.5,1) node[above]{$X$};
	\node at (1,.75){\scs$\wtn$};
\end{tikzpicture}
\stackrel{\eqref{eq:symslide}}{=}
\begin{tikzpicture}[anchorbase]
	\draw[thick,magenta,<-] (-.5,1) to [out=270,in=90] (0,-.25) to [out=270,in=0] (-1.5,-.75)
		to [out=180,in=270] (-3,-.25) to [out=90,in=270] (-1,1);
	\node at (-1.5,-.25){\NB{\Schur[n^{n}](\X_m{+}\X_{m+1})}};
	\fill[white] (.25,0) rectangle (.75,.5);
	\draw[very thick] (.25,0) rectangle (.75,.5);
	\node at (.5,.25) {$\xi$};
	\draw[very thick] (.5,.5) to (.5,1) node[above]{$X$};
	\node at (-1.25,.25){\scs$\wtn{+}\ee_m$};
	\node at (1,.75){\scs$\wtn$};
\end{tikzpicture} \, .
\end{equation}
In weight $\wtn+\ee_m$, 
we have that $\Schur(\X_m) = 0$ if $\parti$ does not fit inside an $(n-1) \times (n+1)$ box 
and $\Schur[\mu](\X_{m+1}) = 0$ if $\mu$ does not fit inside an $(n+1) \times (n-1)$ box. 
If $\parti$ fits inside both an $(n-1) \times (n+1)$ box and an $n \times n$ box, 
then its complement $\parti^c$ with respect to the $n \times n$ box 
does not fit inside an $(n+1) \times (n-1)$ box.
Thus, \eqref{eq:Schurbox} implies that the last morphism in \eqref{eq:deltaf} is zero, 
so $\Im(f) \subset \ker(\delta)$, as desired.
\end{proof}

Combining Theorem \ref{thm:Ctaubraid} and Propositions \ref{prop:SLHM1}, \ref{prop:SLHM2+}, and \ref{prop:SLHM2-} 
with Theorem \ref{thm:Markov}, we arrive at our main result.

\begin{thm}
	\label{thm:typebLH}
The homology of the complex $\SLC{\bgen}$ of $\Z$-graded super $\K$-vector spaces
is an invariant of the framed link $\mathcal{L}_\bgen$, 
up to isomorphism of $\Z^2$-graded super $\K$-vector spaces.
\end{thm}

\begin{proof}
Let $\epsilon(\bgen)$ be the exponent sum of $\bgen$, and consider the renormalized invariant
\begin{equation}
	\label{eq:renorm}
\qsh^{n(n+1) \epsilon(\bgen)} \tsh^{-n \epsilon(\bgen)} \SLC{\bgen}
\end{equation}
which is invariant under the second Markov move
by Propositions \ref{prop:SLHM2+} and \ref{prop:SLHM2-}.
Since $\epsilon(\bgen)$ is unchanged by the braid relations and the first Markov move, 
Theorem \ref{thm:Ctaubraid} and Propositions \ref{prop:SLHM1} imply that \eqref{eq:renorm} 
is invariant under these moves.
Hence, by Theorem \ref{thm:Markov}, \eqref{eq:renorm} 
is an invariant of the (unframed) link $\mathcal{L}_{\bgen}$, 
up to homotopy equivalence.
Since $\epsilon(\bgen)$ equals the writhe of the link diagram for $\mathcal{L}_{\bgen}$ 
given by taking the braid closure, 
the non-renormalized invariant $\SLC{\bgen}$ 
is an invariant of the framed link $\mathcal{L}_{\bgen}$ 
(endowed with the blackboard framing), up to homotopy equivalence.
It follows that the homology of $\SLC{\bgen}$, 
which is $\Z^2$ graded via $\tsh$- and $\qsh$-degree, 
is an invariant of the framed link $\mathcal{L}_{\bgen}$.
\end{proof}

\begin{defn}
Let $H_{\son}(\mathcal{L}_\bgen^\mathcal{S}) \in \sVect_{\K}^{\Z \times \Z}$ 
denote the homology of the complex $\SLC{\bgen}$. 
\end{defn}

We will refer to these link invariants as \emph{Type $B$ spin link homology}, 
or simply as \emph{spin link homology}, 
in light of our decategorification results given in the following section.

\begin{rem}
As graded super vector spaces form a semisimple category, 
the homology of the complex $\SLC{\bgen}$ is uniquely determined by 
the isomorphism class of $\SLC{\bgen}$ in the homotopy category $\hCat[\sVect^\Z_\K]$. 
One can view spin link homology as an invariant valued in $\hCat[\sVect^\Z_\K]$ 
rather than in the category $\sVect_{\K}^{\Z \times \Z}$ 
of bigraded super vector spaces. \end{rem}

\subsection{Decategorification}
	\label{ss:decat}
	
We conclude this section by establishing the results 
from \S \ref{ss:nutshell} and \S \ref{ss:inmoredepth}
that relate our spin link homology to the spin-colored $\son$ link polynomials 
discussed in Section \ref{s:typeB}.

To begin, recall that if $\Cat$ is a $\K$-linear category
equipped with an involution $\sigma$, 
then there is an action of the group 
$\langle \sgn \rangle \cong \Z/2$ on $\ECat$ defined by $\sgn(X,\phi):=(X,-\phi)$.
Further, the $\sigma$-weighted Grothendieck group $\wKzero{\ECat}$
is defined by imposing the relation $[\sgn(X,\phi)] = -[(X,\phi)]$ on $\Kzero{\ECat}$.

Although the category $\sVect_{\K}^{\Z}$ does not arise in this way, 
it also carries a $\Z/2$ action, via the \emph{parity-reversing involution}
\[
\ssh \colon \sVect_{\K}^{\Z} \to \sVect_{\K}^{\Z}
\]
which e.g.~sends $\K^{m|n} \mapsto \K^{n|m}$. 
We can similarly impose the relation $[\ssh V] = - [V]$ on $\Kzero{\sVect_{\K}^{\Z}}$, 
and in the quotient
\[
K_0^{\ssh}(\sVect^{\Z}_{\K}) := \Kzero{\sVect_{\K}^{\Z}} {\big /} \big( [\ssh V] = - [V] \big)
\]
we find that
$[\K^{0|n}]_{\ssh} = - [\K^{n|0}]_{\ssh}$.
It follows that there is an isomorphism $K_0^{\ssh}(\sVect^{\Z}_{\K}) \cong \Z[q^\pm]$ 
given by $[V]_{\ssh} \mapsto \qdim(V)$.
(Recall that dimension of a super vector space is computed as in \eqref{eq:svsdim},
i.e.~it always means super dimension.)

Since the functor $\qsh^{-mn^2}(\RR^+ \oplus \ssh \RR^-)$ intertwines $\sgn$ and $\ssh$, 
there is a unique map $\wKzero{\EBnFoam} \rightarrow \Z[q^{\pm 1}]$ 
such that the following diagram commutes:
\[
\begin{tikzcd}[column sep=6.0em]
\hCat[\EBnFoam] \ar[r,"\qsh^{-mn^2}(\RR^+ \oplus \ssh \RR^-)"] \ar[d,mapsto,"K_0^\tau"'] 
	& \hCat[\sVect^\Z_\K] \ar[d,mapsto,"K_0^\ssh"] \\
\tauKzero{\EBnFoam} \ar[r] & \Z[q^{\pm 1}]
\end{tikzcd}
\]
This is essentially the framework for our decategorification results, 
as we now explain.

For an additive category $\Cat$, 
one can identify the triangulated Grothendieck group of 
the bounded homotopy category $\hCat[\Cat]$ with the split Grothendieck group of $\Cat$,
the map being given by taking the Euler characteristic 
(alternative sum of classes) of each complex.
See e.g.~\cite{RoseG} for a careful proof.
Hence, given a complex $C = \bigoplus_{i\in \Z}\tsh^i C^i$ in $\hCat[\ECat]$, 
we can consider its class $\sum_{i\in \Z}(-1)^i [C^i] \in \Kzero{\ECat}$ 
which therefore determines a well-defined class 
$[C]_{\sigma}  := \sum_{i\in \Z}(-1)^i [C^i]_{\sigma}$
in the quotient $\wKzero{\ECat}$.
Similarly, given $C = \bigoplus_{i\in \mathbb{Z}} \tsh^i C^i \in \hCat[\sVect_{\K}^{\Z}]$, 
we can consider 
$[C]_{\ssh} = \sum_{i\in \mathbb{Z}}(-1)^i \qdim(C^i) \in \Z[q^\pm] \cong K_0^{\ssh}(\sVect_{\K}^{\Z})$.

In all, we arrive at the following diagram:
\begin{equation}
	\label{eq:K0square}
\begin{tikzcd}[column sep=6.0em]
\hCat[\EBnFoam] \ar[r,"\qsh^{-mn^2}(\RR^+ \oplus \ssh \RR^-)"] \ar[d,mapsto,"K_0^\tau"'] 
	& \hCat[\sVect^\Z_\K] \ar[d,mapsto,"K_0^\ssh"] \\
\tauKzero{\EBnFoam} \ar[r] & \Z[q^{\pm 1}]
\end{tikzcd}
\quad , \quad
\begin{tikzcd}
C^{\tau}(\bgen) \ar[r,mapsto] \ar[d,mapsto] & \SLC{\bgen} \ar[d,mapsto] \\
{[C^{\tau}(\bgen)]_{\tau}} \ar[r,mapsto] & {[\SLC{\bgen}]_{\tau}}
\end{tikzcd} \, .
\end{equation}

\begin{nota}
To emphasize dependence on $\tau$, 
we write $[\SLC{\bgen}]_{\tau}$ instead of $[\SLC{\bgen}]_{\ssh}$ 
both in \eqref{eq:K0square} and for the duration.
Since the Euler characteristic of a complex in an abelian category 
always equals that of its homology, 
we have
\[
{[\SLC{\bgen}]_{\tau}} 
	= \dim_{\qsh,\tsh}\big( H_{\son}(\mathcal{L}_\bgen^\mathcal{S}) \big)\big\rvert_{t=-1}
	= \sum_{i} (-1)^i \dim_\qsh \big( H^{i}_{\son}(\mathcal{L}_\beta^\mathcal{S}) \big) \, .
\]
\end{nota}

\begin{thm}
	\label{thm:decat}
Assume that $n=1,2,3$.
If $\beta \in \Br_m$, then
\begin{equation}
	\label{eq:decat}
\bP_{\son}(\mathcal{L}_\beta^{\mathcal{S}}) 
	= (-1)^{n \epsilon(\bgen)+m\binom{n+1}{2}}
	q^{\frac{1}{2} n \epsilon(\bgen)} 
		\sum_{i} (-1)^i \dim_\qsh \big( H^{i}_{\son}(\mathcal{L}_\beta^\mathcal{S}) \big) \, .
\end{equation}
\end{thm}
\begin{proof}
Let $\mirror \bgen$ denote the mirror braid, 
i.e.~the braid obtained from $\bgen=\bgen_{i_1}^{\epsilon_1} \cdots \bgen_{i_\ell}^{\epsilon_\ell}$ 
by replacing each $\bgen_i^{\pm1}$ with $\bgen_i^{\mp}$.
By \eqref{eq:Pbar},
it suffices to show that 
\[
P_{\son}(\mathcal{L}_\beta^{\mathcal{S}}) =
	(-1)^{n \epsilon(\mirror \bgen)+m\binom{n+1}{2}}
	q^{\frac{1}{2} n \epsilon(\mirror \bgen)} {[\SLC{\mirror \bgen}]_{\tau}} \, .
\]
For this, we will use Theorem \ref{thm:characterization}.

Let $A_m^n = \C(q^{\frac{1}{2}}) \otimes_{\Z[q^\pm]} \tauKzero{\EBnFoam}$ and set
$\xx_i^{(k)} := [(\FE_{i}^{(k)}, (-1)^{{n+1 \choose 2}+{n-k+1 \choose 2}+k} \swcl_{k})]_{\tau}
	\in A_m^n$.
Corollaries \ref{cor:K0(X2)} and \ref{cor:K0(Xm)}
then imply that Condition A1.~of Theorem \ref{thm:characterization} holds.
Further,
Definition \ref{def:eRickard} implies that Condition A4.~holds, 
since $(-q^{\frac{1}{2}})^{n \epsilon(\mirror \bgen)} {[C^{\tau}(\mirror \bgen)]_{\tau}}$ 
for $\beta = \bgen_i$ evaluates to
\[
\bgen_i \mapsto 
(-q^{\frac{1}{2}})^{-n} \sum_{k=0}^{n} q^{k} (-1)^{n} \xx_i^{(n-k)}
	= (-q^{\frac{1}{2}})^{-n} (-q)^n \sum_{\ell=0}^{n} q^{-\ell} \xx_i^{(\ell)}
	= q^{\frac{n}{2}} \sum_{\ell=0}^{n} q^{-\ell} \xx_i^{(\ell)} \, .
\]

Let $T_m^n \colon A_m^n \to \C(q^{\frac{1}{2}})$ be given by 
$(-1)^{m \binom{n+1}{2}}$ times 
the bottom horizontal map in \eqref{eq:K0square}. 
For $(Y,\varphi_Y) \in \EBnFoam$, we thus have
\begin{equation}
	\label{eq:K0trace}
\begin{aligned}
T_m^n\big( [(Y,\varphi_Y)]_{\tau} \big) 
&= (-1)^{m\binom{n+1}{2}} q^{-mn^2} \Big(\qdim\big( 
	\Hom_{\EBnFoam}\big((\one_{\wtn},+),(Y,\varphi_Y)\big) \big) \\
	& \qquad \qquad \qquad \qquad \qquad \qquad - \qdim\big( 
		\Hom_{\EBnFoam}\big((\one_{\wtn},-),(Y,\varphi_Y)\big) \big) \Big) \\
&= (-1)^{m\binom{n+1}{2}} q^{-mn^2} \Big(\qdim\big( 
	\Hom_{\BnFoam}\big(\one_{\wtn},Y\big)^{\tau} \big)
		- \qdim\big( \Hom_{\BnFoam}\big(\one_{\wtn},Y\big)^{\sgn} \big) \Big) \\
&= (-1)^{m\binom{n+1}{2}} q^{-mn^2} \tr(\tauast |_{\Hom_{\BnFoam}(\one_{\wtn},Y)}) \, .
\end{aligned}
\end{equation}
In particular, \eqref{eq:K0trace} implies that $T_m^n$ is trace-like, 
since 
\eqref{eq:KhRM1map} gives an isomorphism
$\Hom_{\BnFoam}(\one_{\wtn},Y_1Y_2) \cong \Hom_{\BnFoam}(\one_{\wtn},Y_2Y_1)$ 
that intertwines the $\tauast$-action.

We now verify Condition A3.~from Theorem \ref{thm:characterization}.
First, Remark \ref{rem:gl1} and Definitions \ref{def:CSQ} and \ref{def:BnFoam} 
imply that ${\EuScript B}_1^{n}$ is the 
additive, $\Z$-graded, $\K$-linear, monoidal category generated 
by the monoidal unit $\one_n$. 
This implies that $A_1^n \cong \C(q^{\frac{1}{2}})$.
Further, 
the isomorphism in \eqref{eq:End1=Gr} gives that
$\End_{{\EuScript B}_1^{n}}(\one_n) \cong H^{\ast}\big(\Gr{n}{2n}\big)$,
with basis consisting of the new bubble generators for Schur functions 
$\Schur(\X_1)$ with $\parti$ contained in an $n \times n$ box.
By \eqref{eq:tauonSchur}, we have that 
\begin{equation}
	\label{eq:unknotcategorified1}
T_1^n\big( [(\one_n,+)]_{\tau} \big) 
	= (-1)^{\binom{n+1}{2}} q^{-n^2} \tr(\tauast |_{\End_{\BnFoam[1]}(\one_{n})}) 
	= (-1)^{\binom{n+1}{2}} 
		q^{-n^2} \sum_{\substack{\parti \subset n^n \\ \parti=\parti^{t}}} q^{2|\parti|}
\end{equation}	
and a fun inductive exercise gives that 
\begin{equation}
	\label{eq:unknotcategorified2}	
\sum_{\substack{\parti \subset n^n \\ \parti= \parti^{t}}} q^{2|\parti|} \\
	=q^{n^2} \prod_{i=1}^n (q^{2i-1} + q^{1-2i}) \, .
\end{equation}
so $T_1^n\big( [(\one_n,+)]_{\tau} \big) 
	= (-1)^{\binom{n+1}{2}} \prod_{i=1}^n (q^{2i-1} + q^{1-2i})$
as desired.

Finally, we confirm Condition A2. 
The requisite linear map $\iota \colon A_{m-1}^n \to A_m^n$ 
is induced from the functor $\EBnFoam[m-1] \to \EBnFoam$, 
so it remains to establish \eqref{eq:TrX}.
For this, \eqref{eq:K0trace} shows that 
we must relate the $\tauast$ actions on 
$\Hom_{\BnFoam[m-1]}(\one_{\wtn},Y)$ and $\Hom_{\BnFoam}(\one_{\wtn},\FE_{m-1}^{(k)}Y)$
for $Y \in \BnFoam[m-1]$ 
(for appropriate equivariant structures on $\one_{\wtn}$, $Y$, and $\FE_{m-1}$).

For this, consider the linear map
\begin{equation}
	\label{eq:Hom(1,XY)}
\Hom_{\BnFoam[2]}(\one_{\wtn} , \FE^{(k)}) 
	\otimes_{\K} \Hom_{\BnFoam[m-1]}(\one_{\wtn},Y)
		\to \Hom_{\BnFoam}(\one_{\wtn},\FE_{m-1}^{(k)}Y)
\end{equation}
given by precomposing horizontal composition with the map
\[
\Hom_{\BnFoam[2]}(\one_{\wtn} , \FE^{(k)}) \hookrightarrow
	\Hom_{\BnFoam}(\one_{\wtn} , \FE_{m-1}^{(k)})
\]
which includes into the last $\gln[2]$-string in $\BnFoam$.
(In lay terms, we color an element in $\Hom_{\BnFoam[2]}(\one_{\wtn} , \FE^{(k)})$ 
the color of the final Dynkin node in $\BnFoam$, 
and place it next to an element of $\Hom_{\BnFoam}(\one_{\wtn},Y)$.)
As in the proof of Proposition \ref{prop:SLHM2+},
\cite[Proposition 3.11]{KL3} and \eqref{newbub} imply that \eqref{eq:Hom(1,XY)} is surjective.
Further, it descends to a map
\begin{equation}
	\label{eq:Hom(1,XY)iso}
\Hom_{\BnFoam[2]}(\one_{\wtn} , \FE^{(k)}) 
	\otimes_{H^\ast(\Gr{n}{2n})} \Hom_{\BnFoam[m-1]}(\one_{\wtn},Y)
		\to \Hom_{\BnFoam}(\one_{\wtn},\FE_{m-1}^{(k)}Y)
\end{equation}
where $H^\ast(\Gr{n}{2n})$ acts by new bubbles 
(in the alphabet $\X_1$ in the first factor and the alphabet $\X_{m-1}$ in the second).
Computing dimensions using Proposition \ref{prop:size}, 
we see that \eqref{eq:Hom(1,XY)iso} is an isomorphism. 
Further, it is straightforward to see that this isomorphism intertwines the 
$\tauast \otimes \tauast$ action on the domain with the $\tauast$ action on the codomain 
(provided we compute it with respect to the equivariant structure on $\FE_{m-1}^{(k)}Y$ 
given as the tensor product of the equivariant structures on $\FE_{m-1}^{(k)}$ and $Y$).

We have thus reduced to verifying that
\begin{multline*}
(-1)^{2\binom{n+1}{2}} q^{-2n^2} 
	\tr(\tauast |_{\Hom_{\BnFoam[2]}(\one_{\wtn}, \FE^{(k)})}) \\
= (-1)^{n(k+1)+{n-k \choose 2}} \prod_{t=1}^{n-k} \frac{``[n+1-t][n+t]"}{``[t]^2"}
	 (-1)^{\binom{n+1}{2}} q^{-n^2} \tr(\tauast |_{\End_{\BnFoam[1]}(\one_{n})}) \\
\stackrel{\eqref{eq:unknotcategorified1},\eqref{eq:unknotcategorified2}}{=} 
	(-1)^{n(k+1)+{n-k \choose 2}} \prod_{t=1}^{n-k} \frac{``[n+1-t][n+t]"}{``[t]^2"}
		(-1)^{\binom{n+1}{2}} \prod_{i=1}^n (q^{2i-1} + q^{1-2i})
\end{multline*}
where here the $\tauast$ action is taken with respect to the equivariant objects
$(\FE^{(k)}, (-1)^{{n+1 \choose 2}+{n-k+1 \choose 2}+k} \swcl_{k})$
and $(\one_{\wtn},+)$.
Equivalently, we must show that
\begin{equation}
	\label{eq:lastconj}
q^{-2n^2} 
	\tr(\tauast |_{\Hom_{\BnFoam[2]}(\one_{\wtn}, \FE^{(k)})})
= (-1)^{-nk} 
	\prod_{t=1}^{n-k} \frac{``[n+1-t][n+t]"}{``[t]^2"} \prod_{i=1}^n (q^{2i-1} + q^{1-2i})
\end{equation}
for the action with respect to the equivariant objects 
$(\FE^{(k)}, +\swcl_{k})$ and $(\one_{\wtn},+)$.

When $k=0$, equation \eqref{eq:lastconj} becomes
\begin{equation}
	\label{eq:k=0lastconj}
q^{-2n^2} \tr(\tauast |_{\End_{\BnFoam[2]}(\one_{\wtn})})
	= \prod_{t=1}^{n} \frac{``[n+1-t][n+t]"}{``[t]^2"} \prod_{i=1}^n (q^{2i-1} + q^{1-2i}) \, .
\end{equation}
By Example \ref{eg:devilscircle}, the right-hand side is equal to 
$\big(\prod_{i=1}^n (q^{2i-1} + q^{1-2i})\big)^2$.
Since $\End_{\BnFoam[2]}(\one_{\wtn})$ has a basis consisting of 
(new bubbles for) products $\Schur(\X_1)\Schur[\mu](\X_2)$
with $\parti,\mu$ contained in an $n \times n$ box,
\eqref{eq:k=0lastconj} holds by \eqref{eq:tauonSchur} and \eqref{eq:unknotcategorified2}.

When $k=n$, equation \eqref{eq:lastconj} becomes
\begin{equation}
	\label{eq:k=nlastconj}
\tr(\tauast |_{\Hom_{\BnFoam[2]}(\one_{\wtn}, \FE^{(n)})})
= (-1)^{-n^2} q^{2n^2} \prod_{i=1}^n (q^{2i-1} + q^{1-2i})
\end{equation}
which we again verify directly.
Equation \eqref{dotstonewbubbles} and Proposition \ref{prop:size} imply that 
$\Hom_{\BnFoam[2]}(\one_{\wtn}, \FE^{(n)})$ has basis 
\[
\Bigg\{ 
\begin{tikzpicture}[anchorbase,scale=1]
\draw[ultra thick,->] (-.25,0) node[above=-2pt]{\scs $n$} to [out=270,in=180] (0,-.5) 
	to [out=0,in=270] (.25,0);
\node at (.5,-.75){\NB{\Schur}};
\node at (.75,-.25){\scs$\wtn$};
\end{tikzpicture}
\Bigg\}_{\parti \subset n^n}
\]
so we compute
\[
\tauast\Bigg( 
\begin{tikzpicture}[anchorbase,scale=1]
\draw[ultra thick,->] (-.25,0) node[above=-2pt]{\scs $n$} to [out=270,in=180] (0,-.5) 
	to [out=0,in=270] (.25,0);
\node at (.5,-.75){\NB{\Schur}};
\node at (.75,-.25){\scs$\wtn$};
\end{tikzpicture}
\Bigg)
\stackrel{\eqref{eq:tauonSchur}}{=}
\begin{tikzpicture}[anchorbase,scale=1]
\draw[ultra thick,<-] (.25,.625) to [out=270,in=90] (-.25,0) to [out=270,in=180] (0,-.5) 
	to [out=0,in=270] (.25,0) to [out=90,in=270] (-.25,.625) node[above=-2pt]{\scs $n$};
\node at (.75,-.25){\NB{\Schur[\parti^t]}};
\node at (.75,.25){\scs$\wtn$};
\end{tikzpicture}
\stackrel{\eqref{eq:thickcurl}}{=}
(-1)^{n^2} 
\begin{tikzpicture}[anchorbase,scale=1]
\draw[ultra thick,->] (-.25,0) node[above=-2pt]{\scs $n$} to [out=270,in=180] (0,-.5) 
	to [out=0,in=270] (.25,0);
\node at (.5,-.75){\NB{\Schur[\parti^t]}};
\node at (.75,-.25){\scs$\wtn$};
\end{tikzpicture} \, .
\]
Since a thickness-$n$ cup in $\gln[2]$ weight $(n,n)$ 
has degree $n^2$, equation \eqref{eq:unknotcategorified2} 
implies \eqref{eq:k=nlastconj}.

The remaining cases of \eqref{eq:TrX} are when 
$k=1$ (when $n=2,3$) and $k=2$ (when $n=3$). 
As for $k=0,n$, 
it is possible to verify \eqref{eq:lastconj} directly in these cases; 
however, the computations are tedious so we use a trick instead.
We necessarily have that $T_2^n(\xx_i^{(k)}) = p_{k,n} T_1^n\big( [(\one_n,+)]_{\tau} \big)$ 
for some Laurent polynomials $p_{k,n}$.
Further, the Markov II invariance established in Propositions \ref{prop:SLHM2+} and \ref{prop:SLHM2-}
implies presently that
\begin{equation}
	\label{eq:M2+T}
(-1)^{n+ {n+1 \choose 2}} q^{-n(n+1)} T_1^n\big( [(\one_n,+)]_{\tau} \big) 
	= (-1)^n \sum_{k=0}^n q^{-k} T_2^n(\xx_i^{(n-k)})
\end{equation}
and
\begin{equation}
	\label{eq:M2-T}
(-1)^{n+ {n+1 \choose 2}} q^{n(n+1)} T_1^n\big( [(\one_n,+)]_{\tau} \big) 
	= (-1)^n \sum_{k=0}^n q^{k} T_2^n(\xx_i^{(n-k)}) \, .
\end{equation}
Our computations for $T_2^n(\xx_i^{(0)})$ and $T_2^n(\xx_i^{(n)})$ pair with these equations 
to determine the outstanding values of $T_2^n(\xx_i^{(k)})$.

In detail, when $n=2$, equations \eqref{eq:M2+T} and \eqref{eq:M2-T} yield
\begin{align*}
-q^{-6} &= 1 + q^{-1} \cdot p_{1,2} - q^{-2}(q+q^{-1})(q^3+q^{-3}) \\
-q^{6} &= 1 + q \cdot p_{1,2} - q^{2}(q+q^{-1})(q^3+q^{-3})
\end{align*}
and it is straightforward to check that
\[
p_{1,2} = q^3+q^{-3} 
	= (-1)^{2(1+1)+ {2-1 \choose 2}} \prod_{t=1}^{2-1} \frac{``[2+1-t][2+t]"}{``[t]^2"}
\]
solves both, as desired.
When $n=3$, equations \eqref{eq:M2+T} and \eqref{eq:M2-T} give the system
\begin{align*}
q^{-12} &= 1 + q^{-1} \cdot p_{2,3} + q^{-2} \ p_{1,3} + q^{-3}(q+q^{-1})(q^3+q^{-3})(q^5+q^{-5}) \\
q^{12} &= 1 + q \cdot p_{2,3} + q^{2} \ p_{1,3} + q^{3}(q+q^{-1})(q^3+q^{-3})(q^5+q^{-5})
\end{align*}
and again it is straightforward linear algebra to verify that
\[
p_{2,3} = -(q^5 + q + q^{-1} +q^{-5})
	= (-1)^{3(2+1)+ {3-2 \choose 2}} \prod_{t=1}^{3-2} \frac{``[3+1-t][3+t]"}{``[t]^2"}
\]
and
\[
p_{1,3} = -(q^3+q^{-3})(q^5+q^{-5})
	= (-1)^{3(1+1)+ {3-1 \choose 2}} \prod_{t=1}^{3-1} \frac{``[3+1-t][3+t]"}{``[t]^2"}
\]
are the unique solutions, as desired.
\end{proof}

As initially stated in Theorem \ref{thm:introtrace}, 
we can reformulate Theorem \ref{thm:decat} in terms of 
the involution $\tau$ acting on $\Lambda^n$-colored $\sln[2n]$ link homology.

\begin{cor}
	\label{cor:decatastrace}
Let $\mathcal{L} \subset S^3$ be a link.
\begin{itemize}
\item For all $n \geq 1$, the $n$-colored $\sln[2n]$ Khovanov--Rozansky homology 
$H_{\sln[2n]}(\mathcal{L}^{\wtn})$ admits an involution $\tau$ that preserves the bidegree. 
The bigraded eigenspaces of $\tau$ are link invariants.
\item If $\bgen \in \Br_m$ and $n=1,2,3$, then
\begin{equation}
	\label{eq:introtrace}
\tr \big(
\begin{tikzcd}
H_{\sln[2n]}(\mathcal{L}_{\bgen}^{\wtn}) \arrow[loop left, looseness=3, "\tau"]
\end{tikzcd}
	\big) 
= (-1)^{n \epsilon(\bgen)+m\binom{n+1}{2}} q^{\frac{1}{2} n \epsilon(\bgen)}
	\bP_{\son}(\mathcal{L}_{\bgen}^S) \, .
\end{equation}
\end{itemize}
\end{cor}

\begin{proof}
The arguments establishing Theorems \ref{thm:Ctaubraid} and \ref{thm:typebLH} 
can be repackaged as saying that the $\tauast$-action 
on the chain groups of $\llbracket \bgen \one_{\wtn} \rrbracket_{2n}$ 
descends to a well-defined involution on $H_{\sln[2n]}(\mathcal{L}^{\wtn})$.
The $\pm 1$-eigenspaces of this involution on 
$H_{\sln[2n]}^{i}(\mathcal{L}^{\wtn})$
are exactly the even/odd parts of the super vector space 
$H^{i}_{\son}(\mathcal{L}_\beta^\mathcal{S})$, 
so the result follows from Theorem \ref{thm:decat}.
\end{proof}

More generally, we propose that the equality stated in Theorem \ref{thm:decat} 
(and therefore in Corollary \ref{cor:decatastrace}) holds for all $n$.

\begin{conj}
	\label{conj:decat}
Given $\bgen \in \Br_m$, 
equation \eqref{eq:decat} additionally holds for $n \geq 4$. 
That is, Type $B$ spin link homology categorifies the spin-colored $\son$ link polynomial 
for all $n \geq 1$.
\end{conj}

The proof of Theorem \ref{thm:decat} would suffice to prove Conjecture \ref{conj:decat} 
once additional technical details are verified.
Specifically, the proof of Theorem \ref{thm:decat} uses Theorem \ref{thm:characterization}, 
which holds unqualified only when $n=1,2,3$. 
For $n \geq 4$, 
Theorem \ref{thm:characterization} is conditional on Conjectures \ref{conj:R3} and \ref{conj:TrX}. 
Additionally, in the proof of Theorem \ref{thm:decat}
we only checked Condition A2.~of Theorem \ref{thm:characterization} when $k=0$, $k=n$, 
and for $1 \leq k \leq n-1$ when $n=1,2,3$. 
(Our ``trick'' used for these latter cases does not work for $n\geq 4$.) 
After establishing Condition A2.,~the rest of the proof of Theorem \ref{thm:decat} 
generalizes word-for-word.

\begin{rem}
	\label{rem:conditional}
Conjecture \ref{conj:decat} is therefore an immediate consequence of 
Conjectures \ref{conj:R3} and \ref{conj:TrX}
and the following conjecture.
\end{rem}

\begin{conj}
	\label{conj:TrX2}
Equation \eqref{eq:lastconj} additionally holds for all $n\geq 4$ and all $1 \leq k \leq n-1$.
\end{conj}

\appendix

%
\section{Proof of Lemma \ref{L:rho-recurrence}}\label{S:combinatorial-identities}
%

The goal of this section is to establish the identity \eqref{E:rho-recurrence}.

\subsection{On the devil's product}\label{SS:devil}

We begin with an informal discussion on the devil's arithmetic, 
and the derivation of some useful identities.
Recall from Definition \ref{def:devil} that the devil's product is defined
for $m \leq n\in \Zge$ by
\begin{equation}
	\label{eq:devil}
``[m][n]" = ``[n][m]" := [n+m- 1] - [n+m - 3] + [n+m - 5] \mp \cdots + (-1)^{m-1} [n-m+1] \, .
\end{equation}
Using the relation $[2][n] = [n+1] + [n-1]$
involving the usual product of quantum integers, 
this implies the relation
\begin{equation} \label{E:2mn}
[2]\cdot ``[m][n]" = [n+m] + (-1)^{m-1}[n-m]
\end{equation}

Before presenting additional helpful identities, 
we warn the reader of two unintuitive aspects of devil's arithmetic 
already present above: lack of symmetry and dependence on parity.
When multiplying two quantum numbers, 
one can expand the product as a sum as follows.
\[ 
[2][n] = [n+1] + [n-1] \, , \quad 
[3][n] = [n+2] + [n] + [n-2] \, , \quad 
[4][n] = [n+3] + [n+1] + [n-1] + [n-3] \, , \quad \text{etc.}
\]
Multiplication is commutative, 
which gives rise to two ways to expand any product, 
e.g.
\begin{equation}
	\label{eq:qmult}
[3][5] = [7] + [5] + [3] \, , \quad [5][3] = [7] + [5] + [3] + [1] + [-1] \, ,
\end{equation}
and $[3][5] = [5][3]$ because $[1] + [-1] = 0$. 
In general, the equality between these two expansions follows from
a cancellation of positive and negative quantum integers.

In contrast, 
the devil's product does not admit two equivalent signed expansions! 
Paralleling the example in \eqref{eq:qmult},
we have
\[
``[3][5]" := [7] - [5] + [3] \, , \quad \text{which does \textbf{not} equal } [7] - [5] + [3] - [1] + [-1] \, . 
\]
When deriving formulas for $''[m][n]"$,
one cannot treat $m$ and $n$ symmetrically,
since \eqref{eq:devil} assumes that $m \leq n$.
Equation \eqref{E:2mn} further illustrates
the asymmetry between $m$ and $n$, 
as well as the dependence on the parity of the smaller number $m$.
An easy consequence of that relation is that the $q=1$ 
specialization of the devil's product is given by
\[
``[m][n]"|_{q=1} = \begin{cases} n & \text{ if $m$ is odd,} \\ 
				m & \text{ if $m$ is even.} \end{cases} 
\]

We now record some alternative formulae for the devil's product.

\begin{nota}
If $k \in \Z_{\ge 0}$, then we write $[2]_k := [2]_{q^k} := q^k+ q^{-k}$. 
\end{nota}

Note that $[2] = [2]_1$ divides $[2]_k$ if and only if $k$ is odd 
and that $[k]$ divides $[m]$ when $k$ divides $m$. 
Also, we record the identity
\begin{equation} \label{E:dividequantum} \frac{[k \cdot \ell]}{[k]} = [\ell]_{q^k} \, . \end{equation}

\begin{lem}
If $m\le n$ in $\Z_{\ge 0}$, then
\[
``[m][n]"=\begin{cases}
[2]_{n}[m/2]_{q^2} & \text{if $m$ is even}, \\
[2]_m [n/2]_{q^2} & \text{if $m$ is odd and $n$ is even}, \\
(\frac{[2]_m}{[2]}) [n] & \text{if $m$ and $n$ are odd}.
\end{cases}
\]
\end{lem}

\begin{proof}
Expanding the quantum integers as $[k] = \frac{q^k - q^{-k}}{q-q^{-1}}$ 
on the right-hand side of \eqref{E:2mn}, we obtain
\[
[2]``[m][n]" = \frac{(q^{n+m} - q^{-n-m}) + (-1)^{m-1} (q^{n-m} - q^{m-n})}{q-q^{-1}} 
	= \frac{(q^n + (-1)^m q^{-n})(q^m + (-1)^{m-1} q^{-m})}{q - q^{-1}} \, .
\]
This is clearly equal to $[2]_n [m]$ or $[2]_m [n]$ depending on the parity of $m$. 
Now divide by $[2]$, 
and possibly use the $k=2$ case of \eqref{E:dividequantum}. \end{proof}


\begin{lem} We have
\begin{equation}\label{E:devils-numbers-2}
``[n]^2" = [n]_{q^2}.
\end{equation}
\end{lem}

\begin{proof} 
An immediate consequence of \eqref{E:2mn} and the $k=2$ case of \eqref{E:dividequantum}.
\end{proof}

\subsection{The proof of Lemma \ref{L:rho-recurrence}} \label{SS:rhoforreal}

Recall that equation \eqref{E:rho-recurrence} defines
\[
\rho_{t}^{(\ell)} := (-1)^{\binom{\ell+2-t}{2}} + q^{-1}\dfrac{``[\ell+1-t][\ell+t]"}{``[t]^2"}\rho_{t+1}^{(\ell)} \, .
\]
recursively  for $1 \leq t \leq \ell$ by declaring that $\rho_{\ell+1}^{(\ell)} := 1$.
We aim to prove that $\rho_{1}^{(\ell)} = (q^{-2})^{\binom{\ell+1}{2}}$. 
To do so, we will unravel the recurrence to obtain a formula for $\rho_1^{(\ell)}$, 
which will depend on the value of $\ell$ modulo $4$.

\begin{example}\label{Ex:expand-rho}
Consider the case when $\ell=4n$, for $n\in \Z_{\ge 0}$. 
We find that
\begin{align*}
\rho_{1}^{(4n)} &= 1 + q^{-1}\frac{[2]_{q^{4n+1}}[2n]_{q^2}}{[1]_{q^2}}\rho_{2}^{(4n)} \\
&= 1 + q^{-1}\frac{[2]_{q^{4n+1}}[2n]_{q^2}}{[1]_{q^2}} 
	-q^{-2}\frac{[2]_{q^{4n+1}}[2]_{q^{4n-1}}[2n]_{q^2}[2n+1]_{q^{2}}}{[1]_{q^2}[2]_{q^2}}\rho_{3}^{(4n)} \\
&= \cdots \\
&= 1-q^{-2}[2]_{q^{4n+1}}[2]_{q^{4n-1}}{2n+1\brack 2}_{q^2} 
	+ q^{-4}[2]_{q^{4n+1}}[2]_{q^{4n-1}}[2]_{q^{4n+3}}[2]_{q^{4n-3}}{2n+2\brack 4}_{q^2} \pm \cdots \\
&\qquad +q^{-1}[2]_{q^{4n+1}}{2n\brack 1}_{q^2} 
	- q^{-3}[2]_{q^{4n+1}}[2]_{q^{4n-1}}[2]_{q^{4n+3}}{2n+1\brack 3}_{q^2} \pm \cdots
\end{align*}
where we have reordered the terms in the final equation. 
The assertion is that this sum simplifies to $(q^{-2})^{\binom{4n+1}{2}}$.
\end{example}

\begin{rem}
Before continuing, we sketch a proof of the $q=1$ specialization of the result, 
when $\ell=8$. This suggests a useful heuristic for the general arguments below.
By example \ref{Ex:expand-rho}, 
the assertion is that
\begin{equation}
	\label{eq:k=8,q=1}
\begin{aligned}
1 &= 2^0 {4 \choose 0} - 2^2 {5 \choose 2} + 2^4 {6 \choose 4} - 2^6 {7 \choose 6} + 2^8 {8 \choose 8} \\
&\quad +2^1 {4 \choose 1} - 2^3 {5 \choose 3} + 2^5 {6 \choose 5} - 2^7 {7 \choose 7} \, .
\end{aligned}
\end{equation}
Here, we have suggestively written $1$ as $2^0 {4 \choose 0}$ for the first term on the right-hand side.

We can organize this sum by highlighting the relevant entries of Pascal's triangle 
with solid (for the first row in \eqref{eq:k=8,q=1}) and dashed (for the second) circles, 
and labeling them by the appropriate coefficient $\pm 2^a$:
\begin{equation}
	\label{eq:PT}
\begin{tikzpicture}[anchorbase,scale=.5]
\node at (0,0) {$1$};
\node at (-1,-1) {$1$}; \node at (1,-1) {$1$};
\node at (-2,-2) {$1$}; \node at (0,-2) {$2$}; \node at (2,-2) {$1$};
\node at (-3,-3) {$1$}; \node at (-1,-3) {$3$}; \node at (1,-3) {$3$}; \node at (3,-3) {$1$}; 
\node at (-4,-4) {$1$}; \node at (-2,-4) {$4$}; \node at (0,-4) {$6$}; \node at (2,-4) {$4$}; \node at (4,-4) {$1$}; 
\node at (-5,-5) {$1$}; \node at (-3,-5) {$5$}; \node at (-1,-5) {$10$}; \node at (1,-5) {$10$}; 
	\node at (3,-5) {$5$}; \node at (5,-5) {$1$};
\node at (-6,-6) {$1$}; \node at (-4,-6) {$6$}; \node at (-2,-6) {$15$}; \node at (0,-6) {$20$}; 
	\node at (2,-6) {$15$}; \node at (4,-6) {$6$}; \node at (6,-6) {$1$};
\node at (-7,-7) {$1$}; \node at (-5,-7) {$7$}; \node at (-3,-7) {$21$}; \node at (-1,-7) {$35$}; 
	\node at (1,-7) {$35$}; \node at (3,-7) {$21$}; \node at (5,-7) {$7$}; \node at (7,-7) {$1$};
\node at (-8,-8) {$1$}; \node at (-6,-8) {$8$}; \node at (-4,-8) {$28$}; \node at (-2,-8) {$56$}; \node at (0,-8) {$70$}; 
	\node at (2,-8) {$56$}; \node at (4,-8) {$28$}; \node at (6,-8) {$8$}; \node at (8,-8) {$1$};
\draw (-4,-4) node[left=3pt,yshift=-3pt]{\tiny${+}2^0$} circle (.4);
\draw[dashed] (-2,-4) node[left=3pt,yshift=-3pt]{\tiny${+}2^1$} circle (.4);
\draw (-1,-5) node[left=3pt,yshift=-3pt]{\tiny${-}2^2$} circle (.4);
\draw[dashed] (1,-5) node[left=3pt,yshift=-3pt]{\tiny${-}2^3$} circle (.4);
\draw (2,-6) node[left=3pt,yshift=-3pt]{\tiny${+}2^4$} circle (.4);
\draw[dashed] (4,-6) node[left=3pt,yshift=-3pt]{\tiny${+}2^5$} circle (.4);
\draw (5,-7) node[left=3pt,yshift=-3pt]{\tiny${-}2^6$} circle (.4);
\draw[dashed] (7,-7) node[left=3pt,yshift=-3pt]{\tiny${-}2^7$} circle (.4);
\draw (8,-8) node[left=3pt,yshift=-3pt]{\tiny${+}2^8$} circle (.4);
\end{tikzpicture}
\end{equation}
Using the defining property of Pascal's triangle, 
we can remove each solid circle labeled by $\pm 2^a$ 
and add this value to the coefficient for both spots diagonally above it 
(and ignoring spots outside of the triangle). 
This leaves the corresponding sum unchanged, and yields the schematic:
\[
\begin{tikzpicture}[anchorbase,scale=.5]
\node at (0,0) {$1$};
\node at (-1,-1) {$1$}; \node at (1,-1) {$1$};
\node at (-2,-2) {$1$}; \node at (0,-2) {$2$}; \node at (2,-2) {$1$};
\node at (-3,-3) {$1$}; \node at (-1,-3) {$3$}; \node at (1,-3) {$3$}; \node at (3,-3) {$1$}; 
\node at (-4,-4) {$1$}; \node at (-2,-4) {$4$}; \node at (0,-4) {$6$}; \node at (2,-4) {$4$}; \node at (4,-4) {$1$}; 
\node at (-5,-5) {$1$}; \node at (-3,-5) {$5$}; \node at (-1,-5) {$10$}; \node at (1,-5) {$10$}; 
	\node at (3,-5) {$5$}; \node at (5,-5) {$1$};
\node at (-6,-6) {$1$}; \node at (-4,-6) {$6$}; \node at (-2,-6) {$15$}; \node at (0,-6) {$20$}; 
	\node at (2,-6) {$15$}; \node at (4,-6) {$6$}; \node at (6,-6) {$1$};
\node at (-7,-7) {$1$}; \node at (-5,-7) {$7$}; \node at (-3,-7) {$21$}; \node at (-1,-7) {$35$}; 
	\node at (1,-7) {$35$}; \node at (3,-7) {$21$}; \node at (5,-7) {$7$}; \node at (7,-7) {$1$};
\node at (-8,-8) {$1$}; \node at (-6,-8) {$8$}; \node at (-4,-8) {$28$}; \node at (-2,-8) {$56$}; \node at (0,-8) {$70$}; 
	\node at (2,-8) {$56$}; \node at (4,-8) {$28$}; \node at (6,-8) {$8$}; \node at (8,-8) {$1$};
\draw (-3,-3) node[left=3pt,yshift=-3pt]{\tiny${+}2^0$} circle (.4);
\draw (-2,-4) node[left=3pt,yshift=-3pt]{\tiny${-}2^1$} circle (.4);
\draw (0,-4) node[left=3pt,yshift=-3pt]{\tiny${-}2^2$} circle (.4);
\draw (1,-5) node[left=3pt,yshift=-3pt]{\tiny${+}2^3$} circle (.4);
\draw (3,-5) node[left=3pt,yshift=-3pt]{\tiny${+}2^4$} circle (.4);
\draw (4,-6) node[left=3pt,yshift=-3pt]{\tiny${-}2^5$} circle (.4);
\draw (6,-6) node[left=3pt,yshift=-3pt]{\tiny${-}2^6$} circle (.4);
\draw (7,-7) node[left=3pt,yshift=-3pt]{\tiny${+}2^7$} circle (.4);
\end{tikzpicture}
\]
The corresponding sum is exactly the $q=1$ specialization of $\rho_1^{(7)}$.
Iterating this procedure, we eventually see that
$\rho_1^{(8)}|_{q=1} = \cdots = \rho_1^{(0)}|_{q=1} = 1$.
\end{rem}

Trying to adapt this argument for generic $q$ is not straightforward. 
There are well-known analogues of Pascal's identity, which we record here:
\begin{equation}\label{E:q-pascal}
{x\brack y}_{q} = q^{y}{x-1 \brack y}_{q} + q^{y-x}{x-1 \brack y-1}_{q}
\end{equation}
and
\begin{equation}\label{E:q-pascal2}
{x\brack y}_{q} = q^{-y}{x-1 \brack y}_{q} + q^{-y+x}{x-1 \brack y-1}_{q}.
\end{equation}
However, 
the cancellation between solid and dashed circles in \eqref{eq:PT}, 
which simply used that $2^{m+1} - 2^{m} = 2^{m}$,
now requires a cancellation of Laurent polynomials which is not guaranteed 
by the corresponding $q=1$ cancellation. 
One must keep careful track of the powers of $q$, in both the quantum binomial coefficients 
(choosing between \eqref{E:q-pascal} and \eqref{E:q-pascal2} as appropriate) 
and the complicated ``powers of $2$'' appearing e.g.~in Example \ref{Ex:expand-rho}.

\begin{rem} 
In general, the scalars $\rho_t^{(\ell)}$ are structural coefficients in the type $B$ web algebra 
which describe how to rewrite the expression for the spin braiding 
$R_{S,S} = q^{\frac{n}{2}}\sum_{\ell \ge 0}q^{-\ell}\xx^{(\ell)}$, in the $\II^{(\ell)}$ basis; 
see the proof of Proposition \ref{P:decat-spin-rickard}. 
At $q=1$, one can use similar techniques to analyze the integers $\rho_t^{(\ell)}$ for $t > 1$. 
However, for generic $q$ and $t > 1$, the desired cancellations do not work! 
At present, we do not have a closed formula for $\rho_t^{(\ell)}$ for $t > 1$, 
and it seems to be a difficult problem. \end{rem}

Returning to the task at hand, 
we now introduce notation to help make the sums expressing $\rho_{1}^{(\ell)}$ more compact.

\begin{defn}
Given $c,d\in \Z_{\ge 0}$, 
write $(2)_c^d$ to denote the product consisting of the first $d$ terms of 
\[
[2]_{c+1}[2]_{c-1}[2]_{c+3}[2]_{c-3} \cdots
\]
and write $\{2\}_c^d$ to denote the product consisting of the first $d$ terms of 
\[
[2]_{c-1}[2]_{c+1}[2]_{c-3}[2]_{c+3} \cdots \, .
\]
For $a,b\in \Z_{\ge 0}$, and $\epsilon\in \{\pm 1\}$, 
we then set
\[
A(a,b,\epsilon) = \sum_{i \ge 0}(-1)^iq^{2i\epsilon}(2)_{a}^{2i}{b+i \brack 2i}_{q^2} \, , \quad 
B(a,b,\epsilon) = \sum_{i \ge 0}(-1)^iq^{2i\epsilon}(2)_{a}^{2i+1}{b+i \brack 2i+1}_{q^2}
\]
and
\[
A'(a,b,\epsilon) = \sum_{i \ge 0}(-1)^iq^{2i\epsilon}\{2\}_{a}^{2i}{b+i \brack 2i}_{q^2} \, , \quad 
B'(a,b,\epsilon) = \sum_{i \ge 0}(-1)^iq^{2i\epsilon}\{2\}_{a}^{2i+1}{b+i \brack 2i+1}_{q^2} \, .
\]
\end{defn}

Observe that, for all $a$ and $\epsilon$, 
\begin{equation}
	\label{eq:special-values-AB}
A(a,0, \epsilon) = 1 = A'(a,0, \epsilon) \quad \text{and} \quad 
B(a, 0 , \epsilon) = 0 = B'(a, 0 , \epsilon) \, .
\end{equation}

\begin{lem}\label{L:expand-rho-AB}
For $n \geq 0$,
\[
\rho_{1}^{(4n)} = A(4n, 2n,-1) + q^{-1} B(4n, 2n, -1) \, ,
\]
\[
\rho_{1}^{(4n+1)} = -A'(4n+2, 2n, -1) + q^{-1}B'(4n+2, 2n+1, -1) \, , 
\]
\[
\rho_{1}^{(4n+2)} = -A(4n+2, 2n+1, -1) - q^{-1}B(4n+2, 2n+1, -1) \, ,
\]
and
\[
\rho_{1}^{(4n+3)} = A'(4n+4, 2n+1, -1) - q^{-1}B'(4n+4, 2n+2, -1) \, . 
\]
\end{lem}
\begin{proof}
Following Example \ref{Ex:expand-rho}, expand $\rho_{1}^{(\ell)}$ as a pair of alternating sums.
\end{proof}

\begin{lem}\label{L:A=ABB}
For $b \geq 1$, we have
\begin{equation}\label{E:AB1}
A(a,b,-1) = A(a,b-1,1) - q^{a-2b-1}B(a,b,-1) - q^{-a-2b+1}B(a,b,1) \, ,
\end{equation}
\begin{equation}\label{E:AB-1}
B(a,b,1) = q^{-2}B(a,b-1,-1)+q^{a+2b-1}A(a,b-1,1) + q^{-a+2b-3}A(a,b-1,-1) \, ,
\end{equation}
\begin{equation}\label{E:A'B'1}
A'(a,b,1) = A'(a,b-1, -1) -q^{a+2b+1}B'(a,b,1)-q^{-a+2b-1}B'(a,b,-1) \, ,
\end{equation}
and
\begin{equation}\label{E:A'B'-1}
B'(a,b,-1) = q^2B'(a,b-1, 1) + q^{a-2b+1}A'(a,b-1, -1) + q^{-a-2b+3}A'(a,b-1, 1).
\end{equation}
\end{lem}
\begin{proof}
We prove the first and third formula 
and leave the second and fourth to the reader.
We have
\begin{align*}
A(a,b,-1) &\overset{\eqref{E:q-pascal}}{=}A(a,b-1, 1) 
	+ \sum_{j \ge 0}(-1)^j q^{-2b}[2]_{a-2j+1}(2)_{a}^{2j-1}{b-1+j\brack 2j-1}_{q^2} \\
&=   A(a,b-1, 1) - \sum_{i\ge 0}(-1)^iq^{-2b}(q^{a-2i-1}+ q^{-a+2i+1})(2)_{a}^{2i+1}{b + i\brack 2i+1}_{q^2} \\
&= A(a,b-1,1) - q^{a-2b-1}B(a,b,-1) - q^{-a-2b+1}B(a,b,+1)
\end{align*}
and
\begin{align*}
A'(a,b,1) &\overset{\eqref{E:q-pascal2}}{=} A'(a,b-1, -1) 
	+ \sum_{j\ge 0} (-1)^jq^{2b}[2]_{a+2j-1}\{2\}_a^{2j-1}{b-1+j\brack 2j-1}_{q^2} \\
&= A'(a,b-1, -1) - \sum_{i\ge 0}(-1)^iq^{2b}(q^{a+2i+1} + q^{-a-2i-1})\{2\}_{a}^{2i+1}{b+i\brack 2i+1}_{q^2} \\
&=A'(a,b-1, -1) - q^{a+2b+1}B'(a,b,1) - q^{-a+2b-1}B'(a,b,-1) \, . \qedhere
\end{align*}
\end{proof}

\begin{lem}\label{L:rho-difference}
Fix $n \geq 0$.
For $0 \leq k \leq 2n$, we have
\begin{equation}\label{E:rho1-4n-AB}
\rho_{1}^{(4n)}= (-q^{-8n-2})^k\left(A(4n,2n-k, -1)+q^{2k-1}B(4n,2n-k, -1)\right)
\end{equation}
and
\begin{equation}\label{E:rho1-4n1-AB}
\rho_{1}^{(4n+1)} = (-q^{-8n-2})^{k}\left(-A'(4n+2, 2n-k, -1) + q^{-2k-1}B'(4n+2, 2n+1-k, -1)\right) \, .
\end{equation}
For $0 \leq k \leq 2n+1$ we have
\begin{equation}\label{E:rho1-4n2-AB}
\rho_{1}^{(4n+2)} =(-q^{-8n-6})^{k}\left(-A(4n+2, 2n+1-k, -1)-q^{2k-1}B(4n+2, 2n+1-k, -1)\right)
\end{equation}
and
\begin{equation}\label{E:rho1-4n3-AB}
\rho_{1}^{(4n+3)} =(-q^{-8n-6})^{k}\left(A'(4n+4, 2n+1-k, -1) - q^{-2k-1}B'(4n+4, 2n+2-k,-1)\right) \, .
\end{equation}
\end{lem}
\begin{proof}
Induction on $k$. 
The $k=0$ base case is Lemma \ref{L:expand-rho-AB}. 
The induction step, in the case of equation \eqref{E:rho1-4n-AB}, 
is illustrated as follows: 
\begin{align*}
\rho_{1}^{(4n)} &= (-q^{-8n-2})^k\left(A(4n,2n-k, -1)+q^{2k-1}B(4n,2n-k, -1)\right) \\
&\!\!\!\!\overset{\eqref{E:AB1}}{=} (-q^{-8n-2})^k\Big(A(4n,2n-k-1, 1) - q^{4n-2(2n-k)-1}B(4n,2n-k, -1) \\
&\qquad - q^{-4n-2(2n-k)+1}B(4n,2n-k, 1)+q^{2k-1}B(4n,2n-k, -1)\Big) \\
&=(-q^{-8n-2})^k\left(A(4n,2n-k-1, 1) - q^{-8n+2k+1}B(4n,2n-k, 1)\right) \\
&\!\!\!\!\overset{\eqref{E:AB-1}}{=}(-q^{-8n-2})^k\Big(A(4n,2n-k-1, 1) - q^{-8n+2k+1}\big(q^{-2}B(4n,2n-k-1, -1) \\
&\qquad+ q^{4n+2(2n-k)-1}A(4n,2n-k-1, 1) + q^{-4n+2(2n-k)-3}A(4n,2n-k-1, -1)\big)\Big) \\
&=(-q^{-8n-2})^k\big(- q^{-8n+2k+1-2}B(4n,2n-k-1, -1) - q^{-8n-2}A(4n,2n-k-1, -1)\big) \\
&=(-q^{-8n-2})^{k+1}\big(A(4n,2n-(k+1), -1) + q^{2(k+1)-1}B(4n,2n-(k+1), -1) \big) \, .
\end{align*}
For equation \eqref{E:rho1-4n2-AB}, an identical argument applies. 
For Equations \ref{E:rho1-4n1-AB} and \ref{E:rho1-4n3-AB}, 
the argument is similar, but it instead employs
equation \eqref{E:A'B'-1} and then equation \eqref{E:A'B'1}.
\end{proof}

At last, we arrive at the main result of this appendix.

\begin{proof}[Proof of Lemma \ref{L:rho-recurrence}]
Recall that we must establish the identity
\[
\rho_1^{(\ell)}= (q^{-2})^{\binom{\ell+1}{2}} \, .
\]
Taking $k=2n$ in \eqref{E:rho1-4n-AB} and applying \eqref{eq:special-values-AB} gives
\[
\rho_1^{(4n)} = (-q^{-8n-2})^{2n} = (q^{-2})^{(4n+1)(2n)} = (q^{-2})^{\binom{4n+1}{2}} \, .
\]
A similar argument using 
the $k=2n+1$ case of equation \eqref{E:rho1-4n2-AB} 
shows that $\rho_1^{(4n+2)} = (q^{-2})^{\binom{4n+3}{2}}$. 

For the remaining cases, we also use that $B'(a,1, \epsilon) = [2]_{a-1}$ for all $a, \epsilon$. 
It then follows from \eqref{eq:special-values-AB} 
and the $k=2n$ case of \eqref{E:rho1-4n1-AB} 
that 
\[
\rho_1^{(4n+1)} = (-q^{-8n-2})^{2n}\left(-1 + q^{-4n-1}[2]_{4n+1}\right) 
	=(q^{-2})^{(4n+1)(2n+1)} = (q^{-2})^{\binom{4n+2}{2}} \, .
\]
Finally, the same argument using the $k=2n+1$ case of \eqref{E:rho1-4n3-AB}
gives $\rho_1^{(4n+3)}= (q^{-2})^{\binom{4n+4}{2}}$. 
\end{proof}

\section{Relation to Wenzl's approach using the $q$-Clifford algebra} 
	\label{S:wenzlapproach}

In \cite{Wenzl-Spin}, Wenzl constructs an endomorphism $C\in \End_{\C(q)}(S\otimes S)$ 
which commutes with the action of $U_q(\son)$, 
but does so using a non-standard coproduct. 
We review Wenzl's construction here
and slightly modify his conventions in order to construct an endomorphism of $S\otimes S$ 
which commutes with the action of $U_q(\son)$ via the coproduct in Definition \ref{quantumgroupconventions}. 
Our main result is that Wenzl's endomorphism agrees with the (diagrammatically defined) endomorphism 
\[
\hh^{(1)} = \begin{tikzpicture}[scale=.4, rotate=90, tinynodes, anchorbase]
	\draw[very thick,gray] (-1,0) node[below,yshift=2pt,xshift=2pt]{$S$} to (0,1);
	\draw[very thick,gray] (1,0) node[above,yshift=-4pt,xshift=2pt]{$S$} to (0,1);
	\draw[very thick,gray] (0,2.5) to (-1,3.5) node[below,yshift=2pt,xshift=-2pt]{$S$};
	\draw[very thick,gray] (0,2.5) to (1,3.5) node[above,yshift=-4pt,xshift=-2pt]{$S$};
	\draw[very thick] (0,1) to node[below,yshift=2pt]{$1$} (0,2.5);
\end{tikzpicture}
\]
from \eqref{eq:H1} in Section \ref{s:typeB}.

\begin{defn}
Let $\Cl_q(2n)$ be $\C(q)$ algebra with generators $\psi_i$, $\psi_i^*$, and $\omega_i^{\pm 1}$, for $i=1, \dots, n$, subject to the relations
\begin{gather*}
\psi_i^2=0=(\psi_i^*)^2 \, , \quad \psi_i\psi_j = -\psi_j\psi_i \, , \quad \psi_i\psi_j^*= -\psi_j^*\psi_i \, , \quad \psi_i^*\psi_j^* = -\psi_j^*\psi_i^* \, , \\
\omega_i\omega_i^{-1} = 1 = \omega_i^{-1}\omega_i \, , \quad \omega_i\omega_j = \omega_j\omega_i \, , \\ 
\omega_i\psi_i\omega_i^{-1} = q^2 \psi_i \, , \quad \omega_i\psi_i^*\omega_i^{-1} = q^{-2}\psi_i^* \, , \quad 
	\omega_i\psi_j\omega_i^{-1} = \psi_j \, , \quad \omega_i\psi_j^*\omega_i^{-1} = \psi_j^*\, , \\
\psi_i\psi_i^* + q^2\psi_i^*\psi_i = \omega_i^{-1} \, , \quad  \psi_i\psi_i^* + q^{-2}\psi_i^*\psi_i = \omega_i \, ,
\end{gather*}
for all $1\le i\ne j\le n$, 
\textbf{and} further quotiented by the relations
\begin{equation}\label{wenzlscliffreln}
\omega_i\psi_i = \psi_i \quad \text{and} \quad \psi_i^*\omega_i = \psi_i^* \quad 1\le i \le n. 
\end{equation}
\end{defn}

\begin{rem}
The algebra $\Cl_q(2n)$ is a version of the $q$-Clifford algebra.
The usual definition of the $q$-Clifford algebra \cite{Hayashi} does not impose the relation \eqref{wenzlscliffreln}. 
However, it is essential in establishing that Wenzl's element $C$ commutes with $U_q(\son)$. 
\end{rem}

Multiplying the relation $\psi_i\psi_i^* + q^{2}\psi_i^*\psi_i = \omega_i^{-1}$ on the left by $\omega_i$, 
and using $\omega_i \psi_i = \psi_i$ and $\omega_i\psi_i^*\omega_i^{-1} = q^{-2}\psi_i^*$, 
we deduce that
\begin{equation}
\psi_i\psi_i^* + \psi_i^*\psi_i = 1 \, , \quad 1\le i \le n \, .
\end{equation}

\begin{defn}
The \emph{volume element} in $\Cl_q(2n)$ is defined as
\[
f:=(\psi_1\psi_1^*- \psi_1^*\psi_1)(\psi_2\psi_2^*- \psi_2^*\psi_2)\dots (\psi_n\psi_n^*- \psi_n^*\psi_n) \, .
\]
\end{defn}

For $1\le i \le n$, it satisfies identities:
\[
\omega_if = f\omega_i \, , \quad \psi_if = -f\psi_i \, , \quad \psi_i^*f=-f\psi_i^* \, .
\]

\begin{lem}\label{lem:quantummapstoclif}
There is a homomorphism of algebras 
$U_q(\son)\rightarrow \Cl_q(2n)$ defined for $1\le i \le n-1$ by 
\[
E_i\mapsto \psi_i\psi_{i+1}^* \, , \quad F_i\mapsto \psi_{i+1}\psi_i^*\, , \quad K_i\mapsto \omega_i\omega_i^{-1}
\]
and
\[
E_n\mapsto \psi_n f \, , \quad F_n\mapsto f\psi_n^* \, , \quad K_n\mapsto q\omega_n \, . 
\]
\end{lem}
\begin{proof}
An elementary generators and relations check; see \cite[Theorem 3.2]{Hayashi}.
\end{proof}

\begin{nota}
Given $I\subset \{1, \dots, n\}$ so that $I = \{i_1, \dots, i_d\}$ with $i_1< \cdots < i_d$, 
we write $\psi_I^*:=\psi_{i_1}^* \cdots \psi_{i_d}^*$. 
Also, let $\epsilon_I:=\prod_{i\in I}(-1)^{n-i+1}$.
\end{nota}

\begin{lem}
Let $\mathbf{I}$ be the left ideal $\Cl_q(2n)\cdot \langle \psi_i \mid 1\le i\le n \rangle$. 
There is an isomorphism of vector spaces
\[
\Cl_q(2n)/ \mathbf{I} \xrightarrow{\cong} S
\]
such that $\epsilon_I\psi_I^*\mapsto x_I$. 
Moreover, the induced action of $U_q(\son)$ on $S$ via the homomorphism in Lemma \ref{lem:quantummapstoclif} 
coincides with the action in Definition \ref{def:spin}. 
\end{lem}
\begin{proof}
See \cite[Sections 2.1 and 4.1]{Hayashi}.
\end{proof}

\begin{defn}
Let $\Omega_k:= \omega_1\omega_2\dots \omega_k$ and define 
the following element in $\Cl_q(2n)\otimes \Cl_q(2n)$:
\[
C:=\frac{\Omega_nf\otimes \Omega_n^{-1}f}{[2]} 
	+ \sum_{k=1}^n (\Omega_{k-1}\otimes \Omega_{k-1}^{-1}) 
		\cdot (\psi_k\otimes \psi_k^* + \psi_k^*\otimes \psi_k) \, .
\]
Since $\Cl_q(2n)^{\otimes 2}$ acts on $S^{\otimes 2}$, 
we obtain an operator in $\End_{\C(q)}(S\otimes S)$, (also) denoted $C$. 
\end{defn}

\begin{lem}
The operator $C$ is in $\End_{U_q(\son)}(S\otimes S)$. 
\end{lem}
\begin{proof}
It suffices to show that the elements
\[
\Delta(E_i)\mapsto\psi_i\psi_{i+1}^*\otimes \omega_i\omega_{i+1}^{-1} +  1\otimes \psi_i\psi_{i+1}^* \, , \quad 
\Delta(F_i) \mapsto \psi_{i+1}\psi_i^*\otimes 1 +  \omega_i^{-1}\omega_{i+1} \otimes \psi_{i+1}\psi_i^*
\]
for $1\le i \le n-1$, as well as the elements
\[
\Delta(E_n)\mapsto \psi_nf\otimes q\omega_n + 1\otimes \psi_nf 
\, , \quad \Delta(F_n) \mapsto  f\psi_n^*\otimes 1 + q^{-1}\omega_n^{-1}\otimes f\psi_n^* \, ,
\]
commute with $C$ in $\Cl_q(2n)\otimes \Cl_q(2n)$.
The calculation to verify this follows the rubric outlined in the proof of \cite[Lemma 3.4]{Wenzl-Spin}.
\end{proof}

\begin{lem}\label{lem:Ctriangles}
\[
C\circ \Ya[i] = (-1)^{n-i}\frac{[2(n-i)+1]}{[2]} \Ya[i]
\]
\end{lem}

\begin{proof}
We proceed similarly to the proof of Lemma \ref{SS1triangle}. 
There is some scalar $\chi$ so that $C\circ \Ya[i]= \chi \Ya[i]$. 
Write $\pi_{\lbrace i+1, \ldots, n\rbrace\ot \emptyset}$ to denote the \emph{linear} operator which projects, 
with respect to the basis $\lbrace x_I\ot x_J\rbrace$, to $x_{\lbrace i+1, \ldots, n\rbrace}\ot x_{\emptyset}$. 
Then, $\chi \cdot q^{\emptyset}\cdot x_{\lbrace i+1, \ldots, n\rbrace}\ot x_{\emptyset} 
	= \pi_{\lbrace i+1, \ldots, n\rbrace \ot \emptyset}\circ C\circ \Ya[i](v_i^+)$. 

Using equation \eqref{highestwti}, it follows from weight considerations that
\begin{align*}
\pi_{\lbrace i+1, \ldots, n\rbrace \ot \emptyset}\circ C\circ \Ya[i](v_i^+) 
&= \sum_{\substack {I\cap J = \emptyset \\ I\cup J  = \lbrace i+1, \ldots, n\rbrace}} 
	q^{J} \pi_{\lbrace i+1, \ldots, n\rbrace \ot \emptyset}\circ C\cdot (x_I\otimes x_J) \\
&=\frac{q^{\emptyset}}{[2]}(\Omega_nf\otimes \Omega^{-1}_nf)\cdot x_{\{i+1, \dots, n\}}\otimes x_{\emptyset} \\
&\quad + \sum_{k=i+1}^{n}q^{\{k\}} (\Omega_{k-1}\otimes \Omega_{k-1}^{-1})\circ 
	(\psi_{k}^*\otimes \psi_{k}) \cdot x_{\{i+1, \dots, n\}\setminus \{k\}}\otimes x_{\{k\}} \\
&=\frac{\epsilon_{\{i+1, \dots, n\}}\epsilon_{\emptyset}(q^{-2})^{n-i}(-1)^{n-i}}{[2]} \psi_{\{i+1, \dots, n\}}^*\otimes 1\\
&\quad + \sum_{k=i+1}^nq^{\{k\}}\epsilon_{\{i+1,\dots, n\}\setminus \{k\}}\epsilon_{\{k\}}
	(q^{-2})^{k-i-1}(-1)^{k-i-1} \psi_{\{i+1, \dots, n\}}^*\otimes 1 \\
&=\frac{(q^{-2})^{n-i}(-1)^{n-i}}{[2]}x_{\{i+1, \dots, n\}}\otimes x_{\emptyset} \\
&\quad + \sum_{k=i+1}^n(-q)(-q^2)^{n-k}(q^{-2})^{k-i-1}(-1)^{k-i-1}x_{\{i+1, \dots, n\}}\otimes x_{\emptyset} \\
&=\frac{(-1)^{n-i}}{[2]}\Big(q^{-2(n-i)} + \sum_{k=i+1}^n(q+q^{-1})q^{2n-4k+2i+3}\Big) 
	x_{\{i+1, \dots, n\}}\otimes x_{\emptyset} \\
&=(-1)^{(n-i)}\frac{[2(n-i)+1]}{[2]}x_{\{i+1, \dots, n\}}\otimes x_{\emptyset} \, .
\end{align*}
Thus, $\chi = (-1)^{(n-i)}\frac{[2(n-i)+1]}{[2]}$.
\end{proof}

We now make precise the connection between Wenzl's $C$ 
and the diagrammatically defined endomorphisms from equation \eqref{eq:Hi}.

\begin{prop}\label{P:CequalsH1}
In $\End_{U_q(\son)}(S\otimes S)$, we have $C= \hh^{(1)}$.
\end{prop}
\begin{proof}
Note that $C\cdot x_{\emptyset}\otimes x_{\emptyset} = \frac{1}{[2]}x_{\emptyset}\otimes x_{\emptyset}$. 
Thanks to Lemma \ref{lem:Ctriangles}, 
the proof of Lemma \ref{lem:H1reln} can be applied to $C$ in place of $\hh^{(1)}$, 
in which case we deduce that
\[
C= \dfrac{1}{[2]}\id_{S\ot S} + \sum_{k=1}^{n} (-1)^{\binom{k}{2}}\dfrac{``[k][k+1]"}{d_k} \II^{(n-k)} = \hh^{(1)} \, . 
\]
\end{proof}

\section{More on equivariant categories} \label{SS:equivariantsupplement}

In this appendix we discuss the basic structure of equivariant $\K$-linear categories, 
with minimal assumptions on the commutative base ring $\K$. 
We assume throughout that 
$\K$ is an integral domain and $2$ is invertible in $\K$.

We have two goals. 
The first goal is to justify that most of the results of \S\ref{sec:equiv}, 
including all the results we used when defining our categorical link invariant,
still hold over $\K$. 
The second goal is to shed some light on the difficulties involved in computing 
weighted Grothendieck groups, as they pertain to our conjectures, 
e.g.~the folded skew Howe duality of Conjecture \ref{conj:FSH}.
Even if these conjectures are proven for $\K = \C$, 
there are additional issues one would need to overcome to prove the result for other base rings $\K$, 
and we wish to point these out.

\subsection{When things are nice}\label{ss:nice}

\newcommand{\dd}{D}

We assume we are in the situation of Notation \ref{nota:mixedsigma} henceforth. 
Fix $b\in \BC^{\mathrm{fix}}$.
The main technical issue is that the distinguished object $X_b$ need not be equivariantizable. 

Since $X_b$ is assumed endopositive, 
there is a one-dimensional (by abuse of notation, this means ``free of rank $1$ over $\K$'') 
space of degree zero maps $X_b\rightarrow \sigma(X_b)$. 
Suppose this space is spanned by 
$\varphi_b \colon X_b\rightarrow \sigma(X_b)$, 
which is necessarily invertible, 
as otherwise $X_b$ and $\sigma(X_b)$ would not be isomorphic. 
Since $\sigma$ is an involution (hence an automorphism), 
$\sigma(\varphi_b)$ is also invertible, 
with inverse $\sigma(\varphi_b)^{-1} = \sigma(\varphi_b^{-1})$. 
Since $\End(X_b) = \K \cdot \id$, 
we must have that 
\begin{equation}
	\label{eq:sig(f)f}
\sigma(\varphi_b)\circ \varphi_b = \dd \cdot\id_{X_b}
\end{equation}
for $\dd \in \K^{\times}$.
Note that this implies that $\sigma(\varphi_b^{-1}) = \dd^{-1} \varphi_b$.

Now, in order to equip $X_b$ with an equivariant structure, 
we need $c\in \K$ such that $\sigma(c\cdot \varphi)\circ (c \cdot \varphi) = \id_{X_b}$. 
By \eqref{eq:sig(f)f}, it follows that an equivariant structure exists 
if and only if there exists $c\in \K$ such that $c^2 = \dd^{-1}$.
Equivalently, we see that $X_b$ is equivariantizable if and only if $\dd$ is a square. 
This is not always the case, 
although it is guaranteed if we assume $\K$ is an algebraically closed field
(hence, the comment after Hypothesis \ref{hypo:assumeme}).

\begin{nota}
Let $\EBC^{\mathrm{fix}}$ denote the subset of $b\in \BC^{\mathrm{fix}}$ 
such that there exists $\varphi_b \colon X_b\rightarrow \sigma(X_b)$ so $(X_b, \varphi_b)\in \ECat$.
\end{nota}

Clearly, if $\K$ has all square roots, then $\EBC^{\mathrm{fix}} = \BC^{\mathrm{fix}}$. 
This equality might still hold over general $\K$, 
depending on the structure of the category in question, but it is not guaranteed. 
In the next section, we will study additional equivariant objects associated to 
$b \in \BC^{\mathrm{fix}} \smallsetminus \EBC^{\mathrm{fix}}$ in the event that this set is non-empty.
Note that for $b\in \EBC^{\mathrm{fix}}$, 
Corollary \ref{cor:HomAsInv} implies that the equivariant object $(X_b,\varphi_b)$ is 
endopositive, as it identifies its endomorphism algebra with a subalgebra of $\End(X_b)$.

We now consider a general construction: the induction functor $\Cat \to \ECat$. 
Given any object $Y \in \Cat$, define
\begin{equation} \label{eq:psiY} 
\psi_Y \colon Y \oplus \sigma(Y) \to \sigma(Y) \oplus Y \, , \quad 
\psi_Y := \begin{pmatrix} 0 & \id \\ \id & 0 \end{pmatrix} \, . \end{equation}
It is straightforward to check that
$\ind Y := (Y \oplus \sigma(Y),\psi_Y)$ is then an equivariant object.
It is straightforward to extend this assignment to morphisms, 
thus we obtain a functor 
\begin{equation}
	\label{eq:ind}
\ind \colon \Cat \to \ECat \, , \quad Y \mapsto (Y \oplus \sigma(Y),\psi_Y) \, .
\end{equation}

An equivalent construction, 
tailored to the setting of Notation \ref{nota:mixedsigma}, is as follows. 
Let $b \in \BC$ (whether fixed or not) and choose an isomorphism 
$\varphi_b \colon X_{\sigma(b)} \to \sigma(X_b)$. 
We then can consider the following morphism in $\Cat$: 
\begin{equation} \label{eq:psib} 
\psi_b \colon X_b \oplus X_{\sigma(b)} \to \sigma(X_b) \oplus \sigma(X_{\sigma(b)}) \, , \quad 
\psi_b := \begin{pmatrix} 0 & \varphi_b \\ \sigma(\varphi_b^{-1}) & 0 \end{pmatrix} \, . \end{equation}
It is straightforward to verify that $(X_b\oplus X_{\sigma(b)}, \psi_b)\in \ECat$, 
and that $\varphi_b$ determines an isomorphism between this object and $\ind X_b$.

Next, we compute the space of endomorphisms of $(X_b\oplus X_{\sigma(b)}, \psi_b)$
in the case that $b \in \BC^{\mathrm{free}}$. 
Since there are no non-zero degree zero maps between $X_b$ and $X_{\sigma(b)}$, 
so any degree zero endomorphism of $X_b \oplus X_{\sigma(b)}$ is given by a diagonal matrix
\[ 
\delta = \begin{pmatrix} c_1\id_{X_b} & 0 \\ 0 & c_2\id_{\sigma(X_b)} \end{pmatrix}
\, , \quad c_1,c_2 \in \K \, .
\]
One checks that $\psi_b\circ \delta = \sigma(\delta)\circ \psi_b$ if and only if $c_1=c_2$, 
thus $\End_{\ECat}^0((X_b\oplus X_{\sigma(b)}, \psi_b)) = \K \cdot \id$. 
Since $X_b$ and $X_{\sigma(b)}$ are objects in an endopositive family,  
it follows that there are no endomorphisms of $(X_b\oplus X_{\sigma(b)}, \psi_b)$ of negative degree, 
hence this object is endopositive.
We have thus described the objects appearing in Lemma \ref{lem:phib}, 
and it is easy to verify the remaining claims appearing therein.

Finally, we study multiplicity spaces and decompositions for the objects, 
thus proving Proposition \ref{prop:SummandViaEV} 
under the assumption of Hypothesis \ref{hypo:assumeme}.
As a byproduct, this also establishes Proposition \ref{prop:indECat}.

Fix an equivariant object $(Y,\varphi_Y) \in \ECat$.
Suppose $b \in \EBC^{\mathrm{fix}}$, 
and equip $X_b$ with an equivariant structure $\varphi_b$. 
We hence can consider the $\sigbast$-action on $\Hom^k(X_b,Y)$, 
which we may diagonalize, since $2$ is assumed invertible. 

Suppose we have an inclusion map $\iota \in \Hom^k(X_b,Y)$ lying in the $+1$-eigenspace, 
which pairs against a projection map $p \colon Y \to X_b$ 
to give the identity map of $X_b$. 
By symmetrizing, we can also assume that $p$ lies in the $+1$-eigenspace for the 
$\sigbast$ action on $\Hom^{-k}(Y,X_b)$, 
so $\iota$ and $p$ induce inclusion and projection maps 
between $(X_b,\varphi_b)$ and $(Y,\varphi_Y)$. 
Similarly, inclusion/projection maps in the $-1$-eigenspace induce inclusion and projection maps 
between $(X_b, -\varphi_b)$ and $(Y,\varphi_Y)$. 
Moreover, this argument continues to work for families of orthogonal inclusions and projections. 
Finally, note that $\sigbast$ gives a semisimple action of $\Z/2$ on both $\Hom^k(X_b,Y)$ 
and the kernel of the graded composition pairing, 
so that an eigenbasis of $V^k(X_b,Y)$ can be lifted to eigenvectors in $\Hom^k(X_b,Y)$.

A separate argument holds for $b \in \BC^{\mathrm{free}}$, 
which does not require the $\sigast$ action
(but does involve similar formulae). 
Fix an isomorphism $\varphi_b \colon X_{\sigma(b)} \to \sigma(X_b)$
as above and choose arbitrary maps 
$\iota \in \Hom^{k}(X_b,Y)$ and $p \in \Hom^{-k}(Y,X_b)$. 
We can consider
the maps in $\ECat$
between $(X_b \oplus X_{\sigma(b)},\psi_b)$ and $(Y,\varphi_Y)$
given by the matrices
\begin{equation} \label{eq:inducedmaps} 
I(\iota) := \begin{pmatrix} \iota & \varphi_Y^{-1} \circ \sigma(\iota) \circ \varphi_b \end{pmatrix}
\quad \text{and} \quad
P(p) := \begin{pmatrix} p \\ \varphi_b^{-1} \circ \sigma(p) \circ \varphi_Y \end{pmatrix} \, .
\end{equation}
We leave the reader to confirm that these are indeed equivariant maps. 
Composing, we obtain 
\begin{equation}\label{eq:compositionofinducedmapsatfirst} 
P(p) \circ I(\iota) = \begin{pmatrix} p \circ \iota & p \circ \varphi_Y^{-1} \circ \sigma(\iota) \circ \varphi_b \\ 
	\varphi_b^{-1} \circ \sigma(p) \circ \varphi_Y \circ \iota 
		& \varphi_b^{-1} \circ \sigma(p \circ \iota) \circ \varphi_b \end{pmatrix} \, ,
\end{equation}
which is an endomorphism of $(X_b \oplus X_{\sigma(b)},\psi_b)$.

As observed above, since $b \in \BC^{\mathrm{free}}$, 
this matrix must be a multiple of the identity, 
so, in particular, the off-diagonal entries are zero.
In fact, we can explicitly compute which multiple.
We have $p\circ \iota = \beta^k_{X_b, Y}(p,\iota) \cdot \id_{X_b}$,
and so 
\[
\varphi_b^{-1} \circ \sigma(p \circ \iota) \circ \varphi_b 
	= \beta^k_{X_b, Y} (p, \iota) \cdot \id_{X_{\sigma(b)}}
	\]
is also the same multiple of the identity. 
Thus, the composition pairing of $I(\iota)$ and $P(p)$ 
agrees with the composition pairing of $\iota$ and $p$. 
Consequently, the multiplicity of $X_b$ in $Y$ exactly equals 
the multiplicity of $(X_b \oplus X_{\sigma(b)},\psi_b)$ in $(Y,\varphi_Y)$.

\begin{rem} 
We can say more under a common situation in representation theory, 
that of the \emph{top summand}. 

Let $I \subset \BC$.
Given an object $Z \in \CatBC$, the condition that $[Z] \in \Z[q^{\pm}] \cdot I$ 
is equivalent to the condition that only (shifts of) objects $\{ X_b\}_{b \in I}$ 
appear in the unique decomposition of $Z$ into distinguished summands. 
The associated \emph{thick ideal} $\Ide\subset \CatBC$ 
consists of the morphisms in $\CatBC$ 
that factor through an object $Z$ such that 
$[Z] \in \Z[q^{\pm}] \cdot I$.
If $I$ is closed under the action of $\sigma$ on $\BC$, 
then $\Ide$ is invariant under $\sigma$. 

Now, let $(Y,\varphi_Y)$ be an equivariant object in $\ECatBC$
and suppose that $Y = X \oplus Z$ where $X=X_c$ for $c \in \BC^{\mathrm{fix}} \smallsetminus I$ 
and $[Z] \in \Z[q^{\pm}] \cdot I$.
We think of $X$ as the top summand of $Y$, and $Z$ as consisting of ``lower terms.''
In this context, let us prove that $X$ is equivariantizable. 

Consider $X$ as the image of an idempotent $e \in \End(Y)$,
and note that $\varphi_X := \sigma(e) \circ \varphi_Y \circ e$ is a morphism from $X$ to $\sigma(X)$.
We compute
\begin{align*}
\sigma(\varphi_X) \circ \varphi_X 
&= e \circ \sigma(\varphi_Y) \circ \sigma(e) \circ \varphi_Y \circ e \\
&= e \circ \sigma(\varphi_Y) \circ \sigma(\id_Y + (e - \id_Y)) \circ \varphi_Y \circ e \\
&= e \circ \sigma(\varphi_Y) \circ \varphi_Y \circ e 
	+ e \circ \sigma(\varphi_Y) \circ \sigma(e - \id_Y) \circ \varphi_Y \circ e \, ,
\end{align*}
so $\sigma(\varphi_X) \circ \varphi_X \in e + e \Ide e$.
By endopositivity, the degree zero morphisms in $e \Ide e$ are zero,
thus $\sigma(\varphi_X) \circ \varphi_X = e$.
Since $e$ is the identity map of $X$, 
$\varphi_X$ is an equivariant structure.

This technique can be used to prove that many distinguished objects are equivariantizable.
For example, this provides an alternative route to showing that
all endopositive objects in $\BnFoam[2]$ are equivariantizable, 
via \cite[Example 2.21]{ELau}.
\end{rem}

\subsection{When things are not nice}\label{ss:notnice}

We continue to assume that we are in the setting of Notation \ref{nota:mixedsigma}.
In this section we discuss indecomposable objects associated 
with $b \in \BC^{\mathrm{fix}} \smallsetminus \EBC^{\mathrm{fix}}$. 
We do not know where to find an exploration of this theme in the literature.

Fix $b \in \BC^{\mathrm{fix}}$, possibly in $\EBC^{\mathrm{fix}}$, 
Once and for all, fix an isomorphism $\varphi_b \colon X_b \to \sigma(X_b)$
spanning $\Hom(X_b,\sigma(X_b))$
and define $\dd \in \K^{\times}$ by the formula $\sigma(\varphi_b) \circ \varphi_b = \dd \cdot \id$. 
Let $\indX_b := (X_b \oplus X_b, \psi_b)$ 
be the equivariant object defined in \eqref{eq:psib}. 
Note that
\[ 
\psi_b = \begin{pmatrix} 0 & \varphi_b \\ \sigma(\varphi_b^{-1}) & 0 
\end{pmatrix}
= \begin{pmatrix} 0 & \varphi_b \\ \dd^{-1} \varphi_b & 0 
\end{pmatrix} \, .
\]

To begin, we compute the degree zero endomorphisms of $\indX_b$.
Since $X_b$ is endopositive, 
a degree zero endomorphism of $X_b \oplus X_b$ in $\Cat$ consists of 
a $2 \times 2$ matrix of scalars multiples of $\id_{X_b}$. 
A computation shows that such
a matrix $M$ satisfies $\psi_b\circ M = \sigma(M)\circ \psi_b$
if and only if
\[
M = \begin{pmatrix}
	\alpha & \dd \cdot \beta \\ \beta & \alpha 
		\end{pmatrix}
	\cdot \id_{X_b}
\]
for some $\alpha, \beta \in \K$. 
Thus, 
$\End_{\ECat}^0(\indX_b)$ is spanned over $\K$ 
by the identity and by the map
\[
\gamma := \begin{pmatrix}
	0 & \dd \\ 1 & 0
	\end{pmatrix} \cdot \id_{X_b} \, .
\]
Since $\gamma^2 = \dd \cdot \id$, 
the ring $\End_{\ECat}^0(\indX_b)$ is isomorphic to 
$\K' := \K[x]/(x^2 - \dd)$, where $x$ acts by $\gamma$. 
Since $X_b$ is endopositive in $\Cat$, 
there are no endomorphisms of $\indX_b$ of negative degree.

If $\dd$ has a square root in $\K$,
or equivalently (as discussed in \S \ref{ss:nice} when $b \in \EBC^{\mathrm{fix}}$, 
then it is easy to decompose $\indX_b$ as a direct sum of two indecomposable equivariant objects.
Explicitly, we have
\[ 
\indX_b \cong (X_b, c \varphi_b) \oplus (X_b, -c \varphi_b)
	\]
where $\pm c$ are the two square roots of $\dd^{-1}$.
Correspondingly, in this case
$\K' = \K[x]/(x^2 - c^{-2}) \cong \K \times \K$.

We focus now on the case when $\dd$ does not have a square root in $\K$.
In this case 
$\End_{\ECat}^0(\indX_b)\cong \K'$ is a domain and consequently 
the object $\indX_b$ is indecomposable.

\begin{rem} 
An immediate consequence of this computation is that $\ECat$ 
will \textbf{not} be positively graded over $\K$ 
(in the sense of Definition \ref{def:mixed})
if $\EBC^{\mathrm{fix}} \ne \BC^{\mathrm{fix}}$. 
Indeed, it has an indecomposable object 
whose endomorphism ring in degree zero is not $\K \cdot \id$, 
but rather some degree two extension of $\K$.
Thus, \eqref{homdim} fails. 
It is possible to generalize Definition \ref{def:mixed}
to allows for field extensions, but we will not pursue this matter here. \end{rem}

Let $(Y,\varphi_Y)$ be an equivariant object.
We aim to analyze summands of $(Y,\varphi_Y)$ of the form $\indX_b$
analogously to our treatment in \S \ref{ss:nice} 
of the summands associated to $b \in \BC^{\mathrm{free}}$.
Fix $b \in \BC^{\mathrm{fix}} \smallsetminus \EBC^{\mathrm{fix}}$ and
define operators $\sigast'$ on $\iota \in \Hom(X_b,Y)$ and $p \in \Hom(Y,X_b)$ 
as follows
\begin{equation}
	\label{eq:sigast'}
\sigast' \iota := \varphi_Y^{-1} \circ \sigma(\iota) \circ \varphi_b \, , \quad 
\sigast' p := \varphi_b^{-1} \circ \sigma(p) \circ \varphi_Y \, . \end{equation}
Since $\sigma(\varphi_b)\circ \varphi_b = \dd \cdot \id_{X_b}$ and $\dd\ne 1$, 
the pair $(X_b,\varphi_b)$ is not an equivariant object, 
so, contrasting \eqref{eq:sigastdef}, 
these operators are not involutions! 
Instead, a computation analogous to the one in the proof of Proposition \ref{prop:sigast}
shows that
\begin{equation}
	\label{eq:sigast'2}
\sigast'(\sigast' \iota) = \dd \cdot \iota
\, , \quad \sigast'(\sigast' p) = \dd^{-1} \cdot p \, .
\end{equation}
Hence, we can endow $\Hom(X_b,Y)$ and $\Hom(Y,X_b)$ 
with $\K'$-module structures by letting $x \in \K'$ act 
by $\sigast'$ on $\Hom(X_b,Y)$ and $\dd \cdot \sigast'$ on $\Hom(Y,X_b)$. 

Observe, however, that 
$\End(X_b)$ is only a $\K$-module 
and the graded composition pairing is only $\K$-linear. 
Nevertheless, one has the following adjunction
\begin{equation} \label{wackyadjunction} 
\beta^k_{X_b, Y}(p,\sigast' \iota) = \dd \cdot \beta^k_{X_b, Y}(\sigast' p, \iota) 
\end{equation}
To see why, 
note that since $\sigma(\varphi_b) = D\cdot \varphi_b^{-1}$ and $\sigma(\varphi_Y^{-1}) = \varphi_Y$, 
we have that
\[
\sigma\left(\beta^k_{X_b, Y}(p,\sigast' \iota)\right) 
= \sigma(p \circ \varphi_Y^{-1} \circ \sigma(\iota) \circ \varphi_b )
= \sigma(p) \circ \sigma(\varphi_Y^{-1}) \circ \iota \circ \sigma(\varphi_b) 
	\]
is conjugate to
\[
\dd \cdot \beta^k_{X_b, Y}(\sigast' p, \iota)
= \dd \cdot \varphi_b^{-1} \circ \sigma(p) \circ \varphi_Y \circ \iota
= \sigma(\varphi_b) \circ \sigma(p) \circ \sigma(\varphi_Y^{-1}) \circ \iota \, .
\]
Since both the left- and right-hand sides of \eqref{wackyadjunction} 
are $\K$-multiples of the identity map, 
$\K$-linearity of $\sigma$ establishes \eqref{wackyadjunction}.

Now consider the $\K$-modules 
$\Hom_{\ECat}^k(\indX_b,(Y,\varphi_Y))$ and $\Hom_{\ECat}^{-k}((Y,\varphi_Y), \indX_b)$
which become $\K'$-modules under pre- and post-composition with $\End_{\ECat}^0(\indX_b)$. 
Adapting \eqref{eq:inducedmaps} to the present context, 
define $\K$-linear maps
\begin{equation}
	\label{eq:inducedmaps2}
I \colon \Hom^k(X_b,Y) \to \Hom_{\ECat}^k(\indX_b,(Y,\varphi_Y))
\, , \quad
P \colon \Hom^{-k}(Y,X_b) \to \Hom_{\ECat}^{-k}((Y,\varphi_Y), \indX_b)
\end{equation}
by
\[
I(\iota) := \begin{pmatrix} \iota & \sigast' \iota \end{pmatrix}
\quad \text{and} \quad
P(p) := \begin{pmatrix} p \\ \sigast' p \end{pmatrix} \, .
\]
Using \eqref{eq:sigast'2},
we compute
\begin{equation}
	\label{eq:K'lin}
I(\iota) \circ \gamma = I(\sigast' \iota) \, , \quad
\gamma \circ P(p) = P(\dd \cdot \sigast' p), 
\end{equation}
so the maps in \eqref{eq:inducedmaps2} are in fact $\K'$-module homomorphisms. 
The computation \eqref{eq:compositionofinducedmapsatfirst} is unchanged, 
but now the off-diagonal entries need not be zero.
Instead, \eqref{wackyadjunction} implies that
\begin{equation}\label{E:whackycomppair}
P(p) \circ I(\iota) 
	= \beta^k_{X_b, Y}(p,\iota) \cdot \id_{\indX_b} 
		+ \dd^{-1} \beta^k_{X_b, Y}(p,\sigast' \iota) \cdot \gamma \, .
\end{equation}
In particular, if one defines the \emph{wacky composition pairing} by
\begin{equation} 
\beta^k \colon \Hom^{-k}(Y,X_b) \times \Hom^{k}(X_b,Y) \to \K'
\, , \quad
(p,\iota) \mapsto P(p) \circ I(\iota)
\end{equation}
then \eqref{eq:K'lin} shows that this pairing is $\K'$-bilinear.

We claim the maps 
\eqref{eq:inducedmaps2} send the kernels of the 
graded composition pairing to those of the wacky composition pairing.
To see why, note that if $p\circ \iota = 0$ for all $p$, 
then is straightforward to verify that $p\circ (\sigast' \iota) = 0$ for all $p$.
Equation \eqref{E:whackycomppair} immediately implies that $P(p)\circ I(\iota) = 0$ for all $p$.
A similar argument shows that if $p\circ \iota = 0$ for all $\iota$, 
then $P(p)\circ I(\iota) = 0$ for all $\iota$.
Since the maps 
\eqref{eq:inducedmaps2} respect the kernels of composition pairings, 
one can similarly define a \emph{wacky non-degenerate composition pairing} 
between $V^{-k}(Y,X_b)$ and $V^{k}(X_b,Y)$.

If $\K$ is a field, 
one can continue as follows. 
The spaces $V^{-k}(Y,X_b)$ and $V^{k}(X_b,Y)$ are free over $\K'$, 
and hence their dimension over $\K$ is even. 
One can choose a $\K$ basis $\{p_1, x p_1, p_2, xp_2, \ldots, p_n, x p_n\}$ 
for $V^{-k}(Y,X_b)$ where the action of $x \in \K'$ is in rational canonical form. 
There is a dual basis, and \eqref{wackyadjunction} implies that 
$x$ is in (inverse) rational canonical form as well: 
if $\iota_1$ is dual to $p_1$ then $\sigast' \iota_1$ is dual to $\sigast' p_1$. 
The elements $\{p_1, \ldots, p_n\}$ and $\{\iota_1, \ldots, \iota_n\}$ 
form dual sets as bases over $\K'$, 
and one deduces the following result, 
which extends Proposition \ref{prop:SummandViaEV}.

\begin{prop}
	\label{prop:SummandViaEVfunky}
Suppose $\K$ is a field and let $(Y, \varphi_Y)$ be an object in $\ECat$.
\begin{enumerate}
\item For $b\in \BC^{\mathrm{free}}$, 
	the multiplicity of $\qsh^k (X_b\oplus X_{\sigma(b)}, \psi_b)$ in $(Y,\varphi_Y)$ 
	is equal to the multiplicity of $\qsh^k X_b$ in $Y$.

\item For $b\in \EBC^{\mathrm{fix}}$, 
	the multiplicity of $\qsh^k (X_b,\varphi_b)$ in $(Y,\varphi_Y)$ 
	equals the dimension of the $+1$-eigenspace of $\sigbast$ in $V^k(X_b, Y)$. 

\item For $b\in \EBC^{\mathrm{fix}}$, 
	the multiplicity of $\qsh^k (X_b,-\varphi_b)$ in $(Y,\varphi_Y)$ 
	equals the dimension of the {$-1$-eigenspace} of $\sigbast$ in $V^k(X_b, Y)$. 

\item For $b\in \BC^{\mathrm{fix}} \smallsetminus \EBC^{\mathrm{fix}}$, 
	the multiplicity of $\qsh^k (X_b\oplus X_b, \psi_b)$ in $(Y,\varphi_Y)$ is equal to 
	half the multiplicity of $\qsh^k X_b$ in $Y$, 
	which is equal to the dimension of $V^k(X_b,Y)$ as a $\K'$ vector space.
\end{enumerate}
\end{prop}

Recall from Notation \ref{not:ECatBC} that 
$\ECatBC$ denotes the full subcategory of $\ECat$ whose objects $(Y,\varphi_Y)$ satisfy $Y \in \Cat(\BC)$. 
When working over a field, the following version of Proposition \ref{prop:indECat}
classifies the indecomposable objects in this category without assuming Hypothesis \ref{hypo:assumeme}.

\begin{prop} \label{prop:indECatfunky}
Suppose $\K$ is a field. 
If $(X, \varphi_X)$ is an indecomposable object in $\ECatBC$, 
then exactly one of the following holds:
\begin{itemize}
\item $(X,\varphi_X) \cong \qsh^k (X_b, \varphi_b)$ for some $b\in \EBC^{\mathrm{fix}}$ and $k \in \Z$,
\item $(X,\varphi_X) \cong \qsh^k (X_b, -\varphi_b)$ for some $b\in \EBC^{\mathrm{fix}}$ and $k \in \Z$,
\item $(X,\varphi_X) \cong \qsh^k (X_b\oplus X_{\sigma(b)}, \psi_b) 
	\cong \qsh^k (X_b\oplus X_{\sigma(b)},-\psi_b)$ 
for some $b\in \BC^{\mathrm{free}}$ and $k \in \Z$, or
\item $(X,\varphi_X) \cong \qsh^k (X_b\oplus X_b, \psi_b) \cong \qsh^k (X_b\oplus X_b,-\psi_b)$ 
for some $b\in \BC^{\mathrm{fix}} \smallsetminus \EBC^{\mathrm{fix}}$ and $k \in \Z$.
\end{itemize}
Moreover, $\ECatBC$ is graded Krull--Schmidt.
\end{prop}

\begin{proof} 
We will analyze the possible indecomposable summands of an 
equivariant object $(Y, \varphi_Y)$.
Recall the functor $\ind \colon \Cat \to \ECat$ from \eqref{eq:ind} 
and observe that $(Y,\varphi_Y)$ is necessarily a summand of $\ind Y$.
(Precisely, we have $\ind Y \cong (Y,\varphi_Y) \oplus (Y, -\varphi_Y)$.)
We thus consider the possible indecomposable summands of $\ind Y$.

Since $\ind$ is additive, 
if $Y \in \Cat(\BC)$, 
then $\ind Y$ is a direct sum of shifts of $\ind X_b$ for various $b \in \BC$. 
If $b \in \EBC^{\mathrm{fix}}$, 
then (as above) $\ind X_b \cong (X_b, \varphi_b) \oplus (X_b, -\varphi_b)$. 
Otherwise $\ind X_b$ is either 
$(X_b\oplus X_{\sigma(b)}, \psi_b)$ for $b\in \BC^{\mathrm{free}}$
or
$(X_b\oplus X_b, \psi_b)$ for $b\in \BC^{\mathrm{fix}} \smallsetminus \EBC^{\mathrm{fix}}$.
In either of these latter cases, $\ind X_b$ is indecomposable
having degree zero endomorphism ring $\K$ or $\K'$ and no negative degree endomorphisms.

Regardless, we have a direct sum decomposition of $\ind Y$ into indecomposable objects 
whose endomorphism rings are graded local. 
It follows (e.g.~arguing as in \cite[Theorem 11.50]{soergelbook}) 
that $(Y,\varphi_Y)$ must decompose into these indecomposable summands.
This implies the graded Krull--Schmidt property, 
and further shows that if $(Y,\varphi_Y)$ is itself indecomposable, 
it must be one of the stated objects. \end{proof}

This yields the analogue of Proposition \ref{prop:computeequivGgp} in the present setup.

\begin{prop}
	\label{prop:computeequivGgpfunky}
Suppose $\K$ is a field. The $\sigma$-weighted Grothendieck group $\wKzero{\ECatBC}$ has basis
\[
\big\{ [(X_b, \varphi_b)]_{\sigma} \big\}_{b\in \EBC^{\mathrm{fix}}}
\]
in bijection with $\EBC^{\mathrm{fix}}$. 
Moreover, if $(Y, \varphi_Y)$ is any object in $\ECatBC$, then 
\begin{equation}
[(Y, \varphi_Y)]_{\sigma} 
= \sum_{k\in \Z}\sum_{b\in \EBC^{\mathrm{fix}}} \tr(\sigbast |_{V^k(X_b, Y)}) q^k [(X_b, \varphi_b)]_{\sigma} \, . 
\end{equation}
\qed
\end{prop}

\begin{proof} An immediate consequence of Propositions \ref{prop:SummandViaEVfunky} and \ref{prop:indECatfunky}.
The new indecomposable objects $(X_b\oplus X_b, \psi_b)$ are zero in the weighted Grothendieck group. \end{proof}

\begin{rem} 
When $\K$ is an algebraically closed field, 
Proposition \ref{prop:computeequivGgp} shows that the weighted Grothendieck group 
has a basis in bijection with $\BC^{\mathrm{fix}}$, 
while, for other fields, Proposition \ref{prop:computeequivGgpfunky} shows that
it has a basis in bijection with the subset $\EBC^{\mathrm{fix}}$.
Since the latter could be a proper subset, 
the size of the weighted Grothendieck group can depend on subtle features of $\K$. \end{rem}

When $\K$ is not a field, 
we do not how to imitate the arguments above. 
Even assuming that $V^k(X_b,Y)$ is free as a $\K$-module, 
we do not know how to deduce that $V^k(X_b,Y)$ is free as a $\K'$-module, or equivalently, 
that $x$ should act in rational canonical form. 
As a consequence, we do not know how to deduce that the multiplicity of $\indX_b$ 
in $(Y,\varphi_Y)$ is as large as expected.

\begin{example} 
Let $\K = \Z[\frac{1}{2}, \frac{1}{79}]$ and $\dd = 79$, so that $\K' = \Z[\frac{1}{2}, \frac{1}{79},\sqrt{79}]$. 
The ideal $(3,\sqrt{79}-1)$ inside $\K'$ is not principal, so it is not a free $\K'$-module. 
It is, however, a free $\K$-module of rank $2$. 
The question of extending a free module to a non-free module is related to the ideal class group; 
this is an example\footnote{Thanks to Nick Addington, Francis Dunn, Ellen Eischen, and Samantha Platt 
for finding and communicating this example.} 
where the ideal class group is cyclic of order 3. \end{example}

\begin{rem} As in the proof of Proposition \ref{prop:indECatfunky}, 
we know that every equivariant object $(Y,\varphi_Y)$ is a direct summand of $\ind Y$, 
which has a decomposition into various shifts of $\ind X_b$ whenever $Y \in \Cat(\BC)$. 
If $\K$ is local, and $\K[x]/(x^2 - \dd)$ is local for all $\dd$ arising from summands $X_b$ of $Y$, 
then we can prove Proposition \ref{prop:indECatfunky} with the same proof. 
Otherwise, we do not know how to rule out the possibility of $(Y,\varphi_Y)$ having 
even more exotic direct summands than the $\indX_b$. \end{rem}
	
\subsection{Considerations for categorification} \label{subsec:considerations}

Let $\Cat = \BnFoam$. 
In the body of this paper, we proved directly that our (endopositive) objects $\FEk_i$ are fixed by $\tau$ and have equivariant structures. 
Tensor products of these objects also admit equivariant structures; 
however, it is not obvious that any direct summand of such a tensor product will be equivariantizable. 

\begin{question}
For a given base ring $\K$, 
is every endopositive object in $\Kar(\BnFoam)$ either in a free orbit for $\tau$, 
or fixed by $\tau$ (up to isomorphism) and equivariantizable? 
\end{question}

If not, 
this would complicate the computation of the weighted Grothendieck group, 
which would depend on subtle properties of $\K$ (not just the characteristic). 
When $m=2$, we understand all indecomposable objects in $\BnFoam$: they all take the form $X_b$ for $b \in \tauEBC^{\mathrm{fix}}$. 
For $m > 2$, there is no comprehensive understanding of the indecomposable objects in $\BnFoam$. 
That said, all the direct summands we computed in order to prove the devil's Serre relations in Proposition \ref{T:iserre-in-twisted-K0}
are either in $\tauEBC^{\mathrm{fix}}$ or $\BC^{\mathrm{free}}$.
We are hopeful that $\BC^{\mathrm{fix}} = \tauEBC^{\mathrm{fix}}$, 
but, lacking much hard evidence, we find it irresponsible to posit such a conjecture at this time.

\bibliographystyle{plain}
\bibliography{TypeBLinkHomologyOneFile}
%

%
\end{document}